\documentclass[reqno]{amsart}

\usepackage{amsmath}
\usepackage{amssymb}
\usepackage{amsthm}
\usepackage{bbm}
\usepackage{enumitem}
\usepackage{tensor}
\usepackage[all]{xy}

\numberwithin{equation}{section}

\newtheorem{thm}{Theorem}[section]

\newtheorem{prp}[thm]{Proposition}
\newtheorem{lmm}[thm]{Lemma}
\newtheorem{cor}[thm]{Corollary}

\theoremstyle{definition}

\newtheorem{dfn}[thm]{Definition}

\theoremstyle{remark}

\newtheorem{rmk}[thm]{Remark}
\newtheorem{stp}{Step}

\newcommand{\cf}{cf.\ } 
\newcommand{\ie}{i.e.\ }

\newcommand{\eg}{eg.\ }
\newcommand{\etc}{etc.\ }

\newcommand{\ea}{et al.\ }
\newcommand{\resp}{resp.\ }

\newcommand{\N}{\mathbb{N}}   
\newcommand{\R}{\mathbb{R}}   
\newcommand{\C}{\mathbb{C}}   
\newcommand{\Z}{\mathbb{Z}}   

\newcommand{\e}{\varepsilon}   

\newcommand{\CP}{\mathbb{C}\mathbf{P}}   
\newcommand{\PP}{\mathbb{P}}   

\newcommand{\pr}{\mathrm{pr}}   
\newcommand{\ev}{ev}            
\newcommand{\vp}{\varphi}       
\newcommand{\p}{\mathrm{pt}}    
\newcommand{\End}{\mathrm{End}} 
\newcommand{\Aut}{\mathrm{Aut}} 
\newcommand{\Hom}{\mathrm{Hom}} 

\newcommand{\ee}{\epsilon}     

\newcommand{\ol}[1]{\overline{#1}}   
\newcommand{\ul}[1]{\underline{#1}}  
\newcommand{\wt}[1]{\widetilde{#1}}  
\newcommand{\wh}[1]{\widehat{#1}}    

\newcommand{\Sp}{{\mathbf\mathrm{Sp}}}     
\newcommand{\U}{{\mathbf\mathrm{U}}}       
\newcommand{\SO}{{\mathbf\mathrm{SO}}}     
\newcommand{\Spin}{{\mathbf\mathrm{Spin}}} 
\newcommand{\Sym}{\mathrm{Sym}}        
\renewcommand{\k}{\mathfrak{k}}        
\newcommand{\g}{\mathfrak{g}}          
\renewcommand{\t}{\mathfrak{t}}        
\newcommand{\one}{\mathbbm{1}}         

\newcommand{\D}{\mathcal{D}}   
\newcommand{\CC}{\mathcal{C}} 
\renewcommand{\P}{\mathcal{P}} 
\newcommand{\I}{\mathcal{I}}   
\newcommand{\B}{\mathcal{B}}    
\newcommand{\F}{\mathcal{F}}    
\newcommand{\E}{\mathcal{E}}    
\newcommand{\J}{\mathcal{J}}    
\newcommand{\X}{\mathcal{X}}    
\renewcommand{\O}{\mathcal{O}}  
\newcommand{\M}{\mathcal{M}}    
\newcommand{\NN}{\mathcal{N}}   
\newcommand{\KK}{\mathcal{K}}   
\newcommand{\Rbundle}{\underline{\mathbb{R}}} 
\newcommand{\V}{\mathcal{V}}    
\newcommand{\Vect}{\mathrm{Vect}} 

\newcommand{\ra}{\longrightarrow} 

\renewcommand{\L}{\mathcal{L}}            
\newcommand{\ostd}{\omega_{\mathrm{std}}} 
\newcommand{\Jstd}{J_\std}      
\newcommand{\A}{\mathcal{A}}              

\newcommand{\pp}{\mathfrak{p}}         
\newcommand{\h}{\mathfrak{h}}          
\newcommand{\ww}{\mathfrak{w}}         
\newcommand{\uu}{\mathfrak{v}}         
\newcommand{\ff}{\mathfrak{f}}         

\newcommand{\RS}{\mathrm{RS}} 
\newcommand{\FS}{\mathrm{FS}} 
\newcommand{\Mas}{\mathrm{Mas}} 
\newcommand{\Poz}{\mathrm{Poz}} 
\newcommand{\Vit}{\mathrm{Vit}} 
\newcommand{\Mor}{\mathrm{Mor}} 

\newcommand{\na}{\nabla} 
\newcommand{\pa}{\partial} 
\newcommand{\ns}{\nabla_s} 
\newcommand{\nt}{\nabla_t} 
\newcommand{\py}{\partial_y} 
\newcommand{\px}{\partial_x} 
\newcommand{\pt}{\partial_t} 
\newcommand{\ps}{\partial_s} 
\renewcommand{\d}{d} 
\newcommand{\CR}{\overline \partial} 

\newcommand{\nm}[1]{\left|#1\right|} 
\newcommand{\Nm}[1]{\left\|#1\right\|} 
\newcommand{\Nmm}[1]{\big\|#1\big\|} 
\newcommand{\nmm}[1]{\big|#1\big|} 
\newcommand{\di}[1]{\mathrm{dist}\left(#1\right)} 
\newcommand{\<}{\langle} 
\renewcommand{\>}{\rangle} 
\newcommand{\res}[2]{{#1}_{|_{#2}}}   
\newcommand{\Res}[2]{\left.{#1} \right|_{#2}} 

\newcommand{\reg}{\mathrm{reg}}   
\newcommand{\loc}{\mathrm{loc}}   
\newcommand{\can}{\mathrm{can}}   
\newcommand{\red}{\mathrm{red}}   
\newcommand{\std}{\mathrm{std}}   
\newcommand{\hor}{\mathrm{hor}}   
\newcommand{\ver}{\mathrm{ver}}   
\newcommand{\out}{\mathrm{out}}   
\newcommand{\adm}{\mathrm{adm}}   
\newcommand{\univ}{\mathrm{univ}} 

\DeclareMathOperator{\ind}{ind}   
\DeclareMathOperator{\dom}{dom}   
\DeclareMathOperator{\sign}{sign} 
\DeclareMathOperator{\im}{im}     

\DeclareMathOperator{\grad}{grad}     
\DeclareMathOperator{\crit}{crit}     
\DeclareMathOperator{\coker}{coker}   
\DeclareMathOperator{\Gr}{Gr}         
\DeclareMathOperator{\supp}{supp}     

\renewcommand{\graph}[1]{\mathrm{graph}\,{#1}} 

\newcommand{\regval}{\mathrm{regv}}
\renewcommand{\vert}{\mathrm{vert}}
\newcommand{\RR}{\mathcal{R}} 
 
\newcommand{\sol}{\sigma}   
\newcommand{\glue}{\mathcal{G}} 

\newcommand{\base}{x_*} 
\newcommand{\UU}{\mathcal{U}} 

\title{Floer Homology of Lagrangians in clean intersection}

\author[Schm\"aschke]{Felix Schm\"aschke}

\address{Max-Planck-Institut f\"ur Mathematik in den Naturwissenschaften, Inselstra\ss e 22, 04103, Leipzig, Germany}
\email{felix.schmaeschke@mis.mpg.de}

\begin{document}

\begin{abstract} We consider Floer homology associated to a pair of
  closed Lagrangian submanifolds that satisfy a monotonicty
  assumption. If the Lagrangians intersect cleanly we decribe two
  spectral sequences which help to compute their Floer homology. The
  spectral sequences are constructed using a Morse-Bott version of
  Floer homology.  We give a full treatment of the theory including
  the orientations.
\end{abstract}

\maketitle

\tableofcontents

\section{Introduction}

Floer homology was introduced by A. Floer in~\cite{Floer:Intersection}
to study the fixed points of Hamiltonian diffeomorphism and used to
prove the Arnold conjecture for aspherical symplectic manifolds.  It
is a module $HF_*(L_0,L_1)$ associated to two closed Lagrangian
submanifolds $L_0,L_1 \subset M$ of a symplectic manifold
$(M,\omega)$, which technically speaking is only well-defined under
additional assumptions on $L_0$ and $L_1$. If the requirements are met
and Floer homology is defined, there are natural isomorphisms
$HF_*(L_0,L_1) \cong HF_*(\vp_H(L_0),L_1) \cong HF_*(L_0,\vp_H(L_1))$
for any Hamiltonian diffeomorphisms $\vp_H:M \to M$ and moreover the
rank of $HF_*(L_0,L_1)$ gives a lower bound on the number of points in
$L_0 \cap \vp_H(L_1)$ whenever $L_0$ intersects $\vp_H(L_1)$
transversely. 

Let $\P(L_0,L_1)$ be the space of paths $x:[0,1] \to M$ such that
$(x(0),x(1)) \in L_0 \times L_1$ and fix an element $\base \in
\P(L_0,L_1)$. In this paper we review the construction of Floer
homology under the assumption that there exists constants $\tau >0$
and $N \in \N$ with $N \geq 3$ such that for all smooth maps $u:S^1
\times [0,1] \to M$ with $(u(\cdot,0),u(\cdot,1)) \subset L_0 \times
L_1$ and $u(0,\cdot)=\base$ we have
\begin{equation}
  \label{eq:Assumption}
  \tag{M}
\int u^*\omega = \tau \mu(u),\qquad \mu(u) \in N \Z\,.  
\end{equation}
Here $\mu(u)=\mu(\lambda_1)-\mu(\lambda_0)$ is the difference of the
Maslov indices of loops of linear Lagrangian subspaces $\lambda_0$
and $\lambda_1$ obtained from $u(\cdot,0)^*TL_0$ and
$u(\cdot,1)^*TL_1$ respectively by a symplectic trivialization of
$u^*TM$. These are precisely the assumptions of~\cite{Pozniak} and are
in particular satisfied if $L_0, L_1$ and $M$ are simply connected and
$M$ is monotone in the usual sense with minimal Chern number greater
or equal to $2$. Moreover we show that Floer homology is defined over the
integers whenever the pair $(L_0,L_1)$ is relative spin.

We are interested in the computation of Floer homology whenever
$L_0$ and $L_1$ intersect cleanly; \ie the intersection $L_0 \cap L_1$
is a submanifold of $M$ and for all $p \in L_0 \cap L_1$ we have
\[
T_p(L_0 \cap L_1) = T_p L_0 \cap T_p L_1\,.
\]
Such a situation arises frequently in examples. We consider the set of
connected components $C_1,\dots,C_k$ of $L_0\cap L_1$ which are
connected to the fixed path $\base$ within $\P(L_0,L_1)$ where we
think of a point in $L_0 \cap L_1$ as a constant path in
$\P(L_0,L_1)$. For every $j=1,\dots,k$ fix a point $p_j \in C_j$ and
choose a map $u_j:[-1,1]\times [0,1] \to M$ such that
$(u_j(\cdot,0),u_j(\cdot,1)) \subset L_0 \times L_1$, $u_j(-1,\cdot)
\equiv p_1$ and $u_j(1,\cdot) \equiv p_j \in C_j$.  Define
\[
\A(C_j) := -\int u_j^*\omega, \qquad \mu(C_j):= -\mu(u_j) - \frac
12 \dim C_j + \frac 12 \dim C_1\,,
\]
where $\mu(u_j)$ denotes the Robbin-Salamon index of the pair of
Lagrangian paths obtained from
$(u_j(\cdot,0)^*TL_0,u_j(\cdot,1)^*TL_1)$ using a symplectic
trivialization of $u_j^*TM$ which is constant on $u_j(\pm 1,\cdot)$
(\cf \cite{Robbin:Paths}). By~\eqref{eq:Assumption} we assume that
without loss of generality the maps $u_j$ are chosen such that for
all $j=1,\dots,k$ we have
\[
0 \leq \A(C_j) < \tau N\,.
\]
We denote by $0 \leq a_1 < a_2<\dots < a_\kappa < \tau N$ the values
attained by $\A(C_j)$ for $j=1,\dots,k$. With these preparations we
have the following result, which is based on ideas of
Oh~\cite{Oh:spectral}, Biran-Cornea~\cite{BiranCornea:quantum} and
Seidel~\cite{Seidel:knotted}. For the corresponding statement
including orientations see Theorem~\ref{thm:PozlocII}.
\begin{thm}\label{thm:Pozloc}
  Suppose that $L_0,L_1 \subset M$ are closed Lagrangian submanifolds
  in clean intersection such that~\eqref{eq:Assumption} holds. Then
  Floer homology is well-defined and there exists two spectral
  sequences $E^*_{**}$ and $E^{\loc,*}_{**}$ such that the following
  holds.
  \begin{enumerate}[label=(\roman*)] 
  \item The boundary operator of $E^r_{**}$ and $E^{r,\loc}_{**}$ has
    degree $(-r,r-1)$ for all $r \in \N$.
  \item The first page of $E^{\loc,*}_{**}$ is
    given by \[E_{ij}^{\loc,1} \cong
    \begin{cases} 
   \bigoplus_{\{\ell |\A(C_\ell)=a_i\}}
    H_{i+j-\mu(C_\ell)}(C_\ell;\Z_2)&\text{if }    1\leq i \leq \kappa,\ j \in \Z\\
    0&\text{if otherwise}\,.
    \end{cases}
    \] 
  \item  The sequence $E^{\loc,*}_{**}$ converges to a
    graded vector space $HF^\loc_*$ over $\Z_2$ and we have $E^1_{**}
    \cong \Z_2[\lambda^{\pm 1}] \otimes HF^\loc_*$ with $\deg \lambda
    =-N$.
  \item The sequence $E^*_{**}$ converges and we
    have \[\bigoplus_{i+j=*} E_{ij}^\infty\cong HF_*(L_0,L_1)\,.\]
  \end{enumerate}
\end{thm} 
We immediately obtain some displacement results.  We say that $L_0$ is
\emph{displaceable from $L_1$} if there exists a Hamiltonian
diffeomorphism $\vp_H$ such that $\vp_H(L_0) \cap L_1 = \emptyset$.
\begin{cor}\label{cor:displace}
  Suppose that $L_0$ is displaceable from $L_1$ and intersects $L_1$
  cleanly along a connected manifold $C=L_0 \cap L_1$, then we have
  \[
  N \leq \dim C +1\,.
  \]
  Moreover if $2N > \dim C +1$ then for all $k \in \N$ it holds 
  \[
    H_k(C;\Z_2) \cong
    \begin{cases}
      H_{k+N-1}(C;\Z_2) &\text{if } 0 \leq k \leq \dim C-N+1\,\\
      0 &\text{if } \dim C -N +2 \leq k \leq N-2\,\\
      H_{k-N+1}(C;\Z_2) &\text{if } N-1\leq k \leq \dim C\,.
    \end{cases}\]
\end{cor}
\begin{proof}
  Since there is only one connected component the sequence
  $E^{\loc,*}_{**}$ collapses at the first page and thus $E^1_{**}
  \cong \Z_2[\lambda^{\pm 1}] \otimes H_*(C;\Z_2)$ with $\deg
  \lambda=-N$.  If the boundary operator $\partial^r$ on page $r$ is
  non-trivial then necessarily $r\in N\N$. If $r=N$ we have for all
  $q,\bar p \in \Z$
  \[ H_q(C,\Z_2) \cong E^N_{\bar p N,q}
  \stackrel{\partial^N}{\longrightarrow} E^N_{(\bar p-1)N,q+N-1} \cong
  H_{q+N-1}(C,\Z_2)\,.  \] Suppose by contradiction that $N > \dim
  C+1$, then we conclude that $\partial^N$ is trivial. Inductively
  we show that $\partial^{\bar r N}:H_*(C;\Z_2) \to H_{*+\bar
    rN-1}(C;\Z_2)$ is trivial for all $\bar r \in \N$. Hence
  $HF_*(L_0,L_1) \cong \bigoplus_{p+q=*} E^1_{pq} \neq 0$. But if
  $L_0$ is displaceable from $L_1$ the module $HF_*(L_0,L_1)$
  vanishes. This shows the first claim.  Suppose now that $2N >\dim
  C+1$. In a similarly manner we show that the only possibly
  non-trivial boundary operator is on page $r=N$. Thus $0 \cong
  HF_*(L_0,L_1)\cong \bigoplus_{p+q=*} E^2_{p,q} \cong
  \ker \partial^N/\im \partial^N$. Using $\partial^N:H_*(C;\Z_2) \to
  H_{*+N-1}(C;\Z_2)$ we conclude the second statement.
\end{proof}
The case $L_0=L_1=L$ is a special case of a clean intersection and the
previous corollary implies the well-known result about closed monotone
Lagrangian submanifolds.
\begin{thm}[Polterovich, Oh, Albers] If a monotone Lagrangian submanifold $L$
  is displaceable, then the minimal Maslov number $N_L$ satisfies 
  \[
  N_L \leq \dim L +1\,.
  \]
  Moreover we have $H_k(L;\Z_2)=0$ for all $\dim L-N_L+2 \leq k \leq
  N_L -2$.  
\end{thm}
From the theorem we also obtain Pozniak's result about Lagrangian
submanifolds intersecting cleanly in one connected component.
\begin{thm}[Pozniak]
  Given closed Lagrangian submanifolds $L_0,L_1 \subset M$ satisfying
  assumption~\eqref{eq:Assumption} and such that $L_0$ intersects
  $L_1$ cleanly along a connected manifold $C=L_0 \cap L_1$ with $\dim
  C+1<N$, then $HF_*(L_0,L_1) \cong H_*(C;\Z_2)$ for all $0 \leq *\leq
  \dim C$.
\end{thm}
As a third application, we deduce a new result about the topology of
the intersection of two simply connected Lagrangian submanifolds in $\CP^n \times
\CP^n$. It is a generalization of a result of Fortune~\cite{Fortune}
about fixed points of symplectomorphisms of $\CP^n$. Let $\omega_\FS$
denote the Fubini-Study symplectic form on $\CP^n$.
\begin{prp}
  Let $\CP^n \oplus \CP^n$ be equipped with the symplectic form
  $\omega_\FS \oplus -\omega_\FS$. Give two simply connected
  Lagrangian submanifolds $L_0,L_1 \subset \CP^n \times \CP^n$ intersecting
  cleanly with $L_0 \neq L_1$. Then $L_0 \cap L_1$ has at least two
  connected components. Moreover assume that the intersection $L_0
  \cap L_1$ consists of two disjoint connected components one of which
  is a point, then we have 
  \[H_*(C,\Z_2) \cong H_*(\CP^{n-1},\Z_2)\,,\] where $C$ denotes the
  other connected component.
\end{prp}
\begin{proof}
  The Lagrangians have to intersect by a result of
  Albers~\cite{Peter}. Assume that $L_0 \cap L_1 = C$ has only one
  connected component.  The minimal Chern number of $\CP^n \times
  \CP^n$ equals $n+1$. Since $L_0$ and $L_1$ are simply connected,
  assumption~\eqref{eq:Assumption} holds with $N=2(n+1)$. As in the
  proof of Corollary~\ref{cor:displace} we conclude that
  $HF_*(L_0,L_1) \cong \Z_2[\lambda^{\pm 1}] \otimes H_*(C;\Z_2)$. By
  the quantum action Floer homology $HF_*(L_0,L_1)$ is a module over
  the quantum cohomology ring of $\CP^n \times \CP^n$ which contains
  an invertible element of degree two. Hence $HF_k(L_0,L_1) \cong
  HF_{k+2}(L_0,L_1)$ for all $k \in \Z$. But this leads to the
  contradiction because if  $L_0 \neq L_1$ then $\dim C < 2n$ yet on
  the other hand we have
  \begin{equation*}
    H_0(C;\Z_2)  \cong HF_0(L_0,L_1) \cong HF_2(L_0,L_1) \cong
    \dots \cong HF_{2n}(L_0,L_1)\cong H_{2n}(C;\Z_2) \,.
  \end{equation*}
  Now assume that $L_0 \cap L_1= \{\p\} \cup C$. Without loss of
  generality we assume that the base point lies on the component $C$
  and let $d =\mu(\p) \in \Z$ denote the index of the intersection
  point which does not lie on $C$. Then the local spectral sequence
  collapses at the second page and we have $E^{\loc,\infty}_* \cong
  \ker \partial^{\loc,1} \oplus \coker \partial^{\loc,1}$, where
  $\partial^{\loc,1}:H_*(C;\Z_2) \to \Z_2[d]$ (here $\Z_2[d]$ denotes
  the group $\Z_2$ in degree $d$). We have a case distinction.  In the
  first case we assume that $\partial^{\loc,1} \neq 0$. Then
  $\coker \partial^{\loc,1} = 0$ and $HF^{\loc}_*$ is only supported
  in degrees $0,1,\dots,\dim C < 2n$. Similarly as above we conclude
  that $HF_*(L_0,L_1) \cong \Z_2[\lambda^{\pm 1}] \otimes HF_*^\loc$
  with $\deg \lambda =-N$ which by degree reasons leads to a
  contradiction as before using $HF_*(L_0,L_1) \cong
  HF_{*+2}(L_0,L_1)$.

  In the second case we assume that $\partial^{\loc,1} =0$. Then
  $E^1_{**} \cong \Z_2[\lambda^{\pm 1}] \otimes (H_*(C;\Z_2) \oplus
  \Z_2[d]) $. If $0 \leq d \leq 2n-1$, we obtain a contradiction as
  above. If $d > 2n-1$ we can not conclude by degree reasons that
  $E^*_{**}$ collapses at the first page, since there might possibly
  exist $\bar r \in \N$ such that $\partial^{\bar r N}$ is
  non-trivial. Thus we find a non-trivial map of degree $rN-1$ such
  that $H_*(C;\Z_2) \to \Z_2[d]$. But then the sequence collapses at
  the $\bar rN+1$-page and we have $HF_*(L_0,L_1) \cong
  \Z_2[\lambda^{\pm 1}] \otimes H_*$ in which $H_*\subset H_*(C;\Z_2)$
  is a subspace of codimension one. Unless $H_* =0$ and $n=1$ we again
  obtain a contradiction using $HF_*(L_0,L_1) \cong
  HF_{*+2}(L_0,L_1)$. Yet if $H_*=0$ and $n=1$ then $H_*(C;\Z_2) \cong
  \Z_2 \cong H_*(\CP^0;\Z_2)$ as claimed.  Finally if on the other
  hand $\partial^r=0$ for all $r \geq 1$, then $HF_*(L_0,L_1) \cong
  \Z_2[\lambda^{\pm 1}] \otimes \big(H_*(C;\Z_2) \oplus
  \Z_2[d]\big)$. The only possibility which does not lead to a
  contradiction using $HF_*(L_0,L_1) \cong HF_{*+2}(L_0,L_1)$ is if $d
  = 2n \mod 2n+2$ and
  \[
  H_k(C;\Z_2) \cong
  \begin{cases}
    \Z_2 &\text{if } k=0,2,\dots,2n-n,\\
    0 &\text{otherwise}\,.
  \end{cases}
  \]
  But then $C$ has exactly the same homology as $\CP^{n-1}$ as
  claimed.
\end{proof}

Floer homology of Lagrangians in clean intersection was at first
considered by Pozniak in~\cite{Pozniak}, where he carefully choose
perturbations by Hamiltonian diffeomorphisms to move $L_0$ and $L_1$
into transverse position and then identified holomorphic strips close
to $L_0 \cap L_1$ which appear in the definition of the boundary
operator of the Floer homology complex with Morse trajectories on the
intersection manifold $L_0 \cap L_1$.  Instead we leave the
Lagrangians as they are and treat the action functional for the
degenerate situation. It turns out that the functional is degenerated
in the sense Morse-Bott and thus it is possible to define Floer
homology, which for the sake of distinction we call \emph{pearl
  homology of $L_0$ and $L_1$} and denote it by $QH_*(L_0,L_1)$
although it is naturally isomorphic to $HF_*(L_0,L_1)$. More precisely
the module $QH_*(L_0,L_1)$ is the homology of a chain complex
generated by the critical points of a Morse function on $L_0 \cap L_1$
equipped with the boundary operator obtained by counting so called
\emph{pearl trajectories} or \emph{cascades}. These are tuples of
pseudo-holomorphic strips with boundary on $(L_0,L_1)$ that are
connected by gradient flow lines in $L_0 \cap L_1$. This complex was
previously studied Fukaya-Ohta-Ono-Oh~\cite{FO3:I,FO3:II,FO3:integers}
and Frauenfelder~\cite{Frauenfelder:PhD}. Nonetheless we repeat the
construction of the complex in full detail for completeness. In
particular we include some details which have not been treated (\eg
subjectivity of gluing) and also give a slightly different approach to
orientations which is more adapted to the interpretation of the Floer
complex as a Morse complex.

This paper is part of my PhD thesis, which I have completed under the
supervision of Prof.\ Matthias Schwarz. I am immensely indebted to him
for continuous support and many helpful discussions. Equally I thank
Prof.\ Abbondandolo for keeping me motivated during my time in Bochum
and helping me to wrap my head around the orientations in Floer
homology. I am also grateful to many other mathematicians for sharing
their ideas with me, among them Dr.\ Rotislav Matveyev, Prof.\ Katrin
Wehrheim, Prof.\ Joel Fish, Prof.\ Urs Frauenfelder, Prof.\ Peter
Albers and Dr.\ Urs Fuchs. I thank Murat Saglam, Dr.\ Urs Fuchs, Dr.\
Sonja Hohloch, Dr.\ Roland Voigt and Dr.\ Luca Asselle for carefully
reading parts of my script and giving useful suggestions. Finally I
want to thank Prof. Felix Schlenk for encouraging me to get this
article published.

\section{Background}
\subsection{Symplectic manifolds and Lagrangians}\label{sec:symplag}
A \emph{symplectic manifold} $(M,\omega)$ is a $2n$-dimensional
manifold $M$ equipped with a \emph{symplectic form} $\omega$, which is
a $2$-form that is closed (\ie $\d \omega=0$) and non-degenerated (\ie
$\omega^{\wedge n}\neq 0$).  An \emph{almost complex structure on $M$}
is a complex structure on the tangent bundle $TM$ given by an
endomorphism $J:TM \to TM$ such that $J^2 = -\one$. The almost complex
structure is called \emph{$\omega$-compatible}, if
\[g_J = \omega(\cdot,J\cdot)\] defines a Riemannian metric on $M$. We
denote by $\End(TM,\omega)$ the space of all almost complex structures
on $M$ which are compatible to a fixed $\omega$. A complex structure
on $TM$ induces a first Chern class (see~\cite[Section
20]{BottTu}). Since the space of compatible almost complex structures
is contractible the Chern class does not depent on the choice $J
\in \End(TM,\omega)$ and is denoted $c_1(\omega) \in H^2(M,\Z)$.

Let $H^S_2(M)$ be the image of the Hurewicz morphism $\pi_2(M) \to
H_2(M,\Z)$.  Evaluation of $c_1(\omega)$ and $[\omega]$ on elements
in $H_2^S(M)$ defines two homomorphisms
\[ I_c : H_2^S(M) \to \Z,\qquad \qquad \qquad I_\omega: H_2^S(M) \to
\R \;.\] We define the \emph{minimal Chern number of $M$} as the
smallest positive value of $I_c$, \ie $c_M:=\min\{ I_c(A) \mid A \in
H_2^S(M) \text{ with } I_c(A) > 0\}$.  A symplectic manifold
$(M,\omega)$ is
\begin{itemize}
\item \emph{symplectically aspherical}, if for all classes $a \in
  H^S_2(M)$ we have $I_\omega(a) = I_c(a) = 0$,
\item \emph{monotone} or more precisely \emph{$\tau$-monotone}, if
  there exists a constant $\tau > 0$, such that for all classes $a \in
  H^S_2(M)$ we have $I_\omega(a) = 2 \tau I_c(a)$.
\end{itemize}
These assumptions where introduced by Floer and lead to a
simplification of the analysis. Unless otherwise noted all symplectic
manifolds in this work are either symplectically aspherical or
monotone.

\subsubsection{Lagrangians} A submanifold $L \subset M$ is
\emph{isotropic} if $\omega$ vanishes on all pairs of vectors tangent
to $L$. A \emph{Lagrangian submanifold} is an isotropic submanifold $L
\subset M$ such that $\dim L = n$. By the non-degeneracy of $\omega$,
any isotropic submanifold has a dimension of at most $n$. From that
viewpoint Lagrangian submanifolds are sometimes called \emph{maximally
  isotropic}. Correspondingly we have similar homological
requirements, which where introduced by~\cite{Oh:diskI} and again lead
to simplification of the analysis. Let $H^S_2(M,L)$ be the image of
the relative Hurewicz morphism $\pi_2(M,L) \to H_2(M,L)$.  Evaluation
of $[\omega]$ and the Maslov index (\cf \cite[Section C.3]{Bibel})
defines two homomorphisms
\[I_\mu: H_2^S(M,L) \to \Z,\qquad \qquad \qquad
I_\omega:H_2^S(M,L)\to\R\;.\] Similarly as above we define the
\emph{minimal Maslov number of $L$} as the smallest positive value of
$I_\mu$, \ie $N_L:=\min\{I_\mu(A)\mid A \in H_2^S(M,L) \text{ with }
I_\mu(A) >0\}$. A Lagrangian submanifold $L \subset M$ is
\begin{itemize}
\item \emph{symplectically aspherical} if for all classes $a \in
  H^S_2(M,L)$ we have $I_\omega(a) = I_\mu(a) = 0$,
\item \emph{$\tau$-monotone} if there exists a constant $\tau > 0$
  such that for all classes $a \in H^S_2(M,L)$ we have $ I_\omega(a) =
  \tau I_\mu(a)$.
\end{itemize}
\begin{rmk}
  If $L \subset M$ is symplectically aspherical then $M$ is
  necessarily symplectically aspherical as well and if $L$ is
  $\tau$-monotone then $M$ is $\tau$-monotone or symplectically
  aspherical. For that reason we purposely included the factor $2$ in
  the definition of the monotonicity constant of a monotone symplectic
  manifold. Another basic observation is that the minimal Maslov
  number of a Lagrangian $L \subset M$ always divides $2c_M$.
\end{rmk}
\begin{lmm}\label{lmm:Lmonotone}
  Let $(M,\omega)$ be a monotone symplectic manifold and $L \subset M$
  be a Lagrangian submanifold such that the fundamental group
  $\pi_1(L)$ is finite, then $L$ is monotone. Suppose that $\pi_1(L)$
  is trivial, then the minimal Maslov number of $L$ equals $2c_M$,
  where $c_M$ is the minimal Chern number of $M$.
\end{lmm}
\begin{proof} 
  Let $u:(D,\partial D) \to (M,L)$ be a disc with boundary on
  $L$. After a suitable cover $\vp:D \to D$ of some degree $k \in \N$
  the boundary of the composition $\tilde u = u \circ \vp$ is
  contractible within $L$, \ie there exists $v:D \to L$ such that
  $v|_{\partial D} = \tilde u|_{\partial D}$. Let $w= u \sqcup v/\!\sim$
  be the map defined on $S^2\cong D\sqcup D/\!\sim$ with boundary
  points identified.  Hence
  \[ I_\omega([w])= \int w^* \omega = \int \tilde u^* \omega + \int
  v^*\omega = k \int u^*\omega = k I_\omega(u)\,,
  \]
  and by \cite[Thm.\ C.3.10]{Bibel} 
  \[
  2I_c([w])= 2\<c_1(TM),[w]\> = \mu_\Mas(\tilde u) + \mu_\Mas(v) = k
  \mu_\Mas(u) = k I_\mu(u)\;.\] According to the assumption there
  exists a $\tau >0$ such that $k I_\omega([u]) = I_\omega([w]) =
  2\tau I_c([w]) = \tau k I_\mu([u])$. This shows that $L$ is
  monotone.  If $\pi_1(L)$ is trivial then $k=1$ and $I_\mu([u]) =
  2I_c([w]) \in 2c_M \Z$ for all $u$. This shows that $2c_M$ divides
  the minimal Maslov number of $L$, denoted $N_L$. But since $N_L$
  always divides $2c_M$ we have $N_L=2c_M$. 
\end{proof}

\begin{lmm}\label{lmm:Diagonal}
  The diagonal $\Delta=\{(p,p) \mid p \in M\}$ is a Lagrangian
  submanifold of $(M\times M,\omega\oplus -\omega)$ with minimal
  Maslov number given by twice the minimal Chern number of
  $M$. Moreover $M$ is monotone if and only if $\Delta$ is monotone.
\end{lmm}
\begin{proof}
  Identify a disk $u=(u_0,u_1):(D,\partial D) \to (M\times M,\Delta)$
  uniquely with a sphere $v:\PP^1 \to M$ via $v(z):=u_0(z)$ for
  $\nm{z} \leq 1$ and $v(z):=u_1(1/\bar z)$ for $\nm{z} \geq
  1$. Conversely given a sphere $v$ we obtain a disk $u:(D,\partial D)
  \to (M\times M,\Delta)$ by the same identification.  Choose
  trivializations $\Phi_0:u_0^*TM \to D \times \C^n$ and
  $\Phi_1:u_1^*TM \to D \times \C^n$. Denote by $\Psi:S^1 \to \U(n)$,
  $\theta \mapsto \Phi_1(\theta)\Phi_0(\theta)^{-1}$ the
  trivialization change along $\partial D = S^1$. For every $\theta
  \in S^1$ define the linear Lagrangian subspace $F(\theta) :=
  (\Phi_0(\theta) \oplus \Phi_1(\theta))T_{(u_0,u_1)}\Delta= \graph
  \Psi(\theta) \subset \C^n \times \C^n$. By definition of the Maslov
  index (see \cite[Theorem C.3.6]{Bibel}) we have
  \[
  I_\mu([u]) = \mu_\Mas(F)=\deg \det \Psi^2 = 2 \deg \det \Psi =
  2\<c_1,[v]\>= 2I_c([v])\;.
  \]
  This shows the claim by choosing $u$ such that $I_\mu([u])$ equals
  the minimal Maslov number. The supplement follows directly since
  $I_\omega([u]) = I_\omega([v])$.
\end{proof}

\subsubsection{Lagrangian pairs} Given two Lagrangian submanifolds
$L_0,L_1 \subset M$. We denote the \emph{path space} 
\begin{equation}
  \label{eq:PL0L1}
  \P(L_0,L_1) := \{ x\in C^\infty([0,1],M) \mid x(0) \in L_0,\ x(1) \in
  L_1\}\,.
\end{equation}
Fix an element $\base \in \P(L_0,L_1)$. Given a smooth map $u:[-1,1] \times
[0,1] \to M$ such that
\[ u(-1,\cdot) =u(1,\cdot)= \base,\qquad u(\cdot,0) \subset L_0,\qquad
u(\cdot,1)\subset L_1\,,\] the map $s \mapsto u(s,\cdot)$ defines a
loop in $\P(L_0,L_1)$.  Every loop in $\P(L_0,L_1)$ based in $\base$ is
homotop to a loop of this type. Integrating the symplectic form
over $u$ or by evaluating the Maslov index on $u$ we obtain two ring
homomorphisms
\[
I_\omega:\pi_1(\P(L_0,L_1),\base) \to \R,\qquad \qquad
I_\mu:\pi_1(\P(L_0,L_1),\base) \to \Z\;.
\]
We define the \emph{minimal Maslov number of $(L_0,L_1)$ with respect
  to $\base$} as smallest positive value of $I_\mu$.  We have
corresponding homological requirements. The pair $(L_0,L_1)$ is called
\begin{itemize}
\item \emph{symplectically aspherical with respect to $\base$} if for
  all $a \in \pi_1(\P(L_0,L_1);\base)$ we have $I_\omega(a) = I_\mu(a)
  =0$,
\item \emph{$\tau$-monotone with respect to $\base$} if there exists a
  constant $\tau >0$ such that for all $a \in \pi_1(\P(L_0,L_1),\base)$
  we have $I_\omega(a) = \tau I_\mu(a)$.
\end{itemize}
For simplicity we will write that $(L_0,L_1)$ is \emph{monotone} if
the choice of base point $\base$ is self-understood. If $(L_0,L_1)$ is
monotone with minimal Maslov number $N$, then again $M$, $L_0$ and
$L_1$ is monotone or symplectically aspherical. Moreover $N$ divides
$2c_M$ and the minimal Maslov number of each $L_0$ and $L_1$ (\cf
\cite[Remark 3.3.2]{Pozniak}).
\subsection{Hamiltonian action functional}\label{sec:ham}
Consider a symplectic manifold $(M,\omega)$ and two Lagrangians
submanifolds $L_0,L_1 \subset M$. We denote by $H^{1,2}([0,1],M)$ the
space of all absolutely continous maps with square-integrable first
derivatives and replace the path space $\P=\P(L_0,L_1)$ with the space
$\{x \in H^{1,2}([0,1],M) \mid (x(0),x(1)) \in L_0\times
L_1\}$. Since this space has the same weak homotopy type the previous
considerations still hold true.  Fix an element $\base \in \P$ and a
Hamiltonian function $H\in C^\infty([0,1]\times M)$. We consider the
\emph{Hamiltonian action $1$-form} on $\P$ given by
\[
\alpha_H(x)\xi = \int_0^1 \omega(\dot x,\xi) - dH(x)\xi\, \d t\qquad
\xi \in T_x \P\,.
\]
This form is closed but not exact in general. We define a suitable
cover of $\P$ for which there exists a primitive of $\alpha_H$. Given
a point $x \in \P$, we consider equivalence classes of maps
$u:[-1,1]\times [0,1] \to M$ such that $u(-1,\cdot) = \base$,
$u(1,\cdot)= x$ and $u(s,\cdot) \in \P$ for all $s \in [-1,1]$. Two
such maps $u$ and $v$ are equivalent if $\int u^* \omega = \int
v^*\omega$. Let $\wt \P$ denote the set of equivalence classes and
denote the elements of $\wt \P$ by pairs $[u,x]$. Since $\P$ is
locally path-connected $\wt \P$ carries an induced topology and is in
fact a covering space over the connected component of $\P$ containing
$\base$. The covering map is given by $[u,x] \mapsto x$.  The group of
Deck transformations of this cover is $\pi_1(\P;\base) /\ker I_\omega$
which is acting transitively and effectively via $[v].[u,x] = [v\#
u,x]$ where $\#$ denotes the concatenation of maps. The pull-back of
$\alpha_H$ to $\tilde \P$ is exact. A primitive is called
\emph{Hamiltonian action functional} and given by
\[
\A_H([u,x]) := -\int u^*\omega - \int_0^1 H(x)\, dt\,.
\]
\subsubsection{Critical points of $\A_H$} Critical points of $\A_H$
correspond to solutions of the Hamiltonian equation. We choose the
following convention in order to define the \emph{Hamiltonian vector
  field $X_H$},
\begin{equation}
  \label{eq:XH}
 \omega(X_H,\cdot) = \d H\;.
\end{equation}
Define the \emph{perturbed intersection points}
\begin{equation}
  \label{eq:IH}
  \I_H(L_0,L_1) := \{x:[0,1] \to M \mid \dot x = X_H(x),\ (x(0),x(1))
  \in L_0\times L_1\}\;.
\end{equation}
Note that the set $\I_H(L_0,L_1)$ is in bijection with $\vp_H(L_0)
\cap L_1$, where $\vp_H$ denotes the Hamiltonian diffeomorphism
associated to $H$, \ie the time-one map of the flow associated to
the Hamiltonian vector field $X_H$. 
\begin{lmm}\label{lmm:critA}
  Critical points of $\A_H$ are exactly the points $[u_x,x] \in \wt
  \P$ with $x\in\I_H(L_0,L_1)$.
\end{lmm}
\begin{proof}
  We fix $\e >0$ and let $u:(-\e,\e)\times [-1,1]\times [0,1] \to M$
  be a smooth map such that the maps $u_\tau = u(\tau,\cdot)$ satisfy
  $u_\tau\big|_{t=0,1} \subset L_{0,1}$ and $u_\tau(-1,\cdot)=\base$
  for all $\tau \in (-\e,\e)$. We write $x_\tau = u_\tau(1,\cdot)$ and
  $\xi = \Res{\partial_\tau}{\tau=0} u_\tau(1,\cdot)$.
  \begin{align*}
    &\Res{\frac{d}{d\tau}}{\tau=0}\A_H\left(u_\tau,x_\tau\right) = \\
    &\qquad =
    -\Res{\int_{[-1,1]\times [0,1]} \partial_\tau
      \omega\left(\partial_su_\tau,\partial_tu_\tau\right)\d s\d
      t}{\tau=0}
    - \Res{\int_0^1 \partial_\tau H(t,x_\tau(t))\d t}{\tau=0}\\
    &\qquad = -\Res{\int_0^1 \omega\left(\partial_\tau u_\tau,\partial_t
        u_\tau\right)\,\d t}{\tau=0,s=1} -\int_0^1
    dH(t,\cdot) \xi\,\d t\\
    &\qquad =-\int_0^1 \omega \left(\xi,\dot x\right)\,\d t -\int_0^1\omega (X_H,\xi)\,\d t\\
    &\qquad =\int_0^1 \omega \left(\dot x - X_H,\xi\right) \,\d t\;.
  \end{align*}
  For the second line we use
  \begin{equation*}
    0=d\omega\left(\partial_\tau u,\partial_s u,\partial_t u\right)=\partial_\tau\omega\left( \partial_s u,\partial_t u\right)
    - \partial_s \omega \left(\partial_\tau u,\partial_t u\right)
    +\partial_t\omega \left(\partial_\tau u,\partial_s u\right)\;,
  \end{equation*}
  and that integration over the $\partial_t$-part vanishes by the
  Lagrangian boundary conditions.  By non-degeneracy of the symplectic
  form $\omega$ we conclude that $[u_x,x]$ is a critical point of
  $\A_H$ if and only if $\dot x = X_H(x)$.
\end{proof}
\subsubsection{Gradient of $\A_H$}
A key observation of Floer was that the gradient of $\A_H$ with
respect to a certain $L^2$-metric on $\P$ establishes ties between
Morse theory and $J$-holomorphic curve theory.  More precisely fix a
path $J:[0,1] \to \End(TM,\omega)$, we define an $L^2$-metric on the
path space $\P$ via
\begin{equation}
  \label{eq:gJL2}
  \<\xi,\eta\>_{J} = \int_0^1 \<\xi(t),\eta(t)\>_{J_t} \d t =
  \int_0^1 \omega_{x(t)}\big(\xi(t),J_t(x(t)) \eta(t)\big)\d t\;,
\end{equation}
for all sections $\xi,\eta \in \Gamma(x^*TM)$ and $x \in \P$. The
metric canonically lifts to $\wt \P$. As one sees at the
formula~\eqref{eq:gradAH} of the next lemma, the gradient of $\A_H$ is
independent of the choice of the base point and descends to a vector
field on the path space $\P$.
\begin{lmm}
  The gradient of the functional $\A_H$ with respect to the
  metric~\eqref{eq:gJL2} is given by
  \begin{equation}
    \label{eq:gradAH}
     \grad_{J} \A_H(u_x,x) = J\left( \pt x - X_H(x)\right)\;.
  \end{equation}
\end{lmm}
\begin{proof}   
  Given $\xi \in C^\infty(x^*TM)$. Continuing the computation given in the
  proof of Lemma~\ref{lmm:critA} we see that
  \begin{align*}
    \d \A_H(u_x,x)[\xi] &=\int_0^1 \omega_{x(t)} \left(\dot x(t) -
      X_H(t,x(t)),\xi(t)\right) \,\d t\\
    &=\int_0^1 \omega_{x(t)} \left(\xi(t),J^2_t(x(t))\left(\dot x(t) - X_H(t,x(t))\right)\right) \,\d t\\
    &=\int_0^1 \<\xi(t),J_t(x(t))\left(\dot x(t) -X_H(t,x(t))\right)\>_{J}\,\d t\;.\\
    &=\<\xi,J\left(\pt x - X_H(x)\right)\>_{J}\;.
  \end{align*}  
  This shows the claim.
\end{proof}
Another crucial idea of Floer was that despite the fact that the
negative gradient flow of $\A_H$ is not well-defined, finite energy
negative-gradient flow-lines between any two critical points are. A
gradient flow line between the critical points $[u_-,x_-]$ and
$[u_+,x_+]$ is given by a map $u:\R \times [0,1] \to M$ such that
\begin{equation}
  \label{eq:Floertraj}
  \begin{gathered}
    \ps u + J(u)\left( \pt u - X_H(u)\right) = 0\,,\\
    u\big|_{t=0} \subset L_0 ,\qquad u\big|_{t=1} \subset L_1\,,\\
    \int_{\R \times [0,1]} \nm{\ps u}_J^2 \d s \d t < \infty\,,\\
    \lim_{s \to -\infty} u(s,\cdot)=x_-,\qquad \lim_{s\to \infty}
    u(s,\cdot) = x_+\,,
  \end{gathered}
\end{equation}
where the limits in the last line are in uniform topology.  We call
$u$ a \emph{finite-energy $(J,H)$-holomorphic strip with boundary in
  $(L_0,L_1)$ connecting $x_-$ to $x_+$}.  These ``generalized''
flow-lines satisfy the same properties of negative gradient flow lines
in Morse theory. For example if $[u_+,y] = [u_-\#u,y]$ we have the
\emph{action-energy relation}
\begin{equation}
  \label{eq:AErel}
  E(u) = \int_{\R \times [0,1]} \nm{\ps u}^2_J \d s \d t = \A_H(u_-,x_-) - \A_H(u_+,x_+)\;.  
\end{equation}
There is a standard trick to transform a solution
of~\eqref{eq:Floertraj} into a solution with $H=0$ but changing $L_1$
and $J$. We will use it at several places in the paper.
\begin{lmm}\label{lmm:change}
  Given a Hamiltonian function $H \in C^\infty([0,1]\times M)$ and an
  almost complex structure $J \in C^\infty([0,1],\End(TM,\omega))$.
  Let $u:\R \times [0,1] \to M$ be solution of
  \[
  \ps u + J(u)\big(\pt u - X_H(u)\big) =0\,,
  \]
  which satisfies the boundary condition
  \[
  u(\cdot,0) \subset L_0,\qquad u(\cdot,1) \subset L_1\,.
  \]
  Then the map $v: \R \times [0,1]\to M$ defined by $v(s,t) =
  \vp_H^t(u(s,t))$ is a solution of
  \[
  \ps v + J'(v) \pt v = 0,\qquad J'_t := (\d\vp_H^t)^{-1} \circ J_t
  \circ \d\vp_H^t\,,
  \]
  which satisfies the boundary condition
  \[
  v(\cdot,0) \subset L_0,\qquad v(\cdot,1) \subset \vp_H^{-1}(L_1)\,.
  \]
  Moreover we have $E(u)=E(v)$.
\end{lmm}
\begin{proof}
  Obviously the curve $v$ satisfies the boundary condition by
  construction. We check the differential equation. We have $\ps u= \d \varphi_H
  \ps v$ and $\pt u = X_H(u) + \d \varphi_H \pt v$ and thus
  \[\d \varphi_H \left(\ps v + J'(v) \pt v\right) = \ps u + J(u)\left(\pt u - X_H(u)\right)=0\;.\]
  This shows that $v$ is $J'$-holomorphic.  Then
\[\nm{\ps u}^2 = \omega(\ps u,J\ps u) = \omega(\d \varphi_H \ps
v,J\d \varphi_H \ps v) = \omega(\ps v,J' \ps v) = \nm{\ps
  v}^2\,.\] This shows $E(u)=E(v)$.
\end{proof}

\subsubsection{Hessian of $\A_H$} 
Let $\na^t$ denote the Levi-Civita connection with respect to the
metric $\omega(\cdot,J_t \cdot)$ for each $t \in [0,1]$. Given
$x \in \P$, we define the \emph{Hessian of the Hamiltonian
  action functional} as the operator
\begin{equation}
  \label{eq:HessA}
  A_x:T_x\P(L_0,L_1) \to L^2(x^*TM), \qquad \xi \mapsto J(x)\left(\na_t\xi - \na_\xi X_H\right)\;, 
\end{equation}
with domain $T_x\P(L_0,L_1)\subset L^2(x^*TM)$ given by
\[
T_x\P(L_0,L_1) = \{ \xi \in H^{1,2}(x^*TM) \mid \xi(0) \in T_{x(0)} L_0,\ \xi(1) \in T_{x(1)} L_1\}\;.
\]
\begin{rmk}
  One can show that the operator~\eqref{eq:HessA} is the Hessian of
  the Hamiltonian action functional with respect to the Levi-Civita
  connection of $\P$ induced from the metric~\eqref{eq:gJL2} and
  whenever $x \in \I_H(L_0,L_1)$ the operator is independent of the
  choice of the connection.
\end{rmk}
The eigenvalues and eigenfunctions of $A_x$ play an important role for
the study of the asymptotic behavior of solutions
of~\eqref{eq:Floertraj}.  We have the following result due to
Frauenfelder.
\begin{prp}[{\cite[Theorem 4.1]{Frauenfelder:PhD}}]\label{prp:Hess}
  For any $x \in \P(L_0,L_1)$ the operator $A_x$ is self-adjoint with
  respect to the inner product~\eqref{eq:gJL2} and has a closed
  range. The spectrum $\sigma(A_x) \subset \R$ is discrete and
  consists purely of eigenvalues.
\end{prp}
We prove in Section~\ref{sec:aa} that the gap in the spectrum around
zero of the Hessian $A_x$ controls the decay rate of finite energy
$(J,H)$-holomorphic strips. Given $x \in \I_H(L_0,L_1)$ we define
\begin{equation}
  \label{eq:iotap}
  \iota_x(J,H) :=\inf\left\{ \nm{\alpha}\ \left| \ 0\neq \alpha \in \sigma(A_x)\right\}\right.\;, 
\end{equation}
and moreover for any subset $C \subset \I_H(L_0,L_1)$ we define
\begin{equation}
  \label{eq:iotaC}
  \iota(J,H):=\inf_{x\in \I_H(L_0,L_1)} \iota_x(J,H),\qquad \iota(C;J,H):=\inf_{x \in C} \iota_x(J,H)\,.
\end{equation}
If $H \equiv 0$, then we abbreviate $\iota(x;J,0)$, $\iota(C;J,0)$ and
$\iota(J,0)$ by $\iota(x;J)$, $\iota(C;J)$ and $\iota(J)$
respectively.
\begin{rmk}
  Whenever $H\equiv 0$, $J_t=J_0$ for all $t \in [0,1]$ and $\dim M=2$
  there is a geometric interpretation of the spectrum of $A_x$ as
  angle at the intersection point $x=p \in L_0 \cap L_1$. More precisely
  if $M=\C$, $\omega=\ostd$, $J_t = \Jstd$ for all $t \in [0,1]$, $L_0
  =\R$ and $L_1 = e^{i\alpha}\R$, then the spectrum is given
  by \[\sigma(0;\Jstd,0) = \alpha+\pi \Z\;,\] and
  $\iota:=\iota(0;\Jstd)$ is the unique constant such that $\iota \in
  (0,\pi/2]$ and $\iota=\nm{\alpha + \pi k}$ with $k \in
  \Z$. Geometrically it corresponds to the acute angle of the
  intersection $L_0$ with $L_1$.
\end{rmk}
\begin{lmm}\label{lmm:Hessconj}
  For all $x \in \I_H(L_0,L_1)$ we have
  $\iota_x(J,H)=\iota_p(\vp_H^*J)$ with $p=x(0)\in L_0 \cap
  \vp_H^{-1}(L_1)$ and $(\vp_H^*J)_t = \d \vp^t_H \circ J_t \circ
  \left(\d \vp^t_H\right)^{-1}$.
\end{lmm}
\begin{proof} Abbreviate $J'_t= \left(\vp_H^*J\right)_t$ and $L_1'
  =\vp_H^{-1}L_1$. Consider the operator
  \[A_p:T_p\P(L_0,L_1')\to L^2([0,1],T_p M),\qquad \xi \mapsto J' \pt
  \xi\,.\] It suffices to show that the operators $A_p$ and $A_x$ are
  conjugated by isomorphisms
  \[
  T_p\P(L_0,L_1') \to T_x \P(L_0,L_1),\qquad L^2([0,1],T_pM) \to
  L^2(x^*TM)\,,
  \]
  both given by $\xi \mapsto (t\mapsto \d \vp_H^t \xi(t))$. For that
  it suffices to show that for all smooth $\xi:[0,1] \to T_p M$ we
  have
  \begin{equation*}
     J \left(\na_t \d \vp_H \xi - \na_{\d \vp_H \xi} X_H\right) = \d
  \vp_H J' \pt \xi\;. 
  \end{equation*}
  Suppose that $\pt \xi =0$ for a moment, then since $\na$ is torsion
  free and $\pt x = X_H$ the equation holds after $\na_t \d \vp_H \xi
  = \na_{X_H}\d \vp_H \xi= \na_{\d \vp_H \xi} X_H$. In general any
  other $\xi$ is given as $\xi = \sum_j f_j \xi_j$ with $\xi_j$
  constant and $f_j \in H^{1,2}([0,1],\R)$. We compute
  \begin{align*}
    J \left(\na_t \d \vp_H \xi - \na_{\d \vp_H \xi} X_H\right) &=
    \sum_jJ\big(\na_t f_j\d \vp_H \xi_j - f_j \na_{\d \vp_H \xi_j}
    X_H\big)\\
    &=\sum_j J \d\vp_H  (\pt f_j) \xi_j + f_j \left(\na_t \d \vp_H \xi_j -
      \na_{\d \vp_H \xi_j} X_H\right)\\
    &=\sum_j \d\vp_H J' \pt (f_j \xi_j) = \d \vp_H J' \pt \xi\,.
  \end{align*}
  Thus the last equation holds for all $\xi$. We conclude that $A_x$
  and $A_p$ are conjugated.
\end{proof}

\subsubsection{Clean intersections} Two submanifolds $L_0,L_1 \subset M$
\emph{intersect cleanly along a submanifold $C \subset M$}, if $C
\subset L_0 \cap L_1$ and for all $p$ in $C$ we have
\[ T_pC = T_p L_0 \cap T_p L_1\;.\] Moreover $L_0,L_1$ are \emph{in
  clean intersection} if they intersect cleanly along $L_0 \cap
L_1$. Every transverse intersection is also clean but certainly the
converse is not true. Pozniak~\cite{Pozniak} gave a normal form for
Lagrangian submanifolds in clean intersection. Let $C \subset L$ be a
submanifold of a manifold $L$. The \emph{conormal bundle $T C^\omega
  \subset T^*L$ of $C$} is defined by
\[TC^\omega = \left\{ (q,p) \in T^*L \ | \ q \in C,\ p(v) = 0 \quad
  \forall\ v \in T_q C\right\}\;.\] Note that $TC^\omega \subset
(T^*L,\omega_\can)$ is an exact Lagrangian submanifold, which
intersects the zero section cleanly along $C$.
\begin{prp}[{\cite[Proposition 3.4.1]{Pozniak}}]\label{prp:poz}
  Let $(M,\omega)$ be a symplectic manifold and $L_0,L_1 \subset M$ be
  two Lagrangian submanifolds intersecting cleanly along a compact
  submanifold $C \subset M$, then there exists a vector bundle $E\to
  C$, open sets $V \subset T^*E$, $U\subset M$ and a diffeomorphism
  $\vp:U \to V$ such that such that $C \subset U$, $\vp^*\ostd =
  \omega$ and
    \[\vp\left(L_0 \cap U_\Poz\right) = E \cap V,\qquad
    \vp\left(L_1 \cap U_\Poz\right) = TC^\omega \cap V\,,\] in which
    $E$ and $C$ are identified with their image under the zero section
    in the bundles $T^*E\to E$ and $E\to C$ respectively.
\end{prp}
\begin{lmm}\label{lmm:chart}  
  With the same assumption as Proposition~\ref{prp:poz}. For all $p
  \in C$ there exists open sets $U\subset M$, $V \subset \R^{2n}$ and
  a diffeomorphism $\vp:U \to V$ such that $p \in U$, $\vp(p) =0$,
  $\vp^*\ostd =\omega$ and
  \[
  \vp\big(L_0 \cap U\big) = \Lambda_0 \cap V, \qquad
  \vp\big(L_1 \cap U\big) = \Lambda_1 \cap V\,,
  \]
  in which $\Lambda_0,\Lambda_1 \subset \R^{2n}$ are linear subspaces
  which are Lagrangian with respect to the standard symplectic form
  $\ostd$.
\end{lmm}
\begin{proof} 
  According to Proposition~\ref{prp:poz} we assume without loss of
  generality that $M= T^*L_0$, $\omega =\omega_\can$ and $L_1
  =TC^\omega$ for some submanifold $C \subset L_0$ of dimension
  $k$. There exists local coordinates $\psi:V \stackrel{\cong}{\to} W$
  with $W \subset L_0$ and $V \subset \R^n$ is an open ball such that
  $\psi(V \cap \R^k) = W \cap C$. We define $U:=T^*W$ and $\vp
  =\psi^*:U \to T^*V$.
\end{proof}
\begin{lmm}~\label{lmm:trivi} With the same assumption as
  Proposition~\ref{prp:poz}. Let $J : [0,1] \to \End(TM,\omega)$ be a
  path of compatible almost complex structures. For all $p \in L_0
  \cap L_1$ there exists an open neighborhood $U \subset M$ and a
  local trivialization
  \[ \Phi:[0,1] \times U \times \R^{2n} \to \res{TM}{U},\quad (t,q,v)
  \mapsto \Phi_t(q)v \in T_qM\;,\] such that we have
  \begin{itemize}
  \item $J_t(q) \Phi_t(q) = \Phi_t(q) \Jstd$ for all $t \in [0,1]$ and
    $q \in U$,
  \item $\omega_q(\Phi_t(q)\xi,\Phi_t(q)\xi')=\ostd(\xi,\xi')$ for all
    $t \in [0,1]$, $q \in U$ and $\xi,\xi' \in T_q M$
  \item $T_q L_k = \Phi_k(q)\left( \R^n \oplus \{0\}\right)$ for all
    $q \in L_k \cap U$ and $k =0,1$.
  \end{itemize}  
  where $\Jstd:\R^{2n} \to \R^{2n}$ is the standard complex structure
  with matrix representation given by
  \begin{equation}
    \label{eq:Jstd}
    \Jstd =
    \begin{pmatrix}
      0&-\one\\
      \one&0
    \end{pmatrix}\;.
  \end{equation}
\end{lmm}
\begin{proof}
  Choose local coordinates (\cf Lemma~\ref{lmm:chart}) and assume $U$
  is an open subset of $\R^{2n}$, $\omega=\ostd$,
  $L_0,L_1=\Lambda_0,\Lambda_1$ are linear Lagrangian subspaces and
  the almost complex structure is given by a matrix valued function
  $J: [0,1] \times U \to \R^{2n \times 2n}$. Choose a smooth path of
  linear Lagrangian subspaces $F:[0,1] \to \L(n)$ such that
  $F(0)=\Lambda_0$ and $F(1)=\Lambda_1$. Choose functions
  $e_1,\dots,e_n:[0,1]\times U \to \R^{2n}$ such that
  $(e_1(t,q),\dots,e_n(t,q))$ is a frame of $F_t$ and after
  Gram-Schmidt satisfies $\omega(e_i(t,q),J_t(q)e_j(t,q)) =
  \delta_{ij}$ for all $i,j = 1,\dots,n$, $t \in [0,1]$ and $q \in
  U$. Then the linear map $\Phi_t(q)$ given as matrix with column
  vectors $(e_1,\dots,e_n,J e_1,\dots,J e_n)$ satisfies all required
  properties.
\end{proof}
\subsubsection{Clean and transverse Hamiltonians} Given a Hamiltonian
function $H\in C^\infty([0,1] \times M)$ we denote by $\vp_H:M \to M$
the corresponding Hamiltonian diffeomorphism, \ie the time-one map of
the Hamiltonian flow.
\begin{dfn}\label{dfn:MBreg}
  Given two Lagrangian submanifolds $L_0,L_1 \subset M$. 
  An Hamiltonian $H$ is
  \begin{enumerate}[label=(\roman*)]
  \item \emph{clean for $(L_0,L_1)$}, if $L_0$ and $\vp_H^{-1}(L_1)$ are in
    clean intersection,
  \item \emph{transverse for $(L_0,L_1)$}, if $L_0$ and $\vp_H^{-1}(L_1)$
    are in transverse intersection.
  \end{enumerate}
  If there is no risk of confusion we just write $H$ is \emph{clean} or
  \emph{transverse}.
\end{dfn}
To define Floer homology transverse Hamiltonians are considered and in
that case the action function is Morse, \ie critical points are
non-degenerated. The next lemma shows that, if the Hamiltonian $H$ is
clean, then the action functional $\A_H$ is Morse-Bott.
\begin{lmm}\label{lmm:kerHess}
  Suppose that the Hamiltonian $H$ is clean, then every connected
  component of $\I_H(L_0,L_1)$ is a manifold and for all $x \in
  \I_H(L_0,L_1)$ we have $\ker A_x =T_x\I_H(L_0,L_1)$ as subspaces of
  $T_x\P$.
\end{lmm}
\begin{proof}
  Via $x \mapsto x(0)$ the space $\I_H(L_0,L_1)$ is isomorphic to $L_0
  \cap \vp_H^{-1}(L_1)$ which is component-wise a manifold and
  provides the chart maps. Set $p=x(0)$. Given $\xi_0 \in T_p L_0$,
  consider the vector field $\xi(t):=\d \vp^t_H \xi_0$, which is a
  vector field along $x$. Since $\na$ is torsion free
  and $\pt x = X_H$ we have $ \na_t \xi = \na_{X_H}\xi = \na_{\xi}
  X_H$. We conclude that every element in the kernel of $A_x$ is of
  the form $t \mapsto \xi(t)=\d \vp^t_H\xi_0$ with $\xi_0 \in T_p L_0
  \cap T_p \vp^{-1}_H(L_1)$. If the Hamiltonian is clean then $T_p L_0
  \cap T_p \vp_H^{-1}(L_1) = T_p(L_0 \cap \vp_H^{-1}(L_1))$, which
  under the identification of $\I_H(L_0,L_1)$ with $L_0 \cap
  \vp_H^{-1}(L_1)$ is the tangent space of $\I_H(L_0,L_1)$ at $x$.
\end{proof}
\begin{prp}\label{prp:spectralbound}
  Suppose that the Hamiltonian $H$ is clean, then for any compact
  subset $C \subset \I_H(L_0,L_1)$ we have $\inf_{p \in C}\iota_p(J,H) >0$.
\end{prp}
\begin{proof}
  With loss of generality we assume that $H=0$ and $L_0,L_1$ are in
  clean intersection (\cf Lemmas~\ref{lmm:Hessconj}
  and~\ref{lmm:chart}). Suppose by contradiction that there exists a
  sequence of points $(p_\nu) \subset C$ such that $\lim_{\nu \to
    \infty} \iota_{p_\nu}(J) =0$.  Since $C$ is compact, we assume
  that
  $(p_\nu)$ converges to $p\in C$. Using the trivialization $\Phi$
  from Lemma~\ref{lmm:trivi} we define matrix valued functions
  $\sigma_\nu,\sigma_\infty:[0,1] \to \R^{2n\times 2n}$ by the
  requirement
  \[
  J(p)\pt \Phi(p) \xi = \Phi(p) \left(\Jstd \pt \xi + \sigma_\infty
    \xi\right),\quad J(p_\nu)\pt \Phi(p_\nu)\xi=
  \Phi(p_\nu)\left(\Jstd \pt \xi +\sigma_\nu \xi\right)\;,
  \]
  for all smooth $\xi:[0,1] \to \R^{2n}$. Because $J$ and $\Phi$ are
  smooth there exists a uniform constant $c_1$ such that for all $t
  \in [0,1]$ and $\nu \geq 1$
  \begin{equation*}
    \nm{\sigma_\infty(t)-\sigma_\nu(t)}\leq c_1 \di{p_\nu,p}\;.
  \end{equation*} 
  We define the unbounded operators in the Hilbert space
  $L^2([0,1],\R^{2n})$ via
  \[
  \left(A_\infty \xi\right)(t) = \Jstd \pt \xi(t) +
  \sigma_\nu(t)\xi(t),\qquad \left(A_\nu \xi\right)(t) = \Jstd \pt
  \xi(t) + \sigma_\infty(t)\xi(t)\;,
  \]
  with dense domain $\{\xi \in H^{1,2}([0,1],\R^{2n}) \mid
  \xi(0),\xi(1) \in \R^{n} \times \{0\}\}$.  Being conjugated to the
  Hessians $A_p, A_{p_\nu}$ the operators $A_\infty,A_\nu$ are
  self-adjoint and have a closed range (\cf
  Proposition~\ref{prp:Hess}). The difference $A_\infty-A_\nu$ extends
  to a bounded operator which converges to zero as $\nu$ tends to
  infinity. By Lemma~\ref{lmm:kerHess} the kernels of $A_\infty, A_\nu$
  have the same dimension. Then, by Lemma~\ref{lmm:iotasemicont} there
  exists $\nu_0$ such that for all $\nu \geq \nu_0$ we have
  $\iota_{p_\nu}(J)=\iota(A_\nu) \geq 1/2 \iota(A) >0$ in contradiction
  to $\iota_{p_\nu}(J) \to 0$.
\end{proof}
\subsection{Morse homology}\label{sec:Morse}
Let $C$ be a closed manifold. A \emph{Morse function} $f:C \to \R$ is
a smooth function such that the Hessian at any critical point $p \in
\crit f$ is non-degenerate. Necessarily the set of critical points is
isolated. We choose a Riemannian metric $g$ on $C$ and assume that the
negative gradient flow $\psi:\R \times C \to C$, $\psi^s =
\psi(s,\cdot)$ exists for all times. Define the \emph{unstable}
(resp.\ \emph{stable}) \emph{manifold} of a critical point $p \in
\crit f$ by
\begin{align*}
  W^u(p;f) &:= \{ u \in C \ | \ \psi^s(u) \to p \text{ for } s \to -\infty\}\\
  (\text{resp.\ } W^s(p;f)&:= \{ u \in C \ | \ \psi^s(u) \to p \text{
    for } s \to \infty\} )\;.
\end{align*}
Without risk of confusion we write $W^u(p)$ and $W^s(q)$ to denote
$W^u(p;f)$ and $W^s(q;f)$ respectively. We call the pair $(f,g)$
\emph{Morse-Smale}, if for any two critical points $p,q \in \crit f$
the unstable manifold $W^u(p;f)$ intersects the stable manifold
$W^s(q;f)$ transversely.

If $(f,g)$ is Morse-Smale, then Morse homology is well-defined. We
define the \emph{space of parametrized Morse trajectories} 
\[\wt
\M_0(p,q) = W^u(p) \cap W^s(q)\;.\] The negative gradient flow
preserves $\wt \M_0(p,q)$ and induces an action of $\R$. Whenever
$p\neq q$ the action is free and we denote the quotient by
$\M_0(p,q)$.

\subsubsection{Orientation} We choose an orientation of the unstable
manifolds $W^u(p)$ for each critical point $p \in \crit f$, which is
always possible because $W^u(p)$ is contractible. Once a choice is
made, the stable manifolds $W^s(q)$ are automatically cooriented for
all $q \in \crit f$ and we obtain an orientation of $\wt \M_0(p,q)$
for all pairs of critical points $p,q \in \crit f$ via a canonical
construction (\cf equation~\eqref{eq:capori}). With standard
orientation of $\R$, we also obtain orientations of the quotient
$\M_0(p,q)$, which at elements $[u]$ in the zero dimensional component
is just a number in $\{\pm 1\}$, denoted $\sign u$.

\subsubsection{Morse complex} Let $A$ be any commutative ring with
unit. We define the \emph{Morse chain complex} $C_*(f,A)$ as the free
$A$-module generated by the critical points $\crit f$, graded by
$\nm{p} = \mu_\Mor(p) = \dim W^u(p)$ and equipped with the boundary
operator
\[ \partial:C_*(f;A) \to C_{*-1}(f;A),\qquad p \mapsto \sum_{\mu(q)=\mu(p)-1} \sum_{[u] \in \M_0(p,q)} \sign u \cdot q\;.  \] Note
that if $\nm{p} - \nm{q} =1$ then the sum $\sum_{[u] \in
  \M_0(p,q)}\sign u$ equals the intersection number of $W^u(p)$ with
$W^s(q)$. The next theorem is a classical result. A modern proof is
found in~\cite{AbbMajer:MorseI} or~\cite{Schwarz:Buch}
\begin{thm}\label{thm:HM}
  We have $\partial \circ \partial =0$. The associated homology group
  $H_*(f;A):=\ker \partial/\im \partial$ is independent of the
  function $f$, the metric and the choices of orientations up to
  isomorphism and we have the natural isomorphism
  \begin{equation}
    \label{eq:HM}
    H_*(f;A) \cong H_*(C;A)\;.
  \end{equation}
\end{thm}
\subsubsection{Functoriality} Let $\vp:C \to C'$ be a smooth map between
the manifolds $C$ and $C'$ which are equipped with Morse-Smale pairs
$(f,g)$ and $(f',g')$ respectively. Given two critical points $p \in \crit f$ and $p' \in \crit f'$ define
the space
\begin{equation}
  \label{eq:funcMorse}
  \M^\vp(p,p') := W^u(p) \cap \vp^{-1}\big(W^s(p')\big)\;.  
\end{equation}
Generically the intersection is transverse and hence $\M^\vp(p,p')$ is
a manifold of dimension $\mu(p')-\mu(p)$. If $W^u(p)$ is oriented and
$W^s(p')$ is cooriented, then $\M^\vp(p,p')$ carries an induced
orientation. We define the morphism
\[
C\vp_*: C_*(f;A) \to C_*(f';A),\qquad p \mapsto \sum_{\mu(p')=\mu(p)}
\nolimits \sum_{u \in \M^\vp(p,p')} \nolimits \sign u\cdot p'\;.
\]
The homomorphism $C\vp_*$ is a chain map. We denote the induced map on
homology by $\vp_*:H_*(f;A) \to H_*(f';A)$. In~\cite[\S
2.2]{AbbSchwarz:loop} it is proven that $\vp_*$ is the push-forward in
homology under the identification~\eqref{eq:HM}.
\subsubsection{Cohomology} By definition the cohomology complex $C^*(f;A)$
is given by the module $\Hom(C_*(f;A),A)$ equipped with the boundary
operator $d:C^*(f;A)\to C^{*+1}(f;A)$, $\vp \mapsto (p \mapsto
\vp( \partial p))$. One shows that $d\circ d=0$ and the associated
cohomology is isomorphic to $H^*(C;A)$. For any critical point $p \in
\crit f$ let $\delta_p \in C^*(f;A)$ be the homomorphism that is $1$
on $p$ and $0$ otherwise. Since any element in $C^*(f;A)$ is a linear
combination of these, we see that Morse cohomology is alternatively
defined by the free module generated by critical points $p \in \crit
f$, graded by the Morse index and equipped with differential
\begin{equation}
  \label{eq:dMorse}
  d:C^*(f;A) \to C^{*+1}(f;A),\qquad p\mapsto\sum_{\mu(q)=\mu(p)+1} \sum_{[u] \in
  \M_0(q,p)} \sign u \cdot q\,.
\end{equation}
For more details see~\cite{Schwarz:Buch}.
\subsubsection{Local coefficients} 
Let $\CC$ be the the category of points in $C$ with morphisms given by
homotopy classes of paths and composition law by concatenation of
paths. A \emph{local system $\L$} is a functor from $\CC$ into the
category of $A$-modules. We conclude that an isomorphism class of a
local system is given by a representation of $\pi_1(C)$ on an
$A$-module. On the other hand, starting with a representation
$\rho:\pi_1(C) \to \Aut(\Lambda)$, where $\Lambda$ is an $A$-module,
we obtain a local system as follows. Consider the universal covering
$\tilde C \to C$ and the fibre product $\tilde C\times_{\rho} \Lambda$
which is the quotient of the diagonal action of $\pi_1(C)$ on $\tilde
C \times \Lambda$ given by Deck transformations and $\rho$
respectively. Equipp the bundle $\tilde C \times_\rho \Lambda$ with a
topology by using a discrete topology on $\Lambda$. Then the local
system of $\rho$ associates to a point the fibre of $\tilde C
\times_{\pi_1(C)} \Lambda \to C$ over $p$ and to a homotopy class of
paths we associate the parallel transport. For a local system $\L$
we define \emph{Morse homology with coefficients in $\L$} as the
$A$-module
\[ C_*(f;\L)=\bigoplus_{p \in \crit f} \L(p)\,,\]
graded by the Morse index and equipped with the boundary operator 
\[
\partial:C_*(f;\L) \to C_{*-1}(f;\L),\qquad \L(p) \ni a \mapsto
\sum_{\mu(q)=\mu(p)-1}\sum_{[u] \in \M_0(p,q)} \sign u \cdot \L(u)a\,.
\]
Alternatively define the complex as follows. We choose an isomorphism
$\L(p) \cong \Lambda$ for all $p \in \crit f$. Then
$C_*(f;\L)=C_*(f)\otimes \Lambda$ and the boundary operator is given
by the same formula where now $\L(u)$ is an automorphism of the module
$\Lambda$ given by parallel transport.  In \cite[\S 7.2]{Oancea:Leray}
it is shown (with the minor difference that the argument there is for
cohomology) that $\partial \circ \partial =0$ and the associated
homology is isomorphic to $H_*(A;\L)$, which is the homology of $C$
with values in the local system $\L$. For more details see
also~\cite[Appendix A]{AlbertoSchwarz:Legendre}.
\subsection{Floer homology}\label{sec:Floer}
Fix a symplectic manifold $(M,\omega)$, Lagrangians $L_0,L_1 \subset
M$ such that~\eqref{eq:Assumption} holds and fix coefficient ring $A$
which is eigther $\Z_2$ or $\Z$. We give a short introduction to Floer
homology of the pair $(L_0,L_1)$ with coefficients in the Novikov ring
$\Lambda = A[\lambda,\lambda^{-1}]$.
\subsubsection{Floer trajectories} Choose a Hamiltonian $H \in
C^\infty([0,1]\times M)$ and a path of almost complex structures
$J:[0,1] \to \End(TM,\omega)$, $J_t=J(t,\cdot)$. For two Hamiltonian
arcs $x_-,x_+ \in \I_H(L_0,L_1)$ we define the space of
\emph{parametrized finite energy Floer trajectories}
\[
\wt \M(x_-, x_+;J,H) = \{ u \in C^{\infty}(\R \times [0,1],M) \mid
\eqref{eq:Floertraj} \}\,.\] The Hamiltonian function $H$ is
\emph{transverse} if $\vp_H(L_0)$ intersects $L_1$ transversely, where
$\vp_H$ denotes the Hamiltonian diffeomorphism associated to
$H$. In~\cite{Floer:Trans} it is shown that being transverse is a
generic condition, \ie can always be fulfilled after an arbitrary
small perturbation of $H$.  Moreover for an transverse Hamiltonian
function $H$ it is shown in~\cite{Floer:Trans}, that for a generic
almost complex structure $J$ each connected component of the space
$\wt \M(x_-,x_+;J,H)$ is a manifold and the dimension of a component
containing $u$ is given by the Viterbo index $\mu(u)$. Let us fix
generic data $J$ and $H$. We abbreviate $\wt \M(x,y):=\wt \M(x,y;J,H)$
for any two arcs $x,y \in \I_H(L_0,L_1)$. There exists an $\R$-action
on the space $\wt \M(x,y)$ by translation on the domain, \ie
$(a.u)(s,t)=u(s-a,t)$. If $x \neq y$ the action is free and we denote
the quotient by $\M(x,y)$.
\subsubsection{Grading} Let $N \in \N$ denote the minimal Maslov number of
the pair $(L_0,L_1)$ and let $\P$ denote the space of paths
$\gamma:[0,1] \to M$ with $\gamma(0) \in L_0$ and $\gamma(1)\in
L_1$. For every $x \in \I_H(L_0,L_1)$ which is in the same connected
component of $\P$ as $\base$ we choose $u_x:[-1,1]\times [0,1] \to M$
such that $u(s) \in \P$ for all $s \in [-1,1]$, $u_x(-1)=\base$ and
$u_x(1)=x$. Then define the grading as the Viterbo index of $u_x$, \ie
\begin{equation*}
  \nm{x}:=-\mu(u_x).
\end{equation*}
\subsubsection{Orientation} If the characteristic of $A$ is not two,
we need to orient the spaces $\M(x,y)$. This is done as follows.  Let
$D_{u_x}$ be the linearized Cauchy-Riemann-Floer operator of the cap
$u_x$ extended constantly which by Theorem~\ref{thm:DuFred} is a
Fredholm operator. There is a natural notion of an orientation of a
Fredholm operator and we denote by $|D_{u_x}|$ the space of
orientations of $D_{u_x}$.  For each $x \in \I_H(L_0,L_1)$ connected
to $\base$ within $\P$ chose a cap $u_x$ and an orientation $o_x \in
|D_{u_x}|$. Given $u \in \wt \M(x,y)$ there exists an orientation
gluing operation which lifts the linear gluing map $|D_{u_x}| \otimes
|D_u| \cong |D_{u_x\#u}|$ (\cf Lemma~\ref{lmm:OrDu}). Provided that
the pair $(L_0,L_1)$ is equipped with a relative spin structure
$|D_{u_x\#u}|$ and $|D_{u_y}|$ are naturally isomorphic. Hence by the
orientation gluing map and our choices, we obtain an orientation of
$D_u$ which induces an orientation $o_u$ of $\wt \M(x,y)$ (\cf
Theorem~\ref{thm:ori}). By associativity of the orientation gluing
operation (\cf Lemma~\ref{lmm:associative}), the constructed
orientations satisfy $o_u\#o_v=o_{u\#v}$ for all $(u,v) \in \wt M(x,z)
\times \wt \M(z,y)$. \ie are coherent. With standard orientation on
$\R$ we also obtain an orientation of the quotient space $\M(x,y)$ for
all $x,y \in \I_H(L_0,L_1)_{[\base]}$ (\cf
equation~\eqref{eq:oriquotient}). Let $\M(x,y)_{[0]}$ be the union of
all zero-dimensional components. An orientation of an element
$[u]\in\M(x,y)_{[0]}$ is a number in $\{\pm 1\}$, which we denote by
$\sign u$.
\subsubsection{Floer complex} Let $N \in \N$ be the minimal Maslov number
of the pair $(L_0,L_1)$. Denote by $\Lambda :=
A[\lambda,\lambda^{-1}]$ the ring of Laurent polynomials in the
variable $\lambda$ of degree given by $-N$. Let
$\I_H(L_0,L_1)_{[\base]}\subset \I_H(L_0,L_1)$ denoted the subset of
elements which are connected to $\base$ within $\P$. The \emph{Floer
  complex} is given by the free $\Lambda$-module generated by
$\I_H(L_0,L_1)_{[\base]}$, graded by $\nm{x\otimes \lambda^k} =
-\mu(u_x)-kN$ and equipped with the $\Lambda$-linear operator
\begin{gather*}
  \partial:CF_*(L_0,L_1)\to CF_{*-1}(L_0,L_1), \\
  x \mapsto \sum_{y \in \I_H(L_0,L_1)_{[\base]}} \sum_{[u] \in
    \M(x,y)_{[0]}} \sign u\cdot y \otimes
  \lambda^{(\nm{y}-\nm{x}+1)/N}\;.
\end{gather*}
That $\partial$ is a boundary operator is a highly non-trivial fact
and the central result of Floer's papers~\cite{Floer:Intersection,
  Floer:Action, Floer:Witten} with details of the monotone case worked
out by Oh~\cite{Oh:diskI}. The supplement with orientations is in the
books of Fukaya \ea~\cite{FO3:I, FO3:II} and their
paper~\cite{FO3:integers}. 
\begin{thm}[Floer]\label{thm:HF}
  We have $\partial\circ \partial =0$. The associated
  homology group 
  \begin{equation}
    \label{eq:HF}
    HF_*(L_0,L_1) = \ker \partial/\im \partial\,,
  \end{equation}
  is independent of choices of $J$, $H$ and orientations up to
  isomorphism. 
\end{thm}

\section{Asymptotic analysis}\label{sec:aa}

We study the asymptotic behavior holomorphic strips with boundary
and finite energy on Lagrangian submanifolds. More precisely we show
that if the Lagrangians intersect cleanly such strips decay
exponentially and approach an eigenfunction of the asymptotic operator
up to an error of higher exponential decay. If the Lagrangians
intersect transversely this was proven by Robbin and Salamon
in~\cite{Robbin:Strip}. The generalization to holomorphic strips with
boundary on cleanly intersecting Lagrangians was mainly done by
Frauenfelder in~\cite{Frauenfelder:PhD}. The only part which we have
not found in the literature is the fact that the decay parameter has
an upper bound by the spectral gap of the asymptotic operator and the
above mentioned convergence to the eigenfunction. These improvements
however are necessary to embed the space of holomorphic curves in a
suitable Banach manifold.
\subsection{Main statement}
Given a compact symplectic manifold $(M,\omega)$, two Lagrangian
submanifolds $L_0,L_1 \subset M$. Fix a smooth path of almost complex
structure $J:[0,1] \to \End(TM,\omega)$. In this section we study the
asymptotic behavior of smooth maps $u:[0,\infty) \times [0,1] \to M$
satisfying the Cauchy-Riemann equation
\begin{equation}
  \label{eq:CR}\tag{CR}
    \partial_s u(s,t) + J_t(u) \pt u(s,t) = 0\;,
\end{equation}
and the boundary condition
\begin{equation}
  \label{eq:BC}\tag{BC}
  u|_{t=0} \subset L_0,\qquad \qquad u|_{t=1} \subset L_1\;.  
\end{equation}
For each point $p\in L_0 \cap L_1$ we consider the linear differential
operator (\cf equation~\eqref{eq:HessA})
\begin{gather*}
  A_p: T_p \P(L_0,L_1) \to L^2([0,1],T_pM),\qquad \xi \mapsto J_t(p)\pt \xi\,.
\end{gather*}
In~\cite[Theorem 4.1]{Frauenfelder:PhD} it is shown that $A_p$ is an
operator with discrete spectrum consisting only of eigenvalues. We
define the \emph{spectral gap at $p$}
\begin{equation}
  \label{eq:specgap}
  \iota_p := \min \{\nm{\alpha} \mid \alpha \in
\sigma(A_p)\setminus \{0\}\}\,.
\end{equation}
Up to the quantitative estimate on the decay parameter the next theorem
is proven in~\cite[Thm.\ 3.16]{Frauenfelder:PhD}.
\begin{thm}[exponential decay]\label{thm:remove}
  Assume that $L_0$ and $L_1$ intersect cleanly. Given a map
  $u:[0,\infty)\times [0,1]\to M$ which satisfies \eqref{eq:CR} and
  \eqref{eq:BC}. Then the following three statements are equivalent.
  \begin{enumerate}[label=(\roman*)]
  \item  We have that
    \begin{equation}
      \label{eq:energy}\tag{E}
      E(u) =\int_0^\infty \int_0^1 |\ps u|^2 \d t \d
      s < \infty\;.
    \end{equation}
  \item There exists a point $p\in L_0 \cap L_1$ such that
    \begin{equation}
      \label{eq:limit}
      \lim_{s\to\infty} u(s,t) = p, \qquad \qquad \lim_{s\to \infty}
      \left|\partial_s u(s,t)\right| = 0\;,
    \end{equation}
    where the limits exist uniformly for all $t \in [0,1]$.
  \item For any positive constant $\mu < \iota_p$ with $\iota_p$ given
    in~\eqref{eq:specgap} and integer $k \in \N_0$, there exists a
    constant $c_k=c_k(\mu)$ such that
    \begin{equation}
      \label{eq:decay}
      \Nm{\ps u}_{C^k\left([s,\infty)\times [0,1]\right)} \leq c_k
      e^{-\mu s}\,,
    \end{equation}
    for all $s \geq 0$.
  \end{enumerate}
\end{thm}
Let $u$ be a finite energy $J$-holomorphic strip which approaches the
intersection point $p=\lim_{s\to \infty} u(s,t) \in L_0 \cap L_1$. The
next theorem states that in a chart which is centered at $p$ we
have the approximation $u(s,t) \approx e^{-\alpha s} \zeta(t)$ up to
an error of higher exponential decay for some eigenfunction $\zeta$ of
the asymptotic operator $A_p$ and $\alpha>0$ the corresponding
eigenvalue.
\begin{thm}[Convergence to eigenfunction]\label{thm:eigenval}
  Assume $L_0$ and $L_1$ intersect cleanly. Given a non-constant map
  $u:[0,\infty)\times [0,1] \to M$ satisfying~\eqref{eq:CR},
  \eqref{eq:BC} and~\eqref{eq:limit}. Then there exists an non-zero
  eigenvalue $\alpha$ of $A_p$ with corresponding eigenfunction
  $\zeta \in \ker(A_p - \alpha)$ and a constant $s_0$ such that the
  function $w:[s_0,\infty)\times [0,1] \to T_pM$, $(s,t) \mapsto
  w(s,t)$ defined by
  \[u(s,t) = \exp_p\left(e^{-\alpha s} \zeta(t)+w(s,t)\right)\,,\]
  satisfies the following: For any $\mu<\iota_p$ and number $k \in \N$
  there exists a constant $c_k$ such that for all $s \geq s_0$ we have
  \begin{equation*}
    \Nm{w}_{C^k\left([s,\infty)\times [0,1]\right)} \leq c_k
    e^{- (\mu+\alpha) s}\,.
  \end{equation*}
\end{thm}
\begin{cor}\label{cor:nonzero}
  Assume that $u$ is non-constant and
  satisfies~\eqref{eq:CR},~\eqref{eq:BC} and~\eqref{eq:energy}. There
  exists a point $p \in L_0 \cap L_1$, a non-zero eigenvalue $\alpha \in
  \sigma(A_p)$ and constants $c$, $s_0$ such that we have
  \begin{equation*}
     c^{-1} e^{-\alpha s}\leq \nm{\ps u(s,t)} \leq c e^{-\alpha
    s}\;,
  \end{equation*}
  for all $s \geq s_0$ and $t\in [0,1]$. In particular $\ps u(s,t)$ is
  not zero for all $s \geq s_0$ and $t \in [0,1]$.
\end{cor}
\begin{proof}
  With a uniform bound on the derivative of the exponential map (\cf
  equation~\eqref{eq:pxuxi}) and Theorem~\ref{thm:eigenval} we have
  constants $c_1$ and $c_2$ such that
  \[
  \nm{\ps u(s,t)} \leq c_1 \nm{e^{-\alpha s}\zeta(t) + \ps w(s,t)}
  \leq c_2 e^{-\alpha s} + c_2 e^{-(\alpha+\mu) s} \leq 2c_2 e^{-\alpha s}\,.
  \]
  To show the second inequality observe that since $\zeta$ solves a
  linear first order ordinary differential equation we have
  $c_3:=\inf_{t\in[0,1]} \nm{\zeta(t)}> 0$. Hence with uniform bounds
  on the derivative of the exponential map (\cf
  equation~\eqref{eq:nxxi}) and Theorem~\ref{thm:eigenval}, there
  exists a constant $c_4$ such that
  \begin{multline*}
  c_3e^{-\alpha s} \leq \nm{e^{-\alpha s} \zeta(t)} \leq \nm{\ps(
    -\alpha e^{-\alpha s}\zeta(t) + w(s,t))} + \nm{\ps w(s,t)} \\
  \leq
  c_4\nm{\ps u(s,t)} + c_4 e^{-\mu s}e^{-\alpha s}\,.    
  \end{multline*}
  Since $\mu >0$ we have for $s_0$ sufficiently large that $c_4e^{-\mu
    s} \leq c_3/2$ for all $s \geq s_0$. This shows the second
  estimate by subtracting $c_3e^{-\alpha s}/2$ and dividing by $c_4$
  in the last inequality.
\end{proof}
We prepare the necessary material for the proofs. The proofs
themselves are deferred to the end of the chapter. The proof of the
Theorem~\ref{thm:eigenval} will closely follow \cite[Thm.\
B]{Robbin:Strip} once provided with the adaptation of certain lemmas;
in particular~\cite[Lmm.\ 3.6]{Robbin:Strip}. For the proof of
Theorem~\ref{thm:remove} we differ from the proof given
in~\cite{Robbin:Strip} and make use of the isoperimetric inequality
for arcs between cleanly intersecting Lagrangians (see
Proposition~\ref{prp:iso}). The idea stems
from~\cite{Frauenfelder:PhD} and uses the special nature of the
symplectic action functional. It is a short-cut of the argument. We
just want to state that it is not necessary and one can prove
Theorem~\ref{thm:remove} without the isoperimetric inequality,
sticking with the methods of~\cite{Robbin:Strip}.
\subsection{Mean-value inequality}
The mean-value inequality is a bound of the gradient of
$J$-holo\-morphic curves by their energy provided that the energy is
sufficiently small. The fact is well-known for $J$-holomorphic curves
for which $J$ does do not explicitly depend on the domain (\cf
\cite[Sec.\ 4]{Bibel}). The generalization for almost complex
structures which do depend on the domain was done
in~\cite{Frauenfelder:PhD} with the minor restriction that the
argument was for Lagrangians which are the fixed point set of
anti-symplectic involutions. However if we slightly change the
assumptions the proof is easily adapted for the general case. In the
following we identify the strip $\R \times [0,1]$ with $\Sigma :=
\{z=s+it \in \C \mid t\in [0,1]\}$.

\begin{prp}[Mean value inequality]\label{prp:meanvalue}
  There exists constants $\hbar$ and $c$ such that for any $r<1/2$,
  $z_0:=(s_0,t_0) \in \Sigma$ and map $u:B_r(z_0) \to M$ which
  satisfies~\eqref{eq:CR} and~\eqref{eq:BC} we have
  \[ \int_{B_r(z_0)} |\ps u(s,t)|^2 \d s\d t < \hbar \quad
  \Longrightarrow\quad |\d u(s_0,t_0)|^2 \leq \frac {c}{r^2}
  \int_{B_r(z_0)} |\ps u(s,t)|^2 \d s \d t \,,\] in which
  $B_r(z_0) :=\{z\in \Sigma \mid |z-z_0|<r\}$ denotes the open
  ball of radius $r$ centered at $z_0$.
\end{prp}
  \begin{proof}
    See \cite[Lemma 3.13]{Frauenfelder:PhD}. The present situation is
    slightly different. We bound the radius $r$ by $1/2$ since we need
    to assure that $B_r(s_0,t_0)$ touches at most one of the boundary
    components of $\Sigma$. This is necessary because unlike
    in~\cite{Frauenfelder:PhD} we do not assume symmetries for the
    almost complex structure which was necessary to extend the
    solutions. The rest goes through directly. We give the whole
    argument for completeness. Except for some minor changes the
    computation is the same as in the proof of \cite[Lemma
    4.3.1]{Bibel}.

    Since $M$ is compact any two metrics are equivalent. We assume
    without loss of generality that the metric is given as in
    \cite[Lemma 4.3.3]{Bibel} with respect to $J_0$ and $L_0$ if
    $s_0<1/2$ (\resp $J_1$ and $L_1$ if $s_0 \geq 1/2$). Let $\nabla$
    denote the Levi-Civita connection of that metric. Abbreviate $\xi
    = \ps u$ and $\eta = \pt u$ and define the function
    \[ w:B_r(s_0,t_0) \to \R,\qquad (s,t) \mapsto \frac 12
    \nmm{\xi(s,t)}^2\;.\] Let $\Delta = \ps^2 + \pt^2$ denote the
    Laplace operator.  We want to show that $w$ satisfies the
    inequality
    \begin{equation}
      \label{eq:meanvalueestimate}
      \Delta w \geq - c_1(w +w^2)\,,
    \end{equation}
    for some positive constant $c_1 >0$. We compute
    \[ \Delta w = |\nabla_s \xi|^2 + |\nabla_t \xi|^2 +
    \<\nabla_s\nabla_s\xi + \nabla_t \nabla_t \xi,\xi\>\,.\] We
    abbreviate by $\partial_t J$ the derivative of the path of
    endomorphisms $t \mapsto J_t$ and $\na_\eta J$ the covariant
    derivative of $J_t$ for a fixed $t$ with respect to the vector
    field $\eta$ along the curve $u$. We compute
  \begin{align*}
    \nabla_s\nabla_s\xi + \nabla_t\nabla_t\xi &= \nabla_s(\nabla_s \xi
    + \nabla_t\eta)+ \nabla_t\nabla_s\eta - \nabla_s\nabla_t \eta\\
    &=\na_s\left(\na_s(-J\eta) + \na_t(J \xi)\right) - R(\xi,\eta)\eta\\
    & = \nabla_s\left((\pt J)\xi + (\nabla_\eta J)\xi - (\nabla_\xi
      J)\eta\right) - R(\xi,\eta)\eta\;,
  \end{align*}
  in which we denote by $R(\xi,\eta)\eta := \left(\nabla_s\nabla_t -
    \nabla_t \nabla_s\right)\eta$ and $R$ the curvature tensor. The
  last two equalities combined give
  \[\Delta w = |\nabla_s \xi|^2 + |\nabla_t \xi|^2 -
  \<R(\xi,\eta)\eta,\xi\> + \<\na_s(\pt J) \xi + \na_s (\na_\eta J)\xi
  -\na_s(\na_\xi J)\eta,\xi\>\,.\] Let $\kappa$ denote the last term
  on the right-hand side.  There exists a constant $c_2>1$ depending
  only on the norm of the derivatives of $J$ up to order two such that
\begin{align*}
  \kappa&\geq-c_2 \left(|\xi|^3 + |\xi|\, |\nabla_s\xi| + |\xi|^2
    \left(|\xi|^2 + 2 |\nabla_t \xi|+ 2 |\nabla_s \xi|\right)\right)\\
  &\geq -c_2^2\nm{\xi}^2 - \frac 14 \nm{\xi}^4 -c_2^2 \nm{\xi}^2 -
  \frac 14\nm{\na_s\xi }^2 -c_2 \nm{\xi}^4 -8c_2^2 \nm{\xi}^4-\frac 14( \nm{\na_t \xi}^2  + \nm{\na_s \xi}^2)\\
  &\geq -\frac 12 (|\nabla_s \xi|^2 +  \nm{\nt \xi}^2) -
  10c_2^2(\nm{\xi}^2 + \nm{\xi}^4) \,,
  \end{align*}
  in which for the second estimate we have used the inequality $-ab
  \geq -a^2 - \frac 14 b^2$ for all $a,b \in \mathbb{R}$. Since $M$ is
  assumed to be compact there exists $c_3>0$ depending only on the
  curvature of the metric and the norm of $J$ such that
  \[\<R(\xi,\eta)\eta,\xi\> \geq -c_3 \nm{\xi}^4\;.\] Combining the
  last three estimates we obtain the constant $c_1 >0$ such that
  inequality~\eqref{eq:meanvalueestimate} holds.  Then
  after~\cite[Lemma D.1]{Frauenfelder:PhD} this proves the assertion
  in the case when $B_r(s_0,t_0)$ does not intersect the boundary of
  $\Sigma$. If it does we extend $w$ via $w(s,-t)=w(s,t)$ for $t>0$ if
  $t_0<1/2$ (\resp via $w(s,1+t)=w(s,1-t)$ for $t>1$ if $t_0\geq 1/2$)
  and conclude by the same argument as on~\cite[Page 84]{Bibel}.
\end{proof}
A corollary of the mean-value inequality and bounded gradient
compactness (\cf Lemma~\ref{lmm:bgcomp}) is that the gradient of a
$J$-holomorphic map converges uniformly to zero with all derivatives
as $s$ tends to $\infty$.
\begin{cor}\label{cor:Ckpsu}
  Assume that $u$ satisfies~\eqref{eq:CR}, \eqref{eq:BC}
  and~\eqref{eq:energy}. Then for any $k \in \N$ we have
  \begin{equation*}
    \lim_{s \to \infty} \Nm{\ps u}_{C^k\left([s,\infty)\times [0,1]\right)} = 0\;.
  \end{equation*}
\end{cor}
\begin{proof}
  Suppose by contradiction that we find constants $\e >0$, $k \in \N$
  and a sequence $s_\nu \to \infty$ such that for all $\nu \in N$
  \begin{equation}
    \label{eq:contrasublimit}
  \Nm{\ps u}_{C^k([s_\nu-1,s_\nu+1]\times [0,1])} > \e\;.    
  \end{equation}
  We define $u_\nu:[-2,2]\times [0,1] \to M$ via $u_\nu(s,t) :=
  u(s+s_\nu,t)$. By the mean-value inequality we have that
  \[\sup_{\nu \in \N,(s,t)\in [-1,1]\times [0,1]}\nolimits \nm{\d u_\nu(s,t)} < \infty\;.\]
  By bounded gradient compactness (\cf Lemma~\ref{lmm:bgcomp}) we
  conclude that after possibly passing to a subsequence there exists a
  map $v:[-1,1]\times [0,1] \to M$ such that $(u_\nu)$ converges to
  $v$ uniformly with all derivatives. In particular
  $E(u_\nu;[-1,1]\times [0,1]) \to E(v;[-1,1]\times [0,1])=0$, hence
  $v$ is constant. We conclude that 
\[\Nm{\ps
    u}_{C^k([s_\nu-1,s_\nu+1]\times [0,1])} = \Nm{\ps
    u_\nu}_{C^k([-1,1]\times [0,1])} \to 0\,,\] which
  contradicts~\eqref{eq:contrasublimit}.
\end{proof}
\begin{rmk}
  In the previous corollary we have not used that $L_0$ and $L_1$
  intersect cleanly.
\end{rmk}
\subsection{Isoperimetric inequality}
For paths $\gamma:[0,1]\to M$ with endpoints $\gamma(k) \in L_k$ for
$k=0,1$ and with image sufficiently close to the intersection $L_0 \cap
L_1$ we define the \emph{local action}
\[
\A_\loc(\gamma):=\int \gamma^*\lambda\,,
\]
in which $\lambda$ is any primitive of the symplectic form restricted
to a neighborhood of $L_0 \cap L_1$ such that $\lambda|_{L_k}=0$ (see
Proposition~\ref{prp:poz} to show that such $\lambda$ exist).  That
the inequality which we are about to show is true for \emph{some}
constant $\mu$ is well-known and previously proven in~\cite[Lmm.\
3.17]{Frauenfelder:PhD} or \cite[Lmm.\ 3.4.5]{Pozniak}. 

\begin{prp}[Isoperimetric inequality]\label{prp:iso}
  Assume that $L_0$ and $L_1$ are in clean intersection.  For every
  point $p \in L_0 \cap L_1$ and constant $\mu<\iota_p$ with $\iota_p$
  defined in~\eqref{eq:specgap} there exists a constant $\rho>0$ with
  the following significance: For any smooth curve $\gamma:[0,1] \to
  M$ satisfying $\gamma(0) \in L_0$, $\gamma(1)\in L_1$ and
  $\di{\gamma(t),p} < \rho$ for all $t \in [0,1]$ we have
  \[ 2\mu \nm{\A_\loc(\gamma)} \leq \int_0^1 |\pt \gamma|_J^2\d t\;.\] If moreover $\mu < \inf \{
  \iota_p \mid p \in L_0 \cap L_1\}$ then there exists $\ell_0>0$ such
  that for all $\gamma:[0,1]\to M$ with $\gamma(0) \in L_0$,
  $\gamma(1) \in L_1$ and $\ell(\gamma):=\int_0^1 \nm{\pt \gamma}\d t
  <\ell_0$ the same conclusion holds.
\end{prp}
\begin{proof}
  By Lemma~\ref{lmm:chart} we assume that $\gamma:[0,1]\to \R^{2n}$
  with $\R^{2n}$ equipped with standard symplectic form and the
  Lagrangians $L_0,L_1$ are linear subspaces $\Lambda_0,\Lambda_1$
  respectively. Let $\Phi$ be the trivialization constructed in
  Lemma~\ref{lmm:chart}, which we think of as an matrix valued
  function and abbreviate $\Phi_\gamma(t):=\Phi_t(\gamma(t))$ and
  $J_\gamma(t):=J_t(\gamma(t))$ for all $t \in [0,1]$. The matrix
  $\Phi_\gamma$ is symplectic and satisfies $\Phi_\gamma \Jstd
  =J_\gamma \Phi_\gamma$.  Consider the Hilbert space
  $H=L^2([0,1],\R^{2n})$ equipped with standard inner product
  $\<\cdot,\cdot\>$ and norm $\Nm{\cdot}$. We conclude
  \begin{multline}\label{eq:Egamma}
    \int_0^1 \omega(\pt \gamma,J_\gamma\pt \gamma) \d t = \<\Jstd \pt
    \gamma,J_\gamma\pt \gamma\> =\<\Jstd \pt
    \gamma,\Phi_\gamma\Phi_\gamma^{-1}J_\gamma\pt \gamma\>=\\=
    \Nm{\Phi_\gamma^{-1}J_\gamma\pt \gamma}^2\,.
  \end{multline}
  We extend $\gamma:[0,1]\to \R^{2n}$ to a map $u:[0,1]^2 \to \R^{2n}$
  via $u(s,t)= s \gamma(t)$ and compute
  \begin{equation*}
    \A_{\loc}(\gamma) = \int_{[0,1]}
    \gamma^* \lambda = \int_{[0,1]^2} u^*\omega = \frac 12 \int_0^1
    \omega(\pt \gamma, \gamma) \d t = \frac 12 \<\Jstd \pt \gamma,
    \gamma\>\,,
  \end{equation*}
  in which for the second equality we have used Stokes and the fact
  that by construction $u|_{t=k} \subset \Lambda_k$ for $k=0,1$.
  Abbreviate $\Phi_\infty(t):=\Phi_t(0)$ and $J_\infty(t):=J_t(0)$ for
  all $t \in [0,1]$.  Define the unbounded operator $A_\infty$
  via~\eqref{eq:Ainfdef}. The function $\xi:[0,1] \to \R^{2n}$,
  $t\mapsto \xi(t)=\Phi_\infty(t)^{-1}\gamma(t)$ lies in the domain of
  $A_\infty$.  Continue the computation
  \begin{equation*}
    2\A_\loc(\gamma)=\< \Jstd \pt \gamma,\gamma\>=\<\Jstd \pt \gamma,\Phi_\infty \xi\> = \<\Phi_\infty^{-1} J_\infty \pt \gamma,\xi\>=\<A_\infty\xi,\xi\>\,.
  \end{equation*}
  By construction the operator $A_\infty$ is conjugated to $A_p$.
  In particular these two operators have the same spectral gap.  By
  Corollary~\ref{cor:Aclosed} we have
  \begin{equation}\label{eq:Agamma}
    2 \nm{\A_\loc(\gamma)}=\nm{\<A_\infty\xi,\xi\>} \leq \frac{1}{\iota_p} \Nm{A_\infty\xi}^2=\frac 1 {\iota_p}
    \Nm{\Phi_\infty^{-1}J_\infty \pt \gamma}^2\,.
  \end{equation}
  The matrix $G_t(q):=\Phi_t(q)^{-1}J_t(q)$ is invertible for all
  $(t,q) \in [0,1]\times U$, $\xi \in \R^{2n}$. Moreover we have
  \begin{equation*}
    \Nm{G_t(0) \xi} \leq
    \Nm{G_t(q)\xi} +
    \Nm{G_t(0)-G_t(q)}
  \Nm{G_t(q)^{-1}} \Nm{G_t(q) \xi}\,.    
  \end{equation*}
  Thus there exists a constant $c>0$ such that for all $\rho<1$ and
  curves $\gamma$ with distance to $p$ bounded by $\rho$ we have
  \begin{equation}\label{eq:boundB}
    \Nm{\Phi^{-1}_\infty J_\infty \pt \gamma}^2 =\Nm{G(0) \pt
      \gamma}^2 \leq (1+c\rho) \Nm{G(\gamma) \pt \gamma}^2 =(1+c\rho)
    \Nm{\Phi_\gamma^{-1} J_\gamma \pt \gamma}^2\,.
  \end{equation}
  Together with~\eqref{eq:Egamma} and~\eqref{eq:Agamma} we conclude
  \[
  2\iota_p \nm{\A_\loc(\gamma)} \leq (1+c\rho) \int_0^1 \omega(\pt \gamma,J_\gamma \pt \gamma)\d t\,. 
  \]
  Then the first claim follows if we choose
  $\rho<(\iota_p-\mu)/(c\mu)$.

  We show the second statement. Repeat the argument above with each
  point $p \in L_0 \cap L_1$ and let $c_p$ denote the corresponding
  constant from~\eqref{eq:boundB}.  Since $p \mapsto c_p$ is upper
  semi-continuous and $\iota_p-\mu$ bounded away from zero the
  constant $\rho:=\inf \{\frac{\iota_p -\mu}{c_p\mu} \mid p \in L_0
  \cap L_1\}$ is positive. Abbreviate by $\P$ the space of
  paths $\gamma:[0,1]\to M$ with $\gamma(0) \in L_0$ and $\gamma(1)\in
  L_1$. We denote by $B_\rho(p) \subset M$ the open ball about $p$
  with radius $\rho$. We claim that there exists $\ell_0$ such that
  for any $\gamma \in \P$ we have
  \begin{equation}
    \label{eq:gammain}
    \ell(\gamma):=\int_0^1 |\dot \gamma(t)| \d t<\ell_0 \quad\Rightarrow \quad \gamma(0) \in V:=\bigcup_{p \in L_0 \cap L_1} B_{\rho/2}(p)\,.
  \end{equation}
  If not, there exists sequences $(\gamma_\nu) \subset \P$
  such that for all $\nu \in \N$ we have $\ell(\gamma_\nu)<1/\nu$ and
  $\gamma_\nu(0)$ lies in the complement of $V$. By compactness of
  $L_0$, there exists a sub-sequence, still denoted $(\gamma_\nu)$,
  such that $\gamma_\nu(0)$ converges to a point $p \in L_0$. Moreover
  since $\ell(\gamma_\nu)\to 0$ and $L_1$ is closed we have that
  $\gamma_\nu(1) \in L_1$ converges to $p$ and thus $p \in L_0 \cap
  L_1$, contradicting the fact that $\gamma_\nu(0) \notin V$ for all
  $\nu \in \N$ since $V$ is an open neighborhood of $p$. To show the
  lemma we assume without loss of generality that
  $\ell_0<\rho/2$. Indeed given any $\gamma \in \P$ with
  $\ell(\gamma)<\ell_0$, by~\eqref{eq:gammain} there exists $p
  \in L_0 \cap L_1$ such that $\gamma(0) \in B_{\rho/2}(p)$ and hence
  $\gamma \subset B_\rho(p)$.
\end{proof}
The next lemma is a direct consequence of the isoperimetric
inequality. It is the generalization of a version for $J$-holomorphic
cylinders as given in~\cite[Lemma 4.7.3]{Bibel}. The assertion is that
the energy of a $J$-holomorphic half-strip with boundary in
$(L_0,L_1)$ decays exponentially and the energy of a $J$-holo\-morphic
strip of finite length with boundary in $(L_0,L_1)$ can not spread out
uniformly but must be concentrated at the ends, provided that the
energy is sufficiently small.
\begin{lmm}[Energy decay]\label{lmm:Edecay}
  Assume that $L_0$ and $L_1$ are in clean intersection. For any
  constant $\mu < \inf\{\iota_p \mid p \in L_0 \cap L_1\}$ there
  exists constants $\e_0$ and $c$ with the following significance:
  \begin{enumerate}[label=(\roman*)]
  \item For any map $u:[0,\infty) \times [0,1] \to M$
    satisfying~\eqref{eq:CR},~\eqref{eq:BC} and $E(u)< \e_0$, then
    for all $s \geq 1$ we have
    \begin{equation}
      \label{eq:Edecayinf}
      E(u;[s,\infty)\times [0,1]) \leq E(u) e^{-2 \mu s}\,.
    \end{equation}
    Moreover there exists a point $p\in L_0 \cap L_1$ such that for all
    $s \geq 1$ and $t \in [0,1]$ we have
    \begin{equation}
      \label{eq:duhalfE}
      \di{u(s,t),p}+\nm{\d u(s,t)} \leq c e^{-\mu s}\,.
    \end{equation}
  \item For all $s_0<s_1$ and any map $u:[s_0,s_1]\times [0,1] \to M$
    satisfying~\eqref{eq:CR},~\eqref{eq:BC} and $E(u)<\e_0$ we have
    \begin{equation}
      E(u;[a+s,b-s]\times [0,1]) \leq E(u) e^{-2\mu s}\label{eq:Edecayab}\,.
    \end{equation}
    for all $1\leq s \leq (s_1-s_0)/2$.  Moreover for all
    $\sigma,\sigma' \in [s_0+s,s_1-s]$ and $t,t' \in [0,1]$ we have
    \begin{equation}
      \label{eq:ddecay}
      \nm{\d u(\sigma,t)}+
      \di{u(\sigma,t),u(\sigma',t')} \leq c e^{- \mu s}\,.
    \end{equation}
  \end{enumerate}
  If instead $p \in L_0 \cap L_1$ is a point and $\mu <\iota_p$, then
  there exists constants $\e$, $c$ and $\rho$ satisfying the
  statements above after replacing the manifold $M$ with the open ball
  $B_\rho(p)$.
\end{lmm}
\begin{proof} 
  We show the first statement. Assume that $\varepsilon_0$ is smaller
  than the constant $\hbar$ from the mean-value inequality (\cf
  Prop.~\ref{prp:meanvalue}). Let $u:[0,\infty)\times [0,1] \to M$
  satisfy~\eqref{eq:CR},~\eqref{eq:BC} and $E(u)<\e_0$. That assured
  the mean-value inequality provides a constant $c_1$ independent of
  $u$, such that for any $s > 1/2$ we have
  \begin{equation}
    \label{eq:umeanval}
    \nm{\d u(s,t)}^2 \leq c_1 E(u;[s-1/2,s+1/2]\times[0,1]) \leq c_1 \varepsilon_0\,.
  \end{equation}
  Abbreviate $\gamma_s(t) = u(s,t)$. Let $\ell_0$ denote the constant
  from the isoperimetric inequality (\cf Prop.~\ref{prp:iso}). By
  possibly decreasing $\e_0$ we assume that for all $s>1/2$ we have
  \begin{equation}
    \label{eq:lengthsmall}
  \ell(\gamma_s)=\int_0^1 \nm{\pt \gamma_s(t)} \d t \leq \sqrt{c_1\e_0} <
  \ell_0\,. 
  \end{equation}
  By the choice of $\ell_0$ the point $u(s,t)$ lies inside the Pozniak
  neighborhood where the symplectic form is exact for all $s \geq 1$
  and $t\in[0,1]$. Hence $\omega=\d \lambda$ for some one form
  $\lambda$ and by the isoperimetric inequality we get
  \begin{align*}
    f(s) :=E(u;[s,\infty)\times [0,1]) &= \int_0^1 \gamma_s^*\lambda -
    \lim_{b\to \infty} \int_0^1 \gamma_b^*\lambda\\
    &\leq \frac{1}{2\mu} \int_0^1 \nm{\dot \gamma_s(t)}^2\d t +
    \frac{1}{2\mu}\lim_{b\to \infty} \int_0^1\nm{\dot \gamma_b(t)}^2 \d t\\
    &\leq -\frac{\ps f(s)}{2\mu} + \frac {1}{2\mu}\lim_{b \to \infty}
    E(u;[b-1,b+1]\times [0,1])\\ & = -\frac{\ps f(s)}{2\mu}\;.
  \end{align*}
  Hence $2\mu f(s) + \ps f(s) \leq 0$ which
  gives~\eqref{eq:Edecayinf}. 

  We show~\eqref{eq:duhalfE}. Since $L_0$ is compact there exists a
  sequence $s_\nu \to \infty$ and a point $p \in L_0$ such that
  $p_\nu:=u(s_\nu,0) \to p$. Given any $s$ we find $\nu_0$ such that
  $s_\nu >s$ for all $\nu \geq \nu_0$ and by the exponential decay of
  the energy and the mean-value inequality we have
  \begin{align*}
    \di{u(s,t),p}&\leq \di{u(s,t),u(s_\nu,0)} + \di{p_\nu,p} \\
    &\leq \int_s^{s_\nu} \nm{\partial_\sigma u(\sigma,t)}\d \sigma +
    \int_0^t \nm{\partial_\tau u(s_\nu,\tau)} \d \tau + \di{p_\nu,p}\\
    &\leq \sqrt{c_1\e_0} \int_s^\infty e^{-\mu \sigma} \d \sigma +
    \sqrt{c_1 \e_0} e^{-\mu
      s_\nu} + \di{p_\nu,p}\\
    &\leq \sqrt{c_1 \e_0} e^{-\mu s} + \di{p_\nu,p} \to
    \sqrt{c_1\e_0}e^{-\mu s}\,.
  \end{align*}
  To see that $p \in L_0 \cap L_1$, we consider $p_\nu':=u(s_\nu,1)
  \in L_1$ for all $\nu \in \N$. By the previous estimate we have $p_\nu'
  \to p$. Since $L_1$ is closed we conclude $p\in L_0 \cap L_1$. The
  last estimate together with~\eqref{eq:umeanval}
  and~\eqref{eq:Edecayinf} shows~\eqref{eq:duhalfE}.

  We show~\eqref{eq:Edecayab}. Let $u:[s_0,s_1]\times [0,1]\to M$ be a
  map that satisfies~\eqref{eq:CR}, \eqref{eq:BC} and $E(u)<\e_0$. The
  equations~\eqref{eq:umeanval} and~\eqref{eq:lengthsmall} still
  hold. In particular $u(s,t)$ lies inside the Pozniak neighborhood
  for all $(s,t) \in [s_0+1/2,s_1-1/2]\times [0,1]$ and we have
  \begin{multline*}
    f(s) :=E(u;[s_0+s,s_1-s]\times [0,1]) = \int_0^1
    \gamma_{s_0+s}^*\lambda
    - \int_0^1 \gamma_{s_1-s}^* \lambda \\
    \leq \frac 1 {2\mu} \int_0^1 \nm{\dot \gamma_{s_0+s}}^2 \d t + \frac
    1{2\mu} \int_0^1 \nm{\dot \gamma_{s_1-s}}^2 \d t = - \frac 1{2\mu}
    \ps f(s)\;,
  \end{multline*}
  for all $1/2 \leq s \leq (s_1-s_0)/2$. Hence $\ps f(s) + 2\mu f(s)
  \leq 0$ which implies~\eqref{eq:Edecayab}.

  To show~\eqref{eq:ddecay} we assume without loss of generality that
  $s_0=-s_1$, after possibly replacing $u$ with the shifted map
  $\tilde u$ given by $\tilde u(s,t) = u(s-(s_1+s_0)/2,t)$. By the
  mean value inequality and the energy decay we have for $0 \leq
  \sigma \leq s_1-1$
  \begin{align*}
    \nm{\d u(\sigma,t)}^2 &\leq c_1 E(u;[\sigma-1/2,\sigma+1/2]\times [0,1]) \\
    &\leq c_1E(u;[-\sigma-1/2,\sigma+1/2]\times [0,1]) \\
    &\leq c_1e^{2\mu}\e_0e^{-2\mu(s_1-\sigma)}\;,
  \end{align*}
  where in the last estimate we used~\eqref{eq:Edecayab} with $s =
  s_1-\sigma-1/2$. Note that because $\sigma \leq s_1-1$ we have $s
  \geq 1/2$ as required. Fix some $s \in [1,s_1]$, $\sigma_0 \in
  [0,s_1-s]$ and $t_0 \in [0,1]$. We compute with
  $c_2=\sqrt{c_1\e_0}e^{\mu} $
  \begin{align*}
    \di{u(\sigma_0,t_0),u(0,0)} &\leq \int_0^{\sigma_0} \nm{\ps
      u(\sigma,0)}\d \sigma + \int_0^{t_0} \nm{\pt u(\sigma_0,t)}\d t\\
    &\leq c_2 \int_0^{\sigma_0} e^{-\mu(s_1-\sigma)} \d
      \sigma + c_2\int_0^{t_0} e^{-\mu(s_1-\sigma_0)}\d t \\
    &\leq c_2 (\mu^{-1} + 1) e^{-\mu (s_1-\sigma_0)}\leq c_2(\mu^{-1}+1) e^{-\mu s}\,.
  \end{align*}
  We conclude the same estimate for every $\sigma_1 \in [-s_0+s,0]$
  and $t_1 \in [0,1]$. Hence
  \[\di{u(\sigma_0,t_0),u(\sigma_1,t_1)} \leq 2c_2 (1/\mu +1) e^{-\mu s}\;.\]
  This shows~\eqref{eq:ddecay}.
\end{proof}
\subsection{Linear theory}\label{sec:linaa}
This section is mainly an exposition of the results
from~\cite{Robbin:Strip}. We have included it to introduce the
necessary notations. Following the ideas of~\cite{Robbin:Strip} we
reformulate the linearization of~\eqref{eq:CR} and~\eqref{eq:BC} as an
operator of the form $\ps + A(s) + B(s)$ where $s \mapsto A(s)$ is a
path of unbounded operators converging to a self-adjoined operator as
$s$ tends to $\infty$ and $s \mapsto B(s)$ is a path of anti-symmetric
bounded operators converging to zero as $s$ tends to $\infty$.

\bigskip Fix a point $p \in L_0 \cap L_1$ and neighborhood $U\subset
M$ from  Lemma~\ref{lmm:chart}. Given a map
$u:[0,\infty)\times [0,1] \to M$ which satisfies~\eqref{eq:CR}
and~\eqref{eq:BC}. Assume that the image of $u$ is completely
contained in $U$. We consider the linearized Cauchy-Riemann operator
\begin{equation}
  \label{eq:Duaa}
  D_u:\Gamma(u^*TM) \to \Gamma(u^*TM), \qquad \xi \mapsto \na_s \xi +
J(u)\na_t \xi+ \na_\xi J(u) \pt u\,.
\end{equation}
Let $\Phi$ be the trivialization from Lemma~\ref{lmm:trivi}
and abbreviate $\Phi_u(s,t):=\Phi_t(u(s,t))$ for all $(s,t) \in
[0,\infty)\times [0,1]$. We define the matrix valued function
$S:[0,\infty)\times [0,1] \to \R^{2n\times 2n}$ by
\begin{equation}\label{eq:Sdef}
  \Phi_u\left(\ps \xi + \Jstd \pt \xi + S \xi\right)=D_u \Phi_u \xi\,,
\end{equation}
for all smooth $\xi:[0,\infty)\times [0,1] \to \R^{2n}$. Abbreviate
$\Phi_\infty(t):=\Phi_t(0)$ and $J_\infty(t):=J_t(0)$ for all $t \in
[0,1]$. Similarly we define $S_\infty:[0,1] \to \R^{2n \times 2n}$
via
\begin{equation}
  \label{eq:Sinfdef}
  \Phi_\infty\left(\Jstd \pt \xi + S_\infty \xi \right)=  J_\infty\pt \Phi_\infty\xi\,,  
\end{equation}
for all smooth $\xi:[0,1] \to \R^{2n}$. The next lemma relates the
asymptotic behavior of $S$ to the asymptotic behavior of $u$. Since
the proof does not use the fact that $L_0$ and $L_1$ intersect
transversely, we quote directly from~\cite{Robbin:Strip}.
\begin{lmm}\label{lmm:Sconv}
  The matrix $S_\infty(t)$ symmetric for all $t \in [0,1]$. There
  exists constants $s_0$ and $c>0$ such that
  \[\nmm{S(s,t) - S_\infty(t)} \leq c \Big(\nmm{\ps u(s,t)} + \di{u(s,t),p}\Big)\;,\]
  for all $s \geq s_0$ and $t\in [0,1]$. Moreover if $u$ satisfies an
  uniform $C^k$-bound for some $k \geq 0$, then there exist a constant
  $c_k >0$ such that
  \[\Nm{S - S_\infty}_{C^k\left([s,\infty)\times [0,1]\right)} \leq c_k
  \Big(\Nm{\ps u}_{C^k\left([s,\infty) \times [0,1]\right)} +
  \sup_{s \leq \sigma,0 \leq t \leq 1 }\di{u(\sigma,t),p}\Big)\;.\]
\end{lmm}
\begin{proof}
  See \cite[Lemma 2.2]{Robbin:Strip}.
\end{proof}

\bigskip Consider the Hilbert space $H=L^2([0,1],\R^{2n})$ equipped
with standard inner product $\<\cdot,\cdot\>$ and norm $\Nm{\cdot}$.
Consider the dense subspace $V \subset H$ given by
\[
V= \left\{ \xi \in H^{1,2}([0,1],\R^{2n}) \mid \xi(0),\xi(1) \in \R^n \times \{0\} \right\}\,.
\]
Given $s \in [0,\infty)$ we define the linear operators $A(s): V \to
H$, $\xi \mapsto A(s)\xi$ where
\begin{equation*}
  \left(A(s)\xi\right)(t) = \Jstd\pt \xi(t) + \frac 12 \left(S(s,t) +
    S(s,t)^T \right)\xi(t)\,,
\end{equation*}
and the operator $A_\infty:V \to H$, $\xi \mapsto A_\infty\xi$ where
\begin{equation}
  \label{eq:Ainfdef}
  \left(A_\infty \xi\right)(t) = \Jstd \pt \xi(t) + S_\infty(t)\xi(t)\,.
\end{equation}
Moreover define the linear operator $B(s): H \to H$, $\eta
\mapsto B(s)\eta$ given by
\begin{equation*}
  \left(B(s)\eta\right)(t) = \frac 12 \left(S(s,t)-S(s,t)^T \right)
\eta(t)\,.
\end{equation*}
We quote the next lemma directly from~\cite{Robbin:Strip}. It states
that the paths $s \mapsto A(s)$ and $s \mapsto B(s)$ are continuously
differentiable. We denote the derivatives by $\dot A(s)$ and $\dot
B(s)$ respectively.
\begin{lmm}\label{lmm:AB}
  The operators $A(s) - A_\infty$, $\dot A(s)$, $B(s)$ and $\dot B(s)$
  have extensions to bounded linear operators on $H$. Moreover there exists a
  constant $c>0$ such that for ever $s \geq 0$,
  \[\Nm{A(s)-A_\infty} + \Nm{B(s)} \leq
  c\sup_{t \in [0,1]} \nolimits \left(\nmm{\ps u(s,t)}+
    \di{u(s,t),p}\right)\;,\]
\[\Nmm{\dot A(s)} + \Nmm{\dot B(s)}
\leq c \sup_{t \in [0,1]} \nolimits \left(\nmm{\nabla_s \ps u(s,t)} +
  \nmm{\ps u(s,t)} + \di{u(s,t),p}\right)\;.\] In which
$\Nm{\,\cdot\,}$ denotes the operator norm on bounded linear
operators.
\end{lmm}
\begin{proof}
  \cite[Lemma 2.3]{Robbin:Strip}
\end{proof}
Define the function $\xi_u:[0,\infty)\times [0,1] \to \R^{2n}$, $(s,t) \mapsto
\xi_u(s,t)$ 
\begin{equation}
  \label{eq:xiu} 
 \xi_u(s,t)=\Phi_u(s,t)^{-1}\ps u(s,t)\,.
\end{equation}
Since $u$ solves the Cauchy-Riemann equation~\eqref{eq:CR} and $J$ is
$s$-independent, the vector field $\ps u$ lies in the kernel of $D_u$
and with the above definition we have 
\begin{equation}
  \label{eq:CRxiu}
  \ps \xi_u(s,t) + \Jstd \pt \xi_u(s,t) + S(s,t)\xi_u(s,t) = 0\,.  
\end{equation}
By construction we have $\xi_u(0,\cdot),\xi_u(1,\cdot) \subset \R^n
\times 0$ for $k=0,1$, in particular $\xi_u(s,\cdot) \in V$ for all $s
\geq 0$. Abusing notation we denote the path $[0,\infty)\to V$, $s
\mapsto \xi_u(s,\cdot)$ also by $\xi_u$. According to~\eqref{eq:CRxiu} we
have for all $s \geq 0$
\begin{equation}
  \label{eq:CRAB}
  \ps \xi_u(s) + A(s) \xi_u(s) + B(s) \xi_u(s) =0\,.
\end{equation}
In contrast to the setting of~\cite{Robbin:Strip}, the asymptotic
operator $A_\infty$ is no longer injective in our situation. To be
able to conclude we need that the component of $\xi_u$ is the kernel
is controlled by $\xi_u$ as provided in the next lemma. Let $\ker
A_\infty \subset H$ denote the kernel of $A_\infty$ considered as a
closed subspace of $H$ and $P:H \to \ker A_\infty$ the orthogonal
projection to the kernel of $A_\infty$.
\begin{lmm}\label{lmm:P}
  There exists a uniform constant $c$ such that for all $s \geq 0$ we have
  \[ \Nm{P \xi_u(s)} \leq c\, \sup_{t \in [0,1]} \nolimits \di{u(s,t),p}
  \Nm{\xi_u(s)}\,.\]
\end{lmm}
\begin{proof}
  Via Lemma~\ref{lmm:chart} we assume without loss of generality that
  $U\subset \R^{2n}$ equipped with the standard symplectic structure
  and $L_0,L_1$ are fixed linear Lagrangian subspaces. Moreover we
  think of the almost complex structure $J$ and $\Phi$ as matrix
  valued functions.  Fix $s \geq 0$ and abbreviate
  $\gamma_s(t):=u(s,t)$ for all $t \in [0,1]$. The path $t \mapsto
  \Phi_\infty^{-1}(t)\gamma_s(t)$ is an element of the domain of
  $A_\infty$. Let $e \in \ker A_\infty$ be an element with
  $\Nm{e}=1$. Since $A_\infty$ is symmetric we compute using the
  definition of $A_\infty$ (\cf equations~\eqref{eq:Ainfdef}
  and~\eqref{eq:Sinfdef})
  \[
  \<\Phi_\infty^{-1} J_\infty \pt
  \gamma_s,e\> = \<\Phi_\infty^{-1}J_\infty \pt \Phi_\infty
  \Phi^{-1}_\infty \gamma_s,e\>=\<A_\infty \Phi^{-1}_\infty
  \gamma_s,e\>=0\,.
  \]
  Abbreviate $G_\infty(t) := \Phi_\infty(t)^{-1}J_\infty(t)$ and
  $G_u(t):=\Phi_u(t)^{-1}J_t(u(s,t))$ for all $t \in [0,1]$.  By definition
  of $\xi_u$ (\cf equation~\eqref{eq:xiu}) and since $u$
  solves~\eqref{eq:CR} we have $ \xi_u = \Phi_u^{-1}\ps u =
  -\Phi_u^{-1}J(u) \pt u = -G_u \pt \gamma_s$. Thus
  \begin{multline*}
  \<\xi_u(s),e\>= -\<G_u \pt \gamma_s,e\> =
  \<(G_\infty - G_u) \pt \gamma_s,e\>
  \leq\\\leq \Nm{G_\infty - G_u}_{C^0} \Nm{G_u^{-1}}_{C^0} \Nm{\xi(s)}\,.   
  \end{multline*}
  The matrix $G_t(q):=\Phi_t(q)^{-1}J_t(q)$ is invertible for all
  $(t,q) \in [0,1]\times U$ and satisfies an uniform $C^1$-bound. In
  particular there exists a uniform constant $c$ such that for all $s
  \geq 0$ we have
  \[
  \<\xi_u(s),e\> \leq c \, \sup_{t \in [0,1]} \di{u(s,t),p} \Nm{\xi_u(s)}\,.
  \]
  The claim follows after taking the supremum over all $e \in \ker
  A_\infty$ with $\Nm{e}=1$ of the last estimate.
\end{proof}
\subsection{Proofs}
\begin{proof}[Proof of Theorem~\ref{thm:remove}]\setcounter{stp}{0}
  Given $u$ which satisfies~\eqref{eq:CR} and~\eqref{eq:BC}. Assume
  additionally that $u$ satisfies~\eqref{eq:decay}
  then~\eqref{eq:energy} clearly follows. Also if $u$
  satisfies~\eqref{eq:energy}, then~\eqref{eq:limit} follows by the
  estimate~\eqref{eq:duhalfE}. In order to prove the theorem it
  suffices to show that if $u$ satisfies~\eqref{eq:limit}
  then~\eqref{eq:decay} follows. Provided with the exponential decay
  of the energy this follows from elliptic bootstrapping as explained
  on \cite[Page 594]{Robbin:Strip}. We quickly repeat the argument.

  Given $u$ such that~\eqref{eq:CR},~\eqref{eq:BC}
  and~\eqref{eq:limit} holds. Let $\rho=\rho(\mu,p)$ denote the
  constant from the isoperimetric inequality (\cf Prop.\
  \ref{prp:iso}). By~\eqref{eq:limit} we assume without loss of
  generality that $u(s,t) \in B_{\rho}(p)$ for all $s\geq 0$ and $t\in
  [0,1]$. Moreover we assume that the image of $u$ lies in a suitable
  symplectic chart as considered in Section~\ref{sec:linaa}. The map
  $\xi:=\xi_u$ defined in~\eqref{eq:xiu} solves~\eqref{eq:CRxiu}, \ie
  for all $(s,t)\in [0,\infty)\times [0,1]$ we have
  \begin{equation}\label{eq:CRxi}
    \ps \xi(s,t) + \Jstd \pt \xi(s,t) + S(s,t) \xi(s,t) =0\;.
  \end{equation}
  Fix $k \in \N_0$ and $s \geq k$. For all $\nu \in
  \N_0$ we define the shifted maps
  \[\xi_\nu(\sigma,t) := \xi(\sigma +s + \nu,t),\qquad S_\nu(\sigma,t) := S(\sigma +s+\nu,t)\;.\] 
  For any $a<b$ with possibly $b=\infty$ we abbreviate $\Sigma_a^b =
  [a,b] \times [0,1]$ and $\Sigma_a^\infty = [a,\infty) \times [0,1]$
  if $b=\infty$. For any $\ell \in \N_0$ let
  $\Nm{\cdot}_{\ell,2;\Sigma_a^b}$ denote the standard Sobolev norm of
  $H^{\ell,2}(\Sigma_a^b,\R^{2n})$. Using elliptic bootstrapping (\cf
  \cite[Lemma C.1]{Robbin:Strip}) and since $\xi$
  solves~\eqref{eq:CRxi} we have a constant $c_1=c_1(\ell)$ which
  depends on $\ell$ but is independent of $\xi$, $S$ and $\nu$ such
  that
  \[\Nm{\xi}_{\ell,2;\Sigma_{s}^\infty}^2 = \sum_{\nu=0}^\infty
  \Nm{\xi_\nu}_{\ell,2;\Sigma_0^1}^2 \leq c_1 \sum_{\nu=0}^\infty
  \left(\Nm{S_\nu \xi_\nu}_{\ell-1,2;\Sigma_{-1}^2}^2 +
    \Nm{\xi_\nu}_{\ell-1,2;\Sigma_{-1}^2}^2\right)\;.\] According to
  Lemmas~\ref{lmm:Sconv} and Corollary~\ref{cor:Ckpsu} the smooth maps $S_\nu$
  satisfy an uniform $C^\ell$-bound, hence there exists a uniform
  constant $c_2=c_2(\ell)$ such that 
  \[\Nm{\xi}_{\ell,2;\Sigma_{s}^\infty}^2 \leq c_2 \sum_{\nu=0}^\infty 
  \Nm{\xi_\nu}^2_{\ell-1,2;\Sigma_{-1}^2} = 3c_2
  \Nm{\xi}_{\ell-1,2;\Sigma_{s-1}^\infty}^2\;.\] Repeating the
  previous $k$ times we conclude that for each $k \in \N_0$ we have
  constant $c_3=c_3(k)$ depending on $k$ such that
  \[
  \Nm{\xi}_{k,2;\Sigma_{s}^\infty}^2 \leq c_3
  \Nm{\xi}_{0,2;\Sigma_{s-k}^\infty}^2 = c_3
  E(u;\Sigma_{s-k}^\infty)\;.
  \]
  The $C^k$-norm of $\Phi$ is bounded and after
  Corollary~\ref{cor:Ckpsu} so is the $C^k$-norm of the map $(s,t)
  \mapsto \Phi_u(s,t)=\Phi_t(u(s,t))$, hence there exists a constant
  $c_4$ such that
  \[
  \Nm{\ps u}_{C^k(\Sigma_s^\infty)} \leq c_4 \Nm{\xi}_{C^k(\Sigma_s^\infty)}\;.
  \]
  By Sobolev embedding, the last two estimates and
  Lemma~\ref{lmm:Edecay} we have constants $c_5$ and $c_6$ such that
  \[
  \Nm{\ps u}_{C^k(\Sigma_s^\infty)} \leq c_5
  \Nm{\xi}_{k+2,2;\Sigma_s^\infty} \leq c_3c_5
  \Nm{\xi}_{L^2(\Sigma^\infty_{s-k-2})} \leq c_6
  e^{-\mu(s-k-2)}\,.\] This
  shows~\eqref{eq:decay} and hence the theorem.
\end{proof}

\begin{proof}[Proof of Theorem~\ref{thm:eigenval}]
  We follow closely the line of arguments from the proof
  of~\cite[theorem B]{Robbin:Strip}. By Theorem~\ref{thm:remove} we
  assume without loss of generality that the image of $u$ lies in a
  suitable symplectic chart. With notations from
  Section~\ref{sec:linaa} we see that $\xi:=\xi_u:[0,\infty) \to V$ 
  satisfies~\eqref{eq:CRAB}, \ie for all $s \geq 0$ we have
  \[
   \ps \xi(s) + A(s)\xi(s) + B(s)\xi(s)=0\,.
  \]
  In~\cite[theorem 4.1]{Frauenfelder:PhD} it is proven that $A_\infty$
  is Fredholm and self-adjoined considered as an unbounded operator in
  $H$. Using Lemmas \ref{lmm:AB} and~\ref{lmm:P} together with the
  exponential decay of $u$ given in equation~\eqref{eq:decay} all the
  requirements for Lemma~\ref{lmm:cvvinfty} are fulfilled. Hence there
  exists an eigenvalue $\alpha$ of $A_\infty$, an eigenvector $ \zeta
  \in \ker(A_\infty - \alpha)$ and a constant $c$ such that for all $s
  \geq 0$ we have
  \begin{equation*}
    \int_0^1 \nmm{e^{\alpha s}\xi(s,t) - \zeta(t)}^2 \d t \leq c
    e^{-2 \mu s}\,.
  \end{equation*} 
  Abbreviate $\Sigma_s:=[s,\infty)\times [0,1]$.  We prove by
  induction that for each $k \in \N_0$ there exists a constant $c_k$
  such that for all $s \geq 0$
  \begin{equation}
    \label{eq:Hkzeta}
    \Nm{e^{\alpha s} \xi - \zeta}_{H^{k,2}(\Sigma_s)} \leq c_k e^{-\mu s}\,.
  \end{equation}
  For $k=0$ this follows by the last estimate. Now assume
  that~\eqref{eq:Hkzeta} has been established for some $k\geq
  0$. Abbreviate $\theta(s,t):= e^{\alpha s} \xi(s,t) - \zeta(t)$ for
  all $(s,t) \in [0,\infty)\times [0,1]$. The map
  $\theta:[0,\infty)\times [0,1] \to \R^{2n}$ satisfies
  \[\ps \theta(s,t)+ \Jstd \pt \theta(s,t) = \eta(s,t),\qquad \theta(s,0),\ \theta(s,1) \subset \R^n \times \{0\}\;,\]
  for all $(s,t) \in [0,\infty)\times [0,1]$, in which
  \begin{equation*}
    \eta(s,t) = \left(\alpha-S_\infty(t)\right)\theta(s,t) +
  \left(S_\infty(t) -S(s,t)\right)\left(\theta(s,t)
    -\zeta(t)\right)\,.
  \end{equation*}
  By the $C^k$-bounds of $S$ (\cf Lemma~\ref{lmm:Sconv}) and the
  exponential decay for $\ps u$ (\cf equation~\eqref{eq:decay}) we
  have a constant $c=c(k)$ such that for all $s \geq 0$
  \[\Nm{S-S_\infty}_{C^k\left(\Sigma_s\right)}\leq
  c e^{-\mu s}\,.\] By this estimate and the induction hypotheses
  there exists another constant $c=c(k)$ such that for all $s \geq 0$
  \[\Nm{\ps \theta + \Jstd \pt \theta}_{H^{k,2}\left(\Sigma_s\right)} \leq c  e^{-\mu s}\,.\] 
  Then after elliptic bootstrapping (\cf \cite[Lemma
  C.1]{Robbin:Strip}) we conclude that~\eqref{eq:Hkzeta} holds with
  $k$ replaced with $k+1$. This shows~\eqref{eq:Hkzeta} for all $k \in
  \N$.  By the Sobolev embedding we also conclude for all $k \in
  \N$ we have a possibly larger constant $c_k$ such that
  \begin{equation}
    \label{eq:Ckzeta}
    \Nm{e^{\alpha s} \xi-\zeta}_{C^k(\Sigma_s)} \leq c_k e^{-\mu s}\,.   
  \end{equation}

  By construction the Hessian $A_p$ and the operator $A_\infty$
  are conjugated via $\Phi_\infty$. In particular the
  path $[0,1]\to T_pM$, $t \mapsto \Phi_\infty(t) \zeta(t)$ is an
  eigenvector of $A_p$ with eigenvalue $\alpha$. Define the map
  $w:[s_0,\infty) \times [0,1] \to T_pM$ by
  \[
  u(s,t) = \exp_p( -\alpha^{-1} e^{-\alpha s}\Phi_\infty(t)\zeta(t) +
  w(s,t))\,.
  \]
  We derive the equation by $\ps$ and obtain
  \[
  \ps u  = E(\tilde u) e^{-\alpha s} \Phi_\infty \zeta + E(\tilde u) \ps w\,,
  \]
  in which $E(\tilde u)$ denotes the derivative of the exponential map
  at \[\tilde u := -\alpha^{-1}e^{-\alpha s}\Phi_\infty \zeta + w\,.\]
  Rewriting the last equation gives
  \begin{align*}
    \ps w &= E(\tilde u)^{-1} \ps u - e^{-\alpha s} \Phi_\infty \zeta \\
    &=E(\tilde u)^{-1} \Phi_u (\xi - e^{-\alpha s} \zeta) +
    e^{-\alpha s}(E(\tilde u)^{-1} \Phi_u -\Phi_\infty)\zeta\,.
  \end{align*}
  By the exponential decay of $u$, since the $C^k$-norms of $E$ and
  $\Phi$ are uniformly bounded and $E(0)$ is the identity we conclude
  that there exists a possibly larger constant $c_k$ such that
  \begin{equation*}
    \Nm{E(\tilde u)^{-1}\Phi_u - \Phi_\infty}_{C^k(\Sigma_s)} \leq
    \Nm{(E(\tilde u)^{-1} - \one)\Phi_u}_{C^k(\Sigma_s)} +
    \Nm{\Phi_u-\Phi_\infty}_{C^k(\Sigma_s)} \,,
  \end{equation*}
  is bounded by $c_ke^{-\mu s}$.  Hence with together with
  estimate~\eqref{eq:Ckzeta} we obtain a possibly larger constant
  $c_k$ such that for all $s \geq 0$
  \[
  \Nm{\ps w}_{C^k(\Sigma_s)} \leq c_k e^{-(\mu + \alpha)s}\,.
  \]
  By construction we see that $\lim_{s \to \infty} w(s,t) =0$ for each
  fixed $t \in [0,1]$ and thus
  \[
  w(s,t) = -\int_s^\infty \partial_\sigma w(\sigma,t) \d \sigma\,.
  \]
  Using the previous estimate on $\ps w$ we conclude that $w$
  also satisfies an exponential decay. This proves the theorem.
\end{proof}
\section{Compactness}

We study sequences of (perturbed) holomorphic strips with boundary on
two Lagr\-ang\-ians. We show that if the energy of the sequence is
uniformly bounded, then a subsequence converges in a certain sense to
a broken strip. The convergence is a very crude version of Gromov
compactness, which forgets the so called ``bubbles'' and just
remembers their energies. If the Lagrangians are monotone, this will
prove to be sufficient for our purposes. Convergence of holomorphic
strips has originally been studied by Floer
in~\cite{Floer:Intersection} in which he a priory excluded the bubbles
and later by Oh in~\cite{Oh:diskI} for the monotone case. Both of the
results are formulated under the assumption that the Lagrangians
intersect transversely. Here we give a refinement which allows
cleanly intersecting Lagrangians. In the special case where both
Lagrangians are the same and the almost complex structure does not
depend on the domain a sequence of holomorphic strips is nothing but a
sequence of holomorphic disks and Gromov compactness of these is fully
described in~\cite{Frauenfelder:disks}. Most proofs are straight
forward generalizations of this special case. An alternative approach
is developed Ivashkovich-Shevchishin in
\cite{IvashkovichShevchishin}.

\subsection{Cauchy-Riemann-Floer equation}\label{sec:SRF}
Let $(M,\omega)$ be a symplectic manifold and $L_0,L_1 \subset M$ be
two Lagrangian submanifolds not necessarily in clean intersection. We
abbreviate the strip $\Sigma:=\R \times [0,1]$. Further denote by $X
\in C^\infty(\Sigma,\Vect(X))$ and $J \in
C^\infty(\Sigma,\End(TM,\omega))$ a vector field and an almost complex
structure respectively. A \emph{non-trivial finite-energy
  $(J,X)$-holomorphic strip $u$ with boundary in $(L_0,L_1)$} is a map
$u:\R \times [0,1] \to M$ which satisfies
\begin{equation}
  \label{eq:CRJX}
  \begin{gathered}
    \ps u + J(u) \left(\pt u - X(u)\right) =0\,,\\
    \Res{u}{t=0} \subset L_0,\qquad \Res{u}{t=1} \subset L_1\,,\\
    0< \int \nm{\ps u}^2_J \d s \d t < \infty\;.
  \end{gathered}
\end{equation}
By convenience we often just write that $u$ is a
\emph{$(J,X)$-holomorphic strip}.  For an open subset $\Omega \subset
\R \times [0,1] $, we define the \emph{energy of $u$ on $\Omega$} by
\[E(u):= \int |\ps u|^2_J \, \d s \d t\,\qquad E(u;\Omega) :=
\int_\Omega \nm{\ps u}_J^2 \,\d s \d t\,.\] 
In order to controll the asymptotic behaviour of $(J,X)$-holomorphic strips
we assume that $J$ is asymptotically constant and $X$ is
asymptotically constant to a Hamiltonian vector field of a clean
Hamiltonian (\cf Definition~\ref{dfn:MBreg}).
\begin{dfn}\label{dfn:JXadm}
  Given $J \in C^\infty(\Sigma,\End(TM,\omega))$ and $X \in
  C^\infty(\Sigma,\Vect(X))$,
  \begin{itemize}
  \item we call $J$ \emph{admissible} if there exists $s_0$ and paths
    $J_-$ and $J_+$ such that $J(-s,\cdot) = J_-$ and $J(s,\cdot)=J_+$
    for all $s \geq s_0$ and
  \item we call $X$ \emph{admissible} if there exists $s_0$ and clean
    Hamiltonians $H_-$ and $H_+$ such that $X(-s,\cdot)=X_{H_-}$ and
    $X(s,\cdot) =X_{H_+}$ for all $s \geq s_0$.
  \end{itemize}
  We call $J$ (\resp $X$) \emph{$\R$-invariant} if the same holds for
  $s_0=0$. Necessarily for $\R$-invariant structures we have $J_-=J_+$
  (\resp $H_-=H_+$).
\end{dfn}
\begin{lmm}\label{lmm:ulimits}
  Given admissible $J$ and $X$. For any $(J,X)$-holomorphic strip $u$
  the limits $u(-\infty):=\lim_{s \to -\infty} u(s,\cdot)$ and
  $u(\infty):=\lim_{s \to \infty} u(s,\cdot)$ exists and with the
  notation above we have $u(-\infty) \in \I_{H_-}(L_0,L_1)$ and
  $u(\infty) \in \I_{H_+}(L_0,L_1)$
\end{lmm}
\begin{proof}
  Use Theorem~\ref{thm:remove} and Lemma~\ref{lmm:change}
\end{proof}
\begin{dfn}\label{dfn:comp}
  Given a sequence of admissible almost complex structures
  $(J_\nu)_{\nu\in \N}$ and a sequence of admissible vector fields
  $(X_\nu)_{\nu\in \N}$ converging to $J$ and $X$ respectively.  A
  sequence $(u_\nu)_{\nu\in \N}$ of $(J_\nu,X_\nu)$-holomorphic strips
  \emph{Floer-Gromov converges modulo bubbling} to a tuple
  $v=(v_1,\dots,v_k)$ if there exists sequences $(a_1^\nu), \dots
  ,(a_k^\nu) \subset \R$ and empty or finite sets $Z_1,\dots,Z_k
  \subset \Sigma$ such that for all $j=1,\dots,k$ we have
  \begin{enumerate}[label=(\roman*)]
  \item the sequence $u_\nu \circ \tau_{a_j^\nu}$ converges to $v_j$
    in $C^\infty_\loc(\Sigma \setminus Z_j)$
  \item for all $z\in Z_j$ the limit $m_{j,z} := \lim_{\e \to 0}
    \lim_{\nu \to \infty} E(u_\nu \circ \tau_{a_j^\nu},B_\e(z))$
    exists and is strictly positive,
  \item if $v_j$ is constant then $Z_j$ is not empty,
  \item  $\lim_{\nu \to \infty} u_\nu(-\infty) = v_1(-\infty)$, $\lim_{\nu \to
      \infty} u_\nu(\infty) = v_k(\infty)$ and if $j
    \neq k$ then $v_j(\infty) = v_{j+1}(-\infty)$.
  \end{enumerate}
  Moreover we have
  \[
  \lim_{\nu \to \infty} E(u_\nu) = \sum_{j=1}^k E(v_j) + m,\qquad m:=\sum_{j=1}^k \sum_{z \in Z_j} m_{j,z}\,.
  \]
  If the sets $Z_1,\dots,Z_k$ are all empty we say that $(u_\nu)$
  \emph{Floer-Gromov converges}.
\end{dfn}
\begin{thm}\label{thm:comp}
  Given a sequence of admissible almost complex structures
  $(J_\nu)_{\nu\in \N}$ and a sequence of admissible vector fields
  $(X_\nu)_{\nu\in \N}$ converging to $J$ and $X$ respectively. Any
  sequence $(u_\nu)_{\nu \in \N}$ of $(J_\nu,X_\nu)$-holomorphic with
  uniformly bounded energies has a subsequence which Floer-Gromov
  converges modulo bubbling.
\end{thm}
\begin{proof}
  Iteratively apply Lemma~\ref{lmm:convbub} and Lemma~\ref{lmm:soft} given below.
\end{proof}

\subsection{Local convergence}
In this section we provide local convergence results.  We use a
well-known trick and transform the statement of perturbed holomorphic
curves into a statement for holomorphic curves at the cost of turning
the target space into a non-compact space. Then the results follows
from standard theory on holomorphic curves. Given an open subspace
$\Omega \subset \Sigma$, we say that $u:\Omega \to M$ is a
$(J,X)$-holomorphic map if $u$ satisfies~\eqref{eq:CRJX} wherever it
is defined.
\begin{lmm}[bounded gradient compactness]\label{lmm:bgcomp}
  Given 
  \begin{itemize}
  \item a sequence $\Omega_1,\Omega_2,\dots \subset \Sigma$ of open
  subsets which exhaust $\Omega \subset \Sigma$,
\item a sequence $J_1,J_2,\dots$ such that
  $J_\nu:\Omega_\nu\to \End(TM,\omega)$ are almost complex structures
  converging to $J:\Omega \to \End(TM,\omega)$ in $C^\infty_\loc$,
\item a sequence $X_1,X_2,\dots$ such that $X_\nu:\Omega_\nu\to
  \Vect(X)$ are vector fields converging to $X:\Omega \to \Vect(M)$ in $C^\infty_\loc$,
  \end{itemize}
  then for any sequence $u_1,u_2,\dots$ such that $u_\nu:\Omega_\nu\to
  M$ is a $(J_\nu,X_\nu)$-holo\-morphic map and assume that 
  \[
  \sup_{\nu \in \N} \Nm{\ps u_\nu}_{C^0} < \infty\,,
  \]
  there exists subsequence which converges to a map $u:\Omega\to M$ in
  $C^\infty_\loc$. Moreover the map $u$ is $(J,X)$-holomorphic.
\end{lmm}
\begin{proof}
  Define the manifold $\wt M := \Omega \times M$ with submanifolds
  \[    \widetilde L_0 = \left(\R \times \{0\} \cap \Omega\right)\times
  L_0,\qquad \widetilde L_1 = \left(\R \times \{1\} \cap \Omega\right)
  \times L_1\;.\]
  Define almost complex structures
  $\widetilde J_\nu,\widetilde J\in\mathrm{End}(T\widetilde M)$ via
  \[ \widetilde J_\nu(s,t,p) = \begin{pmatrix}
    0&-1&0\\1&0&0\\X_\nu(s,t,p)&-J_\nu(s,t,p)X_\nu(s,t,p)&J_\nu(s,t,p)\end{pmatrix}\;,\]
  and similarly $\widetilde J$. One checks directly that the
  manifolds $\widetilde L_0$ and $\widetilde L_1$ are totally real
  with respect to $\widetilde J$ and that the curves $\widetilde
  u_\nu(s,t)=(s,t,u_\nu(s,t))$ solve
  \[\CR_{\widetilde J_\nu} \widetilde u_\nu = \ps \widetilde u_\nu
  +  \widetilde J_\nu(\widetilde u_\nu) \pt \widetilde u_\nu =0\;,\]
  with boundary conditions
  \begin{equation}
    \label{eq:BCtilde}
    \wt u_\nu(\cdot,0) \subset \wt L_0, \hspace{2cm} \wt u_\nu(\cdot,1) \subset \wt L_1\,.
  \end{equation}
  We equip $\wt M$ with the product symplectic structure, then $\wt
  J_\nu$ is compatible and for the associated metric we have
  \begin{equation*}
    |\ps \wt u_\nu|^2 = |\pt \wt u_\nu|^2 = 1+
    \omega(\ps u_\nu,\pt u_\nu) = 1+ |\ps u_\nu|^2 + \omega(\ps u_\nu,
    X_\nu)\,.
  \end{equation*}
  We see that the gradient of $\wt u_\nu$ is uniformly bounded. By the
  Theorem of Arzel\`a-Ascoli there exists $\wt u:\Omega \to \wt M$
  such that $\wt u_\nu$ converges to $\wt u$ in $C^0_\loc$. It is easy
  to see that $\wt u$ satisfies the boundary
  condition~\eqref{eq:BCtilde} and $\wt u(s,t)=(s,t,u(s,t))$ for
  all $(s,t) \in \Omega$ with some map $u:\Omega \to M$. To improve
  the convergence and show that $\wt u$ is $\wt J$-holomorphic (thus
  $u$ is $(J,X)$-holomorphic) we proceed as in proof of \cite[Theorem
  B.4.2]{Bibel}.
\end{proof}
\begin{lmm}[Convergence modulo bubbling]\label{lmm:convbub}
  Assume that $\Omega,J,X,\Omega_\nu,J_\nu,X_\nu$ satisfy the
  hypotheses of Lemma~\ref{lmm:bgcomp}. Let $u_1,u_2,\dots$ be a
  sequence of maps such that $u_\nu:\Omega_\nu\to M$ is a $(J_\nu,X_\nu)$-holomorphic map and assume that 
  \[\sup_{\nu \in \mathbbm{N}} E(u_\nu;\Omega_\nu)  < \infty\;,\] then there exists a
  subsequence, still denoted by $(u_\nu)$, a $(J,X)$-holomorphic map
  $u:\Omega \to M$ and an empty or finite set of points $Z=\{z_1,\dots,z_\ell\}
  \subset \Omega$ such that the following holds
  \begin{enumerate}[label=(\roman*)]
  \item $u_\nu$ converges to $u$ in $C^\infty_\loc(\Omega \setminus
    Z)$
  \item for every $i=1,\dots, \ell$ and every $\varepsilon >0$ such
    that $B_\varepsilon(z_i) \cap Z = \{z_i\}$, the limit
    \[m_\varepsilon(z_i):=\lim_{\nu \to \infty}
    E(u_\nu;B_\varepsilon(z_i) \cap \Omega_\nu)\;,\] exists. Moreover
    \[m_i:=m(z_i):=\lim_{\varepsilon \to 0} m_\varepsilon(z_i), \] is
    the energy of a non-constant holomorphic sphere of disk.
  \item For every compact subset $K\subset \Omega$ with $Z \subset
    \mathrm{int}(K)$,
    \[ \lim_{\nu \to \infty}E(u_\nu;K) = E(u;K) +
    \sum_{j=1}^\ell m_i\,.\]
  \end{enumerate}
\end{lmm}
\begin{proof}
  See \cite[Theorem 4.6.1]{Bibel} provided with Lemma~\ref{lmm:bgcomp}.
\end{proof}
\subsection{Convergence on the ends}
In this section we consider convergence of $(J,X)$-holomorphic curves
restricted to the half-strip $\Sigma^+ := [0,\infty) \times [0,1]$. We
assume without loss of generality that $X(s,\cdot)=X_H$ for all $s
\geq 0$ and some clean Hamiltonian function $H$. 
\begin{lmm}[$C^0$-convergence on ends]\label{lmm:convends} Fix a
  clean Hamiltonian $H$ and an almost complex structure $J \in
  C^\infty([0,1],\End(TM,\omega)$.  Given a sequence $(u_\nu)$ of
  $(J,H)$-holomorphic half-strips. Assume that $u_\nu$ converges to a
  half-strip $u$ in $C^\infty_\loc(\Sigma^+,M)$ and
  \[\lim_{\nu \to \infty} E(u_\nu) = E(u) < \infty\,.\]
  Then $(u_\nu)$ converges to $u$ in the topology of
  $C^0(\Sigma^+,M)$. Moreover $ u_\nu(\infty)$ converges to
  $u(\infty)$ as $\nu$ tends to $\infty$.
\end{lmm}
\begin{proof}\setcounter{stp}{0}
  See \cite[Proposition 4.3.10]{Schwarz:PhD} and \cite[Proposition
  4.3.11]{Schwarz:PhD} for the proof for case of $(J,H)$-holomorphic
  cylinders asymptotic to non-degenerate Hamiltonian orbits. The proof
  here is a little different and uses the isoperimetric
  inequality. According to Lemma~\ref{lmm:change} we assume without loss of
  generality that $H=0$ and $L_0$, $L_1$ intersect cleanly. Let
  $U_\Poz$ denote the neighborhood of $L_0 \cap L_1$ given by
  Proposition~\ref{prp:poz}. We decompose
  \[L_0 \cap L_1 = C_1 \cup C_2 \cup \dots \cup C_m\;,\] into
  connected components and by possibly making $U_\Poz$ smaller we
  obtain a respective decomposition
  \[U_\Poz = U_1 \cup U_2 \cup \dots \cup U_m\;,\] such that $C_i
  \subset U_i$ for all $i=1,\dots,m$ and $U_i \cap U_j = \emptyset$
  whenever $i\neq j$. In view of Theorem~\ref{thm:remove} we assume
  without loss of generality that $u(s,t) \in U_1$ for all $s \geq 0$
  and $t \in [0,1]$.
  \begin{stp}
    There exists an $s_0$ and $\nu_0$ such that 
    $u_\nu(s,t) \in U_1$ for all $s \geq s_0$, $t \in [0,1]$ and $\nu \geq \nu_0$.
  \end{stp}
  \noindent By contradiction assume that there exists a sequence
  $(s_\nu,t_\nu)$ with $s_\nu \to \infty$ such that
  \begin{equation}
    \label{eq:outU1}
    u_\nu(s_\nu,t_\nu) \in M \setminus U_1\;,
  \end{equation}
  for all $\nu \geq 0$. For $0 \leq a<b$ we
  abbreviate
  \[E_\nu(a,b)=E(u_\nu;[a,b]\times [0,1]),\qquad
  E(a,b)=E(u;[a,b]\times [0,1])\;,\] and similarly $E_\nu(a,\infty)$
  and $E(a,\infty)$. Since $E(u_\nu) \to E(u)$ we have
  \[0 \leq \lim_{\nu \to \infty} E_\nu(s_\nu-a,\infty) \leq \lim_{\nu
    \to \infty} E_\nu(b,\infty) = E(b,\infty)\;,\] for any $0 \leq a
  <b$. This shows that
  \[\lim_{\nu \to \infty} E_\nu(s_\nu-a,\infty) = 0\;,\]
  for all $a >0$.  In particular $E_\nu(s_\nu-a,s_\nu+a) \to 0$ and
  thus $u_\nu(s_\nu,t_\nu) \to x_2\in L_0 \cap L_1$ in
  $C^\infty_\loc$. Because of~\eqref{eq:outU1} we must have $x_2
  \notin U_1$. Lets assume without loss of generality that $x_2 \in
  U_2$. But since $u_\nu \to u$ in $C^\infty_\loc$ and the image of
  $u$ lies completely in $U_1$ we find another sequence
  $(s^2_\nu,t^2_\nu)$ such that
  \[u_\nu(s^2_\nu,t^2_\nu) \in M \setminus \left(U_1 \cup
    U_2\right)\;,\] for all $\nu \geq 1$. Repeating the same argument
  we see that $u_\nu(s^2_\nu,t^2_\nu) \to x_3 \in L_0 \cap
  L_1$. With $x_3 \notin U_1 \cup U_2$. Lets say $x_3 \in U_3$. Yet
  again we find a sequence $(s^3_\nu,t^3_\nu)$ with
  $u(s^3_\nu,t^3_\nu) \notin U_1 \cup U_2 \cup U_3$ for all $\nu \geq
  1$ and eventually a sequence $(s^m_\nu,t^m_\nu)$ with $s^m_\nu \to
  \infty$ as $\nu$ tends to $\infty$ such that
  \begin{equation}
    \label{eq:Upout}
    u_\nu(s^m_\nu,t^m_\nu) \notin U_1 \cup U_2 \cup \dots \cup U_m = U_\Poz\;,
  \end{equation}
  for all $\nu \geq 1$. On the other hand, as before, we also have
  $u_\nu(s^m_\nu,t^m_\nu) \to x_{m+1} \in L_0 \cap L_1$. This
  contradicts~\eqref{eq:Upout} and consequently shows the claim.
  \begin{stp}
    The map $u_\nu$ convergences to $u$ in $C^0([0,\infty)\times
    [0,1])$.
  \end{stp}
  \noindent Assume by contradiction that there exists a sequence
  $(s_\nu,t_\nu) \in \Sigma_0^\infty$
  and a constant $\varepsilon >0$ such that 
  \begin{equation}
    \label{eq:epscontra}
    \di{u_\nu(s_\nu,t_\nu),u(s_\nu,t_\nu)} \geq \varepsilon\;,
  \end{equation}
  for all $\nu \geq 1$. Since $u_\nu \to u$ in $C^\infty_\loc$ we must
  necessarily have $s_\nu \to \infty$. By the last step,
  Lemma~\ref{lmm:Edecay} and Proposition~\ref{prp:meanvalue} there
  exists constants $c_1,c_2,\delta>0$ such that
  \begin{equation}
    \label{eq:psunu}
  \nm{\ps u_\nu(s,t)}^2 \leq c_1E_\nu(s-1,\infty) \leq c_1E_\nu(s_0,\infty)
  e^{-2\delta(s-s_0-1)}\leq c_2 e^{-2\delta s}\;,    
  \end{equation}
  for all  $s \geq s_0$ and $\nu \geq \nu_0$. Thus 
  \begin{equation}
    \label{eq:distunu}
    \di{u_\nu(a,t),u_\nu(s_\nu,t)} \leq \int_a^{s_\nu} \nm{\ps u_\nu} \d s
    \leq \frac{c_2}{\delta} \left(e^{-\delta a} - e^{-\delta
        s_\nu}\right) \leq \frac{c_2}{\delta} e^{-\delta a}\;, 
  \end{equation}
  for any $a>s_0$. By the same reasoning for $u$ and possibly making
  $c_2$ larger we also have
  \begin{equation}
    \label{eq:distu}
    \nm{\ps u(s,t)}\leq c_2 e^{-\delta s},\qquad \di{u(s,t),u(s_\nu,t)} \leq  \frac{c_2}{\delta} e^{-\delta s}\;,
  \end{equation}
  for all $s>s_0$ and $t \in[0,1]$. Choose $a$ large enough such that $c_2/\delta
  e^{-\delta a} \leq \varepsilon/4$ and $\nu_1$ large enough such that
  the distance from $u_\nu(a,t_\nu)$ to $u(a,t_\nu)$ is smaller than
  $\varepsilon/4$ for all $\nu \geq \nu_1$ then we finally have
  \begin{multline*}
    \di{u_\nu(s_\nu,t_\nu),u(s_\nu,t_\nu)} \leq
    \di{u_\nu(s_\nu,t_\nu),u_\nu(a,t_\nu)} \\ +
  \di{u_\nu(a,t_\nu),u(a,t_\nu)}
  + \di{u(a,t_\nu),u(s_\nu,t_\nu)} \leq \frac {3}{4} \varepsilon\;.
  \end{multline*}
  This is a contradiction to~\eqref{eq:epscontra} hence proves the
  claim. The last inequality also shows that $u_\nu(\infty) \to
  u(\infty)$ as $\nu$ tends to  $\infty$.
\end{proof}
Given a clean Hamiltonian $H\in C^\infty([0,1] \times M)$. We say that
a map $u:\Sigma^+ \to M$ is an \emph{$(J,H)$-holomorphic} half-strip,
if it satisfies~\eqref{eq:CRJX} with $X=X_H$. 
\begin{lmm}[soft rescaling on ends]\label{lmm:soft}
  Given a clean Hamiltonian $H:[0,1] \times M \to \R$ and a path of
  almost complex structure $J:[0,1] \to \End(TM,\omega)$.  Suppose
  that a sequence $(u_\nu)$ of $(J,H)$-holomorphic half-strips
  converges to $u:\Sigma^+\to M$ in $C^\infty_\loc$ such that the
  limit
  \[m = \lim_{s \to \infty} \lim_{\nu \to \infty}
  E(u_\nu;[s,\infty)\times [0,1]) >0\;,\] exists and is positive. Then
  there exists a subsequence of $(u_\nu)$, still denoted $(u_\nu)$, a
  sequence $(b_\nu)\subset \R$ with $b_\nu \to \infty$, a
  $(J,H)$-holomorphic strip $v:\Sigma \to M$ with boundary in
  $(L_0,L_1)$ and a finite set $Z = \{z_1,\dots,z_\ell\} \subset
  \Sigma$ such that
  \begin{enumerate}[label=$(\roman*)$]
  \item\label{nm:vexists} The rescaled sequence $v_\nu := u_{\nu}\circ
    \tau_{b_\nu}$ converges to $v$ in $C^\infty_\loc(\Sigma \setminus
    Z)$,
  \item\label{nm:vbubble} the limit
    \[ m_j := m(z_j) := \lim_{\varepsilon \to 0} \lim_{k \to \infty}
    E(v_k;B_\varepsilon(z_j) \cap \Sigma)\;,\] exists and
    is the energy of a non-constant holomorphic sphere or disk,
  \item\label{nm:vconst} if $v$ is constant then $Z \neq \emptyset$,
  \item\label{nm:venergy} the limits
    \begin{align*}
      m_0 &:= m(-\infty) := \lim_{s \to \infty} \lim_{\nu \to \infty} E(v_\nu;(-b_\nu,-s)\times [0,1]) \\
      m_{\ell+1}&:=m(\infty) := \lim_{s \to \infty}\lim_{\nu \to
        \infty} E(v_\nu;(s,\infty)\times [0,1])\;,
    \end{align*}
    exists and we have
    \[\lim_{\nu \to \infty} E(u_\nu)=E(v) + \sum_{j=0}^{\ell +1} m_j\;,\qquad m = E(v) + \sum_{j=1}^{\ell+1} m_j\;.\]
  \item\label{nm:vconnect} $u(\infty) = v(-\infty)$ .
  \end{enumerate}
\end{lmm}
\begin{proof}\setcounter{stp}{0}
  This is the adaption of~\cite[Theorem 3.5]{Frauenfelder:disks}
  and~\cite[Lemma 3.6]{Frauenfelder:disks} to the setting of
  strips. Essentially all arguments work analogous provided with the
  energy decay (\cf Lemma~\ref{lmm:Edecay}).
  \begin{stp}
    We claim that $m \geq \hbar$, where $\hbar$ is smaller than the
    constant from Proposition~\ref{prp:hbarstrip} and the constant
    from Proposition~\ref{prp:hbardisksphere} for $\J = \{J_{s,t} \mid
    (s,t) \in [-s_0,s_0]\times [0,1]\}$.
  \end{stp}
  \noindent After transforming $u_\nu$ we assume that $H=0$ (see
  Lemma~\ref{lmm:change}) and $L_0$, $L_1$ are in clean and compact
  intersection. Let $U_\Poz \subset M$ denote the neighborhood of
  $L_0 \cap L_1$ given by Proposition~\ref{prp:poz}.  We claim there
  there exists a sequence $(s_\nu,t_\nu) \in \Sigma$ such that
  \begin{equation}
    \label{eq:unuoutpoz}
    \lim_{\nu \to \infty} s_\nu = \infty, \qquad \forall\  \nu \geq 1 \ : \ u_\nu(s_\nu,t_\nu) \in M \setminus U_\Poz\;.
  \end{equation}
  Otherwise we find $s_1 \geq 0$ such that
  $u_\nu(s,t) \in U_\Poz$ for all $s \geq s_1$, $t \in [0,1]$ and $\nu
  \geq 1$. Then by Lemma~\ref{lmm:Edecay} there exists a constant
  $\delta >0$ such that
  \[E(u_\nu;(s,\infty)) \leq E(u_\nu;(s_1,\infty))
  e^{-2\delta(s-s_1)}\;,\] for all $ s \geq s_1$ and $\nu \geq 1$,
  which implies
  \[m = \lim_{s \to \infty} \lim_{\nu \to \infty} E(u_\nu;(s,\infty))
  \leq \sup_\nu E(u_\nu) \lim_{s\to \infty} e^{-2\delta(s-s_1)} = 0\;.\] This
  contradicts the fact that $m >0$ and shows the existence of the
  sequence $s_\nu$ satisfying~\eqref{eq:unuoutpoz}. Now we claim that
  \begin{equation}
    \label{eq:inf}
    \liminf_{\nu \to \infty } \sup_{t \in [0,1]}  \nm{\d u_\nu (s_\nu,t)} >0\;.
  \end{equation}
  If not, then we find a subsequence $\nu_k$ such that
  $u_{\nu_k}(s_{\nu_k},\cdot)$ converges to a constant arc in $L_0
  \cap L_1$ and hence $u_{\nu_k}(s_{\nu_k},t_{\nu_k})\in U_\Poz$ for
  $k$ large enough, contradicting~\eqref{eq:unuoutpoz}. This
  shows~\eqref{eq:inf}. Define the sequences $c_\nu \in \R$ and
  $t'_\nu \in [0,1]$ by
  \[c_\nu = \nm{\d u_\nu(s_\nu,t'_\nu)} = \sup_t \nolimits \nm{\d
    u_\nu(s_\nu,t)}\;.\] We distinguish two cases. If $c_\nu$ is
  unbounded, then after passing to a subsequence (still denoted $\nu$)
  we assume that $c_\nu \to \infty$ and $t'_\nu \to t'_\infty \in
  [0,1]$. By \cite[Lemma 4.6.5]{Bibel} we have that for every
  $\varepsilon>0$ sufficiently small
  \[\hbar \leq \liminf_{\nu \to \infty}
  E(u_\nu;B_\varepsilon(s_\nu,t'_\infty) \cap \Sigma) \leq \lim_{\nu
    \to \infty} E(u_\nu;(s,\infty))\;, \]
  for all $s \geq 0$ and taking the limit
  \[\hbar \leq \lim_{s \to \infty} \lim_{\nu \to \infty}
  E(u_\nu;(s,\infty)) = m\;.\] This shows the claim in the case when
  $c_\nu$ is unbounded. Now assume that $c_\nu$ is bounded.  By
  Lemma~\ref{lmm:convbub} the rescaled sequence $u_\nu\circ
  \tau_{s_\nu}$ converges to a $J$-holomorphic strip $v':\Sigma \to M$
  in $C^\infty_\loc(\Sigma\setminus Z')$ for some finite set $Z'
  \subset \Sigma$ and we have
  \begin{equation}
    \label{eq:energyv}
    E(v';(-s,\infty)) \leq \lim_{\nu \to \infty}
    E(u_\nu;(s_\nu-s,\infty)) \leq \lim_{\nu \to \infty}
    E(u_\nu;(s,\infty))\;,
  \end{equation}
  for all $s \geq 0$. Since $c_\nu$ is bounded we must have $Z' \cap
  \{0\}\times [0,1] = \emptyset$ and thus $C^\infty$-convergence of
  $u_\nu\circ \tau_{s_\nu} \to v'$ on $\{0\}\times [0,1]$. We assume
  without loss of generality that $t'_\nu \to
  t'_\infty$. By~\eqref{eq:inf}
  \[\nm{\d v(0,t'_\infty)} = \lim_{\nu \to \infty} \nm{\d u_\nu(s_\nu,t'_\nu)}  >0\;.\]
  Hence $v'$ is non-constant. Proposition~\ref{prp:hbarstrip} and
  equation~\eqref{eq:energyv} imply
  \[
  \hbar \leq \lim_{s \to \infty} E(v';(-s,\infty)) \leq \lim_{s \to
    \infty} \lim_{\nu \to \infty} E(u_\nu;(s,\infty)) = m\;.\]
  This shows the claim.
  \begin{stp}\label{stp:anu}
    There exists a sequence $a_\nu \to \infty$ such that 
    \[\lim_{\nu \to \infty} E(u_\nu,(a_\nu-s,\infty) \times [0,1]) = m\;,\]
    for all $s \geq 0$.
  \end{stp}
  \noindent Given $a<b$ we abbreviate 
  \[E_\nu(a) = E(u_\nu;(a,\infty)\times [0,1]),\qquad E_\nu(a,b) = E(u_\nu;(a,b)\times [0,1])\;.\]
  For $\ell \in \mathbb{N}$ we find
  $a_\ell >\ell$ and $\nu_\ell$ such that $\nm{E_\nu(a_\ell) - m} \leq  1/ \ell$
  for all $\nu \geq \nu_\ell$. Without loss of generality we assume
  that $\nu_\ell < \nu_{\ell+1}$ and define $a_\nu = a_\ell$ if
  $\nu_\ell\leq \nu <\nu_{\ell+1}$. This shows that 
  \begin{equation}
    \label{eq:anu}
    \lim_{\nu \to \infty} E(u_\nu;(a_\nu,\infty) \times [0,1]) = m,\qquad \lim_{\nu
    \to \infty} a_\nu = \lim_{\ell \to \infty} a_\ell = \infty\;.
  \end{equation}
  Let $\varepsilon >0$ there exists $s_0,\nu_0$ such that
  $E_\nu(s_0) \leq m+\varepsilon$ for all $\nu \geq \nu_0$, by
  definition of $m$. Secondly given any $s \geq 0$, we find
  $\nu_1 \geq \nu_0$ such that $a_\nu -s >s_0$ for all $\nu \geq
  \nu_1$ and hence 
  \[E_\nu(a_\nu) \leq E_\nu(a_\nu -s ) \leq E_\nu(s_0) \leq m +
  \varepsilon\;.\] for any $\nu \geq \nu_1$. Taking the limit of that inequality as $\nu$
  tends to $\infty$ and then as $\varepsilon \to 0$ we have with~\eqref{eq:anu}
  \[m = \lim_{\nu \to \infty} E_\nu(a_\nu) \leq \lim_{\nu \to \infty} E_\nu(a_\nu-s) \leq m \;.\]
  This shows the claim.
  \begin{stp}
    There exists $\nu_0$ and a sequence $b_\nu \to \infty$ such that 
    \begin{equation}
      \label{eq:bnu}
      E(u_\nu;(b_\nu,\infty)\times [0,1]) = m -\hbar/2\;,
    \end{equation}
    for all $\nu \geq \nu_0$ and we have 
    \begin{equation}
      \label{eq:bnulimit}
      \lim_{s \to \infty}
    \lim_{\nu \to \infty}E(u_\nu;(b_\nu -s,\infty)\times [0,1]) = m\;.
    \end{equation}
  \end{stp}
  \noindent  By definition of $m$ and since
  $\sup_\nu E(u_\nu;\Sigma_0^\infty)$ is finite there exists $\nu_0$ such
  that
  \[ E_\nu(0) \geq m,\qquad \qquad \lim_{s \to \infty} E_\nu(s)
  =0\;,\] for all $\nu \geq \nu_0$. Due to the first step $m
  \geq \hbar$. By the intermediate value theorem there exists $b_\nu
  \geq 0 $ such that
  \[ E_\nu(b_\nu) = m - \hbar /2\;,\] for all $\nu \geq \nu_0$. Since for
  every bounded sequence $\sup_\nu s_\nu \leq c$ it holds
  \[\lim_{\nu \to \infty}  E_\nu(s_\nu) \geq \lim_{\nu \to \infty} E_\nu(c) \geq \lim_{s \to \infty} \lim_{\nu \to \infty} E_\nu(s)
  =m\;,\]
  we necessarily have $b_\nu \to \infty$. This shows~\eqref{eq:bnu}. A
  similar argument as in step~\ref{stp:anu} shows that for all $s \geq
  0$ we have
  \[m-\hbar/2 = \lim_{\nu \to \infty} E_\nu(b_\nu) \leq \lim_{\nu \to
    \infty} E_\nu(b_\nu-s) \leq m\;.\] By contradiction, assume
  that~\eqref{eq:bnulimit} is false. Then we find $0<\rho\leq \hbar/2$
  such that
  \begin{equation}
    \label{eq:bnolimit}
    \lim_{\nu \to \infty} E_\nu(b_\nu-s) \leq m -\rho\;,
  \end{equation}
  for every $s \geq 0$. We claim that this implies 
  \begin{equation}
    \label{eq:abinfty}
    \lim_{\nu \to \infty} b_\nu - a_\nu =\infty\;.
  \end{equation}
  Arguing indirectly, we assume that there exists $s_0\geq 0$ such that $b_\nu-a_\nu \leq s_0$ for all
  $\nu \geq 1$. This leads to the following contradiction 
  \[m = \lim_{\nu \to \infty} E_\nu(a_\nu) \leq \lim_{\nu \to \infty}
  E_\nu(b_\nu-s_0) \leq m -\rho < m\;.\] So~\eqref{eq:abinfty} is true and thus
  for any $s \geq 0$ we have
  \begin{multline*}
    \lim_{\nu \to \infty} E_\nu(a_\nu-s,a_\nu+s) \leq \lim_{\nu \to
      \infty} E_\nu(a_\nu-s,b_\nu) \\
    = \lim_{\nu \to \infty} E_\nu(a_\nu-s) - \lim_{\nu \to
      \infty}E_\nu(b_\nu) = m - m +\hbar/2=\hbar/2\;.
  \end{multline*}
  By Lemma~\ref{lmm:convbub} the rescaled strip $w_\nu = u_\nu \circ
  \tau_{a_\nu}$ converges modulo bubbling to a $(J,H)$-holomorphic
  strip $w$ and we have
  \[\hbar/2 \geq \lim_{\nu \to \infty} E(w_\nu;(-s,s)\times [0,1]) =
  E(w;(-s,s)\times [0,1]) + m'\;.\]
  But if $w$ were non-constant or $m'>0$ then right-hand side is
  larger than $\hbar$. This shows that $w$ must be constant and $w_\nu
  \to w$ in $C^\infty_\loc(\Sigma)$. In particular
  \begin{equation}
    \label{eq:ass}
    \lim_{\nu \to \infty} E_\nu(a_\nu-s,a_\nu+s) = 0\;,
  \end{equation}
  for all $s \geq 0$. We also have
  \[\lim_{\nu \to \infty} E_\nu(a_\nu,b_\nu) = \lim_{\nu \to \infty} E_\nu(a_\nu) - \lim_{\nu \to \infty} E_\nu(b_\nu) = \hbar/2\;.\]
  By possibly making $\hbar/2$ smaller, we assume that
  $E_\nu(a_\nu,b_\nu) < \varepsilon_0$ for $\nu$ large enough, where
  $\varepsilon_0$ is given by Lemma~\ref{lmm:Edecay}. Hence there
  exists constants $\nu_0,\delta>0$ such that
  \[E_\nu(a_\nu+s,b_\nu-s) \leq E_\nu(a_\nu,b_\nu) e^{-\delta s} \leq \hbar/2 e^{-\delta s}\;, \]
  for all $s \geq 1$ and $\nu \geq \nu_0$. This shows 
  \begin{equation}
    \label{eq:abs}
    \lim_{s \to \infty} \lim_{\nu \to \infty} E_\nu(a_\nu +s,b_\nu -s) \leq \hbar/2 \lim_{s \to \infty}  e^{-\delta s} = 0\;.
  \end{equation}
  Now for $s \geq 0$ we have
  \[E_\nu(b_\nu-s) = E_\nu(a_\nu-s) - E_\nu(a_\nu-s,a_\nu+s) - E_\nu(a_\nu+s,b_\nu-s)\;.\]
  Combining~\eqref{eq:anu}, \eqref{eq:abs} and~\eqref{eq:ass} this shows that 
  \[\lim_{s \to \infty} \lim_{\nu \to \infty} E_\nu(b_\nu -s) = m\;,\]
  contradicting~\eqref{eq:bnolimit} and proving~\eqref{eq:bnulimit}.
  \begin{stp}
    We show points~\ref{nm:vexists}, \ref{nm:vbubble}
    and~\ref{nm:vconst} of the theorem.
  \end{stp}
  \noindent With $b_\nu$ given in~\eqref{eq:bnu} from last step we
  define $v_\nu:= u_\nu\circ \tau_{b_\nu}$. The existence of the strip
  $v$, the set $Z$ and the limits $m_j$ is provided by
  Lemma~\ref{lmm:convbub}. This shows~\ref{nm:vexists}
  and~\ref{nm:vbubble}. We show~\ref{nm:vconst}. With the following
  argument we even locate the bubbling point. By the definition of
  $b_\nu$ there exists a constant $\nu_0$ such that
  \[E_\nu(b_\nu-s_1,b_\nu -s_0) \leq E_\nu(b_\nu-s_1) - E_\nu(b_\nu)
  \leq m - (m-\hbar/2) = \hbar/2\;,\] for all $\nu \geq \nu_0$ and
  $0<s_0<s_1$. The same argument leading to~\eqref{eq:ass} shows that no bubbling can occur on
  $(b_\nu-s_1,b_\nu-s_0)\times [0,1]$ and provided that $v$ is constant we get
  \[\lim_{\nu \to \infty} E(v_\nu;(-s_1,-s_0)\times [0,1]) = E(v;(-s_1,-s_0)\times [0,1]) = 0\;.\] This
  shows that 
  \begin{multline*}
    \lim_{\nu \to \infty} E(v_\nu;(-s_1,\infty)\times [0,1]) = \lim_{\nu \to \infty}
    E(v_\nu;(-s_0,\infty)\times [0,1]) \\
    + \lim_{\nu \to \infty} E(v_\nu;(-s_1,-s_0)\times [0,1])
  \end{multline*}
  is independent of $s_1$. With~\eqref{eq:bnulimit} we have
  \[\lim_{\nu \to \infty} E(v_\nu;(-s_1,\infty)\times
  [0,1]) = \lim_{s \to \infty}\lim_{\nu \to \infty} E_\nu(b_\nu - s) =
  m\;,\] for all $s_1 >0$. But on the other hand 
  \[\lim_{\nu \to \infty} E(v_\nu;(0,\infty)\times [0,1]) = \lim_{\nu \to \infty} E_\nu(b_\nu) = m - \hbar/2\;.\]
  This implies that there must be a bubbling point on $\{0\}\times
  [0,1]$ for $v_\nu$. 
  \begin{stp}
    We show~\ref{nm:venergy}.
  \end{stp}
  \noindent Let $s_0$ be so large that $Z \subset
  \Sigma_{-s_0}^{s_0}$.  By possible passing to a further subsequence
  still denoted by $(v_\nu)$, we assume that
  \[\rho(s_0) := \lim_{\nu \to \infty} E(v_\nu;(s_0,\infty)\times [0,1]) \;,\]
  is well-defined. Then for any $s\geq s_0$, we have
  $C^\infty$-convergence of $v_\nu$ to $v$ on $\Sigma_{s_0}^s$ and thus
  \[ \rho(s) := \rho(s_0) - \lim_{\nu \to \infty} E(v_\nu;
  (s_0,s)\times [0,1]) = \rho(s_0) - E(v;(s_0,s)\times [0,1])\;.\]
  This shows that $\rho(s)$ is well-defined and monotone
  decreasing. Hence the limit
  \[m_{\ell+1} = \lim_{s \to \infty} \lim_{\nu \to \infty}
  E(v_\nu;(s,\infty)\times [0,1]) = \lim_{s \to \infty }\rho(s) \;,\]
  exists and moreover
  \begin{equation}
    \label{eq:mell}
    \rho(s_0) = m_{\ell+1} + E(v;(s_0,\infty)\times [0,1])\;.
  \end{equation}
  Secondly by definition of $v_\nu$ and after assumption the limit
  \begin{align*}
    m_0 &= \lim_{s \to \infty} \lim_{\nu \to \infty}
    E(v_\nu;(-b_\nu,-s)\times [0,1]) = \lim_{s \to \infty} \lim_{\nu
      \to \infty} E(u_\nu;(0,b_\nu -s)\times [0,1]) \\ &= \lim_{s \to
      \infty } E(u;(0,s)\times [0,1]) = E(u)\;,
  \end{align*}    
  exists. Now by~\eqref{eq:bnulimit} and Lemma~\ref{lmm:convbub} we have 
  \begin{align*}
    m &= \lim_{s \to \infty} \lim_{\nu \to \infty} E(u_\nu;(b_\nu-s,\infty)\times [0,1])\\
    &= \lim_{s \to \infty} \lim_{\nu \to \infty} E(v_\nu;(-s,\infty)\times [0,1])\\
    &= \lim_{s \to \infty} \lim_{\nu \to \infty}
    E(v_\nu;(-s,-s_0)\times [0,1]) + \lim_{\nu \to
      \infty} E(v_\nu;(-s_0,s_0)\times [0,1]) +\rho(s_0)\\
    &= E(v)+\sum_{j=1}^{\ell+1} m_j\;,
  \end{align*}
  where in the last step we used~\eqref{eq:mell}. Finally using the last two equations
  \begin{align*}
    \lim_{\nu \to \infty}E(u_\nu) &= \lim_{s \to \infty} \lim_{\nu \to
    \infty} E(u_\nu;[0,s)\times [0,1]) + \lim_{s \to \infty} \lim_{\nu
    \to \infty} E(u_\nu;(s,\infty)\times [0,1])\\ & = E(u) + m = E(v) + \sum_{j=0}^{\ell+1} m_j\;.
  \end{align*}
  \begin{stp}
    We show~\ref{nm:vconnect}.
  \end{stp}
  \noindent Let $s \geq 0$ be large enough. We have by assumption
  \[ \lim_{\nu \to \infty} E(u_\nu;(s,\infty) \times [0,1]) = m + E(u;(s,\infty)\times [0,1])\;,\]
  and after~\ref{nm:venergy}
  \begin{multline*}
    \lim_{\nu \to \infty} E(u_\nu;(b_\nu-s,\infty)\times [0,1]) = \lim_{\nu \to \infty} E(v_\nu;(-s,\infty)\times [0,1])\\
    = E(v;(-s,\infty)\times [0,1]) + \sum_{j=1}^{\ell+1}m_j
    =m - E(v;(-\infty,-s)\times [0,1])\;.
  \end{multline*}
  Subtracting these two identities gives
  \[\lim_{\nu \to \infty} E(u_\nu;(s,b_\nu-s)\times[0,1])=
  E(u;(s,\infty)\times [0,1]) + E(v;(-\infty,-s)\times [0,1])\;.\] If $s$ tends to
  $\infty$ the right-hand side approaches zero. Hence there exists
  constants $s_0$ and $\nu_0$ such that 
  \[ E(u_\nu;(s,b_\nu-s)\times [0,1]) \leq \varepsilon_0\;,\] for all
  $s \geq s_0$ and $\nu \geq \nu_0$, where $\varepsilon_0$ is the
  constant from Lemma~\ref{lmm:Edecay}. Using this proposition
  with $a=s_0$, $b=b_\nu-s_0$, $\sigma=b_\nu-s$ and $\sigma' = s$
  we see that there exists a constant $c_1$ such that 
  \[
  \di{u_\nu(s,0),u_\nu(b_\nu-s,0)} \leq c_1 e^{-\delta (s-s_0)}\;,
  \]
  for all $s \geq s_0+1$. Now estimate using the triangle inequality
  \begin{multline*}
    \di{u(\infty),v(-\infty)} \leq  \di{u(\infty),u(s,0)} + \di{u(s,0),v(-s,0)} \\
    + \di{v(-s,0),v(-\infty)}  \;,    
  \end{multline*}
  and
  \begin{multline*}
    \di{u(s,0),v(-s,0)} \leq \di{u(s,0),u_\nu(s,0)} + \di{u_\nu(s,0),u_\nu(b_\nu-s,0)} \\
    + \di{v_\nu(-s,0),v(-s,0)}\;.
  \end{multline*}
  Using theorem~\ref{thm:remove} there exists a constant $c_2$ such that 
  \[\di{u(\infty),u(s,0)} + \di{v(-s,0),v(-\infty)} \leq c_2 e^{-\delta s}\;,\]
  for all $s \geq s_0$. Given any $\varepsilon>0$ choose $s\geq s_0+1$ such that 
  \[c_2e^{-\delta s}+ c_1e^{-\delta(s-s_0)} \leq \varepsilon/2\;,\]
  then choose $\nu$ such that 
  \[\di{u(s,0),u_\nu(s,0)} + \di{v_\nu(-s,0),v(-s,0)} \leq \varepsilon/2\;.\]
  That is possible because for $s$ sufficiently large and fixed,
  $u_\nu(s,0)$ converges to $u(s,0)$ and $v_\nu(-s,0)$ converges to
  $v(-s,0)$ as $\nu$ tends to $\infty$. Combining the last six
  estimates shows that the distance from $u(\infty)$ to $v(-\infty)$ is lesser than $\varepsilon$ and
  hence the claim.
\end{proof}
\subsection{Minimal energy}
We establish lower bounds on the energy. We denote by $\J_\adm$ the
space of admissible almost complex structures and $\X_\adm$ the space
of admissible vector fields.
\begin{prp}\label{prp:hbarstrip}
  Given path of almost complex structures $J:[0,1]\to \End(TM,\omega)$
  and a Hamiltonian $H\in C^\infty([0,1] \times M)$ such that
  $\vp_H(L_0)$ and $L_1$ are in clean intersection. There exists a
  positive constant $\hbar>0$ such that for every non-constant
  $(J,H)$-holomorphic strip $u:\Sigma \to M$ with boundary in
  $(L_0,L_1)$ we have $E(u) \geq \hbar$.
\end{prp}
\begin{proof}
  See \cite[Prop. 4.1.4.]{Bibel} for the analogous proposition for
  holomorphic spheres or disks. Note that we can not apply the proof technique from there
  directly because there is no mean-value inequality of large
  radius. We have to argue indirectly. After a transformation we assume
  without loss of generality that $H=0$ (see
  Lemma~\ref{lmm:change}). Assume by contradiction that there exists a
  sequence $u_\nu$ of non-constant $J$-holomorphic strips such that
  \begin{equation}
    \label{eq:unoncnst}
    0 < E(u_\nu),\qquad \qquad  \lim_{\nu \to \infty}E(u_\nu)= 0\;.
  \end{equation}
  Let $U_\Poz$ denote the
  Po\'zniak neighborhood of $L_0 \cap L_1$ given by
  Proposition~\ref{prp:poz}. We claim that there exists $\nu_0$ such
  that $u_\nu(s,t) \in U_\Poz$ for all $(s,t) \in \Sigma$ and $\nu \geq
  \nu_0$. To show that we assume by contradiction that there exists a
  sequence $(s_\nu,t_\nu) \in \Sigma$ such that
  \begin{equation}
    \label{eq:uoutpoz}
    u_\nu(s_\nu,t_\nu) \in M \setminus U_\Poz\;,
  \end{equation}
  for all $\nu \geq 1$. But since $E(u_\nu)\to 0$ there exists a
  subsequence such that $u_{\nu_k}(s_{\nu_k},t_{\nu_k})$ converges to a
  point $x \in L_0 \cap L_1$ as $k$ tends to $\infty$. This contradicts~\eqref{eq:uoutpoz} and
  we have proven that $u_\nu(s,t)\in U_\Poz$ for all $(s,t)\in \Sigma$
  and $\nu \geq \nu_0$. Now inside $U_\Poz$ the symplectic form 
  $\omega = \d \lambda$ is exact with $\Res{\lambda}{TL_k} =0$ for
  $k=0,1$. We have 
  \[E(u_\nu) = \int_\Sigma \nm{\d u_\nu}^2 = \int_\Sigma u_\nu^* \omega =
  \int_{\partial \Sigma} u_\nu^*\lambda = 0\;.\]
  This shows that $E(u_\nu)=0$ for all $\nu \geq \nu_0$, which
  contradicts~\eqref{eq:unoncnst}.
\end{proof}
\begin{prp}\label{prp:hbardisksphere}
  Let  $\J \subset \End(TM,\omega)$ be a compact subset of almost
  complex structures.  There exists a positive constant $\hbar >0$
  such that $\int u^*\omega \geq \hbar$ for any non-constant
  $J$-holomorphic sphere $u:S^2 \to M$ or non-constant $J$-holomorphic
  disk $u:(D^2,\partial D^2) \to (M,L_k)$ with $k=0,1$ and $J \in \J$.
\end{prp}
\begin{proof}
  For every $J \in \J$ let $\hbar(J)$ be the minimal energy of a
  non-constant $J$-holomorphic sphere. For $k=0,1$ let $\hbar_k(J)$ be
  the minimal energy of a non-constant $J$-holomorphic disk
  $u:(D,\partial D^2) \to (M,L_k)$. In \cite[Prop 4.1.4]{Bibel} we see
  that the maps $J \mapsto \hbar(J)$ and $J \mapsto \hbar_k(J)$ are
  lower semi-continuous and everywhere positive. Let $\hbar$ be
  smaller than their minimum which is positive since $\J$ is compact.
\end{proof}
\subsection{Action, energy and index estimates}
We denote by $\J_\adm$ and $\X_\adm$ the space of admissible almost
complex structures and admissible vector fields respectively (\cf
Definition~\ref{dfn:JXadm}).
\begin{lmm}[action-energy estimate]\label{lmm:AErel}
  Given $J \in \J_\adm$ and $X \in \X_\adm$, there exists a constant
  $c>0$ such that for any finite energy $(J,X)$-holomorphic strip $u$
  with boundary in $(L_0,L_1)$ we have
  \[
  \frac 12 E(u) - c\leq \int u^*\omega
  \leq \frac 32 E(u) + c\;.
  \]
\end{lmm}
\begin{proof}
  Fix a $(J,X)$-holomorphic strip $u$ with finite energy.  We denote
  the asymptotic points $u(-\infty)=x_-$ and $u(\infty)=x_+$ and
  estimate
  \[
  |\omega(\ps u,X)| = |\<\ps u,J X\>_J| \leq \frac 12 \nm{\ps u}^2_J +
  \frac 12 \nm{X}^2_J\;.
  \]
  By definition of an admissible vector field we
  have $X(\pm s,\cdot) = X_{H_\pm}$ for all $s \geq s_0$. This shows
  \begin{align*}
    &\int_\Sigma \omega(X,\ps u) \d s \d t =\\
    &= \int_{\Sigma_{-\infty}^{-s_0}} \ps H_-(u) \d
    s \d t + \int_{\Sigma_{-s_0}^{s_0}} \omega(X,\ps u)\d s \d t + \int_{\Sigma_{s_0}^\infty} \ps H_+(u) \d s \d t\\
    &= \int_0^1 H_-(u(-s_0,t)) - H_-(x_-(t)) \d t   + \int_{\Sigma_{-s_0}^{s_0}} \omega(X,\ps u)\d s \d t + \\
    &\qquad  +\int_0^1 H_+(x_+(t)) -H_+(u(s_0,t)) \d t\;.
  \end{align*}
  With the last estimate we see
  \begin{multline*}
   \frac 12 \int_\Sigma \nm{\ps u}^2_J
  + \sup H_- - \inf H_- + \sup H_+ - \inf H_- + s_0 \Nm{X}_\infty^2\geq \int_\Sigma \omega(X,\ps u)  \\
   \geq -\frac 12 \int_\Sigma \nm{\ps u}^2_J - \sup H_- + \inf H_- - \sup H_+ + \inf H_+ - s_0\Nm{X}_\infty^2\;.
  \end{multline*}
  This shows the claim using $|\ps u|^2_J = \omega(\ps u,\pt u) +
  \omega(X,\ps u)$.
\end{proof}
Given an admissible vector field $X \in \X_\adm$ such that $X(\pm
s,\cdot) = X_{H_\pm}$ for all $s \geq s_0$ and an almost complex
structure $J \in \J_\adm$. For every  $(J,X)$-holomorphic strip $u$
asymptotic to $x:=\lim_{s \to -\infty} u(s,\cdot)$ and $y:=\lim_{s \to
  \infty} u(s,\cdot)$ we 
define the \emph{action-energy defect}
\begin{equation}
  \label{eq:EAdefect}
  \Delta(u) := E(u)-\int u^*\omega  - \int_0^1 H_+(t,y(t))\d t + \int_0^1 H_-(t,x(t))\d t\,.
\end{equation}
The quantity is called action-energy defect, because if $X$ is
$\R$-invariant then $\Delta(u)$ vanishes and the equation above is the
action-energy relation (\cf equation~\eqref{eq:AErel}).  The next
lemma states that the defect is continuous under Gromov converge.
\begin{lmm}~\label{lmm:defectconv} Given sequences $(J^\nu) \subset
  \J_\adm$ and $(X^\nu)\subset X_\adm$ converging in
  $C^\infty$-topology to $J$ and $X$ respectively. Let $(u^\nu)$ be a
  sequence of $(J^\nu,X^\nu)$-holomorphic strips with boundary in
  $(L_0,L_1)$. Fix a finite subset $Z \subset \Sigma$ and assume that
  $(u^\nu)$ converges in $C^\infty_\loc(\Sigma \setminus Z)$ to the
  $(J,X)$-holomorphic map $u$, then we have $\lim_{\nu \to \infty}
  \Delta(u^\nu) = \Delta(u)$.
\end{lmm}
\begin{proof}
  Let $B_\e(z) \subset \Sigma$ denote the open ball with radius
  $\e>0$. Fix some $\varepsilon>0$ and denote the thickened set
  \[Z_\varepsilon = \bigcup_{z \in Z} B_\varepsilon(z)\;.\] Then after
  convergence of $u^\nu \to u$ and $(J^\nu,X^\nu) \to (J,X)$ in
  $C^\infty(\Sigma_{-s_0}^{s_0}\setminus Z_\varepsilon)$ there exists
  a $\nu_0$ such that for all $\nu \geq \nu_0$
  \[
  \nm{\int_{\Sigma_{-s_0}^{s_0} \setminus Z_\varepsilon} \<\ps
    u^\nu,J^\nu(u^\nu)X^\nu(u^\nu)\> -\<\ps u,J(u)X(u)\> \d s \d t} \leq
  \varepsilon\;,\] and
    \[
    \int_{\{-s_0\} \times [0,1] \setminus Z_\varepsilon}
    \nm{H_-(u^\nu)-H_-(u)} \d t + \int_{\{s_0\} \times [0,1] \setminus
      Z_\varepsilon} \nm{H_+(u^\nu)-H_+(u)} \d t \leq \varepsilon\;.\]
    Moreover using the Cauchy-Schwarz inequality we obtain for any $z
    \in Z$
  \begin{multline*}
    \nm{\int_{B_\varepsilon(z)} \<\ps u^\nu,J^\nu(u^\nu)X^\nu(u^\nu)\>
      \d s \d t }\\\leq \left(\int_{B_\varepsilon(z)}\nm{X^\nu}^2 \d s
      \d t \int_{B_\e(z)}\nm{\ps u^\nu}^2 \d s \d t \right)^{\frac
      12}\leq \sqrt \pi \e \Nm{X^\nu}_{C^0}\sup_\nu \sqrt{E(u^\nu)}
    \;.
  \end{multline*}
  We have similar estimates for $u^\nu$, $X^\nu$ and $J^\nu$ replaced
  by $u$, $X$ and $J$ respectively. A plain $C^0$-estimate gives for
  every $z \in Z$
  \begin{multline*}
    \int_{\{s_0\}\times [0,1] \cap B_\e(z)} \nm{H_+(u^\nu)-H_+(u)} \d
    t + \int_{\{-s_0\} \times [0,1] \cap B_\e(z)}
    \nm{H_-(u^\nu)-H_-(u)}\d t \\\leq 4 \varepsilon
    \left(\Nm{H_-}_{C^0}+\Nm{H_+}_{C^0}\right)\;.
  \end{multline*}
  Putting all together we obtain a
  constant $c$ independent of $\varepsilon$ such that
  \[\nm{\Delta(u^\nu)-\Delta(u)} \leq c \varepsilon\,,\] for all $\nu
  \in \N$ larger than $\nu_0$. This shows the claim.
\end{proof}
\begin{lmm}[action-index relation]\label{lmm:AIrel}
  Assume that the pair $(L_0,L_1)$ is $\tau$-monotone. Given
  Hamiltonians $H_-,H_+$. Fix connected components $C_- \subset
  \I_{H_-}(L_0,L_1)$ and $C_+ \subset \I_{H_+}(L_0,L_1)$. Given two
  maps $u,v:[-1,1]\times [0,1] \to M$ satisfying the boundary
  condition $(u(\cdot,0),u(\cdot,1)),(v(\cdot,0),v(\cdot,1)) \subset
  L_0 \times L_1$ $x:=u(-1,\cdot),x':=v(-1,\cdot) \in C_-$ and
  $y:=u(1,\cdot),y':=v(1,\cdot) \in C_+$ then we have
  \begin{multline*}
    \tau(\mu(u) - \mu(v)) = \int u^*\omega +\int H_+(y)\, \d t- \int H_-(x)\, \d t \\ - \int v^*\omega - \int H_+(y')\, \d t  + \int H_-(x')\, \d t\,. 
  \end{multline*}
\end{lmm}
\begin{proof}
  Let $u_-:[-1,1]\times [0,1]\to M$ and $u_+:[-1,1]\times [0,1] \to M$
  be such that $u_\pm(s,\cdot) \in C_\pm$ for all $s \in [-1,1]$ and
  $u_-(-1,\cdot) = x'$, $u_-(1,\cdot) = x$ as well as
  $u_+(-1,\cdot) = y'$, $u_+(1,\cdot)=y$.  The
  connected sum $u_-\#u\#u_+^\vee\#v^\vee$ defines a map
  $w:[-1,1]\times [0,1]\to M$ with $w(\cdot,k) \subset L_k$ for
  $k=0,1$ and $w(-1,\cdot)=w(1,\cdot)$. By monotonicity we have $\int
  w^*\omega = \tau \mu_\Mas(w)$. The additivity of the Viterbo index
  shows
  \[
  \int u_-^*\omega + \int u^*\omega - \int u_+^*\omega - \int
  v^*\omega =  \tau \big(\mu(u_-) + \mu(u) - \mu(u_+) - \mu(v)\big)\,.
  \]
  By the zero axiom the index $\mu(u_-)=\mu(u_+)=0$. We compute
  \begin{multline*}
    \int u_-^*\omega = \int \omega(\ps u_-,\pt u_-) = \int \omega(\ps
    u_-,X_{H_-}(u_-)) = -\int \ps H_-(u_-)\, \d s \d t =\\ = \int
    H_-(x')\, \d t - \int H_-(x) \, \d t\,.
  \end{multline*}
  Similarly for $\int u_+^*\omega = \int H_+(y') - \int H_+(y)$. We
  conclude by plugging these two equations into the last one.
\end{proof}

\begin{lmm}\label{lmm:bubblemonotone}
  With the same assumptions as Theorem~\ref{thm:comp}. Assume
  additionally that the pair $(L_0,L_1)$ is monotone. Suppose that
  $(u_\nu)_{\nu \in \N}$ Floer-Gromov converges modulo bubbling to
  $(v_1,\dots,v_k)$ then either all $Z_1,\dots,Z_k$ are empty or we
  have for all $\nu \in \N$ large enough
  \begin{equation}
    \label{eq:indexlimit}
 \mu(u_\nu) \geq \sum_{j=1}^k \mu(v_j)  + N\,,    
  \end{equation}
 where $N$ is the minimum of the minimal
  Maslov numbers of $L_0$ and $L_1$.
\end{lmm}
\begin{proof}
  Assume that for some $j=1,\dots,k$ the set $Z_j$ is non-empty.
  After rescaling and removal of singularities we see that $m_{j,z}$
  is the energy of a non-constant holomorphic sphere or disk with
  boundary on $L_0$ or $L_1$. Hence by monotonicity $m_{j,z}/\tau$ is
  a positive multiple of the minimal Malsov number of $L_0$ or $L_1$
  and thus $m/\tau \geq N$.  Abbreviate $x:=v_1(-\infty)$,
  $y:=v_k(\infty)$, $x_\nu:=u_\nu(-\infty)$ and
  $y_\nu:=u_\nu(\infty)$.  Further abbreviate the connected sum
  $v=v_1\#v_2\#\dots\#v_k$.  We have $\mu(v):=\sum_{j=1}^k \mu(v_j)$
  and $E(v):= \sum_{j=1}^k E(v_j)$. Let $j_0 = 1,\dots,k$ be the
  unique index such that $a_{j_0}^\nu=0$ for all $\nu \in \N$ and
  hence $(u_\nu)$ converges to $v_{j_0}$ in
  $C^\infty_\loc(\Sigma\setminus Z_{j_0})$.  By the action-index
  relation (\cf Lemma~\ref{lmm:AIrel}) we have
  \begin{align*}
    &\tau(\mu(u_\nu) - \mu(v)) \\
    &= \int u_\nu^*\omega + \int H_+(y_\nu)\, \d t - \int H_-(x_\nu)\, \d t - \int v^*\omega - \int H_+(y)\, \d t + \int H_- (x)\, \d t\\
    &=E(u_\nu) - \Delta(u_\nu) - E(v) +\Delta(v_{j_0}) \to m\,.
  \end{align*}
  where we have used Lemma~\ref{lmm:defectconv}.
\end{proof}
\subsection{Convexity}
\begin{thm}
  Let $(\rho,J_0)$ be some convex structure on
  $(M,\omega,L_0,L_1)$. There exists a constant $c$ such that the
  image of a $(J,X)$-holomorphic map $u:\Sigma \to M$ is contained in
  $M^c = \{ p \in M \ | \ \rho(p) \leq c\}$.
\end{thm}
\begin{proof} \cite[Theorem 3.9]{Frauenfelder:PhD} Set $m = \rho \circ
  u$. We have that $\Delta m(s,t) \geq 0$ for all points $(s,t) \in
  \Sigma$ such that $m(s,t) > c$. Suppose by contradiction that
  $m^{-1}((c,\infty))$ is non-empty and let $(s_0,t_0)$ be the point
  where $m$ attains its maximum. After the maximum principle this point must lie on the
  boundary. Without loss of generality assume that $t_0=0$. Let $r>0$
  be small enough such that for all $s \in (s_0 - r,s_0+r)$ we have
  $m(s,0) > c$. Extend $m$ to be defined on $B_r(s_0)$ via
  \[m(s,-t) = m(s,t)\;.\]
  For $m$ to be in $C^2$ it remains to show that $\pt m(s,0)=0$. We
  compute
  \[\pt m(s,0) = \d \rho(u(s,0)) \pt u(s,0) = \<\grad \rho(s,0),J_0\ps
  u(s,0)\>\;,\]
  where we used the fact that $X$ vanishes outside $M^c$. Now
   $\grad \rho(s,0) \in T_{u(s,0)}L_0$ and $\ps u(s,0) \in T_{u(s,0)}
  L_0$. Since $L_0$ is a Lagrangian it shows that $\pt m(s,0)$ vanishes.
  Hence $m$ is of class $C^2$ defined on $B_r(s_0)$ with $\Delta m \geq 0$
  and can not attained its maximum
  at $(s_0,0)$. Contradiction. 
\end{proof}
\section{Fredholm Theory}

We define Banach manifolds and Banach bundles such that the moduli
problem of the perturbed Cauchy-Riemann equation becomes the zero set
of a Fredholm section. This step is part of the standard program in
order to put a smooth structure on the moduli spaces and was pioneered
by Floer in \cite{Floer:Intersection} under the assumption that the
Lagrangians intersect transversely. Frauenfelder constructed the
Banach manifolds for the degenerate case of clean intersecting
Lagrangians in~\cite{Frauenfelder:PhD}. Besides recalling these
well-known constructions we also give a formula for the index in the
degenerate case, which was not done before.
\subsection{Banach manifold}
Given a compact symplectic manifold $(M,\omega)$, two Lagrangian
submanifolds $L_0,L_1 \subset M$ and two clean Hamiltonians $H_-$,
$H_+$ with perturbed intersection points $\I_-$, $\I_+$
respectively (\cf \ref{dfn:MBreg} and~\eqref{eq:IH}). To construct the
Banach manifold we need some auxillary choices. Choose a Riemannian
metric on $M$ and denote by $\e>0$ its injectivity radius.  For two
points $p,q \in M$ which are close enough, we denote by $\Pi_p^q:T_pM
\to T_qM$ the parallel transport along the unique shortest geodesic
joining $p$ to $q$.  More generally if $p,q \in M$ are arbitrary, we
define the linear map
\begin{equation}
  \label{eq:Pihat}
  \widehat \Pi_p^q: T_p M \to T_q M,\qquad
  \widehat \Pi_p^q= \beta(\e^{-1}\di{p,q})\Pi_p^q\,,
\end{equation}
in which $\beta$ is a smooth cut-off function supported in $[0,1]$ and
$\beta \equiv 1$ on $[0,1/2]$. For maps $u, v: \Sigma \to M$ we denote
$\wh \Pi_u^v:u^*TM \to v^*TM$, $(\wh \Pi_u^v)(z) = \wh
\Pi_{u(z)}^{v(z)}$.
\begin{dfn}\label{dfn:B}
  Fix numbers $p>2$ and $\delta > 0$ we define
\[\B^{1,p;\delta}
\subset C^0( \R \times [0,1],M)\,,\] to be the space of maps $u$ such
that
  \begin{enumerate}[label=(\roman*)]
  \item $u$ is of local regularity $H^{1,p}$,
  \item $u$ satisfies the boundary condition $(u(s,0),u(s,1)) \in L_0
    \times L_1$ for all $s \in \R$,
  \item there exists $x_- := u(-\infty) \in \I_-$ and $x_+:=u(\infty)
    \in \I_+$ such that
    \begin{equation}\label{eq:uintegral}
      \begin{aligned}
        \int_{\Sigma_\pm} \left(\di{u,x_\pm}^p+ \nm{\ps u}^p +
          \nmm{\pt u - \wh \Pi_x^u \pt x_\pm}^p\right) e^{\delta p
          \nm{s}} \d s\d t < \infty \,,
      \end{aligned}
    \end{equation}
    with $\Sigma_-=(-\infty,0]\times [0,1]$ and
    $\Sigma_+=[0,\infty)\times [0,1]$.
  \end{enumerate}
  For two subspaces $C_- \subset \I_-$ and $C_+ \subset \I_+$ we denote by
  $\B^{1,p;\delta}(C_-,C_+) \subset \B^{1,p;\delta}$ the subspace of
  all $u$ such that $u(-\infty) \in C_-$ and $u(\infty) \in C_+$.
\end{dfn}
\begin{rmk}
  Since any two metrics on the compact manifold $M$ are equivalent,
  the space $\B^{1,p;\delta}$ does not depend on the specific choice
  of the metric.  For the construction of the charts we employ a
  domain dependent metric, which is explained below but for the mere
  definition of the space it suffices to consider a simple metric on
  $M$.
\end{rmk}
We now are going to construct the tangent space of $\B$. Let $\na$
denote the Levi-Civita connection associated to the axillary
metric.
\begin{dfn}\label{dfn:TB}
  For any $u \in \B^{1,p;\delta}$ we define $T_u\B^{1,p;\delta}$ to be
  the space of sections $\xi$ of $u^*TM$ such that
  \begin{enumerate}[label=(\roman*)]
  \item $\xi$ is of local regularity $H^{1,p}$,
  \item $\xi$ satisfies the linearized boundary condition 
    \begin{equation*}
      \xi(s,0) \in  T_{u(s,0)}L_0, \qquad \xi(s,1) \in T_{u(s,1)}L_1\;,
    \end{equation*}
  \item there exists vector fields $\xi_- \in T_{x_-} \I_-$ and $\xi_+
    \in T_{x_+} \I_+$ such that the following norm is finite
    \begin{equation*}
      \begin{aligned}
      \Nm{\xi}_{1,p;\delta} &:= \Big(\Nm{\xi_-}^p_{L^\infty} +
      \Nm{\xi_+}^p_{L^\infty}+\\ & \qquad + \int_{\Sigma_-}
      \left(\nmm{\xi-\widehat \Pi_{x_-}^{u} \xi_-}^p + \nmm{\na
          \big(\xi -\widehat\Pi_{x_-}^u
          \xi_-\big)}^p\right)e^{\delta p \nm{s}}\d s \d t \\
      &\qquad + \int_{\Sigma_+} \left(\nmm{\xi-\widehat \Pi_{x_+}^u
            \xi_+}^p + \nmm{\na \big(\xi -\widehat\Pi_{x_+}^u
            \xi_+\big)}^p\right)e^{\delta p \nm{s}}\d s \d
        t\Big)^{1/p}.  
      \end{aligned}     
    \end{equation*}
  \end{enumerate}
  Furthermore, we define $\E^{p;\delta}_u$ to be the space of all sections
  $\eta \in \Gamma(u^*TM)$ of local regularity $L^p$ which are bounded
  in the norm
  \begin{equation*}    
    \Nm{\eta}_{p;\delta} := \left(\int_\Sigma \nm{\eta(s,t)}^p e^{\delta p \nm{s}}\d s \d t\right)^{1/p}\,.
  \end{equation*}
  If $u$ is smooth we define the spaces $T_u \B^{1,p;\delta}$ and
  $\E_u^{p;\delta}$ for any constants $p>1$ and $\delta \in \R$.
\end{dfn}
\begin{rmk}
  Since $M$ is compact the norms $\Nm{\,\cdot\,}_{1,p;\delta}$ and
  $\Nm{\,\cdot\,}_{p;\delta}$ with respect to two different
  connections and metrics are equivalent. Hence $T_u\B^{1,p;\delta}$
  is well-defined independently of the choice of the
  connection. Similarly $\E^{p;\delta}_u$ is well-defined independent
  of the metric and connection.
\end{rmk}
\begin{lmm}\label{lmm:Banmfd}
  If the Hamiltonians $H_-$ and $H_+$ are clean (\cf
  Definition~\ref{dfn:MBreg}) each path-connected component of
  $\B^{1,p;\delta}$ is a Banach manifold and the vector bundles
  \begin{equation*}
    \begin{aligned}
      TB^{1,p;\delta} &:= \bigsqcup_{u \in B^{1,p;\delta}}
      T_u\B^{1,p;\delta},\qquad \E^{p;\delta}&:=\bigsqcup_{u \in
        \B^{1,p;\delta}} \E^{p;\delta}_u\;,
    \end{aligned}
  \end{equation*}
  carry the structure of Banach bundles. Moreover $T\B^{1,p;\delta}$
  is the tangent bundle of $\B^{1,p;\delta}$.
\end{lmm}
\begin{proof}
  We give a sketch since the proof is basically already
  given in \cite[Section 4.2]{Frauenfelder:PhD}. We also refer the
  reader to \cite[Theorem 2.1.7]{Schwarz:PhD} and \cite[Theorem
  2.2.1]{Schwarz:PhD}.

  In order to construct local charts we choose a metric which depends
  on the domain, denoted $(g_{s,t})_{(s,t)\in \Sigma}$. For every
  $(s,t) \in \Sigma$ we obtain a Levi-Civita connection, norm,
  distance function, exponential map and parallel transport associated
  to $g_{s,t}$, denoted by $\na^{s,t}$, $\nm{\,\cdot\,}_{s,t}$,
  $\mathrm{dist}_{s,t}$, $\exp^{s,t}$ and $\tensor[^{s,t}]{\Pi}{}$
  respectively. Fix metrics $g_-$ (\resp $g_+$) such that $L_0$ and
  $L_1^-:=\vp_{H_-}^{-1}(L_1)$ (\resp $L_0$ and $L_1^+ :=
  \vp_{H_+}^{-1}(L_1)$) are totally geodesic (\cf
  Lemma~\ref{lmm:totgeodesic} to see that such metrics exists). We
  assume that the family $(g_{s,t})$ satisfies
  \begin{enumerate}[label=(\alph*)]
  \item\label{nm:totgeod} $L_k$ is totally geodesic with respect to
    $g_{s,k}$ for $k =0,1$ and all $s \in \bar \R$,
  \item\label{nm:end} $g_{s,t} = (\vp^t_{H_+})_*g_+$ and
    $g_{-s,t}=(\vp^t_{H_-})_*g_-$ for all $s \geq 1$ and $t\in [0,1]$.
  \end{enumerate}
  Because the family only varies over a compact domain the minimal
  injectivity radius of $g_{s,t}$ is uniformly bounded from below by a
  constant $\e>0$. Given a map $u \in C^0(\Sigma,M)$ and a continuous
  vector field $\xi \in \Gamma(u^*TM)$ such that $\sup_{s,t}
  \nm{\xi(s,t)}_{s,t} < \e$ we define the map $u_\xi \in C^0(\R \times
  [0,1],M)$ via
  \begin{equation}
    \label{eq:uxi}
    u_\xi(s,t) := \exp^{s,t}_{u(s,t)} \xi(s,t)
  \end{equation}
  and the \emph{parallel transport map}
  \begin{equation}
    \label{eq:Pidef}
    \Pi_u^{u_\xi}: \Gamma(u^*TM) \to \Gamma(u_\xi^*TM), \qquad \xi' \mapsto \big(\Pi_u^{u_\xi}\xi')(s,t) := \tensor*[^{s,t}]{\Pi}{_{u(s,t)}^{u_\xi(s,t)}}\xi'(s,t)\;.
  \end{equation}
  Let $u \in \B$ be a strip which is smooth and asymptotically
  constant, \ie there exists $s_0>0$ such that $u(\pm s,\cdot)=u(\pm
  \infty)$ for all $s \geq s_0$. Define the subset
  \[
  \mathcal{V}_u := \{\xi \in T_u\B \mid \sup_{s,t}\nolimits
  \nm{\xi(s,t)}_{s,t} < \e\} \subset T_u\B\,.
  \]
  Since $p>2$ and $\delta>0$, it follows by the Sobolev embedding that
  $\V_u$ is an open subset. We define the chart map by $\exp_u:\V_u
  \to \B$, $\xi \mapsto u_\xi$. 

  We explain why $u_\xi \in \B$. By the property~\ref{nm:totgeod}, the
  map $u_\xi$ satisfies the boundary condition. By
  Corollary~\ref{cor:dwxi} we have $|\d u_\xi| \leq c(|\d u| + |\na
  \xi|)$ for some uniform constant $c$, which readily shows that
  $u_\xi$ is of regularity $H^{1,p}_\loc$. To show that the
  integral~\eqref{eq:uintegral} is finite we need a sharper
  estimate. By symmetry it suffices to consider only the positive end.
  Abbreviate $v(s,t) := (\vp_{H_+}^t)^{-1}(u(s,t))$ and
  $\zeta(s,t):=(\d\vp^t_{H_+})^{-1} (\xi(s,t))$. Let $\exp$ denotes
  the exponential map of the fixed metric $g^+$ and abbreviate
  $v_\zeta(s,t):=\exp_{v(s,t)}\zeta(s,t)$.  By property~\ref{nm:end}
  we have \[\exp^{s,t}\circ \d \vp_{H_+}^t = \vp_{H_+}^t \circ
  \exp\,,\] and thus $u_\xi(s,t) = \vp^t_{H_+}(v_\zeta(s,t))$ for all
  $s\geq 1$ and $t \in [0,1]$. Let $s_0$ be a constant such that
  $u(s,\cdot) = x_+ $ for all $s \geq s_0$. We have $\ns^{s,t}
  \xi(s,t) = \ps \xi(s,t) = \d \vp^t_{H_+} \ps \zeta$. Denote by
  $\nm{\,\cdot\,}_+$ the norm induced by $g_+$. Estimate for all
  $s\geq s_0$ and $t \in [0,1]$ using Corollary~\ref{cor:dwxi}
  \[
  \nm{\ps u_\xi(s,t)}_t = \nmm{\d \vp^t_{H_+} \ps v_\zeta}_t = \nm{\ps
    v_\zeta}_+ \leq c \nm{\ps \zeta}_+ = c \nm{\ps \xi}_t = c \nm{\ps
    (\xi - \xi_+)}_t\,.\] 
  Moreover we have $\pt u_\xi = X_{H_+}(u_\xi) + \d
  \vp^t\,\pt v_\zeta$ and $\na^{s,t}_t \xi = \na^{s,t}_\xi X_{H_+} + \d \vp^t \pt \zeta$. Thus
  \begin{multline*}
    \nm{\pt u_\xi- X_{H_+}(u_\xi)}_t = \nm{\pt v_\zeta}_+\\
    \leq
    c\nm{\pt \zeta}_+ = c\nm{\na_t \xi - \na_\xi X_H}_t = c\nm{\na_t
      \xi - \na_{(\xi-\xi_+)} X_H - \na_t \xi_+} \\ \leq c
    (1+\Nm{X_H})(\nm{\na_t (\xi-\xi_+)}_t + \nm{\xi-\xi_+}_t)\,.
  \end{multline*}
  Now use these two stronger estimates to show that the integral is
  bounded for $u_\xi$. One shows that the collection of all
  $(\V_u,\exp_u)$ indexed over all smooth and asymptotically constant
  curves $u$ gives an atlas of $\B$. This completes the proof of the
  first statement.
  
  We construct local trivializations of the bundles over the images of
  our chart maps given by $\mathcal{U}_u:=\{u' \in \B \mid \sup_{s,t
    \in \Sigma} \mathrm{dist}_{s,t}(u(s,t),u'(s,t)) <\e\}$ where again
  $u$ is smooth an asymptotically constant. The trivializations are
  defined using~\eqref{eq:Pidef}
  \begin{align*}
    \mathcal{U}_u \times T_u \B &\to T\B\big|_{\mathcal{U}_u},\qquad
    (u_\xi,\xi') \mapsto \Pi_u^{u_\xi} \xi' \in T_{u_\xi}\B\,,\\
    \mathcal{U}_u \times \E_u^{p,\delta} &\to
    \E^{p;\delta}\big|_{\mathcal{U}_u},\qquad (u_\xi,\eta) \mapsto
    \Pi_u^{u_\xi} \eta \in \E^{p;\delta}_{u_\xi}\;,
  \end{align*}
  where for the second map we actually use the unique continuous
  extension of the densely defined operator
  $\Pi_u^{u_\xi}:\E^{p,\delta}_u \to \E^{p;\delta}_{u_\xi}$. It is
  again straight-forward to check that the trivialization change is
  smooth using the estimates from Section~\ref{sec:dexp}.
\end{proof}
\begin{lmm}\label{lmm:totgeodesic}
  Let $M$ be a manifold and $L_0,L_1 \subset M$ be two submanifolds in
  clean intersection. There exists a metric on $M$ such that $L_0$ and
  $L_1$ are totally geodesic. Moreover given a submanifold $W \subset
  L_0 \cap L_1$, then there exists a metric such that $W$, $L_0$ and
  $L_1$ are totally geodesic.
\end{lmm}
\begin{proof}
  We construct the metric in suitable charts and patch it together at
  the end. For any point $p \in L_0 \cap L_1$ we find a chart
  identifying a neighborhood of $p$ with a ball in $\R^{2n}$ such that
  $p$ is identified with zero, $L_0$ is identified with the vector
  space $V_0 \subset \R^{2n}$ and $L_1$ is a graph over the vector
  space $V_1 \subset \R^{2n}$ of a function with vanishing
  differential at $0$. Since $L_0$ and $L_1$ intersect cleanly the
  intersection $L_0 \cap L_1$ is a graph over $K:=V_0 \cap
  V_1$. Decompose $\R^{2n} = K\oplus V_0'\oplus V_1' \oplus R$ such
  that $K\oplus V_0'=V_0$ and $K\oplus V_1'=V_1$. In the decomposition
  a point in $L_1$ has coordinates $(x,\vp(x,y),y,\psi(x,y))$ for
  functions $\vp:V_1 \to V_0'$ and $\psi:V_1\to R$ with vanishing
  differentials at $0$ and the property that $\psi(x,0) =0$ for all $x
  \in K$. Consider the map
  \begin{gather*}
    \Phi:K \oplus V_0'\oplus V_1'\oplus R\to     K \oplus V_0'\oplus V_1'\oplus R\\
    (x,x',y,y') \mapsto (x,x'-\vp(x,y),y,y'-\psi(x,y))\,.
  \end{gather*}
  The differential of $\Phi$ at $0$ is the identity, hence by possibly
  making the ball smaller we assume that $\Phi$ is a
  diffeomorphism. By construction $\Phi(L_0)=\Phi(V_0) = V_0$ and
  $\Phi(L_1)=V_1$. Hence composing the chart map with $\Phi$ we have
  found a chart such that $L_0$ and $L_1$ are identified with the
  vector spaces $V_0$ and $V_1$ respectively. Now take any metric on
  $V_0 \cap V_1$ such that $W$ is totally geodesic and extend it over
  the chart such that $V_0$ and $V_1$ are totally geodesic.
\end{proof} 
\subsection{Cauchy-Riemann-Floer section}
We fix a domain dependable vector field $X\in
C^\infty(\Sigma,\Vect(M))$ and an almost complex structure $J \in
C^\infty(\Sigma,\End(TM,\omega))$ which are admissible in the sense of
Definition~\ref{dfn:JXadm}. We define the \emph{non-linear
  Cauchy-Riemann-Floer section}
\[
\CR_{J,X}:C^\infty(\Sigma,M) \to C^\infty(\Sigma,TM), \qquad u \mapsto
\ps u + J(u) \left( \pt u - X(u)\right)\;.
\]
Let $T^\e M$ denote the disk-bundle of vectors $\xi$ with norm bounded
by $\e$. For $\e>0$ small enough we define the local representative of
$\CR_{J,X}$ at a given $u:\Sigma \to M$ by $\F_u:C^\infty(\Sigma,T^\e M) \to
    C^\infty(\Sigma,TM)$ with 
\begin{equation}\label{eq:FuCRX}
    \F_u(\xi)=\Pi_{u_\xi}^{u}\left(\ps
      u_\xi + J(u_\xi)\big(\pt u_\xi - X(u_\xi)\big) \right)\,.
\end{equation}
Here $u_\xi =\exp_u \xi$ and $\Pi_{u_\xi}^{u}$ are given
by~\eqref{eq:uxi} and~\eqref{eq:Pidef} respectively. The
\emph{linearized Cauchy-Riemann-Floer operator} as the differential of
the map~\eqref{eq:FuCRX} at zero, which is given by (\cf
\cite[Prop. 3.1.1]{Bibel})
\begin{equation}
  \label{eq:Du}     
  D_u \xi=\na_s \xi + J(u) \big(\na_t \xi - \na_\xi
    X(u)\big) + \big(\na_\xi J(u)\big)\left(\pt u - X(u)\right)\,.
  \end{equation}
\begin{rmk}
  If $u$ is satisfies $\CR_{J,X} u=0$, the operator $D_u$ is defined
  independently of the choice of $\na$ and is the vertical
  differential of $\CR_{J,X}$ at $u$. 
\end{rmk}
\begin{dfn}\label{dfn:mudecay}
  Given a constant $\mu>0$. A smooth map $u:\Sigma \to M$ has
  \emph{$\mu$-decay} if there exists constants $c$, $s_0$ and a smooth
  paths $x_-,x_+:[0,1] \to M$ such that for all $s \geq s_0$ and $t
  \in [0,1]$ we have 
  \begin{equation}
    \label{eq:mudecay}
    \di{u(s,t),x_+(t)}+ 
    \Nm{\ps u}_{C^1(\Sigma_s^\infty)} + \Nmm{\pt u- \Pi_{x_+}^u \pt x_+}_{C^1(\Sigma_s^\infty)}
    \leq c\, e^{-\mu s}\,,
  \end{equation}
  with $\Sigma_s^\infty:=[s,\infty)\times [0,1]$  and a similar
  estimate for the negative end.  We write $C^{\infty;\mu}(\Sigma,M)$
  for the space of all maps with $\mu$-decay.
\end{dfn}
The next proposition states that the space $\B^{1,p;\delta}$ contains
all finite energy $(J,X)$-holomorphic strips (\cf
equation~\eqref{eq:CRJX}), provided that $\delta$ is sufficiently
small. The constants $\iota(J_-,H_-)$ and $\iota(J_+,H_+)$ are defined
in equation~\eqref{eq:iotaC}.
\begin{prp}\label{prp:modinB}
  Set $\iota := \min\{\iota(J_-,H_-),\iota(J_+,H_+)\}$ with $J_\pm :=
  J(\pm s,\cdot)$ and $H_\pm$ given by $X_{H_\pm} = X(\pm s,\cdot)$
  for some $s$ large enough. For any $\mu<\iota$ we have that all
  $(J,X)$-holomorphic strips have $\mu$-decay. In particular $u \in
  \B^{1,p;\delta}$ for any $\delta < \iota$.
\end{prp}
\begin{proof}
  Define the map $v:\Sigma_0^\infty \to M$, $(s,t) \mapsto
  (\vp_{H_+}^t)^{-1}(u(s,t))$. The map has boundary in $(L_0,L_1')$
  with $L_1' = \vp_{H_+}^{-1}(L_1)$. Since by Lemma~\ref{lmm:change}
  the map $v$ is $J'$-holomorphic with $J'_t := (\vp_{H_+}^t)^*
  (J_+)_t$ for $t\in [0,1]$, and $L_0$ and $L_1'$ are in clean
  intersection by assumption, we conclude that $v(s,\cdot)$ converges
  to an intersection point $p \in L_0 \cap L_1'$ as $s \to \infty$
  (\cf Theorem~\ref{thm:remove}). By Lemma~\ref{lmm:Hessconj} we have
  $\iota(J_+,H_+)=\iota(J')$. Hence $\mu<\iota_p(J')$ and we conclude
  that $\Nm{\d v}_{C^1(\Sigma^\infty_s)} \leq O(e^{-\mu s})$. Moreover
  the path $[s,\infty) \to M$, $\sigma \to v(\sigma,t)$ extends to a
  continuous path from $v(s,t)$ to $p$, hence
  \[
  \di{v(s,t),p} \leq \int_s^\infty|\partial_\sigma v(\sigma,t)|\d
  \sigma \leq O(1)\int_s^\infty e^{-\mu \sigma}\d \sigma \leq
  O(e^{-\mu s})\,.
  \]
  By construction $\ps u = \d\vp_H^t\,\ps v$. Set $x:[0,1]\to M$, $t
  \mapsto \vp_H^t(p)$. Since $\pt x = X_H(x)$ we also have 
  \[\pt u -
  \Pi_x^u \pt x = \d \vp_H^t \pt v + X_H(u)-\Pi_x^u X_H(x)= O(e^{-\mu
    s})\,.\] Using these identities and the estimates for $v$ we
  conclude that $u$ has $\mu$-decay for the positive end (\cf
  estimate~\eqref{eq:mudecay}). We proceed similarly for the negative
  end. Now since $u$ has $\mu$-decay on both ends we conclude that the
  integral~\eqref{eq:uintegral} in the definition of $\B^{1,p;\delta}$
  is finite.
\end{proof}
\begin{thm}\label{thm:DuFred}
  With $\iota>0$ as in Proposition~\ref{prp:modinB}. Choose constants
  $\delta$ and $\mu$ such that $\delta<\mu <\iota$.  For any smooth
  map $u \in \B^{1,p;\delta}$ with $\mu$-decay the linearized
  Cauchy-Riemann operator $D_u$ defined in~\eqref{eq:Du} extends to a
  bounded Fredholm operator $ D_u:T_u\B^{1,p;\delta} \to
  \E_u^{1,p;\delta}$ of index
  \begin{equation}
    \label{eq:indDu}
    \ind D_u = \mu_\Vit(u) + \frac 12 \left(\dim  C_- + \dim C_+\right)\,,
  \end{equation} 
  in which $C_- \subset \I_-$ and $C_+ \subset \I_+$ are connected
 components such that $u(-\infty) \in C_-$ and $u(\infty) \in C_+$.
\end{thm}
\begin{proof}
  We describe how to conjugate the operator $D_u$ to an operator $D =
  \ps + \Jstd \pt + S$ as considered in section~\ref{sec:linfred}. The
  statements then follow from the fact that $D$ is Fredholm proven in
  Lemma~\ref{lmm:DFred} and the index formula.
  
  In a first step we construct trivializations of $u^*TM$. Let $s_0$
  be such that $X(s,t) = X_{H_+}(t)$ for all $s \geq s_0$.  The map
  $v:[s_0,\infty)\times [0,1] \to M$,
  $v(s,t)=(\vp^t_{H_+})^{-1}(u(s,t))$ is a $J'$-holomorphic half-strip
  where $J'=(\vp_{H_+})^* J$ and with boundary condition $v(s,0) \in
  L_0$ and $v(s,1) \in L_1':=(\vp_{H_+})^{-1}(L_1)$ for all $s \geq
  s_0$ (\cf Lemma~\ref{lmm:change}). By asymptotic analysis the point
  $v(s,t)$ lies inside a suitable neighborhood of $p_+ =v(\infty) \in
  L_0 \cap L_1' $ for all $t\in [0,1]$ and $s$ sufficiently large (\cf
  Theorem~\ref{thm:remove}). Let $\Phi_+$ be the trivialization
  constructed in Lemma~\ref{lmm:trivi} with respect to $J'$, $L_0$ and
  $L_1$. Then define $\Phi_u(s,t)=\d \vp_{H_+}^t\circ
  \Phi_+(t,v(s,t)):\R^{2n} \to T_{u(s,t)}M$.  We end up with a
  trivialization $\Phi_u$ of $u^*TM|_{[s_0,\infty)\times [0,1]}$ which
  is
  \begin{itemize}
  \item symplectic, \ie
    $\omega_{u(s,t)}(\Phi_u(s,t)\xi,\Phi_u(s,t)\xi')=\ostd(\xi,\xi')$
    for all $s\geq s_0$, $t\in [0,1]$ and $\xi,\xi' \in
    \R^{2n}$,
  \item complex linear, \ie $J_t(u(s,t)) \Phi_u(s,t)\xi = \Phi_u(s,t)
    \Jstd \xi$ for all $s\geq s_0$, $t\in [0,1]$ and $\xi \in
    \R^{2n}$,
  \item and trivializes the Lagrangians, \ie $T_{u(s,k)} L_k = \Phi_u(
    s,k) \R^n$ for all $k=0,1$ and $s\geq s_0$.
  \end{itemize}
  Similarly we construct $\Phi_u$ over $(-\infty,-s_0]\times
  [0,1]$. Then we extend $\Phi_u$ over the rectangle $[-s_0,s_0]\times
  [0,1]$ such that it is symplectic and complex linear (satisfies the
  first two properties) but not necessarily trivializes the
  Lagrangians. In fact a trivialization which satisfies all three
  properties might not exist.  We define the matrix valued function
  $S:\R\times [0,1] \to \R^{2n\times 2n}$ by
  \begin{equation*}
    \Phi_u\left(\ps \xi + \Jstd \pt \xi + S \xi \right) = D_u \Phi_u \xi\;,    
  \end{equation*}
  for all smooth $\xi:\R\times [0,1] \to \R^{2n}$.  Set $x_-
  =u(-\infty,\cdot)$ and $x_+=u(\infty,\cdot)$. We also define
  $\Phi_{x_+}(t)=\d \vp_{H_+}^t \Phi_+(t,x(t)):\R^{2n}\to T_{x_+(t)}M$
  and similarly $\Phi_{x_-}$. The matrix is asymptotic to the paths
  $\sigma_-,\sigma_+:[0,1] \to \R^{2n\times 2n}$ given by
  \begin{equation*}
    \Phi_{x_\pm}\left(\Jstd \pt \xi+\sigma_\pm
      \xi\right)=
    J(x_\pm) (\na_t - \na X_{H_\pm}(x_\pm))  \Phi_{x_\pm} \xi\,,   
  \end{equation*}
  for all smooth $\xi:[0,1] \to \R^{2n}$. By assumption $u$ has
  $\mu$-decay. We conclude that $S(-s,\cdot)$ converges to $\sigma_-$
  and $S(s,\cdot)$ to $\sigma_+$ as $s$ tends to $\infty$ and moreover
  that $S$ has $\mu$-decay (\cf Lemma~\ref{lmm:Sconv}).  We define the
  paths of linear Lagrangians $F=(F_0,F_1):\R \to \L(n) \times \L(n)$
  via
  \begin{equation*}
    \label{eq:F}
    F_0(s) = \Phi_u(s,0)^{-1}T_{u(s,0)} L_0,\qquad F_1(s) =
    \Phi_u(s,1)^{-1}T_{u(s,1)} L_1\;.   
  \end{equation*}
  By construction the path of Lagrangians $F=(F_0,F_1)$ is
  asymptotically constant. In particular the pair $(F,S)$ is
  admissible in the sense of Definition~\ref{dfn:FSadm} and has the
  asymptotic operators $A_- = A_{\sigma_-}$ and
  $A_+=A_{\sigma_+}$. With notation from Section~\ref{sec:linfred} (in
  particular Definition~\eqref{eq:H1pext}) we have isomorphisms of
  Banach spaces
  \[
  H^{1,p;\delta}_{F,W}(\Sigma,\R^{2n})
  \stackrel{\cong}{\longrightarrow} T_u\B^{1,p;\delta}(C_-,C_+)
  ,\qquad L^{p;\delta}(\Sigma,\R^{2n})
  \stackrel{\cong}{\longrightarrow} \E_u^{p;\delta}\;,
  \]
  both given by $\xi \mapsto ((s,t) \mapsto \Phi_u(s,t)\xi(s,t))$
  where $W=(\ker A_-,\ker A_+)$.  Via these isomorphism the operator
  $D_u$ is conjugated to the operator
  \[
  D:H^{1,p;\delta}_{F,W}(\Sigma,\R^{2n}) \to
  L^{p;\delta}(\Sigma,\R^{2n}),\qquad \xi \mapsto \ps \xi + \Jstd \pt
  \xi + S \xi\;.
  \]
  The asymptotic operators $A_-$ and $A_+$ are conjugated to the
  Hessians $A_{x_-}$ and $A_{x_+}$ respectively, in particular have the
  same spectrum. The claim that $D$ and hence also $D_u$ is
  well-defined and Fredholm now follows by Lemma~\ref{lmm:Dext} using
  the fact that $S$ has $\mu$-decay. It remains to check the index
  formula. By Lemma~\ref{lmm:index} the index of $D$ is given by
  \begin{equation*}
      \mu(\Psi_+\R^n,\R^n) + \mu(F_0,F_1) - \mu(\Psi_-\R^n,\R^n) +\frac
    12 \dim \ker A_- +\frac 12 \dim \ker A_+
  \end{equation*}
  where $\Psi_\pm:[0,1]\to \Sp(2n)$ are the fundamental solutions for
  $\sigma_\pm$ as defined in equation~\eqref{eq:fundsol}. We claim that for
  all $t\in [0,1]$
  \begin{equation}
    \label{eq:PsiPhi}
    \Psi_\pm(t) = \Phi_\pm(t)^{-1}\d \vp_{H_\pm}^t
    \Phi_\pm(0) \,,
  \end{equation}
  in which $t \mapsto \vp_{H_\pm}^t$ denotes the Hamiltonian flow.
  Denote by $e_1,\dots,e_{2n}$ the standard basis of $\R^{2n}$ and for
  each $i =1,\dots,2n$ and $t\in [0,1]$ we define the vector
  \[\xi_i(t) := \Phi_\pm(t)^{-1} \d
  \vp_{H_\pm}^1 \Phi_\pm(0) e_i \in \R^{2n}\,.\] By definition of
  $\sigma_\pm$ we have using the fact that $\na$ is torsion-free and $\pt
  x_\pm = X_{H_\pm}$
  \[
  \Phi_\pm\big(\Jstd \pt \xi_i + \sigma_\pm \xi_i\big)=J(x_\pm) \big(\na_t
  \d \vp_{H_\pm} \Phi_\pm(0)e_i + \na_{\d \vp_{H_\pm} \Phi_\pm(0) e_i}
  X_{H_\pm} \big) = 0\,.\] We see that the function $\xi_i(t)$
  satisfies the same ordinary differential equation as $\Psi_\pm(t)
  e_i$. This implies~\eqref{eq:PsiPhi} and so
  \[\Psi_\pm(t)\R^n = \Phi_\pm(t) \d
  \vp_{H_\pm}^t T_{x_\pm(0)}L_0\,.\] 
  By definition of the Viterbo index we conclude
  \[
  \mu(\Psi_+\R^n,\R^n) + \mu(F_0,F_1) - \mu(\Psi_-\R^n,\R^n) =
  \mu_\Vit(u)\;.
  \]
  Since the asymptotic operators $A_\pm$ are conjugated to the Hessian
  at $x_\pm$ whose kernel is given by $T_{x_\pm} C_\pm$ by
  Lemma~\ref{lmm:kerHess}, we obtain
  \[
  \dim \ker A_\pm  = \dim C_\pm.
  \]
  Obviously $\ind D=\ind D_u$ and the last two equations plugged
  into the index formula for $D$ gives the result.
\end{proof}
\subsection{Linear Theory}\label{sec:linfred}
Denote by $\mathcal{L}(n)$ the space of linear Lagrangian subspaces in
$\left(\R^{2n},\ostd\right)$ and abbreviate $\Sigma = \R \times
[0,1]$. Fix smooth maps $F:\R \to \L(n) \times \L(n)$ and $S:\Sigma
\to \R^{2n\times 2n}$.  In this section we study a differential
operator
\[
D \xi = \ps \xi + \Jstd \pt \xi +S\cdot \xi\,,
\]
defined on some Banach space of maps $\xi:\Sigma \to \R^{2n}$
satisfying the boundary conditions
\begin{equation}
  \label{eq:BClin}
 (\xi(s,0),\xi(s,1)) \in F(s),\qquad \forall\ s\in \R \,.
\end{equation}
The precise definition of the Banach space is given below. We proof
that $D$ is a Fredholm operator and compute its index. We now give
more details.
\begin{dfn}\label{dfn:FSadm}
  Let $(F,S)$ be a pair of smooth maps $F:\R \to \L(n)\times \L(n)$
  and $S:\Sigma \to \R^{2n\times 2n}$ $(F,S)$. We call $(F,S)$
  \emph{admissible} if
  \begin{enumerate}[label=(\roman*)]
  \item $F$ is asymptotically constant, \ie there exists $s_0 >0$ and
    $\Lambda_-,\Lambda_+\in \L(n) \times \L(n)$ such that $F(-s) =
    \Lambda_-$ and $F(s)=\Lambda_+$ for all $s \geq s_0$ and
  \item there exists paths of symmetric matrices
    $\sigma_-,\sigma_+:[0,1] \to \Sym(2n) \subset \R^{2n\times 2n}$
    such that $\lim_{s \to \pm \infty} S(s,\cdot)=\sigma_{\pm}$ uniformly.
  \end{enumerate}
\end{dfn}

Fix constants $\delta \in \R$, $p>1$ and a pair of finite dimensional
subspaces $W_-, W_+ \subset L^p([0,1],\R^{2n})$.  We define the Banach
space
\begin{equation}
  \label{eq:H1pext}
  H^{1,p;\delta}_{F,W}(\Sigma,\R^{2n}) \subset H^{1,p}_\loc(\Sigma,\R^{2n})\,,
\end{equation}
as the space of all functions $\xi:\Sigma \to \R^{2n}$ such that
\begin{enumerate}[label=(\roman*)]
\item $\xi$ is of regularity $H^{1,p}_\loc$,
\item $\xi$ satisfies the boundary condition~\eqref{eq:BClin},
\item there exists $\xi_- \in W_-$ and $\xi_+ \in W_+$ such that the
  following norm is bounded
  \[
  \Nm{\xi}_{1,p;\delta} := \Nm{\xi_-}_{L^p} + \Nm{\xi_+}_{L^p} +
  \Nm{(\xi-\xi_-)\kappa_\delta}_{H^{1,p}(\Sigma_-)} +
  \Nm{(\xi-\xi_+)\kappa_\delta}_{H^{1,p}(\Sigma_+)}\,,
  \]
  with weight-function $\kappa_\delta(s)=e^{\delta|s|}$.
\end{enumerate}
Secondly we define $L^{p;\delta}(\Sigma,\R^{2n})$ as the Banach space
of all $\eta \in L^p_\loc(\Sigma,\R^{2n})$ which are bounded with
respect to the norm
\[
\Nm{\eta}_{p;\delta} := \Nm{\eta \kappa_\delta}_{L^p(\Sigma)}
\] 
If $W_- =0$ and $W_+=0$ are trivial spaces, we
abbreviate~\eqref{eq:H1pext} by $H^{1,p;\delta}_F(\Sigma,\R^{2n})$. If
moreover the weight $\delta=0$ vanishes we abbreviate the space by
$H^{1,p}_F(\Sigma,\R^{2n})$.  We are ready to give a precise
definition of the operator under consideration
\begin{equation}
  \label{eq:DFS}
  D:H^{1,p;\delta}_{F,W}(\Sigma,\R^{2n}) \to
  L^{p;\delta}(\Sigma,\R^{2n}),\qquad \xi \mapsto \ps \xi + \Jstd \pt
  \xi + S\cdot \xi  \,.
\end{equation}
Secondly we define the \emph{asymptotic operators}. To a pair of
Lagrangians $\Lambda\in \L(n)\times \L(n)$ and a path $\sigma:[0,1]
\to \R^{2n\times 2n}$ of symmetric matrices, we associate the Banach
space
\[
H_\Lambda^{1,2}([0,1],\R^{2n}) := \{ \xi \in H^{1,2}([0,1],\R^{2n})
\mid (\xi(0),\xi(1)) \in \Lambda\}\,,
\]
and the operator
\begin{equation}
  \label{eq:AAop}
  A:H^{1,2}_\Lambda([0,1],\R^{2n}) \to L^2([0,1],\R^{2n}),\qquad
  \xi \mapsto \Jstd\pt \xi + \sigma\cdot\xi\,.
\end{equation}
Let $(F,S)$ be an admissible pair such that $\Lambda_\pm = F(\pm
\infty)$ and $\sigma_\pm = S(\pm \infty)$. We define the
\emph{asymptotic operators for $(F,S)$} as the operators $A_-$ and
$A_+$ given by
\begin{equation}
  \label{eq:Aops}
  A_\pm = \Jstd \pt + \sigma_\pm:H^{1,2}_{\Lambda_\pm}([0,1],\R^{2n}) \to L^2([0,1],\R^{2n})\,.
\end{equation}
The following lemma is crucial for the study of the operator $D$.
\begin{lmm}\label{lmm:Aop}
  The operator $A$ given in~\eqref{eq:AAop} is Fredholm of index
  zero. Considered as an unbounded operator with dense domain acting in
  the Hilbert space $L^2([0,1],\R^{2n})$ the operator $A$ is
  self-adjoint with spectrum consisting only of eigenvalues.
\end{lmm}
\begin{proof}
  We have the estimate
  \[
  \Nm{\xi}_{1,2} \leq c (\Nm{A\xi}_2 + \Nm{\xi}_2)\,.
  \]
  for some constant $c$, which shows that $A$ is a semi-Fredholm
  operator. That $A$ is self-adjoint is proved in~\cite[Lmm.\
  4.3]{Frauenfelder:PhD}. Consequently $\coker A=\ker A^*=\ker A$ is finite dimensional, which implies that $A$ is Fredholm of index zero.
\end{proof}

\subsubsection{Formal adjoint} Given an admissible pair $(F,S)$ and some
$q>1$.  We define the operator $D_{F,S}^*:H^{1,q}_F(\Sigma,\R^{2n})
\to L^q(\Sigma,\R^{2n})$ by
\begin{equation}
  \label{eq:Dadjdef}
  (D_{F,S}^*\xi)(s,t)=-\partial_s\xi(s,t) +\Jstd \partial_t \xi(s,t) +S^T(s,t)\xi(s,t)\,.
\end{equation}
The next lemma states that the operators $D_{F,S}$ and $D_{F,S}^*$ are
formally adjoint.
\begin{lmm}\label{lmm:DDadj}
  Given an admissible pair $(F,S)$. Assume that $1=1/p+1/q$ and
  consider the  operators $D=D_{F,S}$ and
  $D^*=D^*_{F,S}$. For all $\xi \in H^{1,p}_F(\Sigma,\R^{2n})$ and
  $\eta \in H^{1,q}_F(\Sigma,\R^{2n})$ we have
  \begin{equation}
    \int_\Sigma \<D\xi,\eta\> \d s \d t=\int_\Sigma \<\xi,D^*\eta\> \d s \d t\,.
  \end{equation}
\end{lmm}
\begin{proof}
  By partial integration we have with $s \in \R$ fixed
  \[\int_0^1\langle \xi,\Jstd\partial_t \eta\rangle \d t =
  \int_0^1\langle \Jstd \partial_t\xi, \eta\rangle \d t\;,\] because
  $\langle \Jstd\xi,\eta\rangle = \ostd(\xi,\eta)$ vanishes for
  $t=0,1$ after the Lagrangian boundary condition. Again by partial
  integration for $\ps$ and the fact that $ \Nm{\xi(s,\cdot)}_{L^2}$
  and $\Nm{\eta(s,\cdot)}_{L^2}$ vanish as $s$ tends to $\pm \infty$,
  \[
  \int_\Sigma \<\ps \xi,\eta\> \d s\d t = \int_\Sigma \<\xi,\ps
  \eta\> \d s \d t\;.
  \]
  We compute
  \begin{align*}
    \langle \xi,D^*\eta\rangle_{L^2} &=-\langle \xi,\partial_s \eta\rangle_{L^2} + \langle\xi,\Jstd\partial_t
    \eta\rangle_{L^2} + \langle \xi,S^T\eta\rangle_{L^2}\\
    &=\langle \partial_s \xi,\eta\rangle_{L^2} + \langle \Jstd \partial_t
    \xi,\eta \rangle_{L^2} +\langle S \xi,\eta\rangle_{L^2} = \langle
    D\xi,\eta\rangle_{L^2}\;,
  \end{align*}
  where $\<\cdot,\cdot\>_{L^2}$ denotes the inner product on
  $L^2(\Sigma,\R^{2n})$. 
\end{proof} 
\subsubsection{Fredholm property}
To show that the operator given in~\eqref{eq:DFS} is Fredholm, we
follow \cite[Section 2]{Salamon:Lecture} and~\cite[Section
3]{Schwarz:PhD}. In these two sources the authors considers the
perturbed Cauchy-Riemann operator defined on the cylinder instead of
the strip with boundary values. However the proofs go through with
almost no change.
\begin{lmm}\label{lmm:DFred}
  Let $(F,S)$ be an admissible pair with asymptotic operators $A_-$
  and $A_+$. If $-\delta$ and $\delta$ is not a spectral value of
  $A_-$ and $A_+$ respectively then the operator
  \[
  D=\ps + \Jstd \pt + S:H_F^{1,p;\delta}(\Sigma,\R^{2n}) \to
  L^{p;\delta}(\Sigma,\R^{2n})\,,
  \]
  is Fredholm.
\end{lmm}
\begin{proof}\setcounter{stp}{0}
  Assume first that $\delta=0$, \ie the operators $A_-$ and $A_+$ are
  invertible. Then proof is completely analogous to the proof
  of~\cite[Thm.\ 2.2]{Salamon:Lecture}. Note that in order to prove
  the precursors \cite[Lmm.\ 2.4]{Salamon:Lecture} all is necessary
  that the operators $A_-$ and $A_+$ are invertible self-adjoint
  operators. This fact is established in Lemma~\ref{lmm:Aop}.  To
  prove the statement for $\delta \neq 0$, pick a smooth function
  $\kappa_\delta : \R \to \R$ such that $\kappa_\delta(s)=
  e^{\delta\nm{s}}$ for all $\nm{s} \geq 1$. We have isomorphisms
  \[H^{1,p;\delta}_F(\Sigma,\R^{2n}) \stackrel{\cong}{\longrightarrow}
  H^{1,p}_F(\Sigma,\R^{2n}),\qquad L^{p;\delta}(\Sigma,\R^{2n})
  \stackrel{\cong}{\longrightarrow} L^p(\Sigma,\R^{2n})\,,\] both
  given by sending $\xi$ to the function
  $\kappa_\delta(s)\xi(s,t)$. Conjugating the operator $D$ with these
  isomorphisms gives the operator $D^\delta=D - (\ps
  \kappa_\delta)/\kappa_\delta$, which has asymptotic operators $A_- +
  \delta$ and $A_+-\delta$ respectively. By assumption these are
  invertible and thus by our first remark $D^\delta$ is
  Fredholm, hence $D$ is Fredholm too.
\end{proof}
\begin{lmm}\label{lmm:Dext}
  With the same assumptions as Lemma~\ref{lmm:DFred} suppose
  additionally that $S$ has $\mu$-decay for some $\mu >\delta$. Set
  $W=(\ker A_-,\ker A_+)$, then the operator
  \begin{equation}
    \label{eq:DFSext}
    D= \ps + \Jstd \pt + S : H^{1,p;\delta}_{F;W}(\Sigma,\R^{2n})
    \to L^{p;\delta}(\Sigma,\R^{2n})\;,
  \end{equation}
  is bounded and Fredholm.
\end{lmm}
\begin{proof}
  Given $\xi \in H^{1,p;\delta}_{F;W}(\Sigma,\R^{2n})$ with limits
  $\xi_\pm = \xi(\pm \infty) \in \ker A_\pm$.  By assumption $\Jstd
  \pt \xi_+ + \sigma_+ \xi_+ =0$ and $\ps \xi_+\equiv 0$. Hence with
  the decay property of $S$ we conclude that for all $s$ large enough
  \begin{equation}
    \label{eq:Sxi}
    \begin{aligned}
      |(\ps + \Jstd \pt + S)\xi| &\leq |(\ps + \Jstd \pt + S)(\xi-\xi_+)| +
      |(S-\sigma_+)\xi_+|\\
      &\leq |\d(\xi-\xi_+)| + O(1)|\xi-\xi_+| + O(e^{-\mu s}) |\xi_+|\,.   
    \end{aligned}
  \end{equation}
  We have a similar estimate for the negative end. If we multiply with
  $e^{\delta|s|}$ and integrate over $\Sigma$ we obtain
  \[
  \Nm{D\xi}_{p;\delta}^p \leq O(1) \Nm{\xi}_{1,p;\delta}+O(1)\int_0^\infty
  e^{-(\mu-\delta)s}\, d s) \Nm{\xi}_{1,p;\delta}\,.
  \]
  We conclude that the operator $D$ is bounded. Restricted to the
  finite co-dimensional subspace
  \[
  H^{1,p;\delta}_F(\Sigma,\R^{2n}) \subset
  H^{1,p;\delta}_{F;W}(\Sigma,\R^{2n})\] the operator $D$ is Fredholm by
  Lemma~\ref{lmm:DFred}. This shows the claim.
\end{proof}
\subsection{Index}\label{sec:RS}
In this section we compute the index of the Fredholm operator $D$ in
terms of the Robbin-Salamon index for paths. This index is defined
in~\cite{Robbin:Paths} and corresponds to the spectral flow of the
family of self-adjoint operators $A(s)$ as given in~\eqref{eq:AAop}.
We give here a quick introduction and review the basic properties.

Let $F=(F_0,F_1):[a,b]\to \L(n) \times \L(n)$ be a path of pairs of
linear Lagrangian spaces. Assume that $F_1(s) = \Lambda$ is constant,
we define the \emph{crossing form} $\Gamma(F_0,\Lambda;s)$ as a
quadratic form on $F_0(s) \cap \Lambda$ given by
\[
\Gamma(F_0,\Lambda;s) v := \left.\frac{\d}{\d
    \sigma}\right|_{\sigma=s} \ostd(v,w(\sigma))\;,
\]
where $v \in F_0(s)\cap \Lambda$ and $w:(s-\e,s+\e) \to \Jstd F(s)$ is
any differentiable map such that $w(\sigma) + v \in F(\sigma)$ for all
$\sigma \in (s - \e,s+\e)$. The proof that $\Gamma(F_0,\Lambda;s)$ is
well-defined is given in \cite[Theorem 1.1 (1)]{Robbin:Paths}. In the
case when $F_1$ is not constant we define the quadratic form on
$F_0(s) \cap F_1(s)$ via
\[
\Gamma(F_0,F_1;s) := \Gamma(F_0,F_1(s);s)  - \Gamma(F_1,F_0(s);s)\;.
\]
A \emph{crossing} $s \in [a,b]$ is a time where $F_0(s) \cap F_1(s)$
is non-trivial.  A crossing is called \emph{regular} if
$\Gamma(F_0,F_1;s)$ is non-degenerate. If $F=(F_0,F_1)$ has only
regular crossings the \emph{Robbin-Salamon index} of $F$ is defined by
\[\mu(F_0,F_1):=\frac 12 \sign \Gamma(F_0,F_1;a) + \sum_{a<s<b} \sign \Gamma(F_0,F_1;s) + \frac 12 \sign \Gamma(F_0,F_1;b)\;,\]
where $\sign$ denotes the signature, that is the number of positive
eigenvalues minus the number of negative eigenvalues. The sum is
finite because regular crossings are isolated. The Robbin-Salamon
index for an arbitrary path $F=(F_0,F_1)$ is defined by the
index of a perturbation that fixes the endpoints and has only regular
crossings. As proven in~\cite[Theorem 2.3]{Robbin:Paths} the index
enjoys the following properties.
\begin{description}
\item[Naturality] For any path $\Psi:[a,b] \to \Sp(2n)$ we have
\[\mu(\Psi F_0,\Psi F_1) =\mu(F_0,F_1)\;.\]
\item[Homotopy] The Robbin-Salamon index is invariant under homotopies
  which fix the endpoints.
\item[Zero] If $F_0(s) \cap F_1(s)$ is of constant dimension for all
  $s \in [a,b]$, then $\mu(F_0,F_1) =0$.
\item[Direct sum] If $F = F' \oplus F''$, then 
  \[\mu(F'_0\oplus F''_0,F'_1\oplus F''_1) = \mu(F'_0,F'_1) +
  \mu(F''_0,F''_1)\;.\]
\item[Concatenation] If $F = F' \# F''$, then
  \[\mu(F_0,F_1) = \mu(F'_0,F_1') + \mu(F''_0,F''_1)\;.\]
\item[Localization] If $F_0(s) = \R^n \times \{0\}$ and $F_1(s) =
  \Gr(B(s))$ for a path $B:[a,b] \to \R^{n\times n}$ of symmetric
  matrices, then the Robbin-Salamon index is given by
  \[
  \mu(F_0,F_1) = \frac 12 \sign B(b) - \frac 12 \sign
  B(a)\;.
  \]
\end{description} 
Given a path $\sigma:[0,1] \to \R^{2n\times 2n}$ of symmetric
matrices, the \emph{fundamental solution for $\sigma$} is a path of
symplectic matrices $\Psi:[0,1]\to \Sp(2n)$ given as the unique
solution of
\begin{equation}
  \label{eq:fundsol}
  \Jstd \pt \Psi + \sigma\Psi =0,\qquad \Psi(0) =\one\,.    
\end{equation}
Recall that $\iota(A) \geq 0$ denotes the spectral gap around zero of
a self-adjoint operator $A$ (\cf equation~\eqref{eq:iotaA}). 
\begin{lmm}\label{lmm:index}
  Let $(F,S)$ be an admissible pair with asymptotic operators $A_-$ and
  $A_+$.  For all $\delta$ with $0< \delta
  <\min\{\iota(A_-),\iota(A_+)\}$ the index of the operator
  \[
  D=\ps + \Jstd \pt + S:H^{1,p;\delta}_F(\Sigma,\R^{2n}) \to
  L^2(\Sigma,\R^{2n})
  \]
 is is given by
  \begin{equation}\label{eq:indD}
    \mu(\Psi^+\Lambda^+_0,\Lambda^+_1)+\mu(F_0,F_1) -
    \mu(\Psi^-\Lambda^-_0,\Lambda_1^-) - \frac 12 \dim C_- - \frac 12
    \dim C_+ \;,
  \end{equation}
  where $\Lambda^\pm=F(\pm \infty)$, $C_\pm = \ker A_\pm$ and
  $\Psi^\pm$ are the fundamental solutions of $S(\pm \infty)$.
\end{lmm}
\begin{proof}
  Assume for the moment that the asymptotic operators $A_-$ and $A_+$
  are injective. We claim that the formula~\eqref{eq:indD} holds for
  $\delta = 0$ and any $p>1$. For $p=2$ this is proven
  in~\cite[Theorem 7.42]{Robbin:Flow} (note that~\cite{Robbin:Flow}
  use different sign convention).  It remains to show that the index
  does not depend on $p>1$. For the case of the perturbed
  Cauchy-Riemann operator defined on the cylinder, the claim is proven
  in~\cite[Prp.\ 3.1.26]{Schwarz:PhD}. The arguments easily adapt to
  our situation. Denote by $D_p$ the operator
  \[
  \ps
  + \Jstd \pt +S:H^{1,p}_F(\Sigma,\R^{2n}) \to L^p(\Sigma,\R^{2n})\,.
  \]
  The claim follows, if we show $\ker D_p \cong \ker D_2$ and $\coker
  D_p \cong \coker D_2$. Since the index is invariant under homotopies
  of Fredholm operators, we assume without loss of generality that $S$
  is asymptotically constant, \ie there exists a constant $s_0$ such
  that $S(\pm s,t) = \sigma_\pm(t)$ for all $s \geq s_0$ and $t \in
  [0,1]$.  By elliptic regularity (\cf \cite[Prp.\ B.4.6]{Bibel}) we
  know every element in the kernel of $D_p$ is smooth. In order to
  show $\ker D_p \cong \ker D_2$ it suffices to show that
  $\Nm{\xi}_{1,2}$ is finite for all $\xi \in \ker D_p$. To show that
  we deduce the following exponential decay condition: there exists
  constants $c$ and $\iota$ such that for all $\nm{s} \geq s_0$ and $t
  \in [0,1]$
  \begin{equation}
    \label{eq:expdecay}
    \nm{\xi(s,t)} + \nm{\d \xi(s,t)} \leq ce^{-\iota \nm{s}}\,.
  \end{equation}
  By analogy we only deduce this inequality for the positive
  end. Denote the Hilbert space
  \[
  H = L^2([0,1],\R^{2n})\,,
  \]
  equipped with the standard norm $\Nm{\cdot}$ and scalar product
  $\<\cdot,\cdot\>$. Consider the positive asymptotic operator $A_+
  = \Jstd \pt + \sigma_+$ defined on $V=\{\xi \in
  H^{1,2}([0,1],\R^{2n}) \mid \xi(0),\xi(1) \in \R^n\}$. By abuse of
  notation think of $\xi:\R \to H$, $s \mapsto \xi(s)(t)=\xi(s,t)$ as
  a path in $H$. Since $\xi \in \ker D_p$ the path solves the
  differential equation for all $s \geq s_0$,
  \[
  \ps \xi(s) + A_+ \xi(s)=0\,.
  \] 
  Define the function $g:[s_0,\infty) \to \R$, $g(s) := \frac
  12\Nm{\xi(s)}^2$. By assumption the asymptotic operator $A_+$ is
  injective. We have shown in Lemma~\ref{lmm:Aop} that $A_+$ is
  unbounded self-adjoint and has a closed range. With
  Lemma~\ref{lmm:Aclosed} we have $\Nm{A_+ \xi(s)} \geq \iota
  \Nm{\xi(s)}$ for all $s \geq s_0$ where $\iota = \iota(A_+)$.  We
  compute
  \begin{equation}
    \label{eq:ddg}
    \ps \ps g(s)
 = -\ps \<A_+ \xi,\xi\> = -2 \<A_+\xi,\ps
    \xi\> = 2\Nm{A_+\xi}^2 \geq 2\iota^2 \Nm{\xi}^2 = 4\iota^2 g(s)\;,
  \end{equation}
For any $s \geq 0$ set $\xi_s:= \xi(s+\cdot)$. 
  The Sobolev embedding $H^{1,2} \hookrightarrow C^0$ for functions
  with one-dimensional domain $[-1,+1]$ implies that we have for all
  $s \geq s_0+2$,
    \[
    \Nm{\xi(s)}\leq O(1) \left(\int_{-1}^{+1} \Nm{\xi_s(\sigma)}^2 +
      \left(\ps \Nm{\xi_s(\sigma)}\right)^2 \d \sigma\right)^{1/2}\]
    We conclude via the Rellich embedding $H^{1,2} \hookrightarrow
    L^p$ for functions on $\Sigma_{-1}^1:=[-1,+1] \times [0,1]$ and elliptic
    regularity for $D$ (\cf \cite[Prop.\ B.4.6]{Bibel}) 
    \begin{equation*}
      \Nm{\xi(s)}\leq O(1) \Nm{\xi_s}_{1,2;\Sigma_{-1}^1} \leq O(1)
      \Nm{\xi_s}_{p;\Sigma_{-1}^1} \leq
      O(1)\Nm{\xi_s}_{1,p;\Sigma_{-2}^2}\,.
      \end{equation*}
  The norm $\Nm{\xi}_{1,p}$ on $\R \times [0,1]$ is finite and
  we conclude that
  \begin{equation}
    \label{eq:limg}
    \lim_{s \to \infty} g(s) \leq O(1) \lim_{s\to\infty} \Nm{\xi}^2_{1,p;\Sigma_{s-2}^{s+2}}=0\;. 
  \end{equation}
  Define the functions $g_0,\psi:[s_0,\infty) \to \R$ by
  \[
  g_0(s):=g(s_0)e^{-2\iota(s-s_0)},\qquad \psi(s) := g(s) -
  g_0(s)\;.\] From~\eqref{eq:ddg} we have $\ddot \psi \geq 4 \iota^2
  \psi$. By~\eqref{eq:limg} we have $\psi(s_0)=0$ and $ \psi(s)
  \to 0$. Hence the maximum of $\psi$ on $[s_0,\infty)$ can not be
  strictly positive, which implies for all $s\geq s_0$ that
  $\psi(s)\leq 0$ or equivalently that
  \[g(s) \leq g_0(s)\,.\] Last inequality, elliptic regularity for $D$
  and Sobolev embeddings show
  \begin{equation*}
      \nm{\xi(s,t)} + \nm{\d \xi(s,t)} \leq O(1)
    \Nm{\xi_s}_{3,2;\Sigma_{-2}^2} \leq O(1) \Nm{\xi_s}_{2;\Sigma_{-2}^2} \leq O(1)e^{-\iota s}\;.
  \end{equation*}
  Therefore~\eqref{eq:expdecay} and $\xi \in H^{1,2}(\Sigma)$. On the
  other hand given any $\xi \in \ker D_2$ we conclude analogously that
  $\xi \in H^{1,p}_F$, thus
  \[
  \ker D_p \cong \ker D_2\;.
  \]
  With Lemma~\ref{lmm:DDadj} any element $\eta \in \coker D_p$ is
  identified with $\eta \in H^{1,q}_F(\Sigma,\R^{2n})$ satisfying
  $D^*\eta=0$ where $1/p+1/q=1$ and $D^* = -\ps + \Jstd \pt + S^T$. As
  before $\eta \in H^{1,2}_F(\Sigma)$ and so
  \[
  \coker D_p \cong \coker D_2\;.
  \]
  This shows formula~\eqref{eq:indD} for the case $\delta=0$ and
  $A_\pm$ injective.

  To show the formula for $\delta \neq 0$ we reduce to the previous
  case.  Choose a smooth function $\kappa_\delta$ such that
  $\kappa_\delta(s) = e^{\delta\nm{s}}$ for all $\nm{s}\geq 1$. As
  explained in the proof of Lemma~\ref{lmm:DFred} we consider the
  conjugated operator $D'=D+ \ps \kappa_\delta/\kappa_\delta$, with
  asymptotic operators $A_\pm' = \Jstd \pt + \sigma_\pm \mp \delta$.
  Since the operators $A_\pm'$ are invertible and by the last step,
  the operator $D'$ has the index
  \[
  \mu\big(\Psi^+_\delta \Lambda^+_0,\Lambda^+_1\big) + \mu(F_0,F_1)
  - \mu\big(\Psi^-_{-\delta} \Lambda^-_0,\Lambda^-_1\big)\;,
  \]
  where $\Psi^+_\delta,\Psi^-_{-\delta}:[0,1] \to \Sp(2n)$ are given
  as the unique solution of
  \[
  \Jstd \pt \Psi^{\pm}_{\pm\delta} + \big(\sigma_\pm \mp
  \delta\big)\Psi^\pm_{\pm \delta} = 0,\qquad \Psi^\pm_{\pm
    \delta}(0)= \one\;.
  \]
  Then the index formula for $D$ follows by the last equality and
  Lemma~\ref{lmm:Adelta}. 
\end{proof}
\begin{cor}\label{cor:indD}
  With assumptions of Lemma~\ref{lmm:index}. Assume additionally that
  $S$ has $\mu$-decay for some $\mu >0$ and set $W=(\ker A_-,\ker
  A_+)$. For any constant $\delta$ such that $0< \delta <
  \min\{\iota(A_-),\iota(A_+),\mu\}$ the index of the operator
  \[
  D=\ps + \Jstd\pt +S :H^{1,p;\delta}_{F;W}(\Sigma,\R^{2n}) \to L^{p;\delta}(\Sigma,\R^{2n})\,,
  \]
  is given by
  \[
  \mu(\Psi^+\Lambda^+_0,\Lambda_1^+)+\mu(F_0,F_1) -
  \mu(\Psi^-\Lambda^-_0,\Lambda^-_1) + \frac 12 \dim\ker A_-
  + \frac 12 \dim \ker A_+ \;.
  \]
\end{cor}
\begin{proof}
  This follows immediately from Lemma~\ref{lmm:index} since the
  codimension of the space $H^{1,p;\delta}_F(\Sigma,\R^{2n})$ as a
  subspace of $H^{1,p;\delta}_{F;W}(\Sigma,\R^{2n})$ is $\dim \ker A_-
  + \dim \ker A_+$.
\end{proof}
\begin{lmm}\label{lmm:Adelta}
  Given a path of symmetric matrices $\sigma:[0,1]\to \R^{2n \times
    2n}$ and a pair of Lagrangians $\Lambda=(\Lambda_0,\Lambda_1) \in
  \L(n) \times \L(n)$.  Consider the operator $A=\Jstd \pt +
  \sigma:H_{\Lambda}\to L^2([0,2],\R^{2n})$. For all $\delta$ with $0
  < \delta < \iota(A)$ we have
  \begin{align*}
    \mu(\Psi_\delta \Lambda_0,\Lambda_1) = \mu(\Psi_0 \Lambda_0,\Lambda_1) - \frac 12 \dim \ker A,\\
    \mu(\Psi_{-\delta} \Lambda_0,\Lambda_1) = \mu(\Psi_0
    \Lambda_0,\Lambda_1) + \frac 12 \dim \ker A\;,
    \end{align*}
    where $\Psi_\rho$ is the fundamental solution of $\sigma-\rho\one$
    for each $\rho =-\delta,0,\delta$.
\end{lmm}
\begin{proof}
  For every $\rho \in [-\delta,\delta]$ consider $\Psi_\rho$ as the
  fundamental solution of $\sigma-\rho\one$. The function $(\rho,t)
  \mapsto \Psi(\rho,t) = \Psi_\rho(t)$ is smooth in both
  variables. Hence $F(\rho,t) := \Psi(\rho,t)\Lambda_0$ defines a
  homotopy with fixed endpoints of $\Psi(\delta,\cdot)\Lambda_0$ to
  the concatenation of the paths $\rho \mapsto
  \Psi(\delta-\rho,0)\Lambda_0$, $t\mapsto \Psi(0,t)\Lambda_0$ and
  $\rho \mapsto \Psi(\rho,t)\Lambda_0$. By the axioms of the
  Robbin-Salamon index we have
  \[
  \mu(\Psi(\delta,\cdot)\Lambda_0,\Lambda_1) =
  -\mu(\Psi(\cdot,0)\Lambda_0,\Lambda_1)+\mu(\Psi(0,\cdot)\Lambda_0,\Lambda_1) +
  \mu(\Psi(\cdot,1)\Lambda_0,\Lambda_1)\;.
  \]
  Since $\Psi(\cdot,0)=\one$ we have
  $\mu(\Psi(\cdot,0)\Lambda_0,\Lambda_1)
  =\mu(\Lambda_0,\Lambda_1)=0$. We claim that the crossings of $\rho
  \mapsto (\Psi(\rho,1)\Lambda_0,\Lambda_1)$ agree with eigenvalues
  of $A$. Indeed, let $\rho \in [0,\delta]$ be a crossing, then there
  exists a non-trivial $w = \Psi(\rho,1)v \in \Psi(\rho,1)\Lambda_0
  \cap \Lambda_1$. Define $\xi(t):=\Psi(\rho,t)v$. By construction we
  have $\xi(0)\in \Lambda_0$, $\xi(1)\in \Lambda_1$ and $\Jstd \pt \xi
  + \sigma\xi=\rho\xi$. This shows that $\xi$ is an eigenvector of $A$
  with eigenvalue $\rho$. By assumption there are no eigenvalues in
  $(0,\delta]$ and hence the only crossing occurs for $\rho = 0$.

  We compute the crossing form
  $\Gamma(\Psi(\cdot,1)\Lambda_0,\Lambda_1;0)$.  Differentiate the
  identity $\Jstd \pt \Psi + (\sigma-\rho) \Psi=0$ by $\partial_\rho$
  to obtain
  \[
  \Psi(\rho,t) = \Jstd \partial_\rho \pt \Psi(\rho,t) +
  \big(\sigma(t)-\rho\big) \partial_\rho \Psi(\rho,t)\;.
  \]
  Using the last equation for $\rho =0$ we compute
  \begin{align*}
    \<\xi,\xi\>=\<\Psi v,\Psi v\>&=\<\Psi v,\Jstd \partial_\rho \pt
    \Psi v +\sigma \partial_\rho \Psi v\>
    \\
    &=\<\Psi v,\Jstd \partial_\rho \pt \Psi v\> +\< \sigma\Psi
    v,\partial_\rho
    \Psi v\> \\
    &=-\<\Jstd\Psi v, \partial_\rho \pt \Psi v\> - \<\Jstd \pt
    \Psi, \partial_\rho \Psi v\>\\
    & =-\pt \<\Jstd \Psi v,\partial_\rho \Psi v\> =-\pt \ostd(\Psi
    v,\partial_\rho \Psi v\>\;.
  \end{align*}
  Integrating the last equation shows (note that
  $\partial_\rho\Psi(\rho,0) =0$)
  \[
  \int_0^1 \<\xi, \xi\> \d t = -\ostd(\Psi(0,1)
  v,\partial_\rho\big|_{\rho=0} \Psi(0,1)v) =
  -\Gamma(\Psi(\cdot,1)\Lambda_0,\Lambda_1;0)w\;.
  \]
  The last identity is established in~\cite[Theorem 1.1
  (2)]{Robbin:Paths}. We see that the crossing form is negative
  definite and defined on the space $\ker A$. By the definition of the
  Robbin-Salamon index we conclude
  \[
  \mu(\Psi(\cdot,1)\Lambda_0,\Lambda_1) = \frac 12 \dim \ker A\;.
  \]
  This shows the claim.
\end{proof}

\section{Transversality}
In terms of last chapter we show that the Fredholm section is
transverse to the zero section for a generic choice of the almost
complex structures. Consequently every connected component of the
moduli space of perturbed holomorphic strips with Lagrangian boundary
conditions is a manifold. In the non-degenerated case of transversely
intersecting Lagrangians this was solved originally by Floer and Hofer
in~\cite{Floer:Trans}. The generalization for the degenerate case of
cleanly intersecting Lagrangians was treated by Frauenfelder
in~\cite{Frauenfelder:PhD}. Besides recapitulating these ideas we
prove some additional transversality results for the evaluation map
based on ideas of Seidel from~\cite{Seidel:sequence}. Transversality
is achieved by allowing the almost complex structure to explicitly
depend on the domain and is based on the existence of regular
  points.
\subsection{Main statement}\label{sec:regsetup}
Let $(M,\omega)$ be a symplectic manifold and $L_0,L_1 \subset M$ be
two Lagrangian submanifolds. Fix an admissible vector field $X$ and
denote the perturbed intersection points $\I_-$ and $\I_+$ of the
Hamiltonian vector field $X(-s,\cdot)$ and $X(s,\cdot)$ respectively for $s$
sufficiently large. Given smooth maps
$\vp_-:W_-\to \I_-$ and $\vp_+:W_+\to \I_+$ we define
\begin{equation}
  \label{eq:MWWJX}
  \wt \M(W_-,W_+;J,X) :=\left\{ (w_-,u,w_+) \left|\
      \begin{aligned}
        &\text{$u$ is $(J,X)$-holomorphic}\\
        &\vp_-(w_-) = u(-\infty)\\
        &\vp_+(w_+) = u(+\infty)      
      \end{aligned}\right\}\right.\,,
\end{equation}
Let $D_{u,J}$ denote the
vertical differential of the Cauchy-Riemann-Floer section (\cf
equation~\eqref{eq:Du}).
\begin{dfn}\label{dfn:reg}
  We way that $J$ is \emph{regular for $X$} if $D_{u,J}$ is surjective
  for all $(J,X)$-holomorphic strips $u$. Moreover $J$ is \emph{regular
    for $X$ and $\vp$} if additionally~\eqref{eq:MWWJX} is cut-out
  transversely, \ie
  \begin{equation}
    \label{eq:evW}
    \{(\xi(-\infty),\xi(\infty)) \mid \xi \in \ker D_{u,J}\} + \im \d_w
    \vp = T_{u(-\infty)}\I_- \oplus T_{u(\infty)}\I_+\,,
  \end{equation}
  for all $(w_-,w_+,u) \in \wt \M(W_-,W_+;J,X)$.
\end{dfn}
If $J$ is regular then each connected component of the
space~\eqref{eq:MWWJX} is a manifold by the implicit function theorem
(\cf \cite[Thm.\ A.3.3]{Bibel}). We now show that regular almost
complex structures exist in abundance. We split up the argument
depending whether $J$ needs to be $\R$-invariant or not.
\subsection{$\R$-dependent structures}
Let $s_0$ be a constant such that $X(- s,\cdot) = X_{H_-}$ and
$X(s,\cdot) =X_{H_+}$ for all $s \geq s _0$. Fix a constant $s_1 >
s_0$ and two paths of almost complex structures $J_-,J_+:[0,1]
\to \End(TM,\omega)$. We look for regular structures in the space
\[
\J := \{J \in C^\infty(\R \times
  [0,1],\End(TM,\omega))\mid J(\pm s,\cdot)=J_\pm\ \forall\ s \geq
  s_1\}\,.
\]A subset of a topological space is \emph{comeager} if it
is a countable intersection of open and dense sets.
\begin{thm}\label{thm:reg} 
  The subset of almost complex structures which are regular for $X$
  and $\vp$ is comeager in $\J$.
\end{thm}
\begin{proof}\setcounter{stp}{0}
  For the proof we follow the original approach of Floer using the
  $\varepsilon$-norms combined with the argument of Taubes as
  described in~\cite[Section 3]{Bibel}. Except for the part of
  transversality with respect to the evaluation the theorem is proven
  in~\cite[Theorem 4.10]{Frauenfelder:PhD}.
  
  Set $W=W_- \times W_+$ and let $W=\bigcup_{k\in\N} W_k$ be an
  exhaustion by compact subsets $W_k \subset W$ with $W_k \subset
  W_{k+1}$ for all $k \in \N$ and denote by $\vp_k$ the restriction of
  $\vp$ to $W_k$. Fix constants $p>2$ and $\mu>0$ small enough. For $k
  \in \N$ we define $\J_{\reg,k} \subset \J$ to be the space $J$ with
  the property that for any $(J,X)$-holomorphic strip $u$ such that
    \begin{equation}
      \label{eq:uR}
      \nm{\ps u(s,t)} \leq k e^{-\mu \nm{s}}
    \end{equation} 
    for all $s \in \R$ and $t \in [0,1]$ it holds that
  \begin{enumerate}[label=(\roman*)]
  \item the operator $D_{u,J}$ is surjective 
  \item for all $w \in W_k$ with $(u(-\infty),u(\infty)) =\vp(w)$ we
    have~\eqref{eq:evW}.
  \end{enumerate}

  It suffices to show that $\J_{\reg,k} \subset \J$ is dense and open
  for all $k\in \N$ because if that is true we write $\J_\reg =
  \bigcap_{k \in \N} \J_{\reg,k}\subset \J$ as a countable
  intersection of open and dense sets.
  \begin{stp}
    The subset $\J_{\reg,k} \subset \J$ is open for all $k \in \N$.
  \end{stp}
  Fix $k\in \N$. We show that the complement of $\J_{\reg,k}$ is
  closed. Take a sequence $J_\nu \in \J\setminus \J_{\reg,k}$ such
  that $J_\nu$ converges to $J$ with respect to the
  $C^\infty$-topology. Because $J_\nu \notin \J_{\reg,k}$ there exists
  a sequence $(u_\nu)_{\nu \in \N}$ of $(J_\nu,X)$-holomorphic strips
  such that for each $\nu \in \N$ we have~\eqref{eq:uR} and at least
  one of the following holds
  \begin{enumerate}[label=(\roman*)]
  \item $D_{u_\nu,J_\nu}$ is not surjective,
  \item there exists $w_\nu \in W_k$ with $
    (u_\nu(-\infty),u_\nu(\infty))=\vp(w_\nu)$ and
    \[\{(\xi(-\infty),\xi(\infty)) \mid \xi \in \ker D_{u_\nu,J_\nu}\}
    + \im \d_{w_\nu} \vp \subsetneq T_{u_\nu(-\infty)} \I_- \oplus
    T_{u_\nu(\infty)} \I_+\,.\]
  \end{enumerate}
  Since the gradient of $(u_\nu)$ is uniformly bounded, a subsequence, 
  still denoted by $(u_\nu)$, converges to a $(J,X)$-holomorphic map
  $u$ in $C^\infty_\loc$ (\cf Lemma~\ref{lmm:bgcomp}). We also have
  with~\eqref{eq:uR}
  \[\lim_{s \to \infty } \lim_{\nu \to \infty} E(u_\nu;[s,\infty)
  \times [0,1]) \leq \lim_{s \to \infty} k^2 \int_s^\infty e^{-2\mu
    s}\d s = 0\,.\] Similarly for the negative end. This shows that
  $E(u_\nu) \to E(u)$.  Provided with the convergence of the energy we
  conclude that $u_\nu$ converges to $u$ uniformly (\cf
  Lemma~\ref{lmm:convends}).  By the mean-value inequality and
  exponential decay of the energy (\cf Lemma~\ref{lmm:Edecay}) we have
  an uniform constant $c$ such that for all $s$ large enough
  \[\nm{\ps u(s,t)}^2\leq c
  E(u;[s,\infty)\times [0,1])) \leq k^2 e^{-2\mu s}\;.\] Similar for
  the negative end. This shows that $u$ satisfies~\eqref{eq:uR} since
  on compact subsets we have $C^1$-convergence.

  We distinguish two cases. In the first case we assume that after
  passing to a subsequence we have that $D_{u_\nu,J_\nu}$ is not
  surjective all $\nu \in \N$. Let $\Pi_u^{u_\nu}$ be the parallel
  transport operator defined in
  equation~\eqref{eq:Pidef}. Lemma~\ref{lmm:JnuJ} shows that $D_\nu:=
  \Pi_{u_\nu}^u D_{u_\nu,J_\nu} \Pi_{u}^{u_\nu}$ converges to
  $D:=D_{u,J}$ in the operator norm.  Since by assumption $D_\nu$ is
  not surjective for all $\nu \in \N$ and surjectivity is an open
  condition this shows that $D$ is not surjective. Hence $J \not\in
  \J_{\reg,k}$ and we are finished in that case.

  Assume in the second case by passing to a subsequence that for
  all $\nu \in \N$ it holds that the operator $D_{u_\nu,J_\nu}$ is
  surjective and there exists $ \zeta_\nu=(\zeta_\nu^-,\zeta_\nu^+)
  \in T_{u_\nu(-\infty)}\I_- \oplus T_{u_\nu(\infty)}\I_+$ such that
  $\nm{\zeta_\nu^-}+\nm{\zeta_\nu^+} =1$, $\zeta_\nu \perp \im
  \d_{w_\nu} \vp$ and
  $\<\zeta_\nu^-,\xi(-\infty)\>+\<\zeta_\nu^+,\xi(\infty)\>=0$ for all
  $\xi \in \ker D_{u_\nu,J_\nu}$. Since $W_k$ is compact we assume
  by passing to another subsequence that $w_\nu \to w \in W$ and
  $\zeta_\nu$ converges to a non-vanishing $\zeta=(\zeta^-,\zeta^+)$
  such that $\zeta \perp \im \d_w \vp$. It remains to show that for
  all $\xi \in \ker D_{u,J}$ we have
  \begin{equation}
    \label{eq:zetaperpkerD}
  \<\zeta^-,\xi(-\infty)\> + \<\zeta^+,\xi(\infty)\>=0\,.    
  \end{equation}
  Let $Q$ be a right-inverse of $D$. For $\nu$ large enough the kernel
  of $D_\nu$ is transverse to the image of $Q$ and since the operators
  $D_\nu$ and $D$ are both surjective with same index their kernels
  have the same dimension. In particular for every $\xi \in \ker D$
  there exists a unique $\xi_\nu \in \ker D_\nu$ such that
  $\xi-\xi_\nu \in \im Q$. With $Q\eta_\nu = \xi-\xi_\nu$ and norm
  $\Nm{\,\cdot\,}$ either $\Nm{\,\cdot\,}_{1,p;\delta}$ or
  $\Nm{\,\cdot\,}_{p;\delta}$ respectively we have
  \begin{multline*}
    \Nm{\xi-\xi_\nu}=\Nm{Q\eta_\nu}\leq O(1) \Nm{\eta_\nu} \leq
    \Nm{DQ\eta_\nu} = \Nm{D(\xi-\xi_\nu)} =\\= \Nm{(D-D_\nu)\xi_\nu}
    \leq o(1) \Nm{\xi_\nu} \leq o(1) + o(1)\Nm{\xi-\xi_\nu}\,.
  \end{multline*}
  This shows that $\Nm{\xi-\xi_\nu} \to 0$.  Since $\xi_\nu(\pm \infty) \perp \zeta_\nu^\pm$
  \begin{multline*}
    |\<\xi(-\infty),\zeta^-\> +\<\xi(\infty),\zeta^+\>| \\\leq
    \Nm{\xi-\xi_\nu} + \nm{\<\xi_\nu(-\infty),\zeta^--\zeta^-_\nu\> +
      \<\xi_\nu(-\infty),\zeta^+-\zeta^+_\nu\>}\\
    \leq \Nm{\xi-\xi_\nu} +
    \Nm{\xi_\nu}\left(\nm{\zeta^--\zeta^-_\nu} +
      \nm{\zeta^+-\zeta^+_\nu}\right) \to 0\,.
  \end{multline*}
  This shows~\eqref{eq:zetaperpkerD} and hence the claim in the second
  case.
  \begin{stp}
    Fix connected components $C_-\subset \I_-$ and $C_+ \subset
    \I_+$. There exists a dense subspace $\J' \subset \J$ such
    that the \emph{universal moduli space}
    \begin{equation*}\label{eq:MJXuniv}
      \wt \M(C_-,C_+;\J',X) =\{ (u,J) \mid u \in \wt \M(C_-,C_+;J,X),\ J \in \J'\}
    \end{equation*}
    is a Banach manifold and the evaluation map $(u,J) \mapsto
    (u(-\infty),u(\infty))$ is a submersion.
  \end{stp}
  By Lemma~\ref{lmm:perturbations} there exists a dense subspace $\J'
  \subset \J$ which is a separable Banach space.  It suffices to see
  that for all $(u,J) \in \wt \M(C_-,C_+;\J',X)$ the operator is
  surjective
  \begin{equation*}
    D^\univ_{u,J} : T_u\B^{1,p;\delta}(C_-,C_+)
    \oplus T_J\J'  \to
    \E_u^{p;\delta}(C_-,C_+),\
     (\xi,Y) \mapsto  D_{u,J}\xi + Y(\pt u
    -X) \,.
  \end{equation*}
  Because $D_{u,J}$ is Fredholm the operator $D^\univ_{u,J}$ has a
  closed range. Take $\eta$ in the cokernel, which is identified with
  an element in the Banach space $\E^{q,-\delta}_u$ where $1/p+1/q=1$
  such that for all $\xi \in T_u\B^{1,p;\delta}$ and $Y \in T_J\J'$
  we have
  \begin{equation}
    \label{eq:etacokernel}
  \int_\Sigma \<\eta,D_u \xi\> \d s \d t =0,\qquad
  \int_\Sigma\<\eta,Y(\pt u - X)\> \d s \d t =0\,.  
  \end{equation}
  The first equation implies that $\eta$ is smooth after elliptic
  regularity. We claim that by the second equation $\eta$ vanishes.
  Given a point $(s,t) \in [s_0,s_1]\times [0,1]$ such that $\ps
  u(s,t) \neq 0$ and assume by contradiction that $\eta(s,t)\neq
  0$. Using an explicit formula (given in \cite[Lemma 3.2.2]{Bibel})
  and a cut-off function we find $Y$ supported in a small
  neighborhood of $(s,t,u(s,t))$ such that
  \[ \int_\Sigma \<\eta,Y(\pt u-X)\> \d s\d t >0\,. \] This is in
  contradiction to the second equation of~\eqref{eq:etacokernel} and
  shows that $\eta(s,t)=0$ for all points $(s,t)$ with $\ps u(s,t)\neq
  0$. Since these points are dense (\cf Lemma~\ref{lmm:finitecrit}) we
  conclude that $\eta$ vanishes restricted to $[s_0,s_1]\times [0,1]$
  and by unique continuation we conclude that $\eta$ vanishes
  altogether. This shows that universal moduli space is a Banach
  submanifold.

  We claim that the operator $D^\univ_{u,J} + \d_u \ev$ is surjective
  for all $(u,J)$ in the universal moduli space. We use an idea
  from~\cite[Lemma 2.5]{Seidel:sequence}.  As above $D^\univ_{u,J} +
  \d_u\ev$ has a closed range. Take $(\eta,\zeta^-,\zeta^+)$ in the
  cokernel. We have~\eqref{eq:etacokernel} and 
  \[
  \<\xi(-\infty),\zeta^-\> + \<\xi(\infty),\zeta^+\> =0\,,
  \]
  for all $\xi \in T_u \B^{1,p;\delta}$. With~\eqref{eq:etacokernel}
  we conclude again that $\eta$ vanishes and by the last identity we
  show that $\zeta^-$ and $\zeta^+$ vanishes because we find $\xi \in
  T_u\B^{1,p;\delta}$ such that $\xi(\pm \infty)=\zeta^\pm$.  This
  shows that $D^\univ_{u,J}+\d_u \ev$ is surjective . Hence given any
  $\zeta^- \in T_{u(-\infty)} C_-$ and $\zeta^+ \in T_{u(\infty)} C_+$
  there exists $\xi \in T_u\B^{1,p;\delta}$ and $Y \in T_J\J'$ such
  that
  \begin{equation*}
    D_{u,J}\xi+Y(\pt u - X) = 0,\qquad \xi(-\infty) = \zeta^-,\qquad \xi(\infty)=\zeta^+\;.
  \end{equation*}
  We see that  $(\xi,Y)$ is an element in the
  tangent space of $\wt \M(C_-,C_+;\J',X)$ and by the second equation that
  $\d_u\ev(\xi)=(\xi(-\infty),\xi(\infty))=\zeta$. This shows the
  claim.

  \begin{stp}\label{stp:regdense}
    We show that the subset $\J_{\reg,k} \subset \J$ is dense for all
    $k \in \N$.
  \end{stp}
  By the last step $\wt \M(W;\J',X):=\{(u,J,w) \mid u \in \wt
  \M(C_-,C_+;J,X),\ J \in \J',\ \ev(u)=\vp(w)\}$ is a Banach manifold
  with tangent space at a point $(u,J,w)$ given by the set of triples
  $(\xi,Y,v)$ such that
  \[ D_{u,J} \xi + Y(\pt u - X) =0,\qquad (\xi(-\infty),\xi(\infty)) =
  \d_w\vp(v)\,.\] Abbreviate $C:=T_{u(-\infty)} C_- \oplus
  T_{u(\infty)} C_+$ and let $\pi : C \to C/\im \d_w \vp$ the canonical
  projection.  We conclude that $(\xi,Y)$ lies in the kernel of the
  operator
  \[(\xi,Y) \mapsto (D_{u,J}\xi+ Y(\pt u - X),
  \pi(\xi(-\infty),\xi(\infty))) \in \E_u \oplus C/\im \d_w \vp\;.\] By
  \cite[Lemma A.3.6]{Bibel} we have that $J\in \J'$ is a regular value
  of the projection $\pr_2:\wt \M(W;\J',X) \to \J'$ if and only if the map $\xi
  \mapsto (D_{u,J}\xi,\pi(\xi(-\infty),\xi(\infty)))$ is surjective.
  Hence let $J$ be a regular value of $\pr_2$ and $[\zeta] \in C/\im
  \d_w \vp$ we find $\xi\in T_u\B$ such that
  \[D_{u,J}\xi=0,\hspace{2cm} \pi(\xi(-\infty),\xi(\infty)) =
  [\zeta]\,.\] We conclude that $\regval(\pr_2)\subset
  \J_{\reg,k}$. Since after Sards theorem the space of regular values
  is comeager and by Baires theorem every complete metric space has
  the property that a comeager subset is dense, the inclusion
  $\regval(\pr_2) \subset \J'$ is dense.  Since $\J' \subset \J$ is
  dense by construction we conclude that $\regval(\pr_2) \subset \J$
  is dense. Thus $\J_{\reg,k} \subset \J$ is dense.
\end{proof}

\subsubsection{Glued structures} Given two admissible vector fields $X_0$
and $X_1$ such that $X_0(s,\cdot)=X_1(-s,\cdot)$ for all $s\geq
s_0$. For every $R \geq R_0$ we define the \emph{glued vector field}
$X_R:=X_0 \#_R X_1$ via
\begin{equation}
  \label{eq:glueX}
  X_R(s,\cdot) :=
  \begin{cases}
    X_0(s+2R,\cdot)&\text{if } s \leq 0\\
    X_1(s-2R,\cdot)&\text{if } s \geq 0\,.
  \end{cases}
\end{equation}
Similarly given two admissible almost complex structures $J_0$ and
$J_1$ such that $J_0(s,\cdot)=J_1(-s,\cdot)$ for all $s$ large enough,
we define for all $R \geq R_0$ the \emph{glued almost complex
  structure} $J_R:=J_0\#_R J_1 $ via
\begin{equation}
  \label{eq:glueJ}
  J_R(s,\cdot):=  \begin{cases}
    J_0(s+2R,\cdot)&\text{if } s \leq 0\\
    J_1(s-2R,\cdot)&\text{if } s \geq 0\;.
  \end{cases}
\end{equation}
Let $\I_-$ and $\I_+$ denote the perturbed intersection points of the
Hamiltonian vector fields $X_0(-s,\cdot)$ and $X_1(s,\cdot)$
respectively for large $s$. Given smooth maps $\vp_-:W_-\to \I_-$ and
$\vp_+:W_+\to \I_+$ we denote the space
\begin{equation}
  \label{eq:MWWJRXR}
  \wt \M(W_-,W_+;J,X):= \left\{ (w_-,u,R,w_+) \left|\
      \begin{aligned}
        &\text{$u$ is $(J_R,X_R)$-holomorphic}\\
        &u(-\infty) = \vp_-(w_-)\\
        &u(+\infty) = \vp_+(w_+)
      \end{aligned}\right\}\right.\,,
\end{equation}
For all pairs $(u,R)$ such that $u$ is $(J_R,X_R)$-holomorphic we
consider the operator
\[\wh D_{u,R}: T_u\B^{1,p;\delta} \oplus \R \to \E^{p;\delta}_u,\qquad
(\xi,\theta) \mapsto D_{u,R} \xi + \theta\,\eta_R\,,\] with
$\eta_R:=(\pa_R J)\,(\pt u-X_R) - J_R(\pa_R X_R)$ and $D_{u,R}$ is the
vertical differential of the Cauchy-Riemann-Floer section (\ie the
operator~\eqref{eq:Du} with $J=J_R$ and $X=X_R$).
\begin{dfn}\label{dfn:regR} 
  The homotopy $J=(J_R)_{R \geq R_0}$ is called \emph{regular for
    $X=(X_R)_{R \geq R_0}$} if $\wh D_{u,R}$ is surjective for all
  pairs $(u,R)$ such that $R \geq R_0$ and $u$ is
  $(J_R,X_R)$-holomorphic. Moreover $J$ is called \emph{regular for
    $X$ and $(\vp_-,\vp_+)$} if additionally~\eqref{eq:MWWJRXR} is
  cut-out transversely, \ie
  \[
  \{(\xi(-\infty),\xi(\infty)) \mid (\xi,\eta) \in \ker \wh D_{u,R}\} + \im \d_w \vp = T_{u(-\infty)} \I_- \times T_{u(\infty)}\I_+\,.
  \]
  for all $(w_-,w_+,u,R) \in \wt \M(W_-,W_+;J,X)$.
\end{dfn}
\begin{rmk}
  This is of course \emph{not} equivalent to demand that for all $R
  \geq R_0$ the structure $J_R$ is regular for $X_R$ and $(\vp_-,\vp_+)$.
\end{rmk}
Fix paths of almost complex structures $J_\infty^-$, $J_\infty$ and
$J_\infty^+$ and some $s_1>s_0$. We denote by $\J_\adm$ the space of
admissible almost complex structures (\cf
Definition~\ref{dfn:JXadm}). We search for regular structures in the
space
\begin{equation*}
\J:=\left\{(J_0,J_1) \in \J_\adm \times \J_\adm \left| \ \begin{aligned}
      J_\infty^-&=J_0(-s,\cdot)\,,\\
      J_\infty&=J_0(s,\cdot)=J_1(-s,\cdot)\,,\\
      J_\infty&=J_1(s,\cdot)\,,\forall\ s \geq s_1
    \end{aligned}
  \right\}\right.\,.
\end{equation*}
\begin{thm}\label{thm:regR}
  The subspace of almost complex structures which are regular for $X$
  and $\vp$ is comeager in $\J$.
\end{thm}
\begin{proof}
  Since we apply the same principle ideas from the proof of
  Theorem~\ref{thm:reg} we give just sketch. Let $\J_\reg \subset \J$
  be the subset of regular pairs. Fix constants $p>2$ and $\mu>0$
  small enough. For any $k \in \N$ we define $\J_{\reg,k} \subset \J$
  consisting of all pairs $(J_0,J_1) \in \J$ with the property that
  for all $R\geq R_0$ with $R\leq k$ the operator $\wh D_{u,R}$ is
  surjective for all $(J_R,X_R)$-holomorphic curves $u$ which satisfy
  additionally $\nm{\ps u(s,t)}\leq ke^{-\mu\nm{s}}$ for all $(s,t)
  \in \Sigma$ and for all $w \in W_k$ with $\vp(w)
  =(u(-\infty),u(\infty))$ we have $\{(\xi(-\infty),\xi(\infty))\mid
  (\xi,\theta) \in \ker \wh D_{u,R}\} + \im \d_w \vp = T_{u(-\infty)}
  \I_- \oplus T_{u(\infty)} \I_+$. By the same arguments as in the
  first step of the proof of Theorem~\ref{thm:reg} we show that
  $\J_{\reg,k}$ is open for all $k\in\N$.

  We show that the corresponding universal moduli space is a Banach
  manifold and the evaluation map is a submersion.  As above we construct a
  separable Banach space $\J'$ such that $\J' \subset \J$ is
  dense. Define the universal moduli space
\[\wt \M(C_-,C_+;\J',X):=\{(u,J,R) \mid u\in\wt
  \M(C_-,C_+;J_R,X_R),\ J\in \J'\}\,.\] The vertical differential of
  the perturbed Cauchy-Riemann operator is
  \begin{equation*}
    \wh D^\univ_{u,R}: T_u\B^{1,p;\delta} \oplus T_J
    \J' \oplus \R \to \E_u^{p;\delta},\qquad
    (\xi,Y_0,Y_1,\theta) \mapsto \wh D_{u,R}(\xi,\theta)  + Y_R\big(\pt u -
    X_R\big)\,,
  \end{equation*}
  where $Y_R:\Sigma \times M \to \End(TM)$ is defined by
  \[
  Y_R(s,\cdot) :=
  \begin{cases}
    Y_0(s+2R,\cdot)&\text{if } s \leq 0\\
    Y_1(s-2R,\cdot)&\text{if } s \geq 0\;.
  \end{cases}
  \]
  Let $D^\univ_{u,R}$ denote the restriction of $\wh D^\univ_{u,R}$ to
  the subspace $T_u\B^{1,p;\delta} \oplus T_J \J'$.  We claim that
  $D^\univ_{u,R}$ is surjective.  Take any element $\eta \in
  \E_u^{p;\delta}$ in the annihilator of the image. Then for all $\xi
  \in T_u\B$ and $(Y_0,Y_1) \in T_J\J'$ we have
  \begin{equation*}
    \int_\Sigma\<D_{u,R} \xi,\eta\> = 0,\qquad
    \int_\Sigma \< Y_R(\pt u - X_R),\eta\> =0\,.
  \end{equation*}
  The first equation shows that $\eta$ is smooth as in the proof of
  Theorem~\ref{thm:reg}. Choose a point $(s,t) \in [2R+s_0,2R+s_1]
  \times [0,1]$ such that $\ps u(s,t)\neq 0$ and suppose by
  contradiction that $\eta(s,t) \neq 0$. We find an infinitesimal
  almost complex structure $Y=(0,Y_1) \in T_J\J'$ supported in a
  small neighborhood about $(s-2R,t,u(s,t))$ such that $\int_\Sigma
  \<Y_RJ_R \ps u,\eta)\> >0 $ in contradiction the the second
  equation. Hence $\eta(s,t)$ vanishes on points $(s,t)\in
  [2R+s_0,2R+s_1]\times [0,1]$ with $\ps u(s,t) \neq 0$. Because such
  points are dense in $[2R+s_0,2R+s_1]\times [0,1]$ we conclude by
  continuity that $\eta$ restricted to $[2R+R_0,2R+R_1] \times [0,1]$
  vanishes and by unique continuation we see that $\eta$ vanishes
  everywhere.  The rest of the proof follows word by word from the
  proof of Theorem~\ref{thm:reg}.
\end{proof}
\subsubsection{Homotopies} Given homotopies $X=(X_R)_{R \in [a,b]}$
and $J=(J_R)_{R \in [a,b]}$ of admissible vector fields and almost
complex structures respectively we similarily denote by $\wt
\M(W_-,W_+;J,X)$ the space of tuples $(w_-,R,u,w_+)$ such that $R\in
[a,b]$, the map $u$ is $(J_R,X_R)$-holomorphic and $u(\pm \infty) \in
\vp_\pm(w_\pm)$, similarily to~\eqref{eq:MWWJRXR}. For technical
reasons we require that $X_R(-s,\cdot)=X_{H_-}$ and
$X_R(s,\cdot)=X_{H_+}$ for all $s \geq s_0$ and two fixed Hamiltonians
$H_-$ and $H_+$ which do not depend on $R$. We say that $J=(J_R)_{R
  \in [a,b]}$ is \emph{regular for $X=(X_R)_{R\in [a,b]}$ and $\vp$}
similarly to Definition~\ref{dfn:regR}. Given two admissible almost
complex structures $J_a$, $J_b$ such that
$J_-:=J_a(-s,\cdot)=J_b(-s,\cdot)$ and
$J_+:=J_a(s,\cdot)=J_b(s,\cdot)$ for all $s \geq s_1$. We search for
regular almost complex structures in the space $\J(J_a,J_b)$ which is
the space of smooth homotopies $(J_R)_{R\in [a,b]}$ from $J_a$ to
$J_b$ such that $J_R(- s,\cdot)=J_-$ and $J_R(s,\cdot)=J_+$ for all $s
\geq s_1$ and $R \in [a,b]$.
\begin{thm}\label{thm:regh}
  If $s_1 >s_0$ the subspace structures which are regular for $X$ and
  $(\vp_-,\vp_+)$ in $\J(J_a,J_b)$ is comeager.
\end{thm}
\begin{proof}
  The proof is completely analogous to the proof of
  Theorem~\ref{thm:reg} and~Theorem~\ref{thm:regR}. Note that the
  space of infinitesimal almost complex structures is given by section
  supported in the compact cube $[a,b]\times [-s_1,s_1]\times [0,1]$
  and hence the resulting Banach spaces are separable.
\end{proof}
\subsection{$\R$-invariant structures}
Let $X=X_H$ be the Hamiltonian vector field for some clean Hamiltonian
$H\in C^\infty([0,1] \times M)$.  In this section we construct regular
structures in a set of almost complex structures
\[\J:=C^\infty([0,1],\End(TM,\omega))\,.\]
For any $J \in \J$ we denote by $\M_J$ the space of
$(J,X)$-holomorphic strips and by $\I$ the space of perturbed
intersection points of $H$. We also give a transversality result of the
evaluation of tuples of $(J,X)$-holo\-morphic maps given by  
\begin{equation}
  \label{eq:ev}
  \begin{gathered}
  \ev:\M_J^m \to \I^{2m},\\
  (u_1,\dots,u_m) \to (u_1(-\infty),u_1(\infty),u_2(-\infty),\dots,u_m(-\infty),u_m(\infty))\,.
  \end{gathered} 
\end{equation}
The difficulty here lies in the fact that we need to perturbed $J$
simultaneously for the curves $u_i$ and $u_j$, which is obviously not
possible if $u_i$ and $u_j$ have the exact same image. For that reason
we define the notion of a \emph{distinct tuple}. 
\begin{dfn}
  A tuple $(u_1,\dots,u_m)$ of maps $\Sigma \to M$ is called
  \emph{distinct}, if for all $i \neq j$ and $a \in \R$ we have
  $u_i \neq u_j(a+\cdot,\cdot)$. 
\end{dfn}
\begin{rmk}
  Distinct tuples should not be confused with the stronger notion of
  \emph{absolutely distinct} tuples as defined in
  \cite{BiranCornea:pearl}. The transversality theory in
  \cite{BiranCornea:pearl} is more difficult since the authors
  achieve transversality for domain-independent almost complex
  structures.
\end{rmk}
\begin{dfn}\label{dfn:regs}
  Given a smooth map $\vp:W \to \I_H(L_0,L_1)^{2m}$. The almost
  complex structure $J \in \J$ is \emph{regular for $X$ and $\vp$} if
  $J$ is regular for $X$ and $\vp$ transverse to the evaluation
  map~\eqref{eq:ev} restricted to the space of distinct tuples.
\end{dfn}
\begin{thm}\label{thm:sreg}
  The subspace of $J$ which are regular for $X$ and $\vp$ is comeager
  in $\J$.
\end{thm}
\begin{proof}\setcounter{stp}{0}
  For $\mu>0$ we define $\J^\mu \subset \J$ as the open subspace of
  all $J$ with $\mu<\iota(J,H)$. Choose an exhaustion $W = \bigcup_k
  W_k$ by compact subsets $W_k$ such that $W_k \subset W_{k+1}$ for
  all $k\in \N$. For $\mu>0$ and $k\in \N$ we denote by
  $\J_{\reg,k}^\mu \subset \J^\mu$ the space of all $J \in \J^\mu$
  such that the operator $D_{u,J}$ is surjective for all
  $(J,X)$-holomorphic strips $u$ which satisfy~\eqref{eq:uR} and for
  all distinct tuples $(u_1,\dots,u_m)$ of $(J,X)$-holomorphic strips
  which satisfy~\eqref{eq:uR} and $w \in W$ such that $\vp(w) = \ev(u)$ 
  we have that the image of $\d_w \vp$ is a complement of
  \begin{equation}\label{eq:evkerDu}
  \{(\xi_1(-\infty),\xi_1(\infty),\xi_2(-\infty),\dots,\xi_m(\infty))
  \mid \xi_j \in  \ker D_{u_j,J},\ j=1,\dots,m \}\,.    
  \end{equation}
  First we that $\J^\mu_{\reg,k} \subset \J^\mu$ is open as in the
  proof of Theorem~\ref{thm:reg}.  To show that that $\J^\mu_{\reg,k}
  \subset \J^\mu$ is dense we proceed as follows.  Let $\J'\subset \J$
  be the dense subspace which is a separable Banach manifold. Then
  $\J'^{\mu} := \J'\cap \J^\mu$ is also separable Banach manifold
  which is dense in $\J^\mu$. Let $C=(C_1,C_2,\dots,C_{2m})$ be a
  tuple of connected components in $\I_H(L_0,L_1)$. Abbreviate
  $\B:=\B^{1,p;\delta}(C_1,C_2) \times \B^{1,p;\delta}(C_3,C_4)\times
  \dots \B^{1,p;\delta}(C_{2m-1},C_{2m})$.  Define the universal
  moduli space
  \[
  \wt \M(C;\J'^\mu,X) \subset \B \times \J'^\mu\,,
  \]
  to be the space of $(u_1,\dots,u_m,J)$ where $J \in \J'^\mu$ and
  $(u_1,\dots,u_m)$ is a distinct tuple of $(J,X)$-holomorphic
  strips. We want to show that $\wt \M(C;\J'^\mu,X)$ is a Banach
  manifold and that the evaluation map is a submersion
  \[ \ev:\wt \M(C;\J'^\mu,X) \to C_1 \times C_2\times \dots \times
  C_m, \qquad (u_1,\dots,u_m,J) \mapsto \ev(u_1,\dots,u_m)\,.\] For
  $(u,J)=(u_1,\dots,u_m,J) \in \wt \M(C;J'^\mu,X)$ with
  $(x_1,x_2,\dots,x_{2m})=\ev(u)$ consider the operator
  \begin{gather*}
    T_u \B \oplus T_J \J' \to \E_{u_1}^{p;\delta} \oplus \E_{u_2}^{p;\delta} \oplus \dots \E_{u_m}^{p;\delta} \oplus T_{x_1}C_1 \oplus T_{x_2} C_2 \oplus \dots  \oplus T_{x_{2m}} C_{2m}\\
    (\xi_1,\xi_2,\dots,\xi_m,Y) \mapsto
    (D^\univ_{u_1,J}(\xi_1,Y),D^\univ_{u_2,J}(\xi_2,Y),
    \dots,D^\univ_{u_m,J}(\xi_m,Y),\d_u \ev(\xi)).
  \end{gather*}
  We claim that the operator is surjective. It suffices to show that
  the cokernel is trivial. Given an element
  $(\eta,\zeta)=(\eta_1,\dots,\eta_m,\zeta_1,\dots,\zeta_{2m})$ in the
  cokernel. For any $\xi=(\xi_1,\dots,\xi_m) \in T_u\B$ and $Y \in T_J
  \J'^\mu$ the following terms vanish
  \begin{equation*}
    \int \< D_{u_j,J} \xi_j,\eta_j\>\d s \d t\,,\qquad
    \sum_{j=1}^m \int \<Y \big(\pt u_j -X\big),\eta_j\> \d s
    \d t\,,
  \end{equation*}
  as well as
  $\<\xi_j(-\infty),\zeta_{2j-1}\>+\<\xi_j(\infty),\zeta_{2j}\>$ for
  all $j=1,\dots,m$. By the first term we see that $\eta_j$ is smooth
  for all $j=1,\dots,m$. By the last term we see that $\zeta$ vanishes
  since there exists $\xi$ with $\zeta = \d_u \ev(\xi)$. To prove that
  also $\eta$ vanishes choose a regular point $(s,t) \in
  \RR(u_1,\dots,u_m)$ (\cf Definition~\ref{dfn:regsequence}). Then for
  any $j=1,\dots,m$ we find $Y$ supported in a small neighborhood of
  $(t,u_j(s,t))$ such that
  \[
  \sum_{j=1}^m \int \<Y\,(\pt u_j - X),\eta_j\> = \int \<Y\, (\pt u_j -X),\eta_j\> >0\,,
  \]
  in contradiction to the second equation. By
  Proposition~\ref{prp:regdenseseq} regular points are dense. Hence
  $\eta$ vanishes. We conclude that $\J^\mu_{\reg,k} \subset \J^\mu$
  is dense.  The union
  \[
  \J_{\reg,k} :=\bigcup_{\mu>0} \J^\mu_{\reg,k} \subset \J\;,
  \]
  is open. It is also dense because for any $J \in \J$, there exists
  $\mu>0$ such that $J \in \J^\mu$ and we find $J_\nu \in
  \J_{\reg,k}^\mu \subset \J_{\reg,k}$ converging to $J$. This shows
  that $\J_{\reg,k} \subset \J$ is dense and open. Hence $\J_\reg =
  \bigcap_k \J_{\reg,k}$ is comeager.
\end{proof}

\subsection{Regular points}
Fix $X=X_H$ be a Hamiltonian vector field for a clean Hamiltonian $H
\in C^\infty([0,1]\times M)$ and $J:[0,1]\to \End(TM,\omega)$ be a
path of almost complex structures.  We abbreviate the strip $\Sigma =
\R \times [0,1]$. Recall that a tuple of maps $(u_1,\dots,u_m)$ is
\emph{distinct} if $u_i = u_j \circ \tau_a$ for some $a \in \R$
implies that $j=i$.
\begin{dfn}\label{dfn:regsequence}
  Let $(u_1,\dots,u_m)$ be a tuple of finite energy
  $(J,X)$-holo\-morphic strips with boundary in $(L_0,L_1)$. A point
  $(s,t) \in \Sigma$ is called a \emph{regular point for
    $(u_1,\dots,u_m)$} if
  \begin{enumerate}[label=(\roman*)]
  \item $\ps u_j(s,t) \neq 0$ for all $j =1,\dots,m$
  \item $u_i(s,t) \neq \lim_{s' \to \pm \infty} u_j(s',t)$ for all $i,j
    =1,\dots,m$
  \item for all $s' \in \R$ we have: $u_i(s,t) = u_j(s',t) \iff s'=s$ and $j=i$.
  \end{enumerate}
  We denote this set of points by $\RR(u_1,\dots,u_m)\subset \Sigma$.
\end{dfn}
\begin{lmm}\label{lmm:finitecrit}
  Let $u$ be a $(J,X)$-holomorphic strip such that $\ps u\not\equiv
  0$, then the set $C(u):=\{ (s,t) \in \R \times [0,1] \mid \ps u(s,t)
  =0 \}$ is finite.
\end{lmm}
\begin{proof}
  By a change of variables we assume $H=0$ and that $L_0$ and $L_1$
  intersect cleanly (see Lemma~\ref{lmm:change}). By the asymptotic
  analysis we know that there exists $s_0 \in \R$ such that $\ps u
  (s,t) \neq 0$ for all $\nm{s} \geq s_0$ and $t \in [0,1]$ (\cf
  Corollary~\ref{cor:nonzero}). By \cite[Lemma 2.3]{Floer:Trans} the
  set of critical points is discrete. This shows the claim.
\end{proof}
\begin{prp}\label{prp:regdenseseq}
  Let $(u_1,\dots,u_m)$ be a tuple of $(J,X)$-holomorphic strips which
  is distinct, then $\RR(u_1,\dots,u_m) \subset \Sigma$ is open and
  dense.
\end{prp}
\begin{proof}\setcounter{stp}{0}
  The proof goes along the lines of \cite[Theorem 4.3]{Floer:Trans} or
  \cite[Theorem 4.9]{Frauenfelder:PhD}. Without loss of generality we
  assume that $H\equiv0$ and $L_0$, $L_1$ intersect cleanly (\cf
  Lemma~\ref{lmm:change}). We abbreviate the points $u_j(\pm \infty):=
  \lim_{s \to \pm \infty} u(s,\cdot) \in L_0 \cap L_1$ for
  $j=1,\dots,m$ and $\RR:=\RR(u_1,\dots,u_m)$.

  \begin{stp}
    We show that $\RR \subset \Sigma$ is open.
  \end{stp}
  By contradiction let $(s_\nu,t_\nu)\subset \Sigma \setminus \RR$ be
  a sequence such that $\lim_{\nu \to \infty} (s_\nu,t_\nu) = (s,t)
  \in \RR$. Hence for all $\nu \in \N$ at least one of the following
  statements holds
  \begin{enumerate}[label=(\roman*)]
  \item $\ps u_j(s_\nu,t_\nu) = 0$ for some $j=1,\dots,m$
  \item $u_j(s_\nu,t_\nu)=u_i(-\infty)$ or $u_j(s_\nu,t_\nu) =
    u_i(\infty)$ for some $i,j=1,\dots,m$,
  \item $u_i(s_\nu,t_\nu) = u_j(s'_\nu,t_\nu)$ for some $s'_\nu \in
    \R$ and $i \neq j$
  \item $u_j(s_\nu,t_\nu)=u_j(s'_\nu,t_\nu)$ for some $s'_\nu \neq
    s_\nu$ and $j=1,\dots,m$
  \end{enumerate}
  In the first case we argue by continuity that $\ps u_j(s,t)=0$ in
  contradiction to $(s,t) \in \RR$. Similarly we exclude the second
  case. Suppose that the third statement holds after passing to a
  subsequence for all $\nu \in \N$.  If $(s_\nu')$ is unbounded then
  without loss of generality we have $s_\nu'\to \infty$ hence
  $u_j(\infty) \leftarrow u_j(s_\nu',t_\nu)=u_i(s_\nu,t_\nu) \to
  u_i(s,t)$, which contradicts the fact that $(s,t)$ is a regular
  point. If $(s_\nu')$ is bounded, then after possibly passing to a
  subsequence we have $s_\nu' \to s'$ and $u_j(s',t) \leftarrow
  u_j(s_\nu',t_\nu)=u_i(s_\nu,t_\nu)\to u_i(s,t)$, hence
  $u_j(s',t)=u_i(s,t)$, which again contradicts the fact that $(s,t)$
  is a regular point. Suppose finally that the last case holds after
  passing to a subsequence for all $\nu \in \N$. If $(s'_\nu)$ is
  unbounded, then without loss of generality $s'_\nu \to \infty$ and
  we obtain $u_j(s,t) \leftarrow u_j(s_\nu,t_\nu) =u_j(s'_\nu,t_\nu)
  \to u_j(\infty)$. This shows that $u_j(s,t) =u_j(\infty)$ in
  contradiction to $(s,t) \in \RR$. If on the other hand $(s'_\nu)$ is
  bounded then $s'_\nu \to s'$ without loss of generality. If $s'\neq
  s$ we conclude that $u_j(s,t)=u_j(s',t)$ in contradiction to the
  fact that $(s,t)$ is regular and if $s'=s$ we conclude that $\ps
  u_j(s,t) =0$ which again contradicts the fact that $(s,t)$ is
  regular. We conclude that $\RR$ is open.
  \begin{stp}
    We show that $\RR \subset \Sigma$ is dense under the additional
    assumption that $m=1$.
  \end{stp}
  Write $u=u_1$.  Given any point $(s_1,t_1) \in \Sigma$ and
  $\e>0$. We have to show that there exists a regular point in the
  ball $B_\e(s_1,t_1)$. Let $\regval(u) \subset M$ be the space of
  regular values of $u$. Since by Lemma~\ref{lmm:finitecrit} the set
  of critical points is finite we assume after possibly replacing
  $(s_1, t_1)$ by a point which is $\e$-close and decreasing $\e$,
  that for all $(s,t) \in B_\e(s_1,t_1)$ we have
  \begin{equation}
    \label{eq:uBeinreg}
    u(s,t) \in \regval (u), \qquad \qquad u(s,t) \neq u(\pm \infty).
  \end{equation}
  We claim that this implies that for all $(s,t) \in B_\e(s_1,t_1)$
  the set $u^{-1}(u(s,t))$ is finite. Indeed, assume by contradiction
  that we find a sequence $(s_\nu,t_\nu) \subset \Sigma$ consisting of
  distinct points and $u(s_\nu,t_\nu)=u(s,t)$ for some $(s,t) \in
  B_\e(s_1,t_1)$ and all $\nu \in \N$. If $(s_\nu)$ is unbounded then
  after possibly passing to a subsequence we have $s_\nu \to \pm
  \infty$ and $u(s,t) = u(s_\nu,t_\nu) = u(\pm \infty)$ in contradiction
  to~\eqref{eq:uBeinreg} and if $(s_\nu)$ is bounded then after
  possibly passing to a subsequence we have $s_\nu \to s'$ and $t_\nu
  \to t'$, which shows that $u(s_\nu,t_\nu)=u(s',t')=u(s,t)$ for all
  $\nu \in \N$ and hence $\d u(s',t')=0$ in contradiction to $u(s,t)
  \in \regval(u)$. Now define
  \begin{equation}
    \label{eq:preus1t1}
  u^{-1}(u(s_1,t_1))\cap \R \times \{t_1\} =
  \{(s_1,t_1),(s_2,t_1),\dots,(s_\ell,t_1)\}\;.
  \end{equation}
  For $\delta>0$ we define 
  \[
  F_\delta:= \{ (s,t) \in \Sigma \mid \exists\ (s',t) \in
  B_\delta(s_1,t_1), \ u(s,t) = u(s',t) \}\;.
  \]
  We claim that for all $r>0$ there exist $\delta >0$ such that
  \begin{equation}
    \label{eq:FdeltainBrho}
    F_\delta \subset  B_r(s_1,t_1) \cup B_r(s_2,t_1) \cup \dots \cup B_r(s_\ell,t_1)\;.
  \end{equation}
  If not then we find $r>0$ and sequences $(s'_\nu), (s_\nu) \subset
  \R$, $(t_\nu) \subset [0,1]$ with $u(s'_\nu,t_\nu)=u(s_\nu,t_\nu)$,
  $(s_\nu,t_\nu) \to (s_1,t_1)$ and $(s_\nu',t_\nu) \not\in
  B_r(s_j,t_1)$ for any $j = 1,\dots, \ell$. If $(s_\nu')$ is
  unbounded we find a diverging subsequence $s_\nu'\to \pm \infty$ and
  we conclude $u(s_1,t_1) \leftarrow u(s_\nu,t_\nu)=u(s_\nu',t_\nu)
  \to u(\pm\infty)$ in contradiction to $u(s_1,t_1) \neq u(\pm
  \infty)$. On the other hand if $(s_\nu')$ is bounded by possibly
  passing to a subsequence we assume without loss of generality that
  $s_\nu' \to s'$, and $u(s',t_1) \leftarrow
  u(s_\nu',t_\nu)=u(s_\nu,t_\nu) \to u(s_1,t_1)$. Hence $u(s',t_1) =
  u(s_1,t_1)$ and by~\eqref{eq:preus1t1} we have $s' \in
  \{s_1,\dots,s_\ell\}$. But this is in contradiction to $s' \not\in
  B_{r/2}(s_j,t_1)$ for all $j = 1,\dots,\ell$. We
  conclude~\eqref{eq:FdeltainBrho}.

  By possibly decreasing $\e$ again we assume that $u$ restricted to
  $B_\e(s_j,t_1)$ is an embedding for all $j = 1,\dots,\ell$ and
  for all $i,j = 1,\dots,\ell $ with $i\neq j$ we have
  \begin{equation}
    \label{eq:sjsifar}
    (s_j-\e,s_j+\e) \cap (s_i-\e,s_i+\e) = \emptyset.
  \end{equation}
  Note that by~\eqref{eq:uBeinreg} the map $u$ is already an
  immersion restricted to $B_\e(s_j,t_1)$. Choose $\delta<\e$ such
  that~\eqref{eq:FdeltainBrho} holds for $r=\e$. We assume that $\ell
  \geq 2$ because otherwise $(s_1,t_1)\in \RR(u)$ and we are
  finished. Let $\mathrm{cl}\big(A\big)$ denote the closure of any
  subset $A \subset \Sigma$. For $j = 2,\dots, \ell$ we define
  \[
  \Sigma_j := \{ (s,t) \in \mathrm{cl}\big(B_{\delta/2}(s_1,t_1)\big)
    \mid \exists\ (s',t) \in B_{\e}(s_j,t_1),\
    u(s,t) =u(s',t)\}\;.
  \]
  By~\eqref{eq:FdeltainBrho} we obtain the same set when replacing
  $B_\e(s_j,t_1)$ with $\mathrm{cl}\big(B_\e(s_j,t_1)\big)$ in the
  definition, which implies that $\Sigma_j$ is closed for all $j
  =2,\dots,\ell$. Again by~\eqref{eq:FdeltainBrho} we have
  \[
  \mathrm{cl}\big(B_{\delta/2}(s_1,t_1)\big) =
  \mathrm{cl}\big(\RR(u)\cap B_{\delta/2}(s_1,t_1)\big) \cup \Sigma_2 \cup \Sigma_3 \cup \dots \cup \Sigma_\ell\;.
  \]
  Suppose by contradiction that $\RR(u) \cap B_{\delta/2}(s_1,t_1) =
  \emptyset$. Since for all $j = 2,\dots,\ell$ the set
  $\Sigma_j$ is closed, there must exists $j_0 =2,\dots,\ell$
  such that $\Sigma_{j_0}$ contains an open subset. We assume without
  loss of generality that there exists $\rho>0$ and $(\hat s_1,\hat
  t_1) \in B_{\delta/2}(s_1,t_1)$ such that $B_\rho(\hat s_1,\hat t_1)
  \subset \Sigma_2$. By possibly making $\rho$ even smaller we assume
  that $B_\e(s_2,t_1) \cap B_\rho(\hat s_1,\hat t_1) = \emptyset$. Define 
  \[
  \Omega:= u^{-1}(u(B_\rho(\hat s_1,\hat t_1)) \cap B_\e(s_2,t_1)
  \subset \Sigma\;,
  \]
  which is an open subset because $u$ restricted to $B_\rho(\hat
  s_1,\hat t_1)$ is an embedding. We have the diffeomorphism
  \[
  \phi := u_2^{-1} \circ u_\rho : B_\rho(\hat s_1,\hat t_1)
  \stackrel{\sim}{\longrightarrow} \Omega\;,
  \]
  where $u_2$ and $u_\rho$ denotes the map $u$ restricted to
  $B_\e(s_2,t_1)$ and $B_\rho(\hat s_1,\hat t_1)$ respectively. In
  particular for all $(s,t) \in B_\rho(\hat s_1,\hat t_1)$ there
  exists uniquely $(s'',t'') = \phi(s,t) \in \Omega$ such that
  $u(s'',t'')= u(s,t)$. On the other hand by construction there exists
  $(s',t) \in B_\e(s_2,t_1)$ such that $u(s,t)=u(s',t)$. This implies
  that $(s',t) \in \Omega$ and by uniqueness $(s',t)=(s'',t'')$. We
  see that $\phi(s,t) = (\kappa(s,t),t)$ for some map
  $\kappa:B_\rho(\hat s_1,\hat t_1) \to \R$ or equivalently
  $u(\kappa(s,t),t) = u(s,t)$. Since $u$ is $J$-holomorphic we compute
  \begin{equation}
    \label{eq:uJkappa}
    0 = \ps u + J \pt u = \ps u \ps \kappa + J (\ps u \pt \kappa +
  \pt u) = \ps u \, ( \ps \kappa -1) + \pt u \,\pt \kappa\;.
  \end{equation}
  Since $u$ restricted to $B_\rho(\hat s_1,\hat t_1)$ is an immersion,
  we see that $\pt \kappa \equiv 0$ and $\ps \kappa \equiv 1$. This
  implies that there exists $a \in \R$ such that
  $\kappa(s,t)=\kappa(s)=s+a$. We claim that $a \neq 0$. Assume by
  contradiction that $a =0$, then we have $\kappa(\hat s_1) = \hat s_1
  \in (s_2-\e,s_2+\e)$ and $\hat s_1 \in (s_1-\e,s_1+\e)$. But
  after~\eqref{eq:sjsifar} the sets $(s_2-\e,s_2+\e)$ and
  $(s_1-\e,s_1+\e)$ have an empty intersection.  We have deduced that
  $u(s+a,t) =u(s,t)$ for all $(s,t) \in B_\rho(\hat s_1,\hat t_1)$ and
  by unique continuation we have $u \equiv u\circ \tau_a$ with $a \neq
  0$. This contradicts the fact that the energy of $u$ is finite.
  \begin{stp}
    We proof that $\RR$ is dense with any $m$.
  \end{stp}
  Given a point $(s_1,t_1) \in \Sigma$ and $\e>0$.  By possibly
  replacing $(s_1,t_1)$ with a point which is $\e$-close and
  decreasing $\e$ we assume that for all $(s,t) \in B_\e(s_1,t_1)$ and
  $i,j = 1,\dots,\ell$ we have
  \begin{equation}
    \label{eq:uiregval}
    u_i(s,t) \in \regval(u_j),\qquad \qquad u_i(s,t) \neq u_j(\pm
    \infty)\;.
  \end{equation}
  We claim that the set $u^{-1}_j(u_i(s,t))$ is finite for all $(s,t)
  \in B_\e(s_1,t_1)$ and $i,j =1,\dots, m$. Suppose by contradiction
  that there exists $i,j$ and a sequence $(s_\nu,t_\nu)$ of distinct
  points such that $u_j(s_\nu,t_\nu)=u_i(s,t)$. If $(s_\nu)$ is
  unbounded, then without loss of generality $s_\nu \to \pm \infty$,
  $t_\nu \to t$ and hence $u_j(\pm \infty) =u_j(s_\nu,t_\nu) =
  u_i(s,t)$ contradicting~\eqref{eq:uiregval}. If $(s_\nu)$ is
  bounded, then without loss of generality $s_\nu \to s'$, $t_\nu \to
  t$, $u_j(s',t)=u_j(s_\nu,t_\nu) = u_i(s,t)$ and hence $\ps
  u_j(s',t)=0$. This contradicts the fact that $u_i(s,t) \in
  \regval(u_j)$.
  
  By the last step and yet again moving $(s_1,t_1)$ and decreasing
  $\e$, we assume without loss of generality that $B_\e(s_1,t_1)
  \subset \RR(u_j)$ for all $j = 1,\dots,m$.  Define
  \[
  \bigcup_{1\leq i,j\leq m} u_i^{-1}(u_j(s_1,t_1)) \cap \R \times
  \{t_1\} = \{ (s_1,t_1),(s_2,t_1),\dots,(s_\ell,t_1)\}\;.
  \]
  For $\delta>0$ define
  \[
  F_{\delta} := \{(s,t)\in \Sigma \mid \exists\ (s',t) \in
  B_\delta(s_1,t_1)\text{ with } u_i(s',t)=u_j(s,t) \text{ for some } i,j\}\;.
  \]
  By the same argument as in last step we conclude that for all $r>0$
  there exists $\delta>0$ such that
  \begin{equation}
    \label{eq:FdeltainBrseq}
    F_\delta \subset B_r(s_1,t_1) \cup B_r(s_2,t_1) \cup \dots \cup B_r(s_\ell,t_1)\;.
  \end{equation}
  Fix $\delta<\e$ such that~\eqref{eq:FdeltainBrseq} holds for
  $r=\e$. For $k=1,\dots, \ell$ and $i\neq j$ we define
  \[
  \Sigma_{i,j,k} := \{ (s,t) \in
  \mathrm{cl}\big(B_{\delta/2}(s_1,t_1)\big) \mid \exists\ (s',t) \in
  B_\e(s_k,t_1) \text{ with } u_i(s',t) = u_j(s,t) \}\;.
  \]
  By~\eqref{eq:FdeltainBrseq} and the assumption that $B_\e(
  s_1,t_1) \subset \RR(u_j)$ for all $j=1,\dots,m$ we have
  \[
  \mathrm{cl}\big(B_{\delta/2}(s_1,t_1)\big) =
  \mathrm{cl}\big(\RR(u_1,\dots,u_m) \cap
  B_{\delta/2}(s_1,t_1)\big) \cup \bigcup_{i,j,k} \nolimits
  \Sigma_{i,j,k}\;.
  \]
  Arguing indirectly assume that $\RR(u_1,\dots,u_m) \cap
  B_{\delta/2}(s_1,t_1) = \emptyset$. Without loss of generality there
  exists $\rho>0$ and $(\hat s_1,\hat t_1) \in B_{\delta/2}(s_1,t_1)$
  such that $B_\rho(\hat s_1,\hat t_1) \subset \Sigma_{2,1,k}$ for
  some $k= 1,\dots,\ell$. Define the open subset
  \[
  \Omega:= u_1^{-1}\big(u_2(B_\rho(\hat s_1,\hat t_1))\big)\cap
  B_\e(s_k,t_1) \subset \Sigma.
  \]
  Hence for all $(s,t) \in B_\rho(\hat s_1,\hat t_1)$ we find
   $(s'',t'') \in \Omega$ such that
  $u_1(s'',t'')=u_2(s,t)$. Moreover $(s'',t'')$ is unique. On the other hand since $B_\rho(\hat
  s_1,\hat t_1) \subset \Sigma_{2,1,k}$ there exists $(s',t) \in
  B_\e(s_k,t_1)$ such that $u_1(s',t) =u_2(s,t)$. This implies that
  $(s',t) \in \Omega$ and by uniqueness $(s'',t'')=(s',t)$. We
  conclude that there exists a map $\kappa:\B_\rho(\hat s_1,\hat t_1)
  \to \R$ such that $u_1(\kappa(s,t),t) = u_2(s,t)$ for all $(s,t) \in
  B_\rho(\hat s_1,\hat t_1)$. Since both $u_1$ and $u_2$ are
  $J$-holomorphic we conclude by a computation similar
  to~\eqref{eq:uJkappa} that $\kappa(s,t)=\kappa(s) =s+a$ for some $a
  \in \R$. Hence $u_1(s+a,t)= u_2(s,t)$ for all $(s,t) \in B_\rho(\hat
  s_1,\hat t_1)$ and by unique continuation $u_1 \equiv u_2 \circ
  \tau_a$ in contradiction to the fact that the tuple
  $(u_1,\dots,u_m)$ is distinct.
\end{proof}

\subsection{Floer's $\e$-norm}
Fix some $p>1$, $\Omega \subset \R^\ell$ a subset with Lipschitz type
boundary and a sequence $(\varepsilon_k)_{k \in \mathbb{N}}$ of
positive numbers $\e_k >0$. Given a smooth function with compact
support $f \in C^\infty_0(\Omega)$ we define the norm
\[
\Nm{f}_\e := \sum_{k \geq 0} \e_k \Nm{f}_{H^{k,p}(\Omega)}\;,
\]
and the subspace $C^{\e}_0(\Omega) \subset C^\infty_0(\Omega)$ by
\[
C^\e_0(\Omega) = \{ f \in C^\infty_0(\Omega) \ | \ \Nm{f}_\e < \infty\}\;.
\]
Floer originally used $C^k$-norms instead of Sobolev norms, but after
the Sobolev embedding theorem the norm defined here is equivalent. We
have chosen this approach because it suits better when considering
domains with boundary.
\begin{lmm}\label{lmm:eBanach}
  If $\Omega$ is bounded then the space $C^\e_0(\Omega)$ with the
  topology induced by the norm $\Nm{\cdot}_\e$ is a complete and
  separable space. In particular $(C^\e_0(\Omega),\Nm{\cdot}_\e)$ is a separable
  Banach space.
\end{lmm}
\begin{proof}
  See~\cite[Lemma 4.2.7]{Schwarz:PhD} and~\cite[Lemma
  4.2.9]{Schwarz:PhD}.
\end{proof}
\noindent Clearly $C^\e_0(\Omega) \subset C^\infty_0(\Omega)$ is
continuous. The next lemma states that for certain sequences this
inclusion is dense. It is a slight generalization of \cite[Lemma
4.2.8]{Schwarz:PhD} allowing boundary values.
\begin{lmm}\label{lmm:edense}
  Given $\ell \in \mathbb{N}$ there exists a sequence
  $(\varepsilon_k)$ such that the inclusion $C^\e_0(\Omega) \subset
  C^\infty_0(\Omega)$ is dense for all subsets $\Omega \subset
  \R^\ell$ with Lipschitz type boundary.
\end{lmm}
\begin{proof}
  Fix some $p>1$ and denote by $B_r \subset \R^\ell$ the ball of
  radius $r >0$ centered at the origin. Choose a smooth function
  $\rho:\R^\ell \to [0,1]$ with $\supp \rho \subset B_1$ and $\int_{\R^\ell}
  \rho =1$. Then set $\rho_\delta(x) = \rho(x/\delta)$ for
  $\delta>0$. Note that we have $\supp \rho_\delta \subset
  B_\delta(0)$ and $\partial^\alpha \rho_\delta =
  \delta^{-k} \partial^\alpha \rho$ with $k=\nm{\alpha}$. Define
  \[\varepsilon_k := (a_k k^k)^{-1}, \qquad \qquad a_k
  :=\Nm{\rho}_{H^{k,p}}\;.\] Now let $\e>0$ and $f \in
  C^\infty_0(\Omega)$ be any given smooth function with compact
  support. Fix $m \in \mathbb{N}$ such that $2^{-m} <
  \e$. Using cut-off functions we find $g \in H^{m,p}_0(\R^\ell)$ such
  that $\Nm{g-f}_{H^{m,p}(\Omega)} \leq \e/4$ (see
  \cite[Exercise B.1.3]{Bibel}). Secondly we find $\delta>0$ such that 
  the smooth and compactly supported function $h = \rho_\delta * g$ satisfies $\Nm{g-h}_{H^{m,p}(\R^\ell)} <
  \e/4$. Then we have 
  \begin{multline*}
  \di{f,h}_{C^\infty(\Omega)} = \sum_{k \geq 0}
  \frac{\Nm{f-h}_{k,p;\Omega}}{1+ \Nm{f-h}_{k,p;\Omega}}2^{-(k+1)}
  \leq \Nm{f-h}_{m,p;\Omega} + 2^{-(m+1)}\\ \leq
  \Nm{f-g}_{m,p;\Omega} + \Nm{g-h}_{m,p;\Omega} + 2^{-(m+1)} \leq \e\;.
  \end{multline*}
  This shows that $h$ lies in the $\e$-ball about the function $f$ in
  the $C^\infty$-topology. It remains to show that $h \in
  C^\e(\Omega)$. Indeed, by Young's inequality we have 
  \[\e_k\Nm{h}_{H^{k,p}} = \e_k\Nm{g * \rho_\delta}_{H^{k,p}} \leq
  \e_k\Nm{g}_{L^1} \Nm{\rho_\delta}_{H^{k,p}} \leq \e_k a_k
  \delta^{-k} \Nm{g}_{L^1} \leq 2^{-k}\Nm{g}_{L^1}\;, \] for every $k
  > 2\delta^{-1}$. This shows that the $\e$-norm of $h$ is finite or
  equivalently that $h \in C^\e(\Omega)$.
\end{proof}
Let $E \to M$ be any vector bundle over a compact Riemannian manifold
with boundary and corners. Choose a connection $\na$ and a Riemannian
metric on $E$. This induces connections and a metric on $E \otimes F$,
where $F$ is any tensor bundle over $M$. Let $\mathrm{vol}_M$ be a
volume form. Define the norm
\[
\Nm{\xi}_p :=\left(\int_M \nm{\xi}^p \mathrm{vol}_M\right)^{1/p}\;,
\]
and recursively for all $k \in \N$ define the norms
\[
\Nm{\xi}_{0,p}:= \Nm{\xi}_p,\qquad \Nm{\xi}_{k,p}:= \Nm{\na
  \xi}_{k-1,p}\;.
\]
\begin{dfn}\label{dfn:enorm}
  Let $p>1$ and $(\e_k)_{k \in \N}$ be a sequence. We define the space
  $\Gamma^\e(E)\subset \Gamma^\infty(E)$ to be the space of all smooth
  sections $\xi$ which are bounded in the norm
  \[
  \Nm{\xi}_\e = \sum_{k \geq 0} \e_k \Nm{\xi}_{k,p} \;.\]
\end{dfn}
\begin{prp}\label{prp:eBanach}
  Suppose that $M$ is compact with boundary and corners. The space
  $\Gamma^\e(E)$ with norm $\Nm{\cdot}_\e$ is a separable Banach space
  and there exists a sequence $(\e_k)_{k\in \N}$ such that the
  inclusion $\Gamma^\e(E) \subset \Gamma^\infty(E)$ is dense for the
  $C^\infty$-topology.
\end{prp}
\begin{proof}
  Choose a local trivialization of $E$ over charts of $M$ which are
  adapted to the boundary $\partial M$ and an associated partition of
  unity. Then the norm $\Nm{\xi}_\e$ of any section $\xi \in
  \Gamma(E)$ is equivalent to the finite sum of the $\e$-norm of its
  local representatives. Then the claim follows from
  Lemma~\ref{lmm:eBanach} and~\ref{lmm:edense}.
\end{proof} 
\begin{lmm}\label{lmm:perturbations}
  Fix paths $J_-$, $J_+:[0,1]\to \End(TM,\omega)$.  For any $s_1>0$
  consider the space $\J := \{J \in C^\infty(\R\times
  [0,1],\End(TM,\omega) \mid J(\pm s,\cdot) = J_\pm \forall\ s \geq
  s_1\}$. There exists a dense subspace $\J' \subset \J$ which is a
  separable Banach manifold. The same holds for $\J :=
  C^\infty([0,1],\End(TM,\omega)$.
\end{lmm}
\begin{proof} 
  For any $J \in \J$ we define the linear bundle $S_J
  \to \Sigma \times M$ where the fibre of $S_J$ over a point $(s,t,p)
  \in \Sigma \times M$ is given by linear maps $Y \in \End(T_pM)$ such
  that
  \begin{equation*}
    YJ(s,t,p)+J(s,t,p)Y = 0,\qquad \omega_p(Y \xi,\xi')+ \omega_p(\xi,Y\xi')=0\;,
  \end{equation*}
  for all vectors $\xi,\xi' \in T_pM$. The tangent space $T_J\J$ is
  given by smooth sections $Y \in \Gamma(S_J)$ with support contained
  in $[-s_1,s_1]\times [0,1]\times M$.  Fix any $J_0 \in \J$, we
  identify the space $\J$ with the $C^0$-unit ball in the space smooth
  sections $Y \in \Gamma(S_{J_0})$ with support in $[-s_1,s_1]\times
  [0,1]\times M$, by $J\mapsto Y_J:=(J+J_0)^{-1}(J-J_0)$. The inverse
  is $Y \mapsto J_0(1-Y)^{-1}(1+Y)$. For further details
  see~\cite[Section 4.2]{Schwarz:PhD}. By
  Proposition~\ref{prp:eBanach} there exists a sequence
  $\e:=(\e_\ell)_{\ell \in \R}$ such that the subspace is dense
  \begin{equation*}
    \J' = \left\{ J \in \J \ | \ \Nm{Y_J}_\e < \infty \right\}
    \subset \J\,.
  \end{equation*}
  By the same lemma we see that the space $\J'$ is identified with
  an open subset in a separable Banach space. 
\end{proof}

\section{Gluing}

All algebraic statements for Floer homology in this work are based on
a gluing result of holomorphic strips, which is in a sense the
converse for the Floer-Gromov breaking phenomenon.  Originally the
problem has been addressed and solved by Floer in the series of papers
\cite{Floer:Intersection}, \cite{Floer:Action} and \cite{Floer:Sphere}
under the assumption that the holomorphic strips have boundary in two
transversely intersecting Lagrangians. The generalization for the
degenerate case in which both Lagrangians are equal was worked out by
Fukaya, Oh, Ohta and Ono \cite[Chapter 7]{FO3:II}. In this chapter we
give a further generalization of the gluing theorem for holomorphic
strips with boundary on two cleanly intersecting Lagrangians. Our
approach is not new and was previously sketched out by Frauenfelder
in~\cite[Chapter 4.7]{Frauenfelder:PhD}. Since we need precise
statements for the construction of coherent orientations we give a
complete proof here. We follow closely the lines of~\cite{FO3:II} as
well as the gluing results of \cite{Abouzaid:spheres}
and~\cite{BiranCornea:pearl}.  Another approach by Sim\v{c}evi\'c has
been developed in~\cite{Tatjana} using completely different methods of
interpolation theory. At the end of the chapter we also give a small
generalization of a gluing result in~\cite{AbbMajer:MorseI} which is
for classical Morse theory.

\subsection{Main statement}\label{sec:glueintro}
Let $(M,\omega)$ be a symplectic manifold and $L_0, L_1 \subset M$
closed Lagrangian submanifolds. Fix admissible vector fields $X_0$,
$X_1$ and admissible almost complex structure $J_0$, $J_1$ (\cf
Definition~\ref{dfn:JXadm}) such that $X_0(s,\cdot)=X_1(-s,\cdot)$ and
$J_0(s,\cdot)=J_1(-s,\cdot)$ for all $s$ large enough. Abbreviate by
$\M_k$ the moduli space of all $(J_k,X_k)$-holomorphic strips modulo
reparametrization. We denote by
\begin{equation}\label{eq:M1WW}
  \M^1(W_-,W_+) := \left\{ (u_0,u_1) \in \M_{0} \times \M_{1} \left|
     \ \begin{aligned}  
&u_0(-\infty) \in W_-\\
&u_0(\infty)=u_1(-\infty)\\
&u_1(\infty) \in W_+      
\end{aligned}\right\}\right.\,,
\end{equation}
for some fixed submanifolds $W_-$ and $W_+$ in the space of perturbed
intersection points (\cf Section~\ref{sec:regsetup}). We distinguish
three cases and define $(J,X)$ 
\begin{enumerate}[label=(\Alph*)]
\item\label{nm:both}   both  $(J_0,X_0)$ and $(J_1,X_1)$ are
  $\R$-invariant, $(J,X):=(J_0,X_0)=(J_1,X_1)$
\item\label{nm:one} either $(J,X):=(J_0,X_0)$ or $(J,X):=(J_1,X_1)$ is
  $\R$-dependent,
\item\label{nm:none} both $(J_0,X_0)$ and $(J_1,X_1)$ are
  $\R$-dependent, then $(J,X)=(J_R,X_R)_{R\geq R_0}$ with $J_R = J_0
  \#_R J_1$ and $X_R=X_0\#_R X_1$ (\cf equation~\eqref{eq:glueJ}
  and~\eqref{eq:glueX})
\end{enumerate}
Mostly the arguments are the same for these three cases and we only
distinguish them at parts where it is necessary.  We glue a pair
$(u_0,u_1)\in \M^1(W_-,W_+)$ at the point
$u_0(\infty)=u_1(-\infty)$ to obtain a family of strips in the space
(\cf equation~\eqref{eq:MWWJX} or~\eqref{eq:MWWJRXR})
\begin{equation}\label{eq:MWW}
  \M(W_-,W_+) :=  \wt \M(W_-,W_+;J,X)/\sim\,.
\end{equation}
We say that $J_0$ and $J_1$ are \emph{regular} if (\cf
Definitions~\ref{dfn:reg} and~\ref{dfn:regR})
\begin{itemize}
\item $J_k$ is regular for $X_k$  for $k=0,1$,
\item in case~\ref{nm:none}, the glued structure $(J_R)$ is regular
  for $(X_R)$ and
\item the spaces $\M^1(W_-,W_+)$ and $\M(W_-,W_+)$ are cut-out
  transversely.
\end{itemize}
Consequently each connected component of the above spaces is a
manifold and as usual we denote with the subscrip $[d]$ the union of
all $d$-dimensional components.
\begin{thm}\label{thm:glue}
  Assume that the almost complex structures $J_0$ and $J_1$ are
  regular.  Given a pair $u=(u_0,u_1) \in \M^1(W_-,W_+)_{[0]}$. There
  exists $R_0$ and a continuous map
  \begin{equation*}
   \glue_u: [R_0,\infty) \to \M(W_-,W_+)_{[1]}\,,\qquad R \mapsto w_R\,,
  \end{equation*}
  such that 
  \begin{enumerate}[label=(\roman*)]
  \item $(w_R)$ Floer-Gromov converges to $u$ as $R \to
    \infty$,
  \item given a sequence $(w^\nu) \subset \M(W_-,W_+)_{[1]}$
    which Floer-Gromov converges to $u$, then $w^\nu$ lies in
    the image of the map $\glue_u$ for all but finitely many $\nu$.
  \end{enumerate}
  Moreover with orientations given in Lemma~\ref{lmm:oriMWWM1WW}, the
  space
  \[\ol \M(W_-,W_+)_{[1]} := \M(W_-,W_+)_{[1]} \sqcup \M^1(W_-,W_+)_{[0]}\,,\]
  is an oriented manifold with oriented boundary
  $(-1)\cdot\M^1(W_-,W_+)_{[0]}$ if $(X_1,J_1)$ is $\R$-invariant and
  $\M^1(W_-,W_+)_{[0]}$ otherwise.
\end{thm}
\begin{proof}
  The proof covers the rest of the chapter. Here we give an overview
  of the principal arguments.  Basically we follow the standard gluing
  procedure, which we quickly recall now. Fix a rigid pair $(u_0,u_1)
  \in \M^1(W_-,W_+)_{[0]}$ and a large enough gluing parameter $R \geq
  R_0$. We denote the glued structures $J_R := J_0 \#_R J_1$ and
  $X_R:=X_0\#_R X_1$ (\cf equations~\eqref{eq:glueJ}
  and~\eqref{eq:glueX}). We define the \emph{preglued map} $u_R$ using
  cut-off functions and then roughly speaking solve the equation $\ps
  w +J_R(w) (\pt w-X_R(w))=0$ for $w$ in a neighborhood of $u_R$ using
  the Newton-Picard theorem. More precisely given a small vector field
  $\xi$ along $u_R$ and write the map $w$ as $w(s,t)=\exp_{u_R(s,t)}
  \xi(s,t)$ with respect to some exponential function associated to an
  axillary Levi-Civita connection. Then $w$ is
  $(J_R,X_R)$-holomorphic if and only if $\xi$ is a zero of a
  non-linear map $\F_R$ defined on an open ball in a Banach space of
  sections of $u_R^*TM$ (\cf equation~\eqref{eq:nonlinearF}). Since we
  work with degenerated asymptotics which require exponential
  weights, the Sobolev norms which we work with have adapted weights
  that depend on the gluing parameter (\cf
  Section~\ref{sec:preglue}). In the assumptions of the Newton-Picard
  theorem we need a bound on the right-inverse of the differential of
  $\F_R$ at zero, denoted $D_R$, which does not depend on $R$. The
  right inverse $Q_R$ is constructed in~\eqref{eq:QR} and the uniform
  bound is established in Corollary~\ref{cor:QR}. Moreover we need a
  quadratic estimate (\cf Lemma~\ref{lmm:quadratic}).  Then all
  $(J_R,X_R)$-holomorphic strips in a neighborhood of $u_R$ are
  modeled on the kernel of $D_R$, \ie for each element $\xi' \in \ker
  D_R$ small enough there exists an unique element
  $\xi'':=\sigma_R(\xi) \in \im Q_R$ such that $(s,t) \mapsto
  \exp_{u_R(s,t)}(\xi'(s,t)+\xi''(s,t))$ is a $(J_R,X_R)$-holomorphic
  strip and any $w$ close enough to $u_R$ is of that form (\cf
  Lemma~\ref{lmm:sol}). In particular the map
  $v_R:=\exp_{u_R}\sol_R(0)$ is $(J_R,X_R)$-holomorphic. We define the
  gluing map $\glue_u(R) =w_R$ where
  \begin{itemize}
  \item in case~\ref{nm:both} $w_R=[v_R]$ the equivalence
    class modulo reparametrizations,
  \item in case~\ref{nm:one} for all $(s,t) \in \R \times [0,1]$ we
    define
    \[w_R(s,t) =
    \begin{cases}
      v_R(s-2R,t)&\text{if } (J_1,X_1) \text{ is $\R$-invariant}\\
      v_R(s+2R,t)&\text{if } (J_0,X_0) \text{ is $\R$-invariant}\,.
    \end{cases}\]
  \item in case~\ref{nm:none} $w_R=v_R$.
  \end{itemize}
  That the gluing map is continuous is proven in
  Lemma~\ref{lmm:solcont}, that it is asymptotically surjective is
  proven in Lemma~\ref{lmm:Gsurj} and the statement about the
  orientations is proven in Proposition~\ref{prp:degglue}.
\end{proof}

\subsection{Pregluing}\label{sec:preglue}
In this section we introduce the Sobolev framework. The main ideas in
this chapter are straight-forward generalizations of the methods
of~\cite[Chapter 7.1]{FO3:II}. We assume for simplicity that
$X_0\equiv 0$, $X_1\equiv 0$ and $W_-,W_+$ lie on different connected
components. Choose an auxiliary metric on $M$ such that $W_-$, $W_+$,
$L_0$ and $L_1$ are totally geodesic (\cf
Lemma~\ref{lmm:totgeodesic}). All norms, parallel transport and
exponential maps in the following sections are induced by this
metric. For the general case where $X_0$, $X_1\not\equiv 0$ or $W_-$
and $W_+$ lie on the same connected component, we need to work with
metrics that depend on the domain as explained in the proof of
Lemma~\ref{lmm:Banmfd}.

\subsubsection{Preglued strip} From now that the pair $u=(u_0,u_1) \in
\M^1(W_-,W_+)_{[0]}$ is fixed. In case~\ref{nm:both} or~\ref{nm:one},
the maps $u_0$ and $u_1$ are unparametrized. We choose
parametrizations and still denote the maps with the same symbol by
abuse of notation. Due to exponential decay (see
Theorem~\ref{thm:remove}) there exists an intersection point $p =
u_0(\infty)=u_1(-\infty) \in C$, a constant $s_0\geq 0$ and two maps
$\zeta_0:[s_0,\infty) \times [0,1] \to T_pM$ , $\zeta_1:(-\infty,-s_0]
\times [0,1] \to T_pM$ such that for all $s\geq s_0$ and $t \in [0,1]$
we have
\begin{equation*}
  u_0(s,t) = \exp_{p} \zeta_0(s,t),\qquad u_1(-s,t) = \exp_{p}
  \zeta_1(-s,t)\;.
\end{equation*}
We fix once and for all smooth cut-off functions
\begin{equation}
  \label{eq:cutoff}
  \beta^-,\ \beta^+: \R \times [0,1] \to [0,1],\qquad
  \beta^-(-s,t)=\beta^+(s,t) =
  \begin{cases}
    1&\text{if } s \geq 1\\
    0&\text{if } s \leq 0\,.
  \end{cases}  
\end{equation}
For any $R \geq s_0$ large enough we define the \emph{preglued strip}
$u_R:\R \times [0,1] \to M$ via
\begin{equation}
  \label{eq:uRdef}
  \begin{aligned}
    u_R(s,t) =
    \begin{cases}
      u_1(s-2R,t) &\text{if } s \geq 1\\
\text{see below} &\text{if } -1 \leq s \leq 1\\
      u_0(s+2R,t) &\text{if } s\leq -1\,,
    \end{cases}
  \end{aligned}
\end{equation}
and if $-1 \leq s \leq 1$ we use the interpolation
\begin{equation*}
u_R(s,t):= \exp_{p}(\beta^-(s,t)\zeta_0(s+2R,t)+\beta^+(s,t)
\zeta_1(s-2R,t))\,.  
\end{equation*}
We frequently use the following decay property of the preglued map
$u_R$ in the neck region.
\begin{lmm}\label{lmm:uRdecay}
  There exists constants $c$, $R_0$ and $\iota$ such that for all
  $R\geq R_0$ and $\mu <\iota$ we have
  \begin{equation*}
    \nm{\d u_R(s,t)}  + \di{u_R(s,t),u_R(0,0)} \leq c e^{-\mu(2R-\nm{s})}\,.
  \end{equation*}
  for all $(s,t) \in [-2R,2R]\times [0,1]$.
\end{lmm}
\begin{proof}
  Set $p:=u_R(0,0)$.  By Proposition~\ref{prp:modinB} the maps $u_0$
  and $u_1$ have $\mu$-decay.  If $s \leq -1$ we have by definition
  $u_R = u_0 \circ \tau_{-2R}$ and the claim follows since $u_0$ has
  $\mu$-decay. Similar for $s \geq 1$. If $\nm{s} \leq 1$ and $R$ is
  large enough $u_R(s,t)$ is close to $p$ for all $t \in [0,1]$. By
  Corollary~\ref{cor:dwxi} we have as $R \to \infty$
   \begin{align*}
     &\nm{\d u_R} + \di{u_R(s,t),p}\\
     &\leq O(1) \left(\nm{\na \big(\beta^- \zeta_0 \circ
         \tau_{-2R} + \beta^+ \zeta_1 \circ \tau_{2R}\big)} + \nm{\zeta_0 \circ \tau_{-2R}} + \nm{\zeta_1 \circ \tau_{2R}}\right) \\
     &\leq O(1) \left(\nm{\na \zeta_0 \circ \tau_{-2R}} + \nm{\na
         \zeta_1 \circ \tau_{2R}} +
       \nm{\zeta_0 \circ \tau_{-2R}} + \nm{\zeta_1 \circ \tau_{2R}}\right)\\
     &\leq O(1) \left(\nm{\d u_0 \circ \tau_{-2R}} + \nm{\d u_1 \circ
         \tau_{2R}} + \di{u_0\circ \tau_{-2R},p} + \di{u_1\circ
         \tau_{2R},p}\right) \\ &\leq O(1)e^{-\mu(2R-\nm{s})}\,.
  \end{align*} 
  This proves the lemma.
\end{proof}
\subsubsection{Linear pregluing and breaking}
Choose $p>2$, $\delta>0$ and abbreviate 
\begin{itemize}
\item $H_0:=T_{u_0}\B^{1,p;\delta}$ and $H_1:=T_{u_1}\B^{1,p;\delta}$,
\item $L_0:=\E_{u_0}^{p;\delta}$ and
  $L_1:=\E_{u_1}^{p;\delta}$,
\item $H_{01} \subset H_0 \oplus H_1$ consisting of pairs $(\xi_0,\xi_1)$ such that $\xi_0(\infty) = \xi_1(-\infty)$,
\item $H_R:=T_{u_R}\B^{1,p;\delta}$ and $L_R:=\E_{u_R}^{p;\delta}$ for
  any $R \geq s_0$.
\end{itemize}
Define the \emph{linear pregluing operator} $\Theta_R:H_{01} \to H_R$,
$(\xi_0,\xi_1) \mapsto \xi_R$ with
\begin{equation}
  \label{eq:xiRdef}
\begin{gathered}
  \xi_R(s,t) = \begin{cases}
    \xi_1(s-2R,t)&\text{if } s \geq R\,,\\
    \text{see below}&\text{if } s \in [-R,R],\\
    \xi_0(s+2R,t)&\text{if }s  \leq -R\,.
\end{cases}  
\end{gathered}
\end{equation}
and if $-R \leq s \leq R$ we use define (omitting the arguments for
convenience)
\begin{equation*}
  \xi_R=\widehat \Pi_p^{u_R} \bar \xi + \beta^+_{-R} \left(\Pi_{u_1 \circ
      \tau_{2R}}^{u_R} \xi_1 \circ \tau_{2R} - \widehat\Pi_p^{u_R} \bar
    \xi\right) + \beta^-_R \left(\Pi_{u_0\circ \tau_{-2R}}^{u_R} \xi_0
    \circ \tau_{-2R} - \widehat \Pi_p^{u_R} \bar \xi\right)\,,
\end{equation*}
with notations $\bar \xi := \xi_0(\infty) = \xi_1(-\infty)$,
$\tau_{R}:\Sigma \to \Sigma$, $(s,t) \mapsto (s-R,t)$, $\beta^+_{-R} =
\beta^+ \circ \tau_{-R}$, $\beta^-_R= \beta^-\circ \tau_R$ and the
parallel transport maps $\Pi$, $\wh \Pi$ as given
in~\eqref{eq:Pihat}. Finally define the \emph{breaking operator}
$\Xi_R:L_R \to L_0 \oplus L_1$, $\eta \mapsto (\eta_{0,R},\eta_{1,R})$
via
\begin{equation}
  \label{eq:JRdef}
\begin{aligned}
  \eta_{1,R}(s,t) &=
  \begin{cases}
    \eta(s+2R,t)&\text{if } s \geq -2R\,,\\
    \text{see below}&\text{if } -2R-1 \leq s \leq -2R\\
    0 &\text{if }  s \leq -2R-1\,,\\
  \end{cases}\\ \\
  \eta_{0,R}(s,t) &=
  \begin{cases}
    0&\text{if } s \geq 2R\,,\\
    \text{see below}& \text{if } 2R-1\leq s \leq 2R\,,\\
    \eta(s-2R,t)&\text{if } s \leq 2R-1\,.
  \end{cases}
\end{aligned}  
\end{equation}
For the interpolation we just use parallel transport. We do not need
to use cut-off functions because the maps are only supposed to be of
regularity $L^p_\loc$. More precisely for $2R-1\leq s\leq 2R$ and
$t\in [0,1]$ we define
\[
\eta_{0,R}(s,t) = \Pi_{u_R(s-2R,t)}^{u_0(s,t)} \eta(s-2R,t)\,,\]
and for $-2R-1 \leq s \leq -2R$ and $t \in [0,1]$ we define
\[
\eta_{1,R}(s,t) = \Pi_{u_R(s+2R,t)}^{u_1(s,t)} \eta(s+2R,t)\,.
\]
We now show that these constructions are uniformly continuous with
respect to an adapted norm. 
\subsubsection{Adapted norms} For $R >0$ we define a weight function
$\gamma_{\delta,R}:\R \to \R$
\[ \gamma_{\delta,R}(s) =
  \begin{cases}
    e^{-\delta(2R+s)} &\text{if } s<-2R\\
    e^{\delta(2R-\nm{s})} &\text{if } \nm{s} <2R\\
    e^{\delta(s-2R)} &\text{if } s>2R\,.
  \end{cases}
  \]
  Given a curve $u\in \B^{1,p;\delta}(C_-,C_+)$, we define weighted
  norms for all vector fields $\eta \in \E^{p;\delta}_u$
  \begin{equation*}
    \Nm{\eta}_{p;\delta,R} :=\Big(\int_{\Sigma}\nm{\eta}^p \gamma_{\delta,R}^p\d s \d t\Big)^{1/p}\;,
  \end{equation*}
  and for all vector fields $\xi \in T_u\B^{1,p;\delta}(C_-,C_+)$ we define the norm $\Nm{\xi}_{1,p;\delta,R}$ via
  \begin{equation}\label{eq:nrmxideltaR}
    \begin{aligned}
      & \Big(|\xi(0,0)|^p
      +\Nm{\xi(-\infty)}^p + \Nm{\xi(-\infty)}^p+\\
      &+\int_{\Sigma_{-\infty}^{-2R}} \left(\nmm{\xi-\widehat
          \Pi_{u(-\infty)}^u \xi(-\infty)}^p +
        \nmm{\na\big(\xi-\widehat \Pi_{u(-\infty)}^u
          \xi(-\infty)\big)}^p\right)\gamma_{\delta,R}^p \d s \d t
      \\
      &+\int_{\Sigma_{-2R}^{2R}} \Big(\nmm{\xi - \widehat
    \Pi_{u(0,0)}^u \xi(0,0)}^p + \nmm{\nabla \big(\xi- \widehat
    \Pi_{u(0,0)}^u \xi(0,0)\big)}^p\Big) \gamma_{\delta,R}^p\d s
  \d t\\
  &+\int_{\Sigma_{2R}^{\infty}} \left(\nmm{\xi-\widehat
      \Pi_{u(\infty)}^u \xi(\infty)}^p + \nmm{\na\big(\xi-\widehat
      \Pi_{u(\infty)}^u \xi(\infty)\big)}^p\right)\gamma_{\delta,R}^p
  \d s \d t\Big)^{1/p}\,.
    \end{aligned}
  \end{equation}
It is straight-forward to check that that for a fixed $R$ these define
equivalent norms (see \cite[Lemma 5.8]{Abouzaid:spheres})
\subsection{A uniform bounded right inverse}\label{sec:uniform}
For the remaining statements to hold true the decay parameter
$\delta>0$ must be sufficiently small. The bound on $\delta$ depends
on the spectral gap of the asymptotic operators given
in~\eqref{eq:iotaC}. More precisely, we assume for the rest of the section:
\begin{equation}
  \label{eq:fixdelta}
  2 \delta < \iota ,\qquad \iota:=\min\{ \iota(J_\infty^-),\iota(J_\infty),\iota(J_\infty^+)\}\,,
\end{equation}
in which $J_\infty^-$, $J_\infty$ and $J_\infty^+$ are paths of almost
complex structures given by  $J_0(-s,\cdot)=J_\infty^-$,
$J_0(s,\cdot)=J_1(-s,\cdot)=J_\infty$ and $J_1(s,\cdot)=J_\infty^+$
for $s$ large enough.
\begin{lmm}\label{lmm:estIR}
  There exists constants $c$ and $R_0$ such that for all
  $(\xi_0,\xi_1) \in H_{01}$ and $R \geq R_0$
  \[\Nm{\Theta_R(\xi_0,\xi_1)}_{1,p;\delta,R} \leq c \left(\Nm{\xi_0}_{1,p;\delta} +
    \Nm{\xi_1}_{1,p;\delta}\right)\,.\]
\end{lmm}
\begin{proof} 
  We follow the proof of~\cite[Prp. 4.7.5]{BiranCornea:pearl}. Fix
  $(\xi_0,\xi_1) \in H_{01}$ and denote $\xi_R :=
  \Theta_R(\xi_0,\xi_1)$ and $p := u_R(0,0)$.  By definition we have
  \begin{equation}\label{eq:startingpoint}
  \begin{aligned}
    \Nm{\xi_R}_{1,p;\delta,R}^p &=\Nmm{\Res{\xi_0}{\Sigma_{-\infty}^0}}_{1,p;\delta}^p + \Nmm{\Res{\xi_1}{\Sigma_0^{\infty}}}_{1,p;\delta}^p + \nm{\xi_R(0,0)}^p\\
    & +\int_{\Sigma_{-2R}^{2R}} \left(\nm{\xi_R - \widehat \Pi_p^{u_R}
        \xi_R(0,0)}^p + \nm{\na \left(\xi_R -\widehat \Pi_p^{u_R}
          \xi_R(0,0)\right)}^p\right)\gamma^p_{\delta,R}\d s \d t\;.
  \end{aligned}
  \end{equation}
  Lets concentrate on the last summand. We deduce a pointwise estimate of
  \begin{equation}
    \label{eq:splitone}
    \xi_R - \widehat \Pi_p^{u_R} \xi_R(0,0) = \left(\xi_R - \widehat
    \Pi_p^{u_R} \bar \xi\right)+\widehat\Pi_p^{u_R} \left(\bar \xi -
    \xi_R(0,0)\right)\;,
  \end{equation}
  and its covariant derivative. Abbreviate $u_{0,R}:=u_0 \circ
  \tau_{-2R}$, $u_{1,R} = u_1\circ \tau_{2R}$, $\xi_{0,R}:=\xi_0 \circ
  \tau_{-2R}$ and $\xi_{1,R}:=\xi_1 \circ \tau_{2R}$. By definition of
  $\xi_R$, the first summand of the right hand side
  of~\eqref{eq:splitone} equals
  \begin{equation}
    \label{eq:firstsummand}
  \beta_R^- \left(\Pi_{u_{0,R}}^{u_R} \xi_{0,R} - \widehat \Pi_p^{u_R} \bar
    \xi\right)+\beta^+_{-R}\left(\Pi_{u_{1,R}}^{u_R} \xi_{1,R}- \widehat \Pi_p^{u_R} \bar \xi \right)\;.  
  \end{equation}
  Focusing now on the first summand of~\eqref{eq:firstsummand} and
  taking into account the support of the cut-off function we have to
  estimate the integral over the smaller strip $[-2R,R]\times [0,1]$
  of the norm of
  \begin{equation}
    \label{eq:u0xi0}
    \Pi_{u_{0,R}}^{u_R} \xi_{0,R} - \widehat
    \Pi_p^{u_R} \bar \xi\;,
  \end{equation}
  and its covariant derivative, since $\beta_R^-$ vanishes on $[R,2R]
  \times [0,1]$. Now on $[-2R,-1] \times [0,1]$ we have that $u_R =
  u_0 \circ \tau_{-2R}$ by Definition~\eqref{eq:uRdef}. Hence by
  substitution $s \mapsto s +2R$ we have
  \begin{multline}\label{eq:-2R-1}
    \int_{\Sigma_{-2R}^{-1}} \left(\nmm{\xi_{0,R} - \widehat \Pi_p^{u_{0,R}} \bar \xi}^p + \nmm{\na\big( \xi_{0,R} -\wh \Pi_p^{u_{0,R}} \bar \xi\big)}^p\right)e^{p\delta(2R-\nm{s})} \d s \d t\\
    = \int_{\Sigma_0^{2R-1}} \left(\nmm{\xi_0 - \widehat \Pi_p^{u_0}
        \bar \xi}^p + \nmm{\na\big(\xi_0 - \widehat \Pi_p^{u_0}
        \bar \xi\big)}^p\right)e^{p\delta s} \d s \d t \leq \Nm{\xi_0}_{1,p;\delta}\,.
  \end{multline} 
  We estimate the same term on $[-1,R]\times [0,1]$. If $R$ is large
  enough the distance of $u_R(s,t)$ to $p$ is less then one third the
  injectivity radius for every $(s,t) \in [-1,R] \times [0,1]$. Hence
  without loss of generality we replace $\widehat \Pi$ by $\Pi$ in the
  formula~\eqref{eq:u0xi0} and continue
  \begin{equation}\label{eq:Pixi0}
    \Pi_{u_{0,R}}^{u_R} \xi_{0,R} - \Pi_p^{u_R}\bar \xi =
    \Pi_{u_{0,R}}^{u_R} \left( \xi_{0,R} - \Pi_p^{u_{0,R}} \bar
      \xi\right) + \left(\Pi_{u_{0,R}}^{u_R} \Pi_p^{u_{0,R}} \bar \xi - \Pi_p^{u_R} \bar \xi\right)\;.
  \end{equation}
  For the first summand on the right-hand side we estimate using
  Corollary~\ref{cor:commutenaPi} and Lemma~\ref{lmm:uRdecay}
  \begin{align*}
    \nmm{\Pi_{u_{0,R}}^{u_R}\left( \xi_{0,R} - \Pi_p^{u_{0,R}} \bar
        \xi\right)} &= \nmm{\xi_{0,R} - \Pi_p^{u_{0,R}} \bar \xi}\\
    \nmm{\nabla \big(\Pi_{u_{0,R}}^{u_R}\left( \xi_{0,R} -
        \Pi_p^{u_{0,R}} \bar \xi\right)\big)} &  \leq \nm{\nabla\left(\xi_{0,R} - \Pi_p^{u_{0,R}} \bar
        \xi\right)}+ O(1) \nm{\xi_{0,R} - \Pi_p^{u_{0,R}} \bar \xi}\,.
  \end{align*}
  For the second summand on the right-hand side of~\eqref{eq:Pixi0} we
  estimate using Corollary~\ref{cor:commutePiPi} and Corollary~\ref{cor:commutenaPi}
  \begin{align*}
    \nm{\Pi_{u_{0,R}}^{u_R} \Pi_p^{u_{0,R}} \bar \xi - \Pi_p^{u_R}
      \bar \xi} &\leq O\big(\di{u_R,p} +\di{u_{0,R},p}\big) \nm{\bar \xi}\\
    \nm{\nabla\big(\Pi_{u_{0,R}}^{u_R} \Pi_p^{u_{0,R}} \bar \xi
      -\Pi_p^{u_R}\bar \xi\big)} &\leq O\big(\nm{\d u_R} + \nm{\d
      u_{0,R}}\big)\nm{\bar \xi}\,.
  \end{align*}
  In particular we see that both quantities are bounded by
  $O(\omega)\nm{\bar \xi}$ with $\omega(s)=e^{-\mu(2R-\nm{s})}$. Use the last two
  estimates and the identity~\eqref{eq:Pixi0} to show
  \begin{align*}
    \nm{\Pi_{u_{0,R}}^{u_R} \xi_{0,R} - \Pi_p^{u_R}\bar \xi} &\leq
    \nm{\xi_{0,R} - \Pi_p^{u_{0,R}} \bar \xi} +
    O(\omega)\nm{\bar
      \xi}\,,\\
    \nm{\nabla (\Pi_{u_{0,R}}^{u_R} \xi_{0,R} - \Pi_p^{u_R} \bar
        \xi)} &\leq O(1)\nm{\xi_{0,R} - \Pi_p^{u_{0,R}} \bar \xi}
    + \nm{\nabla \left(\xi_{0,R} - \Pi_p^{u_{0,R}}\bar \xi\right)} +
    O(\omega) \nm{\bar \xi}\,.
  \end{align*}
  Integrating these pointwise estimates gives
  \begin{multline*}
    \int_{\Sigma^R_{-1}} \left(\nmm{\Pi_{u_{0,R}}^{u_R} \xi_{0,R} -
        \Pi_p^{u_R}\bar \xi}^p + \nmm{\na\big(\Pi_{u_{0,R}}^{u_R}
        \xi_{0,R} - \Pi_p^{u_R}\bar
        \xi\big)}^p\right)e^{p\delta(2R-\nm{s})}\d s d\t \\
    \leq
    O(1) \int_{\Sigma_{2R-1}^{4R}}\left(\nmm{\xi_0 - \Pi_p^{u_0}\bar \xi}^p + \nmm{\na \big(\xi_0 - \Pi_p^{u_0}\bar \xi\big)}^p\right)e^{p\delta \nm{s}} \d s \d t+ \\
    + O\big(e^{-2pR(\mu-\delta)}\big)\nm{\bar \xi}^p \int_{-1}^{R}
    e^{p(\delta - \mu)\nm{s}} \d s\;.
  \end{multline*}
  To show that the factor with $\nm{\bar \xi}^p$ in the last summand
  is uniformly bounded we compute directly assuming without loss of
  generality that $R \geq 1$ 
  \[\int_{-1}^R e^{p (\mu-\delta)\nm{s}}\d s \leq 2 \int_0^R e^{p
    (\mu-\delta)s}\d s = \frac{2}{p (\mu- \delta)} \left(e^{p
      R (\mu- \delta)}- 1\right)\leq O(e^{pR(\mu-\delta)})\;.\] 
  The last estimate and estimate~\eqref{eq:-2R-1} shows that the integral
  \[\int_{\Sigma_{-2R}^{2R}} \left(\nmm{\Pi_{u_{0,R}}^{u_R} \xi_{0,R} - \widehat \Pi_p^{u_R} \bar
      \xi}^p + \nmm{\na\big(\Pi_{u_{0,R}}^{u_R} \xi_{0,R} - \widehat
      \Pi_p^{u_R} \bar \xi\big)}^p\right)\gamma_{\delta,R}^p \d s \d
  t\;.\] is bounded by $O(1)\Nm{\xi_0}^p_{1,p;\delta}$. Similarly we
  proceed with the second term of the summand~\eqref{eq:firstsummand}
  and find that the integral
  \begin{equation*}
    \int_{\Sigma_{-2R}^{2R}} \left(\nmm{\xi_R - \wh \Pi_p^{u_R}\bar
        \xi}^p + \nmm{\na\big(\xi_R - \wh \Pi_p^{u_R} \bar
        \xi\big)}^p\right)\gamma_{\delta,R}^p \d s \d t\,, 
  \end{equation*}
  is bounded by $O(1)(\Nm{\xi_0}_{1,p;\delta} +\Nm{\xi_1}_{1,p;\delta})^p$.
  For the last term of~\eqref{eq:splitone} we use the fact that
  $u_R(0,t) = p$, $\bar \xi= \xi_0(\infty)=\xi_1(-\infty)$,
  Lemma~\ref{lmm:deltadecay}  to show
  \begin{equation*}
    \nm{\xi_R(0,0)- \bar \xi} \leq \nm{\xi_1(-2R,0) -
      \Pi_p^{u_1(-2R,0)} \xi_1(\infty)} + \nm{\xi_0(2R,0) -
      \Pi_p^{u_0(2R,0)}\xi_0(\infty)}
  \end{equation*}
  and with Corollary~\ref{cor:commutenaPi}
  \begin{equation*} 
    \nmm{\nabla \widehat \Pi_p^{u_R}
      \left(\bar \xi -\xi_R(0,0) \right)} \leq O\big(\nm{\d u_R}\big)
    \nm{\bar \xi - \xi_R(0,0)}\,.
  \end{equation*}
  We conclude that both quantities are bounded by 
  $O\big(e^{-2\delta R}\big)(\Nm{\xi_0}_{1,p;\delta} +
    \Nm{\xi_1}_{1,p;\delta})$ and after intergation we have
  \begin{multline*}
    \int_{\Sigma_{-2R}^{2R}} \left(\nmm{\widehat\Pi_p^{u_R} \bar \xi -
        \xi_R(0,0)}^p +\nmm{\na\big(\widehat\Pi_p^{u_R} \bar \xi -
        \xi_R(0,0)\big)}^p\right)\gamma_{\delta,R}^p \d s \d t\\
    \leq O(1) \left(\Nm{\xi_0}_{1,p;\delta}^p +
      \Nm{\xi_1}_{1,p;\delta}^p\right) \int_{-2R}^{2R} e^{-p\delta
      \nm{s}}\d s  \leq O(1)\left(\Nm{\xi_0}_{1,p;\delta} +
      \Nm{\xi_1}_{1,p;\delta}\right)^p\;.
  \end{multline*}
  Now the claim follows from the last four estimates plugged
  into~\eqref{eq:startingpoint}.
\end{proof}
\begin{lmm}\label{lmm:estBR}
  For all $R$ and for all $\eta \in L_R$ we have
  \[\Nm{\eta_{0,R}}_{p;\delta}^p + \Nm{\eta_{1,R}}_{p;\delta}^p =
  \Nm{\eta}_{p;\delta,R}^p\;,\] where
  $(\eta_{0,R},\eta_{1,R})=\Xi_R(\eta)$.
\end{lmm}
\begin{proof}
  Given any $\eta \in L_R$, by definition of the norm
  and~\eqref{eq:JRdef} we have
  \[ \Nm{\eta}^p_{p;\delta,R} =
  \Nm{\Res{\eta_0}{\Sigma_{-\infty}^0}}^p_{p;\delta} +
  \Nm{\Res{\eta_1}{\Sigma_0^\infty}}^p_{p;\delta} +
  \int_{\Sigma_{-2R}^{2R}} \nm{\eta}^p\gamma^p_{\delta,R} \d s \d
  t\;.\]  Again using the
  definition of the norm  we compute
  \begin{align*}
    \int_{\Sigma_{-2R}^{2R}} \nm{\eta}^p\gamma^p_{\delta,R} \d s \d
    t&=\int_{-2R}^0\int_0^1 \nm{\eta}^p e^{p\delta\left(2R+s\right)}\d
    t \d s + \int_0^{2R} \int_0^1 \nm{\eta}^p
    e^{p\delta\left(2R-s\right)}\d
    t \d s\\
    &= \int_{0}^{2R}\int_0^1 \nm{\eta_{0,R}}^p e^{p\delta s}\d t \d s
    + \int_{-2R}^0 \int_0^1
    \nm{\eta_{1,R}}^p e^{-p\delta s}\d t \d s\\
    &= \int_0^\infty\int_0^1 \nm{\eta_{0,R}}^p e^{p\delta s} \d t \d s
    + \int_{-\infty}^0 \int_0^1 \nm{\eta_{1,R}} e^{-p\delta s} \d t\d
    s\;.
\end{align*}
Since $\eta_{0,R}$ vanishes for $s \geq 2R$ and $\eta_{1,R}$ for $s
\leq -2R$. Inserting the identity back into the first equation gives
the results.
\end{proof}
Denote the linearized Cauchy-Riemann operators $D_0=D_{u_0}:H_0 \to
L_0$ and $D_1=D_{u_1}:H_1 \to L_1$ (\cf equation~\eqref{eq:Du}).  We
define the restricted operators $D_{01}=D_0\oplus D_1|_{H_{01}}$ and
$D_{01}' = D_0 \oplus D_1|_{H_{01}'}$, in which $H_{01}' \subset
H_{01}$ is the subspace of pairs $(\xi_0,\xi_1)$ such that
$\xi_0(-\infty) \in T_{p_-} W_-$ and $\xi_1(\infty) \in T_{p_+} W_+$,
where $p_- = u_0(-\infty)$ and $p_+= u_1(\infty)$.
\begin{lmm}\label{lmm:D01surjective}
  The operator $D_{01}':H_{01}' \to L_0 \oplus L_1$ is surjective and
  has a bounded linear right inverse.
\end{lmm}
\begin{proof}
  See \cite[7.1.20]{FO3:II}, \cite[corollary 4.14]{Frauenfelder:PhD}
  or \cite[Lemma 4.9]{Abouzaid:spheres}. Define the subspaces
  \[
  H_0':=\{ \xi \in H_0\mid \xi(-\infty) \in T_{p_-} W_-\}, \qquad H_1':=\{\xi \in H_1\mid \xi(\infty) \in T_{p_+} W_+\}\,.
  \]
  Further define the restrictions $D_0':=D_0|_{H_0'}$ and
  $D_1':=D_1|_{H_1'}$. By assumption the almost complex structures
  $J_0$ and $J_1$ are regular, which implies that the operator is
  surjective
  \begin{equation}
    \label{eq:evalsurj}
    \ker D_0'
    \oplus \ker D'_1 \to T_p C,\qquad(\xi_0,\xi_1) \mapsto
    \xi_0(\infty)-\xi_1(-\infty)\,. 
  \end{equation}
  Given $(\eta_0,\eta_1) \in L_0 \oplus L_1$ we choose lifts
  $(\xi_0',\xi_1') \in H_0' \oplus H_1'$ such that $D_0 \xi_0' =
  \eta_0$ and $D_1 \xi_1' = \eta_1$. Since~\eqref{eq:evalsurj} is
  surjective we find $(\xi_0'',\xi_1'') \in \ker D_0' \oplus \ker
  D_1'$ such that $\xi_0''(\infty) - \xi_1''(-\infty)= \xi_0'(\infty)
  - \xi_1'(-\infty)$. Then the pair $(\xi_0,\xi_1) :=
  (\xi_0'-\xi_0'',\xi_1' - \xi_1'')$ lies in $H_{01}'$ and is a
  preimage of $(\eta_0,\eta_1)$ under the map $D_{01}$.
  
  We have the inclusion $\ker D'_{01} \subset \ker D_{0} \oplus \ker
  D_{1}$. Since $D_0$ and $D_1$ are Fredholm $\ker D'_{01}$ is finite
  dimensional and by the Hahn-Banach theorem we find a closed linear
  complement $H^\perp_{01}$ in $H_{01}'$. Restricted to $H^\perp_{01}$
  the operator $D_{01}'$ is invertible and hence there exists a
  bounded inverse $Q_{01}':L_{01} \to H_{01}^\perp \subset H_{01}'$.
\end{proof}
\subsubsection{Approximate right inverse} Let $Q'_{01}:L_0 \oplus L_1 \to
H'_{01}$ be a bounded right inverse of $D'_{01}$ which exists by
Lemma~\ref{lmm:D01surjective}. Let $D_R:H_R \to L_R$ be the
linearized Cauchy-Riemann-Floer operator at $u_R$. Moreover define
the restricted operator $D_R':=D_R|_{H'_R}$ where $H_R'\subset H_R$ is
the space of $\xi \in H_R$ such that $\xi(-\infty) \in T_{p_-}W_-$ and
$\xi(\infty) \in T_{p_+}W_+$. By construction the linear pregluing
operator $\Theta_R$ sends the subspace $H_{01}'$ to $H'_R$. We define
the operator
\[\widetilde Q_R = \Theta_R \circ Q'_{01} \circ \Xi_R:L_R \to
H_R'\;.\] The next lemma shows that $\widetilde Q_R$ is an uniformly
bounded approximate right inverse of $D_R$ for every $R$ sufficiently
large.
\begin{lmm}\label{lmm:approxQ}
  There exist constants $c$ and $R_0$ such that for all $R \geq R_0$
  and $\eta \in L_R$ we have
  \[\Nmm{\widetilde Q_R \eta}_{1,p;\delta,R} \leq c
  \Nm{\eta}_{p;\delta,R},\qquad \Nmm{D_R \widetilde Q_R \eta - \eta}_{p;\delta,R} \leq
  ce^{-\delta R} \Nm{\eta}_{p;\delta,R}\,.\]
\end{lmm}
\begin{proof}
  The first estimate follows directly by Lemma~\ref{lmm:estIR}
  and~\ref{lmm:estBR}. We show the second estimate we
  follow~\cite[Lemma 7.1.32]{FO3:II}. Fix any $\eta \in L_R$ and
  abbreviate
  \[\xi_R = \widetilde Q_R \eta,\qquad \left(\xi_0,\xi_1\right)=
  \left(Q_{01}' \circ \Xi_R \right) \eta\;,\]
  and moreover
  \begin{align*}
    u_{0,R} &= u_0 \circ \tau_{-2R}, &u_{1,R} &= u_1 \circ \tau_{2R}\\
    \xi_{0,R} &=\xi_0 \circ \tau_{-2R},&\xi_{1,R} &=\xi_1 \circ \tau_{2R}\;.
  \end{align*}
  By construction we have (recall that $\Sigma_{-\infty}^a =
  (-\infty,a]\times [0,1]$ and $\Sigma_a^\infty = [a,\infty)\times
  [0,1]$ for any $a \in \R$)
  \begin{equation}
    \label{eq:D0D1xieta}
    \Pi_{u_{0,R}}^{u_R}D_{u_{0,R}} \xi_{0,R} = \begin{cases} 0 &\text{on } \Sigma_0^{\infty}\\ \eta &\text{on } \Sigma_{-\infty}^0\end{cases},\qquad \Pi_{u_{1,R}}^{u_R}D_{u_{1,R}}\xi_{1,R}
    = \begin{cases} \eta &\text{on } \Sigma_0^\infty\\ 0 &\text{on } \Sigma_{-\infty}^0\end{cases}\;,
  \end{equation}
  and 
  \begin{align*}
    u_R\big|_{\Sigma_{-\infty}^{-R}} &= u_{0,R}\big|_{\Sigma_{-\infty}^{-R}}, &u_R\big|_{\Sigma_R^\infty} &= u_{1,R}\big|_{\Sigma_R^\infty}\\
  \xi_R\big|_{\Sigma_{-\infty}^{-R}} &=
  \xi_{0,R}\big|_{\Sigma_{-\infty}^{-R}},
  &\xi_R\big|_{\Sigma_{R}^\infty} &=
  \xi_{1,R}\big|_{\Sigma_R^\infty}  \;.
  \end{align*}
  Since the operators are local we have
  \[D_R \xi_R\big|_{\Sigma_{-\infty}^{-R}} =
  D_{u_{0,R}}\xi_{0,R}\big|_{\Sigma_{-\infty}^{-R}} =
  \eta\big|_{\Sigma_{-\infty}^{-R}},\qquad D_R
  \xi_R\big|_{\Sigma_R^\infty} =
  D_{u_{1,R}}\xi_{1,R}\big|_{\Sigma_R^\infty} =
  \eta\big|_{\Sigma_R^\infty}\;.\] This shows that $D_R \xi_R - \eta$
  is supported in $[-R,R] \times [0,1]$. According
  to~\eqref{eq:D0D1xieta} and taking into account the the support of
  the cut-off functions we have
  \[\eta = \beta_R^-\Pi_{u_{0,R}}^{u_R} D_{u_{0,R}} \xi_{0,R}
  + \beta_{-R}^+\Pi_{u_{1,R}}^{u_R} D_{u_{1,R}} \xi_{1,R}\;.\]
  and by a zero addition
  \begin{equation}\label{eq:Dxieta}
    \begin{aligned}
      &D_R \xi_R - \eta\\
      &\ = \left(1-\beta_{R}^- -\beta_{-R}^+\right)D_R \Pi_p^{u_R}\bar
      \xi+\\ &\quad + \left(\ps \beta_R^-\right)
      \left(\Pi_{u_{0,R}}^{u_R}\xi_{0,R} - \Pi_p^{u_R}\bar \xi\right)
      + \beta_R^-\left(D_R \Pi_{u_{0,R}}^{u_R} \xi_{0,R} -
        \Pi_{u_{0,R}}^{u_R} D_{u_{0,R}} \xi_{0,R}\right)\\ &\quad
      +(\ps \beta_{-R}^+) \left(\Pi_{u_{1,R}}^{u_R}\xi_{1,R} -
        \Pi_p^{u_R}\bar \xi\right) +\beta_{-R}^+\left(D_R
        \Pi_{u_{1,R}}^{u_R} \xi_{1,R} - \Pi_{u_{1,R}}^{u_R}
        D_{u_{1,R}} \xi_{1,R}\right)\;.
  \end{aligned}    
  \end{equation}
  Focusing on the third summand of the right hand side without the factor
  $\beta_R^{-}$. After a zero addition we obtain
  \begin{multline}
    \label{eq:D0}
    D_R \Pi_{u_{0,R}}^{u_R} \xi_{0,R} - \Pi_{u_{0,R}}^{u_R}
    D_{u_{0,R}} \xi_{0,R} = D_R\Pi_{u_{0,R}}^{u_R} \left(\xi_{0,R} -
      \Pi_p^{u_ {0,R}}\bar \xi\right)- \\- \Pi_{u_{0,R}}^{u_R}
    D_{u_{0,R}} \left(\xi_{0,R} - \Pi_p^{u_{0,R}} \bar \xi\right) +
    D_R \Pi_{u_{0,R}}^{u_R}\Pi_p^{u_{0,R}}\bar \xi -
    \Pi_{u_{0,R}}^{u_R}D_{u_{0,R}}\Pi_p^{u_{0,R}} \bar \xi\;.
   \end{multline}
  By Lemma~\ref{lmm:commuteDPi} and Lemma~\ref{lmm:uRdecay} we
  find $\mu > 2\delta$ such that
  \begin{multline*}
    \nm{D_R\Pi_{u_{0,R}}^{u_R} \left(\xi_{0,R} - \Pi_p^{u_{0,R}}
      \bar \xi\right) - \Pi_{u_{0,R}}^{u_R} D_{u_{0,R}}
    \left(\xi_{0,R} - \Pi_p^{u_{0,R}} \bar \xi\right)} \\ \leq O(1)
  e^{-\mu(2R-\nm{s})} \left(\nm{\xi_{0,R} - \Pi_p^{u_{0,R}} \bar \xi} + \nm{\nabla
      \left(\xi_{0,R} - \Pi_p^{u_{0,R}} \bar \xi \right)}\right)\;,
  \end{multline*}
  and 
  \[\nm{D_R \Pi_{u_{0,R}}^{u_R}\Pi_p^{u_{0,R}} \bar \xi -
    \Pi_{u_{0,R}}^{u_R}D_{u_{0,R}} \Pi_p^{u_{0,R}} \bar \xi} \leq
  c_1e^{-\mu(2R-\nm{s})} \nm{\bar \xi}\;,\]
  for all $ \nm{s}\leq R$  and $t \in [0,1]$.  Hence by~\eqref{eq:D0} we have
  \begin{multline*}
    \nm{\beta_R^{-} \left( D_R \Pi_{u_{0,R}}^{u_R} \xi_{0,R} -
        \Pi_{u_{0,R}}^{u_R} D_{u_{0,R}} \xi_{0,R}\right)} \\ \leq c_1
    e^{-\mu(2R-\nm{s})} \left(\nm{\xi_{0,R} -\Pi_p^{u_{0,R}} \bar \xi}
      + \nm{\nabla \left(\xi_{0,R} -\Pi_p^{u_{0,R}} \bar \xi\right)} +
      \nm{\bar \xi}\right)\;.
  \end{multline*}
  For $p \in \B^{1,p;\delta}(C,C)$ considered as a constant function
  we have $D_p \bar \xi = 0$ and by lemmas~\ref{lmm:commuteDPi}
  and~\ref{lmm:uRdecay} again we have a constant $c_2$ such that
  \[\nm{\left(1-\beta_R^{-} - \beta_{-R}^+\right)D_R\Pi_p^{u_R}
    \bar \xi} \leq \nm{D_R \Pi_p^{u_R} \bar \xi} \leq c_2
  e^{-\mu(2R -\nm{s})} \nm{\bar \xi}\;,\] for all $\nm{s} \leq R$
  and $t \in [0,1]$.  Integrating the point-wise estimate gives
  \begin{align*}
    &\int_{\Sigma_{-R}^R} \nm{\beta_R^{-} \big(D_R \Pi_{u_{0,R}}^{u_R}
      \xi_{0,R} - \Pi_{u_{0,R}}^{u_R}
      D_{u_{0,R}} \xi_{0,R}\big)}^p\gamma_{\delta,R}^p\d s\d t \\
    &\leq 3^pc_1^p e^{-p\mu R} \int_{\Sigma_{-R}^R}
    \left(\nm{\xi_{0,R} - \Pi_p^{u_{0,R}} \bar \xi}^p +
      \nm{\nabla\left(\xi_{0,R} -
          \Pi_p^{u_{0,R}}\bar \xi\right)}^p\right)e^{p\delta (2R+s)} \d s \d
    t\\
    &\qquad + 3^pc_1^pe^{-p(\mu-\delta)2R} \nm{\bar
      \xi}^p\int_{-R}^R e^{p(\mu-\delta)\nm{s}} \d s \\
    &\leq \left(3^pc_1^p e^{-p\mu R} +
      \frac{3^{p+1}c_1^p}{p(\mu-\delta)}
      e^{-p(\mu-\delta)R}\right) \Nm{\xi_0}^p_{1,p;\delta}\;,
  \end{align*}
  where in the last line we used 
  \[
  e^{-p(\mu-\delta)R} \int_{-R}^R e^{p(\mu-\delta)\nm{s}}\d s =
  \frac{2 e^{-p(\mu-\delta)R}}{p(\mu-\delta)}\left(e^{p(\mu-\delta)R}
    - 1 \right)\leq \frac{3}{p(\mu-\delta)}\;.\] Since $\delta <
  \mu/2$ the last estimate shows
  \begin{equation}
    \label{eq:Du0deltaR}
    \int_{\Sigma_{-R}^R} \nmm{\beta_R^{-} \left(D_R \Pi_{u_{0,R}}^{u_R} \xi_{0,R} - \Pi_{u_{0,R}}^{u_R}
        D_{u_{0,R}} \xi_{0,R}\right)}^p \gamma_{\delta,R}^p \d s \d t \leq O(1) e^{-\delta R} \Nm{\xi_0}_{1,p;\delta}^p\;.
  \end{equation}
  Along the same lines we show that
  \begin{equation}
    \label{eq:Du1deltaR}
  \begin{aligned}
    &\int_{\Sigma_{-R}^R} \nmm{\beta_{-R}^+\left(D_R
        \Pi_{u_{1,R}}^{u_R} \xi_{1,R} - \Pi_{u_{1,R}}^{u_R}
        D_{u_{1,R}} \xi_{1,R}\right)}^p \gamma_{\delta,R}^p \d s \d t
    &&\leq O(1) e^{-\delta R} \Nm{\xi_1}_{1,p;\delta}^p\\
    &\int_{\Sigma_{-R}^R} \nmm{\left(1-\beta_R^{-} -
          \beta_{-R}^+\right)D_R\Pi_p^{u_R} \bar
        \xi}^p\gamma_{\delta,R}^p\d s \d t&&\leq O(1)e^{-\delta
        R}\nm{\bar \xi}\;.
  \end{aligned}    
  \end{equation}
  Focusing now on the second term on the right hand side
  of~\eqref{eq:Dxieta} without the factor $\ps \beta_R^{-}$ we have
  \[\Pi_{u_{0,R}}^{u_R} \xi_{0,R} - \Pi_p^{u_R} \bar \xi =
  \Pi_{u_{0,R}}^{u_R} \left(\xi_{0,R} - \Pi_p^{u_{0,R}} \bar
    \xi\right) + \left(\Pi_{u_{0,R}}^{u_R} \Pi_p^{u_{0,R}} \bar \xi -
    \Pi_p^{u_R} \bar \xi\right) \;.\] By
  Lemma~\ref{lmm:deltadecay} and Corollary~\ref{cor:commutePiPi} 
  \begin{multline*}
    \nm{\Pi_{u_{0,R}}^{u_R} \xi_{0,R} - \Pi_p^{u_R} \bar \xi} \leq
  \nm{\xi_{0,R} - \Pi_p^{u_{0,R}}\bar \xi} +\nm{\Pi_{u_{0,R}}^{u_R} \Pi_p^{u_{0,R}} \bar \xi -
    \Pi_p^{u_R} \bar \xi}\\ \leq 
  O(1)e^{-\delta R} \Nm{\xi_0}_{1,p;\delta}  +O(1)
  e^{-\mu(2R-\nm{s})} \nm{\bar \xi} \leq O(1)e^{-\delta R} \Nm{\xi_0}_{1,p;\delta} \;,
  \end{multline*}
  for all $ \nm{s} \leq R$. Since the support of $\ps \beta_R^{-}$ is
  in $[R-1,R] \times [0,1]$ and $\nm{\ps \beta_R^{-}}<2$ for all $R$ we have
  \begin{equation*}
    \int_{\Sigma_{-R}^R} \nmm{\ps \beta_R^{-}\left(\Pi_{u_{0,R}}^{u_R} \xi_{0,R} - \Pi_p^{u_R}
        \bar \xi\right)}^p \gamma_{\delta,R}^p \d s \d t \leq O(1)e^{-\delta
      R}\Nm{\xi_0}_{1,p;\delta}^p\;. 
  \end{equation*}
  By a completely symmetric argument 
  \begin{equation*}
    \int_{\Sigma_{-R}^R} \nmm{\ps \beta_{-R}^+ \left(\Pi_{u_{1,R}}^{u_R}\xi_{1,R} - \Pi_p^{u_R}\bar
        \xi\right)}^p \gamma_{\delta,R}^p \d s \d t\leq O(1) e^{-\delta
    R}\Nm{\xi_1}_{1,p;\delta}^p\;.
  \end{equation*}
  Denote $(\eta_{0,R},\eta_{1,R}) = \Xi_R \eta$. Since $Q_{01}$ is
  bounded and by Lemma~\ref{lmm:estBR} 
  \begin{equation*}
    \Nm{\xi_0}_{1,p;\delta} + \Nm{\xi_1}_{1,p;\delta} \leq O(1)
  \left(\Nm{\eta_{0,R}}_{p;\delta} +
    \Nm{\eta_{1,R}}_{p;\delta}\right)  = O(1) \Nm{\eta}_{p;\delta,R}
  \end{equation*}
  By the identity~\eqref{eq:Dxieta} as well as~\eqref{eq:Du0deltaR},
  \eqref{eq:Du1deltaR} and the last three estimates we have
  \[
  \Nm{D_R \xi_R - \eta}_{p;\delta,R} \leq O(1)e^{-\delta R}
  \left(\Nm{\xi_0}_{1,p;\delta} + \Nm{\xi_1}_{1,p;\delta}\right) \leq
  O(1)e^{-\delta R}\Nm{\eta}_{p;\delta,R}\;.\] This shows the
  claim.
\end{proof} 
\begin{cor}\label{cor:QR}
  There exists uniform constants $c$ and $R_0$ and for all $R \geq
  R_0$ there exists an operator $Q_R:L_R \to H_R'$, which is
  a right inverse for $D_R'$ and we have for all $\eta \in L_R$
  \begin{equation}
    \label{eq:QRbound}
  \Nm{Q_R \eta}_{1,p;\delta,R} \leq c
  \Nm{\eta}_{p;\delta,R}\;.    
  \end{equation}
\end{cor}
\begin{proof}
  Let $R_0$ and $c$ denote the constants from Lemma~\ref{lmm:approxQ}.
  By possibly increasing $R_0$ we assume that $c e^{-\delta R_0}
  <1/2$. With Lemma~\ref{lmm:approxQ} the
  composition $D_R \circ \wt Q_R$ is invertible for all $R \geq R_0$
  and we define
  \begin{equation}
    \label{eq:QR}
    Q_R := \wt Q_R \left(D_R\wt Q_R\right)^{-1} = \wt Q_R \sum_{k=0}^\infty (1-D_R \wt Q_R)^k\,.
  \end{equation}
  Given for some $\eta \in L_R$ we have $ \Nm{Q_R
    \eta}_{1,p;\delta,R} \leq c
  \Nm{\eta}_{p;\delta,R}\sum_{k=0}^\infty 2^{-k} = 2c
  \Nm{\eta}_{p;\delta,R}$.
\end{proof}
\subsection{Quadratic estimate}
We build up the quadratic estimate which is needed to run the
Newton-Picard theorem. Fix some $\e>0$ and denote by $H'_R(\e) \subset
H'_R$ the ball of all $\xi$ with $L^\infty$-norm strictly smaller than
$\e$. Define the non-linear map
\begin{equation}
  \label{eq:nonlinearF}
  \F_{R} :H'_R(\e) \to L_R,\qquad \xi\mapsto
  \Pi_{u_\xi}^{u_R}\big(\ps u_\xi + J_R(u_\xi) \left(\pt
  - X_R(u_\xi)\right) \big)\,,  
\end{equation}
with $u_\xi := \exp_{u_R} \xi$.  By the special choice of the metric
and definition of the space $H'_R$, we have $u_\xi(\pm \infty) \in
W_\pm$. In particular if $\xi$ is a zero of $\F_R$, then $u_\xi$ is an
element of $\wt \M(W_-,W_+)$.
\begin{lmm}[Quadratic estimate]\label{lmm:quadratic}
  There exists constants $R_0$, $\e$ and $c$ such that we have the
  following uniform bounds. For all $R \geq R_0$ it holds
  \begin{equation}
    \label{eq:CRuRdecay}
  \Nm{\F_R(0)}_{p;\delta,R}  \leq c e^{-2\delta R}\;.  
  \end{equation}
  If $\xi,\xi' \in T_{u_R}\B^{1,p;\delta}(C_-,C_+)$ such that
  $\Nm{\xi}_{L^\infty} < \e$ then
  \begin{equation}
    \label{eq:NmdFD}
    \Nm{\d \F_R(\xi)\xi' - D_R\xi'}_{p;\delta,R} \leq c \Nm{\xi}_{1,p;\delta,R} \Nm{\xi'}_{1,p;\delta,R}\;.  
  \end{equation}
\end{lmm}
\begin{proof}
  We show estimate~\eqref{eq:CRuRdecay}.  Since $u_0$ is
  $(J_0,X_0)$-holomorphic, $u_1$ is $(J_1,X_1)$ holomorphic and by
  definition of $u_R$ and the glued structures $(J_R,X_R)$ we have
  that $\F_R(0)=\CR_{J_R,X_R} u_R$ is supported in $[-1,1]\times
  [0,1]$ and moreover for all $s \in [-1,1]$ and $t \in [0,1]$ we have
  \[ (\CR_{J_R,X_R} u_R)(s,t) = \ps u_R(s,t)+ J_\infty(t,u_R(s,t))(\pt
  u_R(s,t)-X_H(t,u_R(s,t)))\;. \] Since $J_\infty$ and $X_H$ is
  uniformly bounded and by the decay of the preglued map in the neck
  region (\cf Lemma~\ref{lmm:uRdecay}) we have
  \[ \int_{\Sigma^1_{-1}} \nm{\CR_{J_R,X_R} u_R}^p e^{p\delta(2R-
    \nm{s})}\d s \d t \leq O\big(e^{-2R
    p(\mu-\delta)}(1-e^{p(\mu-\delta)})\big)=O\big(e^{-2Rp\delta}\big)
  \;.\] 
  
  We show~\eqref{eq:NmdFD}. We have $\d \F_R(0) = D_R$. Integrate the
  pointwise estimate from Lemma~\ref{lmm:estdFDu} to obtain
  \begin{multline}\label{eq:dFRDu}
    \Nm{\d \F_R(\xi)\xi' - D_R\xi'}_{p;\delta,R}\\ \leq O
    \big(\Nm{\xi'}_\infty \Nm{\xi}_\infty \Nm{\d
        u_R}_{p;\delta,R}+ \Nm{\na
        \xi}_{p;\delta,R}\Nm{\xi'}_\infty + \Nm{\xi}_\infty
      \Nm{\na \xi'}_{p;\delta,R}\big)\;.
  \end{multline}
  The norm $\Nm{J_R}_{C^2}$ appearing
  in~\ref{lmm:estdFDu} is independent of $R$.  By definition and
  Lemma~\ref{lmm:uRdecay} we have
  \begin{align*}
    \Nm{\d u_R}_{p;\delta,R}^p &= \Nmm{\Res{\d u_0}{\Sigma_{-\infty}^0}}_{p;\delta}^p + \Nmm{\Res{\d u_1}{\Sigma_0^\infty}}_{p;\delta}^p +
    \int_{\Sigma^{2R}_{-2R}} \nm{\d u_R}^p\gamma^p_{\delta,R} \d
    s\d t\\
    &= \Nmm{\Res{\d u_0}{\Sigma_{-\infty}^0}}_{p;\delta}^p + \Nmm{\Res{\d u_1}{\Sigma_0^\infty}}_{p;\delta}^p +
    O(1)\int_0^{2R}e^{p(\delta-\mu)(2R-s)} \d s \leq O(1)\,.
  \end{align*}
  We obtain~\eqref{eq:NmdFD} by the last estimate and
  Lemma~\ref{lmm:comparenorms} using~\eqref{eq:dFRDu}.  This finishes
  the proof.
\end{proof}
We come to the key result of this section. For some small number
$\e>0$, we denote by $\ker_\e D_R \subset \ker D_R$ all elements in
the kernel with norm smaller than $\e$.
\begin{lmm}~\label{lmm:sol} There exists constants $\e$ and $R_0$ such
  that for all $R \geq R_0$ there exists a map \[ \sol_R :\ker_\e D_R
  \longrightarrow \im Q_R\;,
  \]
  which satisfies the following properties
  \begin{enumerate}[label=(\roman*)]
  \item for all $\xi \in \ker_\e D_R$ the map $\exp_{u_R} (\xi +
    \sol_R(\xi))$ is $(J_R,X_R)$-holomorphic,
  \item for each $R\geq R_0$ the map $\sol_R$ is differentiable and we
    have a constant $c>0$ such that for all $R \geq R_0$
    \[
    \Nm{\sol_R}_{C^1} \leq c e^{-2\delta R}\,,
    \]
  \item for every $\xi' \in T_{u_R}\B$ such that $\exp_{u_R} \xi'$ is
    $(J_R,X_R)$-holomorphic and satisfies $\Nm{\xi'}_{1,p;\delta,R} <
    \e$ there exist $\xi \in \ker_e D_R$ such that $\xi' = \xi +
    \sol_R(\xi)$.
  \end{enumerate}
\end{lmm}
\begin{proof}
  This is a direct consequence of the Newton-Picard theorem provided the
  quadratic estimate of the non-linear map given in
  Lemma~\ref{lmm:quadratic} and the uniform bound on the right inverse
  established in Corollary~\ref{cor:QR}. Set $\mathcal{N}_R(\xi) :=
  \F_R(\xi) -\F_R(0) - D_R\xi$ for all $\xi \in T_{u_R}\B$. Let $\e$
  and $R_0$ be the constants from Lemma~\ref{lmm:quadratic}. Given
  $\xi_0,\xi_1 \in T_{u_R}\B$. By the mean-value theorem there exists
  $\theta \in [0,1]$ such that
  \[
  \F_R(\xi_0) - \F_R(\xi_1) = \d \F_R\big(\theta \xi_0 + (1-\theta)\xi_1\big)(\xi_0 -\xi_1)\;.
  \]
  If $\Nm{\xi_0}_{1,p;\delta}+\Nm{\xi_1}_{1,p;\delta} \leq \e$ we
  conclude using~\eqref{eq:NmdFD} that there are constants $c_1,c_2$
  such that (where for convenience we have dropped the subindex of the
  norms since they are clear from the context)
  \begin{align*}
    \Nm{Q_R \mathcal{N}_R(\xi_0) -Q_R
      \mathcal{N}_R(\xi_1)}&\leq
    c_1\Nm{\NN_R(\xi_0) - \NN_R(\xi_1) }\\
    &=c_1
    \Nm{\F_R(\xi_0)-\F_R(\xi_1) - D_R (\xi-\xi_1)}\\
    &=c_1\Nm{\d \F_R\big(\theta\xi_0 + (1-\theta)\xi_1\big) (\xi_0-\xi_1) -D_R(\xi_0-\xi_1)}\\
    &\leq c_1c_2\Nm{\theta \xi_0 + (1-\theta)\xi_1}\,\Nm{\xi_0-\xi_1}\\
    &\leq c_1c_2\big(\Nm{\xi_0} +
    \Nm{\xi_1}\big) \Nm{\xi_0
      -\xi_1}\\
    &\leq c_1c_2 \e \Nm{\xi_0-\xi_1}\,,
  \end{align*}
  and using~\eqref{eq:CRuRdecay} we find another constant $c_3$ such that
  \[
  \Nm{Q_R \F_R(0)}_{1,p;\delta,R} \leq c_1\Nm{\F_R(0)}_{p;\delta,R} \leq
  c_1c_3e^{-2\delta R}\;.
  \]
  By \cite[Proposition 24]{Floer:Monopoles} and after possibly making
  $R_0$ bigger and $\e$ smaller we conclude that for all $R \geq R_0$
  there exists a map $\sol_R$ satisfying all three properties.  More
  precisely we must have $\e$ so small that $c_1c_2\e \leq 1/4$ and
  $R_0$ so large that $ c_1c_3e^{-2\delta R_0}\leq \e/2$.
\end{proof}
\subsection{Continuity of the gluing map}
The only non-trivial issue is continuity of the solution maps with
respect to the gluing parameter, which essentially reduces to a
question of continuity of the family of right-inverses.  We denote by
$\ker \D$ the vector-bundle over the base $[R_0,\infty)$ with fibre
$\ker D_R$ and $\ker_\e \D \subset \ker \D$ the disk bundle with fibre
$\ker_\e D_R$.
\begin{lmm}\label{lmm:solcont}
  There exists constants $R_0$ and $\e$ such that the map
  \[ \sol:\ker_\e \D|_{[R_0,\infty)} \to T\B^{1,p;\delta},\qquad (R,\xi) \mapsto
  \sol_R(\xi)\,,\] is continuous.
\end{lmm}
\begin{proof}\setcounter{stp}{0}
  We follow the proof of \cite[Prp. 5.5]{Abouzaid:spheres}. Let $R_0$
  and $\e$ denote the constants from Lemma~\ref{lmm:sol}. In the
  course of the proof we possibly have to increase $R_0$ and decrease
  $\e$. Given sequences $(\xi_\nu)$ and $(R_\nu)$ such that $R_\nu\geq
  R_0$ and $\xi_\nu \in \ker_\e D_{R_\nu}$ for all $\nu$. Suppose that
  $R_\nu \to R$ and $\Pi_\nu \xi_\nu \to \xi \in \ker_\e D_R$, where
  we write $\Pi_\nu :=\Pi_{u_{R_\nu}}^{u_R}$ for the parallel
  transport map. We also abbreviate $\sol_\nu := \sol_{R_\nu}$,
  $u:=u_R$, and $u_\nu:=u_{R_\nu}$.  We have to show that $\lim_{\nu
    \to \infty} \Pi_\nu \sol_\nu(\xi_\nu) = \sol_R (\xi)$. Arguing
  indirectly we assume that there exists a subsequence $(\nu_k)
  \subset (\nu)$ such that
  \begin{equation}
    \label{eq:contradictsolcont}
    \lim_{k\to \infty} \nolimits \Nm{\Pi_{\nu_k} \sol_{\nu_k} (\xi_{\nu_k}) - \sol_R (\xi)}_{1,p;\delta}>0\;.    
  \end{equation}
  Without loss of generality we assume that $\nu_k=k$ for all $k\in
  \N$. We now build up a contradiction to~\eqref{eq:contradictsolcont}
  in the following three steps.
  \begin{stp}
    Define $w_\nu := \exp_{u_\nu} (\xi_\nu + \sol_\nu (\xi_\nu))$. A
    subsequence of $w_\nu$ Gromov converges to a $J_R$-holomorphic
    strip $w:\Sigma \to M$.
  \end{stp}
  By the first property of the solution map we know that $w_\nu$ is
  $J_{R_\nu}$-holomorphic. By the second property and general bounds for the derivative of the exponential map (\cf Corollary~\ref{cor:dwxi})
  \[
  \nm{\d w_\nu} \leq O\left(\nm{\d u_{\nu}} +\nm{\na \xi_\nu} +
    \nm{\na \sol_\nu( \xi_\nu)}\right) \leq O\left(\nm{\d u_0}
    +\nm{\d u_1} +\nm{\na \xi} + e^{-\delta R}\right)\;.
  \]
  In particular we conclude that the gradient of $w_\nu$ is uniformy
  bounded. By local compactness we conclude the existence of $w$ such
  that $w_\nu \to w$ in $C^\infty_\loc$ (\cf
  Lemma~\ref{lmm:bgcomp}). In remains to control the convergence on
  the ends. Denote $p_+ := u_{\nu}(\infty)=u(\infty)$. Choose large
  constants $s_0$, $\nu_0$ and estimate for all $s\geq s_0$ and
  $\nu\geq \nu_0$ using exponential decay for $u$ (\cf
  Theorem~\ref{thm:remove}), omitting $(s,t)$ whenever convenient
  \begin{align*}
    \di{w_\nu(s,t),p_+} &\leq \di{w_\nu,u_{\nu}} +
    \di{u_{\nu},u} +\di{u,p_+} \\
    &\leq \nm{\xi_\nu} + \Nm{\sol_{\nu}(\xi_\nu)}_\infty + o(1) + O(e^{-\mu s})\\
    &\leq \e + O\big(e^{-\delta R_0}\big) + o(1) +
    O(e^{-\mu s})\,.
  \end{align*}
  After possibly decreasing $\e$ and increasing $R_0$, $s_0$ and
  $\nu_0$ the right-hand side is smaller than the diameter of a ball
  about $p_+$ which lies completely in the Pozniak neighborhood
  $U_\Poz$ for all $s\geq s_0$ and $\nu\geq \nu_0$. Hence the image of
  $w_\nu$ restricted to $\Sigma_{s_0}^\infty$ lies in $U_\Poz$, where
  the symplectic form is exact $\omega=\d \lambda$ and $\lambda$
  vanishes on $L_0 \cap U_\Poz$ and $L_1 \cap U_\Poz$. By exactness
  and $C^\infty_\loc$-convergence we conclude
  \[
  \int_{\Sigma_{s_0}^\infty} w_\nu^* \omega = \int_0^1
  w_\nu\big|_{s=s_0}^*\lambda \to \int_0^1 w\big|^*_{s=s_0}\lambda =
  \int_{\Sigma_{s_0}^\infty} w^*\omega\,.
  \]
  By convergence of the energy we have $C^0$-convergence on the end, \ie
  $w_\nu$ converges to $w$ in $C^0(\Sigma_{s_0}^\infty)$ (\cf
  Lemma~\ref{lmm:convends}). We proceed similarly for the negative
  end to show that $w_\nu$ converges to $w$ in
  $C^0(\Sigma_{-\infty}^{-s_0})$ and hence $w_\nu \to w$ in $C^0$ and
  $E(w_\nu)\to E(w)$.
  \begin{stp}
    There exists a vector field $\xi'' \in \Gamma(u^*TM)$ such that
    $\exp_u \xi'' = w$ and moreover we have
    \[
    \lim_{\nu \to \infty} \Nm{\Pi_\nu \sol_\nu( \xi_\nu) + \xi -
      \xi''}_{1,p;\delta} = 0\,.
    \]
  \end{stp}
  By the last step the vector field $ \zeta_\nu:=\exp_w^{-1} w_\nu$ is
  well-defined for all $\nu$ sufficiently large. We estimate for any
  $(s,t) \in \Sigma$ omitting the arguments $s$ and $t$ for convenience
  \begin{multline*}
    \di{u,w} \leq \di{u,u_\nu}+ \di{u_\nu,w_\nu} + \di{w_\nu,w} \\\leq
    \di{u,u_\nu}+ \nm{\xi_\nu+\sol_\nu(\xi_\nu)} + \Nm{\zeta_\nu}_\infty
    \leq \e + O\big(e^{-\delta R_0}\big) +o(1)\,,
  \end{multline*}
  since by the last step $\Nm{\zeta_\nu}_\infty \to 0$. Hence
  $\di{u,w} \leq \e + O\big(e^{-\delta R_0}\big)$ and after possibly
  decreasing $\e$ and increasing $R_0$ again we assume that the
  distance from $u(s,t)$ to $w(s,t)$ is smaller than the injectivity
  radius for any $(s,t) \in \Sigma$. In particular the vector field
  $\xi'' := \exp_u^{-1} w$ is well-defined. Because the strips $u$ and
  $w$ are elements of $B^{1,p;\delta}(C_-,C_+)$
  Lemma~\ref{lmm:uvclose} shows that the norm
  $\Nm{\xi''}_{1,p;\delta}$ is finite. By construction it holds that
  $\sol_\nu(\xi_\nu) = \exp_{u_\nu}^{-1} w_\nu -\xi_\nu$ and we estimate
  \begin{multline*}
    \Nm{\Pi_\nu \sol_\nu (\xi_\nu) + \xi - \xi''}_{1,p;\delta} = \Nmm{\Pi_\nu\exp_{u_\nu}^{-1} w_\nu - \Pi_\nu \xi_\nu + \xi -\exp_u^{-1} w}_{1,p;\delta}  \\
    \leq \Nmm{\Pi_\nu \exp_{u_\nu}^{-1} w_\nu - \exp_u^{-1}
      w}_{1,p;\delta} + \Nm{\Pi_\nu\xi_\nu -\xi}_{1,p;\delta}\,.
  \end{multline*}
  To show the claim it remains to see that the first summand on the
  right-hand converges to zero as $\nu$ tends to $\infty$.  Define the
  points $q_\nu:=w_\nu(\infty)$ and $q:=w(\infty)$ as well as the
  vector field
  \[
  \xi^+_\nu :=\Pi_{u_\nu}^u \exp_{u_\nu}^{-1} w_\nu - \exp_u^{-1} w
  -\wh \Pi_{p_+}^u\left(\exp_{p_+}^{-1} q_\nu -
    \exp_{p_+}^{-1}q\right) \in \Gamma(u^*TM)\,.
  \]
  We use corollaries~\ref{cor:dwxi} and~\ref{cor:explip} to estimate the norm of $\xi^+_\nu$ by 
  \begin{equation*}
     \nmm{\Pi_{u_\nu}^u \exp_{u_\nu}^{-1} w_\nu -
      \exp_u^{-1}
      w_\nu} + \nmm{\exp_u^{-1} w_\nu -\exp_u^{-1}w} +\nmm{\exp_{p_+}^{-1} q_\nu - \exp_{p_+}^{-1} q}
  \end{equation*}
  and conclude that 
  \begin{equation*}
    \nmm{\xi^+_\nu}\leq O(\di{u_\nu,u} + \di{w_\nu,w} +
    \di{q_\nu,q})\,.
  \end{equation*}
  With Corollary~\ref{cor:commutenaPi} using the notation
  $\zeta_\nu':=\exp_u^{-1}u_\nu \in \Gamma(u^*TM)$ we bound the norm of $\na \xi^+_\nu$ with
  \begin{multline*}
    \nmm{\big(\na \Pi_{u_\nu}^u-\Pi_{u_\nu}^u \na\big)
      \exp_{u_\nu}^{-1} w_\nu} + \nmm{\Pi_{u_\nu}^u \na
      \exp_{u_\nu}^{-1} w_\nu - \na \exp_u^{-1} w_\nu}
    +\\+\nmm{\na\exp_u^{-1} w_\nu - \na
      \exp_u^{-1}w } +\nmm{\na \wh \Pi_{p_+}^u \big(\exp_{p_+}^{-1} q_\nu -\exp_{p_+}^{-1}q\big)}
   \end{multline*}
   Hence we conclude that
   \begin{multline*}
     \nm{\na\zeta^+} \leq O(1)\Big(\di{u,u_\nu}(\nm{\d u} + \nm{\d u_\nu}) + \nm{\na
       \zeta_\nu'}+\nm{\zeta_\nu}(\nm{\d u}+ \nm{\d w})+ \\
     \nm{\na \zeta_\nu} + \nm{\d u} \di{q_\nu,q}
     \Big)\,.
   \end{multline*}
  The last two estimates and using $C^\infty_\loc$ converges of
  $w_\nu$ to $w$ we conclude that for a fixed $s$ we have
  \begin{equation}
    \label{eq:C1locxi+}
    \lim_{\nu \to \infty} \Nm{\xi^+_\nu}_{C^1(\Sigma_{0}^{s})} = 0\,. 
  \end{equation}
  Choose $\mu$ such that $2\delta<\mu<\iota$ with $\iota$ as defined
  in~\eqref{eq:fixdelta} and let $\e_0 = \e_0(\mu)$ be the associated
  constant from Lemma~\ref{lmm:Edecay}. Now choose $s_0$ large enough
  such that $E(w,\Sigma_{s_0}^\infty) <\e_0/2$. By convergence of the
  energy as established in the last step there exists $\nu_0 \in \N$
  such that $E(w_\nu,\Sigma_{s_0}^\infty) < \e_0$ for all $\nu\geq
  \nu_0$. Thus the assumptions of Lemma~\ref{lmm:Edecay} are met and
  there exists constant $c_1$ independent of $\nu$ such that
  \[
  \di{w_\nu,w_\nu(\infty)} + \nm{\d w_\nu} \leq c_1 e^{-\mu s},\qquad
  \forall\ s \geq s_0\,.
  \]
  Without loss of generality we assume that the same holds with
  $w_\nu$ replaced by $u$, $w$ and $u_\nu$. Then the previous
  estimates show that there exists a constant $c_2$ such that for all
  $s \geq s_0$ and $\nu$ large enough
  \begin{equation}
    \label{eq:naxi+nu}
    \nm{\na \xi^+_\nu} \leq c_2 e^{-\mu s}\,.  
  \end{equation}
  Fix $t \in [0,1]$ and $\nu$ for the moment and define the function
  $f:[s_0,\infty) \to \R$ by $f(s):=\nm{\xi^+_\nu(s,t)}$. We claim
  that $\lim_{s\to \infty} f(s) = 0$. For $s$  large enough we replace
  $\wh \Pi$ with $\Pi$ in the formula for $\xi^+$ and estimate $f(s)$ with
  \[
  \nmm{\exp_{u_\nu}^{-1} w_\nu-\Pi_{p_+}^{u_\nu}
    \exp_{p_+}^{-1} q_\nu} +
  \nmm{\big(\Pi_{p_+}^{u_\nu}-\Pi_u^{u_\nu}\Pi_{p_+}^u\big)\exp_{p_+}^{-1}
    q_\nu} + \nmm{\exp_u^{-1} w - \Pi_{p_+}^u \exp_{p_+}^{-1} q}\,.\]
  The first and the last summand converge to zero as $s$ tends to
  $\infty$ by Lemma~\ref{lmm:uvclose} and so does the second summand
  after Corollary~\ref{cor:commutePiPi}. Hence
  \[ f(s) = \int_s^\infty - \partial_\sigma f(\sigma) \d \sigma \leq
  \int_s^\infty \nmm{\na\xi^+_\nu} \d \sigma \leq c_2 \int_s^\infty
  e^{-\mu \sigma } \d \sigma \leq \frac{c_2}{\mu} e^{-\mu s}\,.\] 
  Using the last estimate and~\eqref{eq:naxi+nu} we see that there exists an universal constant $c_3$ such that for all $s \geq s_0$ 
  \begin{multline*}
    \int_{\Sigma_0^\infty} \left(\nm{\xi^+_\nu}^p +\nm{\na
        \xi^+_\nu}^p\right) \kappa_\delta \d s \d t \leq \\ \leq
    \Nm{\xi^+_\nu}^p_{C^1(\Sigma_0^s)} \int_0^s e^{\delta s}
    \d s + (c_2^p +(c_2/\mu)^p) \int_s^\infty e^{-(\mu-\delta)p s}
    \d s \leq \\ \leq c_3 e^{\delta s}\Nm{\xi^+_\nu}_{C^1(\Sigma_0^s)}^p
    + c_3 e^{-(\mu-\delta)p s}\,.
  \end{multline*}
  According to~\eqref{eq:C1locxi+} the right-hand side converges to
  $c_3 e^{-(\mu-\delta)p s}$ as $\nu \to \infty$ and since $s$ was
  chosen freely, we see that left-hand side converges to zero as $\nu
  \to \infty$. Similar we proceed with the negative end to show that
  \[\Nm{\Pi_\nu \exp_{u_\nu}^{-1}w_\nu- \exp_u^{-1} w}_{1,p;\delta} \to
  0\,,\]
  This shows the claim.
  \begin{stp}
    We show that $\Nm{\Pi_\nu\sol_\nu (\xi_\nu) - \sol_R
      (\xi)}_{1,p;\delta} \to 0$ as $\nu \to \infty$ in contradiction
    to~\eqref{eq:contradictsolcont}.
  \end{stp}
  By the last step, we see that $\Pi_\nu \sol_\nu (\xi_\nu)$ converges
  to $\xi':=\xi''-\xi$.  Define $\eta' := D_R \xi'$ and
  $\eta'_\nu:=D_{R_\nu}\sol_\nu(\xi_\nu)$, then using
  Corollary~\ref{cor:SobcommuteDPi} and~\ref{cor:Dcont}
  \begin{align*}
    \Nm{\eta'-\Pi_\nu \eta'_\nu}_{p;\delta}&= \Nm{D_R \xi'- \Pi_\nu D_{R_\nu} \sol_\nu (\xi_\nu)}_{p;\delta}\\
    &\leq \Nm{D_R \xi' -D_R \Pi_\nu \sol_\nu( \xi_\nu)}_{p;\delta} + \Nm{\Pi_\nu D_{R_\nu} \sol_\nu (\xi_\nu) - D_R \Pi_\nu \sol_\nu (\xi_\nu)}_{p;\delta}\\
    &\leq O(1) \Nm{\xi' -\Pi_\nu \sol_\nu (\xi_\nu)}_{1,p;\delta} + o(1)
    = o(1)\;.
  \end{align*}
  Since $Q_{R_\nu}$ is a right inverse to $D_{R_\nu}$ and
  $\sol_\nu(\xi_\nu) \in \im Q_{R_\nu}$ we have $Q_{R_\nu} \eta'_\nu =
  \sol_\nu( \xi_\nu)$ and using the fact that $R \mapsto Q_R\eta'$ is
  continuous for a fixed $\eta'$ (see Lemma~\ref{lmm:QRcnt}) we have
  (omitting the subscripts of the norms for convenience)
  \begin{align*}
    \Nm{Q_R \eta' - \Pi_\nu \sol_\nu(\xi_\nu)}&= \Nm{Q_R \eta' - \Pi_\nu Q_{R_\nu}\eta'_\nu}\\
    &\leq \Nm{\Pi_\nu Q_{R_\nu} \Pi_\nu^{-1} \eta' - \Pi_\nu Q_{R_\nu} \eta'_\nu} + \Nm{Q_R\eta' - \Pi_\nu Q_{R_\nu} \Pi_\nu^{-1} \eta'}\\
    &\leq O(1) \Nm{\eta'-\Pi_\nu\eta'_\nu} + o(1) = o(1)\;.
  \end{align*}
  Hence $\xi' = Q_R \eta'$ and from $\F_R(\xi+ \xi')=0$ it follows
  that there exists $\xi_0 \in \ker D_R$ such that
  \[\xi + \xi'=\xi_0 +
  \sol_R( \xi_0)\;.\] We have the splitting $T_{u_R} \B = \ker D_R
  \oplus \im Q_R$. Sine $\xi',\sol_R( \xi_0) \in \im Q_R$ we conclude
  that $\xi_0 = \xi$ and $\xi' = \sol_R (\xi_0)=\sol_R (\xi)$. In
  particular 
\[\Nm{\Pi_\nu \sol_\nu( \xi_\nu)-\sol_R( \xi)}_{1,p;\delta}
\to 0\,,\] contradicting~\eqref{eq:contradictsolcont} and proving the
lemma.
\end{proof}
\begin{lmm}\label{lmm:QRcnt}
  Fix $\eta \in L_R$ and given a sequence $R_\nu \to R$ then
  \[
  \lim_{\nu \to \infty} \Nm{Q_{R^\nu}\Pi_{u_R}^{u_{R_\nu}} \eta -
    \Pi_{u_R}^{u_{R_\nu}} Q_R \eta}_{1,p;\delta} =0\,.
  \]
\end{lmm} 
\begin{proof}
  Abbreviate the norm
  $\Nm{\,\cdot\,} := \Nm{\, \cdot\,}_{1,p;\delta}$, the operators
  $D:=D_R$, $D_\nu:=D_{R_\nu}$, $Q:= Q_R$, $Q_\nu := Q_{R_\nu}$, $\wt
  Q:=\wt Q_R$, $\wt Q_\nu:=\wt Q_{R_\nu}$,
  $\Pi_\nu:=\Pi_{u_R}^{u_{R_\nu}}$ and the vector $\eta_j := (1-D \wt
  Q)^j \eta$ for all $j = 0,\dots,k$. We estimate using dominated
  convergence
  \begin{align*}
    \lim_{\nu \to \infty} \Nm{\big(\Pi_\nu Q - Q_\nu\Pi_\nu\big) \eta}
    &\leq \lim_{\nu \to \infty} \sum_{k \geq 0} \Nm{\Pi_\nu\wt Q(1-D\wt Q)^k \eta -\wt Q_\nu(1-D_\nu\wt Q_\nu)^k \Pi_\nu \eta}\\
    &= \sum_{k \geq 0} \lim_{\nu \to \infty}  \Nm{\Pi_\nu\wt Q(1-D\wt Q)^k \eta -\wt Q_\nu(1-D_\nu\wt Q_\nu)^k \Pi_\nu \eta}\\
    &\leq O(1) \sum_{k \geq 0} \sum_{j=0}^k \lim_{\nu \to \infty}
    \Nmm{(\Pi_\nu \wt Q-\wt Q_\nu \Pi_\nu) \eta_j}\,,
  \end{align*}
  where the last inequality follows from
  Corollary~\ref{cor:SobcommuteDPi}. According to the preceding
  consideration we see that it suffices to show the lemma for the
  corresponding approximate right-inverses.

  We have by definition with $\xi:=Q_{01} \Xi_R\eta$
  \begin{equation}
    \label{eq:wtQRdecompose}
    \Nmm{(\Pi_\nu \wt
    Q-\wt Q_\nu \Pi_\nu)\eta}_{1,p;\delta} \leq \Nmm{\big(\Pi_\nu \Theta_R
    -\Theta_{R_\nu}\big)\xi}_{1,p;\delta} + \Nmm{\big(\Xi_R
    -\Xi_{R_\nu}\Pi_\nu\big)\eta}_{L_0 \oplus L_1}\,.
  \end{equation}
  We show that both terms on the right-hand side converge to zero
  separately.  In order to control the second term we define the
  paths of vector fields $\tilde \eta_0:\R \to \Gamma(u_0^*TM)$, $\rho
  \mapsto \tilde \eta_{0,\rho}$ and $\tilde \eta_1:\R \to
  \Gamma(u_1^*TM)$, $\rho \mapsto \tilde \eta_{1,\rho}$ where
  \begin{align*}
    \tilde \eta_{0,\rho} &:= \Pi_{u_R \circ \tau_{2R}}^{u_0} \eta
    \circ \tau_{2R} - \Pi_{u_{R+\rho} \circ \tau_{2(R+\rho)}}^{u_0}
    \Pi_{u_R \circ \tau_{2(R+\rho)}}^{u_{R+\rho}\circ
      \tau_{2(R+\rho)}}
    \eta \circ \tau_{2(R+\rho)}\\
    \tilde \eta_{1,\rho} &:= \Pi_{u_R \circ \tau_{-2R}}^{u_1} \eta
    \circ \tau_{-2R} - \Pi_{u_{R+\rho} \circ \tau_{-2(R+\rho)}}^{u_1}
    \Pi_{u_R \circ \tau_{-2(R+\rho)}}^{u_{R+\rho}\circ
      \tau_{-2(R+\rho)}} \eta \circ \tau_{-2(R+\rho)}\,.
  \end{align*} 
  We assume for the moment that $\eta$ is smooth and compactly
  supported. We have with a standard result on the derivative of
  parallel transport maps (\cf Corollary~\ref{cor:commutenaPi}) that
  norm of $\partial_\rho \tilde \eta_{0,\rho}$ is bounded by
  \[
  O(1)\left(\left(\nm{\partial_\rho u_{R+\rho} \circ \tau_{2(R+\rho)}}
      + \nm{\partial_\rho u_R \circ \tau_{2(R+\rho)}}\right)
    \nm{\eta\circ \tau_{2(R+\rho)}} +
    \nm{\na_\rho \eta \circ \tau_{2(R+\rho)}}\right)\,.
  \]
  We conclude in particular that the $\partial_\rho \tilde
  \eta_{0,\rho}$ is uniformly bounded.  Obviously $\tilde
  \eta_{0,0}\equiv 0$. By the mean-value theorem and the last estimate
  we have $\nm{\tilde \eta_{0,\rho}} \leq O(|\rho|)$. Similarly we
  have $\nm{\tilde \eta_{1,\rho}} \leq O(\nm{\rho})$. Therefore with
  $\rho_\nu=R_\nu-R$
  \[
  \Nm{(\Xi_R-\Xi_{R_\nu}\Pi_\nu)\eta}_{L_0 \oplus L_1}^p \leq O(1)
  \int_\Sigma \nm{\wt \eta_{0,\rho_\nu}}^p + \nm{\wt
    \eta_{1,\rho_\nu}}^p \d s \d t \leq O(\nm{\rho_\nu}^p)\,.
  \]
  Thus the second term in~\eqref{eq:wtQRdecompose} converges to zero
  as $\nu$ tends to $\infty$ if $\eta$ is smooth and compactly
  supported. If $\eta$ is not smooth or not compactly supported we
  find arbitrarily close $\eta' \in L_R$  which is smooth and compactly supported such
  that
  \[
  \Nm{(\Xi_R-\Xi_{R_\nu}\Pi_\nu)(\eta -\eta')} \leq \Nm{\Xi_R(\eta-\eta')}
  +\Nm{\Xi_{R_\nu}\Pi_\nu (\eta-\eta')} \leq O(1) \Nm{\eta-\eta'} \,.
  \]
  In particular we assume that $\eta'$ is choosen such that the
  right-hand side is smaller than some arbitrary $\e>0$.  Using the
  above we conclude
  \[
  \lim_{\nu \to \infty} \Nm{(\Xi_R-\Xi_{R_\nu}\Pi_\nu)\eta} \leq \lim_{\nu
    \to \infty} \Nm{(\Xi_R-\Xi_{R_\nu}\Pi_\nu)\eta'} +
  \Nm{(\Xi_R-\Xi_{R_\nu}\Pi_\nu)(\eta-\eta')} < \e\,.
  \]
  This shows that the second term in~\eqref{eq:wtQRdecompose}
  converges to zero for any $\eta$.

  We now show that the first term in~\eqref{eq:wtQRdecompose}
  converges to zero for any fixed $\xi =(\xi_0,\xi_1)\in H_{01}$. By
  the same argument we assume without loss of generality that $\xi_0$
  and $\xi_1$ are smooth and compactly supported derivative.  We
  define vector fields $\xi_R:=\Theta_R \xi$, and
  $\xi_{R_\nu}:=\Theta_{R_\nu} \xi$. By the definition of the
  interpolation~\eqref{eq:xiRdef} we have the point-wise estimate of
  the norm of difference $\Pi_\nu\xi_R -\xi_{R_\nu}$ by 
  \begin{align*}
    &\nmm{(1-\beta_R^- -\beta_{-R}^+)\big( \Pi_{u_R}^{u_{R_\nu}} \wh
      \Pi_p^{u_R}
      -\wh \Pi_p^{u_{R_\nu}}\big) \bar \xi} +\\
    &+ \nmm{\big(\beta^-_R -\beta^-_{R_\nu}\big)\wh
      \Pi_p^{u_{R_\nu}} \bar \xi} + \nmm{\big(\beta^+_{-R}
      -\beta^+_{-R_{\nu}}\big)\wh \Pi_p^{u_{R_\nu}}
      \bar \xi} +\\
    &+ \nmm{\Pi_{u_R}^{u_{R_\nu}} \Pi_{u_{0,R}}^{u_R} \xi_{0,R}
      - \Pi_{u_{0,R_\nu}}^{u_{R_\nu}} \xi_{0,R_\nu}}
    +\nmm{\Pi_{u_R}^{u_{R_\nu}} \Pi_{u_{1,R}}^{u_R} \xi_{1,R} -
      \Pi_{u_{1,R_\nu}}^{u_{R_\nu}} \xi_{1,R_\nu}}+\\
    &+\nmm{(\beta_R^- -\beta_{R_\nu}^-)\Pi_{u_{0,R}}^{u_R}
      \xi_{0,R}}+\nmm{(\beta^+_{-R}
      -\beta^+_{R_\nu})\Pi_{u_{1,R}}^{u_R} \xi_{1,R}} 
  \end{align*}
  Using again the mean value theorem we show that 
  \[
  \nmm{\Pi_\nu \xi_R - \xi_{R_\nu}} \leq O(1)\nm{R-R_\nu}\,.
  \]
  We deduce the same estimate for the norm of $\na(\Pi_\nu\xi_R
  -\xi_{R_\nu})$ and conclude as above that the first term
  of~\eqref{eq:wtQRdecompose} converges to zero. This shows the claim.
\end{proof}
\begin{rmk}
  An estimate similar to Corollary~\ref{cor:SobcommuteDPi} does not
  hold for the right-inverse. In particular $R \mapsto Q_R$ is not a
  continuous path of operators! The failure of uniform continuity is
  due to the fact that the definition of $Q_R$ involves a
  shift-operator. For our purposes pointwise continuity suffices. It
  does not however, if we were to prove higher regularity of the
  gluing map. Then one would need a more sophisticated analytical
  setup, as for example the theory of polyfolds.
\end{rmk}

\subsection{Surjectivity of the gluing map}
In this section we show that the gluing map is asymptotically
surjective. First we need an auxiliary lemma. 
\begin{lmm}\label{lmm:auffang}
  Given a sequence $[(w_\nu)] \subset \M(W_-,W_+)_{[1]}$ which
  Floer-Gromov converges to $[(u_0,u_1)] \in
  \M^1(W_-,W_+)_{[0]}$. There exists a vector field $\xi_\nu$ along
  $u_{R_\nu}$ and a constant $a_\nu \in \R$ such that $w_\nu =
  \exp_{u_{R_\nu}} \xi_\nu \circ \tau_{a_\nu}$ for all but finitely
  many $\nu \in \N$ and $\lim_{\nu \to \infty}
  \Nm{\xi_\nu}_{1,p;\delta,R^\nu}=0$.  
\end{lmm}
\begin{proof}\setcounter{stp}{0}
  We follow the proof of \cite[Lemma 10.12]{Abouzaid:spheres}. By
  Gromov convergence there are two sequences $(b_\nu), (c_\nu) \subset
  \R$ such that $w_\nu \circ \tau_{b_\nu} \to u_0$ and $w_\nu \circ
  \tau_{c_\nu} \to u_1$ in $C^\infty_\loc$. We define $a_\nu :=
  1/2(b_\nu + c_\nu)$ and $2R_\nu:= 1/2(c_\nu-b_\nu)$.  Set
  \[v_\nu:=w_\nu\circ \tau_{-a_\nu}\,.\] Then $v_\nu$ is
  $(J_{R_\nu},X_{R_\nu})$-holomorphic with respect to the glued
  structures $J_{R_\nu}=J_0\#_{R_\nu} J_1$ and
  $X_{R_\nu}=X_0\#_{R_\nu}X_1$ and satisfies $v_\nu \circ
  \tau_{2R_\nu} \to u_0$ and $v_\nu \circ \tau_{-2R_\nu} \to u_1$ in
  $C^\infty_\loc$.
  \begin{stp}\label{stp:C0vu}
    We have $ \lim_{\nu\to \infty} \sup_{(s,t) \in \Sigma} \nolimits
    \di{v_\nu(s,t),u_{R_\nu}(s,t)} = 0$.
  \end{stp}
  By Floer-Gromov convergence we have convergence the energy $E(v_\nu)
  \to E(u_0)+ E(u_1)=:E$.  For any $\e_0$ there exists $s_0=s_0(\e_0)$
  large enough such that
  \[
  E(u_0;\Sigma_{-s_0}^{s_0})+E(u_1;\Sigma_{-s_0}^{s_0}) \geq
  E-\e_0/2\,.
  \]
  By $C^\infty_\loc$-convergence on the compact set
  $\Sigma_{-s_0}^{s_0}$ this implies that there exists $\nu_0$ such
  that for all $\nu \geq \nu_0$ we have
  \begin{equation}
    \label{eq:wnuEsmall}
      E(v_\nu;\Sigma_{-\infty}^{-2R_\nu-s_0})+E(v_\nu;\Sigma_{-2R_\nu+s_0}^{2R_\nu-s_0}) +E(v_\nu;\Sigma_{2R_\nu+s_0}^\infty)< \e_0\,.
  \end{equation}
  Now assume that $\e_0=\e_0(\mu)$ is the constant given in
  Lemma~\ref{lmm:Edecay} for some $\mu >0$ with $2\delta <\mu<\iota$
  where $\iota$ as given in~\eqref{eq:fixdelta}.  Hence there exists
  an uniform constant $c_1$ which is independent of $\nu$ such that
  for all $s \geq s_0+1$, $\nu\geq \nu_0$ and $t \in [0,1]$ we have
  the decay estimates
  \begin{align*}
    &\forall\ \nm{\sigma} \leq 2R_\nu -s:& \di{w_\nu(-2R_\nu+s,t),w_\nu(\sigma,t)} &\leq c_1 e^{-\mu s}\\
    &\forall\ \sigma \leq -2R_\nu -s:& \di{w_\nu(-2R_\nu-s,t),w_\nu(\sigma,t)} &\leq c_1 e^{-\mu s}\\
    &\forall\ \sigma \geq 2R_\nu+s:& \di{w_\nu(2R_\nu
      +s,t),w_\nu(\sigma,t)} &\leq c_1 e^{-\mu s}\,.
  \end{align*}
  These are all proven using Lemma~\ref{lmm:Edecay}. For the first
  inequality we have used~\eqref{eq:ddecay} of Lemma~\ref{lmm:Edecay}
  with $b=-a=2R_\nu-s_0$, $\sigma'=-2R_\nu+s$, $\sigma =\sigma$ and we
  have replaced the $s$ in that estimate by $s-s_0$. 
  
  We now use these estimates to prove the claim. Abbreviate
  $p:=u_0(\infty)=u_1(-\infty)$ and $u_\nu:=u_{R_\nu}$. We estimate
  for all $\nm{\sigma}\leq 2R_\nu-s$
  \begin{multline*}
    \di{v_\nu(\sigma,t),u_\nu(\sigma,t)} \leq\\
    \leq \di{v_\nu(\sigma,t),v_\nu(-2R_\nu+s,t)} +
    \di{v_\nu(-2R_\nu+s,t),u_0(s,t)}+ \\ +  \di{u_0(s,t),p} + \di{p,u_{\nu}(\sigma,t)}  \leq O(e^{-\mu s}) +o(1)\,,    
  \end{multline*}
  in which we have used the decay for $u_\nu=u_{R_\nu}$ as given in
  Lemma~\ref{lmm:uRdecay} and the fact that $e^{-\mu(2R_\nu
    -\nm{\sigma})} \leq e^{-\mu s}$.  Abbreviate
  $p_+:=u_1(\infty)$ and estimate for all $\sigma \geq 2R_\nu +s$ the distance of $v_\nu(\sigma,t)$ to $u_\nu(\sigma,t)$ by
  \begin{multline*}
    \di{v_\nu(\sigma,t),v_\nu(2R_\nu+s,t)} +\di{v_\nu(2R_\nu+s,t),u_1(s,t)} +\\+ \di{u_1(s,t),p_+} +\di{p_+,u_\nu(\sigma,t)}
  \end{multline*}
  We conclude that all terms are bounded by $O(e^{-\mu s}) +o(1)$. Now
  abbreviate $p_- :=u_0(-\infty)$ and estimate for all $\sigma \leq
  -2R_\nu-s$ the distance from $v_\nu(\sigma,t)$ to $u_\nu(\sigma,t)$
  by
  \begin{multline*}
    \di{v_\nu(\sigma,t),v_\nu(-2R_\nu-s,t)}
    +\di{v_\nu(-2R_\nu-s,t),u_0(-s,t)} +\\ +\di{u_0(-s,t),p_-} +
    \di{p_-,u_\nu(\sigma,t)} \,.
  \end{multline*}
  We see again that all terms are bounded by $O(e^{-\mu s}) +o(1)$.
  Combining the above with $C^\infty_\loc$ convergence we have for any
  $\sigma\in \R$
  \begin{multline*}
    \sup_\Sigma \nolimits \di{v_\nu,u_\nu} \leq \sup_{\Sigma_{-s}^s}
    \nolimits \di{v_\nu\circ \tau_{2R_\nu},u_0} +\\+
    \sup_{\Sigma_{-s}^s}\nolimits \di{v_\nu\circ \tau_{-2R_\nu},u_1} +
    O(e^{-\mu s}) + o(1) \leq O(e^{-\mu s}) +o(1)\,.
  \end{multline*}
  Now the right-hand side converges to $O(e^{-\mu s})$ and since $s$
  was chosen freely we conclude that the left-hand side converges to
  zero. This shows the claim.
  \begin{stp}
    We have $v_\nu =\exp_{u_\nu}\xi_\nu$ for some vector field
    $\xi_\nu$ along $u_\nu$ and the norm $\Nm{\xi_\nu}_{1,p;\delta}$ converges to zero as $\nu\to \infty$.
  \end{stp}
  Because of the last step the vector field $\xi_\nu$ is well-defined
  and the norm $\Nm{\xi_\nu}_{L^\infty}$ converges to zero.  We claim
  that there exists an uniform constant $c_2$ such that for all $s\geq
  s_0+1$, $t \in [0,1]$ and $\nu\geq \nu_0$ the following estimate
  holds omitting the arguments $\sigma$ and $t$ whenever convenient
  \begin{equation}\label{eq:xinudistant}
    \begin{aligned}
      & \nm{\sigma}\leq 2R_\nu-s: & \nmm{\xi_\nu - \widehat
        \Pi_{p}^{u_\nu}\xi_\nu(0)}
      +\nmm{\na\big(\xi_\nu - \widehat \Pi_{p}^{u_\nu} \xi_\nu(0)\big)} &\leq c_2 e^{-\mu (2R_\nu -\nm{\sigma})}\;.\\
      & \sigma \geq 2R_\nu+s:&\nmm{\xi_\nu-\widehat
        \Pi_{p_+}^{u_\nu} \xi_\nu(\infty)} + \nmm{\na\big(\xi_\nu
        -\widehat \Pi_{p_+}^{u_\nu}\xi_\nu(\infty)\big)} &\leq
      c_2e^{-\mu (\sigma -2R_\nu)}\\
      & \sigma \leq
      -2R_\nu-s:&\nmm{\xi_\nu-\widehat \Pi_{p_-}^{u_\nu}
        \xi_\nu(-\infty)} +\nmm{\na\big(\xi_\nu-\widehat
        \Pi_{p_-}^{u_\nu}\xi_\nu(-\infty)\big)} &\leq c_2e^{-\mu (\nm{\sigma}-2R_\nu)}
    \end{aligned}
  \end{equation}
  By analogy we will only deduce the first estimate. First of all we
  assume without loss of generality after possibly increasing $s$ that
  the distance from $u_\nu(\sigma,t)$ to $p$ is small enough to
  replace $\wh \Pi$ with $\Pi$ in the formula. Since the exponential
  function is uniformly Lipschitz (see Corollary~\ref{cor:explip}) and
  the distance of parallel geodesics is uniformly bounded by the
  distance of their starting points (see
  Corollary~\ref{cor:estparallelgeodesics}) we estimate for all
  $\nm{\sigma} \leq 2 R_\nu-s$
  \begin{multline*}
  \nm{\xi_\nu - \Pi_p^{u_\nu} \xi_\nu(0)} \leq O(\di{v_\nu,v_\nu(0)} +
  \di{\exp_p\xi_\nu(0),\exp_{u_\nu} \Pi_p^{u_\nu}\xi_\nu(0)})\\ \leq
  O(\di{v_\nu,v_\nu(0)}+\di{u_\nu,p}) \leq e^{-\mu (2R_\nu-\nm{\sigma})}\,.
  \end{multline*}
  For the last inequality we have used the decay of $u_\nu=u_{R_\nu}$
  (\cf Lemma~\ref{lmm:uRdecay}) and of $v_\nu$ as given by
  Lemma~\ref{lmm:Edecay} which is applicable because the energy of
  $v_\nu$ restricted to $\Sigma_{-2R_\nu+s}^{2R_\nu-s}$ is small (\cf
  equation~\eqref{eq:wnuEsmall}). We deduce the estimate for the
  covariant derivative using Corollary~\ref{cor:commutenaPi} to
  commute $\na$ with $\Pi$ and Corollary~\ref{cor:dwxi} to control
  the covariant derivative of $\xi_\nu$ by the differential of $u_\nu$
  and $v_\nu$.
  \[
  \nm{\na(\xi_\nu -\Pi_p^{u_\nu} \xi_\nu(0))}\leq \nm{\na \xi_\nu} +
  \nm{\na \Pi_p^{u_\nu}\xi_\nu(0)} \leq O(\nm{\d u_\nu} + \nm{\d
    v_\nu}) \leq O(e^{-\mu(2R_\nu-\nm{\sigma})})\,.\] Using the
  point-wise estimate~\eqref{eq:xinudistant} we estimate the norm on
  the neck
  \begin{align*}
    &\int_{\Sigma_{-2R_\nu}^{2R_\nu}} \left(\nmm{\xi_\nu -\wh
        \Pi_p^{u_\nu} \xi_\nu(0)}^p + \nmm{\na (\xi_\nu - \wh
        \Pi_p^{u_\nu} \xi_\nu(0))}^p
    \right)e^{p\delta(2R_\nu-\nm{\sigma})} \d \sigma \d t \leq \\
    &\leq O(1) \int_0^{2R_\nu-s} e^{-p(\mu-\delta)(2R_\nu-\sigma)} \d
    \sigma+\\
    &\hspace{2cm} +
    O\big(\Nm{\xi_\nu}_{\Sigma_{-2R_\nu}^{-2R_\nu+s}}^p
    +\Nm{\xi_\nu}_{\Sigma_{2R_\nu-s}^{2R_\nu}}^p +
    \nm{\xi_\nu(0)}^p\big)
    \int_{2R_\nu -s}^{2R_\nu} e^{p\delta(2R_\nu -\sigma)} \d \sigma\\
    &\leq O(e^{-\mu s}) + o(1)\,,
  \end{align*}
  on the positive end
  \begin{multline*}
    \int_{\Sigma_{2R_\nu}^\infty} \left(\nmm{\xi_\nu -\wh
        \Pi_{p_+}^{u_\nu} \xi_\nu(\infty)}^p + \nmm{\na (\xi_\nu - \wh
        \Pi_p^{u_\nu} \xi_\nu(\infty))}^p
    \right)e^{p\delta(\sigma -2R_\nu)} \d \sigma \d t  \leq\\
    +O(1)\int_s^\infty e^{-p(\mu-\delta)\sigma} \d \sigma + o(1)
    \int_0^s e^{p\delta s} \d s \leq O(e^{-\mu s}) + o(1)\,,
  \end{multline*}
  and similarly for the negative end.  The last estimates amount to
  $\Nm{\xi_\nu}_{1,p;\gamma_{\delta,R_\nu}} \leq O(e^{-\mu s}) +o(1)$
  and in particular show that $\lim_{\nu \to \infty}
  \Nm{\xi_\nu}_{1,p;\gamma_{\delta,R_\nu}} \leq O(e^{-\mu s})$. Now
  since $s$ was chosen freely we see that the limit must vanish.
\end{proof}
\begin{lmm}\label{lmm:Gsurj}
  With the same assumptions as Lemma~\ref{lmm:auffang}, then $w_\nu$
  lies in the image of the gluing map for all but
  finitely many $\nu$.
\end{lmm}
\begin{proof}
  The case~\ref{nm:none} directly follows from Lemma~\ref{lmm:auffang}
  because in that case $\ker D_R$ is zero dimensional and by the
  uniqueness property the solution map we have $w_\nu =
  \exp_{u_{R_\nu}}\sol_{R_\nu}(0)=\glue(R_\nu)$ for all but finitely
  many $\nu$. For case~\ref{nm:one} set $\ee:=+1$ if $(J_0,X_0)$ is
  flow-dependent and $e:=-1$ if $(J_1,X_1)$ is flow-dependent. We
  define for some $\e>0$ small and $R_0$ large enough,
  \[\M_{\e,R_0} := \{ w = (\exp_{u_R} \xi) \circ \tau_{2\ee R}   \mid   R \geq R_0,\
  \Nm{\xi}_{1,p;\delta,R} < \e\} \cap \M(W_-,W_+)\,.\] In
  Lemma~\ref{lmm:auffang} we show that $w_\nu\in \M_{\e,R_0}$ for all
  but finitely many $\nu \in \N$. We claim that $\M_{\e,R_0}$ is
  path-connected. Indeed given two elements $w,w' \in \M_{\e,R_0}$, by
  the third property of the solution map (\cf Lemma~\ref{lmm:sol})
  there exists constants $R,R'\geq R_0$ and $\zeta\in \ker D_R$,
  $\zeta' \in \ker D_{R'}$ such that
  \[
  w=\exp_{u_R} (\zeta +\sol_R (\zeta)) \circ \tau_{2\ee R},\qquad w'
  =\exp_{u_{R'}} (\zeta' + \sol_{R'} (\zeta')) \circ \tau_{2\ee R}\,.
  \]
  We connect $w$ to $w_R=\glue(R)$ via $[0,1] \ni \theta \mapsto
  \exp_{u_R}(\theta \zeta + \sol_R(\theta\zeta)) \circ
  \tau_{2\ee R}$. Similar we connect $w'$ to $w_{R'}$. Assuming without loss
  of generality that $R<R'$, we connect $w_R$ to $w_{R'}$ via $[R,R']
  \ni r \mapsto \glue(r)$.  We identify the connected one-dimensional
  space $\M_{\e,R_0}$ with an half-infinite interval $[0,\infty)$
  such that under the identification the strips $w_\nu$ converges to
  $\infty$. After the identification the gluing map is an unbounded
  continuous map $[R_0,\infty) \to [0,\infty)$ and its image by the
  intermediate value theorem contains all but finitely many elements
  of $w_\nu$. 

  In case~\ref{nm:both} we define similarly $\wt
  \M_{\e,R_0}:=\{ w = \exp_{u_R} \xi \mid R_0 \leq R,\
  \Nm{\xi}_{1,p;\delta} < \e\} \cap \wt \M(W_-,W_+)$ and by
  $\M_{\e,R_0} \subset \M(W_-,W_+)$ the image under the quotient
  map. Again the space is a connected one-dimensional manifold
  containing by Lemma~\ref{lmm:auffang} all but finally many strips
  $w_\nu$. We argue as in case~\ref{nm:one}.
\end{proof}

\subsection{Orientation of the gluing map}
Fix a relative spin structure for $(L_0,L_1)$ and denote by $\O$ the
associated double cover (\cf Definitions~\ref{dfn:relspinMLL}
and~\ref{dfn:O}). Let $C_-$, $C_+\subset L_0 \cap L_1$ denote the
connected components of $W_-$, $W_+$ respectively.
\begin{lmm}\label{lmm:oriMWWM1WW}
  If $W_- \subset C_-$ is equipped with an $\O^\vee$-orientation and
  $W_+ \subset C_+$ is equipped with an $\O$-orientation then the spaces
  $\M(W_-,W_+)$ and $\M^1(W_-,W_+)$ have an induced orientation.
\end{lmm}
\begin{proof}
  Abbreviate $\wt \M(C_-,C_+):=\wt \M(C_-,C_+;J,X)$ equipped with
  obvious evaluations into $C_-$ and $C_+$.  By Theorem~\ref{thm:ori}
  and Lemma~\ref{lmm:orifibre} the fibre product
  \[
  W_- \times_{C_-} \wt \M(C_-,C_+) \times_{C_+} W_+\,,
  \]
  carries a canonical orientation and its quotient space $\M(W_-,W_+)$
  carries the induced orientation by~\eqref{eq:oriquotient}.
  Abbreviate the spaces
\[\wt \M(C_-,C):=\wt \M(C_-,C;J_0,X_0)\quad \text{ and } \quad \wt
\M(C,W_+):=\wt M(C,C_+;J_1,X_1)\,,\] equipped with obvious evaluations
into $C_-$, $C$ and $C_+$. As above the fibre product
  \[
  W_- \times_{C_-} \wt \M(C_-,C) \times_C \wt \M(C,C_+)\times_{C_+}
  W_+\,,
  \]
  carries an induced orientation and hence its quotient
  $\M^1(W_-,W_+)$ too.
\end{proof}
\begin{lmm}\label{lmm:PRTR}
  Linear gluing $P_R \Theta_R:\ker D_{01} \to \ker D_R$ is orientation
  preserving.
\end{lmm}
\begin{proof}
  Consider the intersection points $p_- = u_0(-\infty)$,
  $p=u_0(\infty)=u_1(-\infty)$ and $p_+=u_1(\infty)$ with caps $D_-$,
  $D$ and $D_+$ respectively.  By associativity of linear gluing we
  have a commutative diagram.
  \[
  \xymatrix{|D_-|\otimes |C|\otimes |D_{01}| \ar[d]^{\mathrm{id}\otimes |P_R\Theta_R|} \ar[r]&|D_-| \otimes |D_0| \otimes |D_1|\ar[r]&|C_-| \otimes|D|\otimes  |D_1|\ar[d]\\
    |D_-|\otimes |C|\otimes |D_R|\ar[rr]&&|C_-| \otimes
    |C|\otimes |D_+| }
  \]
  By definition of gluing the orientation on $D_{01}$ is induced by
  following the diagram from the down-right to the top-left corner and
  the orientation of $D_R$ is by definition given by going from the
  down-right to the down-left corner.
\end{proof}
\begin{cor}\label{cor:PRTR}
  The restriction $P_R\Theta_R|_{\ker D_{01}'}:\ker D_{01}' \to \ker
  D_R'$ is orientation preserving. 
\end{cor}
\begin{proof}
  By definition the orientation of $D_{01}'$ and $D'_R$ are given
  by~\eqref{eq:oddeven} on the exact sequences
  \[
  \xymatrix{ 0\ar[r]&\ker D_{01}'\ar[r]\ar[d]^{P_R\Theta_R|_{\ker D_{01}'}}& \ker
    D_{01}\oplus W_- \oplus W_+\ar[r]\ar[d]^{ P_R\Theta_R \oplus \mathrm{id}}&C_- \oplus C_+ \ar[r]\ar[d]^{\mathrm{id}}& 0\\
    0\ar[r]& \ker D_R'\ar[r]&D_R \oplus W_- \oplus W_+\ar[r]&C_-\oplus
    C_+\ar[r]&0}\] The claim follows by naturality
  of~\eqref{eq:oddeven}.
\end{proof}
\begin{prp}\label{prp:degglue}
  The space $\ol \M(W_-,W_+)_{[1]}$ has the structure of a one
  dimensional manifold with boundary. With orientations from
  Lemma~\ref{lmm:oriMWWM1WW} the oriented boundary is 
  \begin{itemize}
  \item $\M^1(W_-,W_+)_{[0]}$ if $(J_0,X_0)$ is $\R$-invariant and
  \item $(-1)\cdot \M^1(W_-,W_+)_{[0]}$ if $(J_1,X_1)$ is
    $\R$-invariant.
  \end{itemize}
\end{prp}
\begin{proof}\setcounter{stp}{0}
  The space $\ol \M(W_-,W_+)_{[1]}$ is a manifold with boundary using
  the gluing map as chart map for a boundary point. It remains to show
  the statement about the degree. We treat each case separately.
  \begin{stp}
    We prove the proposition in case~\ref{nm:none}.
  \end{stp}
  The tangent space at some $(R,u) \in \M(W_-,W_+)$ is given as the
  kernel of the operator $\wh D_u: \R \oplus T_u\B'$, $(\theta,\xi)
  \mapsto D_u\xi + \theta \eta_R$ with
  \[
  \eta_R= (\partial_R J_R(u))(\pt u - X_R(u)) - J_R(u)(\partial_R X_R(u))\,.
  \]
  Here $\B'\subset \B(C_-,C_+)$ is the subspace of all $u$ with
  $u(-\infty) \in W_-$ and $u(\infty) \in W_+$. We assume without loss
  of generality that $D_u$ is surjective and hence an isomorphism when
  restricted to $T_u\B'$. We conclude that there exists a unique
  $\xi_R \in T_u \B'$ such that $D_u \xi_R = \eta_R$. The vector
  $(1,-\xi_R) \in \ker \wh D_u$ is pointing outward has has the same
  orientation as the sign of the isomorphism $D_u$, which by parallel
  transport is the same as the sign of the isomorphism $D_R'$ as
  considered above in Corollary~\ref{cor:PRTR}. We conclude that the
  sign of $D_R'$ is the same as the sign of $D_{01}'$, which by
  definition is the sign of $(u_0,u_1)$. This shows the claim.
  \begin{stp}
    We prove the proposition in case~\ref{nm:one}.
  \end{stp}
  The orientation on $\M(W_-,W_+)_{[1]}$ induces a total order on each
  connected component. We have to distinguish the two sub case when
  $(J_0,X_0)$ is $\R$-invariant or $(J_1,X_1)$ is $\R$-invariant. By
  analogy we only treat the case where $(J_0,X_0)$ is $\R$-invariant.
  Fix a point $([u_0],u_1) \in \M^1(W_-,W_+)_{[0]}$. By surjectivity
  of gluing there exists one connected component of
  $\M(W_-,W_+)_{[1]}$ containing a sequence which converges to
  $([u_0],u_1)$. Let $(w_\nu)_{\nu\in \N}$ be such a sequence which is
  monotone with respect to the total order.  By surjectivity of gluing
  we have for all $\nu$ large enough
  \[w_\nu \circ \tau_{2R_\nu} = \exp_{u_\nu} \xi_\nu,\qquad
  \xi_\nu:=\sol_\nu(0) \in \im Q_{R_\nu}\,.\] Choosing the sequence fine
  enough we assume without loss of generality there exists $\zeta_\nu
  \in \ker D'_{R_\nu}$ such that
  \begin{equation}
    \label{eq:wnu+1wnu}
    w_{\nu+1} \circ \tau_{2R_{\nu}} =
    \exp_{u_\nu}(\zeta_\nu + \eta_\nu)\qquad \eta_\nu:=\sol_\nu(\zeta_\nu) \in \im Q_{R_\nu}\,.
  \end{equation}
  Let $o_\partial$ the orientation of $([u_0],u_1)$ as a boundary
  point of $\M(W_-,W_+)_{[1]}$ and $o$ be the orientation as an
  element of the oriented space $\M^1(W_-,W_+)_{[0]}$. We claim that
  using the orientations of $D'_\nu$ given by
  Lemma~\ref{lmm:oriMWWM1WW} we have
  \begin{equation}
    \label{eq:op=o}
    o_\partial =+1 \iff w_\nu < w_{\nu+1} \iff \zeta_\nu \text{ is pos.} \iff o=+1\,.    
  \end{equation}
  This clearly shows the assertion and we are left to deduce all
  equivalences. The first equivalence is a definition. The second
  equivalence is also a definition, since we have by the properties of
  the solution map (\cf Lemma~\ref{lmm:sol}) an orientation preserving
  path from $w_\nu$ to $w_{\nu+1}$ via $[0,1]\ni \theta \mapsto
  \exp_{u_\nu}(\theta \zeta_\nu + \sol(\theta \zeta_\nu)) \circ
  \tau_{-2R_\nu}$. We show the third equivalence.  For $R >0$ we
  define $u^\tau_R:=u_R \circ \tau_{-2R}$, $\xi^\tau_\nu := \xi_\nu
  \circ \tau_{-2R_\nu}$, $\eta^\tau_\nu:=\eta_\nu\circ \tau_{-2R_\nu}$
  and for $\theta\in [0,1]$
  \[u_{\nu,\theta}^\tau:=u_{R_\nu+\theta(R_{\nu+1}-R_\nu)}^\tau,\qquad
  \xi_{\nu,\theta}^\tau := \Pi_\nu(\theta) \xi^\tau_{\nu+1},\qquad
  \chi_{\nu,\theta}^\tau := \exp_{u^\tau_\nu}^{-1}
  \exp_{u^\tau_{\nu,\theta}} \xi_{\nu,\theta}^\tau\,,
  \]
  in which $\Pi_\nu(\theta)$ denotes the parallel transport from
  $u^\tau_{\nu+1}$ to $u^\tau_{\nu,\theta}$ along the path $[\theta,1]\ni \tau
  \mapsto u^\tau_{\nu,\tau}$. By construction and~\eqref{eq:wnu+1wnu} we have
  \[
  u_{\nu,1}^\tau=u^\tau_{\nu+1},\qquad
  u^\tau_{\nu,0}=u^\tau_\nu,\qquad
  \zeta^\tau_\nu+\eta^\tau_\nu=\chi^\tau_{\nu,1},\qquad
  \chi^\tau_{\nu,0} = \Pi_\nu(0)\xi^\tau_{\nu+1}\,.
  \]
  The path $\theta \mapsto \chi_{\nu,\theta}^\tau$ is differentiable
  and by the mean value theorem we conclude that there exists
  $\theta_\nu \in [0,1]$ such that
  \[
  \zeta^\tau_\nu + \eta^\tau_\nu = \chi^\tau_{\nu,1} =
  \Pi_\nu(0)\xi^\tau_{\nu+1} + \rho_\nu\cdot (\partial_\theta
  \chi^\tau_{\nu,\theta})|_{\theta=\theta_\nu},\qquad
  \rho_\nu:=R_{\nu+1}-R_\nu\,.
  \]
  We apply $\tau_{-2\ee R_\nu}$ and subtract
  $\kappa_\nu:=\partial_\theta u^\tau_{\nu,\theta}|_{\theta=0} \circ
  \tau_{-2R_\nu}$ on both sides
  \[
  \zeta_\nu - \kappa_\nu = \Pi_\nu(0) \xi_\nu -\eta_\nu +
  \rho_\nu\cdot (\partial_\theta \chi^\tau_{\nu,\theta}|_{\theta
    =\theta_\nu} - \partial_\theta u^\tau_{\nu,\theta}|_{\theta=0} )
  \circ \tau_{-2R_\nu}\,.
  \]
  By the property of the solution map the correction terms $\xi_\nu$
  and $\eta_\nu$ converge to zero uniformly (\cf
  Lemma~\ref{lmm:sol}). Moreover by Corollary~\ref{cor:dwxi} the last
  term is uniformly bounded by $O(\rho_\nu^2)$.  Thus we have the
  pointwise estimate
  \[
  |\zeta_\nu-\kappa_\nu| \leq o(1) + O(\rho_\nu^2)\,.
  \]
  Since $u_0$ is holomorphic and non-constant the preglued strip $u_R$
  has an end which is non-constant and holomorphic. By the asymptotic
  behavior (\cf Theorem~\ref{thm:eigenval}) we conclude that
  \[\rho_\nu \leq O(1)
  \di{u_\nu,u_{\nu+1}} \leq O(1) \di{u_\nu,w_{\nu+1}} + o(1) \leq O(1)
  |\zeta_\nu| + o(1)\,.\] Fix constants $s_0<s_1$ and define the
  strips $\Sigma_{\nu}:=[s_0-2 R_\nu,s_1-2 R_\nu]\times
  [0,1]$. We have with uniform constants
  \[
  \Nm{Q_\nu D_\nu \kappa_\nu}_{C^0(\Sigma_{\nu})} \leq O(1) \Nm{Q_\nu
    D_\nu \kappa_\nu}_{1,p;\delta,R} \leq O(1) \Nm{D_\nu
    \kappa_\nu}_{p;\delta} \leq o(1)\,.
  \]
  Combining the last estimates we have for all $(s,t) \in
  \Sigma_{\nu}$
  \[
  |\zeta_\nu- P_\nu \kappa_\nu| \leq |\zeta_\nu -\kappa_\nu| + |Q_\nu
  D_\nu \kappa_\nu| \leq O(1) |\zeta_\nu| + o(1)\,.
  \]
  Let $\e>0$ be some sufficiently small number and assume that the
  constants $s_0<s_1$ are chosen such that $|\ps u_0(s,t)|>\e$ for all
  $(s,t)\in [s_0,s_1]\times [0,1]$. Hence by construction we have for
  all $(s,t) \in \Sigma_\nu$ and all $\nu$ sufficiently large
  \[
  |P_\nu\kappa_\nu| = |\kappa_\nu| + o(1)  \geq \e/2\,.
  \]
  Now define $\alpha_\nu \in \R$ via $\alpha_\nu P_\nu\kappa_\nu =
  \zeta_\nu$ uniquely since $\ker D_{R_\nu}$ is one dimensional. We
  have with the above estimates
  \[
  |1-\alpha_\nu| = |(1-\alpha_\nu)P_\nu\kappa_\nu|/ |P_\nu\kappa_\nu|
  \leq 2/\e \cdot |\zeta_\nu-P_\nu\kappa_\nu| \leq O(1)|\zeta_\nu| +
  o(1)\,.
  \]
  If $\nu$ is large enough and $|\zeta_\nu|$ small enough (by choosing
  the sequence fine enough), we have $|1-\alpha_\nu|<1/2$ in
  particular $\alpha_\nu$ is positive. We conclude that $\zeta_\nu$
  has the same orientation as $P_\nu\kappa_\nu$. A direct computation
  shows that $\kappa_\nu = 4R_\nu \Theta_\nu \ps u_0$ for all $(s,t)
  \in \Sigma_\nu$ which togeher with Lemma~\ref{lmm:PRTR} shows that
  $\zeta_\nu$ has the same orientation as $\ps u_0$ as claimed. 
  \begin{stp}
    We prove the proposition in case~\ref{nm:both}.
  \end{stp}
  Similarly to the last step let $([w_\nu]) \subset \M(W_-,W_+)_{[0]}$
  be a strictly monotone sequence converging to $([u_0],[u_1])$. We
  assume after a possible reparametrization and by surjectivity of
  gluing that $w_\nu = \exp_{u_\nu} \xi_\nu$ for $u_\nu=u_{R_\nu}$ and
  $\xi_\nu \in \im Q_\nu$. Choosing the sequence fine enough we have
  for all $\nu \in \N$ large enough
  \[
  w_{\nu+1} = \exp_{u_\nu} (\zeta_\nu+\eta_\nu),\qquad \zeta_\nu \in
  \ker D_\nu,\quad \eta_\nu \in \im Q_\nu\,.
  \]
  Define the path from $w_\nu$ to $w_{\nu+1}$ via $[0,1] \ni \theta
  \mapsto w_{\nu,\theta}:=\exp_{u_\nu} (\theta \zeta_\nu+
  \sol_\nu(\theta\zeta_\nu))$.  Let $o_{\partial} \in \{\pm 1\}$
  denote the orientation of $([u_0],[u_1])$ as a boundary point of
  $\M(W_-,W_+)$ and $o \in \{\pm 1\}$ as a point of the oriented space
  $\M^1(W_-,W_+)$. We claim that we have the following equivalences
  with orientations on $D_{w_{\nu,\theta}}'$, $D_{R_\nu}'$ and
  $D_{01}'$ given by Lemma~\ref{lmm:oriMWWM1WW},
  \begin{equation}
    \label{eq:claims}
    \begin{aligned}
      o_\partial = +1&\iff    ([w_\nu]) \text{ is incr.}\\
      &\iff \ps w_{\nu,\theta} \wedge \partial_\theta w_{\nu,\theta} \in \det D'_{w_{\nu,\theta}} \text{ is pos.}\\
      &\iff P_\nu \ps u_\nu \wedge P_\nu\kappa_\nu \in \det D'_{R_\nu}
      \text{ is pos.}  \qquad
      \kappa_\nu:=(\partial_R u_R)|_{R = R_\nu}\\
      &\iff \ps u_1 \wedge \ps u_0 \in \det D'_{01} \text{ is pos.}\\
      &\iff o = -1\,.
    \end{aligned}
  \end{equation}
  This clearly implies the assertion of the proposition. We
  show~\eqref{eq:claims}. The first and the last equivalence is a
  definition. The second equivalence also clear by definition of the
  quotient orientation (\cf equation~\eqref{eq:oriquotient}).  We
  claim that the third equivalence of~\eqref{eq:claims} follows if we
  find a norm $\Nm{\,\cdot\,}_\nu$ on $\det D_{R_\nu}$ such that there
  exists an uniform constant $\delta>0$ with
  \begin{equation}
    \label{eq:Pbelow}
    \Nm{P_\nu \ps u_\nu \wedge P_\nu
      \kappa_\nu }_\nu>\delta   
  \end{equation}
  and if $\Pi_{\nu,\theta}$ denotes the parallel transport from
  $w_{\nu,\theta}$ to $u_\nu$ we have
  \begin{equation}
    \label{eq:Pabove}
    \Nm{P_\nu\ps u_\nu \wedge P_\nu \kappa_\nu - P_\nu\Pi_{\nu,\theta}
      \ps w_{\nu,\theta} \wedge P_\nu\Pi_{\nu,\theta} \partial_\theta
      w_{\nu,\theta}}_\nu \leq O(\Nm{\zeta_\nu}_{C^1}) + o(1).
  \end{equation}
  Indeed, suppose that~\eqref{eq:Pbelow} and~\eqref{eq:Pabove} is
  true. Then define $\alpha_\nu \in \R$ by
  \[
  \alpha_\nu P_\nu\ps u_\nu \wedge P_\nu \kappa_\nu =
  P_\nu\Pi_{\nu,\theta} \ps w_{\nu,\theta}\wedge
  P_\nu\Pi_{\nu,\theta} \partial_\theta w_{\nu,\theta}\,.
  \]
  Using~\eqref{eq:Pbelow} and~\eqref{eq:Pabove} we find an uniform
  constant $c>0$ such that for all $\nu$ sufficiently large
  \begin{equation*}
    |1-\alpha_\nu| = \frac{\Nm{(1-\alpha_\nu)P_\nu\ps u_\nu\wedge P_\nu
        \kappa_\nu}_\nu}{\Nm{P_\nu \ps u_\nu \wedge P_\nu \kappa_\nu}_\nu}\leq \frac  {c\Nm{\zeta_\nu}_{C^1} + \delta/3}{\delta}\,.  
  \end{equation*}
  Choosing the sequence $([w_\nu])$ fine enough we assume without loss
  of generality that $\Nm{\zeta_\nu}_{C^1}<\delta/3c$.  We conclude
  that $|1-\alpha_\nu|\leq 2/3$ and hence $\alpha_\nu$ is
  positive. Thus the third equivalence of~\eqref{eq:claims} follows
  (using the fact that the operator $P_\nu\Pi_{\nu,\theta}$ is
  orientation preserving).

  It remains to find a norm $\Nm{\,\cdot\,}_\nu$ such
  that~\eqref{eq:Pbelow} and~\eqref{eq:Pabove} holds.  Fix constants
  $s_0 <s_1$ and consider the Hilbert space $H_\nu:=L^2(\Sigma_\nu)$ on
  the domain $\Sigma_\nu= \Sigma_{\nu,-}\cup \Sigma_{\nu,+}$ with
  $\Sigma_-=[s_0-2R_\nu,s_1-2R_\nu]\times [0,1]$ and
  $\Sigma_{\nu,+}=[s_0+2R_\nu,s_1+2R_\nu]\times [0,1]$. Abbreviate
  $\Nm{\,\cdot\,}_\nu = \Nm{\,\cdot\,}_{L^2(\Sigma_\nu)}$ and
  $\<\cdot,\cdot\>_\nu = \<\cdot,\cdot\>_{L^2(\Sigma_\nu)}$ the
  standard norm and the scalar product on $H_\nu$. We consider the
  norm on $\Lambda^2H_\nu$ given by
  \[\Nm{\xi \wedge \xi'}_\nu :=
  (\Nm{\xi}_\nu^2\Nm{\xi'}_\nu^2 - \<\xi,\xi'\>_\nu^2)^{1/2}\,.\]
  Using the Cauchy-Schwarz inequality we have $\Nm{\xi\wedge \xi'}
  \leq 2 \Nm{\xi}\Nm{\xi'}$ for all $\xi,\xi' \in H_\nu$. With
  Corollary~\ref{cor:dwxi} we obtain
  \[
  \Nm{\Pi_{\nu,\theta}\ps w_{\nu,\theta} -P_\nu \ps u_\nu}_\nu \leq
  o(1) + O(\Nm{\zeta_\nu}_{C^1}).
  \]
  In the last step we show that
  \[
  \Nm{\Pi_{\nu,\theta}\partial_\theta w_{\nu,\theta} -P_\nu
    \kappa_\nu}_\nu \leq o(1) + O(\Nm{\zeta_\nu}_{C^0})\,.
  \]
  This shows~\eqref{eq:Pabove}.

  We show~\eqref{eq:Pbelow}. Choose a small constant $\e>0$ and assume
  that $s_0<s_1$ are such that $|\ps u_0(s,t)| >\e$ and $|\ps
  u_1(s,t)|>\e$ for all $(s,t) \in [s_0,s_1]\times [0,1]$.  Abbreviate
  the norm $\Nm{\,\cdot\,} = = \Nm{\,\cdot\,}_{L^2([s_0,s_1]\times
    [0,1])}$. By construction $\Nm{\ps u_\nu}_\nu^2 =
  \Nm{\kappa_\nu}^2_\nu = \Nm{\ps u_0}^2 + \Nm{\ps u_1}^2$ and $\<\ps
  u_\nu,\kappa_\nu\>= \Nm{\ps u_0}^2 - \Nm{\ps u_1}^2$. We compute
  \begin{align*}
    \Nm{P_\nu \kappa_\nu\wedge P_\nu \ps u_\nu}_\nu^2 &= \Nm{P_\nu\kappa_\nu}^2 \Nm{P_\nu\ps u_\nu}^2 -\<P_\nu\kappa_\nu,P_\nu\ps u_\nu\>^2\\
    & = \Nm{\kappa_\nu}^2\Nm{\ps u_\nu}^2 - \<\kappa_\nu,\ps u_\nu\>^2 + o(1)\\
    &= (\Nm{\ps u_0}^2 + \Nm{\ps u_1}^2)^2
    - (\Nm{\ps u_0}^2 - \Nm{\ps u_1}^2)^2 + o(1)\\
    &= 4 \Nm{\ps u_0}^2\Nm{\ps u_1}^2 + o(1) \geq 2 (s_1-s_0)^4\e^4
    =:\delta^2 >0 \,.
  \end{align*}
  This shows~\eqref{eq:Pbelow} and hence the third equivalence
  of~\eqref{eq:claims}.  

  Finally the fourth equivalence follows because by construction
  \[\Nm{P_\nu \Theta_{\nu}(\ps u_0,\ps u_1) -P_\nu\ps
    u_\nu}_{\nu} \to 0,\qquad\Nm{P_\nu\Theta_{\nu}(\ps u_0,-\ps
    u_1)-1/2 P_\nu\kappa_\nu}_{\nu} \to 0\,.\] We conclude with
  Corollary~\ref{cor:PRTR} that $\ps u_0 \wedge \ps u_1 \in
  \Lambda^2\ker D_{01}'$ has the same orientation as $P_\nu \kappa_\nu
  \wedge P_\nu \ps u_\nu \in \Lambda^2\ker D_R'$.
\end{proof}
\subsection{Morse gluing}\label{sec:Mglue}
In the section we describe a gluing result for Morse trajectories,
which is a little more general than the gluing result
in~\cite{AbbMajer:MorseI}. We have included a complete proof, since we
have not found it in the literature. The methods remain the same and
rely on known results about hyperbolic dynamics, in particular the
graph transform theorem.

Let $f:C \to \R$ be a Morse function and denote by $\psi: \R \times C
\to C$, $\psi^a=\psi(a,\cdot)$ the negative gradient flow of $f$ with
respect the some fixed Riemannian metric. Given sub-manifolds $W_-,
W_+ \subset C$ we define the space
\begin{equation}\label{eq:W}
     W_- \times_\psi W_+ := \{(R,w_-,w_+)  \mid \psi^R(w_-) = w_+\} \subset \R \times W_- \times W_+\,,
\end{equation}
which is the space of finite length flow-lines from $W_-$ to
$W_+$. Standard compactness shows that as a finite length flow-line
gets longer and longer it approaches a broken flow-line, which
generically is a pair
\begin{equation}\label{eq:Wbd}
  w^\infty=(w^\infty_-,w^\infty_+) \in (W_- \cap W^s(p)) \times (W^u(p) \cap W_+)\,,
\end{equation}
for some critical point $p \in \crit f$. The next lemma shows that this
process is reversible, \ie any a broken flow-line can be glued
together to obtain a family of finite length flow-lines.

For the orientation statement we assume that $W_-$ is oriented and
$W_+$ is cooriented. The space~\eqref{eq:W} is cut-out transversely,
if for all points we have the exact sequence
\begin{equation}
  \label{eq:orW}
\xymatrix{ 0 \ar[r]& T_{(R,w_-,w_+)} W_- \times_\psi W_+ \ar[r] & \R \oplus T_{w_-} W_- \ar[r]^-{\d \psi} & T_{w_+} C/ T_{w_+} W_+ \ar[r]&0\,.}  
\end{equation}
From the sequence we obtain via~\eqref{eq:oddeven} an orientation on
the space $W_- \times_\psi W_+$ provided with the fixed
orientations. If $\dim W_- + \dim W_+ = \dim C$ and~\eqref{eq:Wbd} is
cut-out transversely, the space~\eqref{eq:Wbd} is
zero-dimensional. Moreover if $W^u(p)$ is oriented then $W^s(u)$ is
cooriented and by~\eqref{eq:capori} we have signs
$\e_-:=\sign(w^\infty_-)$ and $\e_+:=\sign(w^\infty_+)$. The product
$\e:=\e_-\,\e_+$ does not depend on the choice of the orientation of
$W^u(p)$. If $W_-$ is everywhere transverse to the gradient of $f$ we
write $\grad f \pitchfork W_-$ and define the manifold $\wt W_- := \R
\times W_-$ embedded into $C$ via the flow $\psi$ and oriented by
\begin{equation}
  \label{eq:oriWt-}
  -\grad_w f\R \oplus T_w W_- \cong T_w \wt W_- \,.
\end{equation}
Similarly if $W_+$ is everywhere transverse to the gradient of $f$
we write $\grad f \pitchfork W_+$ and define the manifold $\wt W_+:=
\R \times W_+$ embedded into $C$ via the flow and cooriented by
\begin{equation}
  \label{eq:oriWt+}
 -\grad_w f\R \oplus T_wC_+/T_w \wt W_+ \cong T_w C_+/T_w W_+\,.
\end{equation}
We obtain orientations of $\wt W_- \cap W_+$ and $W_- \cap \wt W_+$
via~\eqref{eq:capori}.
\begin{lmm}[Morse gluing]\label{lmm:Mglue}
  Assume that~\eqref{eq:W} and~\eqref{eq:Wbd} are cut-out transversely
  and $\dim W_- + \dim W_+ = \dim C$. For any element
  $(w^\infty_-,w^\infty_+)$ of~\eqref{eq:Wbd} there exists an
  injective immersion
  \begin{equation}
    \label{eq:Mglue}
    [R_0,\infty) \to W_- \times_\psi W_+,\qquad R \mapsto (R,w_-^R,w_+^R)\,,   
  \end{equation}
  such that
  \begin{enumerate}[label=(\roman*)]
  \item\label{nm:wRwinf} $\lim_{R \to \infty} w_\pm^R=w^\infty_\pm$,
  \item\label{nm:Msurj} there exists $\delta>0$ such that for any
    $(R,w_-,w_+) \in W_- \times_\psi W_+$ with
    $\di{w_-,w^\infty_-}+\di{w_+,w^\infty_+}<\delta$ we have
    $(w_-,w_+)=(w^R_-,w^R_+)$,
  \item\label{nm:Mori} suppose that $W_-$ is oriented and $W_+$ is
    cooriented, then
    \begin{enumerate}[label=\alph*)]
    \item\label{nm:WpW} the orientation of the vector $(1,\pa_R
      w_-^R,\pa_R w_+^R) \in W_- \times_\psi W_+$ is $\e$,
    \item\label{nm:pW} if $\grad f \pitchfork W_-$ the orientation of the
      vector $\pa_R w_+^R \in \wt W_- \cap W_+$ is $\e$,
    \item\label{nm:Wp} if $\grad f \pitchfork W_+$ the orientation of
      the vector $\pa_R w_-^R \in W_- \cap \wt W_+$ is $-\e$.
    \end{enumerate}
    where $\e=\sign(w^\infty_-)\sign(w^\infty_+)$ and the spaces are
    oriented as described above.
  \end{enumerate}
  \end{lmm}
\begin{proof}\setcounter{stp}{0}
  Let $B_\rho(w) \subset C$ denote an open ball with radius $\rho>0$
  centered at $w\in C$. Without loss of generality we replace $W_-$
  with $W_- \cap B_\rho(w^\infty_-)$ and $W_+$ with $W_+ \cap
  B_\rho(w^\infty_+)$ for some sufficiently small $\rho>0$. How small
  $\rho$ needs to be is explained throughout the course of the proof.

  By assumption $w_-^\infty \in W_- \cap W^s(p)$ and $w_+^\infty \in
  W^u(p)\cap W_+$. We identify a neighborhood of the critical point
  $p$ in $C$ with a neighborhood of $0$ in a vector space $H$
  identifying $p$ with $0$. The splitting of the spectrum of the
  Hessian of $f$ at $p$ into negative and positive part induces a
  splitting of $H$ denoted $H^u \oplus H^s$. By $H^u(r)$ (\resp
  $H^s(r)$) we denote the closed $r$-ball centered at $0$ of the
  linear space $H^u$ (\resp $H^s$). We set $Q(r):=H^u(r)\times
  H^s(r)$, which after an identification of $H^u \times H^s$ with $H^u
  \oplus H^s$ is a subset of $H$ and also a neighborhood of $p$ in
  $C$.   
  \begin{stp}\label{stp:graphtransform} 
    We claim that for sufficiently small $\rho$, there exists positive
    constants $R_0$, $r_0$ and continuous paths $\sigma:[R_0,\infty)
    \to C^0(H^u(r_0),H^s(r_0))$, $R \mapsto \sigma_R$ and
    $\tau:[R_0,\infty) \to C^0(H^s(r_0),H^u(r_0))$, $R \mapsto
    \tau_R$ such that for all $R \geq R_0$ the functions $\sigma_R$
    and $\tau_R$ are Lipschitz with constant $\theta<1$ and moreover
    \begin{enumerate}[label=(\roman*)]
    \item $W_-^R:=\psi^R(W_-) \cap Q(r_0)$ is the graph of $\sigma_R$ for
      all $R \geq R_0$,
    \item $\sigma_R$ uniformly converges to $\sigma_\infty$ as $R \to
      \infty$ and $\graph \sigma_\infty =W^u(p) \cap Q(r_0)$,
    \item $W_+^{-R}:=\psi^{-R}(W_+) \cap Q(r_0)$ is the graph of $\tau_R$
      for all $R \geq R_0$,
    \item $\tau_R$ uniformly converges to $\tau_\infty$ as $R \to
      \infty$ and $\graph \tau_\infty = W^s(p) \cap Q(r_0)$.
    \end{enumerate}
  \end{stp}
  A proof is given in \cite[Lmm.\ 11.2]{AbbMajer:MorseI} up to a small
  issue that $W_-$ (\resp $W_+$) is assumed to be an open subset of
  the unstable (\resp stable) manifold of some critical point of index
  $\mu(p)+1$ (\resp $\mu(p)-1$) intersected with a level
  set. Nevertheless the proof goes through without any change for
  general $W_-$ and $W_+$ provided that we have the splitting
  \[T_{w^\infty_-}W_- \oplus T_{w^\infty_-} W^s(p) = T_{w^\infty_-}
  C\quad \text{ and }\quad T_{w^\infty_+}W^u(p) \oplus
  T_{w^\infty_+}W_+ =T_{w^\infty_+} C\,,\] which holds by
  transversality and the assumption $\dim W_- +\dim W_+ =\dim C$.
  \begin{stp}
    We define~\eqref{eq:Mglue}
  \end{stp}
  For all $R_1,R_2 \geq R_0$ the spaces $W_-^{R_1}$ and $W_+^{-R_2}$
  have an unique intersection point inside $Q(r_0)$. Indeed, given any
  $(x,y) \in W_-^{R_1} \cap W_+^{-R_2} \in H^u \oplus H^s$ we conclude
  with step~\ref{stp:graphtransform} that $(x,y) = (x,\sigma_{R_1}(x))
  =(\tau_{R_2}(y),y)$. Hence $x$ is the unique fixed point of the
  contraction $\tau_{R_2} \circ \sigma_{R_1}$ and $y$ is the unique
  fixed point of the contraction $\sigma_{R_1} \circ \tau_{R_2}$. By
  abuse of notation we denote the intersection point $(x,y)$ by
  $W^{R_1}_- \cap W^{-R_2}_+$. For all $R \geq 2R_0$ we define the
  map~\eqref{eq:Mglue} via $R \mapsto (R,w_-^R,w_+^R)$ with $w_-^R :=
  \psi^{-R/2}(W^{R/2}_- \cap W^{-R/2}_+)$ and
  $w_+^R:=\psi^{R/2}(W^{R/2}_- \cap W^{-R/2}_+)$. It remains to check
  the properties.
  \begin{stp}\label{stp:wRwinf}
    We show~\ref{nm:wRwinf}.    
  \end{stp}
  By construction the orthogonal projection onto $H^u$ of the point
  $\psi^{R_0}w_-^R = W_-^{R_0} \cap W_+^{R_0-R} \in H^u\oplus H^s$ is
  the fixed point of the contraction $ \tau_{R-R_0}\circ
  \sigma_{R_0}$. Since $\tau_{R-R_0}$ converges to $\tau_\infty$
  uniformly as $R \to \infty$ the fixed point converges to the fixed
  point of $ \tau_\infty \circ \sigma_{R_0}$. Similarly we show that
  the orthogonal projection onto $H^s$ of $\psi^{R_0} w_-^R$ converges
  to the fixed point of $\sigma_{R_0} \circ \tau_\infty$. In other
  words $\psi^{R_0}w_-^R$ converges to the unique intersection point
  $W_-^{R_0} \cap W^s(p)$ as $R \to \infty$ because $W^s(p)\cap
  Q(r_0)$ is the graph of $\tau_\infty$. After possibly making
  $\rho>0$ smaller again we assume that $w^\infty_-$ is the only point
  in $W_- \cap W^s(p)$. Hence $\psi^{R_0}w_-^R \to
  \psi^{R_0}w^\infty_-$ and thus $w_-^R \to w_-^\infty$. Completely
  analogous we argue that $w_+^R \to w_+^\infty$.

  \begin{stp}
    We show~\ref{nm:Msurj}.  
  \end{stp}
  A standard Morse compactness argument shows that there exists
  $\delta>0$ such that for any $w \in B_{\delta}(w^\infty_-)$ with
  $\psi^R w \in B_{\delta}(w^\infty_+)$ for some $R$ we must have
  $\psi^{R/2}w \in Q(r_0)$. In particular if $(R,w_-,w_+) \in W_-
  \times_\psi W_+$ such that $w_- \in B_\delta(w^\infty_-)$ and
  $w_+=\psi^Rw_- \in B_\delta(w^\infty_+)$, then $\psi^{R/2}w_- \in
  Q(r_0)$. By uniqueness of the intersection point we conclude that
  $\psi^{R/2}w_- =W^{R/2}_- \cap W^{-R/2}_+$, hence $w_-=w_-^R$ and
  $w_+=w_+^R$.
  \begin{stp}\label{stp:Ominus}
    We equip $[0,\infty]:=[0,\infty)\cup \{\infty\}$ with the
    topology of $[0,\pi/2]$ induced by the bijection $[0,\pi/2] \cong
    [0,\infty]$, $s \mapsto \tan(s)$ and $\pi/2 \mapsto \infty$. Let
    $\Gr(H)$ denote the space of linear subspaces in $H$. There exists
    continuous maps
    \begin{equation}
      \label{eq:Ominus}
      \Omega_-:[0,\infty]^2 \to \Gr(H)\,,  
    \end{equation}
    such that for all $R \in [0,\infty]$ we have
    \begin{enumerate}[label=($\alph*$)]
    \item\label{nm:R0} $\Omega_-(R,0) = T_{w_-^R} W_-$ ,
    \item\label{nm:R8} $\Omega_-(R,\infty) = T_{\psi^{-R}(w^\infty_+)}
      W^u(p)$,
    \item\label{nm:0R} $\Omega_-(0,R) \oplus T_{w^R_+}W_+ = H$,
    \item\label{nm:8R} $\Omega_-(\infty,R) \oplus
      T_{\psi^R(w^\infty_-)}W^s(p) = H$.
    \end{enumerate}
  \end{stp}
  Gradient flow lines $R \mapsto \psi^R(w^\infty_-)$ and $R \mapsto
  \psi^{-R}(w^\infty_+)$ extend to continuous functions on
  $[0,\infty]$. By Step~\ref{stp:wRwinf} the paths $R \mapsto w^R_-$
  and $R \mapsto w^{R}_+$ also extend to continuous functions. Write
  $\Omega_-(R_1,R_2) \subset H$ as the graph of a linear map
  $S(R_1,R_2):H^u \to H^s$ of norm $<1$. On two sides of the
  quadrilateral $[0,\infty]^2$ the map $(R_1,R_2) \mapsto S(R_1,R_2)$
  is already determined by the conditions~\ref{nm:R0} and
  \ref{nm:R8}. For these sides the condition on the norm is satisfied
  since by Step~\ref{stp:graphtransform} the spaces $W_- \cap Q(r_0)$
  and $W^u(p) \cap Q(r_0)$ are graphs of maps with Lipschitz-constant
  $<1$. As the space of linear maps with norm $<1$ is convex we extend
  the map $S$ to $[0,\infty]^2$ uniquely up to homotopy. Because the
  norm of $S(0,R)$ is $<1$ and $W_+\cap Q(r_0)$ and $W^s(p) \cap
  Q(r_0)$ are graphs of a map with Lipschitz-constant $<1$ the
  conditions~\ref{nm:0R} and~\ref{nm:8R} are automatically satisfied.
  \begin{stp}\label{stp:WpW}
    We show~\ref{nm:Mori} case \ref{nm:WpW}.  
  \end{stp}
  Without loss of generality assume $R_0 =0$ (if not replace $W_-$ by
  $\psi^{R_0/2} W_-$ and $W_+$ by $\psi^{-R_0/2}W_+$).  Abbreviate
  $W:=W_- \times_\psi W_+$. Set $w=w_-^0 =w^0_+ \in W_- \cap W_+$. By
  sequence~\eqref{eq:orW} the tangent space of $W$ at $(0,w,w)$ is
  identified with the kernel of the map
  \[
  \phi:\R \oplus T_{w} W_- \ra T_{w} C/T_{w} W_+,\quad
  (\theta,\xi^-) \mapsto \xi^--\theta \grad_{w} f\,.
  \]
  By definition the orientation $o_W$ of $W$ is given by 
  \[o_W \wedge o_+ = o_{\R} \wedge o_-\,,\] in which $o_\R$ is the
  standard orientation of $\R$, $o_{-}$ is the fixed orientation of
  $T_{w}W_-$ and $o_+$ is the orientation of a linear complement of
  the kernel of $\phi$ such that $\phi_*(o_+)$ is the fixed
  orientation of $T_{w} C/T_{w} W_+$.  The kernel of $\phi$ is
  one-dimensional. The sign $\delta \in \{\pm 1\}$ of the vector
  $\xi:=(1,\pa_R w^R_-,\pa_R w^R_+)$ is given by
  \[\xi = \delta\, o_W\,,\]
  where by abuse of notation we denote by $\xi$ also the orientation
  on $W$ induced by $\xi$. We see directly that $T_{w} W_-$ is a
  linear complement of $\xi$ in $\R \oplus T_{w}W_-$ and moreover
  \begin{equation*}
    \xi\wedge o_{-} = o_\R \wedge o_{-}\,.
  \end{equation*}
  Since $\phi$ maps the subspace $T_{w}W_- \subset \R \oplus T_w W_-$
  isomorphically onto $T_w C/T_w W_+$ there exists $\alpha \in \{\pm
  1\}$ such that
  \[o_+ = \alpha \,o_-\,.\] Collecting the last four
  equations we have
  \[
  \delta\, o_W \wedge  o_+ = \xi\wedge  o_+ = \alpha
  \,\xi \wedge o_- = \alpha\, o_\R \wedge o_- = \alpha\, o_W \wedge
   o_+\,.
  \]
  It suffices to show that $\alpha = \e
  =\sign(w^\infty_-)\sign(w^\infty_+)$.  We view $\Omega_-$ which is
  constructed in Step~\ref{stp:Ominus} as a vector bundle over
  $[0,\infty]^2$. Because the base is contractible the vector bundle
  is orientable and an orientation is determined by an orientation of
  a fibre. By condition~\ref{nm:R0} the bundle $\Omega_-$ is oriented
  by $W_-$, also denoted $o_-$. Let $o_p$ be an orientation of
  $W^u(p)$. By condition~\ref{nm:R8} the bundle is oriented by $o_p$
  also denoted by $o_p$. Finally by condition~\ref{nm:0R} the bundle
  $\Omega_-$ is oriented by $o_+$, also denoted $o_+$. Abbreviate
  $\e_-:=\sign(w_-^\infty)$ and $\e_+:=\sign(w_+^\infty)$ defined
  above. By definition $\alpha\, o_- = o_+$, $\e_-\, o_- =
  o_p$ and $\e_+\, o_p = o_+$. Putting these three equations together
  shows~\ref{nm:Mori} case~\ref{nm:WpW}.
  
  \begin{stp}\label{stp:localWW}
    For all $R \geq R_0$ we have
    \begin{enumerate}[label=(\alph*)]
    \item $\pa_R w_-^R$ is the projection of $\grad_{w^R_-}f$ onto
      $T_{w^R_-}W_-$ along $T_{w^R_-}\psi^{-R}W_+$,
    \item $\pa_R w_+^R$ is the projection of $-\grad_{w^R_+} f$ onto
      $T_{w^R_+}W_+$ along $T_{w^R_+}\psi^RW_-$.
    \end{enumerate}
  \end{stp}\label{stp:split}
  Without loss of generality $R=0$. We identify a neighborhood of
  $w:=w^0_-=w^0_+$ with an open ball in $H$ identifying $w$ with zero
  and such that the gradient vector field of $f$ is constant under the
  identification. Write $-\grad f =v$ for some vector $v \in H$. We
  have a splitting $H=H_-\oplus H_+$ where $H_\pm = T_w W_\pm$. The
  space $W_-$ is given as a graph $\vp:H_- \to H_+$ with $\d \vp(0)=0$
  and $W_+$ is given as a graph $\phi:H_+ \to H_-$ with $\d
  \phi(0)=0$. Write $v=(v_-,v_+) \in H_-\oplus H_+$. An intersection
  point $(x,y) \in \psi^RW_- \cap W_+$ satisfies 
  \[
  (x+Rv_-,\vp(x)+Rv_+) = (\phi(y),y)\,.
  \]
  In particular $y$ is a fixed point of the map
  $\theta_R(y):=\vp(\phi(y)-Rv_-)+Rv_+$. Up to possibly considering a
  smaller neighborhood $\theta_R$ is a contraction for all $R$ small
  enough and hence the unique fixed point $y_0(R)$ depends smoothly on
  $R$. The corresponding intersection point is
  $w^R_+=(\phi(y_0(R)),y_0(R))$. By deriving the equation
  $\theta_R(y_0(R)) = y_0(R)$ by $R$ shows that $\pa_R w^R_+ =
  (0,v_+)$. Similarly we show $\pa_R w^R_- = (-v_-,0)$.
  \begin{stp}
    We show~\ref{nm:Mori} case \ref{nm:Wp}.   
  \end{stp}
  Let $\Omega_-:[0,\infty]\to \Gr(H)$ be the
  map~\eqref{eq:Ominus}. For each $R \in [0,\infty]$ consider the
  vector $\xi(R) \in \Omega_-(0,R)$ which is defined to be the
  projection of $-\grad f$ at $w^R_+$ onto $\Omega_-(0,R)$ along
  $W_+$. The vector $\xi(R)$ is well-defined by property~\ref{nm:0R}
  of $\Omega_-$ and we have that $\xi(R) \R = \Omega_-(0,R) \cap
  T_{w^R_+} \wt W_+$. Let $\chi(R) \subset \Omega_-(0,R)$ be a linear
  complement of $\xi(R)$ which depends continuously on $R$. The space
  $\xi(R)$ is oriented via the coorientation of $\wt W_+$ and the
  canonical identification
  \[\chi(R)
  \cong T_{w^R_+}C_+/T_{w^R_+}\wt W_+\,.\] We view $\chi$ as a vector
  bundle over $[0,\infty]$ and denote the orientation by $\wt o_+$,
  which by~\eqref{eq:oriWt+} is uniquely determined by $-\grad f
  \wedge \wt o_+ =o_+$. As above we have orientations $o_p$ and $o_-$
  of $\Omega_-$ induced by $W^u(p)$ and $W_-$ respectively.  By
  Step~\ref{stp:split} we have $\pa_R w^R_-|_{R=0} = - \xi(0)$. By
  definition of the orientation of $\wt W_- \cap W_+$ the sign
  $\delta\in \{\pm 1\}$ of $\pa_R w^R_-$ is defined by $-\delta\, \xi
  \wedge \wt o_+ = o_-$. We still have $\e_-o_-=o_p$ and
  $\e_+o_p=o_+=\xi\wedge \wt o_+$. We conclude that
  $-\delta=\e_-\e_+=\e$.
  \begin{stp}
    We show~\ref{nm:Mori} case \ref{nm:pW}.
  \end{stp}
  Similar to step~\ref{stp:Ominus} we show that there exists a
  continuous map
    \begin{equation*}
      \Omega_+:[0,\infty]^2 \to
      \Gr(H)     \,,
    \end{equation*}
    such that for all $R \in [0,\infty]$ we have
    \begin{enumerate}[label=$(\alph*)$]
    \item $\Omega_+(R,0)\oplus T_{w^R_-} W_- = H$,
    \item $\Omega_+(R,\infty) \oplus T_{\psi^{-R}(w^\infty_+)}W^u(p)= H$,
    \item $\Omega_+(0,R)= T_{w^R_+}W_+$,
    \item $\Omega_+(\infty,R) =T_{\psi^R(w^\infty_-)}W^s(p)$.
    \end{enumerate}
    For all $R \in [0,\infty]$ let $\xi(R)\in \Omega_+(R,0)$ be the
    projection of the negative gradient $-\grad f$ at $w^R_-$ onto
    $\Omega_+(R,0)$ along $W_-$. Let $\wt o_-$ be the orientation of
    $\wt W_-$, which by~\eqref{eq:oriWt-} is determined as $\wt o_- =
    -\grad f \wedge o_-=\xi \wedge o_-$. By Step~\ref{stp:split} the
    degree $\delta \in \{\pm 1\}$ of $\pa_R w^R_+ \in \wt W_- \cap
    W_+$ is defined by the requirement $\delta\, \xi \wedge o_+ = \wt
    o_-$. We still have $\e_- o_- = o_p$ and $\e_+ o_p = o_+$ with
    coorientations $o_p$ and $o_+$ of $\Omega_+$ induced by an
    orientation of $W^u(p)$ and a coorientation of $W_+$. We conclude
    that $\delta = \e_-\e_+$ as claimed.
\end{proof} 

\section{Orientation}
We construct orientations for the moduli space of holomorphic strips
with boundary on cleanly intersecting Lagrangians using relative spin
structures. In principle this has been established by Fukaya \ea
in~\cite[\S 8]{FO3:II}. Because our setup is a bit different we repeat
these ideas here using a slightly different language. In particular we
use a different but equivalent notion of relative spin structures due
to Wehrheim\&Woodward~\cite{WW:orient}.

\subsection{Preliminaries}\label{sec:prelim}
\subsubsection{Determinant} To any real finite dimensional vector space
$X$ of dimension $n\geq 0$, we associate the \emph{determinant}
denoted by
\[\det X:=\Lambda^n X\,.\]
The choice of a basis of $X$ gives an isomorphism $\det X \cong
\R$. By definition the determinant of the zero dimensional space is a
fixed copy of $\R$.  It is well-known that given an exact sequence of
finite dimensional real vector spaces
\[
0 \to X_1 \to X_2 \to X_3 \to \dots \to X_k \to 0\,,
\]
 we obtain an natural isomorphism
\begin{equation}
  \label{eq:oddeven}
  \bigotimes_{j \text{ odd}} \det X_j \cong \bigotimes_{j \text{ even}} \det X_j\,.
\end{equation}
The word ``natural'' means that an isomorphism between exact sequences
gives rise to a commuting square (\cf \cite[\S 5]{AbbMajer:inf}).

\subsubsection{Orientation Torsor} To a finite dimensional vector space
$X$ of dimension $n \geq 0$, we associate the set 
\[ |X| := (\Lambda^n X\setminus \{0\})/\R^+\,.\] Here $\R^+$ is the
group of positive real numbers acting freely on $\Lambda^n
X\setminus\{0\}$ by scalar multiplication.  The set $|X|$ has two
elements and choosing a basis picks one of the two elements. We call
$|X|$ the \emph{orientation torsor of $X$} and the elements of $|X|$
are called \emph{orientations}. We say that $X$ is \emph{oriented}, if
an element of $|X|$ is chosen. If $X$ is zero dimensional then we
have a canonical identification $|X| = \{\pm 1\}$.

Let $\Z_2$ denote the group with two elements. The group $\Z_2$ acts
freely and transitively on $|X|$ with action of the non-trivial
element induced by multiplication of $-1$ on $\Lambda^n X$. Given two
vector spaces $X$ and $X'$ we define 
\[|X| \otimes |X'| := (|X|\times |X'|)/\Z_2\,,\] with quotient taken
with respect to the diagonal action $(-1)\cdot(o,o') = (-o,-o')$. The
space $|X| \otimes |X'|$ has again two elements and a free and
transitive $\Z_2$-action induced by
$(-1)\cdot[o,o']:=[-o,o']=[o,-o']$, where $[o,o']$ denotes the
equivalence class of $(o,o')$ in $|X| \otimes |X'|$. It is easy to
check that we have a natural isomorphism $|X|\otimes (|X'|\otimes
|X''|) \cong (|X|\otimes |X'|)\otimes |X''|$ for any three vector
space $X$, $X'$ and $X''$. Thus we do not specify parenthesis for
iterated products. We also define the \emph{dual torsor}
\[|X|^\vee:=\Hom_{\Z_2}(|X|,\Z_2)\,,\]
 consisting of all
$\Z_2$-equivariant maps to $\Z_2$. We have a natural isomorphism $|X|
\otimes |X|^\vee \cong \Z_2$. For two finite dimensional vector spaces
$X$, $Y$ we have natural isomorphism
\begin{equation}
  \label{eq:OrDet}
  \left|X \oplus Y \right| \cong |X| \otimes |Y|\,.  
\end{equation}
Commuting the factors is natural with respect action with
\begin{equation}
  \label{eq:commute}
  (-1)^{\dim
    X\dim Y}\,.
\end{equation}

\subsubsection{Fibre products} For finite dimensional vector spaces $X$,
$Y$ and $Z$ and linear maps $\vp:X \to Z$, $\psi:Y \to Z$, we define
the \emph{fibre product}
\[X \tensor[_\vp]{\times}{_\psi} Y := X\times_Z Y:=\{ (x,y) \mid \vp(x)=\psi(y)\}
\subset X \oplus Y\,.\] We say that a fibre product is
\emph{transverse} if the sequence is exact
\begin{equation}\label{eq:fibresequence}
\xymatrix{ 0 \ar[r] & X\times_Z Y \ar[r]& X \oplus Y \ar[r]^-{\vp-\psi}
  &Z\ar[r]&0}\,.  
\end{equation}
We obtain via~\eqref{eq:oddeven} and~\eqref{eq:OrDet} the canonical
isomorphism
\begin{equation}
  \label{eq:orifibre}
  \nm{X\times_Z Y} \cong |X| \otimes |Z|^\vee \otimes |Y|\,.
\end{equation}
The order of the factors leads to the following associativity property
of the orientation of fibre products.
\begin{lmm}\label{lmm:fibreassociative}
  Given oriented vector spaces $X_0$, $X_{01}$, $X_1$, $Z_0$ and $Z_1$
  as well as maps $\vp_0:X_0 \to Z_0$, $\psi=(\psi_0,\psi_1):X_{01}
  \to Z_0 \oplus Z_1$ and $\vp_1:X_1 \to Z_1$.  Provided that the
  fibre products are transverse we have
  \[X_0 \times_{Z_0} (X_{01} \times_{Z_1} X_1) =(X_0 \times_{Z_0}
  X_{01}) \times_{Z_1} X_1\,\] where the equality holds as oriented
  subspaces of $X_0 \oplus X_{01} \oplus X_1$.
\end{lmm}
\begin{proof}
  This is a slight generalization of \cite[Lmm.\ 8.2.3]{FO3:II} since
  we do not require that the maps are surjective (we do however
  require that their respective differences are surjective). Obviously
  both fibre products define the subspace
  \[Y = \{(x_0,x_{01},x_1) \mid \vp_0(x_0)= \psi_0(x_{01}),\
  \psi_1(x_{01})=\vp_1(x_1)\} \subset X_0 \oplus X_{01} \oplus
  X_1\,.\] It remains to check that the induced orientations on $Y$
  agree. We denote the oriented spaces $Y := X_0 \times_{Z_0} (X_{01}
  \times_{Z_1} X_1)$ and $Y':= (X_0 \times_{Z_0} X_{01}) \times_{Z_1}
  X_1$.  Define the fibre products $Y_0:= X_0 \times_{Z_0} X_{01}$ and
  $Y_1 := X_{01} \times_{Z_1} X_1$. We identify $Z_0$ and $Z_1$ with
  subspaces of $X_0 \oplus X_{01}$ and $X_{01} \oplus X_1$ using
  right-inverses to $\phi_0=\vp_0 -\psi_0$ and $\phi_1=\psi_1-\vp_1$
  respectively. Moreover we identify $Z_0$ and $Z_1$ with subspaces $
  Z_0' \subset X_0 \oplus Y_1$ and $Z_1' \subset Y_0 \oplus X_1$ using
  right-inverses of the restriction of $\phi_0$ and $\phi_1$
  respectively. We use small letters $x_0$, $y_0$, \etc to denote the
  dimensions of the spaces $X_0$, $Y_0$, etc.  By definition we have
  the oriented isomorphisms
  \[
  Y \oplus Z_0' \cong (-1)^{z_0 y_1} X_0 \oplus Y_1,\qquad Y_1 \oplus Z_1 \cong (-1)^{z_1x_1} X_{01} \oplus X_1\,.
  \]
  Hence as subspaces of $X=X_0 \oplus X_{01} \oplus X_1$
  \[
  Y \oplus Z_0' \oplus Z_1 \cong (-1)^{z_0y_1} X_0 \oplus Y_1 \oplus Z_1 \cong (-1)^{z_0 y_1 + z_1x_1} X\,.
  \]
  On the other hand we have similarly
  \begin{equation*}
    Y' \oplus Z_0 \oplus Z_1' \cong (-1)^{z_0z_1+ z_1x_1} Y_0 \oplus X_1 \oplus
    Z_0  \cong
    (-1)^{z_1x_1 +z_0x_1 + z_0 x_{01}+z_0z_1} X\,.
  \end{equation*} 
  By transversality we have $y_1=x_{01} + x_1 -z_1$. By direct
  verification we see that in the last two isomorphisms the
  coefficient on the right-hand side is the same. Since the space of
  linear complements is contractible there exists a homotopy from $Z_0
  \oplus Z_1'$ to $Z_0' \oplus Z_1$. We conclude that the orientation
  of $Y$ and $Y'$ is the same. 
\end{proof}
A special case of a fibre product is obtained if $X, Y \subset Z$ are
subspaces and the maps $\vp$ and $\psi$ are inclusions. Then the fibre
product is isomorphic to the intersection $X \cap Y$. If $X+Y=Z$ we
have the exact sequence
\begin{equation}\label{eq:capori}
\xymatrix{0\ar[r]&X \cap Y \ar[r]& X \ar[r]& Z/Y\ar[r]&0\,,}  
\end{equation}
which is usually used to orient the intersection of two vector spaces
provided with an orientation of $X$ and a coorientation of $Y$ (\ie an
orientation of $Z/Y$). The next lemma shows that the orientation on the
intersection seen as a fibre product agrees with this orientation.
\begin{lmm}\label{lmm:capfibre}
  Given oriented vector spaces $X,Y \subset Z$ such that $X+Y=Z$.  If
  $X\times_Z Y$ and $X \cap Y$ are oriented via~\eqref{eq:orifibre}
  and~\eqref{eq:capori} respectively then the projection to the first
  factor $X\times_Z Y \to X \cap Y$ is orientation preserving.
\end{lmm}
\begin{proof}
  The space $Z/Y$ is oriented by the canonical sequence $0 \to Y \to Z
  \to Z/Y \to 0$. Abbreviate $W := X \cap Y$ equipped with orientation
  given by the sequence~\eqref{eq:capori}. Let $\rho:Z \to X\oplus Y$
  be a right inverse to $X\oplus Y \to Z$, $(x,y) \mapsto x-y$. We
  need to check that the determinant of the isomorphism $W \oplus Z
  \to X \oplus Y$, $(w,z) \mapsto w+\rho(z)$ has sign $(-1)^{\dim
    Y\dim Z}$. Pick a linear complement $\ol X$ of $W$ inside $X$. We
  choose an orientation on $\ol X$ such that $W \oplus \ol X =X$ is
  orientation preserving. If $W$ is oriented via~\eqref{eq:capori}
  then $Y \oplus \ol X = Z$ is orientation preserving. Define the
  right-inverse $\rho:Z=Y\oplus \ol X \mapsto X \oplus Y$, $(y,x)
  \mapsto x-y$. Then the sign of $W \oplus Z \to X\oplus Y$, $(w,y,x)
  \mapsto w+x-y$ is $(-1)^{\dim Y\dim Z}$.
\end{proof}
\subsubsection{Vector bundles and manifolds} All previous observations
extend directly to the category of manifolds and finite rank vector
bundles.  In particular if $\pi:E \to X$ is a finite rank vector
bundle over a locally path-connected space $X$ we define the
\emph{orientation cover} $|E| \to X$ as the double cover with fibre
over $x$ given by $|E_x|$ with vector space $E_x =\pi^{-1}(x)$. The
vector bundle $E$ is \emph{orientable} if there exists a section of
$|E|$, which happens if and only if the \emph{first Stiefel-Whitney
  class} $w_1(E) \in H^1(X,\Z_2)$ vanishes (\cf \cite[Thm.\
II.1.2]{Spin}). If $w_1(E)=0$, then $|E|$ has exactly two
sections. We say that an orientable vector bundle $E$ is
\emph{oriented}, if a section of $|E|$ is chosen.

If $X$ is a finite dimensional manifold, we abbreviate by $|X|=|TX|$
the \emph{orientation cover of $X$}. We say that $X$ is
\emph{oriented} if a section of $|X|$ is chosen.  If $X=\{x\}$ is a
point the space $|X|$ is canonically identified with $\{\pm 1\}$ and
an orientation of $x$ is denoted by $\sign x \in \{\pm 1\}$. If $X$ is
a manifold with boundary $\partial X$, then an orientation of the
interior induces a canonical orientation on $\partial X$ by demanding
that for all $x \in \partial X$ and outward pointing vectors $\xi_\out
\in T_x X$ the isomorphism is orientation preserving
\begin{equation}
  \label{eq:orbdary}
  T_x \partial X \oplus \xi_\out \R \cong T_x X\,.
\end{equation}

Given smooth smooth maps $\vp:X \to Z$, $\psi:Y \to Z$ between smooth
finite dimensional manifolds $X$, $Y$ and $Z$. We define the
\emph{fibre product}
\[
X\tensor[_\vp]{\times}{_\psi}Y= X \times_Z Y = \{(x,y) \mid
\vp(x)=\psi(y) \} \subset X \times Y\,.
\]
If the maps are evident from the context we simply denote the fibre
product by $X \times_Z Y$.  We say that the fibre product is cut-out
transversely if at each point $(x,y) \in X \times_Z Y$ the
differentials $\d_x \vp$ and $\d_y \psi$ are transverse in the sense
above. If so the fibre product is a manifold with tangent space at
$(x,y)$ give by the fibre product of $\d_x \vp$ with $\d_y \psi$.  Let
$\O \to Z$ be a double cover. An \emph{$\O$-orientation on $\vp:X \to
  Z$} is a section of $|X|\otimes \vp^*\O$. Similar an
\emph{$\O$-coorientation on $\psi:Y \to Z$} is a section of
$|Y|\otimes \psi^*|Z|^\vee \otimes \psi^*\O $.
\begin{lmm}\label{lmm:orifibre}
  An $\O$-orientation on $\vp$ and an $\O^\vee$-coorientation on
  $\psi$ induce an orientation on the transverse fibre product
  $X\times_Z Y$.
\end{lmm}
\begin{proof}
  The tangent space of $X\times_Z Y$ is the fibre product of $\d_x
  \vp$ with $\d_y \psi$. Let $z=\vp(x)=\psi(y)$ and pick orientations
  of $\O_z$ and $T_zZ$. We obtain orientations of $T_xX$ and $T_yY$
  using the sections and an orientation on the fibre product
  via~\eqref{eq:orifibre}. It is easy to check that the orientation on
  the fibre product is independent of choices.
\end{proof}
If $G$ is a Lie group acting freely on the manifold $\wt X$, then the
quotient $X:=\wt X/G$ is a manifold. Let $\g=T_e G$ be the Lie algebra
of $G$. We obtain an exact sequence
\begin{equation}\label{eq:oriquotient}
\xymatrix{0\ar[r]&\g\ar[r]&T_x \wt X\ar[r]&T_{[x]} X \ar[r]&0\,,}  
\end{equation}
which is natural with respect to homotopies.  Hence if $\wt X$ and $G$
are oriented we obtain a canonical orientation on $X$
via~\eqref{eq:oddeven} and~\eqref{eq:oriquotient}.

\subsubsection{The determinant bundle over the space of Fredholm
  operators} Let $X,Y$ be Banach spaces and denote $\F(X,Y)$ the space
of Fredholm operators from $X$ to $Y$, equipped with the induced
topology as a subspace of the bounded linear operators from $X$ to
$Y$. We define the \emph{determinant line bundle}, denoted
$\det(X,Y)$, as line bundle on $\F(X,Y)$ with fibre over $D$ given by
\begin{equation}
  \label{eq:detD}
  \det D  := \det(\ker D) \otimes \det(\coker D)^\vee\,.
\end{equation}
The fibre $\det D$ is called \emph{determinant line}.  Although in
general the dimension of the kernel and the cokernel is not constant
as $D$ varies continuously in $\F(X,Y)$, we have the following fact.
\begin{prp}\label{prp:detloctrivial}
  The space $\det(X,Y)$ is a locally trivial line bundle. 
\end{prp}
A proof is given in \cite[Theorem A.2.2]{Bibel} or~\cite[\S
7]{AbbMajer:inf}. A simple observation shows that if $X \cong Y$ and
$X$ is an infinite dimensional Hilbert space the determinant line
bundle is not orientable (\cf \cite[Exercise A.2.5]{Bibel}). Also we
denote by $|D| = |\det D|$ the \emph{orientation torsor} and we call
the elements of $|D|$ the \emph{orientations of $D$}. We say that $D$
is \emph{oriented} if an element of $|D|$ is chosen. If $D$ is an
isomorphism the orientation torsor $|D|$ is canonically identified
with $\{\pm 1\}$ and an orientation of $D$ is denoted by $\sign D \in
\{\pm 1\}$.
\subsection{Spin structures and relative spin structures}\label{sec:spin}
We recall the notion of a spin structure and a relative spin
structure. The definition which we give is due to
Wehrheim-Woodward (\cf \cite[\S 3]{WW:orient}). 
\subsubsection{\v{C}ech cohomology} We give basic definitions which are
taken from \cite[\S 3]{WW:orient} and~\cite[\S 5,\ 10]{BottTu}. Let
$X$ be a manifold and $G$ a topological group which is not necessarily
abelian. For $k \in \N_0$ and an open cover $\UU=\{U_\alpha \subset
X\mid \alpha \in I\}$ with totally ordered index set $I$ a
\emph{\v{C}ech $k$-cochain with values in $G$} is a tuple of
continuous maps
\[
\pp = (\pp_{\alpha_0\alpha_1\dots\alpha_{k}}:U_{\alpha_0}\cap
U_{\alpha_1} \cap \dots \cap U_{\alpha_{k}} \to
G)_{\alpha_0\alpha_1\dots\alpha_k}
\]
indexed over all strictly ordered subsets in $I$ with $k+1$
elements. We write $C^k(\UU,G)$ for the space of all such tuples.  The
space $C^k(\UU,G)$ is a group with group law given by pointwise
multiplication and neutral element given by the cocycle $\one$, which
is the function mapping each point to the neutral element in $G$.  We define
the \emph{\v{C}ech differential}
\[
\d:C^k(\UU,G) \to C^{k+1}(\UU,G),\quad
(\d\pp)_{\alpha_0\alpha_1\dots\alpha_{k+1}}= \prod_{j=0}^{k+1}
\pp_{\alpha_0\alpha_1\dots\hat\alpha_j\dots\alpha_{k+1}}^{(-1)^j}\,,
\]
and the \emph{\v{C}ech $k$-cocycles},
\[Z^k(\UU,G):=\{ \pp \in C^k(\UU,G) \mid \d \pp = \one\}\,.\] The
group $C^0(\UU,G)$ acts from the left on $Z^1(\UU,G)$ by
$(\h.\pp)_{\alpha\beta} =h_\alpha\, \pp_{\alpha\beta}\,
h_{\beta}^{-1}$. We define the \emph{\v{C}ech cohomology groups} by
\[
H^1(\UU,G):= C^0(\UU,G)\setminus Z^1(\UU,G),\qquad
H^0(\UU,G):=Z^0(\UU,G)\,.
\]
If $G$ is abelian then for any $k \in \N_0$ the \v{C}ech $k$-cochains
are abelian groups and together with the \v{C}ech differential form a
cochain complex, which allows us to define the \v{C}ech cohomology
groups, denoted $H^k(\UU,G)$, for all $k \in \N_0$.

A \emph{refinement $\V$ of $\UU$} is a cover $\V=\{V_{\alpha'} \subset
X \mid \alpha' \in I'\}$ such that for all
indices $\alpha \in I$ there exists $\alpha' \in I'$ with
$V_{\alpha'} \subset U_\alpha$. We have natural restriction maps
\[
C^k(\UU,G) \to C^k(\V,G), \qquad \pp_{\alpha_0\alpha_1\dots\alpha_k} \mapsto \pp_{\alpha_0'\alpha_1'\dots\alpha_k'}\,.
\]
On the right-hand side it might be that the indices are not strictly
ordered. To allow indices of any order we use the convention
$\pp_{\alpha_0\alpha_1\dots \alpha_k}
=\pp_{\alpha_{\sigma(0)}\alpha_{\sigma(1)}\dots\alpha_{\sigma(k)}}^{\sign
  \sigma}$ for any permutation $\sigma$ of $k+1$ elements and
$\pp_{\alpha_0\alpha_1\dots\alpha_k}=\one$ if any two indices are the
same (\cite[p. 93]{BottTu}).

A \emph{good cover $\UU$} is a cover such that all multiple
intersections are contractible.  One shows that for any good cover
$\UU$  the group $H^*(\UU,G)$ does not depend on
$\UU$ up to canonical isomorphism and in that case we denote the group
by $H^*(X,G)$. Moreover if $G$ is abelian and equipped with a discrete
topology this group is canonically isomorphic to the usual cohomology
groups with coefficients in $G$, so there is no ambiguity in the
notation.  Any open cover on a manifold has a refinement which is a
good cover (\cf \cite[p. 43]{BottTu}) and by restricting we always
assume that cochains are given with respect to a good cover.

As \v{C}ech cochains are basically maps, we naturally define the
pull-back with respect to continuous maps $\vp :X \to Y$ and the
push-forward with respect to continuous group homomorphisms $\tau:G
\to H$. To define the push-forward we put
\[
\tau_*:C^k(\UU,G)\to C^k(\UU,H),\qquad
(\tau_*\pp)_{\alpha_0\alpha_1\dots\alpha_k} = \tau \circ
\pp_{\alpha_0\alpha_1\dots\alpha_k}\,.\] For the pull-back we define
\[  \vp^*:C^k(\UU_Y,G) \to C^k(\UU_X,G),  \qquad
  (\vp^*\pp)_{\alpha_0\alpha_1\dots\alpha_k}=
  \pp_{\alpha_0\alpha_1\dots\alpha_k} \circ \vp\,,
\] 
where $\UU_Y = \{U_{\alpha} \subset Y \mid \alpha \in I\}$ is an open
cover of $Y$ and $\UU_X = \vp^*\UU_Y :=\{\vp^{-1}(U_\alpha) \subset
X \mid \alpha \in I\}$ is the pull-back cover.

\subsubsection{Spin structures} For $n \geq 2$ the spin group $\Spin(n)$
is by definition the non-trivial double cover of the special
orthogonal group $\SO(n)$. We denote by $\tau:\Spin(n) \to \SO(n)$ the
covering map and identify the kernel of $\tau$ with $\Z_2$.  Write
the action of an element $h \in \Z_2$ on $g \in \Spin(n)$ via
$(-1)^h g$. Let $X$ be a manifold and $\pi:E \to X$ an oriented
finite rank vector bundle equipped with a Riemannian structure. We
denote by $\SO(E)$ the \emph{oriented orthonormal frame bundle}, \ie
the principle bundle over $X$ with fibre over the point $x \in X$
given by all oriented orthonormal bases in the vector space
$E_x:=\pi^{-1}(x)$. The transition maps for local trivializations of
$\SO(E)$ over the open sets of a good cover $\UU_X=\{U_\alpha \subset
X\}_{\alpha \in I}$ define a \v{C}ech cocycle $\ff \in
Z^1(\UU_X,\SO(n))$. If a cocycle $\ff$ arises in such a way for some
local trivializations we say that $\ff$ \emph{represents $\SO(E)$}.
\begin{dfn}\label{dfn:spin}
  A \emph{spin structure on $E$} is a \v{C}ech cocycle $\pp \in
  Z^1(\UU_X,\Spin(n))$ such that $\tau_* \pp$  represents
    $E$. We call two spin structures $\pp$ and $\pp'$ isomorphic, if
  there exists a cochain $\h \in C^0(\UU_X,\Z_2)$ such that
  $\h.\pp=\pp'$.
\end{dfn}
\begin{rmk}\label{rmk:classicspin}  
  Classically a spin structure on $E$ is a $\Spin(n)$-bundle $P \to
  X$ together with a double cover $\rho:P \to \SO(E)$ such that
  $\rho(g.p) = \tau(g).\rho(p)$ for all $g \in \Spin(n)$ and $p \in P$
  (\cf~\cite[Dfn.\ 1.3]{Spin}). A spin structure $\pp$ in the above
  sense is formed by the transition maps from a local trivialization
  of $P$. Conversely, given a spin structure $\pp$ as above,
  we obtain a principle $\Spin(n)$-bundle $P$ by gluing. Since
  $\tau_*\pp$ represents $E$ the map $P \to X$ lifts to a double cover
  $\rho:P \to \SO(E)$ with the required property. See also
  \cite[Prop.\ 3.1.3]{WW:orient}.
\end{rmk}
Not every oriented vector bundle admits a spin structure. The
topological obstruction is given by the \emph{second Stiefel-Whitney
  class} $w_2(E) \in H^2(X,\Z_2)$. The class $w_2(E)$ is defined as
follows: Let $\ff$ be a cocycle representing $\SO(E)$. We find $\pp \in
C^1(\UU_X,\Spin(n))$ such that $\tau_*\pp=\ff$. Then $\d \pp$ is a
cocycle with values in $\Z_2$ and $w_2(E)=[\d \pp]$ is its cohomology class (\cf
\cite[page 83]{Spin}).  If $w_2(E)=0$, the bundle $E$ admits a spin
structure and moreover a free and transitive action of $H^1(X,\Z_2)$
on the isomorphism class of spin structures on $E$. Consequently 
the space of isomorphism classes of spin structures on $E$ is an
affine space, which is (non-canonically) isomorphic to $H^1(X,\Z_2)$
(\cf \cite[Thm.\ II.1.7]{Spin}). If $X$ is an oriented Riemannian
manifold, a \emph{spin structure of $X$} is a spin structure of its
tangent bundle and we call $X$ \emph{spin} if it admits a spin
structure or equivalently if $w_2(TX)=0$.
 
\subsubsection{Relative spin structures} Given a smooth map $\phi:X \to Y$
between smooth manifolds and an oriented finite rank vector bundle
$\pi: E \to X$ equipped with a Riemannian structure. Fix good covers
$\UU_X$ and $\UU_Y$ of $X$ and $Y$ respectively such that $\UU_X$ is a
refinement of the pull-back cover $\phi^*\UU_Y$.
\begin{dfn}\label{dfn:relspin}
  A \emph{ spin structure of $E$ relative to $\phi$} is a pair
  $(\pp,\ww)$ consisting of a cochain $\pp \in C^1(\UU_X,\Spin(n))$
  and a cocycle $\ww \in Z^2(\UU_Y,\Z_2)$ such that $\tau_*\pp$ is a
  cocycle representing $E$ and we have $\d \pp = \phi^* \ww$. Two
  relative spin structures $(\pp,\ww)$ and $(\pp',\ww')$ are
  \emph{isomorphic}, if there exists cochains $\h \in C^0(\UU_X,\Z_2)$
  and $\k \in C^1(\UU_Y,\Z_2)$ such that $\h.\pp=\pp'$ and
  $\k.\ww=\d\k+ \ww =\ww'$.  The cohomology class $w :=[\ww] \in H^2(Y,\Z_2)$ is
  called the \emph{background class}.
\end{dfn}
For example an ordinary spin structure on $E$ is a special case of a
relative spin structure with trivial cochain $\ww$.  That our
definition of a relative spin structure is equivalent to~\cite[Def.\
8.1.2]{FO3:II} is proven in~\cite[Prop.\ 3.1.15]{WW:orient}. Let
$\phi^*:C^*(Y;\Z_2)\to C^*(X;\Z_2)$ be the pull-back. The \emph{cone
  of $\phi^*$}, denoted $C^*(\phi;\Z_2)$, is the complex
$C^*(X;\Z_2)\oplus C^{*+1}(Y;\Z_2)$ equipped with boundary operator
$\d(\h,\k) = (\d \h +\phi^*\k,\d \k)$. Here we have
used the more familiar notation of writing the group law in $\Z_2$
additively. The space of cocycles is denoted by $Z^*(\phi;\Z_2)$ and
the homology by $H^*(\phi;\Z_2)$.  The next proposition is proven in
\cite[Prop.\ 3.1.13]{WW:orient}.
\begin{prp}\label{prp:relspinexists}
  A bundle $E \to X$ admits a spin structure relative to $\phi:X \to
  Y$ if and only if there exists a class $w \in H^2(Y,\Z_2)$ such that
  $\phi^*w = w_2(E)$. If so, $H^1(\phi^*;\Z_2)$ acts freely and
  transitively on the set of isomorphism classes of relative spin
  structures via $[\h,\k].[\pp,\ww]=[(-1)^\h\, \pp,\k+
  \ww]$ for each $(\h,\k) \in Z^1(\phi^*;\Z_2)$.
\end{prp}
\subsubsection{Bundles over strips} The purpose of relative spin
structures is to keep track of homotopy classes of trivializations for
bundles over the boundary of a strip. Abbreviate by $\Sigma = \R
\times [0,1]$ the strip with boundary $\partial \Sigma = \R \times
\{0,1\}$. 
\begin{lmm}\label{lmm:reltriv}
  Given a vector bundle $F \to \partial \Sigma$ with fixed
  trivializations $\Phi_-$ and $\Phi_+$ of the restrictions
  $F|_{(-\infty,-s_0]\times\{0,1\}}$ and
  $F|_{[s_0,\infty)\times\{0,1\}}$ respectively.  A relative spin
  structure of $F$ relative to the inclusion $\partial \Sigma \subset
  \Sigma$ induces a homotopy class of trivializations of $F \oplus \R$
  which agree with $\Phi_-$ and $\Phi_+$ over the ends.
\end{lmm}
\begin{proof}
  See~\cite[Prop.\ 3.1.15, Prop.\ 3.3.1]{WW:orient} for the same
  statement for compact surfaces with boundary.  

  Let $(\pp,\ww)$ be the relative spin structure defined with respect
  to open covers $\UU_\Sigma$ and $\UU_{\partial \Sigma}$. Since
  $\Sigma$ is contractible the cycle $\ww$ is exact and we find $\uu
  \in C^1(\UU_\Sigma,\Z_2)$ such that $\d \uu = \ww$. Fix $k=0,1$ and
  we denote by $\uu_k := \uu|_{\R\times \{k\}}$, $\pp_k:=\pp|_{\R
    \times \{k\}}$ and $\ww_k:=\ww|_{\R \times\{k\}}$ the pull-back of
  the cochains to the boundary, which we identify with cochains on
  $\R$ with respect to an open cover $\UU_\R$.  Consider the cochain
  \[\hat \pp_k:=(\pp_k,-\uu_k) \in C^1(\UU_\R,\Spin(n)\times
  \Z_2)\,.\] Let $\Spin(n)\times_{\Z_2} \Z_2$ denote the quotient of
  $\Spin(n)\times \Z_2$ by the anti-diagonal action of $\Z_2$, which
  is a Lie group because $\Z_2$ acts by central elements. The boundary
  of $\hat \pp_k$ is $(\ww_k,-\ww_k)$ hence the push-forward of $\hat
  \pp_k$ to a chain with values in $\Spin(n) \times_{\Z_2} \Z_2$ is a
  cocycle, which we denote by $\bar \pp_k$. By assumption the
  push-forward of $\hat \pp_k$ to a $\SO(n)\times \{\one\}$-chain is
  $(\ff_k,\one)$, where $\ff_k$ is the cocycle obtained from a
  trivialization of $\SO(F_k)$. By the homotopy lifting principle we
  have the commutative diagram in which the vertical arrows are double
  covers and horizontal arrows are inclusions
  \[
  \xymatrix{\Spin(n)\times_{\Z_2} \Z_2\ar[r]\ar[d]^{2:1}&\Spin(n+1)\ar[d]^{2:1}\\
    \SO(n)\times\{\one\} \ar[r]&\SO(n+1)\,.}
  \]
  The push-forward of $\bar \pp_k$ along the inclusion
  $\Spin(n)\times_{\Z_2}\Z_2 \hookrightarrow \Spin(n+1)$ is denoted by
  $\check \pp_k$. By commutativity we conclude that the push-forward
  of $\check \pp_k$ to $\SO(n+1)$ is $(\ff_k,\one)$. Thus $\check
  \pp_k$ is a spin structure of $F_k \oplus \R$. By gluing (\cf
  Remark~\ref{rmk:classicspin}) we obtain $\Spin(n+1)$-bundles $P_k$
  over $\R$ and maps $P_k \to \SO(F_k \oplus \R)$, which are
  non-trivial double covers on each fibre. Using the trivializations
  $\Phi_-$ and $\Phi_+$ we identify the fibre of $P_k$ over $s$ for
  $\nm{s} \geq s_0$ with $\Spin(n+1)$.  Because the spin group is
  connected there exists a section of $P_k$ which is the identity
  element over $(-\infty,-s_0]$ and $[s_0,\infty)$.  The push-forward
  of the section to $\SO(F_k\oplus \R)$ gives the
  trivialization. Since the spin group is simply connected any two
  choices of the section of $P_k$ are homotopic through a homotopy
  that fixes the endpoints. Hence the trivialization does not depend
  on the choices up to homotopy.
\end{proof}
\subsection{Orientation of caps}\label{sec:theta}
For ordinary Morse theory an orientation on the moduli space of Morse
trajectories is given once an orientation of the space of unstable
directions for each critical point is fixed. Unfortunately for the
moduli spaces of holomorphic strips with boundary on Lagrangians in
clean intersection the situation is not so simple. In fact already for
Morse-Bott functions on finite dimensional manifolds, the space of
Morse trajectories is not necessarily orientable anymore. However it
still holds locally that the orientation of the tangent space of the
moduli space of Morse trajectories at any Morse trajectory is given
canonically in terms of the orientations of the unstable directions of
the critical points which the trajectory connects. If the Lagrangians
are relatively spin the situation is similar for the moduli space of
holomorphic strips where orientation of the caps take the role of the
orientation of the unstable directions.

Given a symplectic manifold $M$ and Lagrangians submanifolds $L_0,L_1
\subset M$ such that there exists a relative spin structure on $TL_0
\sqcup TL_1$ relative to $L_0 \sqcup L_1 \to M$ (\cf
Definition~\ref{dfn:relspin}). We repeat the definition adapted to the
context.
\begin{dfn}\label{dfn:relspinMLL}
  A \emph{relative spin structure for $(L_0,L_1)$} is a triple
  $(\pp_0,\pp_1,\ww)$ such that $\ww$ is a $\Z_2$-cocycle on $M$,
  $\pp_k$ is a $\Spin(n)$-cochain on $L_k$, $\tau_*\pp_k$ represents
  $TL_k$ and $\d \pp_k = \ww|_{L_k}$ for $k=0,1$.
\end{dfn}
Let $X=X_H$ be the Hamiltonian vector field of a clean Hamiltonian $H
\in C^\infty([0,1]\times M)$ and $J:[0,1]\to \End(TM,\omega)$ be a
path of almost complex structures.  The following discussion easily
generalizes when $X$ and $J$ are admissible in the sense of
Definition~\ref{dfn:JXadm}. However for the sake of simplicity we only
consider the case of $\R$-invariant structures.  Use the short-hand
notation $\I:=\I_H(L_0,L_1)$ for the perturbed intersection
points. Choose a constant $\e>0$ and consider the space of strips (\cf
Definition~\ref{dfn:mudecay} for the definition of $C^{\infty;\e}$-regularity)
\[\B :=\{ u\in C^{\infty;\e}(\bar \R\times [0,1], M) \mid
u(\cdot,k) \subset L_k\text{ for }k=0,1 ,\ u(\pm \infty) \in
\I\}\,.\] Fix an element $\base \in \I$ once and for all. A
\emph{cap of $x \in \I$} is an element $u \in \B$ such that
$u(-\infty) =\base$ and $u(\infty)=x$. We denote by $\B(\base,x)
\subset \B$ the subspace of all caps of $x$ and by $\B(\base)$ the
space of all caps.
\begin{rmk}
  The space $\B(\base)$ replaces the space $\tilde \I(R_h)$ from
  \cite[\S 8.8]{FO3:II}.
\end{rmk}
By Theorem~\ref{thm:DuFred} we see if $\e$ is small enough then for
all $u \in \B$ the linearized Cauchy-Riemann operator $D_u$ is
Fredholm and we denote by $|D_u|=|\det D_u|$ the corresponding
orientation torsor.
\begin{lmm}\label{lmm:OrDu}
  Given $u\in \B$ and caps $u_-,u_+ \in \B(\base)$ such that
  $x_-:=u(-\infty)=u_-(\infty)$ and $x_+:=u(\infty)=u_+(\infty)$.  A
  relative spin structure for the pair $(L_0,L_1)$ induces an
  isomorphism
  \begin{equation}
    \label{eq:OrDu}
    |D_{u_-}| \otimes  |D_{u}| \cong   | D_{u_+}| \otimes |T_{x_-} \I|  \,,
  \end{equation}
  which is natural with respect to homotopies, \ie given homotopies
  $(u^\tau)_{\tau \in [a,b]} \subset \B$ and $(u_-^\tau)_{\tau \in
    [a,b]}$, $(u_+^\tau)_{\tau \in [a,b]} \subset \B(\base)$ such that
  $x^\tau_-:=u^\tau(-\infty)=u^\tau_-(\infty)$ and
  $x_+^\tau:=u^\tau(\infty)=u^\tau_+(\infty)$ we have have the
  commutative diagram
  \[
  \xymatrix{
    |D_{u^a_-}| \otimes |D_{u^a}| \ar[r]\ar[d]& |D_{u_+^a}| \otimes
    |T_{x^a_-} \I|\ar[d]\\
    |D_{u^b_-}| \otimes |D_{u^b}| \ar[r]&|D_{u_+^b}| \otimes |T_{x^b_-} \I|\,,}
  \]
  in which the horizontal maps are by~\eqref{eq:OrDu} and the vertical
  are induced by the homotopy.
\end{lmm}
\begin{proof}\setcounter{stp}{0}
  Disclaimer: The construction of the isomorphism involves
  choices. These choices are unique up to homotopy and hence the
  isomorphism on the orientations does not depend on these choices. We
  will not specifically mention this every time.

  Consider trivializations $\Phi_u$, $\Phi_{u_-}$ and $\Phi_{u_+}$ of
  $u^*TM$, $u_-^*TM$ and $u_+^*TM$ respectively with properties listed
  in the proof of Theorem~\ref{thm:DuFred} and such that
  $\Phi_u(-\infty)=\Phi_{u_-}(\infty)$,
  $\Phi_u(\infty)=\Phi_{u_+}(\infty)$ and
  $\Phi_{u_-}(-\infty)=\Phi_{u_+}(-\infty)=\Phi_*$, where $\Phi_*$ is
  a trivialization of $(\base)^*TM$ which is fixed once and for
  all. As explained further in the proof using the trivializations we
  obtain maps $S$, $S_-$, $S_+: \bar \R \times [0,1] \to \R^{2n\times
    2n}$ and $F$, $F_-$, $F_+:\R \to \L(n) \times \L(n)$ such that
  $(F,S)$, $(F_-,S_-)$ and $(F_+,S_+)$ are admissible and the
  operators $D_u$, $D_{u_-}$ and $D_{u_+}$ are conjugated to
  $D_{F,S}$, $D_{F_-,S_-}$ and $D_{F_+,S_+}$ respectively. By
  construction we have asymptotics $\sigma_-:=S(-\infty)=S_-(\infty)$,
  $\sigma_+:=S(\infty)=S_+(\infty)$ and $\sigma_*
  :=S_-(-\infty)=S_+(-\infty)$, where $\sigma_*$ is a path which is
  fixed once and for all.  In particular $\sigma_*$ does not depend on
  the maps $u$, $u_-$ and $u_+$. Using the isomorphism the kernel of
  the operator $A_{\sigma_-}$ is conjugated to $T_{x_-}\I$ (\cf
  Lemma~\ref{lmm:kerHess}). Via pull-back we have relative spin
  structures and by Lemma~\ref{lmm:reltriv} trivializations of $F$,
  $F_-$ and $F_+$ which are standard over the ends. Using the
  trivializations in Lemma~\ref{lmm:OrFS} we obtain integers $\mu_-$,
  $\mu$, $\mu_+ \in \Z$ and isomorphisms of $|D_u|$, $|D_{u_-}|$ and
  $|D_{u_+}|$ with $|\mu.\sigma_-:\sigma_+|$,
  $|\mu_-.\sigma_*:\sigma_-|$ and $|\mu_+.\sigma_*:\sigma_+|$
  respectively. For any $\nu \in \Z$ we fix once and for all an
  orientation in $|\sigma_*:\nu.\sigma_*|$.  The claim follows by
  linear orientation gluing Lemma~\ref{lmm:glueOr} and the canonical
  isomorphism~\eqref{eq:musigma}. Indeed we have
  \begin{equation*}
    |D_{u_+}| \cong |\sigma_*:\mu_+.\sigma_*| \otimes
    |\mu_+.\sigma_*:\sigma_+| \cong |\sigma_*:\sigma_+| \otimes |\sigma_*:\sigma_*|\,,
  \end{equation*}
  and plugging this isomorphism in the next
  \begin{align*}
    |D_{u_-}| \otimes |D_u| &\cong
    |\sigma_*:(\mu_-+\mu).\sigma_*| \otimes |\mu_-.\sigma_*:\sigma_-|
    \otimes |\mu.\sigma_-:\sigma_+| \\
    &\cong   |\sigma_*:(\mu_-+\mu).\sigma_*| \otimes |(\mu_-+\mu).\sigma_*:\mu.\sigma_-|
    \otimes |\mu.\sigma_-:\sigma_+| \\
    &\cong   |\sigma_*:\sigma_+| \otimes |\sigma_*:\sigma_*|  \otimes
    |\sigma_-:\sigma_-|\\
    &\cong |\sigma_*:\sigma_+|\otimes |\sigma_*:\sigma_*|\otimes
    |T_{x_-}\I|\\
&\cong |D_{u_+}|\otimes |T_{x_-}\I|\,.
  \end{align*}
  This shows~\eqref{eq:OrDu}.  Since all isomorphisms are natural with
  respect to homotopies we also have the commutative diagram.
\end{proof}
In particular if $u\in \B$ is such that $u(s,\cdot) =x$ for all $s \in
\R$, we have canonically $|D_u| \cong |T_x \I|$ and we conclude that a
relative spin structure induces a canonical isomorphism for any two
caps $u_-$, $u_+ \in \B(\base,x)$
\begin{equation}
  \label{eq:orcaps}
 |D_{u_-}| \cong |D_{u_+}|\,, 
\end{equation}
which is natural with respect to homotopies. Thus the
following double cover is well-defined.
\begin{dfn}\label{dfn:O}
  Given a relative spin structure for $(L_0,L_1)$. We define the
  double cover $\O \to \I$ with fibre over $x \in \I$ given by
  \[
  \O_x :=  \bigsqcup_{u \in \B(\base,x)} |D_u| \Big/ \sim
  \]
  in which two elements $o \in |D_u|$ and $o'\in |D_{u'}|$ are
  equivalent if they are identified by the isomorphism~\eqref{eq:orcaps}.
\end{dfn}
\begin{rmk}
  If we pull-back the double cover $\O$ to $\B(\base)$ along the
  fibration $\B(\base)\to \I$, $u \mapsto u(\infty)$ it is isomorphic
  to the double cover $\B(\base)^+ \to \B(\base)$ with fibre over $u$
  given by $|D_u|$ (\cf \cite[Prp.\ 8.8.1]{FO3:II}). Using notation
  of~\cite[\S 8.8]{FO3:II} the cover $\O$ is the orientation bundle of
  $\Theta$.
\end{rmk}
We come to the main result of the chapter. Given connected components
$C_-$, $C_+ \subset \I$. We denote by $\wt \M(C_-,C_+;J,X)$ the space
of $(J,X)$-holomorphic strips $u$ such that $u(-\infty) \in C_-$ and
$u(\infty) \in C_+$ which comes equipped with the evaluation map
\[\ev=(\ev_-,\ev_+):\wt \M(C_-,C_+;J,X)\to C_- \times
C_+,\qquad u \mapsto (u(-\infty),u(\infty))\,.\] For more details see
Section~\ref{sec:regsetup}.
\begin{thm}\label{thm:ori}
  Assume $(L_0,L_1)$ is equipped with a relative spin structure and
  let $\O$ be the associated double cover (\cf
  Definition~\ref{dfn:O}).  Given connected components $C_-, C_+
  \subset \I_H(L_0,L_1)$ and suppose that $J$ is regular for
  $X=X_H$. For any $u \in \wt \M(C_-,C_+;J,X)$ connecting
  $x_-=u(-\infty)$ to $x_+ = u(\infty)$ we have the canonical
  isomorphism
    \[
    |\wt \M(C_-,C_+;J,X)|_u  \cong \O_{x_-}^\vee \otimes \O_{x_+}
    \otimes |\I|_{x_-}\,,
    \]
    which is natural with respect to homotopies.
\end{thm}
\begin{proof}
  See \cite[Proposition 8.8.6]{FO3:II}. The result easily follows from
  Lemma~\ref{lmm:OrDu}. By assumption the linearized Cauchy-Riemann
  operator $D_u$ is surjective and the tangent bundle of
  $\wt\M(C_-,C_+;J,X)$ at $u$ is the kernel of $D_u$. We conclude that
  the orientation torsor of $T_u\wt \M(C_-,C_+;J,X)$ equals
  $|D_u|$. Choose caps $u_-$ and $u_+$ of $x_-:=u(-\infty)$ and
  $x_+:=u(\infty)$ respectively. The fibers $\O_{x_-}$ and $\O_{x_+}$
  are represented by orientations of $D_{u_-}$ and $D_{u_+}$
  respectively. Then with Lemma~\ref{lmm:OrDu} we have the natural
  isomorphism $|D_u| \cong |D_{u_-}|^\vee \otimes |D_{u_+}| \otimes
  |T_{x_-} C_-|$.
\end{proof} 
\begin{cor}\label{cor:ori}
  Given a smooth map $\vp_-:W_- \to C_-$ such that $J$ is regular for
  $X$ and $\vp$. An $\O^\vee$-orientation for $\vp_-$ induces an
  $\O^\vee$-orientation for
  \[\ev_+:\wt \M(W_-,C_+;J,X)\to C_+,\qquad (w,u) \mapsto u(\infty)\,.\]
\end{cor}
\begin{proof}
  Using the canonical orientation of the tangent space on the fibre
  product (\cf equation~\eqref{eq:orifibre}) and Theorem~\ref{thm:ori}.
\end{proof}
\subsection{Linear theory}\label{sec:linori}
We proof the orientation gluing for Cauchy-Riemann operators with
degenerated asymptotics defined on strips.  This generalizes the
orientation gluing for Cauchy-Riemann operators the cylinder with
non-degenerated asymptotics as was established in~\cite[Section
3]{Floer:Coherent}.

Fix a constant $\delta>0$, we denote by $\A$ the space of paths
$\sigma:[0,1]\to \R^{2n\times 2n}$ such that $\sigma(t)$ is symmetric
for all $t \in [0,1]$ and the operator (\cf equation~\eqref{eq:AAop})
\[A_\sigma:H^{1,2}_{\Lambda}([0,1],\R^{2n}) \to
L^2([0,1],\R^{2n}),\qquad \xi \mapsto \Jstd \pt \xi + \sigma \cdot
\xi\,,\] with $\Lambda=(\R^n,\R^n)$, has spectral gap
$\iota(A_\sigma)>\delta$ (\cf equation~\ref{eq:iotaA}).  Fix two paths
$\sigma_-,\sigma_+ \in \A$ and a constant $\e$ such that
$\delta<\e<\min\{\iota(A_{\sigma_-},\iota(A_{\sigma_+})\}$. We define
\[
\D(\sigma_-,\sigma_+) :=\{ S \in C^{\infty;\e}(\bar \R \times [0,1],\R^{2n\times 2n}) \mid S(\pm \infty,\cdot) =\sigma_\pm\}\,.
\]
To any $S\in \D(\sigma_-,\sigma_+)$ and $F:\R \to \L(n)\times \L(n)$
such that $(F,S)$ is admissible (\cf Definition~\ref{dfn:FSadm}) we
associate the operator (\cf equation~\eqref{eq:H1pext}
and~\eqref{eq:DFSext})
\[
D_{F,S}:H^{1,2;\delta}_{F;W}(\Sigma,\R^{2n}) \to
L^{2;\delta}(\Sigma,\R^{2n}),\qquad \xi \mapsto \ps \xi + \Jstd \pt
\xi + S \xi\,,
\]
with $W=(\ker A_{\sigma_-},\ker A_{\sigma_+})$.  In case when $F$ is
such that $F(s) = (\R^n,\R^n)$ for all $s \in \R$ we simply write
$D_S$.  By Lemma~\ref{lmm:Dext} the operator $D_{F,S}$ is Fredholm and
we denote by $\det D_{F,S}$ and $|D_{F,S}|$ the associated determinant
line and orientation torsor respectively.  We identify
$\D(\sigma_-,\sigma_+)$ with an open subset in the space of 
Fredholm operators, via $S \mapsto D_S$. Moreover it is easily seen
that $\D(\sigma_-,\sigma_+)$ is convex, hence the pull-back of the
determinant line bundle to $\D(\sigma_-,\sigma_+)$ is orientable. We
denote by
\begin{equation}
  \label{eq:musigma}
  |\sigma_-,\sigma_+|  
\end{equation}
the space of the two possible orientations of the determinant line
bundle over $\D(\sigma_-,\sigma_+)$, \ie the space of sections in the
double cover over $\D(\sigma_-,\sigma_+)$ where the fibre over $S$ is
given by the two possible orientations of the Fredholm operator $D_S$.

For any $\mu \in \Z$ and $t \in [0,1]$ define the unitary matrix
\begin{equation}\label{eq:phimu}
  \phi_\mu(t) =
  \begin{pmatrix}
    e^{\pi i\mu t}&0\\0&\one 
  \end{pmatrix}\,.
\end{equation}
There is an action of $\Z$ on $\A$ given by $\mu.\sigma =
\phi_{-\mu} \, \sigma \,\phi_\mu + \Jstd\, \phi_{-\mu} \pt
\phi_\mu$. Alternatively the action is defined for the associated
operator by $A_{\mu.\sigma} = \phi_{-\mu}\, A_\sigma \,\phi_\mu$. This
easily shows that the spectrum of $A_\sigma $ is not changed under the
action, which implies that the spectral gap is left
invariant. Moreover for all $\mu \in \Z$ we have $\phi_{-\mu}\,
D_S\,\phi_\mu = D_{\phi_{-\mu}\, S \, \phi_\mu + \Jstd \phi_{-\mu}\pt
  \phi_\mu}$ and thus a canonical isomorphism for all $\mu \in \Z$
\begin{equation}
  \label{eq:ksigma}
  |\mu.\sigma_-,\mu.\sigma_+| \cong |\sigma_-,\sigma_+|\,.
\end{equation}
We now state the main lemma.
\begin{lmm}[Orientation gluing]\label{lmm:glueOr}
  Given $\sigma_-,\sigma,\sigma_+ \in \A$, there exists an
  isomorphism
  \begin{equation}
    \label{eq:glueOr}
    \nm{\sigma_-,\sigma} \otimes \nm{\sigma,\sigma_+} \ra |\sigma,\sigma| \otimes \nm{\sigma_-,\sigma_+}\,,    
  \end{equation}
  which is natural with respect to homotopies, \ie given homotopies
  $(\sigma^\tau_-),(\sigma^\tau)$ and $(\sigma^\tau_+)$ in $\A$, then
  there exists a commutative diagram,
  \[
  \xymatrix{ \nm{\sigma^0_-,\sigma^0} \otimes \nm{\sigma^0,\sigma_+^0}
    \ar[r]\ar[d]&
    |\sigma^0,\sigma^0|\otimes \nm{\sigma_-^0,\sigma^0_+} \ar[d]\\
    \nm{\sigma^1_-,\sigma^1} \otimes \nm{\sigma^1,\sigma^1_+}\ar[r]&
    |\sigma^1,\sigma^1|\otimes \nm{\sigma^1_-,\sigma^1_+}\,, }
  \]
  where the horizontal isomorphism is induced by~\eqref{eq:glueOr} and
  the vertical by homotopies.
\end{lmm}
\begin{proof}
  Choose $S_0 \in \D(\sigma_-,\sigma)$ and $S_1 \in
  \D(\sigma,\sigma_+)$ with $S_0(s,\cdot)=S_1(-s,\cdot)$ for all $s
  \geq 2$. For any $R\geq 2$, we consider the glued map $S_0\#_RS_1$
  as given in~\eqref{eq:glueS}. We obtain an isomorphisms $|D_0|
  \otimes |D_1| \cong |\ker A_\sigma| \otimes |D_R|$, constructed in
  Lemma~\ref{lmm:glueOrD0D1D01} below. Note that $\ker A_\sigma$ is
  the same as $\ker D_\sigma$ and $D_\sigma$ is surjective, hence
  $|\ker A_\sigma|=|D_\sigma|$.  Extend the isomorphism uniquely to
  obtain~\eqref{eq:glueOr} via the homotopy lifting principle. We
  explain, why the isomorphism does not depend on the choice of
  $S_0,S_1$: Choose another elements $S'_0 \in \D(\sigma_-,\sigma)$
  and $S'_1 \in \D(\sigma,\sigma_+)$. These are joined to $S_0$ and
  $S_1$ via a homotopy $(S^\tau_0)_{\tau \in [0,1]}$ and
  $(S^\tau_1)_{\tau \in [0,1]}$ respectively. By naturality of the
  gluing construction, the obtained isomorphism on the orientations is
  the same (\cf Lemma~\ref{lmm:glueornat}). We argue similarly to
  show that the isomorphism is natural with respect to homotopies of
  $\sigma_-, \sigma$ and $\sigma_+$.
\end{proof}
\begin{lmm}\label{lmm:associative}
  The isomorphism~\eqref{eq:glueOr} is associative. More precisely,
  given paths $\sigma_0,\dots,\sigma_3 \in \A$, then we have
  a commuting square
  \[
  \xymatrix{ \nm{\sigma_0,\sigma_1}\otimes \nm{\sigma_1,\sigma_2}
    \otimes \nm{\sigma_2,\sigma_3} \ar[r]\ar[d]&
    |\sigma_1,\sigma_1|\otimes \nm{\sigma_0,\sigma_2} \otimes \nm{\sigma_2,\sigma_3}\ar[d]\\
    \nm{\sigma_0,\sigma_1}\otimes |\sigma_2,\sigma_2| \otimes
    \nm{\sigma_1,\sigma_3}\ar[r]& |\sigma_1,\sigma_1| \otimes
    |\sigma_2,\sigma_2| \otimes \nm{\sigma_0,\sigma_3}\,.}
  \]
  in which all but the lower horizontal map is given by the gluing
  map~\eqref{eq:glueOr} and the lower horizontal map is given by
  commuting the factors and the gluing map~\eqref{eq:glueOr}.
\end{lmm}
\begin{proof}
  See \cite[Lmm 2.4.2]{WW:orient} or \cite[Lemma 3.5]{Ekholm:orient}.
\end{proof}
\subsubsection{Path of Lagrangians}
We show that a stable trivialization of $F$ induces an orientation of
$D_{F,S}$ up to data which only depends on the asymptotics and the
index of $F$. We view $F$ as a bundle over $\R \times \{0,1\}$ with
fibre over $(s,k) \in \R\times \{0,1\}$ given by $F_k(s)$.  Suppose
that $F_k(s)=F_k(-s)=\R^n$ for all $s \geq s_0$, $k=0,1$. An
\emph{admissible trivialization} is a trivialization of $F$ given by a
special orthogonal frame which is standard over $(-\infty,-s_0]$ and
$[s_0,\infty)$.
\begin{lmm}\label{lmm:OrFS}
  Given $S \in \D(\sigma_-,\sigma_+)$ and $F:\R \to \L(n)\times \L(n)$
  such that $(F,S)$ is admissible. Let $\mu$ be the Robbin-Salamon
  index of $F$. An admissible trivialization of $F \oplus \R$ induces
  an isomorphism
  \begin{equation}
    \label{eq:OrFS}
    |D_{F,S}| \cong |\mu.\sigma_-,\sigma_+|\,,
  \end{equation}
  which is natural with respect to homotopies, \ie given homotopies
  $(S^\tau)_{\tau \in [a,b]}$ and $(F^\tau)_{\tau \in [a,b]}$ such
  that $S^\tau \in \D(\sigma^\tau_-,\sigma^\tau_+)$ with
  $\sigma^\tau_\pm=S^\tau(\pm\infty,\cdot)$ and $(S^\tau,F^\tau)$ is
  admissible for all $\tau \in [a,b]$, then an admissible
  trivialization of $F\oplus \R$ gives the commutative diagram
  \[
  \xymatrix{ |D_{F^a,S^a}|\ar[r]\ar[d]& |\mu.\sigma^a_-,\sigma^a_+|\ar[d]
    \\|D_{F^b,S^b}|\ar[r]&|\mu.\sigma^b_-,\sigma_+^b|\,, }
  \]  
  in which the horizontal isomorphism is induced by~\eqref{eq:OrFS}
  and the vertical is induced homotopies.
\end{lmm}
\begin{proof}
  For $k=0,1$ a trivialization of $F_k$ is given by a frame
  $e^k_1,\dots,e^k_{n+1}:\R \to \R^{2n+2}$ of $F_k \oplus \R$ which is
  standard over $(-\infty,-s_0]$ and $[s_0,\infty)$. For $s \in \R$
  define the unitary matrix $\Psi_k(s)$ uniquely by
  $\Psi_k(s)e_j=e^k_j(s)$ for all $1\leq j\leq n+1$ where
  $e_1,\dots,e_{n+1}$ denotes the standard basis of $\C^{n+1}$.  Thus
  $F_k(s)\oplus \R=\Psi_k(s)\R^{n+1}$ for all $s \in \R$. The
  Robbin-Salamon index of $F$ is an integer, since $F$ starts and ends
  at $(\R^n,\R^n)$. With $\phi_\mu$ as given in~\eqref{eq:phimu} the
  concatenation $\phi_\mu\#\Psi_0$ has the same Malsov index as
  $\Psi_1$.  By \cite[Thm.\ 4.1]{Robbin:Paths} we conclude that the
  paths are homotopic with fixed endpoints, which implies the
  existence of a map $\Psi:\Sigma \to U(n+1)$ such that
  \begin{itemize}
  \item $\Psi(s,k)\R^{n+1} = F_k(s) \oplus \R \subset
    \C^{n+1}$ for all $s \in \R$ and $k=0,1$,
  \item $\Psi(s,t) = \phi_\mu(t)$ for all $s \leq -s_0$ and $t
    \in [0,1]$,
  \item $\Psi(s,t) = \one$ for all $s \geq s_0$ and $t\in [0,1]$.
  \end{itemize}
  To see that $\Psi$ exists uniquely up to homotopy, let $\Psi'$ be
  another choice. Along the boundary of $[-s_0,s_0]\times [0,1]$ we
  identify $\Psi$ and $\Psi'$ to obtain an map from $S^2$ to
  $\U(n+1)$. Since the unitary group is two-connected we find a
  homotopy from $\Psi$ to $\Psi'$. 

  Now define the map $S_\Psi:=\Psi^{-1} \, S \, \Psi + \Psi^{-1}\ps
  \Psi + \Jstd \Psi^{-1}\pt \Psi$ and $F'=(F'_0,F'_1)$ with $F'_k(s) =
  F_k(s) \oplus \R$ for $k=0,1$. The operators $D_{F',S}$ and
  $D_{S_\Psi}$ are conjugated by $\Psi$. The kernel of $D_{F',S}$
  splits into $\ker D_{F,S} \oplus \R$ where $\R$ is given by the
  space of constant maps $\xi:\R \times [0,1] \to \R^{2n} \oplus
  \R^2$, $\xi(s,t)=(0,a)$ for some $a \in \R$. Moreover the cokernel
  of $D_{F',S}$ and $D_{F,S}$ are the same. Fixing the standard
  orientation on $\R$, then $\Psi$ induces an isomorphism between the
  orientation torsors of $D_{F,S}$ and $D_{S_\Psi}$. By construction
  $S_\Psi(-\infty,\cdot)= \mu.\sigma_-$ and
  $S_\Psi(\infty,\cdot)=\sigma_+$. This shows~\eqref{eq:OrFS}.  

  We show naturality. Using the trivialization of $F$ we define $\Psi$
  as above on the six sides of the cuboid $[a,b]\times
  [-s_0,s_0]\times [0,1]$ such that $\Psi(\tau,-s_0,t)=\phi_\mu(t)$
  and $\Psi(\tau,s_0,t)=\one$ for all $\tau \in [a,b]$ and $t \in
  [0,1]$. Note that the by the homotopy axiom the Robbin-Salamon index
  of $F^\tau$ is independent of $\tau$. Because the unitary group is
  two connected the map extends to a map defined on the cuboid
  $\Psi:[a,b]\times [-s_0,s_0]\times [0,1]\to \U(n+1)$,
  $\Psi^\tau=\Psi(\tau,\cdot)$. In particular the orientation torsor
  of $D_{F,S}$ is isomorphic to $D_{S_{\Psi^\tau}}$ by conjugation
  with $\Psi^\tau$.  Since $S_{\Psi^\tau} \in
  \D(\mu.\sigma^\tau_-,\sigma^\tau_+)$ for all $\tau \in [a,b]$ the
  claim follows by the homotopy lifting principle.
\end{proof}

\subsubsection{Linear gluing}
Given $S_0 \in \D(\sigma_-,\sigma)$ and $S_1 \in \D(\sigma,\sigma_+)$
such that $S_0(s,\cdot)=S_1(-s,\cdot)=\sigma$ for all $s \geq 1$.  For
each $R \geq 1$ we define
\begin{equation}\label{eq:glueS}
S_R:=S_0\#_R S_1 =
\begin{cases}
  S_{0,R}=S_0 \circ \tau_{-2R}&\text{if } s\leq -R\\
  S_0(\infty)=S_1(-\infty)&\text{if } \nm{s}\leq R\\
  S_{1,R}=S_1 \circ \tau_{2R}&\text{if } s \geq R\,,
\end{cases}
\end{equation}
where $\tau_{2R}:\Sigma \to \Sigma$ denotes the translation
$\tau_R(s,t)=(s-2R,t)$. We abbreviate
\begin{itemize}
\item the asymptotic operators $A_-:=A_{\sigma_-}$, $A:=A_\sigma$ and
  $A_+:=A_{\sigma_+}$,
\item their kernels $C_-:=\ker A_-$, $C:=\ker A$ and $C_+:=\ker A_+$,
\item the operators $D_0:=D_{S_0}$, $D_1:=D_{S_1}$ and $D_R:=D_{S_R}$
  which are defined on the Banach spaces $H_0$, $H_1$ and $H_R$
  with target $L_0$, $L_1$ and $L_R$ respectively,
\item the restricted operators $D_{01}:=D_0 \oplus D_1|_{H_{01}}:H_{01}
  \to L_{01}$ where $H_{01} \subset H_0 \oplus H_1$ consists of
  functions $(\xi_0,\xi_1)$ such that $\xi_0(\infty)=\xi_1(-\infty)$
  and $L_{01}:=L_0 \oplus L_1$.
\end{itemize}
\begin{lmm}\label{lmm:indD01=indDR}
  We have $\ind D_{01} = \ind D_R$.
\end{lmm}
\begin{proof}
  We denote by $H_R^\red \subset H_R$ and $H_{01}^\red \subset H_{01}$
  the subspaces of functions with vanishing asymptotics, \ie $\xi \in
  H_R$ with $\xi(\pm \infty)=0$ and $(\xi_0,\xi_1) \in H_{01}$ with
  $\xi_0(\pm \infty) =\xi_1(\pm \infty)=0$. Secondly denote $D^\red_R$
  and $D^\red_{01}$ the operators $D_R$ and $D_{01}$ restricted to the
  spaces $H_R^\red$ and $H_{01}^\red$ respectively. We have
  \begin{equation}\label{eq:D01DR}
    \begin{aligned}
      \ind D_{01} &= \ind D_{01}^\red + \dim C_- + \dim C + \dim C_+\,,\\
      \ind D_R &= \ind D_R^\red + \dim C_- + \dim \ker C_+\,.
    \end{aligned}
  \end{equation}
  The indices of the reduced Fredholm operators  are computed in
  Lemma~\ref{lmm:index}. We have
  \begin{align*}
    \ind D_{01}^\red &= \ind D_0^\red + \ind D_1^\red\\
    &=\mu(\Psi\R^n,\R^n) - \mu(\Psi_-\R^n,\R^n) - \frac 12 \dim C_- - \frac 12 \dim C\\
    &\ + \mu(\Psi_+ \R^n,\R^n) - \mu(\Psi\R^n,\R^n) - \frac 12 \dim C -
    \frac 12 \dim C_+\,,
\end{align*}
where $\Psi_-,\Psi$ and $\Psi_+$ are the fundamental solutions of
$\sigma_-$, $\sigma$ and $\sigma_+$ respectively. Thus
\begin{equation*}
\ind D_{01}^\red = \mu(\Psi_+\R^n,\R^n) - \mu(\Psi_-\R^n,\R^n) - \frac 12 \dim C_- -\dim C - \frac 12 \dim C_+\,.
\end{equation*}
On the other hand we have
\[\ind D^\red_R = \mu(\Psi_+\R^n,\R^n) - \mu(\Psi_-\R^n,\R^n) - \frac 12 \dim C_- - \frac 12 \dim C_+\,.\]
Plugging the last two equations in~\eqref{eq:D01DR} proves the lemma.
\end{proof}
\subsubsection{Adapted norms} For $R>0$, we define a weight function
$\gamma:\R \to \R$
\[
\gamma(s) =
\begin{cases}
  e^{-\delta(2R+s)} &\text{if } s<-2R\\
  e^{\delta(2R-\nm{s})} &\text{if } \nm{s} <2R\\
  e^{\delta(s-2R)} &\text{if } s>2R\,.
\end{cases}
\]
For any $\eta \in L_R:= L^{2;\delta}(\Sigma,\R^{2n})$ we defined the
weighted norm 
\[
\Nm{\eta}_{L_R} := \left(\int_\Sigma \nm{\eta}^2
  \gamma^2_{\delta,R} \d s \d t\right)^{1/2}
\]
and  for any $\xi \in H_R$ with $\xi_\pm = \xi(\pm \infty)$ and $\bar \xi=P
\xi(0,\cdot)$, where $P$  denotes the
orthogonal projector to $\ker A$, we define the norm 
\begin{multline*}
  \Nm{\xi}_{H_R} := \Nm{\xi_-} + \Nm{\bar \xi} +\Nm{\xi_+} +\\+
  \Nm{(\xi-\xi_-)\gamma}_{1,2;\Sigma_{-\infty}^{-2R}} + \Nm{(\xi-\bar
    \xi)\gamma}_{1,2;\Sigma_{-2R}^{2R}} + \Nm{(\xi-\xi_+)\gamma}_{1,2;\Sigma_{2R}^\infty}\,.
\end{multline*}
It is easy to see that these define equivalent norms for every fixed
$R\geq 1$.

\subsubsection{Pregluing and breaking} Fix once and for all a cut-off
function $\beta^+$ and $\beta^-$ given in~\eqref{eq:cutoff}. Further
denote $\beta^-_R = \beta^- \circ \tau_R$ and $\beta^+_R=\beta^+
\circ \tau_R$ we define the \emph{linear pregluing operator} via
\begin{equation*}
  \Theta_R:H_{01} \to H_R,\qquad (\xi_0,\xi_1) \mapsto \bar \xi +
  \beta_{-R}^+ (\xi_1
  \circ \tau_{2R} - \bar \xi)+\beta_R^- (\xi_0\circ \tau_{-2R} - \bar \xi)  \,,
\end{equation*}
with $\bar \xi = \xi_0(\infty)=\xi_1(-\infty)$ and the
\emph{breaking operator} $\Xi_R:L_R \to L_0 \oplus L_1$, $\eta \mapsto
(\eta_{0,R},\eta_{1,R})$ in which
\begin{align*}
\eta_{0,R}(s,t) &=
\begin{cases}
  0 &\text{for } s \geq 2R\\
  \eta(s-2R,t)&\text{for } s \leq 2R
\end{cases}\\ \\
 \eta_{1,R}(s,t) &=
\begin{cases}
  \eta(s+2R,t) &\text{for } s \geq -2R \\
  0 &\text{for } s \leq -2R\,.
\end{cases}
\end{align*}
We define  another \emph{linear pregluing
  operator}  $\Omega_R:L_0 \oplus
L_1 \to L_R$, $(\eta_0,\eta_1)\mapsto \eta_R$ in which
\[
\eta_R =
\begin{cases}
  \eta_0(s+2R,t) &\text{for } s \leq 0\\
  \eta_1(s-2R,t) &\text{for } s \geq 0\,.
\end{cases}
\]
It is easily seen that $\Xi_R$ is a right-inverse for $\Omega_R$.
\begin{lmm}\label{lmm:gluebreakest}
  There exists constants $c$ and $R_0$ such that for all
  $(\xi_0,\xi_1) \in H_{01}$ and $R \geq R_0$ we have
  \[
  \Nm{\Theta_R(\xi_0,\xi_1)}_{H_R} \leq c
  \left(\Nm{\xi_0}_{1,2;\delta} + \Nm{\xi_1}_{1,2;\delta}\right)\,.
  \] Moreover we have for all $\eta \in L_R$, 
  \[
  \Nm{\eta_{0,R}}^2_{2;\delta} + \Nm{\eta_{0,R}}^2_{2;\delta} =
  \Nm{\eta}^2_{L_R}\,,
  \]
  in which $(\eta_{0,R},\eta_{1,R}) = \Xi_R(\eta)$.
\end{lmm}
\begin{proof}
  We have the same estimates as in the proof of
  Lemmas~\ref{lmm:estIR} and~\ref{lmm:estBR}. Note that we are
  now in a much simpler situation where the connection is flat
  and all the parallel transport maps are given by the identity.
\end{proof}

\subsubsection{Approximate pseudo-right inverse} Fix a finite dimensional
subspace $Y \subset L_{01}$ such that $D_{01}$ is transverse to $Y$,
\ie $\im D_{01} + Y = L_{01}$.  Without loss of generality we assume
that all functions in $Y$ are compactly supported. Define $X :=
D_{01}^{-1} Y$.
\begin{lmm}\label{lmm:preglueinj}  The linear pregluing operators are injective when restricted to $X$
  or $Y$ respectively, for all $R$ sufficiently large.
\end{lmm}
\begin{proof}
  We find $s_0$ such that functions in $Y$ are supported in
  $[-s_0,s_0]\times [0,1]$. The fact that $\Omega_R|_{Y}:Y \to L_R$ is
  injective for all $R >s_0$ directly follows by its definition.
  Suppose that $(\xi_0,\xi_1)\in X$. Then $D_0 \xi_0$ is supported
  inside $[-s_0,s_0]\times [0,1]$ and thus by elliptic regularity
  $\xi_0$ must be constant outside. Similarly for $\xi_1$. Yet if
  $\Theta_R(\xi_0,\xi_1)=0$ for $R > s_0$, then $\xi_0$ and $\xi_1$
  have to vanish.
\end{proof}
Let $X^\perp \subset H_{01}$ be some closed complement of $X$ and
define $Y^\perp:= D_{01}(X^\perp)$. Obviously we have a splitting
$L_{01}=Y \oplus Y^\perp$ and since $D_{01}$ restricted to $X^\perp$
is injective, there exists a unique bounded operator $Q_{01}:L_{01}\to
H_{01}$ satisfying
\begin{equation*}
  \im Q_{01} = X^\perp,\qquad \ker Q_{01} = Y,\qquad
  D_{01}Q_{01}\eta =\eta \quad \forall\ \eta \in Y^\perp\,.  
\end{equation*}
We define the \emph{approximate pseudo-right inverse}
\[
\wt Q_R:= \Theta_R \circ Q_{01} \circ \Xi_R : L_R \to H_R\,.
\]
Moreover we define subspaces of $H_R$ and $L_R$ by
\begin{align*}
X_R&:=\Theta_R(X)\,,&X_R^\perp &:=
\Theta_R(X^\perp),\\
Y_R&:=\Omega_R(Y)\,, & Y_R^\perp &:= \Omega_R(Y^\perp)\,.
\end{align*}
By Lemma~\ref{lmm:preglueinj} and since the linear pregluing
maps are surjective we have splittings $H_R = X_R \oplus X_R^\perp$
and $L_R = Y_R \oplus Y_R^\perp$.
\begin{lmm}\label{lmm:Qest}
  We have $\im \wt Q_R = X^\perp_R$, $\ker \wt Q_R = Y_R$ and there
  exists constants $c$ and $R_0$ such that for all $R \geq R_0$ we
  have
  \[
  \Nmm{\wt Q_R \eta}_{H_R} \leq c
  \Nm{\eta}_{L_R},\qquad \Nmm{ D_R \wt
    Q_R\eta  - \eta}_{L_R} \leq ce^{-\delta R} \Nm{\eta}_{L_R}\,.
  \]
  for all $\eta \in Y^\perp_R$.
\end{lmm}
\begin{proof}
  The statements about the kernel and the image directly follow from
  the definition. The first estimate follows directly from
  Lemma~\ref{lmm:gluebreakest}. We turn to prove the second
  estimate. Follow the lines of the proof of Lemma~\ref{lmm:approxQ}
  to show that for all $\xi \in H_{01}$ we have
  \[
  \Nm{D_R \Theta_R \xi - \Omega_R D_{01} \xi}_{L_R} \leq O(e^{-\delta
    R}) \Nm{\xi}_{H_{01}}\,.
  \]
  We abbreviate the norm $\Nm{\cdot} = \Nm{\cdot}_{L_R}$.  For any
  $\eta \in Y^\perp_R$ we have
  \begin{equation*}
    \Nmm{D_R \wt Q_R \eta  - \eta} \leq \Nmm{
      \Omega_R D_{01}Q_{01} \Xi_R \eta -\eta} + O(e^{-\delta
      R}) \Nm{\eta}\,.
  \end{equation*}
  Decompose $\Xi_R \eta = \eta_0 + \eta_1$ along the splitting $Y
  \oplus Y^\perp$ and continue using the fact that $Q_{01}$ is a
  pseudo-inverse we see
  \begin{equation*}
    \Nmm{D_R \wt Q_R \eta - \eta} \leq \Nmm{ \Omega_R \eta_1 -\eta} +
    O(e^{-\delta R})
    \Nm{\eta}    \leq O(e^{-\delta R}) \Nm{\eta}\,.
  \end{equation*}
  The term $\Omega_R \eta_1 -\eta$ vanishes because as $\Xi_R$ is a
  right-inverse to $\Omega_R$ we have $\eta = \Omega_R \Xi_R \eta =
  \Omega_R \eta_0 + \Omega_R \eta_1 = \Omega_R \eta_1$ since by
  assumption $\eta \in Y^\perp_R$ and $\Omega_R \eta_0 \in Y_R$.
\end{proof}
\subsubsection{Gluing construction} Via~\eqref{eq:QR} we use the
approximate pseudo-inverse $\wt Q_R$ to define an actual pseudo
right-inverse $Q_R:L_R \to H_R$ which is uniformly bounded and
satisfies
\[
\im Q_R =X^\perp_R,\qquad \ker Q_R = Y_R,\qquad D_{R}Q_R\eta =
\eta,\quad \forall\ \eta \in Y^\perp_R\,.
\]
Moreover abbreviate $P_R= \one - Q_R D_R$ the projection onto
$D_R^{-1} Y_R$ along $\im Q_R$.
\begin{lmm}
  For all $R$ sufficiently large, the restriction $P_R|_{X_R}:X_R \to
  D_R^{-1}Y_R $ is an isomorphism.
\end{lmm}
\begin{proof}
  We write any $\eta \in L_R$ as $\eta = \eta_0 + \eta_1$ along the
  splitting $L_R=Y_R \oplus Y_R^\perp$. Then $\eta = \eta_0 + D_RQ_R
  \eta_1$. This shows that $\eta \in \im D_R + Y_R$, thus $D_R$ is
  transverse to $Y_R$. Hence the space $D_R^{-1}Y_R$ has dimension
  $\dim Y_R + \ind D_R$. Since pregluing is injective restricted to
  $Y_{01}$ (\cf Lemma~\ref{lmm:preglueinj}), the dimension of $Y_R$
  and $Y_{01}$ are the same. Similarly we conclude that $\dim
  X_R=\dim X_{01} = \dim Y_{01} + \ind D_{01}$. The index of $D_R$ and
  $D_{01}$ agree (\cf Lemma~\ref{lmm:indD01=indDR}). We conclude that
  the spaces $D_R^{-1}Y_R$ and $X_R$ have the same dimension.  It
  suffices to show that $P_R$ is injective.  The kernel of $P_R$ is
  the intersection $D^{-1}_R Y_R \cap \im Q_R$. Suppose by
  contradiction that there exists a non-trivial $\xi \in D^{-1}_R Y_R
  \cap \im Q_R$. Hence $\xi=Q_R \eta$ for some $\eta \in Y_R^\perp$
  with $D_R Q_R \eta \in Y_R$. By the properties of $Q_R$ we have
  $\eta =D_R Q_R \eta =0$.
\end{proof}
Associated to the canonical exact sequences $0 \to \ker D_{01} \to X
\stackrel{D_{01}}{\ra} Y \to \coker D_{01} \to 0$ and $0 \to \ker D_R
\to D_R^{-1} Y_R \stackrel{D_R}{\ra} Y_R \to \coker D_R \to 0$ are the
isomorphisms $|D_{01}| \cong |X| \otimes |Y|^\vee$ and $|D_R| \cong
|D_R^{-1} Y_R| \otimes |Y_R|^\vee$ respectively
via~\eqref{eq:oddeven}. We define the gluing isomorphism as the composition
\begin{equation}\label{eq:glueOrD01DR}
  \xymatrix{|D_{01}| \ar[r]&|X| \otimes |Y|^\vee \ar[rrr]^{|P_R \Theta_R| \otimes |\Omega_R|^\vee} &&&|D_R^{-1}Y_R| \otimes |Y_R|^\vee\ar[r]&  |D_R|}\,.
\end{equation}
In order to glue the operators $D_0$ and $D_1$ we use the following
simple algebraic lemma.
\begin{lmm}\label{lmm:glueOrD0D1D01}
  There exists an exact sequence
  \begin{multline*}
     0 \to \ker D_{01} \to \ker D_0 \oplus
  \ker D_1 \to \ker A \to \coker
  D_{01}\to \\\to \coker D_0\oplus \coker D_1 \to
  0\,,
  \end{multline*}
  which together with~\eqref{eq:glueOrD01DR} induces an isomorphism
  $|D_0| \otimes |D_1| \to |\ker A| \otimes |D_R|$.
\end{lmm}
\begin{proof}
  The claim follows directly from the snake lemma on the commutative
  diagram of short exact sequences.
  \[
  \xymatrix{0 \ar[r] & H_{01} \ar[r]\ar[d]^{D_{01}}& H_0\oplus H_1 \ar[r]\ar[d]^{D_0\oplus D_1} & \ker A \ar[r]\ar[d]&0\\
    0\ar[r]&L \oplus L \ar[r]^{=}\ar[r] &L \oplus L\ar[r]&0\,,}
  \]
  in which the map $H_0\oplus H_1 \to \ker A$ is given by
  $(\xi_0,\xi_1) = \xi_0(\infty)-\xi_1(-\infty)$.  By
  equation~\eqref{eq:oddeven} we have an isomorphism $|D_0| \otimes
  |D_1| \cong |\ker A| \otimes |D_{01}|$.
\end{proof}
\subsubsection{Naturality} We show that the orientation gluing
map~\eqref{eq:glueOrD01DR} is independent of choices and natural with
respect to homotopies (\cf Lemmas~\ref{lmm:extendglueori}
and~\ref{lmm:glueornat} respectively). Given two linear complements
$Y,\wh Y \subset L_{01}$ which are transverse to $D_{01}$ and denote
the corresponding maps from~\eqref{eq:glueOrD01DR} on the level of
determinant lines by $\psi_R, \wh \psi_R: \det D_{01} \to \det D_R $
respectively.
\begin{lmm}\label{lmm:extendglueori}
  For all $R$ sufficiently large, the composition 
  \[\psi^{-1}_R \circ
  \wh \psi_R:\det D_{01} \to \det D_{01}\,,\] is given by multiplication
  with a positive number.
\end{lmm}
\begin{proof}
  Without loss of generality $Y \subset \wh Y$. We denote the spaces
  and maps appearing in the construction using $\wh Y$ with $\wh X$,
  $\wh X^\perp$ etc. To define the maps $\psi_R$ and $\wh
  \psi_R$ we have the commutative diagram
  \begin{equation*}
    \xymatrix{&\det \wh X \otimes \det \wh Y^\vee \ar[rr] &&\det D_R^{-1} \wh Y_R \otimes \det \wh Y^\vee_R \ar[dr]\\
      \det D_{01} \ar[ur]\ar[dr] &&&& \det D_R\\ 
      &\det  X \otimes \det  Y^\vee \ar[uu]\ar[rr] &&\det D_R^{-1}  Y_R \otimes \det  Y^\vee_R \ar[ur]\ar[uu]}
  \end{equation*}
  The map $\psi_R$ is obtained by following the diagram along the
  lower arrows and we get the map $\wh \psi_R$ using the upper
  path. The two vertical arrows are computed in \cite[Lmm.\
  5.2]{AbbMajer:inf} and are given as follows: Consider a splitting
  $\wh Y = Y \oplus Y'$ and set $X':= D^{-1}_{01}(Y')$. We have a
  splitting $\wh X = X \oplus X'$ and $D_{01}|_{X'}:X' \to Y'$ is an
  isomorphism. Fix generators $\alpha \in \det X$, $\beta \in \det Y$
  and $\gamma \in \det X'$.  Then $D_{01} \gamma$ is a generator of
  $\det Y'$ and the map on the left-hand side is given by $\alpha
  \otimes \beta^\vee \mapsto (\alpha \wedge \gamma) \otimes (D_{01}
  \gamma^\vee \wedge \beta^\vee)$. On the other side consider the
  splitting $\wh Y_R = Y_R \oplus Y'_R$ with $Y_R := \Omega_R(Y)$ and
  $Y'_R := \Omega_R(Y')$. We have a corresponding splitting
  $D_R^{-1}\wh Y_R = D_R^{-1}Y_R \oplus D_R^{-1}Y'_R$, generators
  $\alpha_R \in \det D_R^{-1}Y_R$, $\beta_R \in \det Y_R$ and
  $\gamma_R \in \det D_R^{-1}Y_R'$ and the map is given by $\alpha_R
  \otimes \beta_R^\vee \mapsto (\alpha_R \wedge \gamma_R) \otimes (D_R
  \gamma_R^\vee \wedge \beta_R^\vee)$. We assume that the generators
  are picked such that $\alpha_R = P_R\Theta_R \alpha$,
  $\beta_R=\Omega_R(\beta)$ and $D_R \gamma_R =\Omega_R D_{01}\gamma$. We
  conclude by following the definition of the maps around the square,
  that $\psi_R^{-1}\psi_R$ is given by multiplication with the
  determinant of the map
  \[
   \wh P_R \Theta_R :X  \to D_R^{-1}\wh Y_R\,,
  \]
  where $\det X$ is oriented by $\alpha \wedge \gamma$ and $\det
  D_R^{-1} \wh Y_R$ is oriented by $\alpha_R \wedge \gamma_R$. 

  We claim that $(P_R - \wh P_R) \xi \in D_R^{-1}Y_R'$ for any $\xi
  \in X_R$. Indeed given $\xi \in X_R$ and split $D_R \xi = \eta_0 +
  \eta_1 +\eta_2$ along the splitting $Y_R\oplus Y'_R \oplus \wh
  Y_R^\perp$. Then since they are right-inverses
  \begin{multline*}
    D_R(P_R-\wh P_R) \xi = D_R(\wh Q_R-Q_R)D_R \xi = D_R \wh Q_R
  \eta_2 - D_RQ_R (\eta_1 + \eta_2)=\\ = \eta_2 -\eta_1 - \eta_2=
  \eta_1\,.
  \end{multline*}
  This shows the claim and we conclude that $\wh P_R \Theta_R
  \alpha\wedge \wh P_R \Theta_R \gamma= \alpha_R\wedge \wh P_R
  \Theta_R \gamma$. We are left to compute the number $a \in \R$ which
  is defined by
  \[\pi_{D_R^{-1}Y'_R} \wh P_R \Theta_R \gamma = a
  \cdot \gamma_R\,,\] in which $\pi_{D_R^{-1} Y'_R}: D^{-1}_R\wh Y_R
  \to D^{-1}_R Y'_R$ denotes the projection along $D^{-1}_R Y_R$. We
  apply $D_R$ on the last equation and obtain
  \[
  \pi_{Y'_R} D_R \wh P_R \Theta_R \gamma = a \cdot D_R \gamma_R = a
  \cdot \Omega_R D_{01} \gamma\,.
  \]
  Abbreviate the generator $\gamma'_R := \Omega_R D_{01}\gamma \in
  \det Y'_R$. The inverse of the map $\Omega_R D_{01}:X' \to Y'_R$ is
  $Q_{01} \circ \Xi_R$ and hence $\gamma = Q_{01} \Xi_R \gamma'_R$.
  Plugging that back into the last equation we conclude that $a$ is
  the determinant of the map
  \begin{equation*}
    \pi_{Y'_R} D_R \wh P_R\wt Q_R:Y'_R \to Y'_R\,.
  \end{equation*}
  To show that the determinant is positive for all $R$ sufficiently
  large enough it suffices to show the following: There are constants
  $c$ and $R_0$ such that for all $R \geq R_0$ and $\eta \in Y'_R$ we
  have
  \begin{equation*}
    \Nmm{D_R \wh P_R \wt Q_R \eta -
      \eta}_{L_R} \leq c e^{-\delta R}
    \Nm{\eta}_{L_R}\,.   
  \end{equation*}
  Choose any $\eta \in Y'_R$. Abbreviating $\Nm{\cdot} =
  \Nm{\cdot}_{L_R}$ we compute 
  \begin{align*}
    \Nmm{D_R \wh P_R \wt Q_R \eta-\eta}&=\Nmm{D_R(\one - \wh Q_RD_R)\wt Q_R \eta}\\
    &\leq \Nmm{D_R \wt Q_R \eta - \eta} + \Nmm{D_R \wh Q_R D_R \wt Q_R \eta}\\
    &=\Nmm{D_R \wt Q_R \eta - \eta} + \Nmm{D_R \wh Q_R (D_R \wt Q_R
      \eta -\eta)}\leq ce^{-\delta R} \Nm{\eta}\,.
  \end{align*}
  Using the fact that $\wh Q_R \eta =0$ and Lemma~\ref{lmm:Qest}.
\end{proof}
\begin{lmm}\label{lmm:glueornat}
  Given homotopies $(S^\tau_0)_{\tau \in [a,b]}$ and $(S^\tau_1)_{\tau
    \in [a,b]}$  with $S^\tau_0 \in
  \D(\sigma^\tau_-,\sigma^\tau)$, $S^\tau_1 \in
  \D(\sigma^\tau,\sigma^\tau_+)$ and
  $S^\tau_0(s,\cdot)=S^\tau_1(-s,\cdot)=\sigma^\tau$ for all $s \geq
  1$ and $\tau \in [a,b]$. For corresponding operators
  $D^\tau_{01}$ and $D^\tau_R$ and $R$ sufficiently large
    we have a commutative diagram
  \begin{equation*}
    \xymatrix{ |D^a_{01}|
      \ar[r]\ar[d]&|D^a_R|\ar[d]\\|D^b_{01}| \ar[r] &|D^b_R|\,}
  \end{equation*}
  where the horizontal arrows are given by~\eqref{eq:glueOrD01DR} and
  the vertical arrows are induced by the homotopies.
\end{lmm}
\begin{proof}
  Let $\tau \mapsto Y^\tau \subset L$ be a continuous path of subspaces
  such that for each $\tau \in [a,b]$ the space $Y^\tau$ is transverse
  to $D_{01}^\tau$. For $R$ sufficiently large $D_R^\tau$ is
  transverse to the glued space $Y^\tau_R$ for all $\tau \in [a,b]$.
  By definition of the gluing map we have a commutative diagram
  \[ \xymatrix{|D_{01}^\tau| \ar[r]\ar[d]&|
    (D_{01}^\tau)^{-1}Y^\tau|\otimes |Y^\tau|\ar[d]^{|P^\tau_R \Theta^\tau_R|\otimes |\Omega^\tau_R|}\\
    |D_R^\tau|\ar[r] &|(D^\tau_R)^{-1} Y^\tau_R| \otimes |
    Y^\tau_R|\,.}\] The spaces $(D_{01}^\tau)^{-1}(Y^\tau)$, $\wh
  Y_R^\tau$ \etc are the fibers of vector bundles over $[a,b]$ and the
  isomorphisms $P_R^\tau \Theta_R^\tau$ and $\Omega_R^\tau$ are bundle
  maps. By continuity we obtain a commutative diagram.
\end{proof}

\section{Pearl homology}
Pearl homology is a version of Floer homology of Lagrangian
intersection, which has the advantage that if Lagrangians intersect
cleanly we do not need to perturb the Lagrangians into transverse
position.  We call the invariant \emph{pearl homology} because it is a
direct generalization of an invariant associated to a single monotone
Lagrangian introduced by Biran and Cornea
in~\cite{BiranCornea:quantum} with that name. The construction was
already sketched out by Frauenfelder in~\cite[Appendix
C]{Frauenfelder:PhD} under the name of \emph{Floer-Bott
  homology}. Here we give a detailed account of the theory for the
monotone case including orientations. 
\subsection{Overview} We give a quick overview of the pearl homology
complex. All details are provided in the later subsections.
\subsubsection{Pearl trajectories} Choose an auxiliary Morse function
$f:L_0 \cap L_1 \to \R$ and metric on $L_0 \cap L_1$ and a path of
almost complex structures $J:[0,1]\to \End(TM,\omega)$. Given critical
points $p_-, p_+ \in \crit f$, a \emph{pearl trajectory connecting
  $p_-$ to $p_+$} is either a negative gradient flow trajectory $u:\R
\to L_0 \cap L_1$ with $u(-\infty) = p_-$ and $u(\infty) = p_+$ or a
tuple $u=(u_1,\dots,u_m)$ of non-constant finite energy
$J$-holomorphic strips $\R \times [0,1] \to M$ with boundary in
$(L_0,L_1)$ such that $u_1(-\infty) \in W^u(p_-)$, $u_m(\infty) \in
W^s(p_+)$ and for each $j=1,\dots,m-1$ there exists a negative
gradient flow line from $u_j(\infty)$ to $u_{j+1}(-\infty)$. For
reasons of transversality we require that each curve in the tuple is
not a reparametrization of another. The number $m$ is called the
number of \emph{cascades}.  If $u$ is an ordinary Morse flow line, we
say that $u$ has zero cascades.  We denote by $\M(p_-,p_+)$ the space
of all pearl trajectories connecting $p_-$ to $p_+$ modulo
reparametrization with an arbitrary number of cascades. If $J$ is
sufficiently generic every connected component of $\M(p_-,p_+)$ is a
manifold with corners and for $d \in \N_0$ we denote by
$\M(p_-,p_+)_{[d]}$ the union of all $d$-dimensional components.
\subsubsection{Grading} Let $\P$ be the space of paths $\gamma:[0,1] \to
M$ such that $\gamma(0) \in L_0$ and $\gamma(1)\in L_1$. We denote by
$N \in \N$ the minimal Maslov number of the pair $(L_0,L_1)$ with
respect to a fixed element $\base \in \P$. For every critical point $p
\in \crit f$ we choose a map $u_p:[-1,1]\times [0,1] \to M$ such that
$u_p(s,\cdot) \in \P$ for all $s \in [-1,1]$, $u_p(-1,\cdot) =\base$
and $u_p(1,\cdot) \equiv p$.  The grading of a critical point $p \in
\crit f$ is
\begin{equation}
  \label{eq:gradp}
  |p| = \mu(p)-\mu(u_p) - \frac 12 \dim T_p L_0 \cap L_1\,.  
\end{equation}
\subsubsection{Orientation} Let $\O\to L_0 \cap L_1$ be the double cover
associated to a fixed relative spin structure (\cf
Definition~\ref{dfn:O}). For any critical point $p \in \crit f$ fix an
orientation $o_p \in |T_{p} W^u(p)| \otimes \O_p$. In the paragraph
before Lemma~\ref{lmm:coherent} we define orientations on
$\M(p_-,p_+)$ with these choices. An orientation of $[u] \in
\M(p_-,p_+)_{[0]}$ is just a number in $\{\pm 1\}$ which we denote by
$\sign(u)$.
\subsubsection{Pearl complex} Let $\Lambda:=A[\lambda^{-1},\lambda]$ be
the ring of Laurent polynomials in one variable of degree $-N$. The
\emph{pearl chain complex} $CH_*(L_0,L_1)$ is given as the free
$\Lambda$-module generated by all critical points of $f$ with grading
$|p \otimes \lambda^k| = |p| - k N$ and equipped with the
$\Lambda$-linear homomorphism 
\begin{equation}
  \label{eq:bdpearl}
\begin{gathered}
    \partial:CH_*(L_0,L_1) \to CH_{*-1}(L_0,L_1)\\
    p\mapsto  \sum_{q \in \crit f}\nolimits
    \sum_{[u] \in \M(p,q)_{[0]}} \nolimits \sign u \cdot q \otimes
    \lambda^{(\nm{q}-\nm{p}+1)/N}\,.
\end{gathered}  
\end{equation}
The next theorem as proven Fukaya \ea in~\cite{FO3:integers} for the very
general case of semi-positive symplectic manifolds and unobstructed
Lagrangians, which includes the case of monotone Lagrangians. 
\begin{thm}\label{thm:pss} 
  We have $\partial\circ \partial = 0$. The homology group
  $QH_*(L_0,L_1) = \ker \partial/\im \partial$ is called \emph{pearl
    homology} and is independent of choices of $J$, $f$, the metric
  and orientations $o_p$. Moreover we have a natural isomorphism 
  \begin{equation}
    \label{eq:pss}
    QH_*(L_0,L_1) \cong QH_*(L_0,\vp_H(L_1))\,.   
  \end{equation}
  for any Hamiltonian $H$.
\end{thm}
By the invariance we conclude that pearl homology is isomorphic to
Floer homology. Namely if we choose $H$ such that $\vp_H(L_1)$
intersects $L_0$ transversely, then pearl homology for $L_0$ and
$\vp_H(L_1)$ agrees with Floer homology by definition. In the rest of
the section we prove Theorem~\ref{thm:pss} for our present situation
of Lagrangians that satisfy the monotonicity assumption.

\subsection{Pearl complex}
In the following we abbreviate $\I:= L_0 \cap L_1$ and fix an
auxiliary Morse function $f$ on $\I$. We denote by $\psi:\R \times \I
\to \I$, $\psi^a:= \psi(a,\cdot)$ be the negative gradient flow of
$f$ with respect to a sufficiently generic metric. Note that in
general $\I$ has many connected components with possibly different
dimension. Pick a path of almost complex structures $J:[0,1]
\to \End(TM,\omega)$. Given an integer $m \in \N$ and submanifolds
$W_-, W_+ \subset \I$ we define
\begin{equation}
  \label{eq:Mpearl}
  \wt \M_m(W_-,W_+;J) := \left\{ (u_1,\dots,u_m) \subset C^\infty(\Sigma,M) \left|
    \text{ \ref{nm:Ma} -- \ref{nm:Mc}}\right\}\right.\,,
\end{equation}
to be the space of tuples $(u_1,\dots,u_m)$ such that 
\begin{enumerate}[label=\alph*)]
\item\label{nm:Ma} for all $j=1,\dots,m$ the map $u_j$ is
  $J$-holomorphic with boundary in $(L_0,L_1)$,
\item the tuple $(u_1,\dots,u_m)$ is distinct, \ie for all $i \neq j$, $a \in \R$ we have $u_i \circ \tau_a \neq u_j$,
\item for all $j=1,\dots,m-1$ there exists $a_j \geq 0$ such that 
  \begin{equation}
    \label{eq:matching}
    \psi^{a_j}(u_j(\infty)) = u_{j+1}(-\infty)\,,
  \end{equation}
\item\label{nm:Mc} we have $u_1(-\infty) \in W_-$ and $u_m(\infty) \in W_+$.
\end{enumerate}
For $m=0$ we define $\wt \M_0(W_-,W_+;J) := W_- \cap W_+$. The
elements of the space $\wt \M_m(W_-,W_+;J)$ are called
\emph{parametrized $J$-holomorphic pearl trajectory with $m$ cascades
  connecting $W_-$ to $W_+$ and boundary in $(L_0,L_1)$}. If in
particular $W_- = W^u(p_-)$ and $W_+=W^s(p_+)$ for some critical
points $p_-,p_+ \in \crit f$ we abbreviate furthermore \[\wt
\M_m(p_-,p_+;J):=\wt \M_m(W^u(p_-),W^s(p_+);J)\,.\] The elements of
$\wt \M_m(p_-,p_+;J)$ are simply called \emph{pearl trajectories with
  $m$ cascades connecting $p_-$ to $p_+$}.  As we shall see in a
moment the connected components of these spaces are manifolds with
corners if $J$ is chosen sufficiently generic.
\subsubsection{Transversality}
For the correction description of the space $\wt \M_m(W_-,W_+;J)$ we
need the notion of a \emph{manifold with corners}, which is a mild
generalization of the notion of manifolds with boundary. Unfortunately
there is no standard concept in the mathematical literature. We stick
with the definition of~\cite{Joyce:corners}. Briefly a manifold with
corners is a topological space $\M$ equipped with an atlas of charts
locally modeled on open subsets in $[0,\infty)^k\times \R^{n-k}$ and
chart transition maps which extend to smooth maps from $\R^n$ to
$\R^n$. The \emph{dimension} of the manifold is the number $n$.  The
\emph{depth} of a point is the number of zeros among the first $k$
coordinates in a chart. The depth is well-defined independently of
the choice of local coordinates and gives rise to a stratification of
$\M$. The \emph{top stratum} is given as the space of all points with
depth equal to zero. Obviously the each stratum is a manifold in the
usual sense. For more details see~\cite{Joyce:corners}.

Fix an integer $m \in \N$.  Let $\wt \M_m(J)$ denote the space of
distinct $m$-tuples of $J$-holomorphic curves and consider the
evaluation map
\begin{gather*}
  \ev:\wt
  \M_m(J) \to \I^{2m}\\
  (u_1,\dots,u_m) \mapsto
  (u_1(-\infty),u_1(\infty),u_2(-\infty),\cdots,u_m(\infty))
\end{gather*}
On the other hand consider the map 
\begin{equation}
  \label{eq:psipearl}
  \begin{gathered}
    W_- \times \I^{m-1} \times W_+ \times \R^{m-1} \to \I^{2m}\,,\\
    (p_0,\dots,p_m,a_1,\dots,a_{m-1}) \mapsto
    \begin{aligned}
    (p_0,p_1,\psi^{a_1}(p_1),p_2,\psi^{a_2}(p_2),&\dots\\ &\hspace{-3cm}\dots ,p_{m-1},\psi^{a_{m-1}}(p_{m-1}),p_m)\,    
    \end{aligned}
  \end{gathered}
\end{equation}
We say that $J$ is \emph{regular for $W_-$ and $W_+$} if $J$ is
regular for $X \equiv 0$ and the map~\eqref{eq:psipearl} in the sense
of Definition~\ref{dfn:regs}, \ie the operator $D_{u,J}$ is surjective
for all $J$-holomorphic strips $u$ and the evaluation is transverse
to~\eqref{eq:psipearl}.
\begin{lmm}\label{lmm:regfreecasc}
  The subspace of $J \in C^\infty([0,1],\End(TM,\omega))$ which are
  regular for $W_-$ and $W_+$ is comeager. If $J$ is regular then each
  connected component of the space $\wt \M_m(W_-,W_+;J)$ is a manifold
  with corners and the component which contains $u=(u_1,\dots,u_m)$
  has dimension
  \[
  \mu_\Vit(u) + \dim W_- - \frac 12 \dim C_- + \dim W_+ -\frac 12 \dim
  C_+ + m -1\,,
  \]
  in which $\mu_\Vit(u):=\sum_{j=1}^m \mu_\Vit(u_j)$ and $C_- \subset
  \I$, $C_+ \subset \I$ are the connected components containing $W_-$
  and $W_+$ respectively.
\end{lmm}
\begin{proof}
  In Theorem~\ref{thm:sreg} we show that the space of regular
  structures is comeager. Each connected component of the space $W_-
  \times \I^{m-1} \times W_+ \times [0,\infty)^{m-1}$ is clearly a
  manifold with corners. We see that $\wt \M_m(W_-,W_+;J)$ is the
  fibre product of the evaluation map with the map~\eqref{eq:psipearl}
  restricted to the subspace $W_- \times \I^{m-1} \times W_+ \times
  [0,\infty)^{m-1}$. Hence by~\cite[Thm.\ 6.4]{Joyce:corners}
  connected components of $\wt \M_m(W_-,W_+J)$ are also manifolds with
  corners.  To compute the dimension, choose some tuple
  $(u_1,\dots,u_m) \in \wt \M_m(W_-,W_+;J)$ in the top stratum, which
  is equivalent to say that the tuple of non-negative numbers
  $(a_1,\dots,a_{m-1})$ defined by~\eqref{eq:matching} has no
  zeros. Let $C_0,C_1,C_2,\dots,C_m \subset \I$ be the connected
  components of the points
  $u_1(-\infty),u_1(\infty),u_2(\infty),\dots,u_{m}(\infty)$
  respectively.  With the dimension formula from
  Theorem~\ref{thm:sreg} we have
  \begin{equation*}
    \dim T_u \wt \M_m(J) = \sum_{j=1}^m \mu_\Vit(u_j) + \frac 12
    \dim C_{j-1} + \frac 12 \dim C_j\,.
  \end{equation*}
  By the exact sequence~\eqref{eq:fibresequence} we conclude that
  the dimension $d$ of the fibre product $\wt \M_m(C_-,C_+;J)$ at $u$
  is given by 
  \begin{align*}
    d&= \dim T_u \M_m(J) + \dim W_- \times \prod_{j=1}^{m-1} C_j \times W_+ \times \R^{m-1} - \dim \prod_{j=0}^{m-1} C_j \times C_{j+1}\\
    &= \mu(u) + \frac 12 \dim C_0 + \sum_{j=1}^{m-1} \dim C_j +
    \frac 12 \dim C_m + \dim W_- + \dim W_+ + m -1 -\\
   &\qquad-\sum_{j=0}^m \dim
    C_j\\
    &= \mu(u) + \dim W_- - \frac 12 \dim C_0 + \dim W_+ -\frac 12
    \dim C_m + m -1\,.
  \end{align*}
  This shows the claim.
\end{proof}
From now on we fix an almost complex structure $J$, which is
sufficiently generic in the sense that it is regular with respect to
all upcoming pairs of submanifolds and omit the reference to $J$
whenever convenient, \eg we write $\wt \M(p,q)$ to denote $\wt
\M(p,q;J)$.
\subsubsection{Compactness}
A \emph{broken $J$-holomorphic pearl trajectory connecting $p_-$ to
  $p_+$} is a tuple of pearl trajectories $v=(v_1,\dots,v_k)$ such
that $v_i$ connects $p_i$ to $p_{i+1}$ for all $i=1,\dots,k-1$ and
some critical points $p_-=p_1,\dots,p_k=p_+$.  For $i=1,\dots,k$ we
denote by $m(v_i)$ the number of cascades of $v_i$.
\begin{dfn}\label{dfn:Gromovpearl}
  We say that a sequence of pearl trajectories
  $(u^\nu_1,\dots,u^\nu_m)_{\nu\in\N}$ \emph{Floer-Gromov converges}
  to the broken pearl trajectory $v=(v_1,\dots,v_k)$ if
  \begin{itemize}
  \item for each $j=1,\dots,m$ the sequence $(u^\nu_j)_{\nu\in\N}$
    Floer-Gromov converges to $w_j=(w_{j,1},\dots,w_{j,k_j})$ (\cf
    Definition~\ref{dfn:comp})
\item for each $(i,j)$ with $1\leq i \leq k$ and $1 \leq j \leq
  m(v_i)$ there exist a pair $(\ell,\kappa)$ such that
  $v_{i,j}=w_{\ell,\kappa}$,
\item the map $\{(i,j) \mid 1\leq i \leq k, 1\leq j\leq m(v_i)\}\to
  \{(\ell,\kappa)\mid 1\leq \ell\leq m, 1\leq \kappa\leq k_\ell\}$,
  $(i,j) \mapsto (\ell,\kappa)$ mentioned above is surjective and
  strictly monotone with respect to the lexicographic order.
  \end{itemize}
\end{dfn}
\begin{lmm}\label{lmm:compbubcasc}
  Assume that $L_0$ and $L_1$ are monotone with minimal Maslov number
  at least three.  Let
  $(u^\nu)_{\nu\in\N}=(u^\nu_1,\dots,u^\nu_m)_{\nu\in\N}$ be a
  sequence of pearl trajectories connecting $p_-$ to $p_+$ such that
  $\sup_\nu E(u^\nu)<\infty$, then at least one of the following
  holds:
  \begin{enumerate}[label=(\roman*)]
  \item a subsequence of $(u^\nu)$ Floer-Gromov converges to a broken
    pearl trajectory connecting $p_-$ to $p_+$,
  \item there exists a broken pearl trajectory $v$ connecting $p_-$ to
    $p_+$ and $\mu(u^\nu) \geq \mu(v) +3$ for all $\nu$
    sufficiently large.
  \item $p_-=p_+$ and $ \mu(u^\nu) \geq 3$ for all $\nu$ sufficiently
    large.
  \end{enumerate}
\end{lmm}
\begin{proof}
  For each $j=1,\dots,m$ we obtain by Theorem~\ref{thm:comp} a
  subsequence of $(u_j^\nu)$ still denoted by the same sequence, which
  Floer-Gromov converges modulo bubbling to the broken strip
  $w_j=(w_{j,1},\dots,w_{j,k_j})$ possibly containing constant
  components. We also have $u^\nu_j(-\infty) \to w_{j,1}(-\infty)$ and
  $u^\nu_j(\infty)\to w_{j,k_j}(\infty)$. Let $(b^\nu_j) \subset \R$
  be the sequence such that
  $\psi^{b^\nu_j}u_j^\nu(\infty)=u_{j+1}^\nu(-\infty)$. We distinguish
  two cases. In the first case $(b^\nu_j)_{\nu\in \N}$ is bounded,
  then a subsequence converges to $b_j$ and we have
\begin{equation}
    \label{eq:matchv}
    \psi^{b_j}(w_{j,k_j}(\infty)) = w_{j+1,1}(-\infty)\,.
  \end{equation}
  In the second case $(b^\nu_j)_{\nu\in \N}$ is unbounded. For each
  $\nu \in \N$ let $\gamma^\nu_j$ be the Morse trajectory from
  $u^\nu_j(\infty)$ to $u^\nu_{j+1}(-\infty)$.  Then a subsequence of
  $(\gamma^\nu_j)_{\nu\in\N}$ converges to a broken Morse trajectory
  $(\gamma_{j,1},\gamma_{j,2},\dots,\gamma_{j,\ell_j})$ with
  $\ell_j\geq 2$, where $\gamma_{j,1}$ and $\gamma_{j,\ell_j}$ are
  half-trajectories. Set $p_j^-:=\gamma_{j,1}(\infty)$ and
  $p_j^+:=\gamma_{j,\ell_j}(-\infty)$. In particular if $j>2$ the
  critical points $p_j^-,p_j^+$ are joined by the broken Morse
  trajectory $(\gamma_{j,2},\dots,\gamma_{j,\ell_j-1})$ and if
  $\ell_j=2$ they are equal $p_j^-=p_j^+$. In any case we have
  $w_{j,k_j}(\infty) \in W^s(p^-_j)$ and $w_{j+1,1}(-\infty) \in
  W^u(p^+_j)$. Regrouping the non-constant components of the tuples
  $(w_{j,i})$ and $(\gamma_{j,2},\dots,\gamma_{j,\ell_j-1})$
  using~\eqref{eq:matchv} shows they constitute to a broken pearl
  trajectory $v$ which connects $p_-$ to $p_+$. Moreover by
  Lemma~\ref{lmm:bubblemonotone} either each all $w_{j,i}$ are
  non-constant and $(u^\nu)$ Floer-Gromov converges to $v$ or
  $\mu(u^\nu) \geq \mu(v) + 3$ or if all components had been discarded
  $\mu(u^\nu) \geq 3$ for all $\nu$ large enough.
\end{proof}
If $m \geq 1$ the reparametrization group of $\wt \M_m(p_-,p_+)$ is
$\R^m$ which acts freely via $(a_1,\dots,a_m).(u_1,\dots,u_m)= (u_1
\circ \tau_{a_1},\dots,u_m \circ \tau_{a_m})$ with $\tau_a:\R\times
[0,1] \to \R \times [0,1], (s,t) \mapsto (s-a,t)$ for $a \in \R$. If
$m=0$ and $p_- \neq p_+$ there is a free action of $\R$ on $\wt
\M_0(p_-,p_+)$ given by the negative gradient flow. In any case we
denote by $\M_m(p_-,p_+)$ the quotient and moreover
\[
\M(p_-,p_+) := \bigcup_{m \in \N_0} \M_m(p_-,p_+)\,.
\]
For any $d \in \N_0$ we denote by $\M(p_-,p_+)_{[d]} \subset
\M(p_-,p_+)$ the union of components with dimension $d$. Further let
\[\M^\ell_m(p_-,p_+) \subset \M_m(p_-,p_+)\,,\] be the subspace of points with
depth $\ell=0,\dots,m$, \ie given by equivalence classes of pearl
trajectories such that there are exactly $\ell$ zeros in the tuple
$(a_1,\dots,a_{m-1})$ defined by~\eqref{eq:matching}.
\begin{cor}\label{cor:compcasc}
  For all critical points $p,q \in \crit f$, the space $\M(p,q)_{[0]}$
  is finite and the boundary of the Floer-Gromov compactification of
  $\M_m(p,q)_{[1]}$ is given by the union of
  \begin{itemize}
  \item $\M_m^1(p,q)_{[0]}$ ,
  \item $\M^1_{m+1}(p,q)_{[0]}$ ,
  \item $ \M_\ell(p,r)_{[0]} \times \M_k(r,q)_{[0]}$ for all $r\in
    \crit f$ and $\ell,k \in \N_0$ with $\ell+k=m$.
  \end{itemize}
\end{cor}
\begin{proof}
  Fix $d$ and let $(u^\nu) =(u^\nu_1,\dots,u^\nu_{m^\nu}) \in \wt
  \M_{m^\nu}(p_-,p_+)_{[d+m^\nu]}$ be a sequence of pearl
  trajectories. Since every non-constant holomorphic strip carries a
  minimal energy (\cf Proposition~\ref{prp:hbarstrip}) we have an uniform
  constant $\hbar>0$ such that $E(u^\nu_j) = \int (u^\nu_j)^*\omega
  >\hbar$. From the dimension formula and the action-index relation
  (\cf Lemma~\ref{lmm:AIrel}) we conclude that there exists a
  constant $c$ such that
  \begin{equation*}
    d= \sum_{j=1}^{m^\nu} \mu(u^\nu_j) + \mu(p_-) -\mu(p_+) -
    \frac 12 \dim C_- + \frac 12 \dim C_+ -1 > \tau^{-1} m^\nu \hbar + c\,.
  \end{equation*}
  Hence the sequence $(m^\nu)$ is bounded and after possibly passing
  to a subsequence we assume without loss of generality that $m^\nu=m$
  for all $\nu \geq 1$. By the same token we conclude that $d \geq
  \tau^{-1}\sum_{j=1}^m \int E(u^\nu_j) + c$ and that the energy of
  $(u^\nu_j)$ is uniformly bounded for all $j =1,\dots,m$. We apply
  Lemma~\ref{lmm:compbubcasc} and assume by contradiction that the
  second case holds, \ie we obtain a broken pearl trajectory
  $v=(v_1,\dots,v_k)$ which connects $p_-$ to $p_+$ such that
  $\mu(u^\nu) \geq \mu(v) + 3$. Let $p_1,\dots,p_{k-1}$ (\resp
  $C_1,\dots,C_{k-1}$) be critical points (\resp connected components)
  such that $v_\ell$ connects $p_{\ell-1}$ to $p_\ell$ (\resp
  $C_{\ell-1}$ to $C_\ell$) for each $\ell=1,\dots,k-1$. The dimension
  formula for $v_\ell$ implies
  \begin{equation}
    \label{eq:dimcascghost}
    \mu(v_\ell) \geq \mu(p_\ell) -\mu(p_{\ell-1}) + \frac 12 \dim C_\ell
    -\frac 12 \dim C_{\ell-1}+1\,.   
  \end{equation}
  Note that even if $v_\ell$ is not in the top stratum, the
  index is still bigger or equal to the number of
  cascades of $v_\ell$. By the dimension formula for $u^\nu$ and last
  estimate
  \begin{align*}
    d &= \sum_{j=1}^m \mu(u^\nu_j) + \mu(p_-) - \mu(p_+) - \frac 12
    \dim C_- + \frac 12 \dim C_+-1\\
    &\geq \sum_{\ell=1}^k \mu(v_\ell)  +3 + \mu(p_-) - \mu(p_+) - \frac 12 \dim C_- + \frac 12 \dim C_+ - 1\\
    &\geq k + 3 -1 \geq 3\,.
  \end{align*}
  Hence if $d=0,1$ the second case of Lemma~\ref{lmm:compbubcasc} is
  impossible. Now assume that the third case holds, \ie $p_-=p_+$,
  $C_-=C_+$ and $\mu(u^\nu)\geq 3 $. By the same estimate we have $d =
  \mu(u^\nu) -1 \geq 2$, which is again impossible for $d=0,1$. The first
  case implies that a subsequence of $(u^\nu)$ converges to the broken
  pearl trajectory $v$ connecting $p_-$ to $p_+$ with $\mu(u^\nu)
  =\mu(v)$. By the same estimate as above we have
  \begin{align*}
    d &= \sum_{j=1}^m \mu(u^\nu_j) + \mu(p_-) - \mu(p_+) - \frac 12
    \dim C_- + \frac 12 \dim C_+-1\\
    &= \sum_{\ell=1}^k \mu(v_\ell)  + \mu(p_-) - \mu(p_+) - \frac 12
    \dim C_- + \frac 12 \dim C_+-1 \geq k-1\,.
  \end{align*}
  We conclude that if $d=0$, then $k=1$ and the unparametrized curve
  of $v_1$ is in fact an element of $\M(p_-,p_+)_{[0]}$. If $d=1$,
  then possibly $k=2$ and hence is an element of the given
  boundary. To show that every element appears as the boundary
  consider the proof of Lemma~\ref{lmm:coherent} below.
\end{proof}
\subsubsection{Orientations} To define homology with integer
coefficients or more generally with coefficients in a ring of
characteristic $\neq 2$, we need to orient the moduli spaces. As
already mentioned the orientations should be compatible with gluing
and breaking (\ie coherent in the sense of \cite{Floer:Coherent}). We
now explain that with more detail.

Denote by $\ol \M_m(p,q)_{[1]}$ the Floer-Gromov compactification of
$\M_m(p,q)_{[1]}$. In view of Corollary~\ref{cor:compcasc} we see that
the points $\M^1_m(p,q)_{[0]}$ occur as boundary points of $\ol
\M_m(p,q)_{[1]}$ and $\ol \M_{m-1}(p,q)_{[1]}$. Define the space with
common boundary points identified
\[
\M_\#(p,q) := \bigcup_{m\in \N_0} \ol \M_m(p,q)_{[1]} / \sim\,.
\]
\begin{dfn}\label{dfn:coherentcasc}
  Given orientations of the spaces $\M_m(p,q)_{[0]}$ for all $p,q \in
  \crit f$ and $m \in \N_0$.  We say that the orientations are
  \emph{coherent}, if there exists an orientation on $\M_\#(p,q)$ such
  that its oriented boundary is given by
  \[
  \bigcup \M(p,r)_{[0]} \times \M(r,q)_{[0]}\,,\] with union over all
  critical points $r\in \crit f$.
\end{dfn}
\begin{rmk}
  If coherent orientations exists, then pearl homology with integer
  coefficients is well-defined. Given two sets of orientations which
  are coherent, there is no reason why the homology should be the
  same. In particular in~\cite{Cho}, Cho found non-isomorphic Floer
  cohomologies for the same pair of Lagrangians but with different
  choices of orientations associated to non-equivalent relative spin
  structures.
\end{rmk}

\subsubsection{Construction of the orientation} For all $m\in \N$ and any
connected component $C \subset L_0 \cap L_1$ we consider the space
$\wt \M_m(p,C):=\wt \M_m(W^u(p),C)$ equipped with the evaluation map
\[\ev: \wt \M_{m}(p,C)\to C,\qquad
(u_1,\dots,u_{m}) \mapsto u_{m}(\infty).\] Let $\O\to L_0 \cap L_1$ be
the double cover associated to a fixed relative spin structure (\cf
Definition~\ref{dfn:O}). Fix an element in $|T_pW^u(p)|\otimes
\O^\vee_p$ for each critical point $p \in \crit f$.  We construct
recursively an $\O^\vee$-orientation on $\ev:\M_m(p,C) \to C$, \ie a
section of $|\M_m(p,C)|\otimes \ev^*\O^\vee$.  First of all notice
that since $W^u(p)$ is contractible to the point $p$ our choices fix
an $\O^\vee$-orientation on $W^u(p)$ by parallel transport. Then by
Corollary~\ref{cor:ori} we obtain an $\O^\vee$-orientation on $\wt
\M_1(p,C) \to C$ and hence an $\O^\vee$-orientation on the quotient
$\M_1(p,C) \to C$ by~\eqref{eq:oriquotient}. An $\O^\vee$-orientation
on $\ev:\M_m(p,C) \to C$ induces and $\O^\vee$-orientation on
$\ev_\psi:\R \times \M_m(p,C) \to C$, $(a,u) \mapsto \psi^a(\ev(u))$
by parallel transport. Suppose for $m\geq 2$ we have an
$\O^\vee$-orientation on $\ev:\M_{m-1}(p,C') \to C'$ for some connected
component $C'$.  There exists an induced $\O^\vee$-orientation on
\begin{equation} \label{eq:MBori} \big[(\R \times \M_{m-1}(p,C'))
  \tensor[_{\ev_\psi}]{\times}{_{\ev_-}} \wt \M_1(C',C)\big]\big/\sim
  \ \to C \,,
\end{equation}
in which the quotient is with respect to the group of
reparametrizations acting on the last factor.  More precisely the
$\O^\vee$-orientation is constructed by the remarks above,
Corollary~\ref{cor:ori} and the quotient
orientation~\ref{eq:oriquotient}. Here the order of each step
matters. In particular we \emph{first} construct an orientation on the
fibre product and \emph{then} take the associated orientation on the
quotient and not the other way around. Every connected component of
$\M_m(p,C)$ is of the form~\eqref{eq:MBori} for some component $C'
\subset L_0\cap L_1$ and thus by induction we obtain an
$\O^\vee$-orientation on $\M_m(p,C) \to C$ as promised. We have the
isomorphism $|T_q W^u(q)|\otimes \O_q^\vee\cong (\O_q \otimes
|T_qC/T_qW^s(q)|)^\vee$. Thus our choices fix an $\O$-coorientation on
$W^s(q)$. Finally we obtain an orientation on $\M_m(p,C) \times_C
W^s(q) = \M_m(p,q)$ by Lemma~\ref{lmm:orifibre}.
\begin{lmm}\label{lmm:coherent}
  The orientations are coherent.
\end{lmm}
\begin{proof}\setcounter{stp}{0}
  As mentioned in Section~\ref{sec:prelim} the orientation of $
  \M_m(p,q)_{[1]}$ induces an orientation on its boundary points,
  denoted $\partial \M_m(p,q)_{[1]}$. We define an orientation on
  $\M_\#(p,q)$ induced by $(-1)\cdot
  \M_m(p,q)_{[1]}$. Step~\ref{stp:M1m} and Step~\ref{stp:M1m+1} show
  that the orientation is well-defined. Step~\ref{stp:00} to
  Step~\ref{stp:kl} shows that the orientations are coherent.
  \begin{stp}\label{stp:00}
    We show $\M_0(p,r)_{[0]} \times \M_0(r,q)_{[0]} \subset
    (-1)\cdot \partial \M_0(p,q)_{[1]}$
  \end{stp}
  Given $(u,v) \in \M_0(p,r)_{[0]} \times \M_0(r,q)_{[0]}$.  Pick an
  orientation of $\O_p$. By parallel transport along $u$ and $v$ we
  obtain an orientation of $\O_r$ and $\O_q$ respectively. By our
  choices we obtain orientations of $W^u(p)$, $W^u(r)$ and a
  coorientation of $W^s(q)$.  We identify $\M_0(p,q)_{[1]}$ with the
  intersection $W^u_a(p) \cap W^s(q)$ where $a \in \R$ is a regular
  value of $f$ with $f(p) > a>f(q)$. By Lemma~\ref{lmm:capfibre} the
  identification is orientation preserving. By Lemma~\ref{lmm:Mglue}
  we obtain a smooth map $R \mapsto w^-_R \in \M_0(p,q)_{[1]}$ such
  that orientation of $\partial_R w^-_R$ is $-\sign u \sign v$,
  $\lim_{R \to \infty} w^-_R= u$ and $\lim_{R \to \infty} \psi^{2R}
  w^-_R = v$.
  \begin{stp}\label{stp:m0}
    With $m\geq 1$. We show $ \M_m(p,r)_{[0]}\times
    \M_0(r,q)_{[0]}\subset (-1)\cdot\partial \M_m(p,q)_{[1]}$.
  \end{stp}
  Given $(u,v) \in \M_m(p,r)_{[0]} \times \M_0(r,q)_{[0]}$. We have $u
  \in \M_m(p,C)$ for some component $C \subset L_0 \cap L_1$. Identify
  an open neighborhood of $u$ in $\M_m(p,C)$ with the image under the
  evaluation $\M_m(p,C) \to C$. Denote the image by $W_-$, which is
  submanifold in a neighborhood of $\ev(u)$. The space $W_-$ has an
  $\O^\vee$-orientation by construction.  Pick an orientation of
  $\O_q$. We obtain a coorientation of $W^s(q)$ and an orientation of
  $W_-$ for all points in $W_- \cap W^s(q)$ by parallel transport. By
  Lemma~\ref{lmm:capfibre} the identification of an open subset of
  $\M_m(p,q)$ with $W_- \cap W^s(q)$ is orientation preserving. By
  Lemma~\ref{lmm:Mglue} we obtain a smooth map $R \mapsto w^-_R \in
  \M_m(p,q)$ with the same properties as in last step.
  \begin{stp}\label{stp:0m}
    With $m\geq 1$. We show $\M_0(p,r)_{[0]}\times \M_m(r,q)_{[0]}
    \subset (-1)\cdot\partial \M_m(p,q)_{[1]}$
  \end{stp}
  Given $(u,v) \in \M_0(p,r)_{[0]} \times \M_m(r,q)_{[0]}$.
  Abbreviate the space $\wt \M_m(C,q):=\wt \M_m(C,W^s(q))$ with
  quotient $\M_m(C,q)$ for some component $C \subset L_0 \cap L_1$
  such that $v \in \M_m(C,q)$. We identify an open neighborhood of $v$
  in $\M_m(C,q)$ with the image under the evaluation map $\M_m(C,q)
  \to C$. Denoted the image by $W_+ \subset C$, which is a submanifold
  in a neighborhood of $\ev(v)$. Then an open subset of $\M_m(p,q)$ is
  identified with $W^u(p) \cap W_+$. Similarly we identify $\M_m(r,q)$
  with $W^u(r) \cap W_+$. By Lemma~\ref{lmm:Mglue} we obtain $R
  \mapsto w_R \in W^u(p) \cap W_+$ such that $\lim_{R \to \infty} w_R
  = v$ and $\lim_{R \to \infty} \psi^{-2R}w_R = u$. It remains to
  check that the orientation of $\partial_R w_R$ is
  $-\sign(u)\sign(v)$.

  Fix some $w=w_R$ for $R$ sufficiently large.  In the following we
  will often talk a little imprecisely about an orientation of the
  space $\M_m(p,q)$ or $\wt \M_m(p,q)$ where in fact we mean an
  orientation of an open neighborhood of $w$ (or its lift) inside
  this space. We will not mention that every time.

  By Theorem~\ref{thm:ori} and Lemma~\ref{lmm:orifibre} construct an
  orientation of $\wt \M_m(p,q)$ as the fibre product for some
  connected components $C_0, C_1, \dots, C_m$ chosen appropriately
  \begin{equation}\label{eq:oricasc}
    W^u(p) \times_{C_0} (\wt \M \times \R) \times_{C_1} (\wt \M \times
    \R)\times_{C_2} \dots \times_{C_{m-2}}(\wt \M \times \R) \times_{C_{m-1}} \wt \M
    \times_{C_m} W^s(q)\,,
  \end{equation}
  in which $\wt \M$ denotes the space of $J$-holomorphic strips with
  obvious evaluation maps. We obtain an alternative orientation on the
  quotient $\M_m(p,q)$ via~\eqref{eq:oriquotient} and using the
  commutation rule~\eqref{eq:commute} we conclude that this
  orientation differs from the orientation above by the action with $(-1)^{\Delta(w)}$ where
  \[
  \Delta(w) := \Delta_1 + \Delta_2 +\dots + \Delta_{m-1}, \qquad
  \Delta_k = \dim T_{(w_1,\dots,w_k)} \M_k(p,C_k)\,.
  \]
  By the dimension formula (\cf Lemma~\ref{lmm:regfreecasc})
  \[
  \Delta_k = \mu(p) + \mu(w_1)+\mu(w_2)+ \dots + \mu(w_k) - \frac
  12 \dim C_0 + \frac 12 \dim C_k -1\,.
  \]
  Similarly we obtain another orientation of $\M_m(r,q)$ which agrees
  with the old orientation up to action with $(-1)^{\Delta(v)}$ where
  $\Delta(v) = \Delta_1(v)+\Delta_2(v)+\dots+\Delta_{m-1}(v)$ and for
  $k=1,\dots,m-1$ we have
  \[ \Delta_k(v) = \mu(r) + \mu(v_1)+\dots+\mu(v_k) - \frac 12 \dim
  C_0 +\frac 12 \dim C_k -1\,.
  \]Moreover we obtain an $\O$-coorientation of $\wt \M_m(C,q) \to C$
  by writing $\wt \M_m(C,q)$ as fibre product~\eqref{eq:oricasc}
  without the first factor. We obtain a canonical $\O$-coorientation
  on the quotient $W_+$ by~\eqref{eq:oriquotient}. Let $\wt
  \M_m(p,q)'$ and $\M_m(p,q)'$ \etc be the space with orientation
  induced by~\eqref{eq:oricasc}. By associativity of the fibre product
  orientation (\cf Lemma~\ref{lmm:fibreassociative}) we conclude
  \[
  \wt \M_m(p,q)' = W^u(p) \times_{C} \wt \M_m(C,q)\,.
  \]
  Let $G=\R^{m}$ be the group of reparametrizations. By definition
  we have as oriented spaces $\wt \M_m(p,q)'=G \times \M_m(p,q)'$ and
  $\wt \M_m(C,q)=G \times W_+$. Hence
  \[    G \times \M_m(p,q)' = W^u(p) \times_C (G \times W_+)\,.
  \]
  By commuting and canceling the factor $G$ we obtain 
  \[
  (-1)^g \M_m(p,q)' = (-1)^{gw_+} W^u(p) \times_C W_+ = (-1)^{gw_+}
  W^u(p) \cap W_+\,,\] in which we have used small letters to denote
  the dimensions of the spaces. Similarly we show
  \[
  \M_m(r,q)' = (-1)^{gw_+} W^u(r) \cap W_+\,.
  \]
  Pick an orientation of $\O_p$. By parallel transport along $u$ we
  obtain an orientation of $\O_r$. By our choices we obtain an
  orientation on $W^u(p)$ and $W^u(r)$ as well as a coorientation of
  $W_+$ such that $\sign u=\e_0$, where $\e_0$ is the usual Morse
  trajectory orientation and $\sign v=(-1)^{\Delta(v) + gw_+}\e_1$
  where $\e_1$ is the orientation of the point $v$ as an element of
  the intersection $W^u(r) \cap W_+$. Moreover we conclude with
  Lemma~\ref{lmm:Mglue} that the orientation of $\partial_R w_R$ is
  \[
  (-1)^{\Delta(w)+gw_++g}\e_0\e_1 =(-1)^{\Delta(w) + 2gw_+ +g +
    \Delta(v)}\, \sign(u) \sign(v)\,.
  \]
  We have $\Delta_j(w)=\Delta_j(v)+1$ for all $j=1,\dots,m-1$. Hence
  $\Delta(w)=\Delta(v)+m-1$. Moreover $g=m$. Hence $\Delta(w)+\Delta(v)+g\equiv 1 \mod 2$. This shows the claim.
  \begin{stp}\label{stp:kl}
    With $m \geq 2$ and $k+\ell=m$ with $k \neq 0,m$. We show
    $\M_k(p,r)_{[0]} \times \M_\ell(r,q)_{[0]} \subset
    (-1)\cdot\partial
    \M_m(p,q)_{[1]}$.
  \end{stp}
  Given $(u,v) \in \M_k(p,r)_{[0]}\times \M_\ell(r,q)_{[0]}$. We
  consider $u$ and $v$ as elements in $\M_k(p,C)$ and $\M_\ell(C,q)$
  for some component $C \subset L_0 \cap L_1$ respectively. Identify
  open neighborhoods of $u$ and $v$ in $\M_k(p,C)$ and $\M_\ell(C,q)$
  with submanifolds $W_- \subset C$ and $W_+ \subset C$ respectively
  using the evaluation map. Then an open subset of $\M_m(p,q)$ is
  identified with $W_- \times_{\psi} W_+$ (\cf
  equation~\eqref{eq:W}). By Lemma~\ref{lmm:Mglue} there exists $R
  \mapsto w_R \in \M_m(p,q)$ such that $\lim_{R\to \infty} w_R = u$
  and $\lim_{R \to \infty} \psi^{2R} w_R = v$. It remains to show that
  the orientation of $\partial_R w_R$ is $-\sign(u)\sign(v)$.

  With alternative orientations as described in the last step we have
  \[
  \wt \M_m(p,q)' = (\wt \M_k(p,C)' \times \R) \times_C \wt \M_\ell(C,q)'\,.
  \]
  Let $G_-=\R^k$ (\resp $G_+ =\R^\ell$) be the group of
  reparametrizations of $\wt \M_k(p,C)$ (\resp
  $\wt\M_\ell(C,q)$). Pick an orientation on $\O_r$. We obtain an
  orientations of $\wt \M_k(p,C)$ and a coorientation of $ \wt
  \M_\ell(C,q) \to C$. Hence an orientation on $W_- \subset C$ and a
  coorientation on $W_+ \subset C$ such that $\wt \M_k(p,C)' =
  G_-\times W_-$ and $\wt \M_\ell(C,q)'=G_+\times W_+$. Thus (\cf
  equation~\eqref{eq:W})
  \begin{equation}
    \label{eq:MWpsiW}
  (-1)^{g_+}\M_m(p,q)'= (-1)^{g_+w_+}(W_- \times \R) \times_C W_+ =
  (-1)^{g_+w_+ +w_-} W_- \times_{\psi} W_+\,.    
  \end{equation}
  Similarly we conclude that
  \[
  \M(p,r)' = W_- \cap W^s(r),\qquad \M(r,q)' = (-1)^{g_+w_+} W^u(r)\cap W_+\,.
  \]
  Hence we have $\sign u =(-1)^{\Delta(u)} \e_0$ with $\e_0$ is the
  orientation of $u$ as an element in the intersection $W_- \cap
  W^s(r)$ and $\sign v = (-1)^{\Delta(v)+gw_+} \e_1$ with $\e_1$ is
  the orientation of $v$ as an element in the intersection $W^u(r)\cap
  W_+$. By Lemma~\ref{lmm:Mglue} the orientation of $\partial_R w_R$
  is
  \[
  (-1)^{\Delta(w)+g_+w_++w_-+g_+}\e_0\e_1 =
  (-1)^{\Delta(w)+\Delta(u)+\Delta(v) + w_- +g_+} \sign u \sign v\,.
  \]
  We have $\Delta_j(u)=\Delta_j(w)$ for all $j=1,\dots,k-1$. Moreover
  $\mu(r) = w_-=\Delta_k(w)$. By definition we have the recursive
  formula for all $j=1,\dots, m-1$
  \[ \Delta_j(w) = \Delta_{j-1}(w) + \mu(w_j) + \frac 12\big(\dim T_{w_j(\infty)}
  L_0\cap L_1 - \dim T_{w_j(-\infty)} L_0 \cap L_1\big)
  \]
  and similarly for $\Delta_j(v)$. Since for all $j=1,\dots,\ell-1$
  the index of $v_j$ is the same as $w_{j+k}$ and the asymptotics lie
  on the same connected components we have
  $\Delta_j(v)=\Delta_{j+k}(w)-1$. We conclude that
  $\Delta(w)=\Delta(u)+\Delta(v)+w_-+\ell-1$ and
  \[
  \Delta(w) + \Delta(u)+\Delta(v) + w_- +g_+ \equiv \ell - 1 +g_+
  \equiv 1 \mod 2\,.
  \]
  This shows the claim.
  \begin{stp}\label{stp:M1m}
    For $m\geq 2$. We show $\M^1_m(p,q)_{[0]} \subset
    (-1)\cdot\partial \M_m(p,q)_{[1]}$.
  \end{stp}
  We have not yet constructed an orientation on $\M^1_m(p,q)$ so
  technically the statement does not make sense on the level of
  oriented spaces. But our orientation algorithm easily generalizes to
  give an orientation of the space $\M^1_m(p,q)$, if we leave out the
  $\R$-factor in~\eqref{eq:MBori} at the appropriate place. So we
  assume in the following that the space $\M^1_m(p,q)$ is oriented in
  that way.

  Also similarly as above we construct an orientation on $\wt
  \M^1_m(p,q)$ via~\eqref{eq:oricasc} in which we omit the $\R$-factor
  at the appropriate place. We obtain an alternative orientation on
  the quotient and we denote the space equipped with the alternative
  orientation by $\M^1_m(p,q)'$.  Given $u \in \M^1_m(p,q)$. Lets say
  we have $u_k(\infty)=u_{k+1}(-\infty)$ for some
  $k=1,\dots,m-1$. Then the orientation as an element of
  $\M^1_m(p,q)'$ is changed with action of $(-1)^{\Delta(u)}$ where
  $\Delta(u)=\Delta_1(u)+\dots+\Delta_{m-1}(u)$ and $\Delta_j(u) =
  \dim T_{(u_1,\dots,u_j)} \M_j(p,C_j)$ if $j<k$, $\Delta_k(u)=0$ and
  $\Delta_j(u)= \dim T_{(u_1,\dots,u_j)} \M^1_j(p,C_j)$ if $j>k$. We
  identify $\M_k(p,C)$ as a submanifold $W_- \subset C$ and
  $\M_{\ell}(C,q)$ as a submanifold $W_+$ respectively. Hence by the
  same computation as~\eqref{eq:MWpsiW} we have
  \[\M_m(p,q)' = (-1)^{g_++g_+w_+ +w_-} W_- \times_\psi W_+\,.\]
  Similarly we conclude
  \[
  \M_m^1(p,q)' = (-1)^{g_+w_+} W_- \cap W_+\,.
  \]
  We have $\sign u = (-1)^{g_+w_+ + \Delta(u)} \e$, where $\e$ is the
  sign of $u$ seen as an element of $W_- \cap W_+$. A local
  construction gives a family $R \mapsto w_R = (R,w^-_R,w^+_R) \in W_-
  \times_\psi W_+$ with $w_0=(0,u,u)$ (\cf Step~\ref{stp:localWW} in
  the proof of Lemma~\ref{lmm:Mglue}). Moreover $\partial_R w_R$
  induces an orientation on $W_- \times_\psi W_+$ which is $\e\,o$
  where $o$ is the canonical orientation on $W_- \times_{\psi} W_+$
  (\cf Step~\ref{stp:WpW} of the same proof). For $\M_m(p,q)$ the
  vector $\partial_R w_R$ gives an orientation that is changed by the
  action with
  \[
  (-1)^{\Delta(w) + g_+ + w_- + \Delta(u)} \sign(u)\,.
  \]
  We have $\Delta_j(w) = \Delta_j(u)$ for all $j=1,\dots,k-1$,
  $\Delta_k(w)=w_-$, $\Delta_k(u)=0$ and $\Delta_j(w)=\Delta_j(u)+1$
  for all $j=k+1,\dots,m-1$. We conclude that $\Delta(w) = \Delta(u) +
  w_- + \ell -2$. Since $g_+ =\ell$ and with the last equation the
  vector changes orientation by $\sign(u)$. But since this time the
  vector $\partial_R w_R$ points inward the induced boundary
  orientation is $-\sign(u)$.
  \begin{stp}\label{stp:M1m+1}
    For $m\geq 1$. We have $\M^1_{m+1}(p,q) \subset \partial
    \M_m(p,q)_{[1]}$.
  \end{stp}
  With oriented spaces as explained in the last step we have
  \begin{gather*}
    \wt \M^1_{m+1}(p,q)' = (\wt \M_k(p,C_-)'\times \R) \times_{C_-}
    \wt \M \times_C \wt \M \times_{C_+} (\R \times \wt
    \M_\ell(C_+,q))\,,\\
    \wt \M_m(p,q)'\ = (\wt \M_k(p,C_-)' \times \R) \times_{C_-} \wt \M
    \times_{C_+} \times  (\R \times \wt \M_\ell(C_+,q))\,,
  \end{gather*}
  with obvious evaluation maps.  Let $G = \R^{m}$ (\resp
  $G^1=\R^{m+1}$, $G_- = \R^k$ and $G_+=\R^\ell$) be the
  reparametrization group of $\wt \M_m(p,q)$ (\resp $\wt
  \M^1_{m+1}(p,q)$, $\wt \M_k(p,C_-)$ and $\wt \M_\ell(C_+,q)$). Via
  the evaluation map we identify the quotients $\M_k(p,C_-)$ and
  $\M_\ell(C_+,q)$ with submanifolds $W'_- \subset C_-$ and $W'_+
  \subset C_+$ respectively. Moreover we embed $W_- = \R \times W'_-$
  (\resp $W_+ = \R \times W'_+$) into $C_-$ (resp. $C_+$) via the
  Morse-flow. We have
  \begin{gather*}
    G^1 \times \M^1_{m+1}(p,q)' = (-1)^{g_+w_+} G_- \times (W_-
    \times_{C_-}
    \wt \M \times_C \wt \M \times_{C_+} W_+) \times G_+\,,\\
    G \times \M_m(p,q)' = (-1)^{g_+w_+} G_- \times (W_- \times_{C_-}
    \wt \M \times_{C_+} W_+) \times G_+\,.
  \end{gather*}
  Denoting the quotients of the terms inside the parentheses on the
  right-hand side by $\M^1(W_-,W_+)$ and $\M(W_-,W_+)$ respectively,
  we conclude that
  \begin{gather*}
    \M^1_{m+1}(p,q) = (-1)^{\Delta(u) + g_+w_+} \M^1(W_-,W_+)\,,\\
    \M_m(p,q) = (-1)^{\Delta(w) + g_+w_+ +g_+} \M(W_-,W_+)\,.
  \end{gather*}
  By Theorem~\ref{thm:glue} we have $\M^1(W_-,W_+) =
  (-1)^{g_1}\cdot \partial \M(W_-,W_+)$ in which $g_1=1$ is the
  dimension of the group of reparametrizations acting on the second
  factor of $\wt \M^1(W_-,W_+)$. Thus
  \[
  \M_{m+1}^1(p,q) \subset (-1)^{\Delta(u)+\Delta(w)+g_++g_1}\partial \M_m(p,q)\,.
  \]
  Again we have $\Delta_j(u)= \Delta_j(w)$ for all $j=1,\dots,k$ and
  $\Delta_k(u)=0$. By the additivity axiom for the Viterbo index
  $\mu(u_k)+\mu(u_{k+1}) = \mu(w_k)$ and so $\Delta_j(u) =
  \Delta_{j-1}(w)$ for all $j=k+1,\dots,m$. So $\Delta(u) + \Delta(w)
  = m-k-1=\ell+1 \mod 2$. This shows the claim because $g_+=\ell$.  
\end{proof}
\subsection{Invariance}
In this section we construct a canonical isomorphism
\begin{equation}
  \label{eq:invariance}
  QH_*(L_0,L_1)\cong QH_*(L_0,\vp_H(L_1))\,,
\end{equation}
for any clean Hamiltonian function $H$.  The construction of the
isomorphism is well-known and goes along the lines of
\cite{Schwarz:PSS}, \cite{Peter:PSS}, \cite[Section
7]{Frauenfelder:PhD} and \cite{BiranCornea:rigidity}. The novelty we
present here is the gluing analysis for cleanly intersecting
Lagrangians and the orientations.

Fix a vector field $X\in C^\infty(\Sigma,\Vect(M))$ such that
$X(-s,\cdot)=X_-$ and $X(s,\cdot)=X_+$ for all $s \geq 1$. Moreover
let $\J$ denote the space of almost complex structures $J \in
C^\infty(\Sigma,\End(TM,\omega))$ such that $J(-s,\cdot)=J_-$ and
$J(s,\cdot)=J_+$ for all $s \geq 2$. We abbreviate by $\I_-$ (\resp
$\I_+$) the space of perturbed intersection points of $H_-$ (\resp
$H_+$), fix a Morse function $f_-$ (\resp $f_+$) and denote the
negative gradient flow by $\psi_-$ (\resp $\psi_+$) with respect to
some sufficiently generic metric. For any $m \geq 1$, $J \in \J$ and
submanifolds $W_- \subset \I_-$, $W_+ \subset \I_+$ we define the
space
\[
\wt \M_m(W_-,W_+;J,X):=\{(u_1,\dots,u_m) \in C^\infty(\Sigma,M) \mid
\text{\ref{nm:chia} -- \ref{nm:chib}}\}\,,
\] 
as the space of tuples $u=(u_1,\dots,u_m)$ such that there exists
$k\in \N$ and
\begin{enumerate}[label=\alph*)]
\item\label{nm:chia} the tuple $(u_1,\dots,u_{k-1})$ is a $(J_-,X_-)$-holomorphic pearl
\item the map $u_k$ is a $(J,X)$-holomorphic strip,
\item the tuple $(u_{k+1},\dots,u_m)$ is a $(J_+,X_+)$-holomorphic pearl
\item all strips have boundary in $(L_0,L_1)$,
\item there exists $a_-, a_+\geq 0$ such that $
  \psi^{a_-}_-(u_{k-1}(\infty)) = u_k(-\infty)$ and also
  $\psi^{a_+}_+(u_k(\infty)) = u_{k+1}(-\infty)$,
\item\label{nm:chib} we have $u_1(-\infty) \in W_-$ and $u_m(\infty)\in W_+$.
\end{enumerate}
The elements of $\wt\M_m(W_-,W_+;J,X)$ are called
\emph{$(J,X)$-holomorphic pearl trajectories with $m$ cascades
  connecting $W_-$ to $W_+$}. If in particular $W_-=W^u(p)$ and
$W_+=W^s(p)$ for critical points $p \in \crit f_-$ and $q \in \crit
f_+$, we abbreviate
\[
\wt \M_m(p,q;J,X):= \wt \M_m(W^u(p),W^s(q);J,X)\,,
\]
with elements called \emph{$(J,X)$-holomorphic pearl trajectories with
  $m$ cascades connecting $p$ to $q$}. 
\subsubsection{Transversality}
For $m\in \N$ and a connected component $C_- \subset \I_-$ abbreviate
the space $\wt\M_m^-:=\wt \M_m(W_-,C_-;J_-,X_-)$ and the evaluation
map
\[
\ev^-:\R \times \wt \M_m^- \to C_-, \qquad (a,u_1,\dots,u_m) \mapsto \psi^a(u_m(\infty))\,.
\]
Similarly given $C_+\in \I_+$ let $\wt \M_m^+:=\wt
\M_m(C_+,W_+;J_+,X_+)$ equipped with the evaluation map
\[
\ev^+:\R \times \wt \M^+_m \to C_+, \qquad (a,u_1,\dots,u_m) \mapsto
\psi^a(u_1(-\infty))\,.
\]
Assume that $W_-$ and $W_+$ are chosen such that $J_-$ is regular for
$X_-$ and $W_-$ and $J_+$ is regular for $X_+$ and $W_+$ (\cf
Definition~\ref{dfn:regs}). We call $J$ regular if
it is regular for $X$ and the above evaluation maps for any connected
components, \ie the fibre product 
\[
(\R_+ \times \wt \M^-_k) \times_{C_-} \wt
\M(C_-,C_+;J,X) \times_{C_+} (\R_- \times \wt
\M_{\ell}^+)
\]
is cut-out transversely with $\R_+ = [0,\infty)$ and $\R_- =
(-\infty,0]$ for any connected components $C_-$ and $C_+$ and
integers $k$ and $\ell$. We show in Theorem~\ref{thm:regR} that the
subspace of regular almost complex structures in $\J$ is comeager and
we fix a regular $J \in \J$ for the rest of the section. Obviously the
space $\wt \M_m(W_-,W_+;J,X)$ is the union of the fibre product over
all connected components and integers $k,\ell$ such that $k+\ell=m-1$
and every connected component of it is a manifold with corners. We
conclude with a computation similarly as in the proof of
Lemma~\ref{lmm:regfreecasc} the dimension of the component in $\wt
\M_m(W_-,W_+;J,X)$ contain the element $u$ is
\[
\dim W_- + \dim W_+ - \frac 12 \dim C_0 - \frac 12 \dim C_m + m-1 + \mu(u)\,,
\]
in which $C_0 \subset \I_-$ is the connected component containing
$u_1(-\infty)$ and $C_m \subset \I_+$ is the connected component
containing $u_m(\infty)$ respectively. There is a free
$\R^{m-1}$-action on the space $\wt \M_m(p,q;J,X)$ given by
reparametrizations and we denote the corresponding quotient by
$\M_m(p,q;J,X)$. Since $J$ is fixed, we omit the reference to it and
write $\M_m(p,q)$ to denote the space $\M_m(p,q;J,X)$. Moreover for
some $\ell,d \in \N_0$ we denote by $\M_m(p,q)_{[d]}$ the union of all
$d$-dimensional components and $\M^\ell_m(p,q)$ the union of all
elements of depth $\ell$.
\subsubsection{Compactness}
A \emph{broken $(J,X)$-holomorphic pearl trajectory connecting $p$ to
  $q$ of height $k$} is a tuple of pearl trajectories
$(v_1,\dots,v_k)$ such that $v_i$ is a $(J_i,X_i)$-holomorphic pearl
connecting $p_i$ to $p_{i+1}$ for all $i=1,\dots,k-1$ for critical
points $p=p_1,\dots,p_\ell \in \crit f_-$, $p_{\ell+1},\dots,p_m =q\in
\crit f_+$ for some $\ell\leq k$ and we have
\[
(J_i,X_i) =
\begin{cases}
  (J_-,X_-)&\text{if } i=1,\dots,\ell-1\\
  (J,X)&\text{if } i = \ell\\
  (J_+,X_+)&\text{if } i=\ell+1,\dots,k\,.
\end{cases}
\]
Completely analogous to the case without the Hamiltonian perturbations
we define the notion of Floer-Gromov convergence, prove that the space
$\M_m(p,q)_{[0]}$ is finite and $\M_m(p,q)_{[1]}$ is compact up to
breaking of height two, \ie the Floer-Gromov boundary consists of the union of
\begin{itemize}
\item $\M^1_m(p,q)_{[0]}$,
\item $\M^1_{m+1}(p,q)_{[0]}$,
\item $\M_k(p,r_-)_{[0]} \times \M_\ell(r_-,q)_{[0]}$ for all $r_- \in
  \crit f_-$ and such that $k+\ell=m$,
\item $\M_k(p,r_+)_{[0]} \times \M_\ell(r_+,q)_{[0]}$ for all $r_+ \in
  \crit f_+$ and such that $k+\ell=m$.
\end{itemize}
\subsubsection{Orientations}
Let $\ol \M_m(p,q)_{[1]}$ be the Floer-Gromov compactification of
the space $\M_m(p,q)_{[1]}$ and define 
\[
\M_\#(p,q) := \bigcup_{m \in \N} \ol \M_m(p,q)_{[1]} /\sim\,,
\]
as disjoint union with double boundary points identified.
\begin{dfn}\label{dfn:coherentMchi}
  Orientations of the spaces $\M(p,q)_{[0]}$ for all $p\in \crit f_-$
  and $q \in \crit f_+$ are \emph{coherent}, if there exists an
  orientation on $\M_\#(p,q)$ such that its oriented boundary is given
  by the union of 
  \begin{itemize}
  \item $(-1)\cdot\M_k(p,r_-)_{[0]} \times \M_\ell(r_-,q)_{[0]}$ for all $r_- \in
    \crit f_-$ and such that $k+\ell=m$,
  \item $\M_k(p,r_+)_{[0]} \times \M_\ell(r_+,q)_{[0]}$ for all $r_+ \in
    \crit f_+$ and such that $k+\ell=m$.
  \end{itemize}
\end{dfn}
We construct orientations on the spaces $\M_m(p,q)$ completely
analogous as described in the paragraph before
Lemma~\ref{lmm:coherent}. A little more precisely we construct the
orientation recursively using the fibre-product~\eqref{eq:MBori} and
replace $\wt \M_1(C',C)$ in the expression with $\wt
\M_1(C',C;J_-,X_-)$, $\wt \M_1(C',C;J,X)$ or $\wt \M_1(C',C;J_+,X_+)$
appropriately.
\begin{lmm}\label{lmm:coherentinv}
  The orientations on $\partial \M_\#(p,q)$ are coherent.
\end{lmm}
\begin{proof}\setcounter{stp}{0}
  The proof follows the steps from the proof of
  Lemma~\ref{lmm:coherent}.
  \begin{itemize}
  \item To show $\M_m(p,r_+)_{[0]} \times \M_0(r_+,q)_{[0]} \subset
    (-1)\cdot \partial \M_m(p,q)_{[1]}$ proceed as Step~\ref{stp:0m}.
  \item To show $\M_0(p,r_-)_{[0]} \times \M_m(r_-,q)_{[0]}
    \subset \partial \M_m(p,q)_{[1]}$ proceed as in
    Step~\ref{stp:m0} but now the dimension of the group of
    reparametrizations is $g=m-1$.
  \item To show $\M_k(p,r_-)_{[0]} \times \M_\ell(r_-,q)_{[0]}
    \subset \partial \M_m(p,q)_{[1]}$ for $r_- \in \crit f_-$ and
    $k\neq 0$ proceed as in Step~\ref{stp:kl}. Now the dimension of
    the reparametrization group  on the right is $g_+=\ell-1$.
  \item To show $\M_k(p,r_+)_{[0]} \times \M_\ell(r_+,q)_{[0]} \subset
    (-1)\cdot \partial \M_m(p,q)_{[1]}$ for $r_+ \in \crit f_+$ and
    $\ell\neq m$ proceed as in Step~\ref{stp:kl} but now the
    dimension of the group of reparametrization on the right is
    $g_+=\ell$.
  \item To show $\M_k(p,C)\times_C
    \M_\ell(C,q)_{[0]} \subset \partial \M_m(p,q)_{[1]}$ if $C\subset
    \I_-$ proceed as in Step~\ref{stp:M1m}  with $g_+ = \ell-1$.
  \item To show $\M_k(p,C)\times_C
    \M_\ell(C,q)_{[0]} \subset (-1)\cdot \partial \M_m(p,q)_{[1]}$ if
    $C\subset \I_+$ proceed  as in Step~\ref{stp:M1m} with $g_+=\ell$.
  \item To show $\M_{k+1}(p,C)\times_C\M_{\ell+1}(C,q)_{[0]} \subset
    (-1)\cdot \partial \M_m(p,q)_{[1]}$ for $C \subset \I_-$ proceed
    as in Step~\ref{stp:M1m+1} with $g_++g_1 = \ell$.
  \item To show $\M_{k+1}(p,C)\times_C\M_{\ell+1}(C,q)_{[0]} \subset
    \partial \M_m(p,q)_{[1]}$ for $C \subset \I_+$ proceed as
    Step~\ref{stp:M1m+1} with $g_++g_1 = \ell+1$.
  \end{itemize}
  The above steps show that the orientation on $\M_\#(p,q)$ induced by
  $(-1)\cdot \M_m(p,q)_{[1]}$ is well-defined and shows that the
  orientations are coherent.
\end{proof}
\subsubsection{Chain map}
Define  the $\Lambda$-linear homomorphism 
\begin{gather*}
    C\chi_*(J,X): CH_*(L_0,L_1^-) \to CH_*(L_0,L_1^+)\\
    p \mapsto \sum_{q \in \crit f_+}\nolimits \sum_{[u] \in \M(p,q)_{[0]}}
    \nolimits \sign(u) \cdot q \otimes \lambda^{(\nm{q}-\nm{p})/N}\,.
\end{gather*}
Because the orientations are coherent the homomorphism $C\chi_*(J,X)$
is a chain map. To see that we have to show $C\chi_{*-1}
\circ \partial_* = \partial_*\circ C\chi_*$. By definition the
coefficient of $C\chi_{*-1} (\partial_* p) - \partial_* (C\chi_* p)$
in front of $q \in \crit f_+$ is given by
\[\sum \sign u_0 \sign u_1 - \sign v_0 \sign v_1\,,
\]
with summation over all tuples $(u_0,u_1) \in \M(p,r_-)_{[0]}\times
\M(r_-,q)_{[0]}$ and $(v_0,v_1) \in \M(p,r_+)_{[0]} \times
\M(r_+,q)_{[0]}$ and all critical points $r_- \in \crit f_-$ and $r_+
\in \crit f_+$. Since the summation agrees with the summation of the
signs of the oriented boundary points of $\M_\#(p,q)$ it
vanishes. Having established that $C\chi_*(J,X)$ is a chain map we
denote the induced map on homology by
\begin{equation}
  \label{eq:chi}
  \chi_*(J,X) : QH_*(L_0,L_1^-) \to QH_*(L_0,L_1^+)\,.
\end{equation}
\subsubsection{Naturality}
\begin{lmm}\label{lmm:naturality}
  The homomorphism $\chi_*(J,X)$ does not depend on the choice of $X$
  and $J$.
\end{lmm}
\begin{cor}\label{cor:isosame}
  The map~\eqref{eq:chi} is the identity if $L_1^-=L_1^+$, $J_-=J_+$
  and $f_-=f_+$.
\end{cor}
\begin{proof}
  By Lemma~\ref{lmm:naturality} we are free to choose $X$ and $J$.
  Choose $X$ and $J$ which are $\R$-invariant. Then there is an
  additional free $\R$-action on the space of tuples $(u_1,\dots,u_m)$
  in $\M_m(p,q)$ with $u_k$ non-constant. Thus such tuples are not
  counted in the definition of the morphism $\chi(J,X)$. This shows
  that $\chi(J,X)$ is the identity on chain level.
\end{proof}
\begin{proof}[Proof of Lemma~\ref{lmm:naturality}]
  Let $(X_a,J_a)$ and $(X_b,J_b)$ be two different choices. Fix a map
  $X \in C^\infty([a,b]\times \Sigma, \Vect(M))$ such that
  $X(a,\cdot)=X_a$, $X(b,\cdot)=X_b$ and $X(R,\pm s,\cdot)=X_R(\pm
  s,\cdot)= X_\pm$ for all $s \geq 1$ and $R \in [a,b]$. We abbreviate
  by $\J$ the space of $J \in C^\infty([a,b]\times
  \Sigma,\End(TM,\omega))$ such that $J(a,\cdot)=J_a$,
  $J(b,\cdot)=J_b$ and $J(R,\pm s,\cdot)=J_R(\pm s,\cdot) = J_\pm$ for
  all $s \geq 2$ and $R \in [a,b]$. For critical points $p \in \crit
  f_-$, $q\in \crit f_+$ a natural number $m \in \N$ and $J \in \J$ we
  define the space
  \[
  \wt \M_m(p,q;J,X)\,,
  \]
  as the space of pairs $(u,R)$ where $R \in [a,b]$ and $u$ is a
  $(J_R,X_R)$-holomorphic pearl trajectory connecting $p$ to $q$. The
  space $\wt \M_m(p,q;J,X)$ is equally defined as the union of the fibre
  products
  \[
  (\R_+ \times \wt \M^-_{k-1}\times_{C_-} \wt
  \M_1(C_-,C_+;J,X)\times_{C_+}(\R_- \times \wt \M_\ell^+)\,,
  \]
  over all possible connected components $C_-$, $C_+$ and with $k+\ell
  =m$. We say that $J$ is \emph{regular} if it is regular with respect
  to $X$ (\cf Definition~\ref{dfn:regR}) and for any connected
  components and critical points the fibre product is cut-out
  transversely. We conclude by Theorem~\ref{thm:regR} that a generic $J
  \in \J$ is regular and we fix such a regular $J$. The components of
  the the space $\wt \M_m(p,q;J,X)$ are manifolds with corners. We
  compute the dimension of a component containing $(R,u)$ to be
  \begin{equation}
    \label{eq:dimMhomotopy}
    \mu(u) + \mu(p) - \mu(q) +m - \frac 12 \dim C_0 + \frac 12 \dim C_m\,,  
  \end{equation}
  in which again $C_0 \subset \I_-$ (\resp $C_m \subset \I_+$) denotes
  the connected component containing $p$ (\resp $q$).  The group of
  reparametrizations is $\R^{m-1}$ and we denote the quotient by
  $\M_m(p,q;J,X)$. With the same arguments and notations as in the
  sections above we show that $\M_m(p,q;J,X)_{[0]}$ is finite and the
  Floer-Gromov boundary of $\M_m(p,q;J,X)_{[1]}$ is given by broken
  trajectories of height at most two.  Moreover we obtain orientations
  on the spaces constructed recursively using the fibre
  product~\eqref{eq:MBori} in which $\wt\M(C',C)$ is given by $\wt
  \M(C',C;J_-,X_-)$, $\wt \M(C',C;J,X)$ or $\wt \M(C',C';J_+,X_+)$
  appropriately. Let $\M_\#(p,q;J,X)$ be the disjoint union of all the
  Floer-Gromov compactifications of $\M_m(p,q;J,X)_{[1]}$ with double
  boundary points identified.  We say that orientations on the
  boundary of $\M_\#(p,q;J,X)$ are \emph{coherent}, if there exists an
  orientation on $\M_\#(p,q;J,X)$ such that its oriented boundary is
  given by the union of
  \begin{itemize}
  \item $(-1)\cdot \M_m(p,q;J_a,X_a)_{[0]}$,
  \item $\M_m(p,q;J_b,X_b)_{[0]}$,
  \item $(-1)\cdot \M_k(p,r;J_-,X_-)_{[0]}\times \M_\ell(r,q;J,X)_{[0]}$ for all $k+\ell = m$ and $r \in \crit f_-$,
  \item $\M_k(p,r;J,X)_{[0]}\times \M_\ell(r,q;J_+,X_+)_{[0]}$ for all $k+\ell = m$ and $r \in \crit f_+$.
  \end{itemize}
  If we show that the orientations are coherent we are done because,
  then we define the $\Lambda$-linear homomorphism
  \begin{gather*}
    \Theta_*: CH_*(L_0,L_1^-) \to CH_{*+1}(L_0,L_1^+)\\
    p \mapsto \sum_{q \in \crit f_+}  \sum_{[u]\in
      \M(p,q;J,X)_{[0]}} \sign(u)\cdot q \otimes \lambda^{(|q|-|p|-1)/N}\,.
  \end{gather*}
  Now since the orientations are coherent we conclude that $\Theta_*$ is
  a chain homotopy from $C\chi_*(J_a,X_a)$ to $C\chi_*(J_b,X_b)$, \ie
  \[
  \partial_* \circ \Theta_*-\Theta_{*-1} \circ \partial_* =
  C\chi_*(J_a,X_a)-C\chi_*(J_b,X_b)\,,
  \]
  which implies that the induced morphism on homology agree.
\end{proof}
\begin{lmm}
  The orientations on $\partial \M_\#(p,q;J,X)$ are coherent.
\end{lmm}
\begin{proof}
  Abbreviate $\partial \M:=\partial \M_m(p,q;J,X)_{[1]}$.  The proof
  of coherence follows the steps from the proof of
  Lemma~\ref{lmm:coherent}.
  \begin{itemize}
  \item To show $\M_m(p,r;J,X)_{[0]}\times \M_0(r,q;J_+,X_+)_{[0]}
    \subset (-1)\cdot \partial \M$ proceed as in
    Step~\ref{stp:m0},
  \item To show $\M_0(p,r;J_-,X_-)_{[0]} \times \M_m(r,q;J,x)_{[0]}
    \subset \partial \M$ proceed as in
    Step~\ref{stp:0m}. We have $g=m-1$.
  \item To show $\M_k(p,r;J_-,X_-)_{[0]}\times
    \M_\ell(r,q;J,X)_{[0]}\subset \partial \M$
    proceed as in Step~\ref{stp:kl}. We have $g_+ = \ell-1$.
  \item To show $\M_k(p,r;J,X)_{[0]}\times \M_\ell(r,q;J_+,X_+)_{[0]}
    \subset (-1)\cdot \partial \M$ proceed as in
    Step~\ref{stp:kl}. We have $g_+ = \ell$.
  \item We show $\M_m(p,q;J_a,X_a)_{[0]} \subset
    (-1)\cdot \partial\M$. For appropriate
    submanifolds $W_-$ and $W_+$ we have $\M_m(p,q;J,X) =
    \M(W_-,W_+;J,X)$ as well as $\M_m(p,q;J_a,X_a) =
    \M(W_-,W_+;J_a,X_a)$. Fix some element $(a,u) \in \M(W_-,W_+;J_a,X_a)$. By
    the implicit function Theorem we obtain $R \mapsto (R,w_R)$ with
    $w_a=u$. Using notation as the proof of
    Proposition~\ref{prp:degglue} in case~\ref{nm:none} we conclude
    that $\sign (u) = \sign D_R'$ which equals the orientation induced
    by the vector $(1,\partial_R w_R) \in \ker \wh D_R$. Since
    $(1,\partial_R w_R)$ points inward the orientation of the boundary
    point is $-\sign(u)$.
  \item To show $\M_m(p,q;J_b,X_b)_{[0]} \subset \partial
    \M_m$ proceed as above but this time the vector
    points outward.
  \item To show $ \M_k(p,C;J_-,X_-)\times_C \M_\ell(C,q;J,X)_{[0]}
    \subset \partial \M_m$ with $k+\ell=m$ proceed as
    in Step~\ref{stp:M1m} with $g_+=\ell-1$,
  \item To show $\M_k(p,C;J,X) \times_C
    \M_\ell(C,q;J_+,X_+)_{[0]}\subset (-1)\cdot \partial
    \M_m(p,q;J,X)_{[1]}$ proceed as in Step~\ref{stp:M1m} with $g_+=\ell$,
  \item To show $\M_{k+1}(p,C;J_-,X_-)\times_C
    \M_{\ell+1}(C,q;J,X)_{[0]} \subset (-1)\cdot\partial
    \M$ with $\ell+k = m-1$ proceed as in
    Step~\ref{stp:M1m+1} with $g_++g_1 =\ell$
  \item To show $\M_{k+1}(p,C;J,X)\times_C
    \M_{\ell+1}(C,q;J_+,X_+)_{[0]} \subset \partial \M$
    with $\ell+k = m-1$ proceed as in Step~\ref{stp:M1m+1} with $g_++g_1 =\ell+1$.
  \end{itemize}
  This shows the claim by putting the orientation on $\M_\#(p,q;J,X)$
  induced by $(-1)\cdot\M_m(p,q;J,X)_{[1]}$ for any $m\in \N$.
\end{proof}

\subsubsection{Functoriality}
\begin{lmm}\label{lmm:functoriality}
  We show that the map~\eqref{eq:chi} is functorial, \ie the
  map defined in~\eqref{eq:chi} gives rise to a commutative triangle
  \[
  \xymatrix{QH(L_0,L_1)\ar[r]\ar@/^1pc/[rr]&QH(L_0,L_2)\ar[r]&QH(L_0,L_3)\,,}
  \]
  for any Lagrangians $L_0$, $L_1$, $L_2$ and $L_3$ such that $L_1$,
  $L_2$ and $L_3$ are Hamiltonian isotopic.
\end{lmm}
\begin{cor}\label{cor:isodifferent}
  The map~\eqref{eq:chi} is an isomorphism.
\end{cor}
\begin{proof}
  Using Lemma~\ref{lmm:functoriality} we see that the map
  $QH(L_0,L_1^+) \to QH(L_0,L_1^-)$ is an inverse to $QH(L_0,L_1^-)
  \to QH(L_0,L_1^+)$ since their composition is $QH(L_0,L_1^-)\to
  QH(L_0,L_1^-)$ which by Corollary~\ref{cor:isosame} is the identity.
\end{proof}
\begin{proof}[Proof of Lemma~\ref{lmm:functoriality}]
Suppose that $L_1 =\vp_{H_-}(L)$, $L_2 = \vp_{H}(L)$ and
$L_3=\vp_{H_+}(L)$ for some Hamiltonians $H_-$, $H$ and $H_+$ and a
fixed Lagrangian $L$. We abbreviate the perturbed intersection points
$\I_-:=\I_{H_-}(L_0,L)$, $\I:=\I_{H}(L_0,L)$ and
$\I_+:=\I_{H_+}(L_0,L)$. Fix vector fields $X_0,X_1 \in
C^\infty(\Sigma,\Vect(M))$ such that $X_0(-s,\cdot)=X_-:=X_{H_-}$,
$X_0(s,\cdot)=X_1(-s,\cdot)=X:=X_{H}$ and $X_1(s,\cdot) =
X_+:=X_{H_+}$ for all $s \geq 1$. Denote the Morse functions $f_-$,
$f$ and $f_+$ and paths of almost complex structures $J_\infty^-$,
$J_\infty$ and $J_\infty^+$ with respect to which the pearl homology
groups are defined. Denote by $\J$ be the space pairs $(J_0,J_1) \subset
C^\infty(\Sigma,\End(TM,\omega))$ where $J_0(-s,\cdot) =J_\infty^-$,
$J_0(s,\cdot)=J_1(-s,\cdot)=J_\infty$ and $J_1(s,\cdot)=J_\infty^+$
for all $s \geq 2$. 

Given some $J=(J_0,J_1)\in \J$ critical points $p\in \crit f_-$, $q\in
\crit f_+$ and a number $m\in \N$ with $m \geq 2$ we denote by
\[
\wt \NN_m(p,q;J,X)\,,
\]
the space of tuples $u=(u_1,\dots,u_m)$ such that for some $1 \leq k \leq m-1$
\begin{enumerate}[label=\alph*)]
\item the tuple $(u_1,\dots,u_k)$ is a $(J_0,X_0)$-holo\-morphic pearl trajectory,
\item the tuple $(u_{k+1},\dots,u_m)$ is a $(J_1,X_1)$-holo\-morphic
  pearl trajectory,
\item all trips have boundary in $(L_0,L)$,
\item there exists $a \geq 0$ such that $\psi^a(u_k(\infty)) =
  u_{k+1}(-\infty)$,
\item we have $W^u(p)\in u_1(-\infty)$ and $W^s(q) \in u_m(\infty)$.  
\end{enumerate}
As usual we conclude that for generic $J$ each component of these
spaces are manifolds with corners. The component containing the
element $u \in \wt \NN_m(p,q;J,X)$ has dimension
\[
\mu(u) + \mu(p) - \mu(q) + \frac 12 \dim C_0 - \frac 12 \dim C_m + m -1\,.
\] 
The group of reparametrizations has dimension $m-2$ and acts
freely. With usual notations as explained in the last sections we show
that the quotient $\NN_m(p,q;J,X)_{[0]}$ is finite and the
Floer-Gromov boundary of $\NN_m(p,q;J,X)_{[1]}$ is given by breaking
of height at most two. Also the spaces are oriented using the
algorithm given in the paragraph before Lemma~\ref{lmm:coherent}. We
denote by $\NN_\#(p,q;J,X)$ the disjoint union of the Floer-Gromov
compactification of $\NN_m(p,q;J,X)$ over all $m \in \N_2$ with double
boundary points identified. We show as above that there exists an
orientation on $\NN_\#(p,q;J,X)$ such that its oriented boundary is
given by
\begin{itemize}
\item $(-1)\cdot\KK_m(p,q;J,X)_{[0]}$ (see definition below)
\item $\NN_k(p,r;J,X)_{[0]}\times \M_\ell(r,q;J_\infty^+,X_+)_{[0]}$ for $r \in \crit f_+$ and $k+\ell=m$,
\item $(-1)\cdot\M_k(p,r;J_\infty^-,X_-)_{[0]}\times \NN_\ell(r,q;J,X)_{[0]}$ for $r \in \crit f_-$ and $k+\ell=m$,
\item $\M_k(p,r;J_0,X_0)_{[0]} \times \M_\ell(r,q;J_1,X_1)_{[0]}$ for $r \in \crit f$ and $k+\ell=m$.
\end{itemize}
Here  $\KK_m(p,q;J,X) \subset \NN^1_m(p,q;J,X)$ is the subspace of
equivalence classes of tuples $(u_1,\dots,u_m)$ such that $u_k$ is
$(J_0,X_0)$-holo\-morphic, $u_{k+1}$ is $(J_1,X_1)$-holo\-morphic and
$u_k(\infty)=u_{k+1}(-\infty)$ for some $k\in \N$. 
\begin{rmk}
  The space $\KK_m(p,q;J,X)_{[0]}$ appears as a boundary of the glued
  space $\NN_\#(p,q;J,X)$ because if we glue elements in
  $\KK_m(p,q;J,X)_{[0]}$ we do not obtain elements inside the space
  $\NN_{m-1}(p,q;J,X)_{[1]}$.
\end{rmk}
We need to define another moduli space. For critical points $p\in
\crit f_-$, $q\in \crit f_+$ and a number $m\in \N$ we denote by
\[
\wt \M_m(p,q;J,X)\,,
\]
the space of pairs $(u,R)$ such that $R \geq 2$ and $u$ is a
$(J_R,X_R)$-holomorphic pearl trajectories connecting $p_-$ to $p_+$
with glued structures $X_R= X_0\#_R X_1$ and $J_R=J_0\#_R J_1$ as
defined in~\eqref{eq:glueX} and~\eqref{eq:glueJ} respectively.  By
Theorem~\ref{thm:regR} we conclude that for a generic $J$ each
connected component of the spaces is a manifold with corners. The
dimension of a component containing $u$ is~\eqref{eq:dimMhomotopy}.
The group of reparametrizations has dimension $m-1$ and acts freely.
The quotient $\M_m(p,q;J,X)_{[0]}$ is finite and the Floer-Gromov
boundary of $\M_m(p,q;J,X)_{[1]}$ is give by broken trajectories of
height at most two. There also exists an orientation on the spaces as
explained in the paragraph before Lemma~\ref{lmm:coherent}. Let
$\M_\#(p,q;J,X)$ denote the union of all $\M_m(p,q;J,X)_{[1]}$ over $m
\in \N$ with double boundary points identified. We show as above there
exists an orientation on $\M_\#(p,q;J,X)$ such that its oriented
boundary is given by
\begin{itemize}
\item $(-1)\cdot\M_k(p,r;J_\infty^-,X_-)_{[0]}\times \M_\ell(r,q;J,X)_{[0]}$ for all $r \in \crit f_-$ and $k+\ell=m$,
\item $\M_k(p,r;J,X)_{[0]}\times \M_\ell(r,q;J_\infty^+,X_+)_{[0]}$
  for all $r \in \crit f_+$ and $k+\ell=m$,
\item $(-1)\cdot\M_m(p,q;J_R,X_R)_{[0]}$ with $R=2$,
\item $\KK_m(p,q;J,X)_{[0]}$ (which appears considering sequences
  $(u_\nu,R_\nu)$ with $R_\nu \to \infty$).
\end{itemize}  
We define the $\Lambda$-linear homomorphism
\begin{gather*}
  \Theta_*:CH_*(L_0,L_1)\to CH_*(L_0,L_3)\\
  p \mapsto \sum_{q \in \crit f_+} \sum_{[u]} \sign u \cdot q \otimes
  \lambda^{|q|-|p|-1}\,,
\end{gather*}
in which the second summation is over all elements $[u]$ in the space
$\NN(p,q;J,X)_{[0]}$ or $\M(p,q;J,X)_{[0]}$. If the orientations on
$\partial \NN_\#(p,q;J,X)$ and $\partial \M_\#(p,q;J,X)$ are coherent,
we conclude that $\Theta_*$ is a chain-homotopy from the composition
$\chi(J_1,X_1)\circ \chi(J_0,X_0)$ to $\chi(J_R,X_R)$ with $R=2$. Note
that the boundary components $\KK_m(p,q;J,X)$ appears in both spaces
but with opposite signs.
\end{proof}
\subsection{Spectral sequences}\label{sec:spectral}
In this section we prove the main Theorem~\ref{thm:Pozloc}.  Recall that a
spectral sequence  is a sequence of
complexes 
\[
(E^1_*,\partial), (E^2_*,\partial),\dots
\]
such that $E^{r+1}_* \cong \ker \partial^r/\im \partial^r$
for all $r\in \N$. We say that the spectral sequence
$(E^r_*,\partial^r)_{r \in \N}$ \emph{collapses (at page $r_0$)} if
there exists $r_0 \in \N$ such that $\partial^r = 0$ for all $r \geq
r_0$. In that case we have $E^{r+1}_* \cong E^r_* $ for all $r \geq
r_0$ and we define $E^\infty_* := E^{r_0}_*$. We say that a spectral
sequence \emph{converges} to the graded module $H_*$ if there exists a
filtration $\F$ on $H_*$ such that $E^\infty_* \cong \bigoplus_p \F^p
H_* /\F^{p-1} H_*$.  If $H_*$ is a vector space this always implies
that $E^\infty_* \cong H_*$ although the isomorphism is not canonical.
The spectral sequence is \emph{bigraded} if there exists a
decomposition $E^r_* =\bigoplus_{k+\ell =*} E^r_{k,\ell}$ for all
$r\in \N$ and the boundary operator $\partial^r$ has \emph{degree}
$(i,j)$ if $\partial^r(E^r_{k,\ell}) \subset E^r_{k+i,\ell+j}$. We
abbreviate a bigraded spectral sequence by $E^*_{**}$. For more
details see \cite{McCleary}.

Let $N$ denote the minimal Maslov number of the pair $(L_0,L_1)$ and
$\tau$ be the monotonicity constant. We decompose $L_0 \cap L_1$ into connected
components $C_1,\dots,C_k$ and choose maps $u_j:[-1,1]\times [0,1] \to
M$, $u(s,\cdot) \in \P(L_0,L_1)$ such that $u_j(-1)= x_1$ and $u_j(1)
= x_j \in C_j$. By concatenating to $u_j$ with path we assume without
loss of generality that the caps $u_p$ for all critical points $p \in
\crit f$ we have caps $u_p$ that satisfy
\begin{equation}\label{eq:fixing}
  \A(u_p) = \int u_p^* \omega \in [0,\tau N)\,.
\end{equation}
We call a pearl trajectory $u$ connecting $p$ to $q$ \emph{local} if
$\mu(u) = \mu(u_p) -\mu(u_q)$. Moreover we define the \emph{local
  pearl chain complex} $CH^\loc_*(L_0,L_1)$ as the free $A$-module
over all critical points $p\in \crit f$ graded by~\eqref{eq:gradp} and
equipped with the boundary operator~\eqref{eq:bdpearl} without the
$\lambda$ factor and summation only over local trajectories. The next
lemma shows that local pearl homology, denoted $QH^\loc(L_0,L_1)$, is
well-defined.
\begin{lmm}\label{lmm:localbreak}
  Let $(u^\nu)$ be a sequence of local pearl trajectories Floer-Gromov
  converging to the broken trajectory $v=(v_1,\dots,v_k)$. Then $v_i$
  is local for all $i = 1,\dots,k$.
\end{lmm}
\begin{proof}
  Let $\bar v_i$ be cap of $v_i(\infty)$ with $0 \leq \int \bar
  v_i^*\omega <\tau N$ for all $i=1,\dots,k$ and $\bar v_0$ be a cap
  of $v_1(-\infty)$ with $0 \leq \int \bar v_0^*\omega < \tau N$.
  Define $k_i:=\mu(v_i) + \mu(\bar v_{i-1}) - \mu(\bar v_i)$ for all
  $i=1,\dots,k$. We have to show that $k_i$ vanishes for all
  $i=1,\dots,k$. Let $m_i$ denote the number of cascades in $v_i$ and
  assume without loss of generality that $m_i\geq 1$ since otherwise
  $v_i$ is local by definition.  By the integer axiom we have $k_i \in
  \Z$ and by monotonicity
  \begin{equation*}
    \tau k_i N = \sum_{j=1}^{m_i} \int v_{i,j}^*\omega + \int \bar v_{i-1}^*\omega -\int
    \bar v_i^*\omega\,.
  \end{equation*}
  Due to the energy condition on the caps and the fact that $v_i$
  consists of holomorphic strips we have $\tau k_i N > -\tau N$.
  Hence $k_i \geq 0$. Again by monotonicity and Floer-Gromov
  convergence we have
  \begin{align*}
    \tau N\sum_{i=1}^k k_i &= E(v) + \int \bar v_0^*\omega -\int
    \bar v_k^*\omega \\
    &= \lim_{\nu \to \infty} E(u^\nu) +  \int \bar v_0^*\omega - \int \bar v_k^*\omega\\
    &=\tau \left(\lim_{\nu \to \infty} \mu_\Vit(u^\nu) + \mu_\Vit(\bar
      v_0) - \mu_\Vit(\bar v_k) \right) = 0\,.
  \end{align*}
  This shows that $\sum_{i=1}^k k_i =0$. Since all $k_i$ are
  non-negative, this shows that $k_i=0$ for all $i=1,\dots,k$.
\end{proof}
For convenience we restate the main theorem in its final form
including the statement about the orientations. Given two Lagrangian
submanifolds $L_0, L_1 \subset M$ such that $L_0$ intersects $L_1$
cleanly. We denote by $C_1,\dots,C_k$ the collection of all connected
components of $L_0 \cap L_1$ that are connected to the path $\base$
within $\P$. If $(L_0,L_1)$ is relative spin, we choose a relative
spin structure. For $1\leq j\leq k$ we denote by $\L_j$ the local
system on $C_j$ associated to the choice of the relative spin
structure (\cf Definition~\ref{dfn:O}). Whenver $(L_0,L_1)$ are not
relative spin and $A=\Z_2$ we denote by $\L_j=\Z_2$ in the next theorem.
\begin{thm}\label{thm:PozlocII}
  Suppose that $L_0,L_1 \subset M$ are closed Lagrangian submanifolds
  in clean intersection such that~\eqref{eq:Assumption} holds. Denote
  by $A$ a commutative unital ring.  If $2A\neq 0$ we require
  addionally that the pair $(L_0,L_1)$ is relative spin and chose a
  relative spin structure. Then Floer homology is well-defined and
  there exists two spectral sequences $E^*_{**}$ and $E^{\loc,*}_{**}$
  such that the following holds.
  \begin{enumerate}[label=(\roman*)] 
  \item The boundary operator of $E^r_{**}$ and $E^{r,\loc}_{**}$ has
    degree $(-r,r-1)$ for all $r \in \N$.
  \item\label{nm:E1Pozloc} The first page of $E^{\loc,*}_{**}$ is
    given by \[E_{ij}^{\loc,1} \cong
    \begin{cases} 
      \bigoplus_{\{\ell |\A(C_\ell)=a_i\}}
      H_{i+j-\mu(C_\ell)}(C_\ell;\L_\ell)&\text{if }    1\leq i \leq \kappa,\ j \in \Z\\
      0&\text{if otherwise}\,.
    \end{cases}
    \] 
  \item\label{nm:E1Poz} The sequence $E^{\loc,*}_{**}$ converges to a
    graded module $HF^\loc_*$ over $A$ and we have $E^1_{**} \cong
    A[\lambda^{\pm 1}] \otimes HF^\loc_*$ with $\deg \lambda =-N$.
  \item\label{nm:E8Poz} The sequence $E^*_{**}$ converges and we
    have \[\bigoplus_{i+j=*} E_{ij}^\infty\cong HF_*(L_0,L_1)\,.\]
    Here the right-hand side denotes the Floer-homology with respect
    to the specific choice of the relative spin structure.
  \end{enumerate}
\end{thm} 
\begin{proof}
  The proof is motivated by \cite[proof of Theorem
  5.2.A]{Biran:nonintersection}.  Let $(C_*,\partial)= (C_*(f) \otimes
  \Lambda,\partial)$ denote the pearl-complex. A spectral sequence is
  canonically determined by an \emph{increasing filtration}, \ie a
  sequence of subcomplexes $(\F^kC_*)_{k \in \Z}$ such that
  \begin{equation}
    \label{eq:increasing}
    \dots \subset \F^{k-1}C_* \subset \F^kC_* \subset \dots \subset C_*,\qquad k \in \Z\;.
  \end{equation}
  We construct a filtration by the degree of the Novikov
  variable. More precisely for every $k \in \Z$ we define the
  free $A$-module
  \[
  \F^k C_* := \< p\otimes \lambda^\ell \mid \ell \geq -k \>\;.
  \]
  Clearly the sequence $(\F^kC_*)_{k \in \Z}$
  satisfies~\eqref{eq:increasing}. To show that $(\F^kC_*)_{k\in \Z}$
  is a filtration, it remains to show that the modules are
  subcomplexes, \ie $\partial \F^k C_* \subset \F^k C_*$ for all $k
  \in \Z$. By $\Lambda$-linearity of the boundary operator $\partial$
  it suffices to check this for $k=0$. Moreover it suffices to check
  this on generators of the form $p\otimes \one$ for some critical
  point $p$. Assume that the coefficient for $\partial p$ in front of
  $q \otimes \lambda^\ell$ is not zero. We need to show that $\ell
  \geq 0$.  By definition there exists a rigid trajectory $u$
  connecting the critical points $p$ to $q$.  We have two cases. In
  the first case $u$ has zero cascades. Then necessarily $p$ and $q$
  lie on the same connected component and $\ell N = \nm{p} - \nm{q} +
  1 =\mu(p) - \mu(q) +1=0$. Hence $\ell =0$ and we are finished. In
  the second case $u=(u_1,\dots,u_m)$ has at least one cascade. Then
  by the dimension formula of Lemma~\ref{lmm:regfreecasc} and the
  definition of the grading~\eqref{eq:gradp} we have with connected
  components $C_-$ and $C_+$ of $p$ and $q$ respectively
  \begin{align*}
    m &= \sum_{j=1}^m \mu(u_j) + \mu(p) - \mu(q) - 1/2 \dim C_- + 1/2 \dim C_+ + m -1\\
    &= \sum_{j=1}^m \mu(u_j) + \mu(u_p) -
    \mu(u_q) + \nm{p} - \nm{q} -1+m\\
    &= \sum_{j=1}^m \mu(u_j) + \mu(u_p) -
    \mu(u_q) -\ell N +m\;.
  \end{align*}
  By monotonicity we have
  \[
  \tau \ell N = \sum_{j=1}^m \int u_j^*\omega + \int
  u^*_p\omega - \int u_q^*\omega\;.
  \]
  Since $u_j$ is $J$-holomorphic $\int u_j^*\omega > 0$ for all $j$
  and since by our choice~\eqref{eq:fixing} we have $\int u_p^* \omega
  \geq 0$ and $\int u_q^*\omega >-\tau N$ we obtain $\tau \ell N >
  -\tau N$.  This shows that $\ell\geq 0$ and that we have truly
  defined a filtration.

  \bigskip We claim the filtration is \emph{bounded}, \ie for every
  $m \in \Z$ there exists $k_-,k_+ \in \Z$ 
  such that
  \[
  0 = \F^{k_-}C_m \subset \F^{k_-+1}C_m \subset \dots \subset
  \F^{k_+ -1}C_m \subset \F^{k_+} C_m = C_m\,,
  \]
  where $\F^kC_m := \<p \otimes \lambda^\ell \mid \nm{p} - \ell N = m,\ \ell
  \geq -k\>$. Indeed, define the integers $k_- := \lfloor (m-\ol
  m)/N\rfloor-1$ and $k_+ := \lceil (m-\ul m)/N\rceil$, where $\ul m:=
  \min_{p \in \crit f} \nm{p}$ and $\ol m := \max_{p \in \crit f}
  \nm{p}$. For every $p \otimes \lambda^\ell \in C_m$ we have \[ m = \nm{p}
  - \ell N \geq \ul m - \ell N \Longrightarrow \ell \geq (\ul m - m)/N \geq -
  k_+\;.\] Hence $p \otimes \lambda^\ell \in \F^{k_+}C_m$ and this shows
  $\F^{k_+}C_m = C_m$. On the other hand arguing indirectly assume
  that there exists $p \otimes \lambda^\ell \in \F^{k_-}C_m$ then
  \[
  m = \nm{p} - \ell N \leq \ol m - \ell N \leq \ol m + k_- N\
  \Longrightarrow\ (m-\ol m)/N\leq k_-\;.  \] This gives the
  contradiction $k_- = \lfloor (m-\ol m)/N \rfloor - 1 < (m-\ol m)/N
  \leq k_-$ and shows $\F^{k_-}C_m = \{0\}$. We have deduced that the
  filtration is bounded.
  
  \bigskip The rest of the proof follows from standard algebraic
  arguments. Our main reference here is~\cite{McCleary}. In
  particular the next result is valid for any complex $(C_*,\partial)$
  equipped with a bounded increasing filtration $\F$ and a boundary
  operator of degree $-1$. For every $p,q\in \Z, r \in \N$ we define
  \begin{align*}
    Z^r_{p,q} &:= \F^pC_{p+q} \cap \partial^{-1} \F^{p-r} C_{p+q-1},\\
    B^r_{p,q} &:= \F^pC_{p+q} \cap \partial \F^{p+r}C_{p+q+1},\\
    E^r_{p,q} &:= Z^r_{p,q}/(Z^{r-1}_{p-1,q+1} + B^{r-1}_{p,q})\;.
  \end{align*}
  A simple computation shows $\partial Z^r_{p,q} = B^r_{p-r,q+r-1}
  \subset Z^r_{p-r,q+r-1}$ and that $\partial$ induces a morphism
  \begin{equation}
    \label{eq:bidegree}
    \partial^r:E^r_{p,q} \to E^r_{p-r,q+r-1}\;.
  \end{equation}
  The proof of~\cite[Theorem 2.6]{McCleary} adapted to this setting
  shows that we obtain a spectral sequence and moreover we have the
  isomorphisms
  \begin{enumerate}[label=\emph{\alph*})]
  \item\label{nm:ss} $E^{r+1}_{p,q} \cong \ker
    \left(\partial^r:E^r_{p,q} \to E^r_{p-r,q+r-1}\right)/\  \im
    \left(\partial^r:E^r_{p+r,q-r+1}\to E^r_{p,q}\right)$,
  \item\label{nm:E1} $E^1_{p,q} \cong
    H_{p+q}(\F^pC_*/F^{p-1}C_*,[\partial])$,
  \item\label{nm:Einfty} $E^\infty_{p,q} \cong \F^pH_{p+q}(
    C_*)/\F^{p-1}H_{p+q}(C_*)$,
  \end{enumerate}
  where $\F^pH_{p+q}(C_*) := \im\big(H_{p+q}(\F^pC_*) \to
  H_{p+q}(C_*)\big)$ and $E^\infty_{p,q}$ denotes $E^r_{p,q}$ with
  sufficiently large $r$. Coming back to our specific case, consider
  the index transformation \[\tilde E^r_{k,\ell} :=
  \begin{cases}
    E^r_{k/N,\ell+(N-1)k/N}&\text{if } k \in N\Z\,,\\
    0&\text{otherwise}\,.  
  \end{cases}
  \] Then by~\eqref{eq:bidegree} we have $\partial^r:\tilde
  E^r_{k,\ell} \to \tilde E^r_{k-Nr,\ell+Nr-1}$. We interpret that as
  the $rN$-th boundary operator setting all other boundary operators
  to zero. This shows that $\tilde E^*_{**}$ is a homological spectral
  sequence (\ie the $r$-th boundary operator has degree
  $(-r,r-1)$). To obtain the first page of $\tilde E^*_{**}$ we
  use~\ref{nm:E1} and compute
  \begin{multline*}
    E^1_{p,q}=H_{p+q}( \F^pC_*/\F^{p-1}C_*,[\partial]) \cong
    H_{p+q}(C_*(f) \otimes \<\lambda^{-p}\>,\partial_0 \otimes \one) \\ \cong
    H_{p+q-p N} (C_*(f),\partial_0) \otimes \<\lambda^{-p}\>\;,
  \end{multline*}
  where $\partial_0:C_*(f) \to C_{*-1}(f)$ is precisely the boundary
  operator of the local pearl complex. Hence $E^1_{p,q} \cong
  QH_{q+(1-N)p}^\loc \otimes \<\lambda^{-p}\> \iff \tilde E^1_{k,\ell}
  = QH_\ell^\loc \otimes \<\lambda^{-p}\>$. This shows
  the statement~\ref{nm:E1Poz} of Theorem~\ref{thm:Pozloc}.

  \bigskip We show statement~\ref{nm:E8Poz} of
  Theorem~\ref{thm:Pozloc}. Abbreviate $H_*:=QH_*(L_0,L_1)$. By
  invariance we have $QH_*(L_0,L_1) \cong HF_*(L_0,L_1)$ (\cf
  equation~\eqref{eq:invariance}).  By~\ref{nm:Einfty} and we obtain
  an isomorphism of $E^\infty_* \cong \bigoplus_p
  \F^pH_*/\F^{p-1}H_*$. It remains to show that the graded module is
  isomorphic to $H_*$ even in the case when the ground ring $A$ is not
  a field. We define the valuation
  \[ \nu:\ker \partial \to \Z \cup \{\infty\}, \qquad z \mapsto
  \begin{cases}
    \infty&\text{if } z =0\\
    \min \{ k \in \Z \mid x_k \neq 0\}&\text{if } z = \sum_k x_k
    \otimes \lambda^k\neq 0\;.
  \end{cases}\] Since $\partial$ is $\Lambda$-linear, the module
  $\ker \partial$ is a $\Lambda$-module.  In particular there is an
  automorphism of $\ker \partial$ given by multiplication with $\lambda$. It
  is immediate from the definition that for all $z \in \ker \partial$
  and $\ell \in \Z$ we have
  \begin{equation}
    \label{eq:valuation}
    \nu(\lambda^\ell z) = \ell + \nu(z)\;.   
  \end{equation}
  For every $p \in \Z$ we define
  \begin{align*}
    Z_p^\infty &:= \{z \in \ker \partial \mid \nu(z) \geq -p\} =
    \ker \partial \cap \F^p C_*\,,\\
    B_p^\infty &:= \{z \in \im \partial \mid \nu(z) \geq -p\} =
    \im \partial \cap \F^p C_*\;.
  \end{align*}
  It is easy to see that $\F^pH_* \cong Z^\infty_p/B^\infty_p$.  We
  abbreviate the quotient $\bar H_* := \F^0H_*/\F^{-1}H_*$ and define
  \[
  \phi:\ker \partial \to \bar H_* \otimes \Lambda,\qquad z \mapsto
  [\lambda^{-\nu(z)}z] \otimes \lambda^{\nu(z)}\;.
  \]
  To check that $\phi$ is well-defined, we need to see that
  $\lambda^{-\nu(z)}z \in Z^\infty_0$. But with~\eqref{eq:valuation} this is
  obvious since $\nu(\lambda^{-\nu(z)}z) = -\nu(z) + \nu(z) =0$. The
  morphism $\phi$ is surjective, since every element of $\bar H_*
  \otimes \Lambda $ is a linear combination of elements of the form
  $[z]\otimes \lambda^\ell$ with $\nu(z) =0$ and such elements have the
  preimage $\lambda^{\ell}z$. The kernel of $\phi$ is given by
  $\im \partial$, because $ z\in \ker \phi \iff \lambda^{-\nu(z)} z \in
  \im \partial \iff z \in \im \partial$. Hence $\phi$ induces an
  isomorphism
  \begin{equation}
    \label{eq:structureQH}
    H_* \cong \bar H_* \otimes \Lambda\;.
  \end{equation}
  Restricting $\phi$ to $Z^\infty_p$ shows that we have the
  isomorphism $\F^p H_* \cong \bar H_* \otimes \F^p\Lambda$, where
  $\F^p\Lambda := \<\lambda^\ell\mid \ell \geq -p\>$. Since the quotient
  $\F^p\Lambda/\F^{p-1}\Lambda = \<\lambda^{-p}\>$ is free we have
  $\F^pH_*/F^{p-1}H_* \cong \bar H_* \otimes \<\lambda^{-p}\>$. Together
  with~\ref{nm:Einfty} and~\eqref{eq:structureQH} we
  obtain the statement.

  \bigskip

  We show statement~\ref{nm:E1Pozloc}.  The spectral sequence is
  constructed again by a filtration. This time we use the local pearl
  complex. Abbreviate by $C_*:=CH_*^\loc(L_0,L_1)$ the local pearl
  complex. Define $\F^jC_* = 0$ if $j\leq 0$ and $\F^jC_* =C_*$ if $j
  \geq \kappa+1$. If $1 \leq j \leq \kappa$ define $\F^jC_* \subset
  C_*$ to be the submodule generated by all critical points $p$ with
  $\A(u_p) \leq a_j$. We need to show that this defines a
  filtration. By construction $\F^{k-1}C_*\subset \F^kC_*$. Given a
  critical point $p$ with $\A(u_p) \leq a_k$ and suppose that the
  coefficient of $\partial p$ in front of $q$ is not zero. We need to
  show that $\A(u_q) \leq a_k$.  There must exists a rigid local pearl
  $u$ from $p$ to $q$. If $u$ has zero cascades then $p$ and $q$ are
  on the same connected component and we are done because then
  $\A(u_q)=\A(u_p)\leq a_k$. If $u=(u_1,\dots,u_m)$ has at least one
  cascade, then since the trajectory is local we have
  \[
  0= \sum_{j=1}^m \mu(u_j) + \mu(u_p) - \mu(u_q) = \tau^{-1}
  \left(\sum_{j=1}^m \int u_j^*\omega - \A(u_p) + \A(u_q)\right)\,.\]
  Since $u_j$ is non-constant and holomorphic we have $\int
  u_j^*\omega >0$ for all $j$ and hence
  \begin{equation}
    \label{eq:Edrop}
    a_k \geq \A(u_p) = \sum_{j=1}^m \int u_j^*\omega + \A(u_q) > \A(u_q)\,. 
  \end{equation}
  We have deduced that $(\F^kC_*)_{k\in\Z}$ truly defines a
  filtration, which is evidently bounded by construction. As above we
  obtain a spectral sequence $E^{\loc,*}_{**}$ with first page given
  by
  \[E^{\loc,1}_{k,\ell}\cong
  H_{k+\ell}(\F^kC_*/\F^{k-1}C_*,[\partial])\,.\] The complex
  $\F^kC_*/\F^{k-1}C_*$ is generated by critical points $p$ such that
  $\A(u_p)=a_k$ and the boundary operator $[\partial]$ of $p$ is
  $\partial p$ projected to $\F^kC_*/\F^{k-1}C_*$ (\ie we forget any
  critical points of lower action). Suppose that there exists a
  non-trivial contribution of $[\partial]p$ in front of $q$. Then by
  definition there exists a trajectory $u$ connecting $p$ to $q$. If
  $u$ has at least one cascade we know by estimate~\eqref{eq:Edrop}
  that $\A(u_q) \neq \A(u_p)$. Hence trajectories connecting critical
  points with the same action value are only Morse trajectories and
  hence $[\partial]$ is counting standard Morse trajectories. Taking
  our orientation algorithm (\cf paragraph before
  Lemma~\ref{lmm:coherent}) and the degree-shift into account shows
  the claim.
\end{proof}

\appendix

\section{Estimates}
\subsection{Derivative of the exponential map}\label{sec:dexp}
In the section we have collected estimates for the derivative of the
exponential map of the Levi-Civita connection. These results are
well-known, yet we have always included the proofs, since we have not
found a good reference. Let $M$ be a compact Riemannian manifold with
Levi-Civita connection $\na$. The connection induces a splitting of
the tangent space $T_\xi(TM)$ at $\xi \in T_pM$ into horizontal and
vertical space and we define the horizontal and vertical lift
\[
L^\hor(\xi): T_pM \to T_\xi^\hor(TM),\qquad L^\ver(\xi):T_pM \to
T_\xi^\ver (TM)\,.\] Associated to the connection is an exponential
map $\exp:TM \to M$. Using the horizontal and vertical lifts we define
the horizontal and vertical differential of the exponential map at
some $\xi \in T_pM$ 
\begin{align*}
    E_p^\hor(\xi) := \d_\xi \exp \circ L^h(\xi)&: T_pM
    \longrightarrow T_{\exp(\xi)} M\,,\\
    E_p^\ver(\xi) := \d_\xi \exp \circ L^v(\xi)&: T_p M
    \longrightarrow T_{\exp(\xi)} M\,.
\end{align*}
Given a smooth curve $u:(a,b) \to M$ and a smooth vector field $\xi
\in \Gamma(u^*TM)$ along $u$, we write $u_\xi = \exp_u \xi:(a,b) \to
M$ where $u_\xi(x)=\exp_{u(x)}\xi(x)$. With the above definitions we
have
\begin{equation}
  \label{eq:dwxi}
  \px u_\xi = E^\hor_u(\xi) \px u + E^\ver_u(\xi) \na_x \xi\;.  
\end{equation}
\begin{prp}\label{prp:dexp}
  For all $\e>0$ there exists an universal constant $c$ with the
  following significance:
  \begin{itemize}
  \item Given vectors $\xi,\xi' \in T_p M$
    with $\nm{\xi} < \e$, we have the estimates
    \[
    \nm{E_p(\xi)\xi'} \leq c\nm{\xi'},\qquad
    \nm{E_p(\xi)\xi'-\Pi_p(\xi)\xi'} \leq c \nm{\xi}\nm{\xi'}\,,
    \]
  \item Let $u:(a,b) \to M$ be a smooth curve and given vector
    fields $\xi,\xi' \in \Gamma(u^*TM)$ such that $\Nm{\xi}_{L^\infty}
    <\e$ then we have the estimates
    \begin{equation}
      \label{eq:commutenaE}
      \nm{\na_x E_u(\xi)\xi' - E_u(\xi)\na_x \xi'} \leq c \nm{\xi'} \nm{\xi}
      \left( \nm{\px u} + \nm{\na_x \xi}\right)\,,   
    \end{equation}
  \end{itemize}
  where $\Pi_p(\xi):T_pM \to T_{\exp(\xi)}M$ is the parallel transport
  along the geodesic curve $y \mapsto \exp_p(y \xi)$ and $E_p(\xi)$
  denotes either $E^\ver_p(\xi)$ or $E^\hor_p(\xi)$.    
\end{prp}
\begin{proof} 
  By~\eqref{eq:dwxi} the vector field $Y(y):=E_p(y\xi)\xi'$ is a
  Jacobi vector field along the geodesic $c:[0,1] \to M$, $y \mapsto
  \exp_p(y \xi)$, \ie solves the equation $\na_y \na_y Y = R(\dot
  c,Y)\dot c$ where $R$ denotes the curvature tensor. Given any Jacobi
  field $Y$, we define the function $f:[0,1] \to \R$, $y \mapsto
  \nm{Y(y)} + \nm{\na_y Y(y)}$. We have
  \[
  f'(y) \leq \nm{\na_y Y} + \nm{\na_y \na_y Y} \leq \nm{\na_y Y} +
  \Nm{R}_\infty \nm{\xi}^2 \nm{Y} \leq (1+ \Nm{R}_\infty \e^2) f\,.
  \]
  Hence $f(y) \leq c_1 f(0)$ with constant
  $c_1:=e^{(1+\Nm{R}\e^2)}$ and so
  \begin{equation}
    \label{eq:Jacfieldest}
    \nm{Y(1)}+\nm{\na_y Y(1)} \leq c_1 (\nm{Y(0)} + \nm{\na_y Y(0)})\,.
  \end{equation}
  Since the estimate holds for any Jacobi field $Y$ we have in
  particular the estimates $\nm{E^\hor(\xi)\xi'} \leq c_1 \nm{\xi'}$ and
  $\nm{E^\ver(\xi)\xi'} \leq c_1 \nm{\xi'}$ as required.

  We show the second inequality. We define the vector field $X \in
  \Gamma(c^*TM)$ via
  \[X(y):=\Pi_p(y \xi) Y(0)+ y\Pi_p(y\xi)\na_y Y(0)\,.\] Consider the
  function
  \[f:[0,1] \to
  \R,\ f(y) := \nm{Y(y) - X(y)} + \nm{\na_y Y(y) - \na_y X(y)} +
  c_1 \e \Nm{R} \nm{\xi}\nm{\xi'}\,.\]
  We derive 
  \[
  f'(y) \leq \nm{\na_y Y - \na_y X} + \nm{\na_y \na_y Y} \leq
  \nm{\na_y Y - \na_y X} + \nm{R(\dot c,Y)\dot c } \leq f(y)\,.
  \]
  This shows that $\nm{E_p(\xi)\xi' - \Pi_p(\xi)\xi'} \leq f(1) \leq e
  f(0) = e c_1 \e \Nm{R} \nm{\xi}\nm{\xi'}$.

  We come to the third inequality. Define the map
  $w(x,y):=\exp_{u(x)}y \xi(x)$ and the family of geodesics
  $c_x:=w(x,\cdot)$. The vector fields $Y(x,y)=E(y\xi(x))\xi'(x)$ and
  $Z(x,y):=E(y\xi(x))\na_x \xi'(x)$ are vector fields along $w$ which
  are Jacobi fields when restricted to $c_x$. We claim that there
  exists a uniform constant $c_2$ such that
  \begin{equation}
    \label{eq:yysYyyZ}
    \nm{\na_y\na_x\na_y Y -\na_y\na_y Z} \leq c_2 \nm{\xi}\nm{\xi'}(\nm{\px u} + \nm{\na_x \xi}) + \e^2 \Nm{R}  \nm{\na_x Y - Z}\;.
  \end{equation}
  Indeed use~\eqref{eq:Jacfieldest} to show in particular that
  $\nm{\px w} + \nm{\na_y \px w} \leq c_1 (\nm{\px u} +
  \nm{\na_x\xi})$ and $\nm{Y} + \nm{\na_y Y} \leq c_1
  \nm{\xi'}$. Abbreviate $R(\px,\py) = R(\px w,\py w)$ \etc and
  estimate
  \begin{align*}
    & \nm{\na_y\na_x\na_y Y-\na_y\na_y Z}\\
    & \leq \nm{R(\py,\px)\na_y Y} +
    \nm{\na_xR(\py,Y)\py w - R(\py,Z)\py w}\\
    & \leq \nm{R(\py,\px )\na_y Y}+
    \nm{\na_xR(\py ,Y)\py w - R(\py ,\na_x Y)\py w}+\nm{R(\py,\na_x Y-Z)\py w}\\
    &\leq \Nm{R}\nm{\xi}\nm{\px w}\nm{\na_y Y} + \Nm{\na R}\nm{\xi}^2
    \nm{\px w} \nm{Y}+ 2 \Nm{R}\nm{\xi}\nm{\na_y \px w}\nm{Y}+\\
    &\hspace{9cm}+
    \Nm{R}\nm{\xi}^2 \nm{\na_x Y -Z }\\
    & \leq c_2 \nm{\xi}\nm{\xi'}(\nm{\px u}+\nm{\na_x \xi}) + \e^2
    \Nm{R} \nm{\na_x Y -Z} \,.
  \end{align*}
  Define the function $f:[0,1] \to M$, where $c_3:=c_2+c_1^2\e
  \Nm{R}$,
  \[ f(y):=\nm{\na_x Y- Z} + \nm{\na_x\na_yY -\na_yZ} +
  c_3\nm{\xi}\nm{\xi'}(\nm{\px u} + \nm{\na_x \xi})\,.\] We derive
  using~\eqref{eq:yysYyyZ}
  \begin{align*}
    f'(y) &\leq \nm{\na_y\na_x Y-\na_y Z} + \nm{\na_y\na_x\na_y
      Y-\na_y\na_y Z}\\
    &\leq \nm{R(\py,\px)Y} + \nm{\na_x \na_y Y -\na_y Z} + c_2
    \nm{\xi}\nm{\xi'}(\nm{\px u} + \nm{\na_x\xi}) +\\
    &\hspace{8cm}+ \Nm{R}\e^2\nm{\na_x Y -Z}\\
    &\leq (1 + \Nm{R}\e^2)f(y)\,.
  \end{align*}
  This shows that $\nm{\na_x Y - Z} \leq f(1) \leq c_1 f(0) =c_1c_3 \nm{\xi}\nm{\xi'}(\nm{\px u}+ \nm{\na_x \xi})$.  
\end{proof}
\begin{cor}\label{cor:dwxi}
  There exists universal constants $c$ and $\e$ with the following
  significance. Given a smooth curve $u:(a,b) \to M$ and a vector
  field $\xi \in \Gamma(u^*TM)$ such that $\Nm{\xi}_{L^\infty} \leq
  \e$ we have
  \begin{align}
    \nm{\partial_x u_\xi} &\leq c \left(\nm{\partial_x u} +
      \nm{\nabla_x
        \xi} \right)\,,\label{eq:pxuxi}\\
    \nm{\nabla_x \xi} &\leq c
    \left(\nm{\partial_x u} + \nm{\partial_x u_\xi}\right)\,,\label{eq:nxxi}\\
    \nm{\px u_\xi - \Pi_u^{u_\xi} \px u} &\leq c \left(\nm{\px
        u}\nm{\xi} + \nm{\na_x \xi}\right)\label{eq:pxuxiPipxu}\,.
  \end{align}
  Moreover for all vector fields $\xi,\xi' \in \Gamma(u^*TM)$ with
  $\Nm{\xi}_\infty + \Nm{\xi'}_\infty < \e$ we have
  \begin{align}
    \nmm{\exp_{u_\xi}^{-1} \exp_u \xi' - \Pi_u^{u_\xi} \xi'}&\leq c
    \nm{\xi},\label{eq:expPi} \\
    \nmm{\na_x \exp_{u_\xi}^{-1} \exp_u \xi' - \Pi_u^{u_\xi} \na_x
      \xi'} &\leq c \left( \nm{\px u} \nm{\xi} + \nm{\na_x
        \xi}\right)\label{eq:naexpPi}\;.
  \end{align}
\end{cor}
\begin{proof}
  Estimate~\eqref{eq:pxuxi} follows by the first inequality of
  Proposition~\ref{prp:dexp} and~\eqref{eq:dwxi}. We show the
  estimate~\eqref{eq:nxxi}. By the second inequality of
  Proposition~\ref{prp:dexp} we have an universal constant $c_1$ such
  that
  \[\nmm{\one- \Pi_{u_\xi}^u E^\ver_u(\xi)} = \nm{\Pi_u^{u_\xi}-E^\ver_u(\xi)} \leq c_1 \nm{\xi}\;.\]
  If $\nm{\xi}<1/2c_1$ then the operator $E^\ver_u(\xi)$ is invertible
  with inverse $\sum_{k\geq 0} (1-\Pi_{u_\xi}^u
  E^\ver_u)^k\circ\Pi_{u_\xi}^u$ which is bounded by
  $2$. By~\eqref{eq:dwxi} we have 
  \[\nm{\na_x \xi} =
  \nm{(E^\ver_u)^{-1}E^\hor_u\ps u - (E^\ver_u)^{-1}\ps u_\xi} \leq
  2c_1\nm{\ps u} + 2\nm{\ps u_\xi}\,.\]
  This shows the claimed bound.

  Estimate~\eqref{eq:pxuxiPipxu} follows because after
  Proposition~\ref{prp:dexp} we estimate the norm of $\px u_\xi -\Pi_u^{u_\xi}\px u$ by
  \[
  \nm{E^\hor_u(\xi)\px
    u - \Pi_u^{u_\xi}\px u} + \nm{E^\ver_u(\xi)\na_x \xi}\,,
  \]
  which is bounded by $O(1)\nm{\ps u} \nm{\xi} +  O(1)\nm{\na_x \xi}$.
  
  For estimate~\eqref{eq:expPi} we define the curve $w:(a,b)\times
  [0,1] \to M$, via $w(x,y):=\exp_{u(x)}y \xi(x)$. If $\xi',\xi$ are
  sufficiently small we define implicitly a vector field $\zeta$ along
  $w$ via
  \[
  \exp_{u(x)}\xi'(x) = \exp_{w(x,y)} \zeta(x,y)\,,
  \]
  for all $x \in (a,b)$ and $y \in [0,1]$.  Deriving the last equation
  by $\py$ using~\eqref{eq:dwxi} gives
  \begin{equation}
    \label{eq:nayzeta}
     E^\hor_w(\zeta)\py w + E^\ver_w(\zeta)\na_y \zeta=0\,.
   \end{equation}
   Fix $x \in (a,b)$ and define the function 
   \[f:[0,1] \to \R,\qquad y \mapsto \nmm{\zeta(x,y) -
     \Pi_{u(x)}^{w(x,y)} \xi'(x)}\,.\] By construction we have
   $w(x,0)=u(x)$, $w(x,1) =u_\xi(x)$, $\zeta(x,0) =\xi'(x)$ and
   $\zeta(x,1) = \exp_{u_\xi(x)}^{-1}\exp_{u(x)}\xi'(x)$ for all $x
   \in (a,b)$, which implies that $f(0)=0$ and that we have to
   estimate $f(1)$.  We compute using~\eqref{eq:nayzeta} and the
   mean-value theorem omitting the arguments $x,y$ whenever convenient
   \[\nmm{\exp_{u_\xi}^{-1} \exp_u\xi'-\Pi_u^{u_\xi}\xi'}=f(1) = \py f(y) \leq \nm{\na_y \zeta} =
   \nmm{E^{\ver}_w(\zeta)^{-1}E^\hor_w(\zeta)\py w} \,,\]
   which is in particular bounded by $O(1)\nm{\xi}$.

  We show~\eqref{eq:naexpPi}. By definition we have $\zeta(1) =
  \exp_{u_\xi}^{-1}\exp_u \xi'$ and after the mean-value theorem
  \begin{equation*}\label{eq:naexpPistart}
    \nm{\na_x \zeta- \Pi_u^w \na_x \xi'} = \py \nm{ \na_x \zeta-\Pi_u^w \na_x \xi'} \leq\nm{\na_y \na_x \zeta} \leq \nm{R(\py w,\px w) \zeta} + \nm{\na_x \na_y \zeta}\,.
  \end{equation*}
  By~\eqref{eq:pxuxi} the first term on the right-hand side is in
  $O(\nm{\px u}\nm{\xi} + \nm{\na_x \xi})$. Hence it suffices to
  estimate $\nm{\na_x\na_y \zeta}$. Abbreviate
  $E_w^\ver:=E^\ver_w(\zeta)$ and estimate
  \begin{multline*}
    \nm{\na_x\na_y \zeta} \leq O(1)\nm{E^\ver_w\na_x\na_y \zeta}
    \leq \\ \leq O(1)\nm{\na_x E^\ver_w\na_y \zeta}+ O(1) \nm{(\na_x E^\ver_w -E^\ver_w \na_x)\na_y \zeta}
  \end{multline*}
  With~\eqref{eq:nayzeta} we have $\nm{\na_y \zeta} \leq O(1)$.
  By~\eqref{eq:commutenaE} the second term on the right-hand side
  of the last estimate is in $O(\nm{\px u}\nm{\xi} + \nm{\na_x \xi})$.
  We continue to estimate the first. Using~\eqref{eq:nayzeta} again we
  have
  \begin{equation*}
    \nm{\na_x E^\ver_w\na_y \zeta}  = \nm{\na_x E^\hor_w\partial_y w} \leq  \nm{(\na_x E^\hor_w -E^\hor_w\na_x)\partial_y w} + O(1)\nm{\na_x \partial_y w}\,.
  \end{equation*}
  Again by~\eqref{eq:commutenaE} the first term is in
  $O(\nm{\px u}\nm{\xi} + \nm{\na_x \xi})$ and it suffices to bound
  $\nm{\na_x \partial_y w}=\nm{\na_y \px w}$.
  \begin{align*}
    \nm{\na_y \px w} &\leq \nmm{\na_y E^\hor_w(y\xi)\px u} + \nm{\na_y
      E^\ver_w(y \xi)(y\na_x \xi)}  \\
    &\leq O(1) \nm{\xi}\nm{\px u} +
    O(1)\nm{\xi}\nm{\na_x \xi} +\nmm{E^\ver_w(y \xi)\na_x
      \xi}\\
    &\leq O(1)(\nm{\xi}\nm{\px u} + \nm{\na_x \xi})\,.
  \end{align*}
  This shows the claim using the last four estimates.
\end{proof}
\begin{cor}\label{cor:explip}
  There exists constants $c$ and $\e$ with the following
  significance. Given a point $p \in M$ and two vectors $\xi_0,\xi_1
  \in T_pM$ such that $\nm{\xi_0}+\nm{\xi_1} <\e$, then we have
  \begin{equation}\label{eq:lip}
    1/c\nm{\xi_0-\xi_1} \leq \di{\exp_p(\xi_0),\exp_p(\xi_1)} \leq c
    \nm{\xi_0-\xi_1}\,.
  \end{equation}
  Moreover let $u:[a,b] \to M$ be a curve and $\xi,\xi'\in
  \Gamma(u^*TM)$ be a vector field along $u$ with
  $\Nm{\xi}_\infty+\Nm{\xi'}_\infty<\e$, then
  \begin{equation}
    \label{eq:C1lipxi}
    \nmm{\na_x \exp_{u_{\xi'}}^{-1} u -\na_x \exp_{u_{\xi'}}^{-1}u_\xi} \leq c\big( \nm{\xi}\,\nm{\px u} + \nm{\xi}\,\nm{\na_x \xi'} +\nm{\na_x \xi}\big)\,.
  \end{equation}
\end{cor}
\begin{proof}
  To show the estimate on the right-hand side of~\eqref{eq:lip} set
  $\xi(x) = (1-x)\xi_0 + x\xi_1$ and $u(x) = p$ for all $x \in
  [0,1]$. Then the inequality follows after integrating the norm of
  $\px u_\xi$ over $[0,1]$ and using estimate~\eqref{eq:pxuxi} of
  Corollary~\ref{cor:dwxi}.

  To show the estimate on the left-hand side of~\eqref{eq:lip} let
  $c:[0,1]\to M$ be the unique shortest geodesic from $\exp_p\xi_0$ to
  $\exp_p\xi_1$. We define the path of vectors $\xi:[0,1] \to T_pM$
  via $\xi(y) := \exp_p^{-1} c(y)$. By the mean-value theorem we have
  a value $y \in [0,1]$ such that $\xi_1 -\xi_0 = \na_y \xi(y)$ and
  hence after~\eqref{eq:nxxi} there exists a uniform constant $c_1$
  such that $ \nm{\xi_1 - \xi_0} = \nm{\na_y \xi(y)} \leq c_1\nm{\py
    c} = c_1 \di{\exp_p \xi_0,\exp_p \xi_1}$.

  We show~\eqref{eq:C1lipxi}. Abbreviate the curve $v:=\exp_u \xi'$ and
  define implicitly the vector field $\zeta:[a,b]\times [0,1]\to
  v^*TM$ via
  \[\exp_{v(x)}\zeta(x,y) = \exp_{u(x)} y \xi(y)\,.\]
  Deriving the equation by $\na_x$ and by $\na_y$ we get
  \begin{equation}
    \label{eq:zetapxpy}
    \begin{aligned}
    E_v^\ver(\zeta) \na_x \zeta &= E_v^\hor(y\xi)\px u + y E_u^\ver(y\xi)\na_x \xi - E_v^\hor(\zeta)\px v\\
       E_v^\ver(\zeta)\na_y \zeta &= E^\ver_u(y\xi) \xi\,.
    \end{aligned}
  \end{equation}
  By construction we have $\zeta_0:= \zeta(\cdot,0) = \exp_{v}^{-1}u$
  and $\zeta_1:=\zeta(\cdot,1) = \exp_{v}^{-1}u_\xi$. Hence after the
  mean-value theorem 
  \begin{equation*}
    \nmm{\na_x \exp_v^{-1} u -\na_x \exp_v^{-1} u_\xi} =\nm{\na_y \na_x \zeta} \leq \nm{\na_x \na_y \zeta} + \nm{R(\py u,\px u)\zeta}\,.
  \end{equation*}
  Since $u$ does not depend on $y$ the last term vanishes and we are
  left to estimate the norm of $\na_x\na_y \zeta$. Abbreviate
  $E_v=E_v^\vert(\zeta)$, then 
  \begin{equation*}
    \nm{\na_x\na_y \zeta} \leq O(1) \nm{E_v \na_x \na_y \zeta} \leq O(1) \nm{(E_v \na_x -\na_xE_v) \na_y \zeta} + O(1)\nm{\na_x E_v \na_y \zeta}\,.
  \end{equation*}
  Via~\eqref{eq:commutenaE} the first term on the right-hand side is
  bounded by
  \begin{multline*}
    \nm{(E_v \na_x -\na_xE_v) \na_y \zeta} \leq O(1) \nm{\na_y \zeta} \nm{\zeta}(\nm{\px v} + \nm{\na_x \zeta}) \leq \\ \leq O(1)\nm{\xi}(\nm{\px u} + \nm{\na_x \xi'} + \nm{\na_x \xi})\,.
  \end{multline*}
  For the last estimate we have used~\eqref{eq:zetapxpy}. To show the
  claim it suffices to estimate the norm of $\na_xE_v \na_y
  \zeta=\na_xE^\ver_u(y \xi) \xi$. Abbreviate $E^\ver_u(y\xi)=E_u$ and
  estimate
  \begin{equation*}
    \nm{\na_x E_u\xi} \leq \nm{(\na_x E_u -E_u \na_x)\xi} + \nm{E_u \na_x \xi} \leq O(1)\nm{\xi}(\nm{\px u} + \nm{\na_x \xi}) + O(1) \nm{\na_x \xi}\,.
  \end{equation*}
  This shows the claim using the last three estimates.
\end{proof}
The next corollary states that the distance between parallel geodesics
is uniformly bounded by the distance of their starting point.
\begin{cor}\label{cor:estparallelgeodesics}
  There exists positive constants $\e$ and $c$ such that given points
  $p,q\in M$ and a vector $\xi \in T_pM$ satisfying 
  $\di{p,q}+\nm{\xi} \leq \e$ then we have
  \[ \di{\exp_p \xi,\exp_q\Pi_p^q \xi} \leq c\, \di{p,q}\;.\]
\end{cor}
\begin{proof}
  Let $u:[0,1]\to M$ be the unique shortest geodesic from $p$ to
  $q$. Extend $\xi$ to a parallel vector field along $u$.  Integrate
  the estimate~\eqref{eq:pxuxi} of Corollary~\ref{cor:dwxi} over
  $[0,1]$, using $\na_x \xi=0$ and that $|\partial_x u| = \di{p,q}$.
\end{proof}
\subsection{Parallel Transport}
Let $M$ be a compact Riemannian manifold equipped with a metric
connection $\na$. For any curve $\gamma:[a,b]\to M$, let
\[\Pi(\gamma):T_{\gamma(a)}M\to T_{\gamma(b)}M\,,\] denote the parallel
transport along $\gamma$ with respect to the connection $\na$. 
\begin{lmm}\label{lmm:commutePiPi}
  There exists a constant $c$ such that for any map $w:[0,1]^2 \to M$
  we have
  \begin{equation*}
    \nm{\Pi(\gamma_1)\Pi(u_0) - \Pi(u_1)\Pi(\gamma_0)} \leq
    c \int_{[0,1]^2} \nm{\px w(x,y)}\nm{\py w(x,y)} \d x
    \d y \;.
  \end{equation*}
  with curves $u_\tau = w(\tau,\cdot)$ and $\gamma_\tau=w(\cdot,\tau)$
  for $\tau =0,1$.
\end{lmm}
\begin{proof}
  Fix $\xi_0 \in T_{w(0,0)}M$ and define vector fields $\xi,\eta \in \Gamma(w^*TM)$ along $w$ such that 
  \begin{align*}
      \nabla_x \xi(x,y) &= 0\,, &  \nabla_x \eta(x,0) &= 0\,,\\
      \nabla_y \xi(0,y) &= 0\,, &   \nabla_y \eta(x,y) &= 0\,,\\
      \xi(0,0) &= \xi_0\,,&        \eta(0,0) &=\xi_0\,.
  \end{align*}
  We have to estimate the norm of $\xi(1,1) - \eta(1,1)$. Let $R$ be
  the curvature tensor. We have
  \[ \partial_y \nm{ \nabla_x \eta(x,y)} \leq \nm{ \nabla_y \nabla_x
    \eta} = \nm{ R(\partial_x w,\partial_yw)\eta} \leq \Nm{R}_\infty
  \nm{\partial_x w } \nm{ \partial_y w} \nm{\xi_0} \,.\] Then
  $\nabla_x \eta(x,0) = 0$ and by integrating the last inequality we
  obtain
  \[\nm{\nabla_x \eta(x,y)} \leq \Nm{R}_\infty \nm{\xi_0} \int_0^1 \nm{\partial_x w(x,y)}
  \nm{\partial_y w(x,y)} \d y\,.\] We have $\px \nm{\xi - \eta} \leq
  \nm{\na_x \eta}$ and integrate again using the last estimate and
  $\xi(0,y)-\eta(0,y)=0$ we show the claim.
\end{proof}
\begin{lmm}\label{lmm:commutenaPi}
  There exists a constant $c$ such that for any curve $w:[0,1]^2 \to
  M$ and section $\xi \in \Gamma(w(\cdot,0)^*TM)$ we have
  \begin{equation*}
    \nm{\na_x \Pi(\gamma_x) \xi - \Pi(\gamma_x)\na_x \xi} \leq c \nm{\xi}
    \int_0^1 \nm{\px w(x,y)}\nm{\partial_y w(x,y)} \d y\,,
  \end{equation*} 
  with curves $\gamma_x = w(x,\cdot)$ for $x \in [0,1]$.
\end{lmm}
\begin{proof}
  Define $\xi,\eta \in \Gamma(w^*TM)$ via
  \begin{align*}
    \na_y \xi(x,y) &= 0,&\na_y \eta(x,y) &=0\,,\\
    \xi(x,0)&=\xi(x),&\eta(x,0)&=\na_x \xi(x)\,.
  \end{align*}
  We compute
\[
     \partial_y \nm{\na_x \xi - \eta} \leq
  \nm{\na_y\na_x \xi} = \nm{R(\partial_y
    w,\px w)\xi} \leq \Nm{R}_{\infty}
  \nm{\partial_x w} \nm{\partial_y w} \nm{\xi}\,.\]
  Since $\na_x\xi(x,0) = \eta(x,0)$ the result follows by integration.
\end{proof}
\begin{cor}\label{cor:commutePiPi}
  There exists uniform constants $\e$ and $c$ such that for all curves
  $u:[0,1]\to M$ and vector fields $\xi \in \Gamma(u^*TM)$ with
  $\Nm{\xi}_\infty< \e$ we have
  \[
  \nm{\Pi(\gamma_1)\Pi(u_0) - \Pi(u_1)\Pi(\gamma_0)}
  \leq c \int_0^1\nm{\xi} (\nm{\px u} + \nm{\na_x \xi })\d x\,.
  \]
  with curves $\gamma_x,u_y:[0,1] \to M$ given by
  $\gamma_x(y)=u_y(x):=\exp_{u(x)} y \xi(x)$ for all $x,y\in [0,1]$.
  In particular if $u_0$ and $u_1$ are short geodesics we have
  \[
  \nm{\Pi(\gamma_1)\Pi(u_0) - \Pi(u_1)\Pi(\gamma_0)} \leq c
  \left(\di{u_0(0),u_0(1)} + \di{u_1(0),u_1(1)}\right)\,.
  \]
\end{cor}
\begin{proof}
  Define $w(x,y):=\gamma_x(y)$. We have $\nm{\pa_y w} = \nm{\xi}<\e$.
  By Corollary~\ref{cor:dwxi} there exists constants $c$ and $\e$ such
  if $\nm{\xi}<\e$ we have $\nm{\pa_x w} \leq c (\nm{\px u} + \nm{\na_x
    \xi})$ and $\nm{\px w} \leq c(\nm{\px u_0} + \nm{\px u_1})$ .
  Then conclude by Lemma~\ref{lmm:commutePiPi}.
\end{proof} 
\begin{cor}\label{cor:commutenaPi}
  There exists constants $\e$ and $c$ such that for all curves
  $u:[0,1] \to M$ and vector fields $\xi,\xi' \in \Gamma(u^*TM)$ with
  $\Nm{\xi}_\infty <\e$ we have
  \[
  \nm{\na_x\Pi(\gamma_x) \xi' - \Pi(\gamma_x)\na_x \xi'} \leq c
  \nm{\xi}\nm{\xi'} (\nm{\px u} + \nm{\na_x \xi})\,,\] with curves
  $\gamma_x:[0,1]\to M$ given by $\gamma_x(y):=\exp_{u(x)}y \xi(x)$
  for $x,y \in [0,1]$.
\end{cor}
\begin{proof}
  Define $w(x,y):=\gamma_x(y)$. Conclude by
  Lemma~\ref{lmm:commutenaPi}, $\nm{\py w} = \nm{\xi}$ and
  Corollary~\ref{cor:dwxi}.
\end{proof}

\subsection{Estimates for strips}
For a smooth map $u:\R\times [0,1]\to M$ consider the differential
operator $\F_u$ and its linearization $D_u$ given by
equation~\eqref{eq:FuCRX} and~\eqref{eq:Du} respectively. In this
section we establish auxiliary estimates for these operators. We
assume for simplicity that $X \equiv 0$.
\subsubsection{Pointwise estimates}
For the following estimate $\Nm{J}_{C^2}$ denotes the $C^2$-norm of
the tensor $J$ using the induced norm on $\End(TM)$ coming from the
fixed Riemannian metric on $M$. All universal constants are independent
of $J$.
\begin{lmm}\label{lmm:commuteDPi}
  There exists universal constants $c$ and $\e$ such that for all
  smooth maps $u:\R \times [0,1] \to M$ and vector fields $\xi \in
  \Gamma(u^*TM)$ with $\Nm{\xi}_\infty<\e$ we have
  \begin{equation}
    \label{eq:commuteDPi}
    \nm{D_{u_\xi} \Pi_u^{u_\xi} \xi' - \Pi_u^{u_\xi} D_{u} \xi'} \leq
    c(1+\Nm{J}_{C^2}) \left(\nm{\xi}\nm{\xi'}\nm{\d u} + \nm{\na \xi}\nm{\xi'} + \nm{\xi}\nm{\na \xi'} \right)\,.
  \end{equation}
  with $u_\xi:\R\times [0,1]\to M$ defined by
  $u_\xi(s,t)=\exp_{u(s,t)}\xi(s,t)$.
\end{lmm}
\begin{proof}
  We have to estimate the norm of
  \begin{align*}
    D_{u_\xi} \Pi_u^{u_\xi} \xi'-\Pi_u^{u_\xi} D_u \xi' &=\ns \Pi_u^{u_\xi} \xi' - \Pi_u^{u_\xi}\ns \xi' + J(u_\xi)\nt \Pi_u^{u_\xi} \xi' - \Pi_u^{u_\xi} J(u)\nt \xi'\\
    &\qquad + \left(\na_{\Pi_u^{u_\xi}\xi'} J(u_\xi)\right)\pt
      u_\xi  - \Pi_u^{u_\xi}\left(\na_{\xi'}
      J(u)\right)\pt u
  \end{align*}
  Denote the norm of the successive differences of the right-hand side
  by $T_1,T_2$ and $T_3$.  By Corollary~\ref{cor:commutenaPi} we have
  a constant $\e$ such that if $\nm{\xi} < \e$ then
   \[ T_1 \leq O(1)\nm{\xi}\nm{\xi'}\left(\nm{\ps u} + \nm{\ns
        \xi}\right)\,.\]
    Similarly
    \begin{align*}
      T_2 &\leq \nm{J(u_\xi)\nt \Pi_u^{u_\xi} \xi' -
        J(u_\xi)\Pi_u^{u_\xi} \nt \xi'}+ + \nm{J(u_\xi)\Pi_u^{u_\xi}\nt
        \xi' - \Pi_u^{u_\xi}
        J(u)\nt \xi'} \\
      &\leq \Nm{J}_\infty \nm{\nt \Pi_u^{u_\xi} \xi'-\Pi_u^{u_\xi}\nt
        \xi'} + \Nm{\na J}_\infty \nm{\nt \xi'}\nm{\xi}\\
      &\leq O(1)\Nm{J}_{C^2}\nm{\xi}\nm{\xi'}\left(\nm{\pt u} + \nm{\nt \xi}\right)
      + \Nm{J}_{C^2}\nm{\xi}\nm{\nt \xi'}\,,
    \end{align*}
    and
    \begin{align*}
      T_3 &\leq \nm{\left(\na_{\Pi_u^{u_\xi}\xi'}
          J(u_\xi)\right)\left(\pt u_\xi - \Pi_u^{u_\xi} \pt u
        \right)}+\\ &\hspace{2cm}
      +\nmm{\left(\na_{\Pi_u^{u_\xi}\xi'}J(u_\xi)\right)\Pi_u^{u_\xi}
        \pt u - \Pi_u^{u_\xi} \left(\na_{\xi'}
          J(u)\right) \pt u}\\
      &\leq \Nm{\na J}_\infty \nm{\xi'}\nm{\pt u_\xi -
        \Pi_u^{u_\xi}\pt u} + \Nm{\na(\na J)}_\infty
      \nm{\xi'}\nm{\xi} \nm{\pt u} \\
      &\leq O(1)\Nm{J}_{C^2}\left(\nm{\pt u} \nm{\xi}\nm{\xi'} +
        \nm{\xi'}\nm{\nt \xi}\right)\,.
 \end{align*}  
 This shows the claim.
\end{proof}
\begin{lmm}\label{lmm:estdFDu} There exists universal constants $c$
  and $\e$ such that for all smooth $u:\R \times [0,1] \to M$ we have
  and vector fields $\xi,\xi' \in \Gamma(u^*TM)$ with
  $\Nm{\xi}_\infty<\e$ we have
  \[\nm{\d \F_u(\xi)\xi'-D_u\xi'} \leq c(1+\Nm{J}_{C^2}) \left(\nm{\d u}\nm{\xi}\nm{\xi'} + \nm{\na \xi} \nm{\xi'} + \nm{\xi}\nm{\na
      \xi'}\right)\,.\]
\end{lmm}
\begin{proof}
  For $\tau \in \R$ small enough denote $u_{\xi_\tau} :=\exp_u(\xi+
  \tau\xi')$, $u_\xi:=\exp_u \xi$ and the vector field $\eta_\tau:=
  \F_u(\xi+ \tau \xi')$. By definition we have $\Pi_u^{u_{\xi_\tau}}
  \eta_\tau= \Pi_u^{u_{\xi_\tau}} \F_u(\xi+\tau \xi') = \CR_J
  u_{\xi_\tau}$ and after deriving that equation covariantly for
  $\tau$ and restricting to $\tau =0$ we obtain
  \[D_{u_\xi}\left(E_u^\ver(\xi)\xi'\right)=\na_\tau
  \Pi_u^{u_{\xi_\tau}} \eta_\tau\big|_{\tau=0} = \na_\tau
  \Pi_u^{u_{\xi_\tau}} \eta_\tau -\Pi_u^{u_{\xi_\tau}} \partial_\tau
  \eta_\tau\big|_{\tau=0} + \Pi_u^{u_\xi} \d \F_u(\xi)\xi'\;.\] For
  the second identity we just added zero and used that by definition
  $\partial_\tau \eta_\tau|_{\tau=0}= \d\F_u(\xi)\xi'$. Hence
  \begin{multline}\label{eq:dFDu}
    \nm{\d \F_u(\xi)\xi' - D_u \xi'} = \nm{\Pi_u^{u_\xi}\d
      \F_u(\xi) \xi' - \Pi_u^{u_\xi} D_u \xi'}\\
    \leq \nmm{ \na_\tau \Pi_u^{u_{\xi_\tau}} \eta_\tau
      -\Pi_u^{u_{\xi_\tau}} \partial_\tau \eta_\tau|_{\tau=0}} +
    \nm{D_{u_\xi} E^\ver_u(\xi)\xi'- \Pi_u^{u_\xi} D_u \xi'}\,.
  \end{multline}
  To estimate the first term of the right-hand side, we use
  corollaries~\ref{cor:commutenaPi} and~\ref{cor:dwxi}
  \begin{multline}\label{eq:firstterm}
    \nmm{ \na_\tau \Pi_u^{u_{\xi_\tau}} \eta_\tau
      -\Pi_u^{u_{\xi_\tau}} \partial_\tau \eta_\tau|_{\tau=0}} \leq
    O(1) \nm{\partial_\tau
      u_{\xi_\tau}}\nm{\xi_\tau}\nm{\eta_\tau}\big|_{\tau=0} \leq O(1)
    \nm{\xi'}\nm{\xi} \nm{\CR_{J}u} \\
\leq O(1) \nm{\xi'}\nm{\xi}
    (1+\Nm{J}_\infty)\nm{\d u}
  \end{multline}
  We now focus on the second summand of the right-hand side
  of~\eqref{eq:dFDu}. The differene $D_{u_\xi} E^\ver_u(\xi)\xi' -
  \Pi_u^{u_\xi} D_u \xi'$ equals
  \begin{multline*}
    \ns E^\ver_u(\xi)\xi' - \Pi_u^{u_\xi} \ns \xi'
    + J(u_\xi) \nt E^\ver_u(\xi)\xi'  - \Pi_u^{u_\xi} J(u)\nt \xi'\\
    +\left(\na_{E^\ver_u(\xi)\xi'} J(u_\xi)\right) \pt u_\xi -
    \Pi_u^{u_\xi}\left(\na_{\xi'} J(u) \right)\pt u\,.
  \end{multline*}
  Denote the norm of the successive differences of the right-hand side
  by $T_1,T_2$ and $T_3$. By Proposition~\ref{prp:dexp} and
  Corollary~\ref{cor:dwxi} we have a constant $\e>0$ such that if
  $\nm{\xi} <\e$
  \begin{align*}
    T_1 &\leq \nm{\ns E^\ver_u(\xi)\xi' - E^\ver_u(\xi)\ns\xi'} +
    \nm{E^\ver_u(\xi)\ns \xi'
      - \Pi_u^{u_\xi}\ns \xi'}\\
    &\leq O(1)\nm{\xi'}\left(\nm{\ps u} \nm{\xi} + \nm{\ns \xi}\right)
    + O(1)\nm{\xi}\nm{\ns \xi'}\;,
  \end{align*}
  Abbreviate $E(\xi)\xi'=E^\ver_u(\xi)\xi'$) and estimate using
  Proposition~\ref{prp:dexp} and Corollary~\ref{cor:dwxi}
  \begin{align*}
    T_2 &\leq \nm{ J(u_\xi) \nt E(\xi)\xi'- J(u_\xi)\Pi_u^{u_\xi} \nt
      \xi'} +\nm{J(u_\xi) \Pi_u^{u_\xi}\nt \xi' - \Pi_u^{u_\xi}J(u) \nt \xi'}\\
    &\leq \Nm{J}_\infty \nm{\nt E(\xi)\xi' - \Pi_u^{u_\xi} \nt \xi'}
    + \Nm{\na J}_\infty \nm{\xi}\nm{\nt \xi'} \\
    & \leq O(1)\Nm{J}_\infty \nm{\xi'}\left(\nm{\pt u}\nm{\xi} +
      \nm{\nt \xi} \right) + \Nm{J}_{C^1} \nm{\xi}\nm{\nt\xi'}\;.
  \end{align*}
  and
  \begin{align*}
    T_3 &\leq\nm{\left(\na_{E(\xi)\xi'} J(u_\xi)\right) \pt u_\xi -
      \left(\na_{E(\xi)\xi'} J(u_\xi)\right)\Pi_u^{u_\xi}\pt u}+\\
    &\qquad+ \nm{\left(\na_{E(\xi)\xi'} J(u_\xi)\right)\Pi_u^{u_\xi}
      \pt u - \left(\na_{\Pi_u^{u_\xi}\xi'} J(u_\xi)\right)
      \Pi_u^{u_\xi}
      \pt u}+\\
    &\qquad+\nm{ \left(\na_{\Pi_u^{u_\xi}\xi'} J(u_\xi)\right)
      \Pi_u^{u_\xi} \pt u- \Pi_u^{u_\xi}\left(\na_{\xi'} J(u)
      \right)\pt u }\\
    & \leq \Nm{\na J}_\infty \nm{E(\xi)\xi'} \nm{\pt u_\xi -
      \Pi_u^{u_\xi} \pt u}
    +\Nm{\na J}_\infty \nm{E(\xi)\xi' - \Pi_u^{u_\xi}\xi'} \nm{\pt u}+\\ &\hspace{8.5cm} + \Nm{\na^2 J}_\infty \nm{\xi'}\nm{\xi}\nm{\pt u}\\
    &\leq O(1)\Nm{\na J}_\infty \nm{\xi'}(\nm{\pt u}\nm{\xi}+\nm{\nt \xi}) + \Nm{\na J}_\infty \nm{\xi}\nm{\xi'}\nm{\pt u}+\\ &\hspace{8cm} + \Nm{\na^2 J}_\infty\nm{\xi'}\nm{\xi}\nm{\pt u}\\
    &\leq O(1)(1+\Nm{J}_{C^2})\left(\nm{\xi}\nm{\xi'}\nm{\pt u} +
      \nm{\xi'}\nm{\nt \xi}\right)\,.
  \end{align*}
  Putting everything together we obtain the result by the last three
  estimates and estimate~\eqref{eq:firstterm} plugged into the
  identity~\eqref{eq:dFDu}.
\end{proof}
\subsubsection{Sobolev estimates}
Here we have collected estimates for the weighted Sobolev norms. For
any $a < b$ consider the domain $\Sigma_a^b:=[a,b]\times [0,1]$. We
also abbreviate the half-open strips $\Sigma_a^\infty=[a,\infty)\times
[0,1]$ and $\Sigma_{-\infty}^b = (-\infty,b]\times [0,1]$. The next
lemma states that functions on $\Sigma_a^b$ satisfy a Sobolev estimate
with a constant independent of $a$ and $b$ as long as the strip is
long enough.
\begin{lmm}\label{lmm:uniSob}
  For all constants $p >2$, $a$ and $b$ with possibly $a=-\infty$ or
  $b=\infty$ satisfying $b-a \geq 1$ and all functions $f \in
  H^{1,p}(\Sigma_{a}^{b},\R)$ we have
  \[
  \Nm{f}_{C^0(\Sigma_{a}^{b})} \leq \frac{4p}{p-2}\left(\int_{\Sigma_a^b}\nm{f}^p + \nm{\d f}^p \d s \d t\right)^{1/p}\,.
  \]
\end{lmm}
\begin{proof} It suffices to show the estimate for any given smooth
  function $f:\Sigma_{a}^{b} \to \R$.  Fix an arbitrary $z=s+it \in
  \Sigma_a^b$. An easy geometric observation shows there exists $s_0
  \in \R$ such that $\nm{s-s_0}<1$ and $\Sigma_{s_0}^{s_0+1} \subset
  \Sigma_{a}^{b}$. We have for every $z' \in [s_0,s_0+1] \times [0,1]$
  \[
  f(z) = f(z') + \int_0^1 df(\theta z +(1-\theta)z')[z-z'] \d \theta\,.
  \]
  Integrate $z'$ over $\Sigma^{s_0+1}_{s_0}:=[s_0,s_0+1]\times [0,1]$
  and estimate using $\nm{z-z'}<2$ and the H\"older inequality
  \begin{align*}
    \nm{f(z)} &\leq \int_{\Sigma^{s_0+1}_{s_0}} \nm{f(z')} \d z' + 2
    \int_0^1 \int_{\Sigma^{s_0+1}_{s_0}} \nm{\d f(\theta z
      +(1-\theta)z')}
    \d z' \d \theta \\
    &\leq \Nm{f}_{L^p} +2 \int_0^1 \left(\int_{\Sigma^{s_0+1}_{s_0}}
      \nm{\d f(\theta z + (1-\theta)z')}^p \d
      z'\right)^{1/p} \d \theta\\
    &= \Nm{f}_{L^p} + 2 \int_0^1
    \left(\int_{(1-\theta)\Sigma_{s_0}^{s_0+1} + \theta z}
      \frac{\nm{\d f(\hat z)}^p}{(1-\theta)^2} \d \hat z\right)^{1/p}
    \d \theta\\
    &\leq \Nm{f}_{L^p} + 2 \Nm{\d f}_{L^p} \int_0^1
    (1-\theta)^{-2/p} \d \theta \\
    &\leq \frac{2p}{p-2}\left(\Nm{f}_{L^p} + \Nm{\d
        f}_{L^p}\right)\leq \frac{ 4p}{p-2}\Nm{f}_{H^{1,p}}\,.
  \end{align*}
  This shows the claim by taking the supremum over all $z$
\end{proof}
\begin{lmm}\label{lmm:comparenorms}
  For all $p>2$, there exists a constant $c$ such that for all strips
  $u \in \B^{1,p;\delta}(C_-,C_+)$, vector fields $\xi \in
  T_u\B^{1,p;\delta}$ and $R \geq 1$ we have
  \begin{equation}
    \label{eq:comparenorms}
    \begin{aligned}
      \Nm{\xi}_\infty &\leq c \Nm{\xi}_{1,p;\delta},& \Nm{\na
        \xi}_{p;\delta} &\leq c (1+\Nm{\d u}_{p;\delta}) \Nm{\xi}_{1,p;\delta}\\
      \Nm{\xi}_\infty &\leq c \Nm{\xi}_{1,p;\delta,R},& \Nm{\na
        \xi}_{p;\delta,R} &\leq c(1+ \Nm{\d u}_{p;\delta,R}) \Nm{\xi}_{1,p;\delta,R}\,,
  \end{aligned} 
  \end{equation}
  in which the norm $\Nm{\cdot}_{1,p;\delta,R}$ is defined
  in~\eqref{eq:nrmxideltaR}.
\end{lmm}
\begin{proof}
  The proof is given in \cite[Lemma 10.8]{Abouzaid:spheres} and
  \cite[Lemma 10.9]{Abouzaid:spheres}.  For some vector field $\xi \in
  \Gamma(u^*TM)$ we have the inequality
  \begin{equation}
    \label{eq:CSxi}
    \nm{\d \nm{\xi}} = \nm{\<\na \xi,\xi\>}/\nm{\xi} \leq \nm{\na \xi}\,. 
  \end{equation}
  For any $(s,t) \in \Sigma_0^\infty$ we have by
  Lemma~\ref{lmm:uniSob}
  \begin{equation*}
    \nm{\xi}^p \leq 2^p\big(\nmm{\xi - \wh \Pi_{u(\infty)}^u
      \xi(\infty)}^p + \nm{\xi(\infty)}^p\big)
  \end{equation*}
  Then by estimate~\eqref{eq:CSxi} and since $e^{\delta \nm{s}} \geq 1$ we have
  \begin{multline*}
   \nm{\xi}^p   \leq O(1)
    \int_{\Sigma_0^\infty}\nmm{\xi-\wh \Pi_{u(\infty)}^u
      \xi(\infty)}^p + \nmm{\d \nmm{\xi-\wh \Pi_{u(\infty)}^u
        \xi(\infty)}}^p \d s \d t + 2^p\nm{\xi(\infty)}^p 
  \end{multline*}
  which is easily bounded by $O(1)\Nm{\xi}_{1,p;\delta}$. Similarily we
  proceed for the negative end and for $\Sigma_{-2R}^{2R}$ appearing
  in norm $\Nm{\, \cdot\,}_{1,p;\delta,R}$. Note that the Sobolev
  constant of $\Sigma_{-2R}^{2R}$ is independent of $R$ by
  Lemma~\ref{lmm:uniSob}. This shows the two inequalities on the
  left-hand side of~\eqref{eq:comparenorms}.
  
  For the two inequalities on the right-hand side we use
  Corollary~\ref{cor:commutenaPi} to see that the norm of $\na \xi$ is bounded by 
  \[
  \nmm{\na\big(\xi -\wh \Pi_{u(\infty)}^u \xi(\infty)\big)} + \nmm{\na
    \wh \Pi_{u(\infty)}^u \xi(\infty)} \leq \nmm{\na\big(\xi -\wh
    \Pi_{u(\infty)}^u \xi(\infty)\big)} + c_3 \nm{\d u}
  \nm{\xi(\infty)}\,.
  \]
  Multiply the estimate with $e^{\delta \nm{s}}$, use the inequality
  $(a+b)^p \leq 2^p(a^p + b^p)$ for all positive $a,b$ and integrate
  over $\Sigma_0^\infty$ to conclude that there exists a constant
  $c_4$ such that
  \begin{multline*}
  \int_{\Sigma_0^\infty} \nm{\na \xi}^p e^{\delta\nm{s}p} \d s \d t
  \\ \leq c_4 \int_{\Sigma_0^\infty}
  \nmm{\na\big(\xi-\wh\Pi_{u(\infty)}^u \xi(\infty)\big)}^p
  e^{\delta\nm{s}p} \d s \d t + c_4 \nm{\xi(\infty)}^p
  \int_{\Sigma_0^\infty}\nm{\d u}^p e^{\delta \nm{s}p} \d s \d t \,.    
  \end{multline*}
  Similar we proceed with the negative end and
  $\Sigma_{-2R}^{2R}$. This shows the claim.
\end{proof}
\begin{lmm}\label{lmm:deltadecay}
  For all $p>2$, there exists a constant $c$ such that for all
  $\delta\geq 0$, strips $u \in \B^{1,p;\delta}(C_-,C_+)$, vector
  fields $\xi\in T_u\B^{1,p;\delta}$ and $s \geq 0$ we have
  \[\Nmm{\xi - \wh \Pi_{u(\infty)}^u \xi(\infty)}_{C^0\left([s,\infty]\times
      [0,1]\right)} \leq c e^{-\delta s} \Nm{\xi}_{1,p;\delta}\,.\] A
  similar estimate holds for the negative end.
\end{lmm}
\begin{proof}
  A proof is given in \cite[Lemma 4.4]{Abouzaid:spheres}. Abbreviate
  $\xi^+ := \xi-\wh \Pi_{u(\infty)}^u \xi(\infty)$. By
  Lemma~\ref{lmm:commutenaPi}, the Sobolev estimate~\ref{lmm:uniSob}
  and estimate~\eqref{eq:CSxi} we have uniform constants $c_1$ and
  $c_2$ such that
  \[
  \nmm{e^{\delta s}\xi^+}^p \leq c_1 \int_{\Sigma_{s}^\infty}
  \left(\nm{\xi^+}^p + \nm{\d \nm{\xi^+}}^p\right) e^{\delta \sigma }
  \d \sigma \leq c_2 \Nm{\xi}^p_{1,p;\delta}\,.
  \]
  This shows the claim after multiplying with $e^{-p \delta s}$ on
  both sides and taking the $p$-th root.
\end{proof}

\begin{cor}\label{cor:SobcommuteDPi}
  There exists constant $\e$ and $c$ such that for all elements $u \in
  \B^{1,p;\delta}(C_-,C_+)$ and smooth vector fields $\xi,\xi' \in
  \Gamma(u^*TM)$ such that $\Nm{\xi}_\infty< \e$ we have
  \[
  \Nm{D_{u_\xi} \Pi_u^{u_\xi} \xi' - \Pi_u^{u_\xi} D_u
    \xi'}_{p;\delta} \leq c(1+\Nm{J}_{C^2})(1+\Nm{\d
    u}_{p;\delta})\Nm{\xi}_{1,p;\delta} \Nm{\xi'}_{1,p;\delta}\,.
  \]
\end{cor}
\begin{proof}
  Integrate the pointwise estimate from Lemma~\ref{lmm:commuteDPi} and
  then use the Sobolev estimates from Lemma~\ref{lmm:comparenorms}. A
  completely similar argument appears in the proof of
  Lemma~\ref{lmm:quadratic}.
\end{proof}
\begin{cor}\label{cor:Dcont}
  There exists a constant $c$ such that for all $u \in
  \B^{1,p;\delta}(C_-,C_+)$ and $\xi \in T_u\B^{1,p;\delta}$ we have
  \[
  \Nm{D_u\xi}_{p;\delta} \leq c(1+\Nm{J}_{C^2})(1 +\Nm{\d u}_{p;\delta})
  \Nm{\xi}_{1,p;\delta}\,.
  \]  
\end{cor}
\begin{proof}
  By definition of the operator $D_u$ we have the point-wise estimate
  \[
  \nm{D_u \xi} \leq (1+\Nm{J}_\infty) \nm{\na \xi} +
  \Nm{J}_{C^1}\nm{\xi}\nm{\d u}\,.
  \]
  Now integrate the estimate and use the estimates given in
  Lemma~\ref{lmm:comparenorms}. 
\end{proof}
\begin{lmm}\label{lmm:uvclose}
  There exists a constant $\e>0$ such that for all elements $u \in
  \B^{1,p;\delta}(C_-,C_+)$ and $\xi \in \Gamma(u^*TM)$ with $u_\xi
  =\exp_u \xi \in \B^{1,p;\delta}(C_-,C_+)$ and $\Nm{\xi}_\infty < \e$
  we have $\xi \in T_uB^{1,p;\delta}$.
\end{lmm}
\begin{proof}
  Abbreviate $p=u(\infty)$ and $q=u_\xi(\infty)$. Define the vector
  $\xi(\infty) \in T_pM$ via $q=\exp_p \xi(\infty)$. Since the
  distance between parallel geodesics is uniformly bounded (\cf
  Corollary~\ref{cor:estparallelgeodesics}) we have for all $(s,t)\in
  \R \times [0,1]$ with $s$ large enough
   \begin{align*}
     \nm{\xi-\Pi_p^u \xi(\infty)} &\leq O\big(\di{u_\xi,\exp_u\Pi_p^u
       \xi(\infty)}\big)\\
     &\leq O\big(\di{u_\xi,q} +\di{\exp_p\xi(\infty),\exp_u
       \Pi_p^u\xi(\infty)}\big)\\
     &\leq O\big(\di{u_\xi,q} + \di{u,p} \big)\,,
   \end{align*}
   and using bounds on the derivative of the exponential map (\cf
   Corollary~\ref{cor:dwxi}) as well as the bound for the commutator
   of $\Pi^u_p$ with $\na$ (\cf Corollary~\ref{cor:commutenaPi}) we
   obtain
  \[
  \nmm{\na\big(\xi -\Pi_p^u \xi(\infty)\big)} \leq O\left(\nm{\d
      u} + \nm{\d u_\xi} + \di{u,p}\right)\,.
  \]
  Since $u$ and $u_\xi$ are elements of $\B^{1,p;\delta}(C_-,C_+)$ the
  integral is finite
  \begin{multline*}
    \int_{\Sigma_0^\infty} \left(\nmm{\xi -\wh \Pi_p^u \xi(\infty)}^p
      + \nmm{\na
        \big(\xi -\wh \Pi_p^u \xi(\infty)\big)}^p \right)e^{\delta p s} \d s \d t \leq \\
    O(1) \int_{\Sigma_0^\infty} \left(\nm{\d u}^p + \nm{\d v}^p +
      \di{u,p}^p+\di{v,q}^p\right)e^{\delta p s} \d s \d t < \infty\,.
  \end{multline*}
  Similar we proceed on the negative end. This shows the claim.
\end{proof}
\begin{lmm}\label{lmm:convends1p}
  Gromov topology is finer than the topology of
  $\B^{1,p;\delta}(C_-,C_+)$ if $\delta >0$ is sufficiently small, \ie
  given a sequence $(u_\nu)_{\nu\in\N}$ of $(J_\nu,X_\nu)$-holomorphic
  curves which Floer-Gromov converges to the $(J,X)$-holomorphic strip
  $u$, then for all $\delta>0$ small enough and $\nu\in\N$ large
  enough we have $u_\nu=\exp_u \xi_\nu$ for some vector field
  $\xi_\nu\in T_u\B^{1,p;\delta}$ and moreover
  $\Nm{\xi_\nu}_{1,p;\delta}$ converges to zero.
\end{lmm}
\begin{proof}
  By Floer-Gromov convergence we have in particular that $u_\nu$
  converges to $u$ uniformly on $\R \times [0,1]$ (\cf
  Lemma~\ref{lmm:convends}). Hence there exists $\xi_\nu \in
  \Gamma(u^*TM)$ such that $u_\nu=\exp_u \xi_\nu$ for all $\nu$ large
  enough. Lemma~\ref{lmm:uvclose} shows that the norm
  $\Nm{\xi_\nu}_{1,p;\delta}$ is finite for all $\nu$ large enough. It
  remains to show that $\Nm{\xi_\nu}_{1,p;\delta}$ converges to zero.
  
  By Lemma~\ref{lmm:Edecay} we see that there exists a constant $\mu$
  such that for all $s >0$ and $\nu \in \N$ large enough
  \[\di{u_\nu,u_\nu(\infty)} + \nm{\d u_\nu} \leq c_1 e^{-\mu s} \,.\]
  The same holds with $u_\nu$ replaced by $u$. Abbreviate
  $p_+:=u(\infty)$, $p_+^\nu:=u_\nu(\infty)$,
  $\xi_\nu(\infty):=\exp_{p_+}^{-1} p_+^\nu$ and $\xi_\nu^+ := \xi_\nu
  - \wh \Pi_{p_+}^u \xi_\nu(\infty)$. We use
  Corollary~\ref{cor:commutenaPi} to get
  \[
  \nm{\na \xi_\nu^+} \leq O(\nm{\d u} + \nm{\d u_\nu} + \di{u,p}) \leq
  O(e^{-\mu s}) \,.
  \]
  By Lemma~\ref{lmm:uvclose} we have $\lim_{s \to \infty}
  \nm{\xi_\nu^+(s,t)} =0$ and hence
  \begin{multline*}
  \nm{\xi_\nu^+(s,t)} = \int_s^\infty
  -\partial_\sigma\nm{\xi_\nu^+(\sigma,t)} \d \sigma \leq
  \int_s^\infty \nm{\na \xi_\nu^+(\sigma,t)} \d \sigma \leq \\ \leq
  O(1)\int_s^\infty e^{-\mu \sigma} \d \sigma \leq O(e^{-\mu s})\,.   
  \end{multline*}
  For $\delta<\mu$ and $s>s_0$ with $s_0$ large enough we conclude
  \begin{multline*}
  \int_{\Sigma_0^\infty} \big(\nm{\xi_\nu^+}^p + \nm{\na
    \xi_\nu^+}^p\big)e^{\delta p s} \d s \d t \leq
  \Nm{\xi_\nu}_{C^1(\Sigma_0^{s_0})}^p \int_0^s e^{\delta \sigma} \d
  \sigma + O(1)\int_s^\infty e^{-(\mu-\delta) \sigma} \d \sigma\\
 \leq
  O( e^{\delta s}) \Nm{\xi_\nu}_{C^1(\Sigma_0^s)} + O(e^{-(\mu-\delta)s})\,.
  \end{multline*} 
  Similar we proceed with the negative end to show that for all $s
  \geq s_0$ we have
  \[
  \Nm{\xi_\nu}_{1,p;\delta} \leq O( e^{\delta s})
  \Nm{\xi_\nu}_{C^1(\Sigma_{-s}^s)} + O(e^{-(\mu-\delta)s}) \leq o(1)
  + O(e^{-(\mu-\delta) s})\,.
  \]
  Because $s$ was chosen freely the left-hand side converges to zero.
\end{proof}

\begin{lmm}\label{lmm:JnuJ}
  With the same assumptions as Lemma~\ref{lmm:convends1p}. There
  exists a constant $c$ such that for all $\xi \in
  T_u\B^{1,p;\delta}(C_-,C_+)$ and $\nu \in \N$ large enough we have
  $u_\nu =\exp_u \xi_\nu$ for some $\xi_\nu \in T_u\B^{1,p;\delta}$
  and
  \[
  \Nm{\big(D_{u_\nu,J_\nu}\Pi_u^{u_\nu} -
    \Pi_u^{u_\nu}D_{u,J}\big)\xi}_{p;\delta} =
  c(\Nm{\xi_\nu}_{1,p;\delta}+\Nm{J_\nu-J}_{C^1})
  \Nm{\xi}_{1,p;\delta}\,.\] In particular the operator
  $D_{u_\nu,J_\nu}\Pi_u^{u_\nu} -\Pi_u^{u_\nu}D_{u,J}$ converges to
  zero in operator norm.
\end{lmm}
\begin{proof}
  By Lemma~\ref{lmm:convends1p} the vector field $\xi_\nu$
  exists. Corollary~\ref{cor:SobcommuteDPi} implies that there exists a
  constant $c_1$ possibly depending on $u$ and $J$ but independent of
  $\nu$ such that for all sections $\xi \in \Gamma(u^*TM)$ and $\nu$
  large enough we have
  \[
  \Nm{\left(D_{u_\nu,J_\nu} \Pi_u^{u_\nu} -
      \Pi_u^{u_\nu}D_{u,J_\nu}\right)\xi}_{p;\delta} \leq c_1
  \Nm{\xi_\nu}_{1,p;\delta}\Nm{\xi}_{1,p;\delta}\;.
  \]Directly from the Definition we have
  \[
  \Nm{\big(D_{u,J_\nu} - D_{u,J}\big)\xi}_{p;\delta} \leq
  2^{p+1}\Nm{J-J_\nu}_{C^1} (\Nm{\xi}_\infty \Nm{\pt u}_{p;\delta} + \Nm{\na
    \xi}_{p;\delta})\,.
  \]
  This shows the estimate using Lemma~\ref{lmm:comparenorms}. With the
  estimate we conclude convergence of the operator since by
  Lemma~\ref{lmm:convends1p} the norm $\Nm{\xi_\nu}_{1,p;\delta}$
  converges to zero.
\end{proof}

\section{Operators on Hilbert spaces}\label{app:operator}
Let $H$ be a separable real Hilbert space. We will write
$\<\cdot,\cdot\>$ and $\Nm{\cdot}$ for the inner product and
respectively the norm of $H$. In this section we consider unbounded
self-adjoint operators $A$ and with dense $\dom A \subset H$. Let
$\L(H)$ denote the space of bounded linear operators of $H$ and for $B
\in \L(H)$ we denote by $\Nm{B}$ the operator norm.
\subsection{Spectral gap}
Given a self-adjoint operator $A:\dom A \to H$ we denote by $\sigma(A)
\subset \R$ its spectrum and by
\begin{equation}
  \label{eq:iotaA}
  \iota(A):=\inf\{\nm{\lambda} \mid \lambda \in \sigma(A) \setminus
\{0\}\}
\end{equation}
the \emph{spectral gap of $A$}.
\begin{lmm}\label{lmm:Aclosed}
  Let $A$ be self-adjoint operator with domain $\dom A \subset
  H$. Assume that the spectrum $\sigma(A) \subset \R$ is bounded from
  below, then for all $\xi \in \dom A$ we have
  \[
  \<A \xi,\xi\> \geq \inf \sigma(A) \Nm{\xi}^2\,.
  \]
  If additionally the range of $A$ is closed then the spectral gap of
  $A$ is positive and satisfies
  \[
  \Nm{A \xi} \geq \iota(A) \Nm{\xi}\,,
  \]
  for all $\xi \in \dom A$ with $\xi \perp \ker A$. Both inequalities
  are sharp.
\end{lmm}
\begin{proof}
  The first part is proven in~\cite[Section 10]{Kato}. To show the
  second inequality use the first inequality with $A^2$. It remains to
  show that $\iota(A)$ is positive. Assume without loss of generality
  that $A$ is injective. Since $A$ is self-adjoint with closed range
  it is invertible. Since $A$ is a closed operator, the closed graph
  theorem implies that $A^{-1}$ is bounded. Hence for some constant
  $c>0$ we have $\Nm{A^{-1}\eta}\leq c \Nm{\eta}$ for all $\eta \in H$
  which implies $\Nm{\xi} \leq c \Nm{A \xi}$ for all $\xi \in \dom
  A$. This shows that $\iota(A) \geq 1/c$ since the inequalities are
  sharp.
\end{proof}
\begin{cor}\label{cor:Aclosed}
  Let $A$ be a self-adjoint operator with domain $\dom A$ and closed
  range, then we have for all $\xi \in \dom A$
  \[
  \Nm{A\xi}^2 \geq \iota(A) \<A\xi,\xi\>\,.
  \]
\end{cor}
\begin{proof}
  Let $P:H \to \ker A$ denote the orthogonal projector to $\ker A$ as
  subset of $H$.  With Lemma~\ref{lmm:Aclosed} we have
  \[
  \<A\xi,\xi\> = \<A(1-P)\xi,\xi\> = \<(1-P)\xi,A\xi\> \leq
  \Nm{(1-P)\xi}\Nm{A\xi} \leq \iota(A)^{-1} \Nm{A\xi}^2\,.
  \]
  This shows the claim.
\end{proof}
The next lemma states that the spectral gap is lower semi-continuous
for bounded perturbations of $A$ which preserve the dimension of the
kernel.
\begin{lmm}\label{lmm:iotasemicont}
  Let $A$ be a self-adjoint operator with closed range and finite
  dimensional kernel. For all $\e>0$ there exists a constant
  $\delta>0$ such that for any bounded symmetric operator $B$ with
  $\Nm{B} <\delta$ and $\dim \ker A+B = \dim \ker A$ we have
  $\iota(A+B) \geq \iota (A)-\epsilon$.
\end{lmm}
\begin{proof}
  Write $A'=A+B$ and denote by $P,P'$ the orthogonal projection to the
  kernel of $A,A'$ respectively. We claim that for any $\e>0$ there
  exists $\delta$ such that
  \begin{equation}
    \label{eq:PPclose}
    \Nm{A-A'}<\delta \quad \Rightarrow \quad \Nm{P-P'}< \frac{\e}{2\iota(A)}\,.
  \end{equation}
  Let $\{E(\lambda)\}, \{E'(\lambda)\}$ be the spectral families
  associated to $A,A'$ respectively. The spectrum of $A$ has a gap at
  $\pm \iota(A)/2$. By \cite[Thm. 5.10]{Kato} we have for any $\e>0$ a
  constant $\delta$ such that ($\iota:=\iota(A)$)
  \begin{equation}
    \label{eq:Kato}
    \Nm{A-A'} <\delta\quad \Rightarrow\quad \Nm{E(\iota/2)-E(-\iota/2) -
      (E'(\iota/2) -E'(-\iota/2))} <\frac{\e}{2\iota}\,.
  \end{equation}
  Note that in our situation the quantity $\hat \delta(A,A')$ as
  defined in \cite[p. 197]{Kato} reduces to $\Nm{A-A'}$.  Since zero
  is the only spectral value in the interval $[-\iota/2,\iota/2]$ we
  have $E(\iota/2) -E(-\iota/2) =P$. To show~\eqref{eq:PPclose} it
  remains to show $E'(\iota/2)-E'(-\iota/2)=P'$. By monotonicity of
  the spectral family we have
  \begin{equation}
    \label{eq:PE}
    \im P' = \im \big(E'(0)- \lim_{\lambda \uparrow 0} E'(\lambda)\big)
    \subset \im \big(E'(\iota/2) - E'(-\iota/2)\big)\,,
  \end{equation}
  By~\eqref{eq:Kato} the projection $E'(\iota/2) -E'(-\iota/2)$
  converges to $P$, in particular their images have the same
  dimension. Hence 
  \[
  \dim \im P' \leq \dim \im \big(E'(\iota/2) -E'(-\iota/2)\big) =\dim
  \im P\,.
  \]
  By assumption we have $\dim \im P'=\dim \im P$, hence we have
  equality in the last estimate, which shows that we have equality
  in~\eqref{eq:PE} and thus $P'
  =E'(\iota/2)-E'(-\iota/2)$. Hence~\eqref{eq:PPclose} follows
  from~\eqref{eq:Kato}.

  We now proof the lemma. Since $A'$ is a bounded perturbation we have
  $\dom A'=\dom A$.  By possibly decreasing $\delta$ we assume that
  $\delta<\frac{\e}{2}$ and estimate using~\eqref{eq:PPclose} for any
  $\xi \in \dom A$ with $\Nm{\xi}=1$ and  $\xi\perp
  \ker A'$
  \begin{align*}
    1=\Nm{\xi} &\leq \Nm{(\one -P)\xi} + \Nm{(P-P')\xi} \\
    &\leq \frac 1{\iota}\Nm{A\xi} + \Nm{P-P'} \\ & \leq \frac
    1{\iota} \Nm{A'\xi} + \frac 1{\iota}\Nm{A-A'} +\Nm{P-P'}\\
    &\leq \frac 1{\iota} \Nm{A'\xi} + \frac{\e}{2\iota}
    +\frac{\e}{2\iota}\,.
  \end{align*}
  Hence $\iota -\e \leq \Nm{A'\xi}$ and the lemma follows by taking
  the infimum over all $\xi \in \dom A$ with
  $\Nm{\xi}=1$ and $\xi \perp \ker A'$.
\end{proof}

\subsection{Flow operator}
Given a Banach space $V$ such that there exists a compact and dense
inclusion $V \subset H$. Let $\L(V,H)$ denote the space of bounded
operators from $V$ to $H$. In this section we analyze the asymptotic
properties of bounded functions $\xi:[0,\infty) \to V$ which solve the
differential equation
\begin{equation}
  \label{eq:xisolves}
  \ps \xi(s) + A(s) \xi(s) + B(s) \xi(s) =\eta(s)\,,
\end{equation}
where $\eta:[0,\infty) \to H$ and $A:[0,\infty) \to \L(V,H)$,
$B:[0,\infty) \to \L(H)$ are continuously differentiable functions
satisfying the assumptions:
\begin{enumerate}[label=(\roman*)]
\item The operator $A(s)$ is symmetric for every $s$.  There exists an
  operator $A_\infty \in \L(V,H)$ such that $A(s) -A_\infty$ and $\ps
  A(s)$ extend to bounded linear operators on $H$ and we have
  \begin{equation}
    \label{eq:A}
    \lim_{s\to \infty}\Nm{A(s) - A_\infty } = \lim_{s\to \infty}
    \Nm{\ps A(s)} =0\,.
  \end{equation}
\item The operator $A_\infty$ is Fredholm but not necessarily
  injective.
\item The operator $B(s)$ is skew-symmetric for every $s \geq 0$ and
  \begin{equation}
    \label{eq:B}
    \lim_{s\to \infty} \Nm{B(s)} = 0\,.
  \end{equation}
\end{enumerate}
\begin{rmk}
  These assumptions are almost identical to the assumptions in
  \cite[Section 3]{Robbin:Strip} except that we do not suppose that
  $A_\infty$ is injective.  
\end{rmk}
\begin{lmm}\label{lmm:delta}
  Let $P:H \to \ker A_\infty$ denote the orthogonal projection. Assume
  that $\lim_{s \to \infty} \Nm{\xi(s)} =0$ and for every constant
  $\e$ there exists $s_0$ such that for all $s \geq s_0$ we have
  \begin{equation}
    \label{eq:P}
    \Nm{P\xi(s)} \leq \e \Nm{\xi(s)}.
  \end{equation}
  Further suppose that there exists positive constants $\delta$ and
  $c$ such that for all $s \geq 0$ we have
  \[
  \Nm{\eta(s)} + \Nm{\ps \eta(s)} \leq c e^{-\delta s}\,. 
  \]
  Then for any $\mu<\min\{\iota(A_\infty),\delta\}$ there exists a
  constant $s_0=s_0(\mu)$ such that $\Nm{\xi(s)} \leq e^{-\mu s}$ for
  all $s\geq s_0$.  Moreover if $\eta=0$ then we have
  \begin{equation*}
    \<A(s) \xi(s),\xi(s)\> \geq \mu \Nm{\xi(s)}^2\,,
  \end{equation*}
   for all $s \geq s_0$.
\end{lmm}
\begin{proof}
  We follow closely the lines of the proof of \cite[Lemma
  3.1]{Robbin:Strip}. We just need to insert
  assumption~\eqref{eq:P} at the right place. Consider the
  function $g:[0,\infty) \to \R$ given by
  \[ g(s) := \frac 12 \Nm{\xi(s)}^2\;.\] We suppress the argument $s$
  whenever convenient and write $\dot g$ \etc to denote the derivative
  by $\ps$. Since $B(s)$ is skew-symmetric we have
  \begin{equation}
    \label{eq:dotg}
    \dot g(s) = \<\xi,\dot \xi\> = \<\xi,\eta -A \xi\>\,.
  \end{equation}
  Differentiating again we have with assumptions~\eqref{eq:A}
  and~\eqref{eq:B} for any $\e>0$
  \begin{align*}
    \ddot g &= \<\dot \xi,\eta- 2A\xi\> + \< \xi,\dot \eta -\dot A \xi\> \\
    &=2\Nm{A \xi}^2 + \Nm{\eta}^2 - \<A \xi,3 \eta\> - \<B \xi,\eta\> + \<2B\xi,A\xi\> + \<\xi,\dot \eta - \dot A\xi\> \\
    &\geq (2-\e)\Nm{A\xi}^2 -
    \big((1+4\e^{-1})\Nm{B}^2 + \|\dot A\|+\e\big) \Nm{\xi}^2 - (1+10\e^{-1})(\Nm{\eta}^2+\Nm{\dot \eta}^2)\\
    &\geq (2-\e)\Nm{A \xi}^2 - (o(1)+\e)\Nm{\xi}^2 -
    c^2(1+10\e^{-1})e^{-2 \delta s}\,,
  \end{align*}
  where we have used the Cauchy-Schwarz inequality and the estimate $-ab
  \geq -\e a^2 -b^2/\e$ for all $a,b >0$. Similarly we have
  \begin{align*}
    \Nm{A_\infty\xi}^2 &= \Nm{(A-A_\infty) \xi}^2 +
    2\<(A-A_\infty)\xi,A\xi\> + \Nm{A \xi}^2\\
    &\leq (1+\e^{-1})\Nm{A-A_\infty}^2 \Nm{\xi}^2 + (1+\e)\Nm{A\xi}^2\\
    &\leq o(1) \Nm{\xi}^2 + (1+\e)\Nm{A\xi}^2\,.
  \end{align*}
  Combining the last two estimates we get a constant $c_1=c_1(\e)$
  such that
  \begin{align*}
    \ddot g &\geq \frac{2-\e}{1+\e} \Nm{A_\infty \xi}^2 - (o(1)+\e)
    \Nm{\xi}^2 - c_1e^{-2 \delta s}\\
    & \geq (2-4\e)\Nm{A_\infty \xi}^2 - (o(1)+\e) \Nm{\xi}^2 -
   c_1e^{-2\delta s}
  \end{align*}
  Let $\iota :=\iota(A_\infty)$ denote the spectral gap of $A_\infty$.
  With Lemma~\ref{lmm:Aclosed} we have
  \[
  \Nm{A_\infty \xi}^2 \geq \iota^2\Nm{(1-P)\xi}^2 = \iota^2
  \Nm{\xi}^2 -\iota^2 \Nm{P\xi}^2 \geq \iota^2 \Nm{\xi}^2 -
  o(1)\Nm{\xi}^2\,,
  \]
  in which we have used the assumption~\eqref{eq:P}. Combining the last
  two estimates shows
  \begin{align*}
    \ddot g(s) &\geq \iota^2 (2-4 \e) \Nm{\xi}^2
    -(o(1)+\e)\Nm{\xi}^2 - c_1e^{-2\delta s}\\
    &\geq (2\iota^2 -4\iota^2 \e -\e -o(1))\Nm{\xi}^2 - c_1 e^{-2\delta
      s}\,.
  \end{align*}
  In particular there exists a constant $s_0=s_0(\e)$  such
  that for all $s \geq s_0$ we have
  \[
  \ddot g(s) \geq 2(2\iota^2-4\iota^2\e  -4\e)g(s) - c_1 e^{-2\delta
    s}\,.\] The previous computation holds with any $\e$. Now choose
  $\e < (\iota^2-\mu^2)/(2\iota^2+2)$ to conclude
  \begin{equation*}
    \ddot g(s) \geq 4 \mu^2 g(s) - c_1 e^{-2\delta s}\,.
  \end{equation*}
  By assumption we also have $\lim_{s \to \infty} g(s) =0$.  Provided
  with the last estimate the rest of the proof is completely analogous
  to the proof of~\cite[Lemma 3.1]{Robbin:Strip}.
\end{proof}

\begin{lmm}\label{lmm:kernel}
  Assume that $\eta=0$ and the integral is finite
  \[
  \int_0^\infty \Nm{A(s)-A_\infty} + \Nm{B(s)} \d s\,.
  \]
  Then there exists an element $\zeta \in \ker A_\infty$ such that
  $\lim_{s \to \infty} \xi(s) = \zeta$. Moreover assume that there
  exist constants $\e>0$ and $c_1$ such that for all $s \geq 0$
  \[
  \Nm{A(s) -A_\infty} + \Nm{B(s)} \leq c_1e^{-\e s}\,,
  \]
  then for all $\mu<\min\{\e,\iota(A_\infty)\}$ we have a constant
  $s_0$ such that $\Nm{\xi(s)-\zeta} \leq e^{-\mu s}$ for
  all $s\geq s_0$.
\end{lmm}
\begin{proof}
  Let $P:H \to \ker A_\infty$ be the orthogonal projection. Apply $P$
  on~\eqref{eq:xisolves} to show that
  \[
  \ps P \xi = -PA \xi -PB\xi = P(A_\infty-A)\xi - PB \xi\,.
  \]
  Since $\xi$ is bounded we conclude that there exists a constant $c$
  such that
  \[
  \Nm{\ps P \xi} \leq c (\Nm{A_\infty-A} +\Nm{B})\,. 
  \]
  Since $\ps P \xi$ is integrable and $\ker A_\infty$ is finite
  dimensional the path $s \mapsto P\xi(s)$ converges to an element
  $\zeta \in \ker A_\infty$. The difference $s \mapsto \xi(s)-\zeta$
  solves the equation
  \[
  \ps (\xi-\zeta) + A (\xi-\zeta) + B( \xi-\zeta) = \eta\,, 
  \]
  with $\eta(s) = (A_\infty-A(s)) \zeta - B(s) \zeta$. We conclude
  using Lemma~\ref{lmm:delta}.
\end{proof}
For the rest of the section we assume that $\eta=0$
in~\eqref{eq:xisolves}. The proof of Agmon-Nirenberg Lemma (\cf
\cite[Lemma 3.3]{Robbin:Strip}) goes through without any change. Thus
$\xi(s) \neq 0$ for all $s \geq 0$ and we define
\[ v(s) = \frac{\xi(s)}{\Nm{\xi(s)}},\qquad \lambda(s) =
\<v(s),A(s)v(s)\>\;.\] The proof of~\cite[Lemma 3.4]{Robbin:Strip} requires an
adjustment.
\begin{lmm}\label{lmm:lambdainfty}
  With the assumptions of Lemma~\ref{lmm:delta} and $\eta=0$. Suppose
  that the two integrals are finite
  \[\int_0^\infty \Nm{A(s) - A_\infty} + \Nm{B(s)} \d s, \quad N
  :=\lambda(0) + \int_0^\infty \Nm{B(s)}^2 + \Nmm{\dot A(s)} \d s\;.\]
  Let $\mu$ denote the constant from Lemma~\ref{lmm:delta}. Then
  the limits
  \[ N \geq \lim_{s \to \infty} \lambda(s) = \lambda_\infty \geq \mu, \qquad
  \lim_{s \to \infty} v(s) = v_\infty\;,\] exist, where the latter convergence is
  in $H$ and we have $A_\infty v_\infty =\lambda_\infty v_\infty$.
\end{lmm}
\begin{proof} 
  We differentiate
  \begin{equation}
    \label{eq:dotv}
     \dot v =  \lambda v - A v  - B v\;,
  \end{equation}
  and
  \begin{equation}\label{eq:dotlambda}
    \begin{aligned}
      \dot \lambda &= 2 \<\dot v, A v\> + \<v, \dot A v\>\\
      &= 2\<\lambda v - A v - B v,A v\> + \<v,\dot A v\>\\
      &=-2\Nm{\lambda v- Av}^2 + 2 \<Bv, \lambda v-Av\> + \<v,\dot
      Av\>\\
      &\leq \|B\|^2 + \|\dot A\| - \Nm{\lambda v - A v}^2\;.
    \end{aligned}  
  \end{equation}
  Consider the function $\gamma:[0,\infty) \to \R$ defined by
  \[ \gamma(s) := \lambda(s) + \int_s^\infty \Nm{B(\sigma)}^2 +
  \Nmm{\dot A(\sigma)} \d \sigma\;.\] Since $\dot \gamma = \dot
  \lambda - \Nm{B}^2 - \Nmm{\dot A}$ it follows
  from~\eqref{eq:dotlambda} that for every $s \geq 0$
  \begin{equation}
    \label{eq:dmu}
    \dot \gamma + \Nm{Av - \lambda v}^2 \leq 0\,.
  \end{equation}
  We see that $\gamma$ is decreasing. Moreover $\gamma$ is bounded
  from below because with Lemma~\ref{lmm:delta} we have $\gamma(s)
  \geq \lambda(s) \geq \mu$.  Hence $\gamma(s)$ converges to a
  positive real number
  \[ \lambda_\infty := \lim_{s \to \infty} \gamma(s) = \lim_{s \to
    \infty} \lambda(s)\,.\] Since $\gamma(0)=N$ and $\gamma(s)
  \geq \mu$ for all $s\geq 0$ we have that
  \[ \mu \leq \lambda_\infty \leq N\;.\] We claim that
  $\lambda_\infty$ is an eigenvalue of $A_\infty$. By contradiction we
  assume that $A_\infty - \lambda_\infty$ is injective. Since $V
  \hookrightarrow H$ is compact and $A_\infty$ is a Fredholm operator,
  $A_\infty - \lambda_\infty$ is a Fredholm operator as well. In
  particular $A_\infty - \lambda_\infty$ is closed and there exists a
  constant $c_1>0$ such that
  \[ 1 =\Nm{v} \leq c_1 \Nm{A_\infty v - \lambda_\infty v}\,.\]
  We estimate
  \[
  \Nm{A_\infty v - \lambda_\infty v} \leq \Nm{A-A_\infty} + |\lambda
  -\lambda_\infty| + \Nm{A v - \lambda v}\,.
  \]
  Hence by assumption and definition of $\lambda_\infty$ for any
  $\e>0$ we find $s_0$ such that for all $s \geq s_0$ we have
  \begin{equation}
    \label{eq:AinftyvAv}
    \Nm{A_\infty v - \lambda_\infty v} \leq \e + \Nm{A v- \lambda v}\,.    
  \end{equation}
  Suppose $\e < 1/2c_1$ then the last estimates show that $\Nm{Av -
    \lambda v} \geq 1/2c_1$ and by~\eqref{eq:dmu} also $\dot \gamma(s)
  \leq - 1/4c^2_1 <0$ for all $s \geq s_0$. This contradicts the fact
  that $\gamma(s)$ converges. Thus we have proved that
  $\lambda_\infty$ is an eigenvalue.

  Consider the eigenspace $E = \ker (A_\infty - \lambda_\infty)$ and
  the orthogonal projection $P:H \to E$. We want to show that
   \begin{equation}
     \label{eq:sigmazero}
     \lim_{s\to \infty}  \Nm{v(s) - P v(s)}^2  = 0\;.
   \end{equation}
   Set $\sigma(s):=1/2 \Nm{v(s)-Pv(s)}^2$ and compute with $\<\dot
   v,v\>=0$ and~\eqref{eq:dotv}
   \begin{multline*}
      \dot \sigma = \<\dot v,(1-P) v\> =-\<\dot v,Pv\>=\<Bv,Pv\> +\<Av-\lambda v,Pv\>=\\=\<Bv,Pv\> + \<Av-A_\infty v,Pv\> +(\lambda_\infty-\lambda) \<v,Pv\> +\<A_\infty v -\lambda_\infty v,Pv\>\,.
   \end{multline*}
   The last term on the right-hand side vanishes because $P$ projects
   to the kernel of $A_\infty -\lambda_\infty$ and we conclude that
   the derivative of $\sigma$ converges to zero. Now suppose by
   contradiction that~\eqref{eq:sigmazero} does not hold. Then we find
   a constant $\varepsilon >0$ and a sequence $s_\nu \to \infty$ such
   that $\sigma(s_\nu) \geq \varepsilon$ for all $\nu \in \N$. But
   since the derivative of $\sigma$ converges to zero we have for all
   $\nu \in \N$ sufficiently large
   \[ \nmm{s_\nu- s} \leq 1 \Longrightarrow \sigma(s) \geq
   \varepsilon/2\,.\] Since $A_\infty - \lambda_\infty$ is closed
   there exists a constant $c_2$ such that
   \[ \Nm{v-Pv} \leq c_2 \Nm{A_\infty v - \lambda_\infty v}\,.\]
   By~\eqref{eq:AinftyvAv} we conclude that for all $\nu$ sufficiently
   large and $s \in [s_\nu -1,s_\nu+1]$ we have
   \[ \frac \e 2 \leq \sigma(s) \leq c_2^2\Nm{A(s)v(s)-\lambda(s) v(s)}^2
   + \frac \e 4\,.\] Again by~\eqref{eq:dmu} it follows $\dot
   \gamma(s) \leq - \varepsilon/4c_2^2 < 0$ for all $s \in
   [s_\nu-1,s_\nu+1]$, which contradicts that $\gamma(s)$
   converges. Thus we have proved that
   \begin{equation}
     \label{eq:limP}
     \lim_{s\to \infty} \Nm{v(s) - Pv(s)} = 0,\qquad \lim_{s\to \infty}
   \Nm{Pv(s)} = \lim_{s\to \infty}\<v(s),Pv(s)\> = 1\;. 
   \end{equation}
   Hence for $s$ large enough we have $\Nm{P v(s)} \neq 0$ and define 
   \[ w(s) = \frac{P\xi(s)}{\Nm{P\xi(s)}} =
   \frac{Pv(s)}{\Nm{Pv(s)}}\;.\]
   By assumption $\xi$ solves $\ps \xi + A \xi + B \xi=0$ and
   $\lambda_\infty P \xi = P A_\infty \xi$. We conclude
   \[ P \dot \xi = P (A_\infty - A) \xi - P B \xi - \lambda_\infty P
   \xi\;. \] Abbreviate $\zeta := (A_\infty -A)\xi-B\xi$ and plug this 
   into the computation of the derivative of $w$ to obtain
   \[ \dot w =\frac{P\dot \xi}{\Nm{P\xi}} - \frac{\<P\dot
     \xi,\xi\>}{\Nm{P\xi}^2} w = \frac{P\zeta}{\Nm{P\xi}} -
   \lambda_\infty w - \frac{\<\zeta,P\xi\>}{\Nm{P\xi}^2} w +
   \lambda_\infty w = \frac{P\zeta}{\Nm{P\xi}} -
   \frac{\<\zeta,P\xi\>}{\Nm{P\xi}^2} w \;.\] According
   to~\eqref{eq:limP} we have for all $s$ large enough
   \[ \Nm{\xi(s)} \leq 2\Nm{P\xi(s)},\qquad \Nm{\zeta(s)} \leq 2
   \left(\Nm{A(s)-A_\infty} + \Nm{B(s)}\right)\Nm{P\xi(s)}\;, \] for
   hence 
   \begin{equation}
     \label{eq:normdotw}
     \Nm{\dot w} \leq 4\Nm{A-A_\infty} +4\Nm{B}\;,
   \end{equation}
   Thus $\Nm{\dot w}$ is integrable after
   the assumptions and $w(s)$ converges to some element $v_\infty \in
   E$. By~\eqref{eq:limP} we have furthermore
   \[v_\infty = \lim_{s \to \infty} w(s) = \lim_{s\to\infty} P v(s) =
   \lim_{s\to \infty} v(s)\;.\]
   This proves the lemma.
\end{proof}
Provided Lemma~\ref{lmm:lambdainfty} and
equations~\eqref{eq:limP},~\eqref{eq:normdotw} the proof of
\cite[Lemma 3.5]{Robbin:Strip} and \cite[Lemma 3.6]{Robbin:Strip} goes
through up to very small changes. We state it here.
\begin{lmm}\label{lmm:cvvinfty}
  In the situation of Lemma~\ref{lmm:lambdainfty}. Assume that $s
  \mapsto B(s)$ is continuously differentiable and that there exists
  positive constants $c$ and $\e$ such that
  \[\|A(s)-A_\infty\| + \|B(s)\| + \|\dot A(s)\| + \|\dot B(s)\| \leq
  c e^{-\e s}\;,\] then there exists a non-zero eigenvalue
  $\lambda_\infty$ of $A_\infty$ with corresponding eigenvector
  $v_\infty$ and for every $\mu<\e$ there exists a constant $c$ such that for all $s \geq 0$ we have
  \[\|\xi(s) - e^{-\lambda_\infty s}v_\infty\| \leq c
  e^{-\left(\lambda_\infty + \mu \right)s }\,.\]
\end{lmm}
\section{Viterbo index }\label{app:Vit}
In order to relate the Fredholm index of the Cauchy-Riemann-Floer
operator to topological data we generalize the index defined by
Viterbo~\cite{Viterbo} and Floer~\cite{Floer:relative} to maps with
boundary on not necessarily transversely intersecting Lagrangians in
terms of they Robbin-Salamon index for paths $\mu_\RS$ given in
\cite{Robbin:Paths}.

Let $L_0,L_1 \subset M$ be any two Lagrangian submanifolds and
$H_-,H_+: M\times [0,1] \to \R$ be any two Hamiltonian
functions. Consider the perturbed intersection points
$\I_{H_-}(L_0,L_1)$ and $\I_{H_+}(L_0,L_1)$ as defined
in~\eqref{eq:IH}. To a continuous map $u:[-1,1]\times [0,1] \to M$
satisfying
\begin{equation}
  \label{eq:vitcond}
  u(\pm 1, \cdot)  \in \I_{H_{\pm}}(L_0,L_1),\qquad u(\cdot,0)
  \subset L_0, \qquad u(\cdot,1) \subset L_1\;, 
\end{equation}
we assign an half-integer $\mu_\Vit(u)$.  Let us explain the
construction. Since the base $[-1,1]\times [0,1]$ is contractible the
symplectic bundle $u^*TM$ is trivial and any two trivializations are
homotopic. Choose a symplectic trivialization
\[\Phi_u:[-1,1]\times [0,1] \times \R^{2n} \to u^*TM,\qquad (s,t,\xi) \mapsto \Phi_u(s,t)\xi
\in T_{u(s,t)}M\;, \] that is a bundle isomorphism $\Phi_u$ such that
for all $(s,t) \in [-1,1]\times [0,1]$ and $\xi,\xi'\in \R^{2n}$
\[
\omega_{u(s,t)}(\Phi_u(s,t) \xi,\Phi_u(s,t)\xi') = \ostd(\xi,\xi')\;.
\]
Denote by $\L(n)$ the space of linear Lagrangian subspaces in
$\R^{2n}$ and by $\vp_{H_\pm}:[0,1] \times M \to M$,
$\vp^t_{H_\pm}=\vp_{H_\pm}(t,\cdot)$ the Hamiltonian flow associated
to $H_\pm$. For $s \in [-1,1]$ and $t \in [0,1]$ define the Lagrangian
spaces
\begin{equation}
  \label{eq:L0L1LmLp}
\begin{aligned}
  F_0(s) &= \Phi_u(s,0)^{-1}T_{u(s,0)}L_0, & F_-(t) &=
  \Phi_u(-1,t)^{-1} \d \vp_{H_-}^t T_{u(-1,0)} L_0\\
  F_1(s) &= \Phi_u(s,1)^{-1}T_{u(s,1)}L_1,&F_+(t)
  &=\Phi_u(+1,t)^{-1} \d \vp^t_{H_+} T_{u(1,0)} L_1\;.
\end{aligned}
\end{equation}
We denote by $F_0$, $F_1$, $F_-$ and $F_+$ the
continuous paths of Lagrangian spaces defined by $s \mapsto
F_k(s)$ for $k=0,1$ and $t \mapsto F_\pm(t)$
respectively. Finally the index $\mu_\Vit(u)$ is defined by
\begin{equation}
  \label{eq:Vit}
  \mu_\Vit(u) :=\mu_\RS(F_0,F_1) +  \mu_\RS(F_+,F_1(1)) -  \mu_\RS(F_-,F_1(-1))\;.
\end{equation}
\begin{lmm}\label{lmm:Vitindependent}
  The index $\mu_\Vit(u)$ is defined independently of the choice of
  $\Phi_u$. Moreover given two Hamiltonian functions $H_-,H_+:M \times
  [0,1] \to \R$ and let $u:[0,1]\times [-1,1]\times [0,1] \to M$ be
  such that $u_\tau := u(\tau,\cdot)$ satisfies~\eqref{eq:vitcond}
  for all $\tau \in [0,1]$, then $\mu_\Vit(u_\tau)= \mu_\Vit(u_0)$ for
  all $\tau \in [0,1]$.
\end{lmm}
\begin{proof}
  Given two trivializations $\Phi_0$ and $\Phi_1$. Because
  $[-1,1]\times [0,1]$ is contractible there exists a homotopy
  $\Phi_\tau$ between $\Phi_0$ and $\Phi_1$.  Using $\Phi_\tau$ we
  define via~\eqref{eq:L0L1LmLp} homotopies of paths. Note that
  throughout the homotopy the dimension of the intersection of the two
  spaces at the endpoints is constant. A homotopy satisfying that
  property is called stratum homotopy and leaves the Robbin-Salamon
  index invariant (see \cite[Theorem 2.4]{Robbin:Paths}). Similar we
  proceed for the family $u_\tau$.
\end{proof}
\begin{prp}\label{prp:Vit}
  The index $\mu_\Vit$ has the following properties
  \begin{enumerate}[label=(\roman*)]
  \item\label{nm:Vitold} If $H_+=H_-=0$ vanishes and $L_0, L_1$
    intersect transversely, then the index $\mu_\Vit(u)$ agrees with the
    one given in \cite{Viterbo}.
  \item\label{nm:VitMas1} Given $u:[-1,1]\times [0,1] \to M$
    satisfying~\eqref{eq:vitcond}. Let $v:S^2 \to M$ be a sphere
    (resp.\ $v:(D,\partial D) \to (M,L_k)$ be a disk with $k \in
    \{0,1\}$), such that $v$ has an point (resp.\
    a boundary point)  in common with $u$, then we have
    \begin{align*}
      \mu_\Vit(u\#v)&= \mu_\Vit(u) + 2\<c_1(TM),[v]\>,\\
    (\text{resp. } \mu_\Vit(u\#v)&=\mu_\Vit(u)+\mu_{\Mas}(v))\;,
    \end{align*}
    where by $\#$ we denote the connected sum at the common point (see
    the proof for the concrete definition).
  \item\label{nm:VitMas2} Given $u:[-1,1]\times [0,1]\to M$
    satisfying~\eqref{eq:vitcond} and $v:[0,1]^2 \to M$ satisfying
    $v(0,\cdot)=v(1,\cdot)$, $v(\cdot,0) \subset L_0$ and
    $v(\cdot,1)\subset L_1$. Assume that there exists $s_0 \in [-1,1]$
    such that $u(s_0,\cdot)=v(0,\cdot)$ then it holds
    \[\mu_\Vit(u\#v)= \mu_\Vit(u) + \mu_\Mas(v)\;,\]
    here $u\#v$ denotes the connected sum along the path $u(s_0,\cdot)$.
  \item\label{nm:Vitadd} Given three Hamiltonian functions
    $H_0,H_1,H_2:[0,1]\times M \to \R$ and maps $u_0,u_1$ satisfying
    the boundary condition $u_k|_{t=0} \subset L_0$, $u_k|_{t=1}
    \subset L_1$ for $k \in \{0,1\}$ and
    \begin{gather*}
      u_0(-1,\cdot) \in I_{H_0}(L_0,L_1),\\ 
    u_0(1,\cdot)=u_1(-1,\cdot) \in I_{H_1}(L_0,L_1),\\
    u_1(1,\cdot) \in I_{H_2}(L_0,L_1)\;.
    \end{gather*}
    Then we have
    \[\mu_\Vit(u_0\#u_1) = \mu_\Vit(u_0)+\mu_\Vit(u_1)\;.\]
    Where $u_0\#u_1$ denotes the connected sum.
  \item\label{nm:Vitzero} Assume that $H=H_-=H_+$ is clean for the pair
    $(L_0,L_1)$ and given $u:[-1,1]\times [0,1] \to M$ such that
    $u(s,\cdot) \in \I_H(L_0,L_1)$ for all $s \in [-1,1]$ then
    $\mu_\Vit(u) = 0$.
  \item\label{nm:muint} Let $H_-,H_+$ be clean. Given two connected
    components $C_- \subset \I_{H_-}(L_0,L_1)$ and $C_+ \subset
    \I_{H_+}(L_0,L_1)$.  Suppose $u:[-1,1] \times [0,1] \to M$
    satisfies~\eqref{eq:vitcond} and $u(\pm 1,\cdot) = x_\pm \in
    C_\pm$, then
    \[
    \mu_\Vit(u) + \frac 12 \left(\dim C_+ + \dim C_-\right) \in \Z\;.
    \]
  \end{enumerate}
\end{prp}

\begin{proof}
  We will deduce these properties from the axioms of the
  Robbin-Salamon index.
  \begin{stp}
    We show~\ref{nm:Vitold}
  \end{stp}
  Since $L_0,L_1$ intersect transversely, the intersection $L_0 \cap
  L_1$ is a discrete set of points. Necessarily the map $t \mapsto
  u(\pm 1,t) = p_\pm$ is constant. We choose a trivialization $\Phi$
  satisfying for all $t \in [0,1]$
  \[\Phi(\pm 1, t)T_{p_\pm}L_0 = \R^n \oplus \{0\},\qquad \Phi(\pm 1,
  t)T_{p_\pm} L_1= \{0\} \oplus \R^n\;.\] Since $H_-=H_+=0$ the paths
  $F_\pm$ defined in~\eqref{eq:L0L1LmLp} are constant. Define the path
  $\ol F:[0,1] \to \L(n)$ via
  \[
  t\mapsto \ol F(t)= \left(\cos(\pi t/2)\one + \sin(\pi t/2 )\Jstd\right) \left(\R^n \oplus \{0\}\right)\;.
  \]
  Denote by $\ol F^\vee$, $F_1^\vee$ \etc the path $\ol F$, $F_1$ run with
  reverse orientation. We define a loop $F_{\mathrm{loop}}:S^1 \to
  \L(n)$ via the concatenation of the paths
  \[
  F_{\mathrm{loop}} = F_0 \# \ol F \# F_1^{\vee} \# \ol F^\vee\;.
  \]
  In~\cite{Viterbo} the index is defined as the Maslov index of the
  loop $F_{\mathrm{loop}}$.  Using~\cite[Remark 2.6]{Robbin:Paths} we
  have fix 
  \begin{equation}
    \label{eq:Vdefold}
  \mu_\Mas(F_{\mathrm{loop}}) = \mu_\RS(F_{\mathrm{loop}},F_1(0))\;.
  \end{equation}
  By the concatenation axiom, the homotopy axiom and the zero axiom we
  have
  \[\mu_\RS(\ol F,F_1(1)) + \mu_\RS(\ol F^\vee,F_1(1)) = \mu_\RS(\ol F
  \# \ol F^\vee,F_1(1)) = \mu_\RS(\ol F(0),F_1(1))=0.  \] Similarly
  $\mu_\RS(F_1,F_1(1)) + \mu_\RS(F_1^\vee,F_1(1))= 0$ and
  \[\mu_\RS(F_1(0),F_1) + \mu_\RS(F_1,F_1(1)) = 0\,,\] which shows
  $\mu_\RS(F^\vee_1,F_1(1)) = \mu_\RS(F_1(0),F_1)$. Since $F_1$ is a
  loop we have 
  \[\mu_\RS(F_1(0),F_1) = \mu_\RS(F_0(1),F_1)\,.\]
  Using the last identities we continue with~\eqref{eq:Vdefold}
  \begin{align*}
    \mu_\Mas(F_{\mathrm{loop}}) &=\mu_\RS(F_0 \#
    \ol F \# F_1^\vee \# \ol F^\vee, F_1(1))\\
    &=\mu_\RS(F_0,F_1(1)) + \mu_\RS(\ol F,F_1(1)) -
    \mu_\RS(F_1,F_1(1)) - \mu_\RS(\ol F,F_1(1))\\
    &=\mu_\RS(F_0,F_1(1)) + \mu_\RS(F_1(0),F_1)\\
    &=\mu_\RS(F_0,F_1)\\
    &=\mu_\Vit(u)\;.
  \end{align*}
  \begin{stp}
    We show~\ref{nm:VitMas1}
  \end{stp}
  Let $v:(D,\partial D) \to (M,p)$ be a sphere and assume for
  simplicity that the sphere has an interior point in common with $u$,
  say $u(1/2,0) = p$. Without loss of generality we assume that
  $D:=\{(s,t) \in \R^2 \mid s^2+(t-1/2)^2\leq 1/4\}\subset
  [-1,1]\times [0,1]$. Let $\vp:[-1,1]\times [0,1] \to [-1,1]\times
  [0,1]$ be a continuous map, which maps $D$ to the point
  $\{(1/2,0)\}$, is a homeomorphism on the complement and fixes each
  arc of the boundary. We define the connected sum
  \[
  (u\#v)(s,t) :=
  \begin{cases}
    v(s,t)&\text{if } (s,t) \in D\\
    u(\vp(s,t))&\text{if } (s,t) \not\in D\;.
  \end{cases}
  \]
  Choose symplectic trivializations $\Phi_u:u^*TM \to \Rbundle^{2n}$
  and $\Phi_v:v^*TM \to \Rbundle^{2n}$. The change of trivialization
  defines a loop $\psi:\partial D \to \Sp(2n)$, $(s,t) \mapsto
  \Phi_v(s,t)\circ \Phi_u(1/2,0)^{-1}$. We assume with loss of
  generality that $\psi$ is unitary. Let $\Psi:[-1,1]\times [0,1]
  \setminus D \to \U(n)$ be an extension of $\psi$ such that
  $\Psi(s,t) = \psi(s,t)$ for all $(s,t) \in \partial D$ and
  $\Psi(s,t)=\one$ for $(s,t) \in \{-1\}\times [0,1]\cup [-1,1]\times
  \{1\} \cup \{1\}\times [0,1]$. Abbreviate $w=u\#v$. We define
  the symplectic trivialization $\Phi_w:w^*TM \to \Rbundle^{2n}$ via
  \[\Phi_w(s,t)=
  \begin{cases}
    \Phi_v(s,t)&\text{if } (s,t) \in D\\
    \Psi(s,t) \circ \Phi_u(\vp(s,t)) &\text{if } (s,t) \not\in D\;.
  \end{cases}
  \]
  Define $F^u_0$, $F^u_1$, $F^u_\pm$ (\resp
  $F^w_0$, $F^w_1$, $F^w_\pm$) via~\eqref{eq:L0L1LmLp} using $\Phi_u$
  (\resp $\Phi_w$). By construction we have $F^w_0(s) = \Psi(s,0)
  F^u_0(\vp(s,0))$ for all $s \in [-1,1]$, $F^u_\pm = F^w_\pm$ and
  $F^u_1 = F^w_1$. We have after a homotopy in the domain
  $\psi'F_0^u = \psi'F_0(-1)\#F_0$. Hence by the concatenation axiom
  \begin{equation}
    \label{eq:Vitwut} 
\begin{aligned}
    \mu_\Vit(w) &= \mu_\RS(F^w_0,F^w_1)+ \mu_\RS(F^w_+,F^w_1(1)) -
    \mu_\RS(F^w_-,F^w_1(-1))\\
    &=\mu_\RS(\psi' F_0^u,F^u_1) + \mu_\RS(F^u_+,F^u_1(1)) - \mu_\RS(F^u_-,F^u_1(-1))\\
    &= \mu_\RS(\psi' F_0^u(-1),F_1^u(-1)) + \mu_\RS(F^u_0,F^u_1) +
    \mu_\RS(F^u_+,F^u_1(1))-\\ &\hspace{8cm}-\mu_\RS(F^u_-,F^u_1(-1))\\
    &= \mu_\Mas(\psi' F_0^u(-1)) + \mu_\Vit(u)\;.
  \end{aligned}   
  \end{equation}
  Abbreviate $F:=F_0^u(-1)$. Since $\psi$ and $\psi'$ are homotop
  within $\U(n)$ (via $\Psi$), the Maslov index of the loop $\psi' F$
  is the same as $\psi F$, which by definition is given as the degree
  of the map $\det(\psi \circ \psi):S^1 \to S^1$ (see
  \cite[C.3.1]{Bibel}). On the other hand it is a classical fact that
  $\deg \det \psi = \<c_1(TM),[v]\>$. We summarize the argument
  \[ 
  \mu_\Mas(\psi' F) = \mu_\Mas(\psi F) = \deg \det (\psi \circ \psi) = 2 \deg \det \psi = 2\<c_1(TM),[v]\>\;.
  \]
  This shows the identity for spheres using equation~\eqref{eq:Vitwut}. 

  Now let $v:(D,\partial D) \to (M,L_0)$ be a disk. Denote by $z_0
  \in \partial D$ and $s_0 \in [-1,1]$ the points such that $v(z_0) =
  u(s_0,0)$. We give the definition of the connected sum: Define
  $\Omega_0 :=\{ (s,t) \in [-1,1]\times [0,1] \mid s^2 + t^2 \leq
  1/2\}$ and let $\vp_0:\Omega_0 \to D$ be a continuous map which maps
  the arc $\gamma:=\{(s,t)\in [-1,1]\times [0,1] \mid s^2 + t^2 =
  1/2\}$ to the point $\{z\}$ and which is a homeomorphism on the
  complement. Secondly define $\Omega_1 :=\{(s,t) \in [-1,1]\times
  [0,1] \mid s^2 + t^2 \geq 1/2\}$ and let $\vp_1:\Omega_1\to
  [-1,1]\times [0,1]$ be a continuous map such that $\vp(\gamma) =
  \{(s_0,0)\}$, which is a homeomorphism on the complement and which
  fixes the arcs $\{-1\}\times [0,1]$, $[-1,1]\times \{1\}$ and
  $\{1\}\times [-1,1]$. With these preparations we define the
  \emph{connected sum}
  \[
  (u\# v) (s,t) :=
  \begin{cases}
    v(\vp_0(s,t))&\text{if } (s,t) \in \Omega_0\\
    u(\vp_1(s,t))&\text{if } (s,t) \in \Omega_1\;.
  \end{cases}
  \]
  We deduce the equation. Choose symplectic trivializations
  $\Phi_u:u^*TM \to \Rbundle^{2n}$ and $\Phi_v:v^*TM \to
  \Rbundle^{2n}$ that agree on $u(s_0,0) = v(z_0)$. Using $\Phi_u$ and
  $\Phi_v$ we obtain a trivialization $\Phi_w:w^*TM\to \Rbundle^{2n}$
  where $w=u\#v$. Define $F^u_0$, $F^u_1$, $F^u_\pm$ (\resp $F^w_0$,
  $F^w_1$, $F^w_\pm$) via~\eqref{eq:L0L1LmLp} using $\Phi_u$ (\resp
  $\Phi_w$). Secondly define $F^v(s) :=
  \Phi_v(\vp_0(s,0))T_{v(\vp_0(s,0))}L_0$ for all $s\in
  [-1/2,1/2]$. By construction $F^v(-1/2) = F^v(1/2)$, $F^w_1=F^u_1$,
  $F^w_\pm = F^u_\pm$ and $F^w_0 = F^u_0|_{[-1,s_0]} \#
  F^v\#F^u_0|_{[s_0,1]}$. Then using the concatenation axiom
  \begin{align*}
    \mu_\Vit(w) &= \mu_\RS(F^w_0,F^w_0) + \mu_\RS(F^w_+,F^w_1(1)) - \mu_\RS(F^w_-,F^w_1(-1))\\
    &=\mu_\RS(F^u_0|_{[-1,s_0]},F^u_1|_{[-1,s_0]}) +
    \mu_\RS(F^v,F^u_1(s_0)) + \mu_\RS(F^u_0|_{[s_0,1]},F^u_1|_{[s_0,1]})\\
   &\qquad\mu_\RS(F^u_+,F^u_1(1)) - \mu_\RS(F^u_-,F^u_1(-1))\\
    &=\mu_\RS(F^u_0,F^u_1) + \mu_\RS(F^u_+,F^u_1(1)) -
    \mu_\RS(F^u_-,F^u_1(-1))+ \mu_\RS(F^v,F^u_1(s_0))\\
    &=\mu_\Vit(u) + \mu_\Mas(v)\;.
  \end{align*}
 \begin{stp}
    We show~\ref{nm:VitMas2}
  \end{stp}
  Choose a symplectic trivialization $\Phi_u:u^*TM \to \Rbundle^{2n}$
  and $\Phi_v:v^*TM\to \Rbundle^{2n}$ that agree over
  $v(0,\cdot)=u(s,\cdot)$. We obtain a symplectic trivialization
  $\Phi_w$ of $w^*TM$ where $w$ is the connected sum $u\#v$. Define
  paths of Lagrangians $F^u_0,F^u_1,F^u_-$ and $F^u_+$
  via~\eqref{eq:L0L1LmLp} using $\Phi_u$. Similarly define
  \[F^v_0(s):= \Phi_v(s,0)T_{v(s,t)}L_0,\quad F^v_1(s):=\Phi_v(s,1)T_{v(s,1)}L_1\,,\] for all $s \in [0,1]$. Further
  define $F^w_0,F^w_1,F^w_-$ and $F^w_+$ via~\eqref{eq:L0L1LmLp} using
  $\Phi_w$. By construction we have $F^v_k(0) = F^v_k(1)=F^u_k(s_0)$,
  $F^w_k = F_k^u|_{[-1,s_0]} \# F_k^v \#F_k^u|_{[s_0,1]}$ for $k =0,1$
  and $F^w_\pm =F^u_\pm$. Using the concatenation axiom we have
  \begin{align*}
    \mu_\Vit(w)& = \mu_\RS(F^w_0,F^w_1) + \mu_\RS(F^w_+,F^w_1(1)) - \mu_\RS(F^w_-,F^w_1(-1))\\
    &= \mu_\RS(F^u_0|_{[-1,s_0]},F^u_1|_{[-1,s_0]}) +
    \mu_\RS(F^v_0,F^v_1) + \mu_\RS(F^u_0|_{[s_0,1]},F^u_0|_{[s_0,1]})\\
    &\qquad + \mu_\RS(F^u_+,F^u_1(1)) - \mu_\RS(F^u_-,F^u_1(-1))\\
    &= \mu_\RS(F^u_0,F^u_1) + \mu_\RS(F^u_+,F^u_1(1)) - \mu_\RS(F^u_-,F^u_1(1)) + \mu_\RS(F^v_0,F^v_1)\\
    &=\mu_\Vit(u) + \mu_\Mas(v)\;.
  \end{align*}
  \begin{stp}
    We show~\ref{nm:Vitadd}
  \end{stp}
  Choose symplectic trivialization $\Phi_k:u_k^*TM \to \Rbundle^{2n}$
  for $k=0,1$ that agree over $u_0(1,\cdot)$ and $u_1(-1,\cdot)$. We
  obtain a symplectic trivialization of the connected sum $u_2 =
  u_0\#u_1$ denoted $\Phi_2$.  For $k=0,1,2$ denote the paths
  $F_0^{\Phi_k}$,$F_1^{\Phi_k}$ and $F_\pm^{\Phi_k}$ associated to
  $\Phi_k$ by~\eqref{eq:L0L1LmLp}. We have $F_+^{\Phi_0}
  =F_-^{\Phi_1}$ and $F_k^{\Phi_0}(1) = F_k^{\Phi_1}(-1)$ for
  $k=0,1$. By the concatenation axiom
  \[
  \mu(F_0^{\Phi_0},F_1^{\Phi_0}) + \mu(F_0^{\Phi_1},F_1^{\Phi_1}) =
  \mu(F_0^{\Phi_2},F_1^{\Phi_2})\,.
  \]
  The shows $\mu(u_0) + \mu(u_1) = \mu(u_2)$ inserting the definitions.
  \begin{stp}
    We show \ref{nm:Vitzero}.
  \end{stp}
  Choose a symplectic trivialization $\Phi_0$ of $u(\cdot,0)^*TM$ and
  define the trivialization $\Phi$ of $u^*TM$ by
  \[\Phi(s,t) = \d \varphi_H^t \Phi_0(s)\;,\]
  where $\varphi_H^t$ denotes the Hamiltonian flow of $H$. By the
  property of the Hamiltonian flow $\Phi$ is a symplectic
  trivialization. By definition
  \begin{multline*}
     F_{\pm}(t) = \Phi(\pm 1,t)^{-1}\d \varphi_H^t T_{u(\pm 1,0)}L_0 =
  \Phi_0(\pm 1)^{-1} \left(\d \varphi_H^t\right)^{-1} \d \varphi_H^t
  T_{u(\pm 1,0)}L_0=\\ = F_0(\pm 1)\;.
  \end{multline*}
   Thus $\mu_{\RS}(F_{\pm},F_1(\pm
  1))$ vanishes after the zero axiom.  Since $\varphi_H(L_0)$
  intersects $L_1$ cleanly we have for all $s \in [0,1]$
  \begin{align*}
    F_0(s) \cap F_1(s) &= \Phi(s,0)^{-1}T_{u(s,0)}L_0 \cap
    \Phi(s,1)^{-1} T_{u(s,1)}L_1\\
    &= \Phi(s,1)^{-1} \left(\d \varphi^1_H T_{u(s,0)}L_0
      \cap T_{u(s,1)} L_1\right)\\
    &= \Phi(s,1)^{-1} \left( T_{u(s,1)} \vp_H^1(L_0)
      \cap T_{u(s,1)} L_1\right)\\
    &= \Phi(s,1)^{-1}T_{u(s,1)}\left( \varphi^1_H \left(L_0\right)\cap
      L_1\right)\;.
  \end{align*}
  We see that the dimension of $F_0(s) \cap F_1(s)$ is constant for
  all $s \in [-1,1]$. This shows that $\mu_{\RS}(F_0,F_1)=0$ by the
  zero axiom.
  \begin{stp}
    We show~\ref{nm:muint}
  \end{stp}
  Choose a symplectic trivialization $\Phi$. Define the Lagrangian
  paths $F_0,F_1,F_-$ and $F_+$ by~\eqref{eq:L0L1LmLp}. We have by
  definition $F_0(-1) = F_-(0)$ and $F_0(1) = F_+(0)$. As a result the
  concatenation $F_{-,0,+}= F_-^{-1}\# F_0 \# F_+$ is a well-defined
  continuous path of Lagrangian subspaces starting from $F_-(1)$ and
  ending at $F_+(1)$. By the concatenation axiom and~\cite[Theorem
  2.4]{Robbin:Paths} we have
  \[
  \mu_\Vit(u) = \mu_\RS(F_{-,0,+},F_1) = \frac 12 \left(\dim F_+(1)
    \cap F_1(1) - \dim F_-(1) \cap F_1(-1)\right) + \Z\;.
  \]
  This shows the claim since 
  \[T_{x_+} C_+ = \Phi(1,1)\left(F_+(1) \cap F_1(1)\right),\qquad
  T_{x_-} C_- = \Phi(-1,1) \left(F_-(1) \cap F_1(-1)\right)\;.\] 
\end{proof}

\bibliography{bibl}

\begin{thebibliography}{10}

\bibitem{AbbMajer:MorseI}
Alberto Abbondandolo and Pietro Majer.
\newblock A {M}orse complex for infinite dimensional manifolds. {I}.
\newblock {\em Adv. Math.}, 197(2):321--410, 2005.

\bibitem{AbbMajer:inf}
Alberto Abbondandolo and Pietro Majer.
\newblock Infinite dimensional {G}rassmannians.
\newblock {\em J. Operator Theory}, 61(1):19--62, 2009.

\bibitem{AbbSchwarz:loop}
Alberto Abbondandolo and Matthias Schwarz.
\newblock Floer homology of cotangent bundles and the loop product.
\newblock {\em Geom. Topol.}, 14(3):1569--1722, 2010.

\bibitem{AlbertoSchwarz:Legendre}
Alberto Abbondandolo and Matthias Schwarz.
\newblock The role of the {L}egendre transform in the study of the {F}loer
  complex of cotangent bundles.
\newblock arXiv:1306.4087v1, 2013.

\bibitem{Abouzaid:spheres}
Mohammed Abouzaid.
\newblock Framed bordism and {L}agrangian embeddings of exotic spheres.
\newblock {\em Ann. of Math. (2)}, 175(1):71--185, 2012.

\bibitem{Peter}
Peter Albers.
\newblock {\em On functoriality in {F}loer homology}.
\newblock PhD thesis, University of Leipzig, 2004.

\bibitem{Peter:PSS}
Peter Albers.
\newblock A {L}agrangian {P}iunikhin-{S}alamon-{S}chwarz morphism and two
  comparison homomorphisms in {F}loer homology.
\newblock {\em Int. Math. Res. Not. IMRN}, 4:56, 2008.

\bibitem{Biran:nonintersection}
P.~Biran.
\newblock Lagrangian non-intersections.
\newblock {\em Geom. Funct. Anal.}, 16(2):279--326, 2006.

\bibitem{BiranCornea:pearl}
Paul Biran and Octav Cornea.
\newblock {Q}uantum structures for {L}agrangian submanifolds.
\newblock arXiv:0708.4221[math.SG], 2007.

\bibitem{BiranCornea:quantum}
Paul Biran and Octav Cornea.
\newblock A {L}agrangian quantum homology.
\newblock In {\em New perspectives and challenges in symplectic field theory},
  volume~49 of {\em CRM Proc. Lecture Notes}, pages 1--44. Amer. Math. Soc.,
  Providence, RI, 2009.

\bibitem{BiranCornea:rigidity}
Paul Biran and Octav Cornea.
\newblock Rigidity and uniruling for {L}agrangian submanifolds.
\newblock {\em Geom. Topol.}, 13(5):2881--2989, 2009.

\bibitem{BottTu}
Raoul Bott and Loring~W. Tu.
\newblock {\em Differential forms in algebraic topology}, volume~82 of {\em
  Graduate Texts in Mathematics}.
\newblock Springer-Verlag, New York, 1982.

\bibitem{Cho}
Cheol-Hyun Cho.
\newblock Holomorphic discs, spin structures, and {F}loer cohomology of the
  {C}lifford torus.
\newblock {\em Int. Math. Res. Not.}, 35:1803--1843, 2004.

\bibitem{Ekholm:orient}
Tobias Ekholm, John Etnyre, and Michael Sullivan.
\newblock Orientations in {L}egendrian contact homology and exact {L}agrangian
  immersions.
\newblock {\em Internat. J. Math.}, 16(5):453--532, 2005.

\bibitem{Floer:Coherent}
A.~Floer and H.~Hofer.
\newblock Coherent orientations for periodic orbit problems in symplectic
  geometry.
\newblock {\em Math. Z.}, 212(1):13--38, 1993.

\bibitem{Floer:Intersection}
Andreas Floer.
\newblock Morse theory for {L}agrangian intersections.
\newblock {\em J. Differential Geom.}, 28(3):513--547, 1988.

\bibitem{Floer:relative}
Andreas Floer.
\newblock A relative {M}orse index for the symplectic action.
\newblock {\em Comm. Pure Appl. Math.}, 41(4):393--407, 1988.

\bibitem{Floer:Action}
Andreas Floer.
\newblock The unregularized gradient flow of the symplectic action.
\newblock {\em Comm. Pure Appl. Math.}, 41(6):775--813, 1988.

\bibitem{Floer:Sphere}
Andreas Floer.
\newblock Symplectic fixed points and holomorphic spheres.
\newblock {\em Comm. Math. Phys.}, 120(4):575--611, 1989.

\bibitem{Floer:Witten}
Andreas Floer.
\newblock Witten's complex and infinite-dimensional {M}orse theory.
\newblock {\em J. Differential Geom.}, 30(1):207--221, 1989.

\bibitem{Floer:Monopoles}
Andreas Floer.
\newblock Monopoles on asymptotically flat manifolds.
\newblock In {\em The {F}loer memorial volume}, volume 133 of {\em Progr.
  Math.}, pages 3--41. Birkh\"auser, Basel, 1995.

\bibitem{Floer:Trans}
Andreas Floer, Helmut Hofer, and Dietmar Salamon.
\newblock Transversality in elliptic {M}orse theory for the symplectic action.
\newblock {\em Duke Math. J.}, 80(1):251--292, 1995.

\bibitem{Fortune}
Barry Fortune.
\newblock A symplectic fixed point theorem for {${\bf C}{\rm P}^n$}.
\newblock {\em Invent. Math.}, 81(1):29--46, 1985.

\bibitem{Frauenfelder:PhD}
Urs Frauenfelder.
\newblock {\em {F}loer homology of symplectic quotients and the
  {A}rnold-{G}ivental conjecture}.
\newblock PhD thesis, ETH Z{\"u}rich, 2003.

\bibitem{Frauenfelder:disks}
Urs Frauenfelder.
\newblock Gromov convergence of pseudoholomorphic disks.
\newblock {\em J. Fixed Point Theory Appl.}, 3(2):215--271, 2008.

\bibitem{FO3:I}
Kenji Fukaya, Yong-Geun Oh, Hiroshi Ohta, and Kaoru Ono.
\newblock {\em Lagrangian intersection {F}loer theory: anomaly and obstruction.
  {P}art {I}}, volume~46 of {\em AMS/IP Studies in Advanced Mathematics}.
\newblock American Mathematical Society, Providence, RI, 2009.

\bibitem{FO3:II}
Kenji Fukaya, Yong-Geun Oh, Hiroshi Ohta, and Kaoru Ono.
\newblock {\em Lagrangian intersection {F}loer theory: anomaly and obstruction.
  {P}art {II}}, volume~46 of {\em AMS/IP Studies in Advanced Mathematics}.
\newblock American Mathematical Society, Providence, RI, 2009.

\bibitem{FO3:integers}
Kenji Fukaya, Yong-Geun Oh, Hiroshi Ohta, and Kaoru Ono.
\newblock Lagrangian {F}loer theory over integers: spherically positive
  symplectic manifolds.
\newblock {\em Pure Appl. Math. Q.}, 9(2):189--289, 2013.

\bibitem{Hofer:Floer}
H.~Hofer and D.~A. Salamon.
\newblock Floer homology and {N}ovikov rings.
\newblock In {\em The Floer memorial volume}, volume 133 of {\em Progr. Math.},
  pages 483--524. Birkh\"auser, Basel, 1995.

\bibitem{IvashkovichShevchishin}
S.~Ivashkovich and V.~Shevchishin.
\newblock Reflection principle and {$J$}-complex curves with boundary on
  totally real immersions.
\newblock {\em Commun. Contemp. Math.}, 4(1):65--106, 2002.

\bibitem{Joyce:corners}
Dominic Joyce.
\newblock On manifolds with corners.
\newblock In {\em Advances in {G}eometric {A}nalysis}, Advanced {L}ectures in
  {M}athematics, pages 225 -- 258. International press, Boston, 2012.

\bibitem{Kato}
Tosio Kato.
\newblock {\em Perturbation theory for linear operators}.
\newblock Springer-Verlag, Berlin-New York, second edition, 1976.
\newblock Grundlehren der Mathematischen Wissenschaften, Band 132.

\bibitem{Spin}
H.~Blaine Lawson, Jr. and Marie-Louise Michelsohn.
\newblock {\em Spin geometry}, volume~38 of {\em Princeton Mathematical
  Series}.
\newblock Princeton University Press, Princeton, NJ, 1989.

\bibitem{Martin}
Shaun Martin.
\newblock Symplectic quotients by a nonabelian group and by its maximal torus.
\newblock arXiv:math/0001002, 2000.

\bibitem{McCleary}
John McCleary.
\newblock {\em A user's guide to spectral sequences}, volume~58 of {\em
  Cambridge Studies in Advanced Mathematics}.
\newblock Cambridge University Press, Cambridge, second edition, 2001.

\bibitem{Intro}
Dusa McDuff and Dietmar Salamon.
\newblock {\em Introduction to symplectic topology}.
\newblock Oxford Mathematical Monographs. The Clarendon Press Oxford University
  Press, New York, second edition, 1998.

\bibitem{Bibel}
Dusa McDuff and Dietmar Salamon.
\newblock {\em {$J$}-holomorphic curves and symplectic topology}, volume~52 of
  {\em American Mathematical Society Colloquium Publications}.
\newblock American Mathematical Society, Providence, RI, 2004.

\bibitem{Milnor:classes}
John~W. Milnor and James~D. Stasheff.
\newblock {\em Characteristic classes}.
\newblock Princeton University Press, Princeton, N. J., 1974.
\newblock Annals of Mathematics Studies, No. 76.

\bibitem{Oancea:Leray}
Alexandru Oancea.
\newblock Fibered symplectic cohomology and the {L}eray-{S}erre spectral
  sequence.
\newblock {\em J. Symplectic Geom.}, 6(3):267--351, 2008.

\bibitem{Oh:diskI}
Yong-Geun Oh.
\newblock Floer cohomology of {L}agrangian intersections and pseudo-holomorphic
  disks. {I}.
\newblock {\em Comm. Pure Appl. Math.}, 46(7):949--993, 1993.

\bibitem{Oh:spectral}
Yong-Geun Oh.
\newblock Floer cohomology, spectral sequences, and the {M}aslov class of
  {L}agrangian embeddings.
\newblock {\em Internat. Math. Res. Notices}, 7:305--346, 1996.

\bibitem{Schwarz:PSS}
S.~Piunikhin, D.~Salamon, and M.~Schwarz.
\newblock Symplectic {F}loer-{D}onaldson theory and quantum cohomology.
\newblock In {\em Contact and symplectic geometry ({C}ambridge, 1994)},
  volume~8 of {\em Publ. Newton Inst.}, pages 171--200. Cambridge Univ. Press,
  Cambridge, 1996.

\bibitem{Pozniak}
Marcin Po{\'z}niak.
\newblock {\em Floer homology, {N}ovikov rings and clean intersections}.
\newblock PhD thesis, University of Warwick, 1994.

\bibitem{Robbin:Paths}
Joel Robbin and Dietmar Salamon.
\newblock The {M}aslov index for paths.
\newblock {\em Topology}, 32(4):827--844, 1993.

\bibitem{Robbin:Flow}
Joel Robbin and Dietmar Salamon.
\newblock The spectral flow and the {M}aslov index.
\newblock {\em Bull. London Math. Soc.}, 27(1):1--33, 1995.

\bibitem{Robbin:Strip}
Joel~W. Robbin and Dietmar~A. Salamon.
\newblock Asymptotic behaviour of holomorphic strips.
\newblock {\em Ann. Inst. H. Poincar\'e Anal. Non Lin\'eaire}, 18(5):573--612,
  2001.

\bibitem{Salamon:Lecture}
Dietmar Salamon.
\newblock Lectures on {F}loer homology.
\newblock In {\em Symplectic geometry and topology (Park City, UT, 1997)},
  volume~7 of {\em IAS/Park City Math. Ser.}, pages 143--229. Amer. Math. Soc.,
  Providence, RI, 1999.

\bibitem{Schwarz:Buch}
Matthias Schwarz.
\newblock {\em Morse homology}, volume 111 of {\em Progress in Mathematics}.
\newblock Birkh\"auser Verlag, Basel, 1993.

\bibitem{Schwarz:PhD}
Matthias Schwarz.
\newblock {\em Cohomology Operations from $S^1$-Cobordisms in {F}loer
  Homology}.
\newblock PhD thesis, Ruhr-Universit{\"a}t Bochum, 1995.

\bibitem{Seidel:knotted}
Paul Seidel.
\newblock Lagrangian two-spheres can be symplectically knotted.
\newblock {\em J. Differential Geom.}, 52(1):145--171, 1999.

\bibitem{Seidel:sequence}
Paul Seidel.
\newblock A long exact sequence for symplectic {F}loer cohomology.
\newblock {\em Topology}, 42(5):1003--1063, 2003.

\bibitem{Tatjana}
Tatjana Sim\v{c}evi\'c.
\newblock {\em {A} {H}ardy-space approach to {L}agrangian {F}loer gluing}.
\newblock PhD thesis, ETH Z\"urich, 2014.

\bibitem{Viterbo}
Claude Viterbo.
\newblock Intersection de sous-vari\'et\'es lagrangiennes, fonctionnelles
  d'action et indice des syst\`emes hamiltoniens.
\newblock {\em Bull. Soc. Math. France}, 115(3):361--390, 1987.

\bibitem{WW:orient}
Katrin Wehrheim and Chris~T. Woodward.
\newblock Orientations for pseudoholomorphic quilts, 2007.

\end{thebibliography}
\bibliographystyle{plain}

\end{document}